\documentclass[a4paper,10pt]{memoir}
\usepackage[utf8]{inputenc}
\usepackage[left=4cm, right=3cm]{geometry}

\usepackage{standalone}

\usepackage{amsmath, amssymb, amsthm}
\usepackage[utf8]{inputenc}
\usepackage{alltt}
\usepackage{amscd}
\usepackage{url}
\usepackage{setspace}
\usepackage{xcolor}
\usepackage{tikz}
\usepackage{float}
\usepackage{subfloat}
\usepackage{stmaryrd}
\usepackage{MnSymbol}
\usepackage{enumitem}
\usepackage[all]{xy}
\usepackage{mdframed}
\usepackage{imakeidx}
\makeindex
\usepackage{listings}
\usepackage{marginnote}

\usetikzlibrary{matrix}

\theoremstyle{definition}
\newtheorem{defn}{Definition}[section]

\newtheorem{ex}[defn]{Example}

\newtheorem{sitn}[defn]{Situation}

\newtheorem{opbm}[defn]{Open problem}

\theoremstyle{plain}
\newtheorem{prop}[defn]{Proposition}

\newtheorem{cor}[defn]{Corollary}

\newtheorem{lemma}[defn]{Lemma}

\newtheorem{thm}[defn]{Theorem}
\newtheorem{lemdef}[defn]{Lemma and Definition}

\newtheorem{var}[defn]{Variant}

\theoremstyle{remark}
\newtheorem{rem}[defn]{Remark}

\newtheorem{constr}[defn]{Construction}

\newcommand{\remove}[1]{}

\gdef\notesoff{\gdef\note##1{}}

\notesoff

\newcommand{\kk}{\mathbf{k}}

\newcommand{\bAA}{\mathbb{A}}

\newcommand{\CC}{\mathbb{C}}

\newcommand{\GG}{\mathbb{G}}
\newcommand{\HH}{\mathbb{H}}
\newcommand{\NN}{\mathbb{N}}
\newcommand{\PP}{\mathbb{P}}
\newcommand{\QQ}{\mathbb{Q}}
\newcommand{\RR}{\mathbb{R}}

\newcommand{\TT}{\mathbb{T}}
\newcommand{\UU}{\mathbb{U}}
\newcommand{\ZZ}{\mathbb{Z}}

\newcommand{\A}{\mathcal{A}}
\newcommand{\B}{\mathcal{B}}
\newcommand{\C}{\mathcal{C}}
\newcommand{\D}{\mathcal{D}}
\newcommand{\E}{\mathcal{E}}
\newcommand{\F}{\mathcal{F}}
\newcommand{\G}{\mathcal{G}}
\newcommand{\cH}{\mathcal{H}}
\newcommand{\I}{\mathcal{I}}

\newcommand{\K}{\mathcal{K}}
\newcommand{\cL}{\mathcal{L}}
\newcommand{\M}{\mathcal{M}}
\newcommand{\N}{\mathcal{N}}
\newcommand{\cO}{\mathcal{O}}
\newcommand{\cP}{\mathcal{P}}
\newcommand{\Q}{\mathcal{Q}}
\newcommand{\R}{\mathcal{R}}
\newcommand{\cS}{\mathcal{S}}
\newcommand{\T}{\mathcal{T}}

\newcommand{\U}{\mathcal{U}}
\newcommand{\V}{\mathcal{V}}
\newcommand{\W}{\mathcal{W}}
\newcommand{\X}{\mathcal{X}}

\newcommand{\Z}{\mathcal{Z}}

\newcommand{\m}{\mathfrak{m}}

\newcommand{\bA}{\mathbf{A}}

\newcommand{\bC}{\mathbf{C}}

\newcommand{\TW}{\mathrm{TW}}

\newcommand{\Spec}{\mathrm{Spec}\ }

\newcommand{\hb}{\hbar}

\newcommand{\dl}{\partial}
\newcommand{\db}{\bar\partial}
\newcommand{\fpsu}{\llbracket u\rrbracket}
\newcommand{\fLsu}{(\hspace{-0.22em}(u)\hspace{-0.22em})}
\newcommand{\fpsh}{\llbracket \hbar\rrbracket}
\newcommand{\fLsh}{(\hspace{-0.22em}(\hbar)\hspace{-0.22em})}

\newcommand{\eps}{\varepsilon}

\newcommand{\Diff}{\D\hspace{-1pt}\mathit{iff}}
\newcommand{\bnull}{\mathbf{0}}

\newcommand{\lookup}{\colorbox{yellow}{Look it up!}}

\title{Global \\ logarithmic \\ deformation theory}
\author{Simon Felten}

\begin{document}

\frontmatter

\maketitle

\begin{abstract}
 A classical problem in algebraic geometry is to construct smooth algebraic varieties with prescribed properties. In the approach via smoothings, one first constructs a degenerate scheme with the prescribed properties, and then shows the existence of a smooth variety degenerating to this scheme. Logarithmic geometry has given important
 new impulses to the second step of this approach, which we explore in this book. Degenerations, in particular in the context of mirror symmetry, often enjoy similar formal properties as smooth morphisms once considered from the
logarithmic perspective. Logarithmic deformation theory has therefore become an effective tool for the construction of smoothings and the transfer of properties between smooth nearby fibers and the singular special fiber.

The strongest existence result for deformations in classical algebraic geometry is the Bogomolov--Tian--Todorov theorem for Calabi--Yau varieties. A logarithmic variant, once established, constructs log smooth deformations. However, the logarithmic Bogomolov--Tian--Todorov theorem has resisted efforts to its proof for a while. Finally, a method to prove it was discovered in 2019 by Chan, Leung, and Ma.

In this book, we explore this new approach to the logarithmic Bogomolov--Tian--Todorov theorem. We prove several variants of the abstract unobstructedness theorem, some of which are new and stronger than previously known results. We investigate its application to the global deformation theory of log smooth and mildly log singular spaces, obtaining unobstructedness results for log Calabi--Yau spaces, some log Fano spaces, and line bundles. Special care is taken to allow sufficiently mild log singularities, including all log singularities that appear in the Gross--Siebert construction of toric log Calabi--Yau mirror pairs.
\end{abstract}

\newpage

\chapter*{About this version}

This is a preprint of 
\begin{quote}
	Simon Felten. Global logarithmic deformation theory. \textit{Lecture Notes in Mathematics 2373}, Springer, 2025.
\end{quote}
which the publisher allows me to keep here. It does not reflect any editorial review or peer review on behalf of the publisher but is, with minor corrections, the version originally submitted. The final and greatly expanded version is available online at: 
\begin{quote}
	http://dx.doi.org/10.1007/978-3-031-98751-9
\end{quote}

\newpage

\chapter*{Preface}

We give a short and gentle introduction to the subject of this book here. A more comprehensive introduction with precise statements can be found in Chapter~\ref{intro-sec}.

\vspace{\baselineskip}

A classical problem in algebraic geometry is to construct smooth algebraic varieties with prescribed properties. The approach via smoothings of degenerate schemes aims to solve this problem in two steps. The first step is to construct a degenerate scheme with the prescribed property, for example the four coordinate hyperplanes $H = \{XYZW = 0\} \subset \PP^3$. The second step is to show that some smooth variety exists which degenerates to the already constructed one, for example any smooth quartic surface $E \subset \PP^3$.

Assume for a moment that we have the four coordinate hyperplanes $H$, but that we do not know a pencil of hyperplanes in $\PP^3$ to smooth it. Since $H$ is a normal crossing space with $\omega_H \cong \cO_H$, we might try to find a semistable degeneration to $H$, i.e., one with local models $z_1 \cdot ... \cdot z_r = t$, where $t$ is a parameter on the base. According to a classical result of Friedman in \cite{Friedman1983}, a necessary condition is that $H$ is $d$-semistable, meaning that $\T^1_H = \E xt^1(\Omega^1_H,\cO_H) \cong \cO_D$, where $D$ is the singular locus of $H$. This is, however, not the case. A direct computation shows that $\T^1_H|_{L_i} \cong \cO_{L_i}(4)$ for each of the six lines $L_i \cong \PP^1$ which form the singular locus $D$ of $H$. Nonetheless, we can modify $H$ so that it becomes $d$-semistable by choosing four points in the interior of each line $L_i$ and blowing them up in one of the two components meeting at the chosen point $p$. There are $24$ continuous moduli for choosing the points, and $2^{24}$ discrete choices which side to blow up. In other words, this procedure is very far from being unique, but we obtain a normal crossing modification $\tilde H$ of $H$ with $\T^1_{\tilde H} \cong \cO_{\tilde D}$. Additionally, we still have $\omega_{\tilde H} \cong \cO_{\tilde H}$, i.e., the resolution is crepant. Now $\tilde H$ is a $d$-semistable K3 surface of type III, and hence it admits a semistable smoothing by another classical result of Friedman.

We can find an explicit example of such a semistable smoothing as follows: We choose the Fermat quartic 
 $$E = \{X^4 + Y^4 + Z^4 + W^4 = 0\} \subset \PP^3$$
 as our smooth fiber so that the pencil defined by $H$ and $E$ has total space 
$$\X = \{T_0(X^4 + Y^4 + Z^4 + W^4) - T_1XYZW = 0\} \subset \PP^1 \times \PP^3.$$
The first projection defines a flat projective family $\varphi: \X \to \PP^1$.
We denote by $S \subseteq \PP^1$ some Zariski open neighborhood of $0 = [0:1]$ such that $H = \varphi^{-1}(0)$ is the only singular fiber of the restricted family $f: X \to S$, i.e., $X := \varphi^{-1}(S)$ and $f := \varphi|_X$. The set 
$$Z := D \cap \{X^4 + Y^4 + Z^4 + W^4 = 0\} \subset \PP^3 $$
consists of 24 points, four on each line. It is the singular locus of the total space $X$, and the family $f: X \to S$ is semistable outside $Z$. In each point $p \in Z$, each of the two components of $H$ meeting at $p$ is not a Cartier divisor but only a Weil divisor. By blowing up one component in the total space $X$, we resolve the singularities of $X$ and obtain a semistable degeneration $g: \tilde X \to S$. The 24 points of $Z$ can be resolved independently because $\pi: \tilde X \to X$ is an isomorphism in a punctured neighborhood of $p \in Z$.

Let us consider the semistable degeneration $g: \tilde X \to S$ from the perspective of logarithmic geometry. We endow the base $S$ with the compactifying (also called divisorial) log structure defined by $0 \in S$, and we endow the total space $\tilde X$ with the compactifying log structure defined by $\tilde H = g^{-1}(0) \subset \tilde X$. This turns the degeneration $g: \tilde X \to S$ into a log morphism $g: (\tilde X,\tilde H) \to (S,0)$ which enjoys, from the perspective of log geometry, the same formal properties as smooth families from the perspective of classical algebraic geometry. We say that $g: (\tilde X,\tilde H) \to (S,0)$ is \emph{log smooth}. 

Let $S_0 \subset S$ be the point $0 \in S$, considered as a closed subscheme. There is a log structure on $S_0$ which is obtained as a pull-back of the log structure on $(S,0)$. We say that $S_0$ is a \emph{log point}. We can pull back the log structure from $(\tilde X,\tilde H)$ to the fiber $\tilde H$ as well, thus constructing a log morphism $g_0: \tilde H \to S_0$. The property of being log smooth is stable under base change. In other words, the log morphism $g_0: \tilde H \to S_0$ is log smooth. Similarly, we can endow the thick points $S_k = \Spec \CC[t]/(t^{k + 1}) \subset S$ with a pull-back log structure, and we obtain log smooth morphisms $g_k: \tilde X_k \to S_k$ by base change. They are infinitesimal log smooth deformations of $g_0: \tilde H \to S_0$. Conversely, when we have enhanced the $d$-semistable K3 surface $\tilde H$ to a log smooth morphism $g_0: \tilde H \to S_0$, and when we have constructed a compatible system of infinitesimal log smooth deformations $g_k: \tilde X_k \to S_k$, it gives rise to a formal semistable smoothing of $\tilde H$, and hence also to a holomorphic semistable smoothing. We obtain a new perspective on the smoothing results of Friedman for semistable Calabi--Yau surfaces, and of Kawamata and Namikawa for semistable Calabi--Yau varieties of all dimensions: while both Friedman and Kawamata--Namikawa analyze the infinitesimal flat deformations, we could have instead analyzed the infinitesimal log smooth deformations. This is technically much more appealing due to the strong analogy between classical smooth deformations and log smooth deformations.

The classical Bogomolov--Tian--Todorov theorem states that infinitesimal flat deformations of a smooth and proper Calabi--Yau variety $Y$ are unobstructed. In other words, whenever we have a flat deformation $Y_B$ over $\Spec B$ and a surjection $B' \to B$ of Artin rings, there is a flat deformation $Y_{B'}$ over $\Spec B'$ extending $Y_B$. The logarithmic analog of this is exactly what we want. If the logarithmic Bogomolov--Tian--Todorov theorem is true, then we can lift the log smooth morphism $g_0: \tilde H \to S_0$ to any $S_k$ and obtain compatible log smooth deformations $g_k: \tilde X_k \to S_k$. This does not only recover the results of Friedman and Kawamata--Namikawa but strengthens them because it removes technical hypotheses which we do no longer need, and it applies to log smooth morphisms which are not semistable. A method that leads to the proof of the logarithmic Bogomolov--Tian--Todorov theorem was provided in 2019 by Chan, Leung, and Ma in \cite{ChanLeungMa2023}, building on earlier work of Kontsevich, Manetti, and others. The proof is considerably more difficult than of its classical counterpart, and it is one of the main subjects of this book.

\vspace{\baselineskip}

Let us go back to the family $g: \tilde X \to S$. We have obtained it as a birational modification of the---in terms of defining equations---simpler family $f: X \to S$ in order to force semistability. Just as $g: \tilde X \to S$, we can also turn $f: X \to S$ into a log morphism $f: (X,H) \to (S,0)$. This family is \emph{not} log smooth at the 24 points of $Z$. Thus, we cannot apply log smooth deformation theory. The log structure does not even admit a chart at $p \in Z$. Locally in the \'etale topology, each of the 24 singularities is equivalent to 
$$Y = \Spec \CC[x,y,z,t]/(xy - tz) \to \bAA^1_t,$$
where we endow $Y$ with the compactifying log structure defined by $Y_0 = \{t = 0\}$. Although this family is not log smooth, we can use it to recover a key feature of log smooth deformations of $g_0: \tilde H \to S_0$: they are locally rigid. Any two log smooth deformations of an affine log scheme are isomorphic. Let us say that a logarithmic deformation of $f_0: H \to S_0$ is \emph{divisorial} if it is log smooth outside $p \in Z$, and if it is \'etale equivalent to $Y \times_{\bAA^1} S_k \to S_k$ around $p \in Z$. It turns out that divisorial deformations are locally rigid as well. In the example, this is essentially because the log singularities are isolated, but a more sophisticated version works for certain non-isolated log singularities as well. 

Obviously, the log deformations $f_k: X_k \to S_k$ obtained by base change of $f: (X,H) \to (S,0)$ are divisorial because $f: X \to S$ is \'etale equivalent to $Y \to \bAA^1_t$. Thus, if the logarithmic Bogomolov--Tian--Todorov theorem applies also for divisorial deformations of $f_0: H \to S_0$, then we can show the existence of a compatible system of divisorial deformations $f_k: X_k \to S_k$ by applying the theorem. Simplifying the requirements of the method discovered by Chan, Leung, and Ma for the sake of exposition,\footnote{Strictly speaking, the list is not sufficient.} we need the following three ingredients in addition to the log Calabi--Yau property:
\begin{enumerate}[label=(\arabic*)]
 \item a log de Rham complex $\W^\bullet_{X_k/S_k}$ for every divisorial deformation $f_k: X_k \to S_k$ of $f_0: H \to S_0$ which commutes with base change along $S_{k + 1} \to S_k$;
 \item the Hodge--de Rham spectral sequence of $\W^\bullet_{X_0/S_0}$ must degenerate at $E_1$;
 \item the map $\HH^m(X_k,\W^\bullet_{X_k/S_k}) \to \HH^m(X_0,\W^\bullet_{X_0/S_0})$ in hypercohomology must be surjective.
\end{enumerate}
For $\W^\bullet_{X_k/S_k}$, we choose the direct image of the log de Rham complex $\Omega^\bullet_{U_k/S_k}$ on $U_k = X_k \setminus Z$. Then (1) is a direct computation in the local model $Y \to \bAA^1_t$, (2) can be proven with the method of Deligne--Illusie by using the local model to construct Frobenius lifts in positive characteristic, and (3) follows from a method originally discovered by Katz, combined with a heavy calculation. In short, divisorial deformations of $f_0: H \to S_0$ are indeed unobstructed.

\vspace{\baselineskip}

The original treatment of Chan--Leung--Ma shows that pairs of a deformation and a volume form are unobstructed, which is sufficient to construct a smoothing. We have strengthened the argument to find that indeed the deformations themselves are unobstructed. In particular, the hull in the sense of Schlessinger is smooth, i.e., it is of the form $\CC\llbracket t,s_1,...,s_r\rrbracket$ for $r = \mathrm{dim}\, H^1(X_0,\Theta^1_{X_0/S_0})$, with an additional parameter $t$ to account for the log structure. The parameters $s_1,...,s_r$ correspond to locally trivial log smooth (or divisorial) deformations over $S_\eps = \Spec \CC[s]/(s^2)$. The log scheme $S_\eps$ is different from $S_1$ in that it has a simpler log structure which does not come from the pull-back along $S_1 \subset (S,0)$ but from the pull-back along the projection $S_\eps \to S_0$. Unlike deformations in the direction of $S_1$, deformations in the direction of $S_\eps$ are not smoothings but preserve the singularities of the underlying space.

The log Hodge numbers are constant in divisorial deformations. Thus, we have $$h^1(X_0,\Theta^1_{H/S_0}) = h^1(X_0,\W^1_{H/S_0}) = h^1(E,\Omega^1_E) = 20$$
for the smooth Fermat quartic $E \subset \PP^3$ from above. For the underlying flat deformations of $H$, we have 
$$\mathrm{dim}\, Ext^1(\Omega^1_H,\cO_H) = 23,$$
which can be computed with Macaulay2. However, locally trivial flat deformations are classified by $H^1(H,\T^0_H)$, for which we have $\mathrm{dim}\, H^1(H,\T^0_H) = 1$. Thus, the log picture, where we have a $20$-dimensional locus of locally trivial deformations inside a $21$-dimensional space of all divisorial deformations, is very different from the classical picture, where we have a one-dimensional (tangent space to the) locus of locally trivial deformations inside a $23$-dimensional (tangent space of the) space of flat deformations. This difference is due to the difference in infinitesimal automorphisms of the respective structures.

For the semistable K3 surface $\tilde H$ of type III, the log picture is the same since it is a log smooth degeneration of the same general fiber. Friedman has elaborated on the flat picture: The base of the holomorphic semi-universal family is of the form $V_1 \cup V_2$, where $V_1 \subset Ext^1(\Omega^1_{\tilde H},\cO_{\tilde H}) =: \TT^1_{\tilde H}$ is a smooth divisor corresponding to locally trivial deformations, and $V_2 \subset \TT^1_{\tilde H}$ is a smooth subvariety of dimension $20$. Their intersection $V_1 \cap V_2$ has dimension $19$ and corresponds to locally trivial deformations which preserve $d$-semistability. By \cite{Stevens2002}, the dimension of $\TT^1_{\tilde H}$ is $23$ if $\tilde H$ is in $(-1)$-form, which is the case if we have blown up each component of $H$ in two points of each of the six lines in $D$. Then $\mathrm{dim}\, V_1 = 22$. In particular, both the $22$ moduli of locally trivial classical deformations and the $19$ moduli of locally trivial classical deformations which preserve $d$-semistability differ from the $20$ moduli of locally trivial log smooth deformations. The missing modulus from $19$ to $20$ corresponds to a $\CC^*$-scaling action of the log morphism, which is not seen on the underlying morphism of schemes or complex spaces. The difference between $20$ and $22$ comes from the fact that locally trivial log smooth deformations have less automorphisms than locally trivial classical deformations.

\vspace{\baselineskip}

It is well known that of the $20$ holomorphic moduli of a smooth projective K3 surface only $19$ correspond to projective deformations. In the algebraic setting, we want to know which formal families, say over $\Spec \CC\llbracket t\rrbracket$, are induced from an algebraic family $X \to \Spec \CC\llbracket t\rrbracket$ of schemes. In this book, we show that deformations of pairs of a proper log Calabi--Yau space with a line bundle are unobstructed as well. So let $\cL_0$ be a polarization on $H$. Any lift of $\cL_0$ to a divisorial deformation $f_k: X_k \to S_k$ will be ample again, and thus we can use the family $\{\cL_k\}_k$ to algebraize the formal family of divisorial deformations. We have an exact sequence 
$$0 \to \cO_H \to \Theta^1_{H/S_0}(\cL_0) \to \Theta^1_{H/S_0} \to 0,$$
and $H^1(H,\Theta^1_{H/S_0}(\cL_0))$ is the tangent space to the deformation functor of pairs. We obtain an exact sequence 
$$0 \to H^1(H,\Theta^1_{H/S_0}(\cL_0)) \to H^1(H,\Theta^1_{H/S_0}) \to H^2(H,\cO_H) \to H^2(\Theta^1_{H/S_0}(\cL_0)) \to 0.$$
The dimensions of the two middle terms are $20$ and $1$. Then the dimension on the left must be $19$, and on the right it must be $0$, essentially because these numbers are constant in the family, and not every infinitesimal deformation of the generic fiber, a smooth K3 surface, is algebraic. Thus, the hull of the deformations of pairs $(H,\cL_0)$ is $\CC\llbracket t,s_1,...,s_{19}\rrbracket$, which forms a smooth hypersurface inside the hull of the divisorial deformations of $H$. The argument also shows that for a projective Calabi--Yau threefold $Y$, deformations of $Y$ and of pairs $(Y,\cL_0)$ coincide; namely, $h^1(Y,\cO_Y) = h^2(Y,\cO_Y) = 0$ by definition.\footnote{This is no longer true for vector bundles, see Example~\ref{R-Thomas-ex-1}. Note also that this result is classical for smooth Calabi--Yau threefolds.}

\vspace{\baselineskip}

Above, we have constructed a log structure on $\tilde H$ and $H$ from a given family. In the case of $\tilde H$, we could also use the existence result of Friedman to obtain the family. However, to use logarithmic deformation theory to show the existence of a family, we have to construct a log structure on a given degenerate scheme $V$ without reference to a family. For the sake of exposition, let us assume that $V$ is a normal crossing space. In particular, when $W \subseteq V$ is a small enough open subset (possibly in the \'etale topology), then we can find a semistable degeneration to $W$. Now we can define a log structure on $W$ and a log morphism $W \to S_0$ to the log point as above. This structure of a log morphism is not unique since it depends on the chosen semistable degeneration, but any such structure has no non-trivial automorphisms. Thus, their isomorphism classes form a sheaf $\cL\cS_V$, whose sections over open subsets $W \subseteq V$ correspond to log morphisms $W \to S_0$ which are locally induced from some semistable degeneration. By construction, any such log morphism is log smooth.

The vector space $\TT^1_V = Ext^1(\Omega^1_V,\cO_V)$ classifies first-order flat deformations of $V$, and sections of the sheaf $\T^1_V = \E xt^1(\Omega^1_V,\cO_V)$ correspond to choices of local first-order flat deformations which are isomorphic on overlaps. On a normal crossing space $V$, the sheaf $\E xt^1(\Omega^1_V,\cO_V)$ is a line bundle on the double locus $D \subset V$ (and $0$ on $V \setminus D$). Given a section $s \in \Gamma(W,\cL\cS_V)$ for an affine open subset $W \subseteq V$, there is a unique log smooth deformation over $S_1$. Taking the underlying flat deformation defines a map $\eta: \cL\cS_V \to \T^1_V$ of sheaves of sets. It turns out that this map is injective and identifies $\cL\cS_V$ with the subsheaf $(\T^1_V)^* \subseteq \T^1_V$ of local generators of the line bundle $\T^1_V$ on $D$.

Recall that we have $\T^1_{\tilde H} \cong \cO_D$ on the semistable K3 surface $\tilde H$, which is the $d$-semistability condition. The global section $1 \in \cO_D$ gives rise to a global section $s \in \cL\cS_{\tilde H}$, which defines a global structure of a log smooth morphism $\tilde H \to S_0$. Any semistable degeneration to $\tilde H$ induces this section up to multiplication with $\CC^* = \Gamma(D,\cO_D^*)$. On $H$, we have $\T^1_H|_{L_i} \cong \cO_{L_i}(4)$ for each line $L_i$ in $D$. Thus, every global section of $\T^1_H$ has four zeroes on each line $L_i$, counted with multiplicities. These zeroes are precisely the 24 log singular points of the log morphism $H \to S_0$ constructed above when $s \in \Gamma(H \setminus Z,\cL\cS_H)$ is the corresponding section of $\cL\cS_H$ of the log structure outside the log singularities.

This construction can be applied in situations where no obvious degeneration to $V$ is available. For an example, let $P = \PP^3$, and let $D \subset P$ be a smooth quartic surface, i.e., an anti-canonical hypersurface. We obtain a normal crossing threefold $M = P \amalg_D P$ with two irreducible components by gluing two copies of $P = \PP^3$ along the closed subset $D \subset P$. We have $\omega_M \cong \cO_M$ and $\T^1_M \cong \N_{D/P} \otimes \N_{D/P} \cong \cO_D(8)$. By Bertini's theorem, we can find a section $s \in \Gamma(D,\cO_D(8))$ whose zero locus is a smooth hypersurface $Z \subset D$. This gives rise to a log smooth morphism $M \setminus Z \to S_0$ whose log singularities in $Z$ are given by $\bAA^1 \times Y_0 \to S_0$, where $Y_0$ is the central fiber of $\Spec \CC[x,y,z,t]/(xy - tz) \to \bAA^1_t$ from above. Divisorial deformations of $M \to S_0$ are unobstructed so that we obtain the existence of a formal smoothing.
 
\vspace{\baselineskip}

The original motivation to study global deformations of mildly log singular log morphisms $V \to S_0$ comes from mirror symmetry. To construct mirror pairs $(X,\check X)$ of smooth Calabi--Yau varieties, Gross and Siebert first construct mirror pairs $(X_0,\check X_0)$ of degenerate Calabi--Yau schemes by gluing toric varieties along closed toric subvarieties. A space $X_0$ arising from their construction admits a sheaf $\cL\cS_{X_0}$ classifying log smooth morphisms $X_0 \to S_0$. Unfortunately, most $X_0$ do not have a global section of $\cL\cS_{X_0}$ (just like $H$ from above), but nonetheless, $\cL\cS_{X_0}$ has a canonical log singular locus $Z \subset X_0$ and a canonical section $s \in \Gamma(X_0 \setminus Z, \cL\cS_{X_0})$. It turns out that any log morphism $X_0 \to S_0$ arising from their construction admits a well-behaved theory of divisorial deformations. Thus, one may attempt to construct a smoothing of $X_0$ through logarithmic deformation theory. However, at the time, no abstract theory to deduce the existence of deformations from conditions on the local deformation theory and the global geometry of $X_0$ was available. Instead, Gross and Siebert had to construct the gluings of the local divisorial deformations explicitly from a subtle combinatorial analysis of so-called scattering diagrams. Global logarithmic deformation theory in its present form shows the existence of a formal semi-universal family of divisorial deformations over a smooth base of the form $\CC\llbracket t,s_1,...,s_r\rrbracket$ from abstract principles. Furthermore, there is a formal semi-universal family of polarized divisorial deformations over a (different) smooth base as well.

\vspace{\baselineskip}

The crucial insight of Chan--Leung--Ma's method to prove the logarithmic Bogomolov--Tian--Todorov theorem is to replace logarithmic deformations with deformations of the induced Gerstenhaber algebra of polyvector fields. The deformed Gerstenhaber algebras admit acyclic resolutions which become all isomorphic upon forgetting the differential. Classifying deformations becomes equivalent to classifying differentials on a given bigraded Gerstenhaber algebra, and this in turn allows to classify deformations in terms of gauge equivalence classes of solutions of an extended Maurer--Cartan equation. In the Calabi--Yau case, the bigraded (or rather curved) Gerstenhaber algebra can be enhanced to a curved Batalin--Vilkovisky algebra, and this structure allows to conclude the existence of deformations.

\vspace{\baselineskip}

In this book, we explore thoroughly the new approach to the logarithmic Bogomolov--Tian--Todorov theorem and its applications in logarithmic geometry. In the first part, we study curved Gerstenhaber and Batalin--Vilkovisky algebras (and calculi), and we prove several variants of the abstract unobstructedness theorem for Maurer--Cartan solutions in these algebras, some of which are new and stronger than previously known results. These results may also be of independent interest outside logarithmic geometry. In the second part, we discuss how to pass from logarithmic geometry to curved Gerstenhaber and Batalin--Vilkovisky algebras. We investigate applications to the global deformation theory of log smooth and mildly log singular spaces, obtaining unobstructedness results for log Calabi--Yau spaces, some log Fano spaces, and line bundles on them. We take special care to allow sufficiently mild log singularities everywhere, including all log singularities that appear in the toric\footnote{We call the mirror construction of \cite{GrossSiebertI}, which we have discussed above, the \emph{toric} Gross--Siebert mirror construction to distinguish it from the \emph{intrinsic} mirror construction of \cite{GrossSiebertIntrinsic}, also by Gross and Siebert.} Gross--Siebert mirror construction. In short, we treat the second step of the construction of smooth varieties with prescribed properties---the passing from a degenerate scheme to a smooth variety---fairly completely.

We review and expand the elementary theory of toroidal crossing spaces in an appendix. They are a tool to construct (mildly singular) log structures on a given degenerate scheme so that we can apply logarithmic deformation theory to obtain (partial) smoothings. The first step of the construction of smooth varieties with prescribed properties, constructing degenerate schemes and the structure of toroidal crossing spaces on them, is left to other studies who may construct them according to the desired properties. An example of this is the construction of $X_0$ in the Gross--Siebert program.

\vspace{\baselineskip}

The original motivation for this book was threefold: First, we wanted to provide a completely algebraic and expanded account of the argument of Chan, Leung, and Ma in \cite{ChanLeungMa2023}. Secondly, we wanted to clarify the definitions of the structures employed by Chan--Leung--Ma, which yielded the quite lengthy definitions of $\Lambda$-linear curved Gerstenhaber and Batalin--Vilkovisky calculi as the correct concepts. Thirdly, we wanted to generalize the theorem to include also deformations of line bundles on log Calabi--Yau spaces, which allows in particular the passage from formal deformations to algebraic ones in the projective case. 

With the present book, we hope to make logarithmic deformation theory and in particular its use in the Gross--Siebert program more complete and more accessible, and to facilitate further research into degenerations of varieties,  compactification of their moduli, and classification of the components of moduli spaces. In particular, we want to popularize Gerstenhaber algebras and calculi as a tool to study algebraic varieties. They provide at least a deformation theory and a cohomology theory for algebraic varieties, and they are closely related to Hodge structures and the mirror symmetry $B$-model approach of Barannikov--Kontsevich as well.

\begin{flushright}
 Simon Felten, October 2023
\end{flushright}

\newpage

\tableofcontents

\listoffigures

\mainmatter


\chapter{Introduction}\label{intro-sec}\note{intro-sec}

We work over a field $\kk$ of characteristic $0$ throughout.

\section{Logarithmic geometry}

We work in the realm of logarithmic geometry throughout this book. Let us first give a quick overview for the occasional reader unacquainted with this branch of algebraic geometry. 

Logarithmic geometry, as pioneered by Fontaine and Illusie, first described by K.~Kato in \cite{kkatoFI}, and popularized by Tsuji in \cite{Tsuji1999} through the proof of the $C_{st}$-conjecture, replaces a scheme $X$ with a \emph{log scheme} $(X,\M_X,\alpha)$ where $\alpha: \M_X \to \cO_X$ is a morphism of sheaves of monoids on $X$, usually in the \'etale topology, such that $\alpha: \alpha^{-1}(\cO_X^*) \to \cO_X^*$ is an isomorphism. By adding a compatible map $\M_Y \to f_*\M_X$ of sheaves of monoids to a morphism $f: X \to Y$, log schemes form a category $\mathbf{LSch}$, which admits fiber products. There is a forgetful morphism $\mathbf{LSch} \to \mathbf{Sch}$ forgetting $\alpha: \M_X \to \cO_X$, and there is an embedding $\mathbf{Sch} \to \mathbf{LSch}$ by taking $\M_X = \cO_X^*$. There are also versions for analytic spaces, algebraic spaces, and stacks.

A \emph{log ring} is a map $\alpha: M \to R$ from a monoid $M$ to a ring $R$ such that $\alpha(m + m') = \alpha(m) \cdot \alpha(m')$ and $\alpha(0) = 1$. We have a spectrum construction which associates with every log ring a log scheme $\Spec(M \to R)$, whose underlying scheme is $\Spec R$.

When $f: X \to Y$ is a morphism of schemes and $(Y,\M_Y)$ is a log scheme, then we have a \emph{pull-back log structure} $f^*_{log}\M_Y$ on $X$ which turns $f: X \to Y$ into a morphism of log schemes. Such a morphism of log schemes is called \emph{strict}.

When $D \subseteq X$ is a closed subscheme, then the \emph{compactifying log structure} $\M_{(X,D)}$ on $X$ is given by 
$$\M_{(X,D)}(W) = \{f \in \cO_X(W) \ | \ f|_{W \setminus D}\ \mathrm{is} \ \mathrm{invertible}\}.$$
The name comes from considering $X$ a (partial) compactification of $U = X \setminus D$. When $D \subseteq X$ is a divisor, this log structure is also called the \emph{divisorial} log structure.

Every monoid $P$ admits a universal morphism $P \to P^{gp}$ to an abelian group. The monoid $P$ is called \emph{toric} if $P^{gp} \cong \ZZ^r$ for a finite number $r$, the map $P \to P^{gp}$ is injective, and, for any $p \in P^{gp}$, we have that $n \cdot p \in P$ for $n \geq 0$ implies $p \in P$. The toric monoids are precisely the monoids of the form $\sigma \cap M$ for a lattice $M \cong \ZZ^r$ and a rational polyhedral cone $\sigma \subseteq M_\RR$. A monoid $P$ is called \emph{sharp} if $0 \in P$ is the only invertible element. When $P$ is a sharp toric monoid, we write $A_P := \Spec (P \to \kk[P])$. This coincides with the affine toric variety $\Spec \kk[P]$ endowed with the compactifying log structure defined by the full toric boundary divisor $D_P$. A log scheme $X$ is \emph{fine and saturated} if it admits, locally in the \'etale topology, strict morphisms to $A_P$ for various sharp toric monoids $P$.\footnote{In general, the base ring $\kk$ has to be replaced with $\ZZ$.}

A morphism $f: X \to Y$ of log schemes is \emph{log smooth} if it is locally of finite type, both log schemes $X$ and $Y$ are fine and saturated,\footnote{Here, usually a weaker condition is given, but we can assume this throughout the book.} and the morphism satisfies a logarithmic variant of Grothendieck's geometric-functorial characterization of smoothness.

\section{Logarithmic deformation theory}

A \emph{log point} is a log scheme of the form $S_0 = \Spec(Q \to \kk)$ for a sharp toric monoid $Q$. Here, the log ring $Q \to \kk$ is given by $0 \mapsto 1$ and $q \mapsto 0$ for $q \not= 0$. The log point $S_0$ is a point $\Spec \kk$ endowed with a log structure determined by $Q$; more precisely, we have $\M_{S_0}(S_0) \cong Q \oplus \kk^*$. We obtain thickenings of $S_0$ as follows: Let $\kk\llbracket Q\rrbracket$ be the completion of the monoid ring $\kk[Q]$ in the maximal ideal $\kk[Q^+]$, where $Q^+ = Q \setminus \{0\}$. Then we denote the category of Artinian local $\kk\llbracket Q\rrbracket$-algebras with residue field $\kk$ by $\mathbf{Art}_Q$. Any $A \in \mathbf{Art}_Q$ gives rise to a log scheme $S_A := \Spec(Q \to A)$, where $Q \to A$ is given by the $\kk\llbracket Q\rrbracket$-algebra structure on $A$. Then $S_0 \to S_A$ is a strict closed immersion cut out by a nilpotent ideal in $S_A$.

In infinitesimal logarithmic deformation theory, we start with a log scheme $f_0: X_0 \to S_0$, and we want to classify all infinitesimal deformations $f_A: X_A \to S_A$ of a specific type. The archetypal situation, studied by F.~Kato in \cite{Kato1996}, is when $f_0: X_0 \to S_0$ is log smooth (and has the technical property of being saturated), and we wish to classify all infinitesimal deformations $f_A: X_A \to S_A$ which are log smooth. As it turns out, this situation is technically very similar to flat deformations of smooth schemes in that, on an affine open subset $W_0 \subseteq X_0$, there is, over any $S_A$, a log smooth deformation $W_A \to S_A$, which is unique up to non-unique isomorphism. The infinitesimal automorphisms are governed by the sheaf $\Theta^1_{X_0/S_0}$ of relative \emph{log derivations}, and then, for a first order thickening $B' \to B$ with kernel $I \subseteq B'$, liftings of a log smooth deformation $f: X_B \to S_B$ to $S_{B'}$ have automorphisms in $H^0(X_0,\Theta^1_{X_0/S_0}) \otimes I$, are classified by $H^1(X_0,\Theta^1_{X_0/S_0}) \otimes I$, and have obstructions in $H^2(X_0,\Theta^1_{X_0/S_0}) \otimes I$. Consistent with the usage of the expression in \cite{Gross2011}, we say that log smooth deformations are \emph{locally rigid}. Isomorphism classes of infinitesimal log smooth deformations form the \emph{log smooth deformation functor}
$$\mathrm{LD}_{X_0/S_0}: \enspace \mathbf{Art}_Q \to \mathbf{Set}.$$

While, in the classical setting, smooth deformations are always locally trivial, this is no longer true in the logarithmic setting. For $Q = \NN$, the most important case, consider, on the one hand side, the $\kk\llbracket t\rrbracket$-algebras $C_k = \kk[s]/(s^{k + 1})$ with $t \mapsto 0$, and on the other hand side the $\kk\llbracket t\rrbracket$-algebras $A_k = \kk[t]/(t^{k + 1})$ with $t \mapsto t$. Then we have $C_k \cong A_k$ as $\kk$-algebras but not as $\kk\llbracket t\rrbracket$-algebras. The simplest non-trivial example of a log smooth morphism is the central fiber $f_0: X_0 \to S_0$ of $\bAA^2 \to \bAA^1, \: t \mapsto xy$. Then the log smooth deformations over $C_k$ are trivial, but the log smooth deformations over $A_k$ are the thickenings of $\{xy = 0\}$ inside $\bAA^2 \to \bAA^1$. Since the general fiber of $\bAA^2 \to \bAA^1$ is smooth, the system of infinitesimal log smooth deformations $(f_k: X_k \to S_k)_k$ over $A_k$ forms a \emph{formal smoothing} of $X_0$.\footnote{See Definition~\ref{formal-smoothing-defn} for a precise definition.} In this specific example, this can be algebraized to an \emph{algebraic smoothing} over $\Spec \kk\llbracket t\rrbracket$ and even over $\bAA^1$. From this basic observation, we see that log smooth deformation theory is useful to study formal and sometimes algebraic smoothings of degenerate varieties $X_0$.

The most classical case of this is when $X_0$ is a $d$-semistable simple normal crossing space as defined by Friedman in \cite{Friedman1983}. In this situation, $X_0$ can be endowed with the structure of a log smooth log morphism $f_0: X_0 \to S_0$ for $Q = \NN$. Then log smooth deformations over $A_k$ are precisely the infinitesimal deformations which are coming from the local models $x_1 \cdot ... \cdot x_r = t$ of a semistable degeneration. If a log smooth lift exists to any order over $A_k$, then we obtain a semistable formal smoothing of $X_0$; it can often be algebraized to a semistable degeneration over $\Spec \kk\llbracket t\rrbracket$, or to a holomorphic family over a small disk in the analytic setup.

\vspace{\baselineskip}

From the perspective of the degenerate variety $X_0$ that we want to deform, the case of $d$-semistable normal crossing spaces is rather the exception than the rule. Many degenerate varieties $X_0$ do not admit a global structure of a log smooth morphism $f_0: X_0 \to S_0$, even if such a structure exists everywhere locally. This situation is our main object of study. To handle this situation, we have introduced in \cite{FFR2021} the notion of a \emph{generically log smooth family}. Given a fine and saturated log scheme $S$ as the base, such a structure essentially consists of a flat morphism $f: X \to S$ of schemes with reduced Cohen--Macaulay fibers\footnote{A saturated log smooth morphism has Cohen--Macaulay fibers, and we wish to retain that assumption even for our log singularities.} which is endowed with an open subset $j: U \to X$ on which a log structure is defined, log smooth and saturated over the base $S$. With this definition, we capture the situation that we cannot define a log smooth log structure globally on $X$. In a sense, $f: X \to S$ has \emph{log singularities} in the complement $Z := X \setminus U$, but we circumvent the question of the nature of the log structure in the singular points by just defining it around them.
Unlike the name \emph{generically} might suggest, we require that the \emph{log singular locus} $Z$ has relative codimension $\geq 2$ over $S$. This is in parallel with the fact that normal varieties are smooth in codimension $1$, and it is usually sufficient. A consequence of our requirement on the codimension and our Cohen--Macaulay assumption is that $j_*\cO_U = \cO_X$. In principle, this allows us to define a global log structure $\M_X := j_*\M_U$, but the precise nature of $\M_X$ remains unclear as of now, and it plays no role in our current analysis.

The following is a standard example of a generically log smooth family, given, for example, in the introduction of \cite{GrossSiebertI}. The central fiber $X_0$, a normal crossing space, acquires log singularities from the geometry of a smoothing. In fact, there is no way that $X_0$ can be the central fiber of a semistable degeneration since $X_0$ is not $d$-semistable.

\begin{figure}
 \begin{mdframed}
 \begin{center}
   \includegraphics[scale=0.8]{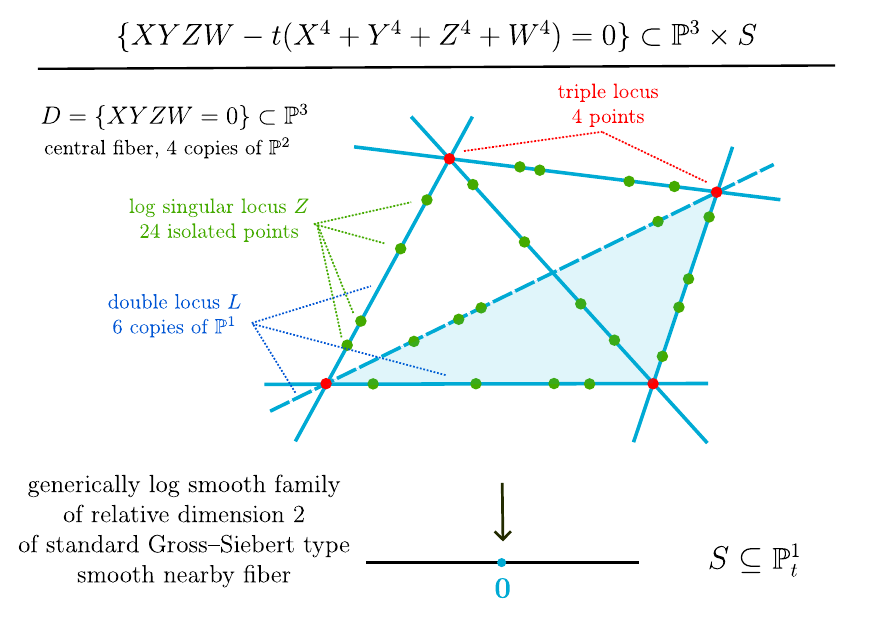}
 \end{center}
 \caption{The central fiber $D$ consists of four copies of $\PP^2$ intersecting in a union $L$ of six copies of $\PP^1$. On each $\PP^1$ there are $4$ singularities of $X$.}\label{intersectingPlanes}
 \end{mdframed}
 
\end{figure}

\begin{ex}\label{smooth-quartic-deg}\note{smooth-quartic-deg}
 We consider the degeneration of a smooth quartic surface 
 $$E = \{X^4 + Y^4 + Z^4 + W^4 = 0\} \subset \PP^3$$ into an arrangement of four planes $D := \{XYZW = 0\} \subset \PP^3$. The pencil defined by $D,E$ has total space 
$$\X = \{T_0(X^4 + Y^4 + Z^4 + W^4) - T_1XYZW = 0\} \subset \PP^1 \times \PP^3 $$
where $X, Y, Z, W$ are homogeneous coordinates of $\PP^3$, and $T_0, T_1$ are homogeneous coordinates of $\PP^1$. The space $\X$ is the blowup of $\PP^3$ in $D \cap E$ via the second projection, and the first projection defines a flat projective family $\varphi: \X \to \PP^1$.
We denote by $S \subseteq \PP^1$ some Zariski open neighborhood of $0 = [0:1]$ such that $D = \varphi^{-1}(0)$ is the only singular fiber of the restricted family, which we denote by $f: X \to S$, i.e., $X := \varphi^{-1}(S) \subseteq \X$. The singular fiber $D$ consists of four copies of $\PP^2$ intersecting in a union $L$ of six lines $\PP^1$ like the faces of a tetrahedron, exhibiting $D$ as four toric varieties $\PP^2$ glued along toric divisors. 
It is depicted in Figure~\ref{intersectingPlanes}. The set 
$$Z := L \cap \{X^4 + Y^4 + Z^4 + W^4 = 0\} \subset \PP^3 $$
consists of 24 points, four on each line. It is the singular locus of the total space $X$, i.e., $U = X \setminus Z$ is regular. The divisor 
$D|_U \subset U$ is a simple normal crossing divisor whereas $D \subset X$ is not a normal crossing divisor since $X$ is not regular in the points $z \in Z$. The log morphism $(X,D) \to (S,0)$ is log smooth outside $Z$ because there it is a semistable degeneration, but not log smooth in $Z$. We obtain a toric degeneration by base change to a (formal) neighborhood $\Spec \kk\llbracket t\rrbracket$ of $0 \in S$.
\end{ex}

Given a generically log smooth family $f_0: X_0 \to S_0$, the notion of an infinitesimal generically log smooth deformation makes sense and yields a deformation functor 
$$\mathrm{LD}_{X_0/S_0}^{gen}: \enspace \mathbf{Art}_Q \to \mathbf{Set}.$$
These deformations are, however, rather wild since this notion allows many flat deformations in the log singular locus $Z$, and there is no control locally on $X$ how a log smooth deformation of $U$ may ``twist'' around the log singular locus $Z$. The question of classifying these generically log smooth deformations is a \emph{local} question which remains, except for some special cases, widely open as of now. The tools that we discuss in this book are instead designed to study the passage from local deformations to global deformations. In order to apply them, we fix the local isomorphism type of those generically log smooth families which we wish to admit in our study. Our device to do so is a \emph{system of deformations $\D$}, which consists of an open cover $\V = \{V_\alpha\}_\alpha$ of $X_0$ and a choice of a generically log smooth deformation $V_{\alpha;A} \to S_A$ for each $V_\alpha$ and $A \in \mathbf{Art}_Q$, which serves as a \emph{local model} of those generically log smooth deformations which we wish to admit; the local models $V_{\alpha;A} \to S_A$ have to be compatible with base change and to be isomorphic on overlaps. In doing so, we mimic the situation of log smooth deformations; now the deformations \emph{of type $\D$} are \emph{locally rigid} in that any two generically log smooth deformations of type $\D$ are locally isomorphic---by the very definition of that notion.
We have a deformation functor 
$$\mathrm{LD}_{X_0/S_0}^\D: \enspace \mathbf{Art}_Q \to \mathbf{Set}$$
which classifies generically log smooth deformations of type $\D$; it behaves much like the log smooth deformation functor.

\begin{rem}
 The log singularities as we study them here are different from morphisms $f: X \to S$ between \emph{fine and saturated} log schemes which are just not log smooth. Such log singularities may be studied from the perspective of \emph{log flat deformations}, which are not the subject of this book. For fine and saturated log singularities, some form of resolution of singularities is known, see for example \cite{Quek2022}. These results do not apply to our situation. Our log singularities are, at least sometimes, not even (quasi-)coherent.
\end{rem}

When we want to use this deformation functor to construct global formal smoothings, then we need to lift infinitesimal deformations from order to order along $A_k = \kk\llbracket Q\rrbracket/\m_Q^{k + 1}$. In other words, we need some form of \emph{unobstructedness} for $\mathrm{LD}_{X_0/S_0}^\D$. In the classical case, there is a well-known example of that: the Bogomolov--Tian--Todorov theorem, stating that the smooth deformation functor
$$\mathrm{Def}_X: \enspace \mathbf{Art}_\kk \to \mathbf{Set}$$
is unobstructed if $X$ is a smooth and proper \emph{Calabi--Yau manifold}, meaning that $\omega_X \cong \cO_X$. If $f_0: X_0 \to S_0$ is a generically log smooth family of relative dimension $d$, then the \emph{log canonical sheaf} is 
$$\omega_{X_0/S_0} := j_*\omega_{U_0/S_0} = j_*\Omega^d_{U_0/S_0}$$
where $\Omega^d_{U_0/S_0}$ is the $d$-th piece of the \emph{log de Rham complex} $\Omega^\bullet_{U_0/S_0}$. This sheaf may not always be a line bundle, but if it is a line bundle and moreover, we have $\omega_{X_0/S_0} \cong \cO_{X_0}$, we say that $f_0: X_0 \to S_0$ is \emph{log Calabi--Yau}. In many cases, $X_0$ is a Gorenstein scheme, and the canonical bundle $\omega_{X_0}$ is isomorphic to the log canonical bundle $\omega_{X_0/S_0}$.

The classical Bogomolov--Tian--Todorov theorem suggests that $\mathrm{LD}_{X_0/S_0}^\D$ may be unobstructed if $f_0: X_0 \to S_0$ is log Calabi--Yau, suggesting some form of a \emph{logarithmic Bogomolov--Tian--Todorov theorem}. Historically, some results in this direction have been obtained. In \cite{Friedman1983}, Friedman shows that all semistable K3 surfaces which occur in Kulikov's classification of semistable degenerations of K3 surfaces indeed occur as the central fiber of some semistable degeneration. This article predates the advent of log geometry and uses only more classical methods. In \cite{KawamataNamikawa1994}, Kawamata and Namikawa give a generalization to higher-dimensional $d$-semistable normal crossing spaces under some cohomological conditions, this time using log geometry although they still rely on the analysis of a mixed Hodge structure as Friedman does. A proof of the actual logarithmic Bogomolov--Tian--Todorov theorem has not been possible before a new method was presented by Chan, Leung, and Ma in \cite{ChanLeungMa2023} in 2019 (year of the preprint); this is one of the main topics of this book.

\section{Examples of generically log smooth families}\label{examples-intro-sec}\note{examples-intro-sec}

Before we go on, we illustrate possible central fibers $f_0: X_0 \to S_0$ with a couple of examples. We have seen log singularities of the form in Example~\ref{xy-tz-example} below already in Example~\ref{smooth-quartic-deg} above. The reader may want to recall that by K.~Kato's toroidal characterization of log smoothness, every log smooth map is \'etale locally the base change of some toric morphism $A_\theta: A_P \times \bAA^r \to A_Q$ of toric varieties, where $A_P$ and $A_Q$ carry the divisorial log structure from the (full) toric boundary, and $\bAA^r$ carries the trivial log structure. We have defined in \cite[Defn.~4.1]{FFR2021} the notion of a \emph{log toroidal} family by taking the same local models on the level of schemes but allowing to take only a part of the toric boundary to define the log structure.\footnote{There are some subtleties concerning the log singular locus $Z$ and the map from $S$ to the base of the local model in the various definitions of log toroidal families. The reader may ignore them at this point. More information can be found in \cite{FFR2021,FeltenThesis} as well as in Section~\ref{elem-GS-type-sec}.} Log toroidal families form generically log smooth families which are not necessarily smooth, as the reader can see below. In Chapter~\ref{elem-GS-type-sec}, we define log toroidal families of \emph{Gross--Siebert type} by admitting only local models which arise from a specific construction which was first considered by Gross and Siebert in \cite{GrossSiebertII}. If the polytope $\Delta_+$ in the construction is an elementary respective a standard simplex, then the family is called log toroidal of \emph{elementary/standard} Gross--Siebert type. Not every generically log smooth family is log toroidal. Additional illustrated examples of generically log smooth families can be found in Chapter~\ref{gen-log-smooth-sec} and Chapter~\ref{elem-GS-type-sec}.

\begin{ex}\label{xy-t-example}\note{xy-t-example}
 Let $f: \bAA^2 \to \bAA^1, \: t\mapsto xy$, endowed with the divisorial log structures given by $t = 0$ on source and target. This map is log smooth of relative dimension $1$. The nearby fiber is smooth, as is the total space. The central fiber consists of two copies of $\bAA^1$, intersecting in a point. The ghost stalk at this point is $\NN^2$ while the ghost stalk at every other point is $\NN$.
\end{ex}

\begin{figure}
 \begin{mdframed}
  \begin{center}
   \includegraphics[scale=0.8]{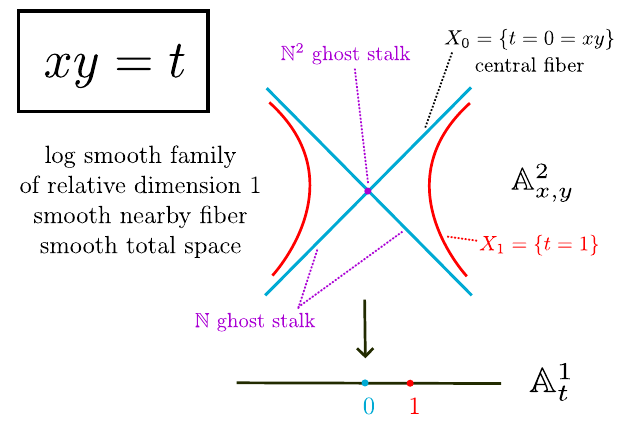}
  \end{center}
  \caption{Example~\ref{xy-t-example}}
 \end{mdframed}
\end{figure}

\begin{ex}\label{xy-tk-example}\note{xy-tk-example}
 Let $X = \Spec\kk[x,y,z]/(xy - z^k)$ and $f: X \to \bAA^1, \: t \mapsto z$, endowed with the divisorial log structure given by $t = 0$ on source and target. The total space is $\Spec \CC[P_k]$ for the monoid $P_k$ depicted in Figure~\ref{xy-tk-figure}, and the map $f: X \to \bAA^1$ is a toric morphism such that $f^{-1}(0)$ is the toric boundary of $X$. Hence $f$ is log smooth of relative dimension $1$. While the total space in the previous Example~\ref{xy-t-example} is smooth, now the total space has an $A_{k - 1}$-singularity. The stalk of the ghost sheaf $\overline\M_X$ at the intersection of the two components of the central fiber $X_0$ is a copy of the monoid $P_k$.
\end{ex}

\begin{figure}\label{xy-tk-figure}\note{xy-tk-figure}
 \begin{mdframed}
  \begin{center}
   \includegraphics[scale=0.8]{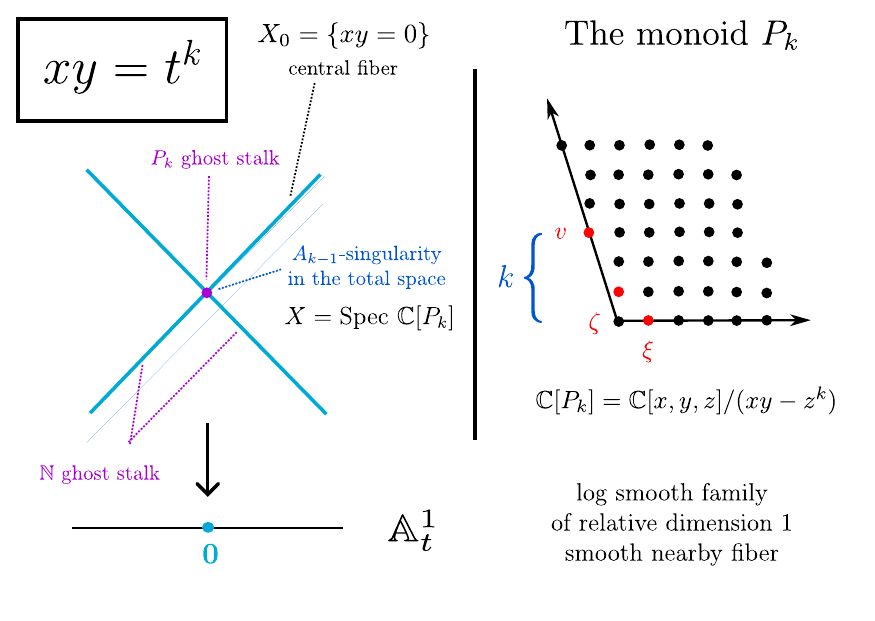}
  \end{center}
  \caption{Example~\ref{xy-tk-example}}
 \end{mdframed}
\end{figure}

\begin{ex}\label{xyz-t-example}\note{xyz-t-example}
 Let $X = \Spec \kk[x,y,z]$ and $f: X \to \bAA^1, \: t \mapsto xyz,$ endowed with the divisorial log structures given by $t = 0$ on source and target. This map is log smooth of relative dimension $2$. It is the semistable degeneration with three components intersecting in a single point. The stalk of the ghost sheaf $\overline\M_X$ at the triple intersection point is $\NN^3$. On the double locus, it is $\NN^2$, and in the interior of the irreducible components, it is $\NN$. The general fiber of the family is smooth.
\end{ex}

\begin{figure}
 \begin{mdframed}
  \begin{center}
   \includegraphics[scale=0.8]{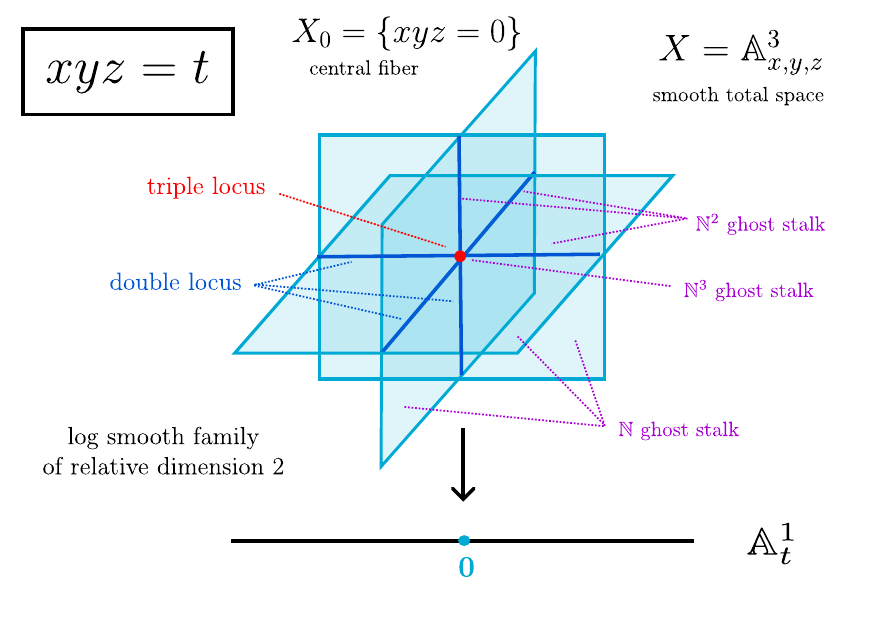}
  \end{center}
  \caption{Example~\ref{xyz-t-example}}
 \end{mdframed}
\end{figure}

\begin{ex}\label{xy-wk-zw-t-example}\note{xy-wk-zw-t-example}
 This is a generalization of the previous example. For $k \geq 1$, let 
 $$X = \Spec \kk[x,y,z,w]/(xy - w^k) = \Spec \kk[P_k \oplus \NN]$$
 and $f: X \to \bAA^1, \: t \mapsto zw,$ endowed with the divisorial log structures given by $t = 0$ on source and target. The total space has a $cA_{k - 1}$-singularity in the sense that we have an $A_{k - 1}$-singularity multiplied with $\bAA^1$. The central fiber $X_0$ has three irreducible components; two of them, $V_{yz} = \{x = w = 0\} = \Spec \kk[y,z]$ and $V_{xz} = \{y = w = 0\} = \Spec \kk[x,z]$, are smooth; the remaining component $V_{xy} = \{z = 0\} = \Spec \kk[x,y,w]/(xy - w^k)$ has an $A_{k - 1}$-singularity which lies in the triple locus of $X_0$. The singular locus of the total space is then $V_{yz} \cap V_{xz} = \Spec \kk[w]$. The stalk of the ghost sheaf $\overline\M_X$ at the single triple intersection point is $P_k \oplus \NN$. Along $V_{yz} \cap V_{xz}$, this is Example~\ref{xy-tk-example}; in particular, the ghost stalk is $P_k$. Along the other two lines of the double locus, this is $xy = t$, hence the stalk of the ghost sheaf $\overline\M_X$ is $\NN^2$.
\end{ex}

\begin{figure}
 \begin{mdframed}
  \begin{center}
   \includegraphics[scale=0.8]{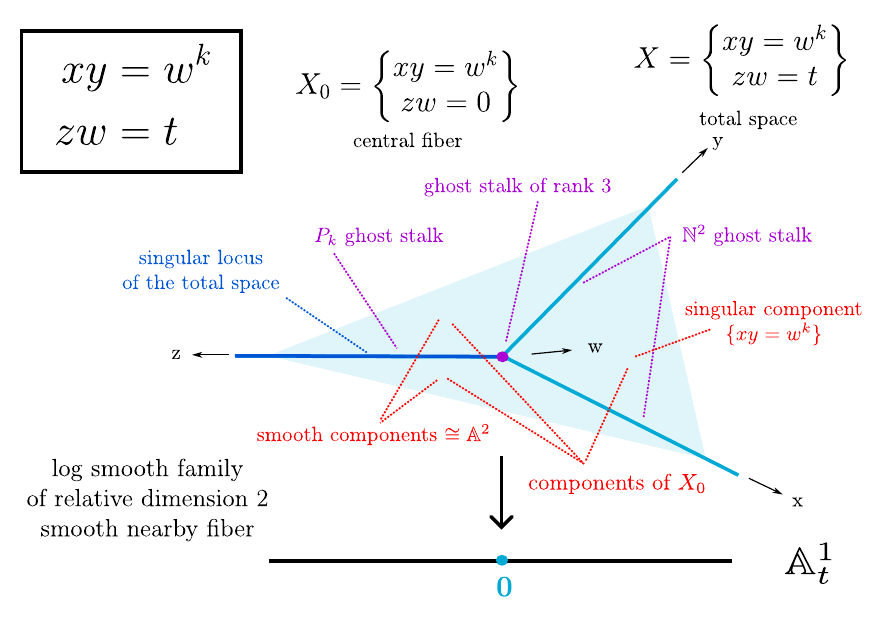}
  \end{center}
  \caption{Example~\ref{xy-wk-zw-t-example}}
 \end{mdframed}
\end{figure}

\begin{ex}\label{xy-tz-example}\note{xy-tz-example}
 Let $X = \Spec \kk[x,y,z,t]/(xy - zt)$, and let $f: X \to \bAA^1, \: t \mapsto t$, be endowed with the divisorial log structure defined by $t = 0$ on source and target. Then the central fiber has two smooth irreducible components. The log structure has an isolated singularity in $x = y = z = t = 0$, the $A_1$-singularity of the total space. This is the simplest example of a generically log smooth family which is not log smooth. When we consider $V = X_0$ as a toroidal crossing space (see below), then we have $\cL\cS_V \hookrightarrow \T^1_V = \kk[z]$, and the section defining the log structure under consideration is $s = z \in \cL\cS_V$. Outside $Z$, the ghost stalk is $\NN^2$ on the double locus $\{x = y = 0\}$ of the central fiber, but in $0$, the ghost stalk is only $\NN$. The sheaf of log differential forms $\Omega^1_{X/S}$, defined via the usual universal property, is not coherent. In fact, it is neither quasi-coherent nor of finite type. See \cite[Ex.~1.11]{GrossSiebertII} and the author's master thesis \cite[Thm.~4.7]{FeltenMasterThesis} for details. This shows that the log structure is \emph{not coherent} in $Z$, i.e., the map is not only not log smooth, but there is no chart for the log structure at all. This motivates both that we work with a coherent replacement $\W^1_{X/S}$ of $\Omega^1_{X/S}$ and that we ignore the log structure in $Z$ and just work with the log structure on the log smooth locus $U = X \setminus Z$.
 
 This family is log toroidal of standard Gross--Siebert type as defined in Definition~\ref{log-tor-GS-type}. To see this, we take $\tau = [0,1] \subseteq \RR^1$ in the construction in Section~\ref{local-model-GS-type-sec}; we set $q = 1$ and $\Delta_1 = \tau$. Then 
 $$\check\psi_0(n) = -\mathrm{inf}\{0,n\} = \check\psi_1(n),$$
 and we have 
 $$P = \{(n,a_0,a_1) \ | \ a_i \geq \check\psi_i(n), \ i = 1,2\}.$$
 This monoid is generated by 
 $$\bar x = (1,0,0),\quad \bar y = (-1,1,1),\quad \bar t = (0,1,0), \quad \bar z = (0,0,1),$$
 which induces the isomorphism $\kk[x,y,z,t]/(xy - zt) \cong \kk[P]$. Note also that this example is a special case of Example~\ref{codim-1-monoid}. That $f: X \to \bAA^1$ is of \emph{standard} Gross--Siebert type follows from Lemma~\ref{standard-simplex-gen-fib} since the general fiber of $f: X \to \bAA^1$ is smooth.
\end{ex}

\begin{figure}
 \begin{mdframed}
  \begin{center}
   \includegraphics[scale=0.8]{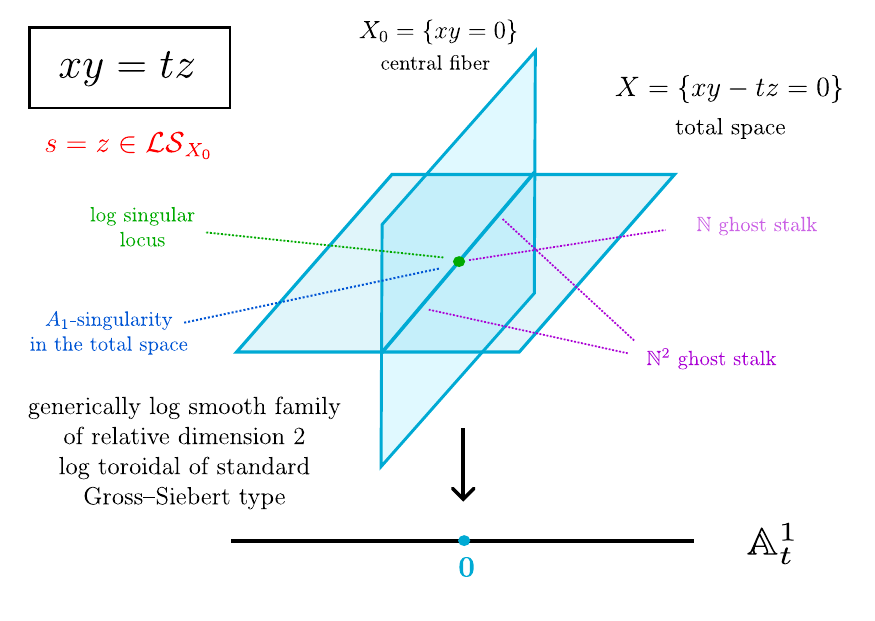}
  \end{center}
  \caption{Example~\ref{xy-tz-example}}
 \end{mdframed}
\end{figure}

\begin{ex}\label{xy-tz-inv-example}\note{xy-tz-inv-example}
 This example is just a central fiber $f_0: X_0 \to S_0$ without a deformation. As underlying scheme, we take $X_0 = V := \Spec \kk[x,y,z]/(xy)$ as in the previous example. This is a normal crossing space, and hence it has a canonical structure of a toroidal crossing space as explained in Example~\ref{normal-crossing-toroidal-crossing}. We take $z^{-1} \in \cL\cS_V$ to define a log structure outside $x = y = z = 0$. In plain terms, we take the log structure from $xy = tz^{-1}$ on $\{z \not= 0\}$, and we take the induced log structure from the base $S_0$ on both $\{x \not= 0\}$ and $\{y \not= 0\}$. This gives rise to a generically log smooth family with log singular locus a single point $Z = \{0\}$. There is no generically log smooth deformation of this family over $S_1 = \Spec(\NN \to \kk[t]/(t^2))$. Namely, a flat deformation over $S_1$ defines a class in $\mathrm{Ext}^1(\Omega^1_{X_0},\cO_{X_0}) = \kk[z]$ whose restriction to $\{z\not= 0\}$ is the underlying flat deformation of the log smooth deformation; this is $z^{-1} \in \kk[z]_z$, which does not extend to the whole of $X_0$.
\end{ex}

\begin{figure}
 \begin{mdframed}
  \begin{center}
   \includegraphics[scale=0.8]{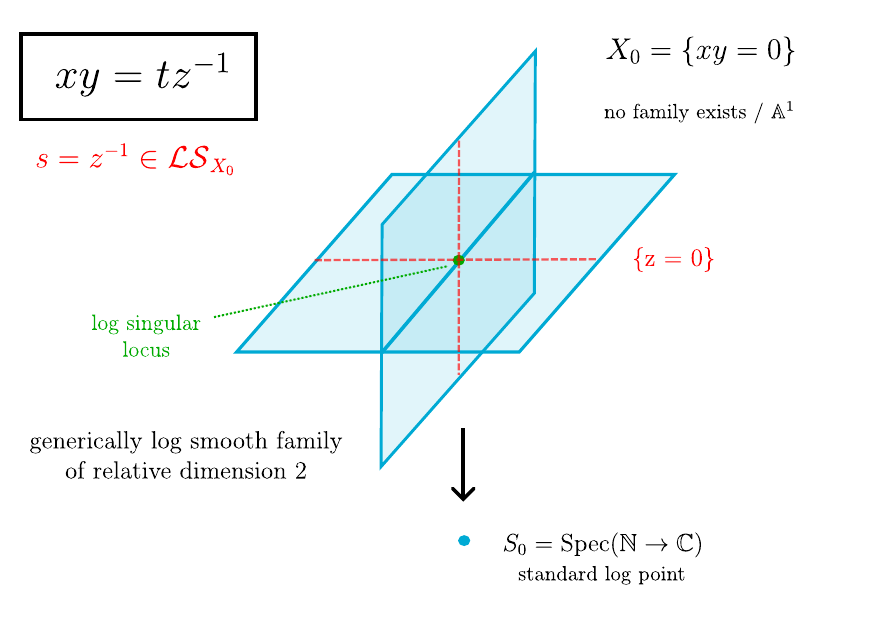}
  \end{center}
  \caption{Example~\ref{xy-tz-inv-example}}
 \end{mdframed}
\end{figure}

\begin{ex}\label{xy-tzw-example}\note{xy-tzw-example}
 Let 
 $$X = \Spec \kk[x,y,z,w,t]/(xy - tzw),$$
 and let $f: X \to \bAA^1, \: t \mapsto t$, be endowed with the divisorial log structures given by $t = 0$ on source and target. The central fiber $X_0 = \Spec \kk[x,y,z,w]/(xy)$ has two smooth irreducible components. The log singular locus $Z$ is given by $zw = 0$ inside the double locus $\bAA^2_{z,w}$ of $X_0$, i.e., it consists of two lines intersecting transversely in a point. Outside this point, the log singularity is of the form $xy = tz$ as in Example~\ref{xy-tz-example}. The family $f: X \to \bAA^1$ is identical with the family $L(1;1,1) \to \bAA^1$ in Example~\ref{codim-1-monoid}. Thus, it is log toroidal of Gross--Siebert type but not of elementary Gross--Siebert type since $\Delta_+$ is a square and hence not an elementary simplex. The general fiber has a $3$-dimensional $A_1$-singularity; a log toroidal family of elementary Gross--Siebert type of relative dimension $3$ would have a smooth general fiber since, in this case, singularities in the general fiber can occur only in codimension $4$.
\end{ex}

\begin{figure}
 \begin{mdframed}
  \begin{center}
   \includegraphics[scale=0.8]{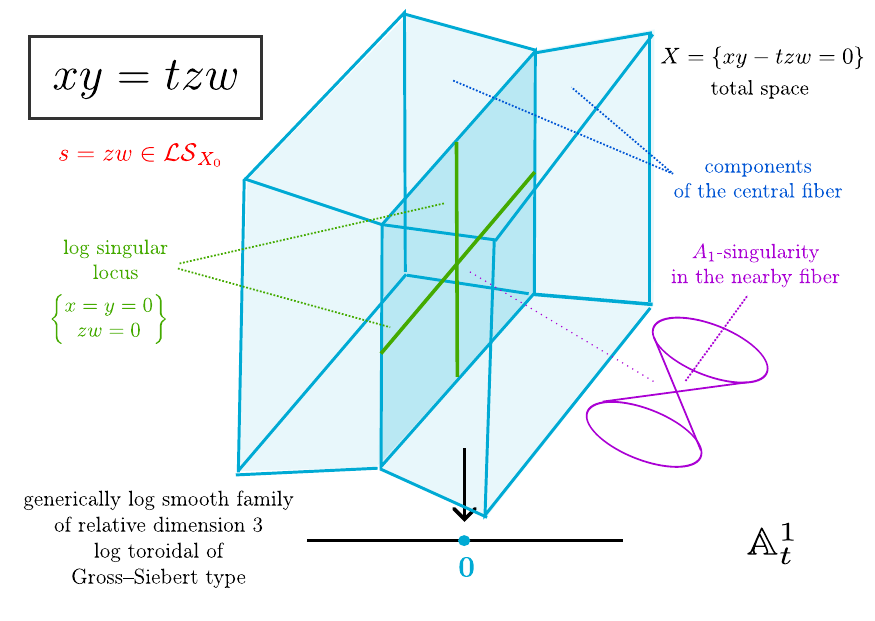}
  \end{center}
  \caption{Example~\ref{xy-tzw-example}}
 \end{mdframed}
\end{figure}
\begin{figure}
 \begin{mdframed}
  \begin{center}
   \includegraphics[scale=0.8]{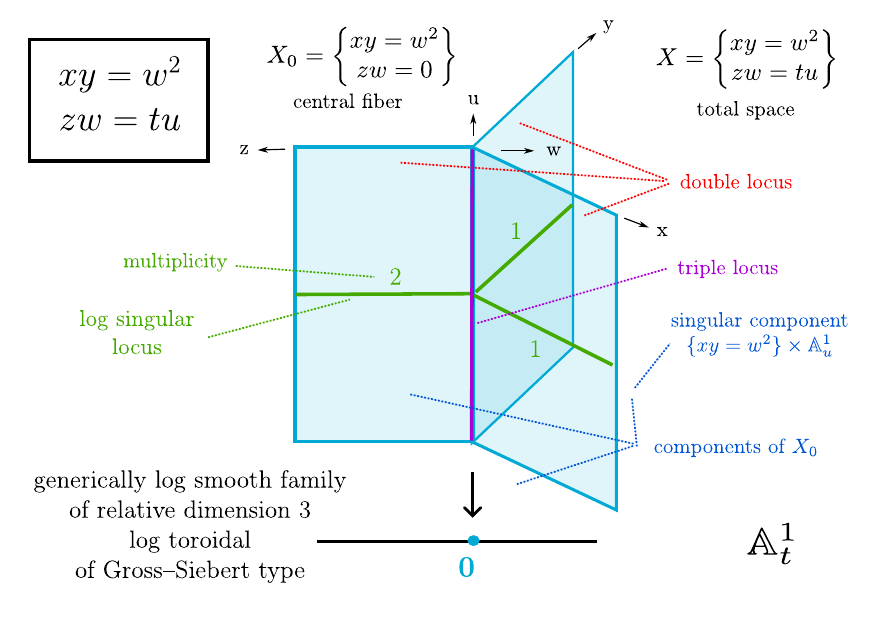}
  \end{center}
  \caption{Example~\ref{xy-w2-zw-tu-example}}
 \end{mdframed}
\end{figure}

\begin{ex}\label{xy-w2-zw-tu-example}\note{xy-w2-zw-tu-example}
 Let 
 $$X = \Spec \kk[x,y,z,w,u,t]/(xy - w^2, zw - tu),$$
 and let $f: X \to \bAA^1, \: t \mapsto t$, be endowed with the divisorial log structure defined by $t = 0$ on source and target. The general fiber of this map is a singular Gorenstein toric variety $\Spec \kk[x,y,z,w]/(xy - w^2)$, i.e., a $cA_1$-singularity. The central fiber is 
 $$X_0 = \Spec \kk[x,y,z,w,u]/(xy - w^2, zw),$$
 i.e., the product of the central fiber in Example~\ref{xy-wk-zw-t-example} with $\bAA^1_u$. In particular, the central fiber consists of three irreducible components of dimension $3$, any two of them intersecting in a stratum of dimension $2$, and all three intersecting in a stratum of dimension $1$. The log singular locus $Z$ is given by $u = 0$ in the double locus of $X_0$. More precisely, there is a single branch in each of the two components $V_x = \bAA^1_{u,x}$ and $V_y = \bAA^1_{u,y}$ of the double locus, and a double branch in $V_z = \bAA^1_{z,u}$. The multiplicity of this branch comes from the fact that the local model of the singularity is $xy = t^2z^2$; the $t^2$ comes from the monoid $P_2$ of the ghost sheaf of the morphism in Example~\ref{xy-wk-zw-t-example}, and the $z^2$ corresponds to the fact that the branch is a double line. 
 
 This family is log toroidal of Gross--Siebert type. To see this, we follow the construction in Section~\ref{local-model-GS-type-sec} with $\tau \subseteq \RR^2$ the so-called $(1,1,2)$-triangle given by the vertices $(0,0)$, $(2,0)$, and $(0,1)$. We set $q = 1$ and $\Delta_1 = \tau$. Then we have 
 $$\check\psi_0(n_1,n_2) = -\mathrm{inf}\{0,2n_1,n_2\} = \check\psi_1(n_1,n_2).$$
 This function is piece-wise linear on the dual fan of $\tau$. The monoid
 $$P = \{(n_1,n_2,a_0,a_1) \ | \ a_i \geq \check\psi_i(n_1,n_2),\ i = 1,2\}$$
 gives rise to a family $f: \bAA_P = \Spec \kk[P] \to \bAA^1$ by mapping $t \mapsto z^{(0,0,1,0)}$; this is a log toroidal family of Gross--Siebert type by definition. To see that $f: \bAA_P \to \bAA^1$ is isomorphic to $f: X \to \bAA^1$, note that
 \begin{center}
  \begin{tabular}{lll}
  $\bar x = (1,0,0,0)$, & $\bar y = (-1,-2,2,2)$ & $\bar z = (0,1,0,0)$, \\
  $\bar w = (0,-1,1,1)$, & $\bar t = (0,0,1,0)$, & $\bar u = (0,0,0,1)$, \\
 \end{tabular}
 \end{center}
 are elements of $P$. They satisfy the relations $\bar x + \bar y = 2 \cdot \bar w$ and $\bar z + \bar w = \bar t + \bar u$; thus, we obtain a ring homomorphism 
 $$\kk[x,y,z,w,u,t]/(xy - w^2, zw - tu) \to \kk[P]$$
 in the obvious way. A slightly cumbersome computation shows that the map is surjective; then $\bAA_P \subset X$ is a closed subscheme, and both are integral of the same dimension, i.e., $\bAA_P \cong X$. Both are endowed with the divisorial log structure coming from the central fiber, so they are isomorphic as log schemes as well. However, this family is not log toroidal of \emph{elementary} Gross--Siebert type because the nearby fiber has singularities in codimension $2$, which is not possible in a log toroidal family of elementary Gross--Siebert type by Lemma~\ref{elem-simplex-gen-fib}.
\end{ex}

\begin{figure}
 \begin{mdframed}
  \begin{center}
   \includegraphics[scale=0.8]{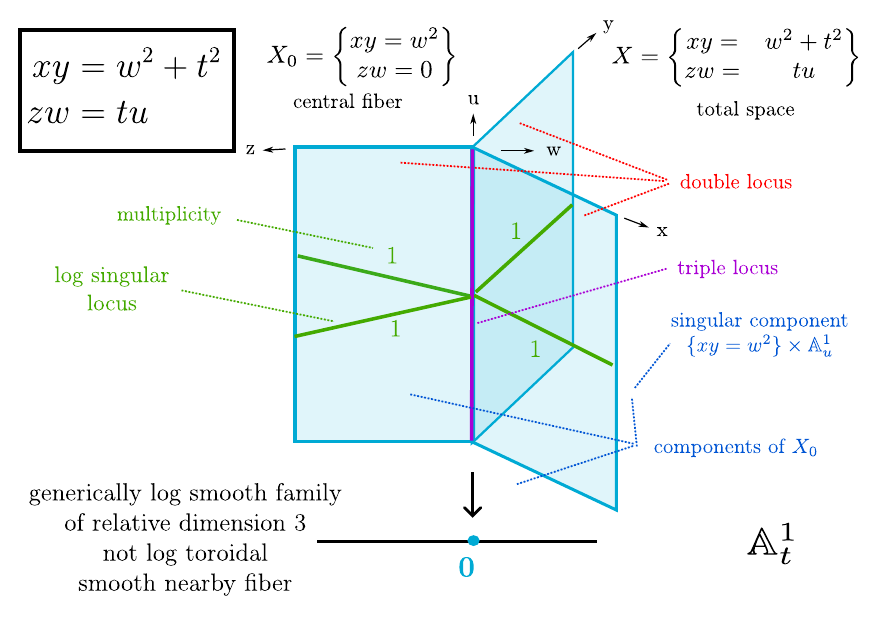}
  \end{center}
  \caption{Example~\ref{xy-w2-t2-zw-tu-example}}
 \end{mdframed}
\end{figure}

\begin{ex}\label{xy-w2-t2-zw-tu-example}\note{xy-w2-t2-zw-tu-example}
 Let 
 $$X = \Spec \kk[x,y,z,w,u,t]/(xy - w^2 - t^2, zw - tu),$$
 and let $f: X \to \bAA^1, \: t \mapsto t$, be endowed with the divisorial log structure given by $t = 0$ on source and target. This is a generically log smooth family of relative dimension $3$. It is very similar to the previous Example~\ref{xy-w2-zw-tu-example}, but now the double line in the log singular locus has been separated into two single lines such that the log singular locus $Z$ has now four branches. The general fiber is now smooth. This generically log smooth family is log toroidal with local model $xy = tz$ outside $\{0\}$, but \emph{not} log toroidal around the point $\{0\}$. In fact, it does not satisfy the base change property of Definition~\ref{bc-prop}, which is satisfied by all log toroidal families. More precisely, when we denote by $\W^1_{X/\bAA^1}$ the reflexive log differential forms of $f: X \to \bAA^1$, then $\W^1_{X/\bAA^1} \otimes \cO_{X_0}$ is not reflexive on $X_0$. Nonetheless, $\W^2_{X/\bAA^1}$ remains reflexive after pull-back to $X_0$. For the similar family in Example~\ref{base-change-violation-t3-w3}, also the pull-back of $\W^2$ is not reflexive.
\end{ex}

\section{Toroidal crossing spaces and the Gross--Siebert program}

The initial motivation to study deformations of spaces which are not everywhere log smooth comes from mirror symmetry. The idea to use pairs of dual reflexive polytopes as the backbone of mirror symmetry goes back to Batyrev's foundational work \cite{Batyrev1994}. In \cite{GrossSiebertI} and subsequent works, Gross and Siebert present a much refined and somewhat different variant of this idea. They start with an integral affine manifold $B$ with singularities, endowed with a polyhedral decomposition $\cP$, some form of polarization $\varphi$, and so-called open gluing data $s$. Given $(B,\cP,s)$, they construct two degenerate schemes $\check X_0(B,\cP,s)$ and $X_0(B,\cP,s)$ which are thought to be mirror pairs. The first one, referred to as the \emph{cone picture}, is glued from quasi-projective toric varieties $X_\tau$ corresponding to polyhedra $\tau \in \cP$. It is covered by open subsets corresponding to vertices $v \in \cP$. The second one, referred to as the \emph{fan picture}, is covered by open subsets corresponding to maximal cells $\sigma \in \cP$. A duality construction, the \emph{discrete Legendre transform}, associates with $(B,\cP,\varphi)$ a dual triple $(\check B,\check \cP,\check \varphi)$. At least in good cases, we then have $\check X_0(B,\cP,s) = X_0(\check B,\check \cP,\check s)$ and $\check X_0(\check B,\check \cP,\check s) = X_0(B,\cP,s)$ for appropriate open gluing data $s$ and $\check s$. The schemes $X_0(B,\cP,s)$ and $\check X_0(B,\cP,s)$ are reduced,  Gorenstein, and Calabi--Yau, but they have many irreducible components intersecting in a complicated way. 

The fundamental insight of Gross and Siebert is that actual mirror pairs of smooth Calabi--Yau manifolds should be obtained by smoothing both $X_0(B,\cP,s)$ and $\check X_0(B,\cP,s)$ to obtain some mirror pair $X(B,\cP,s)$ and $\check X(B,\cP,s)$. Since the two central fibers are Calabi--Yau, the existence of such a smoothing is plausible.

In order to construct a degeneration to $\check X_0 = \check X_0(B,\cP,s)$, one crucial step is to find a log structure on $\check X_0$. Putting this problem aside for a moment, it is much easier to define the \emph{ghost sheaf} of a log structure on $\check X_0$. The space $\check X_0$ admits an open cover by the toric boundaries $V(v)$ of some affine toric varieties $U(v)$, indexed by vertices $v \in \cP$. The affine toric varieties $U(v)$ are constructed from the affine structure on $B$ together with the polarization $\varphi$. While the log structures coming from $U(v)$ on $V(v)$ do not glue, their ghost sheaves do, giving rise to a global sheaf of monoids on $\check X_0$, which is usually denoted by $\cP$, unfortunately the same symbol as the polyhedral decomposition. Also the image of $1 \in \overline\M_{S_0}$ glues, giving rise to a global section $\bar\rho \in \cP$. In the language that we define in 
Chapter~\ref{toroidal-cr-sp-sec}, this is a \emph{toroidal crossing space}, but the concept really goes back to Gross and Siebert's work \cite{GrossSiebertI} as well as to \cite{SchroerSiebert2006}, which appeared in the same year. On a toroidal crossing space $(V,\cP,\bar\rho)$, there is a whole \emph{sheaf} $\cL\cS_V$ whose sections correspond to log smooth log structures on the respective open subset. Gross and Siebert have computed this sheaf explicitly on $\check X_0$. As it turns out, it usually has \emph{no global section}. In other words, there is no way of endowing $\check X_0$ with a global log structure which is everywhere similar to $U(v)$. The only thing one may do is to choose a closed subset $Z \subseteq X_0$, which will become the \emph{log singular locus}, and a section $s \in \Gamma(X_0 \setminus Z,\cL\cS_{\check X_0})$. In fact, there is a canonical way of choosing $Z \subseteq \check X_0$, and a canonical section $s \in \cL\cS_{\check X_0}$ on its complement, what Gross and Siebert call the \emph{normalized} section. In our language, that turns $\check X_0(B,\cP,s) \to S_0$ into a generically log smooth family, but the construction of $\check X_0(B,\cP,s)$ predates our language and is the original context where our notion is coming from. In the language of Gross--Siebert, $\check X_0(B,\cP,s)$ is now a \emph{toric log Calabi--Yau space}.

The next step is to define infinitesimal deformations of $\check X_0$ which lead up to the desired smoothing $\check X$. Since the family $\check X_0(B,\cP,s) \to S_0$ is not log smooth, there is no obvious choice, and special care must be taken. The first insight at this point is that the log singularities around $Z$ are not just any log singularities, but they admit very special local models in the form of taking a map $\NN \to P$ of sharp toric monoids, and then endowing the source $\Spec(\kk[P])$ not with the divisorial log structure of the full toric boundary, but with the divisorial log structure of the partial toric boundary $\{t = 0\}$. In \cite{FFR2021}, we introduced the name \emph{log toroidal family} for such generically log smooth families. Moreover, $\NN \to P$ is not arbitrary but of a very specific form; we call them \emph{log toroidal families of elementary Gross--Siebert type}, honoring their role as the local models of $\check X_0(B,\cP,s) \to S_0$.

While there are many generically log smooth deformations of $\check X_0 \to S_0$, any two of them which are both log toroidal of elementary Gross--Siebert type are locally isomorphic. This is the key insight to study deformations of $\check X_0 \to S_0$. Such deformations are called \emph{divisorial deformations} by Gross and Siebert in \cite{GrossSiebertII}. Thus, divisorial deformations form a system of deformations $\D$ in our sense. The original proofs of this fact can be found in \cite[\S 2]{GrossSiebertII}; in Chapter~\ref{elem-GS-type-sec}, we review this proof for the reader's convenience. Up to now, this is the most important source of examples of systems of deformations which are not log smooth, showing the usefulness of the concept.

As Gross and Siebert explain in \cite[Rem.~2.19]{GrossSiebertII}, they originally expected to apply some form of Bogomolov--Tian--Todorov theorem in their situation, thus obtaining the deformation to an irreducible and often smooth space, at least analytically. However, as the logarithmic Bogomolov--Tian--Todorov theorem was not available until 2019, they developed instead the fascinating but intricate method given in \cite{Gross2011}, which uses an explicit combinatorial device on $B$, the \emph{scattering diagram}, to track the comparison isomorphisms of the local deformations to conclude that a global deformation exists up to any order.

\vspace{\baselineskip}

As already mentioned above, we have abstracted from the analysis in \cite{GrossSiebertI} of $\check X_0(B,\cP,s)$ and its ghost sheaf $\cP$ the notion of a \emph{toroidal crossing space}. In its current form briefly defined in \cite{FFR2021}, it consists of a scheme $V$ together with a sheaf of monoids $\cP$ in the \'etale topology and a global section $\bar\rho \in \cP$ such that the triple $(V,\cP,\bar\rho)$ is locally isomorphic to what we obtain on the central fiber $V(\sigma) \to S_0$ of a morphism $U(\sigma) \to \bAA^1$ obtained from a Gorenstein toric monoid $P$ with Gorenstein degree $\bar\rho \in P$, see Construction~\ref{prototype-constr}. We discuss this notion in some detail in Chapter~\ref{toroidal-cr-sp-sec} since we haven't yet done so elsewhere. Even beyond the Gross--Siebert program, choosing a section $s \in \cL\cS_V(V \setminus Z)$ on a toroidal crossing space $(V,\cP,\bar\rho)$ is a rich source of generically log smooth families, not all of which admit nice local models as in the case of elementary Gross--Siebert type, for example the family in Example~\ref{xy-w2-t2-zw-tu-example}, which is not log toroidal for any local model of the form $A_{P,\F} \to A_Q$.\footnote{We will later consider this family as an \emph{enhanced} generically log smooth family to remedy the situation that the formation of the Gerstenhaber calculus does not commute with base change here.}

Singularities related to the one in Example~\ref{xy-w2-t2-zw-tu-example} occur in the following example, which we learned from Alessio Corti and sketch only roughly (details are to appear elsewhere).

\begin{ex}\label{Fano-3fold-exa}\note{Fano-3fold-exa}
 Let $\Delta$ be a reflexive polytope of dimension $3$. Then there is a canonical structure of a toroidal crossing space on the reducible toric scheme $X_0(\Delta)$ associated with the central subdivision of $\Delta$. The canonical toric degeneration from the Gorenstein toric Fano variety $X(\Delta)$ to $X_0(\Delta)$ induces a section $s_0 \in \cL\cS$. An analysis of $\cL\cS$ shows that we can find another section $s \in \cL\cS$ which is log toroidal of elementary Gross--Siebert type. More precisely, outside finitely many points, the log singularities are given by $xy = tz$. These simple branches of $Z$ meet in isolated log singular points in the strata of codimension $2$ to give rise to more complicated log singularities. A system of deformations $\D$ may be defined such that any formal deformation of type $\D$ constitutes a formal smoothing.
\end{ex}

This also illustrates one more method to construct systems of deformations $\D$: The family is log toroidal of elementary Gross--Siebert type outside finitely many points, so we can choose some local model $V_{\alpha;A}$ around the isolated very singular points which is log toroidal of elementary Gross--Siebert type. By local rigidity of deformations of elementary Gross--Siebert type, they must coincide on overlaps as soon as each $V_\alpha$ contains at most one very singular point.

\section{A curved Lie algebra controlling the deformation functor}

Let $X/\kk$ be a smooth variety. In classical infinitesimal smooth deformation theory, isomorphism classes of flat deformations are encoded in a \emph{deformation functor} 
$$\mathrm{Def}_X: \mathbf{Art}_\kk \to \mathbf{Set},$$
which is controlled by a dg Lie algebra $L_X^\bullet$, i.e., when we take gauge equivalence classes of Maurer--Cartan elements in $\m_A \otimes_\kk L_X^\bullet$, then we recover $\mathrm{Def}_X$. The most classical approach is complex analytic, using the Kodaira--Spencer dg Lie algebra as defined e.g.~in \cite[Defn.~8.3.1]{ManettiLieMethods2022}, but there is a purely algebraic approach as well, replacing the Dolbeault resolution with the Thom--Whitney resolution, as described in \cite[Ex.~2.1]{AlgebraicBTT2010} by Iacono and Manetti.

This is no longer true in logarithmic deformation theory, even in the log smooth case. Namely, when $L^\bullet_{X_0/S_0}$ is a $\kk\llbracket Q\rrbracket$-linear dg Lie algebra which controls the log smooth deformation functor $\mathrm{LD}_{X_0/S_0}$ for a log smooth morphism $f_0: X_0 \to S_0$, then $0 \in L^1_{X_0/S_0} \otimes_{\kk\llbracket Q\rrbracket} A$ is always a Maurer--Cartan solution. In other words, over any $A \in \mathbf{Art}_Q$, there is at least one log smooth deformation over $S_A$. Unfortunately, this is not always the case.

\begin{ex}[Persson--Pinkham]\label{obstructed-defo-nc}\note{obstructed-defo-nc}
 Let $X = Y \amalg_D Z$ as in \cite[Thm.~1']{PerssonPinkham1983}. Then $X$ is a $d$-semistable normal crossing space with two smooth components $Y$ and $Z$ intersecting in a smooth variety $D$. In particular, it can be endowed with the structure of a log smooth log morphism $f_0: X_0 \to S_0$. The space $X_0$ is, however, not the central fiber of any analytic semistable degeneration, which would exist if we could find an infinitesimal deformation over all $S_k$.
\end{ex}

Thus, there is no $\kk\llbracket Q\rrbracket$-linear dg Lie algebra controlling $\mathrm{LD}_{X_0/S_0}$. However, Chan, Leung, and Ma observed fairly recently in \cite{ChanLeungMa2023} that this problem can be solved by replacing a dg Lie algebra with a $\kk\llbracket Q\rrbracket$-linear \emph{curved} Lie algebra, which is a graded Lie algebra with an operator $\bar\partial$ for which the condition $\bar\partial^2 = 0$ is relaxed.

\begin{defn}
 Let $\Lambda$ be a complete local Noetherian $\kk$-algebra with residue field $\kk$. Then a \emph{$\Lambda$-linear curved Lie algebra} is a quadruple $(L^\bullet,[-,-],\bar\partial,\ell)$ where each $L^i$ is a flat and complete $\Lambda$-module, $(L^\bullet,[-,-])$ is a graded $\Lambda$-linear Lie algebra, $\bar\partial: L^i \to L^{i + 1}$ is a $\Lambda$-linear map with 
 $$\bar\partial[\theta,\xi] = [\bar\partial\theta,\xi] + (-1)^{|\theta| + 1}[\theta,\bar\partial\xi],$$
 and $\ell \in \m_\Lambda \cdot L^2$ is an element with $\bar\partial^2(\theta) = [\ell,\theta]$ and $\bar\partial(\ell) = 0$.
\end{defn}

The operator $\bar\partial$ is called the \emph{predifferential}, and $\ell \in L^2$ is called the \emph{curvature}. We set $L_A^\bullet := L^\bullet \otimes_\Lambda A$ for $A \in \mathbf{Art}_\Lambda$. Every element $\phi \in \m_A \cdot L^1_A$ gives rise to a modified predifferential $\bar\partial_\phi := \bar\partial + [\phi,-]$, which forms a new $\Lambda$-linear curved Lie algebra with the curvature 
$$\ell_\phi := \bar\partial\phi + \frac{1}{2}[\phi,\phi] + \ell.$$
This is called the \emph{twisting procedure}.\index{twisting procedure}\footnote{The twisting procedure is a more general phenomenon which is discussed in the nice new book \cite{Dotsenko2024} by Dotsenko, Shadrin, and Vallette. It applies classically to curved $A_\infty$-algebras as well as curved $L_\infty$-algebras, and moreover, there is an elaborate theory of the twisting procedure on the level of operads.} In order to obtain a deformation functor, we consider elements $\phi \in \m_A \cdot L^1_A$ with $\ell_\phi = 0$, i.e., elements which satisfy the \emph{extended} Maurer--Cartan equation 
$$\bar\partial\phi + \frac{1}{2}[\phi,\phi] + \ell = 0.$$
These elements satisfy $\bar\partial_\phi^2 = [\ell_\phi,-] = 0$. The usual formulae, which can be found in Lemma~\ref{gauge-action}, form a \emph{gauge action} of elements in $\m_A \cdot L_A^0$ on the Maurer--Cartan solutions, and equivalence classes form the deformation functor 
$$\mathrm{Def}(L^\bullet,-): \enspace \mathbf{Art}_\Lambda \to \mathbf{Set}.$$
Now $\phi = 0$ is not automatically a Maurer--Cartan solution for all $A \in \mathbf{Art}_\Lambda$, thus allowing for the existence of a $\kk\llbracket t\rrbracket$-linear curved Lie algebra controlling the deformation functor in the examples of Persson--Pinkham.
It has been worked out in \cite{Felten2022} that for log smooth and saturated $f_0: X_0 \to S_0$, there is indeed a $\kk\llbracket Q\rrbracket$-linear curved Lie algebra $L_{X_0/S_0}^\bullet$ which controls the log smooth deformation functor $\mathrm{LD}_{X_0/S_0}$.

\vspace{\baselineskip}

The curved Lie algebra $L_{X_0/S_0}^\bullet$ is constructed in a complex process, which we generalize (as compared to \cite{Felten2022}) in this book.  

When $f: X \to S$ is a generically log smooth family of relative dimension $d \geq 1$, then we have the log de Rham complex $\Omega^\bullet_{U/S}$ on $U$. We obtain a complex on the whole of $X$ by simply defining $\W^\bullet_{X/S} := j_*\Omega^\bullet_{U/S}$ as the direct image. Similarly, we define the  polyvector fields as $\V^\bullet_{X/S} := j_*\Theta^{-\bullet}_{U/S}$. They live in negative degrees $[-d,0]$, a convention that is by now customary to us and goes at least back to \cite{ChanLeungMa2023}. We usually assume that the formation of $\V^\bullet_{X/S}$ and $\W^\bullet_{X/S}$ commutes with base change. The pair $\V\,\W^\bullet_{X/S} := (\V^\bullet_{X/S},\,\W^\bullet_{X/S})$ carries a number of operations, which turn it into what we call a \emph{two-sided Gerstenhaber calculus}, see Proposition~\ref{G-A-construction}.

In a first step toward the construction of $L^\bullet_{X_0/S_0}$, we forget about the log structure on $f_0: X_0 \to S_0$ and the local models $V_{\alpha;A} \to S_A$ for its infinitesimal deformations, and we just retain the two-sided Gerstenhaber calculi. This yields a deformation problem for \emph{geometric families of two-sided Gerstenhaber calculi}. Now the central fiber is a morphism $f_0: X_0 \to S_0$ of schemes, carrying a two-sided Gerstenhaber calculus $\V\,\W^\bullet_{X_0/S_0}$, and the local models for deformations are morphism of schemes $V_{\alpha;A} \to S_A$, carrying a two-sided Gerstenhaber calculus $\V\,\W^\bullet_{\alpha;A}$. For morphisms $B' \to B$ in $\mathbf{Art}_Q$, they come with restriction maps, and on overlaps $V_\alpha \cap V_\beta$, we have isomorphisms 
$$\V\,\W^\bullet_{\alpha;A}|_{\alpha\beta} \cong \V\,\W^\bullet_{\beta;A}|_{\alpha\beta}$$
which are induced from the isomorphisms of generically log smooth families. Now a deformation of $f_0: X_0 \to S_0$ is a morphism $f_A: X_A \to S_A$ of schemes together with a two-sided Gerstenhaber calculus $\V\,\W^\bullet_A$ which is, on $V_\alpha$, isomorphic to the local model $\V\,\W^\bullet_{\alpha;A}$. Unlike in the case of generically log smooth families, we now have to track the isomorphisms with the local models. The reason for this is that we need to distinguish between \emph{inner} automorphisms of a geometric family of two-sided Gerstenhaber calculi, which can be constructed explicitly from the operations in the calculus, and \emph{outer} automorphisms, which are simply structure-preserving invertible self-maps. The inner automorphisms are in one-to-one correspondence with automorphisms of the generically log smooth family, but there may be additional outer automorphisms. From the one-to-one correspondence between automorphisms in the two cases, 
we obtain that 
$$\mathrm{LD}_{X_0/S_0}^\D \cong \mathrm{GDef}^\D_{X_0/S_0}$$
as functors of Artin rings, where the latter deformation functor classifies equivalence classes deformations of geometric families of two-sided Gerstenhaber calculi. Two deformations are equivalent if there is an (outer) isomorphism such that the induced automorphisms of the local models $V_{\alpha;A} \to S_A$ via the fixed comparison isomorphisms is an \emph{inner} automorphism of the local model.

In the next step, we apply methods of homological algebra to the geometric families of Gerstenhaber calculi. We fix a second cover $\U = \{U_i\}_i$ of $X_0$, which defines a notion of \emph{Thom--Whitney resolution} $\TW^\bullet(\F)$ of a sheaf of $\kk$-vector spaces $\F$ on $X_0$. Applying $\TW^\bullet(-)$ to the two-sided Gerstenhaber calculus $\V\,\W^\bullet_0$ on the central fiber yields what we call a \emph{differential bigraded two-sided Gerstenhaber calculus}. Each sheaf $\V^p_{X_0/S_0}$ and $\W^i_{X_0/S_0}$ is replaced with an acyclic resolution, and all operations extend to the resolutions. The original two-sided Gerstenhaber calculus can be recovered as the \emph{cohomology} of the differential $\bar\partial$. We apply the same resolution construction also to the local models, obtaining the curved two-sided Gerstenhaber calculi $\V\,\W^{\bullet,\bullet}_{\alpha;A}$. When we have a deformation $f_A: X_A \to S_A$ with two-sided Gerstenhaber calculus $\V\,\W^\bullet_A$, then we can apply the resolution as well. The resulting curved two-sided Gerstenhaber calculus $\V\,\W^{\bullet,\bullet}_A$ is locally isomorphic to the local models $\V\,\W^{\bullet,\bullet}_{\alpha;A}$. It is a curious but in fact central fact of our theory that, once we forget the differential $\bar\partial$ and keep just the \emph{bigraded two-sided Gerstenhaber calculus}, all $\V\,\W_A^{\bullet,\bullet}$ are isomorphic, no matter what the original deformation was. This is essentially because isomorphism classes are classified by $H^1(X_0,\V_0^{-1,0}) = 0$, which vanishes due to acyclicity. In fact, there is a unique gluing $(PV_{X_0/A}^{\bullet,\bullet},DR^{\bullet,\bullet}_{X_0/A})$ of the $\V\,\W^{\bullet,\bullet}_{\alpha;A}$, even if there is no deformation $f_A: X_A \to S_A$; this gluing carries all operations except the differential $\bar\partial$, which is not invariant under inner automorphisms of the local models and hence has no unique action on $(PV_{X_0/A}^{\bullet,\bullet},DR^{\bullet,\bullet}_{X_0/A})$. We can, however, choose some differential operator $\bar\partial$ on $(PV_{X_0/A}^{\bullet,\bullet},DR^{\bullet,\bullet}_{X_0/A})$ which is of the form $\bar\partial + [\varphi,-]$ on $\V\,\W^{\bullet,\bullet}_{\alpha;A}$ for $\varphi \in \V^{-1,1}_{\alpha;A}$ depending on the chosen isomorphism with $\V\,\W^{\bullet,\bullet}_{\alpha;A}$. This differential operator is called a \emph{predifferential} because we may have $\bar\partial^2 \not = 0$.

A \emph{differential} on $(PV_{X_0/A}^{\bullet,\bullet},DR^{\bullet,\bullet}_{X_0/A})$ is now a differential operator $\bar\partial' := \bar\partial + [\varphi,-]$ with $(\bar\partial')^2 = 0$. When we have such a differential, then the cohomology is a deformation of geometric families of two-sided Gerstenhaber calculi, locally isomorphic to $\V\,\W^\bullet_{\alpha;A}$. For two differentials, the associated deformations are equivalent if and only if the two differentials are \emph{gauge equivalent}, and any deformation of families of calculi is the cohomology of some differential. Thus, gauge equivalence classes of differentials form a functor of Artin rings isomorphic to $\mathrm{GDef}^\D_{X_0/S_0}$.

When taking a limit, then we obtain a curved two-sided Gerstenhaber calculus 
$$(PV_{X_0/\Lambda}^{\bullet,\bullet},DR^{\bullet,\bullet}_{X_0/\Lambda})$$
over $\Lambda = \kk\llbracket Q\rrbracket$. Now $L^\bullet_{X_0/S_0} = \Gamma(X_0,PV^{-1,\bullet}_{X_0/\Lambda})$ is a $\Lambda$-linear curved Lie algebra, and gauge equivalence classes of Maurer--Cartan solutions are nothing but gauge equivalence classes of differentials as above. Thus, we have an isomorphism 
$$\mathrm{LD}^\D_{X_0/S_0} \cong \mathrm{Def}(L^\bullet_{X_0/S_0},-): \enspace \mathbf{Art}_Q \to \mathbf{Set}$$
of functors of Artin rings.

\vspace{\baselineskip}

In the earlier work \cite{Felten2022} of the author, we have carried out the same construction for log smooth morphisms only, and with the Gerstenhaber \emph{algebra} $\V^\bullet_{X/S}$ only, there denoted by $\G^\bullet_{X_0/S_0}$. For the construction of $L^\bullet_{X_0/S_0}$, this makes little difference, but we have indeed $\W^\bullet_{X_0/S_0}$ as well at our disposal, which greatly amplifies the possibilities to study properties of the deformations. For example, compare Theorem~\ref{first-abs-unob-thm-intro} with Theorem~\ref{second-abs-unob-thm-intro}. The article \cite{Felten2022} was in turn inspired by \cite{ChanLeungMa2023}, where the Gerstenhaber calculus including $\W^\bullet_{X_0/S_0}$ is covered, albeit with one operation less (the left contraction $\vdash$ is missing), only in the analytical setting, and only for global sections, not on the sheaf level. We need the part of $\W^\bullet_{X_0/S_0}$ in order to study the logarithmic Bogomolov--Tian--Todorov theorem.

\section{The logarithmic Bogomolov--Tian--Todorov theorem}

If $X/\kk$ is a smooth and proper Calabi--Yau variety, i.e., $\Omega^d_X \cong \cO_X$, where $d = \mathrm{dim}(X)$, then the famous Bogomolov--Tian--Todorov theorem states that $\mathrm{Def}_X$ is \emph{unobstructed}, i.e., if $B' \to B$ is a surjection in $\mathbf{Art}_\kk$, then $\mathrm{Def}_X(B') \to \mathrm{Def}_X(B)$ is surjective. Several proofs are known for this theorem, including Ran's proof via the $T^1$-lifting criterion in \cite{Ran1992} and Iacono--Manetti's proof via the homotopy abelianity of the dg Lie algebra $L_X^\bullet$ in \cite{AlgebraicBTT2010}. 

The $\kk\llbracket Q\rrbracket$-linear curved Lie algebra $L^\bullet_{X_0/S_0}$ is closely related to the dg Lie algebra $L_X^\bullet$ of Iacono--Manetti in that both are constructed via a Thom--Whitney resolution. However, the proof of unobstructedness via homotopy abelianity cannot be generalized to $L^\bullet_{X_0/S_0}$; indeed, it is not even clear what homotopy abelianity of $L^\bullet_{X_0/S_0}$ should mean since it is not a complex. The argument of Iacono--Manetti shows in our situation that $L^\bullet_{X_0/S_0} \otimes_{\kk\llbracket Q\rrbracket} \kk$ is homotopy abelian, as we have explained in \cite[Thm.~3.3]{FeltenPetracci2022}, showing that $\mathrm{LD}_{X_0/S_0}$ is unobstructed on the subcategory of Artinian local $\kk$-algebras where the $Q$-action is trivial. This captures only locally trivial log smooth deformations of $f_0: X_0 \to S_0$ but no deformations in the smoothing direction.

At this point, the argument of Chan, Leung, and Ma in \cite{ChanLeungMa2023} takes advantage of the fact that we do not only have a $\kk\llbracket Q\rrbracket$-linear curved Lie algebra but much more structure. The centerpiece of their argument is \cite[Thm.~5.6]{ChanLeungMa2023}, which they call the \emph{abstract unobstructedness theorem}. They state the theorem solely for the structures as they arise in this article, but the abstract unobstructedness theorem is a purely algebraic fact once the correct definitions are collected. At the center of formulating the abstract unobstructedness theorem in purely algebraic terms is the notion of a \emph{$\Lambda$-linear curved Batalin--Vilkovisky algebra}. Strictly speaking, the setup can be further weakened, but in practice, at least for this work, we have that structure. Our definition has some minor improvements in that we stipulate necessary relations which are nowhere explicit in \cite{ChanLeungMa2023}.

\begin{defn}
 A \emph{$\Lambda$-linear curved Batalin--Vilkovisky algebra} of dimension $d \geq 1$ is a tuple 
 $$(G^{\bullet,\bullet}, \wedge, 1, [-,-],\Delta,\bar\partial,\ell,y)$$
 where $G^{p,q}$ is a flat and complete $\Lambda$-module for $-d \leq p \leq 0$ and $q \geq 0$, the operation $\wedge$ is a $\Lambda$-bilinear product with unit $1 \in G^{0,0}$, the operation $[-,-]$ is a $\Lambda$-bilinear Lie bracket, $\Delta: G^{p,q} \to G^{p + 1,q}$ is a Batalin--Vilkovisky operator satisfying $\Delta^2 = 0$, the operation $\bar\partial: G^{p,q} \to G^{p,q + 1}$ is a predifferential, $\ell \in \m_\Lambda \cdot G^{-1,2}$ is such that $\bar\partial^2(\theta) = [\ell,\theta]$, and $y \in \m_\Lambda \cdot G^{0,1}$ is such that $\Delta\bar\partial(\theta) + \bar\partial\Delta(\theta) = [y,\theta]$. The precise relations can be found in Chapter~\ref{BV-alg-sec}.
\end{defn}

When $G^{\bullet,\bullet}$ is a $\Lambda$-linear curved Batalin--Vilkovisky algebra, then $L^\bullet := G^{-1,\bullet}$ is a $\Lambda$-linear curved Lie algebra, giving rise to a deformation functor $\mathrm{Def}(L^\bullet,-)$. In order to show that this deformation functor is unobstructed, it is sufficient (and necessary) to show that the Maurer--Cartan functor $\mathrm{MC}(L^\bullet,-)$ is unobstructed. The subject of the abstract unobstructedness of \cite{ChanLeungMa2023} is however not the Maurer--Cartan functor $\mathrm{MC}(L^\bullet,-)$ but the \emph{semi-classical Maurer--Cartan functor} $\mathrm{SMC}(G^{\bullet,\bullet},-)$. This functor of Artin rings classifies solutions $(\phi,f)$ with $\phi \in \m_A \cdot G_A^{-1,1}$ and $f \in \m_A \cdot G_A^{0,0}$ of the \emph{semi-classical extended Maurer--Cartan equation}
\begin{align}
 \bar\partial\phi + \frac{1}{2}[\phi,\phi] + \ell &= 0 \nonumber \\
 \bar\partial f + [\phi,f] + y + \Delta \phi &= 0. \nonumber
\end{align}
Obviously, if $(\phi,f)$ is a semi-classical Maurer--Cartan solution, then $\phi$ is a classical Maurer--Cartan solution. These names come from the fact that we consider also a \emph{quantum} Maurer--Cartan equation in proving the unobstructedness theorem. For $G^{\bullet,\bullet}$, the most natural condition which suffices for the unobstructedness theorem is the following:

\begin{defn}
 A $\Lambda$-linear curved Batalin--Vilkovisky algebra is \emph{quasi-perfect} if the following conditions hold:
 \begin{enumerate}[label=(\roman*)]
  \item The first spectral sequence $'\!E_{0;r}$ associated with the double complex $(G_0^{\bullet,\bullet},\Delta,\bar\partial)$ over $A_0 = \kk$ degenerates at $E_1$.
  \item For every $A \in \mathbf{Art}_\Lambda$, the induced map $H^k(G_A^\bullet,\check d) \to H^k(G_0^\bullet,\check d)$ with 
  $$\check d = \Delta + \bar\partial + (\ell + y) \wedge (-)$$
  is surjective.
 \end{enumerate}
\end{defn}
Over $\kk$, we have $\bar\partial^2 = 0$ since $\ell \in \m_\Lambda \cdot G^{-1,2}$; thus, $(G_0^{\bullet,\bullet},\Delta,\bar\partial)$ is indeed a double complex. The map $\check d$ in the definition turns out to always satisfy $\check d^2 = 0$, and hence it is a differential on the total complex $G^\bullet$. The notion of \emph{perfectness} incorporates an additional finiteness condition which is not necessary for the unobstructedness.

\begin{thm}[First abstract unobstructedness theorem]\label{first-abs-unob-thm-intro}\note{first-abs-unob-thm-intro}
 Let $G^{\bullet,\bullet}$ be a quasi-perfect $\Lambda$-linear curved Batalin--Vilkovisky algebra. Then $\mathrm{SMC}(G^{\bullet,\bullet},-)$ is unobstructed.
\end{thm}

In fact, a weaker condition, called \emph{semi-perfectness} and defined in Definition~\ref{semi-perf-defn}, is sufficient. Also note that Theorem~\ref{first-abs-unob-thm-intro} does \emph{not} show unobstructedness of $\mathrm{Def}(L^\bullet,-)$---it only shows that \emph{pairs} $(\phi,f)$ can be lifted along surjections $B' \to B$ in $\mathbf{Art}_\Lambda$. We will come back to this point shortly.

In the case without log structures, the usage of Batalin--Vilkovisky algebras to show unobstructedness is by now classical. Early uses in the mathematics literature include \cite{Manin1999,BarannikovKontsevich1998}; a transition from the original use in physics to deformation theory seems to be made in \cite{Stasheff1997}. The degeneracy of the spectral sequence as a condition for the unobstructedness of the deformation functor seems to be first mentioned in \cite[\S 4.2.2]{KKP2008}; in Proposition~\ref{spec-seq-alternative-formulation}, we check that the degeneration of the spectral sequence at $E_1$ is indeed equivalent to the criterion as stated in \cite[Defn.~4.13]{KKP2008}. The degeneration of a somewhat different spectral sequence is given as a criterion in \cite{Terilla2008}.

While Theorem~\ref{first-abs-unob-thm-intro} is enough to show the \emph{existence} of some deformation, it is insufficient to conclude that $\mathrm{LD}_{X_0/S_0}^\D$ is unobstructed, i.e., that \emph{every} infinitesimal deformation to some $B$ can be lifted to $B'$. To show this stronger statement, we take once again advantage of additional structure that we have in the geometric situation. Instead of a Batalin--Vilkovisky algebra, we use a Batalin--Vilkovisky \emph{calculus}.

\begin{defn}
 A \emph{$\Lambda$-linear curved Batalin--Vilkovisky calculus} of dimension $d \geq 1$ is a tuple 
 $$(G^{\bullet,\bullet}, \wedge, 1_G, [-,-],\Delta,\bar\partial,\ell,y, A^{\bullet,\bullet},\wedge,1_A,\partial,\, \invneg \, ,\omega,\bar\partial, \cL)$$
 with a $\Lambda$-linear curved Batalin--Vilkovisky algebra, flat and complete $\Lambda$-modules $A^{i,j}$ for $0 \leq i \leq d$ and $j \geq 0$, a product $\wedge$ on $A^{\bullet,\bullet}$ with unit $1_A$, a de Rham differential $\partial: A^{i,j} \to A^{i + 1,j}$, a contraction map $\invneg: G^{p,q} \times A^{i,j} \to A^{p + i, q + j}$, a volume form $\omega \in A^{d,0}$, a predifferential $\bar\partial: A^{i,j} \to A^{i, j + 1}$ with $\bar\partial^2(\alpha) = \cL_\ell(\alpha)$, and a Lie derivative $\cL$. These data must satisfy the relations given in Chapter~\ref{BV-alg-sec}; in particular, 
 $$\kappa: G^{p,q} \to A^{p + d,q}, \quad \theta \mapsto (\theta \ \invneg \ \omega),$$
 must be an isomorphism, and we have $\Delta(\theta) \ \invneg \ \omega = \partial(\theta \ \invneg \ \omega)$.
\end{defn}

By Lemma~\ref{perfectness-stated-on-G}, a $\Lambda$-linear curved Batalin--Vilkovisky calculus $(G^{\bullet,\bullet},A^{\bullet,\bullet})$ is quasi-perfect (in the sense of Definition~\ref{G-calc-q-perf-def}) if and only if the Batalin--Vilkovisky algebra $G^{\bullet,\bullet}$ is quasi-perfect as defined above. When we exploit this additional structure, we can show that any classical Maurer--Cartan solution $\phi$ can be complemented to a semi-classical Maurer--Cartan solution $(\phi,f)$, see Theorem~\ref{SMC-MC-smooth}. This yields a second abstract unobstructedness theorem, the actual unobstructedness result which we have desired in the first place. This somewhat stronger result is not contained in \cite{ChanLeungMa2023}.

\begin{thm}[Second abstract unobstructedness theorem]\label{second-abs-unob-thm-intro}\note{second-abs-unob-thm-intro}
 Let $(G^{\bullet,\bullet},A^{\bullet,\bullet})$ be a quasi-perfect $\Lambda$-linear curved Batalin--Vilkovisky calculus. Then $\mathrm{MC}(L^\bullet,-)$ is unobstructed for the $\Lambda$-linear dg Lie algebra $L^\bullet = G^{-1,\bullet}$.
\end{thm}

Recall that, in the geometric situation of a generically log smooth family with a system of deformations $\D$, we have constructed a $\kk\llbracket Q\rrbracket$-linear curved \emph{Gerstenhaber} calculus 
\begin{equation}\label{PV-DR-intro}
 (PV^{\bullet,\bullet}_{X_0/\Lambda},DR^{\bullet,\bullet}_{X_0/\Lambda}).
\end{equation}
In other words, we do not have the volume form $\omega$, the Batalin--Vilkovisky operator $\Delta$, and the element $y \in \m_\Lambda \cdot G^{0,1}$ on this structure. However, when assuming that $f_0: X_0 \to S_0$ is log Calabi--Yau, then a volume form $\tilde \omega_0 \in \W^d_{X_0/S_0}$ gives rise to a volume form 
$$\omega_0 \in DR^{d,0}_{X_0/\kk} = DR^{d,0}_{X_0/\Lambda} \otimes_\Lambda \kk.$$
This can be lifted order by order to obtain a (even after the initial choice of $\tilde\omega_0$ non-canonical) volume form $\omega \in DR^{d,0}_{X_0/\Lambda}$. Applying Proposition~\ref{BV-calc-construction}, we find a structure of Batalin--Vilkovisky calculus on our Gerstenhaber calculus after establishing the necessary properties of $\omega$. This strategy is slightly different from the original approach in \cite{ChanLeungMa2023}, which constructs $\omega$ and $\Delta$ directly along with the other structures of the Gerstenhaber calculus.

\vspace{\baselineskip}

With Theorem~\ref{second-abs-unob-thm-intro}, the final step to achieve the logarithmic Bogomolov--Tian--Todorov theorem is to show that \eqref{PV-DR-intro} is indeed quasi-perfect. The first part, the degeneration of the spectral sequence, has been known from the beginning in the log smooth case, see \cite[Thm.~4.12]{kkatoFI}. In the log toroidal case, this is a recent insight by Filip, Ruddat, and the author in \cite{FFR2021}; the method is the same as the original algebraic approach by Deligne--Illusie, heavily relying on the local models of log toroidal families to study the behavior of the method under the presence of log singularities. The second part of the quasi-perfectness condition is somewhat subtle. If there is some generically log smooth deformation $f_A: X_A \to S_A$ of type $\D$, then it is equivalent to that the relative Hodge--de Rham spectral sequence of $f_A$ degenerates at $E_1$. However, the condition is stronger in that it does not rely on the existence of some gluing of the local models $V_{\alpha;A} \to S_A$ to a global deformation, and must be satisfied before we know the existence of a global gluing in order to conclude this very existence from Theorem~\ref{first-abs-unob-thm-intro}. Thus, that we have shown in \cite{FFR2021} that the relative spectral sequence degenerates at $E_1$ for infinitesimal log toroidal deformations does not show that the condition is satisfied. Nonetheless, a subtle variant of this argument, given by Chan, Leung, and Ma in \cite[Lemma~4.18]{ChanLeungMa2023}, yields the desired condition. In Chapter~\ref{perfect-G-calc-sec}, we show that this argument applies to our algebraically constructed Gerstenhaber calculus \eqref{PV-DR-intro} as well, not only to the analytic version of Chan--Leung--Ma. Interestingly, we have to analytify it nonetheless in order to conclude the quasi-perfectness of the algebraic Gerstenhaber calculus---the argument is genuinely analytical. In \cite{FFR2021}, the situation is the same.

\begin{thm}
 Let $f_0: X_0 \to S_0$ be a proper log Gorenstein log toroidal family. Let $\D$ be a system of deformations as a generically log smooth family all of whose local models $V_{\alpha;A} \to S_A$ are log toroidal. Then the $\kk\llbracket Q\rrbracket$-linear curved (two-sided) Gerstenhaber calculus \eqref{PV-DR-intro} is quasi-perfect.
\end{thm}
\begin{cor}[Logarithmic Bogomolov--Tian--Todorov theorem]\label{log-BBT-intro}\note{log-BBT-intro}
 Let $f_0: X_0 \to S_0$ be a proper log Calabi--Yau log toroidal family. Let $\D$ be a system of deformations which is log toroidal. Then $\mathrm{LD}^\D_{X_0/S_0}$ is unobstructed. In particular, this holds when $f_0: X_0 \to S_0$ is log smooth and $\D$ is the system of log smooth deformations.
\end{cor}

In the log smooth case, we have stated this result in \cite[Thm.~3.7]{FeltenPetracci2022} in our joint work with Petracci. There, we deduce via \cite[Prop.~2.8]{FeltenPetracci2022} that, in order for $\mathrm{LD}_{X_0/S_0}$ to be unobstructed, it is sufficient that the restriction map along 
$$A_{k + 1} = \kk[Q]/\m_Q^{k + 2} \to \kk[Q]/\m_Q^{k + 1} = A_k$$
is surjective. Our description of this key step has, however, been very sketchy, deferring the details to a future article, and just mentioning the method of \cite{ChanLeungMa2023}. Here, we deliver the promised elaborate account of the theory in its algebraic version. The reduction to $A_{k + 1} \to A_k$ is no longer necessary due to the more general form of Theorem~\ref{second-abs-unob-thm-intro}.

In general, many interesting generically log smooth families are not log toroidal. It would be interesting to find a version of Corollary~\ref{log-BBT-intro} for other log singularities than log toroidal ones. By the theory that we already have, it will be sufficient to show that the Gerstenhaber calculus \eqref{PV-DR-intro} is quasi-perfect.

\section{Constancy of the Hodge numbers}

Let $f_0: X_0 \to S_0$ be a proper log Gorenstein generically log smooth family with a system of deformations $\D$. Assume that \eqref{PV-DR-intro} is quasi-perfect, as is the case for example if everything is log toroidal. We do \emph{not} assume that $f_0: X_0 \to S_0$ is log Calabi--Yau here. On any deformation $f: X_A \to S_A$ of type $\D$, we have an isomorphism 
$$\HH^m(X_A,\W^\bullet_{X_A/S_A}) \cong H^m(DR^\bullet_{X_0/A}, d_\phi)$$
for the operator $d_\phi = \partial + \bar\partial_\phi$, where $\bar\partial_\phi$ is the differential corresponding to the Maurer--Cartan solution $\phi$ corresponding to $f: X_A \to S_A$. The right hand side commutes with base change as part of the definition of quasi-perfectness. Thus, also the left hand side commutes with base change. Then we can deduce from Proposition~\ref{spectral-sequence-Artin-ring} that the Hodge--de Rham spectral sequence 
$$E_1^{pq} = R^qf_*\W^p_{X_A/S_A} \Rightarrow \W^\bullet_{X_A/S_A}$$
degenerates at $E_1$, that $E_1^{pq}$ is a free $A$-module, and that its formation commutes with base change. Moreover, also the abutment $\HH^m(X_A,\W^\bullet_{X_A/S_A})$ is a free $A$-module, and its formation commutes with base change.

In the case of a log toroidal family $f_A: X_A \to S_A$, at least for some bases $A$, we have known this before by \cite[Thm.~1.10]{FFR2021}, but here, the argument is slightly different in that we know that $H^m(X_A,\W^\bullet_{X_A/S_A})$ is free, and that its formation commutes with base change, \emph{prior to the existence} of $f_A: X_A \to S_A$. This hypercohomology is not so much associated with $f_A: X_A \to S_A$ as it is associated with the curved Gerstenhaber calculus \eqref{PV-DR-intro}, and thus with the \emph{deformation problem}. Although we have chosen the specific total differential $d_\phi$ associated with a Maurer--Cartan solution $\phi$ above, the cohomology is independent of $\phi$ in that we can transform $d_\phi$ into $d_{\phi'}$ via an automorphism of \eqref{PV-DR-intro}, and $d_\phi$ is a differential even if $\phi$ does not satisfy the Maurer--Cartan equation, see Lemma~\ref{hypercohom-comparison}.

\section{Unobstructedness of line bundles}

When $X/\kk$ is a smooth and proper Calabi--Yau variety, and $\cL$ is a line bundle on $X$, then a deformation of the pair $(X,\cL)$ is a flat deformation $X_A \to S_A$ of $X/\kk$ together with a line bundle $\cL_A$ on $X_A$ and an isomorphism $\cL_A|_X \cong \cL$. Isomorphism classes of deformations of the pair form a deformation functor 
$$\mathrm{Def}_X(\cL,-): \enspace \mathbf{Art} \to \mathbf{Set}.$$
This deformation functor is unobstructed. As explained in \cite[Rem.~2.6]{Iacono2021defPairs}, this follows from Ran's $T^1$-lifting criterion. However, more is true: the main result of Iacono and Manetti's article \cite{Iacono2021defPairs} is that the dg Lie algebra controlling $\mathrm{Def}_X(\cL,-)$ is homotopy abelian. Their proof is as follows: they construct a new log Calabi--Yau variety $Y = \PP(\cO_X \oplus \cL)$, endowed with the divisorial log structure from two sections $\Delta = \Delta_0 + \Delta_\infty$ of this $\PP^1$-bundle, and then the dg Lie algebra controlling $\mathrm{Def}_X(\cL,-)$ is quasi-isomorphic to the dg Lie algebra controlling $\mathrm{Def}_{(Y,\Delta)}$.

Given a generically log smooth family $f_0: X_0 \to S_0$ with a system of deformations $\D$, and given a line bundle $\cL_0$ on $X_0$, we can consider similarly generically log smooth deformations of type $\D$ of the pair $f_0: (X_0,\cL_0) \to S_0$---they consist of a generically log smooth deformation $f_A: X_A \to S_A$ of type $\D$ together with a line bundle $\cL_A$ on $X_A$ which deforms $\cL_0$. Such deformations form a deformation functor 
$$\mathrm{LD}_{X_0/S_0}^\D(\cL_0,-): \enspace \mathbf{Art}_Q \to \mathbf{Set}.$$
In this situation, we can form a similar $\PP^1$-bundle $p_0: P_0(\cL_0) \to X_0$, where we add the divisorial log structure of $\Delta = \Delta_0 + \Delta_\infty$ to the pull-back log structure from $X_0$ with the procedure described in Chapter~\ref{fiber-bundle-constr-sec}. For $g_0: P_0(\cL_0) \to S_0$, we can construct a new system of deformations $\D(\cL_0)$ by applying the construction of $P_0(\cL_0)$ to the system of deformations $\D$. Then we obtain an isomorphism 
$$\mathrm{LD}_{X_0/S_0}^\D(\cL_0,-) \cong \mathrm{LD}^{\D(\cL_0)}_{P_0(\cL_0)/S_0}$$
of functors of Artin rings. If $f_0: X_0 \to S_0$ is log Calabi--Yau, then $g_0: P_0(\cL_0) \to S_0$ is log Calabi--Yau as well. Thus, for $\mathrm{LD}_{X_0/S_0}^\D(\cL_0,-)$ to be unobstructed in the case where $f_0: X_0 \to S_0$ is log Calabi--Yau, it is sufficient that the characteristic curved (two-sided) Gerstenhaber calculus of $g_0: P_0(\cL_0) \to S_0$ is quasi-perfect. If $f_0: X_0 \to S_0$ is log toroidal, and the system of deformations $\D$ is log toroidal, then $g_0: P_0(\cL_0) \to S_0$ and the system of deformations $\D(\cL_0)$ are log toroidal as well. Thus, in this case, if $f_0: X_0 \to S_0$ is additionally proper, then $\mathrm{LD}_{X_0/S_0}^\D(\cL_0,-)$ is unobstructed by Corollary~\ref{log-BBT-intro}. In other words, the pair $f_0: (X_0,\cL_0) \to S_0$ is unobstructed in this situation:

\begin{thm}
 Let $f_0: X_0 \to S_0$ be a proper lo Calabi--Yau log toroidal family, and let $\cL_0$ be a line bundle on $X_0$. Let $\D$ be a system of deformations which is log toroidal. Then $\mathrm{LD}^\D_{X_0/S_0}(\cL_0,-)$ is unobstructed. In particular, this holds when $f_0: X_0 \to S_0$ is log smooth and $\D$ is the system of log smooth deformations.
\end{thm}

\begin{rem}
 The forgetful map 
 $$\mathrm{LD}_{X_0/S_0}^\D(\cL_0,-) \to \mathrm{LD}_{X_0/S_0}^\D$$
 is in general \emph{not} smooth, even in the situation where we can show that each of the two functors is unobstructed. In other words, sometimes a line bundle $\cL_B$ over $f: X_B \to S_B$ cannot be extended to a given thickening $f': X_{B'} \to S_{B'}$.
\end{rem}
\begin{rem}
 It is currently an open question if quasi-perfectness of \eqref{PV-DR-intro} implies quasi-perfectness of $(PV^{\bullet,\bullet}_{P_0(\cL_0)/\Lambda},DR^{\bullet,\bullet}_{P_0(\cL_0)/\Lambda})$. In practice, this does not seem to be too much of a problem as we may expect that most methods showing the quasi-perfectness of one of them also shows the quasi-perfectness of the other.
\end{rem}

In \cite[Rem.~1.32]{Gross2011}, Gross and Siebert have announced that the polarizing line bundle on $f_0: \check X_0(B,\cP,s) \to S_0$ can be deformed as well over $\mathrm{Spf}\,\kk\llbracket t\rrbracket$ by applying their reconstruction procedure to the total space of the line bundle. At this point, we see that indeed deformations of $f_0: \check X_0(B,\cP,s) \to S_0$ with any line bundle $\cL_0$ are unobstructed, confirming their claim, yet with a totally different method.

\section{Modifications of the log structure}

With Corollary~\ref{log-BBT-intro}, we have a powerful unobstructedness result for log toroidal log Calabi--Yau spaces $f_0: X_0 \to S_0$. However, sometimes we want to deform a space which is not log Calabi--Yau but, say, log Fano. In some cases, we can \emph{modify} the log structure and obtain a log Calabi--Yau. In general, given a line bundle $\cL_0$ on $X_0$ and a section $s_0 \in \cL_0$, we can define a modification $g_0: X_0(s_0) \to S_0$ with the property $\W^d_{X_0(s_0)/S_0} \cong \W^d_{X_0/S_0} \otimes \cL_0$. Thus, when $\cL_0 = (\W^d_{X_0/S_0})^\vee$ is the anti-canonical bundle, then $g_0: X_0(s_0) \to S_0$ is log Calabi--Yau. We have already used a variant of this construction in \cite[\S 6]{FFR2021}. Here, we study a more general version systematically.

Given a generically log smooth family $f: X \to S$ and a line bundle $\cL$, we can add the log structure of the zero section $X \subseteq L$ to the pull-back log structure on $L$ and obtain a log smooth log morphism $p: L(X) \to X$. Then the \emph{modification} of $f: X \to S$ in $s \in \cL$ is given by pulling back the log structure from $L(X)$ via the embedding $s: X \to L$ defined by $s \in \cL$. We write $f \circ h: X(s) \to S$ for this modified log structure. 

The key difficulty is to find conditions on $s \in \cL$ under which this construction is well-behaved. As such a condition, we propose \emph{log regularity} as defined in Definition~\ref{log-qreg-defn} and Definition~\ref{log-reg-defn}.\footnote{Log quasi-regularity and log regularity coincide in the case of a plain generically log smooth family and are different only in the enhanced case.} For us, the starting point in formulating this condition has been \cite[IV,\,Thm.~3.2.2]{LoAG2018}, which gives conditions for a strict closed subscheme in a log smooth scheme to be log smooth itself---in a sense, this is what we want for $X(s)$ inside $f \circ p: L(X) \to S$.

Now let $f_0: X_0 \to S_0$ be a generically log smooth family with a system of deformations $\D$, let $\cL_0$ be a line bundle on $X_0$, and let $s_0 \in \cL_0$ be a log regular section. Then $g_0: X_0(s_0) \to S_0$ is generically log smooth as well. For any deformation $\cL$ of $\cL_0$, a section $s \in \cL$ with $s|_0 = s_0$ is log regular again. After choosing a deformation of $\cL_0$ over the local model $V_{\alpha;A} \to S_A$ together with a section restricting to $s_0$, we obtain a generically log smooth deformation $\tilde V_{\alpha;A} \to S_A$ of $g_0: X_0(s_0) \to S_0$. Up to isomorphism, this deformation does not depend on the choices we made; thus, we obtain a system of deformations $\D(s_0)$ for $g_0: X_0(s_0) \to S_0$, and hence a deformation functor $\mathrm{LD}_{X_0(s_0)/S_0}^{\D(s_0)}$. On the other hand side, we can also consider deformations $f: X_A \to S_A$ of $f_0: X_0 \to S_0$ together with a line bundle $\cL_A$ and a section $s \in \cL_A$ with $s|_0 = s_0$. This gives rise to a deformation functor $\mathrm{LD}^\D_{X_0/S_0}(\cL_0,s_0,-)$. There is an obvious map 
$$\mathrm{LD}_{X_0/S_0}^\D(\cL_0,s_0,-) \to \mathrm{LD}_{X_0(s_0)/S_0}^{\D(s_0)},$$
which is in fact an isomorphism of functors of Artin rings. This ties the deformation theory of $f_0: X_0 \to S_0$ closely to the one of $g_0: X_0(s_0) \to S_0$. In the case where $\cL_0$ is the anti-canonical line bundle, unobstructedness of $\mathrm{LD}^\D_{X_0/S_0}$ is often equivalent to unobstructedness of $\mathrm{LD}^{\D(s_0)}_{X_0(s_0)/S_0}$, see Corollary~\ref{original-modif-unob-equiv}.

We call log regular sections which satisfy the stronger condition in \cite[\S 6]{FFR2021} \emph{log transversal}.\footnote{It is actually an open question if the two conditions are equivalent.} If $f: X \to S$ is log toroidal, and $s \in \cL$ is a log transversal section, then $g: X(s) \to S$ is log toroidal as well. Thus, in the case of a proper log toroidal family family $f_0: X_0 \to S_0$ with a log transversal section $s_0 \in (\W^d_{X_0/S_0})^\vee$, the deformations of $g_0: X_0(s_0) \to S_0$ of type $\D(s_0)$ are unobstructed.

An application of the theory of modifications is \cite[Thm.~1.1]{FFR2021}. In Chapter~\ref{smoothing-nc-sec}, we discuss the following variant.

\begin{thm}[Smoothing normal crossing spaces] Let $V/\CC$ be a proper normal crossing space such that every open stratum $S^\circ \in [\cS_kV]$ is quasi-projective for $k \geq 0$. Suppose that $\omega_V^\vee$ is a globally generated line bundle on $V$, and that $\T^1_V$ is a globally generated line bundle on the double locus $D = V_{sing}$. Then $V$ admits a formal smoothing.
\end{thm}

If $V$ is projective, then the formal smoothing can be algebraized to a smoothing $f: X \to \Spec \CC\llbracket t\rrbracket$. In any case, there is an analytic smoothing $F: \X \to \Delta$ over a small disk which coincides with the formal smoothing up to any prescribed order (but not necessarily for all orders simultaneously). In \cite[Thm.~1.1]{FFR2021}, the assumptions differ slightly. Due to the assumption of a well-behaved section of $\omega_V^\vee$, the global generatedness of $\omega_V^\vee$ is dropped as well as the quasi-projectivity of open strata $S^\circ \in [\cS_0V]$, but then, concerning open strata in $[\cS_kV]$ for $k \geq 1$, the stronger condition that the double locus $D$ is projective is required. The argument is more or less the same, using a more refined version of Bertini's theorem to get the, in our opinion, more natural condition that the strata are quasi-projective.

\section{Enhanced generically log smooth families}

The generically log smooth family in Example~\ref{xy-w2-t2-zw-tu-example} does not have the base change property. Thus, its two-sided Gerstenhaber calculus $\V\,\W^\bullet_{X/S}$ is not well suited to study the family. However, in this example, and in many others as well, we can remedy this situation by choosing subsheaves $\G^p_{X/S} \subseteq \V^p_{X/S}$ and $\A^i_{X/S} \subseteq \W^i_{X/S}$. With this trick, we can enforce compatibility under base change by just \emph{defining} the base change as the subsheaves given by the pull-backs of $\G^p_{X/S}$ and $\A^i_{X/S}$. In general, we may have to relax the condition that they are \emph{subsheaves}, but usually the injectivity of the embedding is preserved as long as we are working with $Q = \NN$. When $\G^p_{X/S}$ and $\A^i_{X/S}$ are chosen, we call the resulting structure an \emph{enhanced generically log smooth family}. In the main text below, we provide the whole theory for enhanced generically log smooth families with often little additional effort compared to the plain case, but in the introduction, we work only with plain generically log smooth families, which are simpler and, as of now, more important.\footnote{In fact, we do not yet know any example of an enhanced generically log smooth family where we can prove quasi-perfectness of the characteristic curved (two-sided) Gerstenhaber calculus, or rather, we know the examples where we wish to have this property, but it has so far evaded our attempts to prove quasi-perfectness.}

\section{Prerequisites}

Throughout the book, we assume that the reader is familiar with algebraic geometry at the level of Hartshorne's \cite{Hartshorne1977} (and some standard topics not discussed there), and logarithmic geometry at the level of Ogus' \cite{LoAG2018}, especially Chapters~III and IV. A basic understanding of complex geometry as in Grauert's and Remmert's \cite{GrauertRemmert1984} is helpful but not needed except in Chapter~\ref{perfect-G-calc-sec}. We also assume acquaintance with infinitesimal deformation theory as in Schlessinger's short article \cite{Schlessinger1968} and the books \cite{Sernesi2006} or \cite{Hartshorne2010}, but most of the material there is not necessary for this book. The reader not acquainted with F.~Kato's log smooth deformation theory should read \cite{Kato1996} first. We also assume the theory of log toroidal families as contained in our article \cite{FFR2021}, of which a more detailed exposition can be found in the author's PhD thesis \cite{FeltenThesis}. Except for basic acquaintance with the concept of log toroidal families, the reader may consult these two works as needed. Prior knowledge of the Gross--Siebert program like \cite{GrossSiebertI,GrossSiebertII,Gross2011} is helpful as well, in particular to put results in context, but we do not assume that the reader is familiar with these works.

\section{Structure of the book}

In Chapter~\ref{alg-str-sec}, we define the algebraic structures which are the main subject of our study. In particular, we define Gerstenhaber and Batalin--Vilkovisky calculi as well as their bigraded and curved counterparts.

In Chapter~\ref{gauge-trafo-sec}, we study gauge transforms of our structures in some generality.

In Chapter~\ref{extended-MC-eqn-sec}, we study the classical extended Maurer--Cartan equation of a curved Lie algebra, in particular a curved Gerstenhaber calculus, as well as the semi-classical extended Maurer--Cartan equation of a curved Batalin--Vilkovisky calculus. We show that solutions of these equations parametrize the different ways of turning the curved structures into a differential graded one by modifying the predifferential $\bar\partial$. Furthermore, we define perfectness and quasi-perfectness for curved Gerstenhaber calculi and related structures.

In Chapter~\ref{two-abstract-unobstructed-sec}, we prove the two abstract unobstructedness theorems Theorem~\ref{first-abs-unob-thm-intro} and Theorem~\ref{second-abs-unob-thm-intro} by following the argument in \cite{ChanLeungMa2023}. In doing so, we also define a quantum extended Maurer--Cartan equation and study liftings of its solutions first.

In Chapter~\ref{gen-log-smooth-sec}, we review the basic theory of generically log smooth families, in particular the notions of being log Gorenstein and having the base change property. We give the definition of a system of deformations $\D$ for $f_0: X_0 \to S_0$. We also define enhanced generically log smooth families and their systems of deformations.

In Chapter~\ref{gen-log-sm-vector-bundle}, we review deformations of a pairs consisting of an (enhanced) generically log smooth family $f_0: X_0 \to S_0$ and a vector bundle $\E_0$ on $X_0$. We briefly investigate the relationship between deformations of $(X_0,\E_0)$, deformations of $(X_0,\mathrm{det}(\E_0)$, and deformations of $X_0$.

In Chapter~\ref{geom-fam-P-alg-sec}, we introduce a deformation theory of geometric families of algebraic structures. The most important example of this is obtained by replacing a generically log smooth deformation $f_A: X_A \to S_A$ with the deformation on the level of two-sided Gerstenhaber calculi. In doing so, we may forget about the log structures and just keep the Gerstenhaber calculi.

In Chapter~\ref{char-alg-constr-sec}, we give the very important construction of a characteristic curved algebra associated with a deformation problem of algebraic structures. This construction essentially relies on the Thom--Whitney resolution. It allows to form a single algebraic structure which controls the deformation problem and, in the case of Gerstenhaber calculi, computes the Hodge numbers. It is likely that this structure is closely tied to other important invariants of $f_0: X_0 \to S_0$ as well, like Hodge structures, a Frobenius structure similar to \cite{BarannikovKontsevich1998},\footnote{Cf.~also to \cite[\S 6.2]{ChanLeungMa2023}, where a semi-infinite Hodge structure is constructed.} or even Gromov--Witten invariants of the mirror, although the characteristic curved algebra alone is certainly not sufficient to compute these invariants.

In Chapter~\ref{perfect-G-calc-sec}, we show that the characteristic curved Gerstenhaber calculus is indeed perfect if $f_0: X_0 \to S_0$ is proper and log toroidal, and all deformations in a system of deformations $\D$ are log toroidal as well. The argument in this chapter relies on analytification, so we work over $\kk = \CC$. Through the Lefschetz principle, the result should be true over all fields $\kk$ of characteristic $0$.

In Chapter~\ref{fiber-bundle-constr-sec}, we study deformations of pairs $f_0: (X_0,\cL_0) \to S_0$ with a line bundle $\cL_0$ on $X_0$.

In Chapter~\ref{alg-degen-sec}, we apply the results of Chapter~\ref{fiber-bundle-constr-sec} to an ample line bundle on a projective $f_0: X_0 \to S_0$ over the standard log point $S_0 = \Spec(\NN \to \CC)$ to enhance a formal degeneration to an algebraic degeneration over $\Spec \CC\llbracket t\rrbracket$. We briefly discuss the geometric properties of such an algebraic degeneration. 

In Chapter~\ref{log-modif-sec}, we present our theory of modifications $g: X(s) \to S$ of log morphisms $f: X \to S$, which often allows us to turn log Fano spaces into log Calabi--Yau spaces in a well-behaved way. As an application, we discuss the smoothing of normal crossing spaces.

\section{About additional material}

Throughout the main text, we use a number of results for which we have no appropriate references, or on which we wanted to elaborate for the reader's convenience.

In Chapter~\ref{fiber-total-space-sec}, given a morphism $f: X \to S$, we study the relationship between that a coherent sheaf $\F$ on $X$ is reflexive, and that its fibers $\F_s$ are reflexive.

In Chapter~\ref{multilin-diff-op-sec}, we introduce the notion of a \emph{multilinear differential operator} of finite order
$$\mu: \F_1 \times ... \times \F_n \to \G,$$
generalizing the case of differential operators $D: \F \to \G$ between two quasi-coherent sheaves $\F$ and $\G$. As it turns out, this theory is quite well-behaved. Our primary interest is in the Schouten--Nijenhuis bracket $[-,-]$ and the Lie derivative $\cL_{-}(-)$ as bilinear differential operators.

In Chapter~\ref{spectral-seq-sec}, we prove the criterion for $E_1$-degeneration of spectral sequences over $\kk$, which is alluded to in \cite[Defn.~4.13]{KKP2008}. We also give a more precise criterion for partial degeneration at $E_1$, and we study degeneration of spectral sequences over Artinian local rings $A$.

In Chapter~\ref{analyt-sec}, we study the analytification of schemes $X$ of finite type over $\CC$. In particular, we show that all differential operators of first order between quasi-coherent algebraic sheaves can be analytified.

In Chapter~\ref{toroidal-cr-sp-sec}, we discuss the notion of a toroidal crossing space $(V,\cP,\bar\rho)$, which goes back to \cite{SchroerSiebert2006}, was essentially developed in \cite{GrossSiebertI}, and briefly defined in \cite{FFR2021}. This seems to be the first introductory discussion of this notion available, and some of the results seem to have not yet been stated anywhere else (for example that $V$ has slc singularities).

In Chapter~\ref{elem-GS-type-sec}, we discuss log toroidal families. We introduce the notion of a log toroidal family of (elementary) Gross--Siebert type for those whose singularities are controlled by local models discussed by Gross--Siebert. For the reader's convenience, we also discuss the proof of the important fact that any two deformations which are log toroidal of elementary Gross--Siebert type (\emph{divisorial} deformations in the terminology of Gross--Siebert) are locally isomorphic.

\section{What is new in this book?}

Many results in this book have been known before, in some form or another. Here is what is new in this book:

\begin{enumerate}[label=(\arabic*)]
 \item the precise definitions of a Gerstenhaber calculus and Batalin--Vilkovisky calculus, in all their variants;
 \item the Second Abstract Unobstructedness Theorem~\ref{second-abstract-unob-thm};
 \item the criterion for partial degenerations of spectral sequences in Proposition~\ref{spec-seq-alternative-formulation}; 
 \item the unobstructedness results for line bundles on log toroidal log Calabi--Yau spaces in Chapter~\ref{fiber-bundle-constr-sec}; in particular, the deformations of the line bundle claimed in \cite[Rem.~1.32]{Gross2011} exist indeed, albeit that we have a different proof than the one announced there;
 \item as a consequence of the unobstructedness of a space with a line bundle, the existence of actual algebraic families over $\Spec \kk\llbracket t\rrbracket$ in the case where the central fiber is a projective log Calabi--Yau space, not only formal families;
 \item the precise formulation of changing the log structure of a log Fano variety to obtain a log Calabi--Yau one in Chapter~\ref{log-modif-sec}; a weaker and less precise version of this has been given in \cite{FFR2021};
 \item a theory of multilinear differential operators in algebraic geometry, given in Chapter~\ref{multilin-diff-op-sec};
 \item a weakening of the notion of generically log smooth family to the notion of \emph{enhanced generically log smooth family}, which has better stability properties under base change;
 \item the first comprehensive presentation of the theory of toroidal crossing spaces in our sense, given in Chapter~\ref{toroidal-cr-sp-sec};
 \item an interpretation of the sheaf $\B$ occurring in \cite[\S 2.2]{GrossSiebertII}, given in Chapter~\ref{elem-GS-type-sec}.
\end{enumerate}

\section{Related works}

\textbf{Logarithmic geometry.} Logarithmic geometry was developed in the late 1980s by Fontaine and Illusie after structures with logarithmic poles had appeared in various contexts, most notably crystalline cohomology, the original context of logarithmic geometry. The first published account is by K.~Kato in \cite{kkatoFI} in 1989.\footnote{According to private communication with Luc Illusie, he and Jean-Marc Fontaine conceived the concept of a log structure on the train from Paris to Oberwolfach in July 1988. There, Illusie quickly wrote four or five pages sketching the construction and handed them to Kazuya Kato who seemed interested. A few months later, Kato surprised them by presenting his preprint of \cite{kkatoFI}, containing an elaborate theory of log structures with additional concepts like the chart of a log structure. While Illusie did not further pursue log structures in the meantime, Fontaine had started to work out the theory himself. He completely abandoned the project after he received Kato's preprint. Indeed, he never published on the subject of log geometry.} A parallel theory, nowadays called Deligne--Faltings log structures, appears around the same time in \cite{Faltings1990} by Faltings. The first highlight of logarithmic geometry is Tsuji's celebrated proof of the $C_{st}$-conjecture in \cite{Tsuji1999} in 1999, which contributed greatly to the popularity of the method.

In 2001, Olsson gives in \cite{Olsson2001} a new stack-theoretic interpretation of logarithmic geometry. Starting from there and from F.~Kato's \cite{Kato2000}, log structures on stacks and stacks over the category of log schemes have been the interest of mostly recent research, see \cite{Olsson2007,Gillam2012,GrossSiebert2013,Chen2014,AbramovichChen2014,Shentu2016,ACGS2020}. In 2018, Talpo and Vistoli developed in \cite{TalpoVistoli2018} yet another approach to logarithmic geometry via infinite root stacks. This approach is particularly suited to define a category of quasi-coherent sheaves on a log scheme which reflects the log geometry and not just the classical geometry of the underlying scheme.

\vspace{\baselineskip}

\noindent \textbf{Semistable degenerations.} The simplest and best-known case of a log smooth family is a semistable family. Interest in semistable families seems to arise around 1969 in the one-dimensional case with the Deligne--Mumford compactification $\overline\M_g$ of the stack of $\M_g$ of curves of genus $g$ in \cite{DeligneMumford1969}; this compactification is achieved precisely by adding curves with nodal singularities. In \cite{KKMS1973}, it was proven in 1973 that any degeneration $f: X \to C$ of varieties of arbitrary dimension over a smooth curve $C$ can be brought in semistable form after a finite base change; this is the well-known semistable reduction theorem. 

For any (projective) holomorphic family $f: X^* \to \Delta^*$ over a small punctured disk, Schmid has constructed in 1973 in \cite{Schmid1973} a limiting mixed Hodge structure on the cohomology $H^k(X_t,\CC)$ of the general fiber. If the family can be completed to a semistable degeneration $f: X \to \Delta$, the central fiber also carries a canonical mixed Hodge structure by work of Deligne. However, this mixed Hodge structure is not the same as Schmid's limiting mixed Hodge structure. In 1976, Steenbrink has given in \cite{Steenbrink1976} an interpretation of the limiting mixed Hodge structure in terms of the geometry of the semistable family $f: X \to \Delta$. This work already contains much of the later log differential calculus of semistable degenerations.

In \cite{Kulikov1977}, Kulikov classified in a 1977 article (in Russian) all possible central fibers of a semistable degeneration of K3 surfaces;\footnote{Kulikov's claim on the classification of degenerations of Enriques surfaces in the same article has been found to be false.} a more rigorous account has been published by Persson and Pinkham in \cite{PerssonPinkham1981} in 1981. Two years later, Friedman has shown in \cite{Friedman1983} that indeed all possible central fibers do occur in an analytic semistable degeneration of K3 surfaces. This is the first important smoothing result which turns out to be a special case of the unobstructedness of the logarithmic deformation functor in the Calabi--Yau case. Another important insight of \cite{Friedman1983} is that a normal crossing space $V$ which occurs as the central fiber of a semistable degeneration must satisfy $\T^1_V \cong \cO_{V^{sing}}$, i.e., $V$ is $d$-semistable. It was observed by Persson and Pinkham in the same year in \cite{PerssonPinkham1983} that $d$-semistability is not enough, i.e., that there are $d$-semistable normal crossing spaces which do not occur as the central fiber of a semistable degeneration.

In 1994, Friedman's smoothing result for normal crossing surfaces has been generalized to arbitrary dimension by Kawamata and Namikawa in \cite{KawamataNamikawa1994}. They use an early version of the language of logarithmic deformation theory; the argument essentially relies on flat deformation theory and the analysis of a specific mixed Hodge complex, which is related to Steenbrink's construction of the limiting mixed Hodge structure in a semistable degeneration. In the same year, in \cite{Steenbrink1995}, Steenbrink himself, too, gives a construction of a mixed Hodge complex on what he calls a logarithmic deformation. By this, he means a normal crossing space together with the structure of a log smooth log morphism to the standard log point; it generalizes his earlier construction of the limiting mixed Hodge structure as well.

The smoothing result of Kawamata--Namikawa has been applied and refined in a couple of instances. Already in 1990, Fujita has classified the possible semistable degenerations of del Pezzo surfaces in \cite{Fujita1990}, and in 2007, Kachi has applied the result of Kawamata--Namikawa in \cite{Kachi2007} to show that all semistable del Pezzo surfaces occur indeed as the central fiber of a semistable degeneration. In 2010, Lee has computed in \cite{Lee2010} the topological invariants of many smoothings arising from Kawamata--Namikawa. In 2020, Lee constructs in \cite{Lee2020} new examples of normal crossing Calabi--Yau threefolds which are smoothable by Kawamata--Namikawa, giving rise to smooth Calabi--Yau threefolds with topological invariants that no previously known Calabi--Yau threefold has. In 2021, Sano shows in \cite{Sano2021} that, for every $N \geq 4$, there are infinitely many non-K\"ahler Calabi--Yau $N$-folds $X$ with pairwise different second Betti number $b_2(X)$. In 2023, Hashimoto and Sano show in \cite{HashimotoSano2023} that the same is true for $N = 3$, thus giving infinitely many different topological types of Calabi--Yau threefolds, albeit that none of them is projective. Infinitely many non-K\"ahler Calabi--Yau threefolds with $b_2(X) = 0$ and  pairwise different third Betti number $b_3(X) = 2h$ for $h \geq 103$ have already been constructed by Friedman in \cite{Friedman1991} in 1991, attributing the construction to Clemens who has used a similar construction in \cite{Clemens1983}. In 1994, Lu and Tian give a construction with $h = 25$ in \cite{LuTian1994}.

In 2023, Doi and Yotsutani give in \cite{DoiYotsutani2023} an entirely different approach to smoothing normal crossing spaces in differential geometry.

\vspace{\baselineskip}

\noindent\textbf{Higher methods in deformation theory.} Every dg Lie algebra $L^\bullet$ gives rise to an associated deformation functor $\mathrm{Def}_{L^\bullet}$, and many deformation problems admit a dg Lie algebra which controls the deformation problem in this way. For example, the holomorphic smooth deformation functor is governed by the Kodaira--Spencer dg Lie algebra $KS^\bullet_X$, i.e., the Dolbeault resolution of the holomorphic tangent sheaf $\Theta^1_X$. Similarly, every $L_\infty$-algebra $L^\bullet$ gives rise to a deformation functor, which depends only on the homotopy class of $L^\bullet$. Since every $L_\infty$-algebra is homotopy equivalent to a dg Lie algebra, proven for example by Kriz and May in \cite{KrizMay1995} in 1995, the deformation problems governed by some $L_\infty$-algebra are the same as the deformation problems governed by some dg Lie algebra.

$L_\infty$-algebras feature prominently in Kontsevich's celebrated proof of the formality conjecture in \cite{Kontsevich2003}, published in 2003: The dg Lie algebras $T_{poly}(X)$ and $D_{poly}(X)$ are homotopy equivalent via an $L_\infty$-morphism. An important technical result on $L_\infty$-algebras is the homotopy transfer, which appears clearest and most general in the 2007 work \cite{FiorenzaManetti2007} of Fiorenza and Manetti, but is much older, going back to a result of Kadeishvili in \cite{Kadeishvili1982} on $A_\infty$-algebras. This result essentially states that, when a complex is a homotopy retract of an $L_\infty$-algebra, then there is an induced homotopy equivalent $L_\infty$-structure on that complex.

It has long been a well-established philosophy that \emph{every} deformation problem should be governed by dg Lie algebras. Indeed, for a large number of deformation problems, such a dg Lie algebra is known. This philosophy has been established independently as a theorem by Pridham in \cite{Pridham2010} in 2010 and Lurie in \cite{LurieX} in 2011: They define an $\infty$-category of formal moduli problems as an $\infty$-category of certain functors of $\mathbb{E}_\infty$-algebras, and this $\infty$-category is equivalent to the $\infty$-category obtained from dg Lie algebras by localization of quasi-isomorphisms. A nice survey has been given by Calaque and Grivaux in \cite{CalaqueGrivaux2021}. A new conceptual approach to the Lurie--Pridham theorem has been given by Le Grignou and Roca i Lucio in \cite{GrignouLucio2023}.

Given a deformation problem, its deformation functor maps a local Artin ring $A$ to the \emph{set} of isomorphism classes. Similarly, given a dg Lie algebra $L^\bullet$, its deformation functor maps $A$ to the \emph{set} of equivalence classes. There is, however, an obvious $1$-categorical enhancement: The deformation groupoid (over $A$) of a deformation problem has as objects deformations over $A$, and as morphisms isomorphisms between deformations. Similarly, the Deligne groupoid $\C(L^\bullet)$ of a dg Lie algebra $L^\bullet$ has as objects Maurer--Cartan solutions, and as morphisms gauge transforms that send one solution to the other. Hinich has proven in \cite{Hinich1997} in 1997 that the two groupoids are equivalent for certain deformation problems. In the course of doing so, he defines a Kan complex $\mathrm{MC}_\bullet(L^\bullet)$ out of the semisimplicial dga $(A_{PL})_\bullet$ of polynomial differential forms on simplices. Then the Deligne groupoid $\C(L^\bullet)$ is isomorphic to the Poincar\'e groupoid of $\mathrm{MC}_\bullet(L^\bullet)$, which is the category obtained from the left adjoint of the nerve functor $N: \mathrm{Cat} \to \mathrm{sSet}$. Since $\mathrm{MC}_\bullet(L^\bullet)$ is a Kan complex, we may consider it as an $\infty$-categorical realization of the deformation problem governed by $L^\bullet$.

The Kan complex $\mathrm{MC}_\bullet(L^\bullet)$ is not equivalent to the nerve $N_\bullet\C(L^\bullet)$ of the Deligne groupoid. In 2009, Getzler has constructed in \cite{Getzler2009} in the generality of $L_\infty$-algebras a subcomplex $\gamma_\bullet(L^\bullet) \subseteq \mathrm{MC}_\bullet(L^\bullet)$ such that $\gamma_\bullet(L^\bullet)$ is isomorphic to $N_\bullet\C(L^\bullet)$ if $L^\bullet$ is a nilpotent dg Lie algebra concentrated in degrees $[0,\infty)$, which is the relevant case in deformation theory. In 2020, Robert-Nicoud and Vallette have given a more explicit and more general interpretation of $\gamma_\bullet(L^\bullet)$ in \cite{RobertVallette2020}.

The algebraic variant of the Kodaira--Spencer dg Lie algebra $KS^\bullet_X$, which is the classical counterpart of the global logarithmic deformation theory presented in this book, takes a Thom--Whitney resolution $L_X^\bullet$ of $\Theta^1_X$ as the dg Lie algebra that controls $\mathrm{Def}_X$. Then a canonical homotopy retract to the \v{C}ech complex $\check\C(\U,\Theta^1_X)$ turns the latter into an $L_\infty$-algebra controlling $\mathrm{Def}_X$ via the homotopy transfer. The idea to use the Thom--Whitney resolution of $\Theta^1_X$ to control the smooth deformation functor goes at least back to the 1997 articles \cite{HinichSchechtman1997I,HinichSchechtman1997II} of Hinich and Schechtman. In these articles, the construction is used to describe the hull of $\mathrm{Def}_X$ in terms of $\Theta^1_X$ in a case where $\mathrm{Def}_X$ is assumed to be unobstructed. The article \cite{Hinich1997} of Hinich discussed above in the context of $\mathrm{MC}_\bullet(L^\bullet)$ then also shows that the Thom--Whitney resolution of $\Theta^1_X$ controls $\mathrm{Def}_X$. Iacono and Manetti achieve the present form of the theory in \cite{AlgebraicBTT2010} in 2010. Manetti and others have further elaborated on various parts of the theory in \cite{Manetti2009,FiorenzaManetti2009,FiorenzaManettiMartinengo2012,ManettiMeazzini2019,ManettiMeazzini2020}. The Thom--Whitney resolution has a long history; the variant which we use is the one used in Iacono--Manetti \cite{AlgebraicBTT2010}; it has been developed by Navarro Aznar in \cite{AznarHodgeDeligne1987} in 1987. Hinich and Schechtman give their own work \cite{HinichSchechtman1987}, also from 1987, as the source of the Thom--Whitney construction they use in \cite{HinichSchechtman1997I,HinichSchechtman1997II}.

\vspace{\baselineskip}

\noindent\textbf{Non-logarithmic Bogomolov--Tian--Todorov theorems.} The BBT theorem, asserting that infinitesimal deformations of a compact Calabi--Yau manifold are unobstructed, was first proven in a special case by Bogomolov in \cite{Bogomolov1978} in 1978. The general case has been achieved independently by Tian in \cite{Tian1987} in 1987 and by Todorov in \cite{Todorov1989} in 1989. Both proofs use differential geometry. The first algebraic proof was given by Ran in \cite{Ran1992} in 1992 via the so-called $T^1$-lifting criterion, in the case of the complex numbers as the base field. In the same year, Kawamata has generalized the algebraic proof to an arbitrary base field in \cite{Kawamata1992}. In 1999, Fantechi and Manetti have generalized in \cite{FantechiManetti1999} the $T^1$-lifting criterion by weakening the necessary assumptions.

In the case of compact complex manifolds, the infinitesimal deformations are governed by the Kodaira--Spencer dg Lie algebra $KS^\bullet_X$, i.e., the Dolbeault resolution of the holomorphic tangent sheaf. In 1990, Goldman and Millson have shown in \cite{GoldmanMillson1990} that this dg Lie algebra is homotopy abelian in the Calabi--Yau case, a statement which implies the unobstructedness of the deformation functor. Their proof relies on the old trick of \cite{DGMS1975} to show that $KS^\bullet_X$ is quasi-isomorphic to the $\partial$-cohomology of the Dolbeault resolution $(A^{n - 1,\bullet}_X,\bar\partial)$, which shows formality, and then the (Bogomolov--)Tian--Todorov formula shows that this $\partial$-cohomology is an abelian dg Lie algebra. The proof of unobstructedness of the extended deformation functor in Barannikov--Kontsevich \cite{BarannikovKontsevich1998} mentioned above relies on the same method.

In the algebraic case, Iacono and Manetti show the unobstructedness of the dg Lie algebra $L^\bullet_X$, constructed via the Thom--Whitney resolution, in \cite{AlgebraicBTT2010} in 2010. Their proof relies on the theory of period maps and Cartan homotopies developed by Fiorenza and Manetti in \cite{FiorenzaManetti2009}. Iacono generalizes the method to give an abstract unobstructedness theorem in \cite{AbstractBTT2017} in 2017, giving a very explicit criterion for a dg Lie algebra to be homotopy abelian. Iacono also uses this method in \cite{Iacono2015} in 2015 to show homotopy abelianity for a pair $(X,D)$ of a smooth variety $X$ with a divisor $D \subset X$, and in \cite{Iacono2021defPairs} in 2021 (with Manetti) to show homotopy abelianity for a pair $(X,L)$ of a smooth variety $X$ and a line bundle $L$. In \cite{IaconoManetti2019} in 2019, they study deformations of pairs $(X,\F)$ where $\F$ is a coherent sheaf.

The argument of Goldman--Millson, in particular before the background of extended deformations described above, suggests to study Batalin--Vilkovisky algebras instead of dg Lie algebras. In 2008, Terilla gives in \cite{Terilla2008} a smoothness criterion for the deformation functor associated with a BV algebra in terms of a quantization thereof. In the same year, Katzarkov, Kontsevich, and Pantev give in \cite{KKP2008} a new perspective on the homotopy abelianity of the dg Lie algebra inside a (dg) BV algebra; in particular, they define the degeneration property. These two articles lead up to the logarithmic variants pioneered in Chan--Leung--Ma \cite{ChanLeungMa2023}, which we discuss below.

\vspace{\baselineskip}

\noindent\textbf{Formal Frobenius manifolds and Batalin--Vilkovisky algebras.} Given a Calabi--Yau manifold $\check X$, its genus-$0$ Gromov--Witten invariants can be encoded in the quantum cohomology ring, which in turn may be viewed as a structure of a formal Frobenius supermanifold on the formal stalk $\check\M$ at $0$ of the total cohomology ring $H^\bullet = H^\bullet(\check X,\Omega^\bullet_{\check X})$.\footnote{We are ignoring here that, in general, one has to work with coefficients in the Novikov ring.} After forgetting the metric/pairing, the latter is equivalent to a sequence of graded-symmetric products $m_n: (H^\bullet)^{\otimes n} \to H^\bullet$ of degree $2(2 - n)$ for $n \geq 2$, satisfying certain associativity relations. This is what Getzler calls a hypercommutative algebra in \cite{Getzler1995} in 1995. 

In the article \cite{Kontsevich1995} of 1995, when expounding the foundational ideas of homological mirror symmetry, Kontsevich suggests that this structure of a formal Frobenius supermanifold on $\check\M$ should correspond to a structure of a formal Frobenius supermanifold on the formal stalk $\M$ of the total Hochschild cohomology $H^\bullet(X,\bigwedge^\bullet\Theta^1_X)$ of the mirror $X$, and that $\M$ should be considered as an extended moduli space of $X$, extending the interpretation of $H^1(X,T_X)$ as the tangent space of the smooth deformation functor $\mathrm{Def}_X$. Barannikov and Kontsevich give in \cite{BarannikovKontsevich1998} in 1998 a construction of such a formal Frobenius supermanifold on $\M$ by studying the differential BV algebra $\mathbf{t} = \bigoplus_{p,q} \Gamma(X,\bigwedge^q\overline T^*_X \otimes \bigwedge^pT_X)$ of polyvector fields; this formal Frobenius supermanifold conjecturally corresponds to the quantum cohomology of the mirror $\check X$. In \cite{Manin1999}, Manin gives, among many other things, an expanded exposition of the construction of the formal Frobenius manifold. The construction has been further investigated, in particular from the perspective of Hodge theory, by Barannikov in \cite{Barannikov1999,Barannikov2001,Barannikov2002}. Merkulov defines in \cite{Merkulov1999} in 1999 a structure of an $A_\infty$-algebra on the kernel of the BV operator on $\mathbf{t}$; he gives in \cite{Merkulov2000} in 2000 an extension of the construction of the Frobenius manifold to the general case where $X$ needs no longer be Calabi--Yau, with values in $F_\infty$-manifolds. The construction of Barannikov--Kontsevich may be viewed as endowing the cohomology $H^\bullet(A^\bullet,d)$ of any dg Batalin--Vilkovisky algebra $A^\bullet$ with the structure of a hypercommutative algebra. This viewpoint has been further expanded by Losev and Shadrin in \cite{Losev2007}, and by Park in \cite{Park2007}, both in 2007.

Commutative $BV_\infty$-algebras have been first defined by Kravchenko in \cite{Kravchenko2000} in 2000, partially inspired by Barannikov--Kontsevich. A canonical notion of (not necessarily commutative) homotopy BV algebra has been provided in 2012 by G\'alvez-Carrillo, Tonks, and Vallette in \cite{Vallette2012} by taking the Koszul resolution of the $BV$-operad. In 2013,  Drummond-Cole and Vallette have given in \cite{Vallette2013} an explicit description of the minimal model of this operad; an algebra over this operad is called a skeletal homotopy Batalin--Vilkovisky algebra. This allows them to extend the construction of Barannikov--Kontsevich: Every dg Batalin--Vilkovisky algebra $A^\bullet$ which is endowed with so-called non-commutative Hodge--de Rham degeneration data gives rise to a structure of \emph{homotopy} hypercommutative algebra on the cohomology $H^\bullet(A^\bullet,d)$. Here, a homotopy hypercommutative algebra is defined via the Koszul resolution of the operad of hypercommutative algebras. Since $H^\bullet(A^\bullet,d)$ has trivial differential, this homotopy hypercommutative algebra has an underlying hypercommutative algebra, and this hypercommutative algebra is precisely the one constructed by Barannikov--Kontsevich. Unlike in the construction of Barannikov--Kontsevich, the dg Batalin--Vilkovisky algebra $A^\bullet$ can be reconstructed up to homotopy from the homotopy hypercommutative algebra $H^\bullet(A^\bullet,d)$, again by \cite{Vallette2013}. The results have been strengthened and extended by Dotsenko, Shadrin, and Vallette in \cite{Dotsenko2015}. As a consequence of results by Koroshkin, Markarian, and Shadrin in \cite{Khoroshkin2013}, there is also a hypercommutative structure on the dg BV algebra $A^\bullet$ (endowed with additional data), which is homotopy equivalent to the homotopy hypercommutative structure on its cohomology. In 2023, Cirici and Horel prove in \cite{Cirici2023} that, for a Calabi--Yau manifold, the Dolbeault resolution of polyvector fields is formal as a hypercommutative algebra, i.e., homotopy equivalent to its cohomology. Thus, in the case of Calabi--Yau manifolds, the hypercommutative algebra structure on the Hochschild cohomology looses no homotopical information compared to the dg BV algebra structure on polyvector fields. In the case of surfaces and threefolds, this has been proven before by Drummond-Cole in \cite{Drummond2014}. This result is particularly interesting for us because it suggests that the Batalin--Vilkovisky calculi considered in this book may be interpreted as an incarnation of the $B$-model in Barannikov--Kontsevich mirror symmetry.

In 2013, Braun and Lazarev give in \cite{BraunLazarev2013} a version of the theory in Katzarkov--Kontsevich--Pantev (discussed above) for commutative $BV_\infty$-algebras. They define the degeneration property in this context, which implies the homotopy abelianity of the associated $L_\infty$-algebra. This degeneration property is, on the one hand side, a generalization of Katzarkov--Kontsevich--Pantev's degeneration property, and on the other hand side a variant of the degeneration properties considered by Vallette and others in \cite{Vallette2012,Vallette2013,Dotsenko2015}. There has been some additional recent interest in commutative $BV_\infty$-algebras, see \cite{BashkirovVoronov2017,MarklVoronov2017,Voronov2020}. 

A similar structure as the $A_\infty$-algebra structure defined by Merkulov on the kernel of the BV operator also occurs in the recent work \cite{MandelRuddat2023} of Mandel and Ruddat, where tropical quantum field theories are defined, and two new methods to compute the multiplicities of tropical curves are provided.

\vspace{\baselineskip}

\noindent\textbf{Logarithmic deformation theory.} In 1996, F.~Kato gives in \cite{Kato1996} the modern formulation of log smooth deformation theory as a functor from $\mathbf{Art}_Q$ to $\mathbf{Set}$. The notion of log smoothness itself has been already introduced by K.~Kato in the seminal paper \cite{kkatoFI} mentioned above, but we find a fully developed deformation theory for the first time in \cite{Kato1996}. Two years later, in 1998, F.~Kato developed a more general theory of functors of \emph{log} Artin rings, which does not seem to have been in use since then. In 2000, F.~Kato comes in \cite{Kato2000} back to the original motivation of $\overline\M_g$, giving an interpretation of $\overline\M_g$ as the log stack of log smooth curves.

In 2015, Tziolas has shown in \cite{Tziolas2015} that a $d$-semistable Fano variety $X_0$ is smoothable in a semistable family by showing that the obstruction space $H^2(X_0,\Theta^1_{X_0/S_0})$ vanishes, applying F.~Kato's log smooth deformation theory. In 2020, Shentu has treated in \cite{Shentu2020} the more general case of semistable Fano varieties over the log point $\Spec (\NN^r \to \kk)$; such families are also known as polystable families, following the terminology of Berkovich in \cite{Berkovich1999}. 

Iacono's article \cite{Iacono2015} on the deformations of pairs $(X,D)$ with a divisor $D \subset X$, and Iacono and Manetti's article \cite{Iacono2021defPairs} on deformations of pairs $(X,L)$ with a line bundle also rely on logarithmic deformation theory over the trivial log point; namely, in both articles, the compactifying log structure of a divisor is considered.

In the case of an snc divisor $D \subset X$, logarithmic deformations need to be distinguished from deformations of the closed immersion $D \subset X$ where both $X$ and $D$ are allowed to vary. If $(X,D)$ is a log CY pair, the first are always unobstructed, but the latter can be obstructed, as shown by Petracci, Robins, and the author in \cite{FPR2022}.

\vspace{\baselineskip}

\noindent\textbf{The Gross--Siebert program.} The Gross--Siebert program is an explicit constructive approach to mirror symmetry via degenerations to spaces which are glued from toric varieties. Mirror pairs of these so-called toric log Calabi--Yau spaces admit a version of dual Lagrangian torus fibrations to an integral tropical manifold $B$. The original idea goes back to Batyrev's mirror symmetry of \cite{Batyrev1994} from 1994, and the more general Batyrev--Borisov mirror symmetry of \cite{BatyrevBorisov1996} from 1996. In the foundational article \cite{GrossSiebertI} from 2006, Gross and Siebert describe the construction of the degenerate Calabi--Yau spaces $X_0(B,\cP,s)$ and $\check X_0(B,\cP,s)$ as well as a log structure on these spaces and a mirror construction of $(B,\cP,\varphi)$, the discrete Legendre transform $(\check B,\check\cP,\check\varphi)$. In the sequel \cite{GrossSiebertII} of 2010, they describe the deformation theory of these spaces. The first highlight of the Gross--Siebert program is \cite{Gross2011} from 2011, where Gross and Siebert construct infinitesimal deformations of $\check X_0(B,\cP,s)$ up to arbitrary order via an intricate algorithm. In \cite{Gross2005}, Gross explains how Batyrev--Borisov duality can be considered as a special case of their construction. There has been tremendous interest in toric degenerations since.
In 2006, Nishinou and Siebert use the degeneration in \cite{NishinouSiebert2006} to count curves on toric varieties. In 2010, Ruddat generalizes in \cite{Ruddat2010} the computation of certain Hodge numbers in the construction. In the same year, Nakayama and Ogus give with \cite{Nakayama2010} a study on properties of relatively log smooth morphisms, the type of log singularity occurring in the Gross--Siebert construction. In 2014, Casta\~no-Bernard and Matessi study conifold transitions in this context in \cite{Matessi2014}. In 2015, Gross, Hacking, and Keel give a variant of toric degenerations for Looijenga pairs $(Y,D)$ in \cite{GrossHackingKeel2015}. In \cite{Prince2018}, Prince applies the Gross--Siebert algorithm to study Fano surfaces, similar to Example~\ref{Fano-3fold-exa}; in \cite{Prince2021}, he constructs new Calabi--Yau threefolds via the Gross--Siebert algorithm. In \cite{Arguz2021}, Arg\"uz studies the real as opposed to complex geometry of toric degenerations. In \cite{RuddatSiebert2020}, Ruddat and Siebert show that the deformations constructed via the Gross--Siebert algorithm are formally versal as deformations of log schemes. In the same article, they construct canonical coordinates on the base of the formally versal deformation. In \cite{GrossHackingSiebert2022}, Gross, Hacking, and Siebert construct canonical bases of ample line bundles for toric degenerations (as well as more general situations). There is also an extensive literature on questions related to Gromov--Witten invariants in the Gross--Siebert construction and more general log contexts, see \cite{GrossPandharipandeSiebert2010,GrossSiebert2013,FilippiniStoppa2015,Bousseau2019,Bousseau2020,Qile2014,AbramovichWise2018,ACGS2020,Graefnitz2022} to give just a subjective sample.

\vspace{\baselineskip}

\noindent\textbf{Toroidal crossing spaces.} A toroidal crossing space $(V,\cP,\bar\rho)$ consists of a scheme $V$ together with a sheaf of monoids $\cP$ and a global section $\bar\rho \in \cP$ which is locally the ghost sheaf of a vertical and saturated log smooth log morphism to the standard log point $S_0$. This notion is relevant in the construction of log singular log varieties, as one may construct $(V,\cP,\bar\rho)$ first, and then a section $s \in \cL\cS_V$ in the sheaf of actual log smooth log structures with the specified ghost sheaf. The classical case is when $V$ is a normal crossing space; in this situation, $(\cP,\bar\rho)$ can be canonically constructed from the geometry of $V$. The existence of a global section of $\cL\cS_V$ is equivalent to $d$-semistability, introduced by Friedman in \cite{Friedman1983}, and further investigated in \cite{KawamataNamikawa1994,Steenbrink1995,Kato1996}. In 2003, Olsson gives a more general treatment, still in the semistable case, in \cite{Olsson2003}. This, in turn, has been further generalized by Li in \cite{Li2007} in 2007, where the case of weakly semistable families over a higher-dimensional base is treated. Also related to normal crossing spaces as toroidal crossing spaces, Tziolas has provided in \cite{Tziolas2011} methods to compute $\T^1_V$.

The case of toroidal crossings, where a semistable degeneration is replaced with a more general vertical and saturated log smooth morphism, has been first investigated (and named) by Schr\"oer and Siebert in \cite{SchroerSiebert2006} in 2006, under the name of Gorenstein toroidal crossings. In the same year, the variant of the theory that we consider now standard has been developed by Gross and Siebert in \cite{GrossSiebertI} while expounding their mirror symmetry construction. In 2021, Filip, Ruddat, and the author have studied deformations of toroidal crossing spaces (already endowed with a section of $\cL\cS_V$) in \cite{FFR2021}, where also the name \emph{toroidal crossing space} for a triple $(V,\cP,\bar\rho)$ as above is established. In the appendix of this book, we give the first extended exposition on toroidal crossing spaces.

\vspace{\baselineskip}

\noindent\textbf{The smoothing method of Chan--Leung--Ma.} In 2019, Chan, Leung, and Ma achieve in \cite{ChanLeungMa2023} the major breakthrough needed to generalize the classical Bogomolov--Tian--Todorov theorem into a logarithmic one. Their key insight is that the dg Lie algebras of classical deformation theory must be replaced with curved Lie algebras in order to describe log deformations as solutions of a Maurer--Cartan equation. The pre-dg Lie algebra can be constructed by an adaptation of Iacono--Manetti's approach via the Thom--Whitney resolution. Then a pre-dg BV algebra is constructed, whose analysis in the spirit of Katzarkov--Kontsevich--Pantev \cite{KKP2008} and Terilla \cite{Terilla2008} yields unobstructedness of pairs of a log deformation together with a volume form. The method can be applied whenever we have locally fixed an isomorphism type of log deformations, in particular for log smooth deformations and in the Gross--Siebert program. The article \cite{ChanLeungMa2023} also contains a treatment of semi-infinite variations of Hodge structures in that setting, generalizing the work of Barannikov.

In 2021, Filip, Ruddat, and the author have applied in \cite{FFR2021} the method of Chan--Leung--Ma to show a very general smoothing result for normal crossing spaces, relying on the local deformation theory for certain log singularities developed by Gross and Siebert in \cite{GrossSiebertII}. In 2022, the author has in \cite{Felten2022} slightly modified the method to prove explicitly that the log smooth deformation functor $\mathrm{LD}_{X_0/S_0}$ is controlled by a curved Lie algebra $L^\bullet_{X_0/S_0}$. Also in 2022, Petracci and the author have announced and partially proven in \cite{FeltenPetracci2022} the logarithmic Bogomolov--Tian--Todorov theorem, i.e., the unobstructedness of $\mathrm{LD}_{X_0/S_0}$ if $f_0: X_0 \to S_0$ is log Calabi--Yau; a full proof is contained in this book. If the base carries the trivial log structure, the log Bogomolov--Tian--Todorov theorem also follows from the method of Iacono--Manetti, as described in \cite{FeltenPetracci2022} in full generality, and in \cite{Iacono2015} in a special case.

In 2022, Chan and Ma apply in \cite{CLM2022pairs} their smoothing method to pairs of a space $X$ and a bounded perfect complex $\C^*$, very similar to the study of vector bundles and in particular line bundles in the present work. However, \cite{CLM2022pairs} misses the unobstructedness of line bundles on a log CY space, and hence needs to assume $H^2(X,\cO_X) = 0$ for the most interesting results. In the same year, Chan, Ma, and Suen apply in \cite{ChanMaSuen2022} the theory of \cite{CLM2022pairs} to deform certain vector bundles on some $X_0(B,\cP,s)$ of dimension $2$.

Chan, Leung, and Ma have as well introduced methods of asymptotic analysis into the Gross--Siebert program in \cite{CLM2022scattering}. These methods have been applied in \cite{ChanMa2020tropical} and \cite{LeungMaYoung2021} to curve counting, and in \cite{CLM2022fukaya} to produce a variant of the above smoothing method where the Thom--Whitney resolution, applied on $X_0$, is replaced with a resolution constructed from asymptotic analysis, applied on $B$. This allows to tie solutions of the Maurer--Cartan equation explicitly to scattering diagrams, thus connecting the two methods to construct log deformations.

\vspace{\baselineskip}

\noindent\textbf{Global smoothings of normal singularities.} Smoothings and deformations constructed via logarithmic deformation theory must be distinguished from smoothings constructed via the classical flat deformation functor $\mathrm{Def}_X$. There is, in fact, a vast literature on the latter, both for normal and non-normal singularities, usually for either Calabi--Yau or Fano varieties. In 1994, Namikawa has shown in \cite{Namikawa1994} that $\mathrm{Def}_X$ of a projective $3$-dimensional $\QQ$-Calabi--Yau variety with terminal singularities is unobstructed, after a special case with Kleinian singularities has been treated by Ran in \cite{Ran1992Kleinian} in 1992. The deformations arising from these unobstructedness results need not always be smoothings, but under some special hypotheses they are. Related questions have been further investigated in \cite{Gross1997,Namikawa1997inv,Namikawa1997Fano,Minagawa2001,Namikawa2002,Minagawa2003} mostly by Namikawa and Minagawa, and more recently in \cite{Sano2016QI,Sano2017,Sano2017QII,Sano2018} by Sano exclusively in the Fano case (and variants thereof). Other works focus on the relation between $X$ and its smoothing; in 2011, Jahnke and Radloff show in \cite{JahnkeRadloff2011} that the Picard numbers of a Gorenstein terminal Fano threefold $X$ and its smoothing coincide. 

New interest in this topic has arisen recently with a series of articles by Friedman and Laza on $k$-rational and $k$-du Bois singularities, see \cite{FriedmanLaza2023defo,FriedmanLaza2023local,FriedmanLaza2023higher,FriedmanLaza2023higherIsol,FriedmanLaza2023liminal,FriedmanLaza2023log}.

In a different direction, much work has been done to study the flat deformations of Gorenstein toric Fano varieties, starting with the work of Altmann \cite{Altmann1994,Altmann1995,Altmann1997,Altmann1997ob,Altmann2000} in the 1990s. An overview can be found in \cite{Petracci2022}. Conversely, different methods have been developed to degenerate a smooth Fano variety to a Gorenstein toric Fano variety, for example the classical \cite{GonciuleaLakshmibai1996} for flag varieties, or Anderson's construction of degenerations to toric varieties from Okounkov bodies in \cite{Anderson2013} of 2013. The general philosophy is that moduli spaces of Fano varieties may be classified by the Gorenstein toric Fano varieties contained in them, as opposed to the more intricate philosophy of the Gross--Siebert program that moduli spaces can be classified by the degenerations to reducible toric varieties contained in them.

A study of deformations in the non-normal case with classical methods has been undertaken by Tziolas in \cite{Tziolas2010} in 2010. In \cite{FantechiFP2021} and \cite{FantechiFP2022}, Fantechi, Franciosi, and Pardini apply the result of Tziolas to semi-smooth varieties. Related to and motivated by this, Nobile reviews in \cite{Nobile2023} the relation between formal smoothings and geometric smoothings. Technically, also Friedman's classical \cite{Friedman1983} as well as Kawamata and Namikawa's \cite{KawamataNamikawa1994} apply classical methods to construct global deformations of non-normal varieties, but we have discussed these works already above.

\vspace{\baselineskip}

\noindent\textbf{Log Hodge structures.} Just as the curved Gerstenhaber calculi that we study here, log Hodge structures are another tool to study log spaces from a cohomological perspective. Log Hodge structures have been developed by K.~Kato, Nakayama, and Usui in a series of articles starting around 2000, see \cite{KatoUsui2000,KatoUsui2009,KatoNakayamaUsui2009,KatoNakayamaUsui2010,KatoNakayamaUsui2011,KatoNakayamaUsui2013}. Somewhat surprisingly, the concern of these articles is not so much the construction of log Hodge structures from families of log schemes but rather the classification of log Hodge structures; in fact, the definition of a log Hodge structure is adapted to make their classification feasible, not their construction. An exception to that is the early work \cite{Matsubara1998} of Matsubara in 1998, and possibly a preprint of K.~Kato mentioned therein, as well as the work \cite{KMN2002} of K.~Kato, Matsubara, and Nakayama where they construct log Hodge structures if the base is log smooth over the trivial log point. The question of constructing a log Hodge structure geometrically over non-trivial log points (and more general bases) has been addressed quite recently by Fujisawa and Nakayama in \cite{FujisawaNakayama2015,FujisawaNakayama2020}; they obtain partial results which are already quite satisfactory. Usui's article \cite{Usui2014} suggests that certain mirror phenomena can be captured by log Hodge structures, and this topic certainly deserves further investigation, as well as the relation with the log semi-infinite variations of Hodge structures that show up in \cite{ChanLeungMa2023}.

\section{Outlook}

Several closely related problems of varying difficulty and importance have remained unaddressed here.

\vspace{\baselineskip}

\noindent\textbf{Deformations of vector bundles.} We show that deformations of pairs $f_0: (X_0,\cL_0) \to S_0$ are unobstructed when $\cL_0$ is a vector bundle, and $f_0: X_0 \to S_0$ is log Calabi--Yau. By repeating the argument, we expect that simultaneous deformations of a finite collection $(\cL_{0;1},...,\cL_{0;m})$ of line bundles are unobstructed as well. It is also open to consider deformations of vector bundles of rank $r \geq 2$, or coherent sheaves, or simultaneous deformations of the category of coherent sheaves, including their morphisms.

\vspace{\baselineskip} 

\noindent\textbf{Construction and characterization of log regular sections.} As mentioned below, it is open if every log regular section $s_0 \in \cL_0$ is also log transversal. If not, it is also open to show that the modification of a log toroidal family in a log regular section is always log toroidal. Furthermore, it is largely open to give criteria for the existence of global log regular sections, but cf.~the results in \cite{FFR2021} on the construction of log Calabi--Yau log structures on certain normal crossing spaces, a version of which we discuss in Theorem~\ref{smoothing-nc-spaces}.

\vspace{\baselineskip}

\noindent\textbf{Propagation of perfectness to modifications and $\PP^1$-bundles.} Assume that $f_0: X_0 \to S_0$ is a proper enhanced generically log smooth family with perfect curved Gerstenhaber calculus. Then it is open if, given a line bundle $\cL_0$, the family $P_0(\cL_0) \to S_0$ constructed in Chapter~\ref{fiber-bundle-constr-sec} has perfect curved Gerstenhaber calculus. It is also open if, given a log regular section $s_0$ in some line bundle $\cL_0$, the modification $X_0(s_0) \to S_0$ constructed in Chapter~\ref{log-modif-sec} has perfect curved Gerstenhaber calculus. Although it seems plausible to assume this, the perfectness of the new curved Gerstenhaber calculus does not appear to be an easy formal consequence in these settings.

\vspace{\baselineskip}

\noindent\textbf{Homotopy theory of Gerstenhaber calculi.} Although our construction of the characteristic curved Gerstenhaber calculus depends on the choice of two coverings, it is reasonable to expect that the resulting curved Gerstenhaber algebra is independent of these choices up to homotopy. Some results are already known in this direction, see \cite[\S 3.3.2]{ChanLeungMa2023}. In the non-logarithmic case, where a homotopy theory is more straightforward due to the absence of true predifferentials, this seems to be classical, at least for dg Lie algebras/$L_\infty$-algebras. Note also that there has been recent progress on the homotopy theory of curved Lie algebras (\cite{Maunder2017,Bellier2020,Lucio2022}).

\vspace{\baselineskip}

\noindent\textbf{Semi-derived deformations.} In practice, it might be difficult to construct a system of deformations $\D$ for a generically log smooth family $f_0: X_0 \to S_0$. One idea is to use local resolutions of singularities, and then take a (derived) direct image of the Gerstenhaber calculi on the resolutions. It is reasonable to expect that the resolutions do not glue, and that the derived direct images agree only up to homotopy. A variant of the theory should be developed that works in this setting. Here, we may start directly with curved Gerstenhaber calculi. While there may not be any well-defined actual Gerstenhaber calculus on such a deformation, it is reasonable to expect that there is a well-defined underlying deformation of schemes because many resolutions of log singularities satisfy $R\pi_*\cO_{\tilde X} = \cO_X$. It is then however unclear which invariants are preserved in such a deformation. It will probably be difficult to extract information about a smoothing from the general fiber in these settings. Using derived Gerstenhaber calculi (which are yet to be defined) may also rectify situations in which the curved Gerstenhaber calculus of a non-derived system of deformations is not perfect.

\vspace{\baselineskip}

\noindent\textbf{Integral log smooth morphisms.} We assume throughout that our morphisms are saturated. There should be a similar theory for morphisms that are only integral. This situation is less rigid, so we have less information about the families. In the absence of non-coherent log singularities, one can take a base change which makes the map saturated, but this alters the underlying scheme, and then the behavior of our log singularities is not yet clear.

\vspace{\baselineskip}

\noindent\textbf{Local deformation theory.} In this book, we have studied extensively the passage from local deformation theory of log schemes to global deformation theory. For this, we assume that local deformations are fixed in a compatible way. However, some generically log smooth families admit many generically log smooth deformations while others admit none. It is an open problem to classify all generically log smooth log singularities over $S_0$ and their local deformations.

\section{Acknowledgements}

I thank my postdoc mentor Aise Johan de Jong for discussion and ongoing support around most of the topics covered in this study. I thank Richard Thomas for discussion around his example of an obstructed vector bundle of rank $2$, see Example~\ref{R-Thomas-ex-1}. I thank Luc Illusie for clarifications about the early history of log structures. I thank Andrea Petracci and Andr\'es G\'omez Villegas for comments on the draft of this book. I thank Taro Sano for bringing up the question about unobstructedness of line bundles on log Calabi--Yau spaces. His question largely guided me to the present form of the material in this book. I thank S\'andor Kov\'acs for helpful comments on the notion of semi-log-canonical singularities. I thank Bruno Vallette for pointing out many important references in operads and homotopical algebra. I thank my former PhD mentor Helge Ruddat for further discussion around toroidal crossing spaces and smoothings.

I thank Columbia University for its hospitality. This study has been financially supported by the grant FE 2102/1-1 of the German research foundation.

\section{Notations and conventions}

We fix a field $\kk$ of characteristic $0$.

\subsection{Complexes and double complexes}

Let $\bA$ be an abelian category. Complexes $C^\bullet$ in $\bA$ have cohomological grading, i.e., the differential is $\partial: C^i \to C^{i + 1}$. For $s \in \ZZ$, the shift is $C^i[s] := C^{i + s}$ with differential $\partial[s] := (-1)^s \cdot \partial$ acquiring a sign $(-1)^s$. In a double complex $(C^{\bullet,\bullet},\partial,\bar\partial)$, we assume 
$$\partial\bar\partial + \bar\partial\partial = 0;$$
then the differential of the total complex is $\partial + \bar\partial$.

\subsection{Artinian rings}
Here we fix notations for the general study of functors of Artin rings. Let $\Lambda$ be a complete local Noetherian $\kk$-algebra with residue field $\kk$. We denote the maximal ideal by $\m_\Lambda \subset \Lambda$. We write $A_k = \Lambda/\m_\Lambda^{k + 1}$ for integers $k \geq 0$; in particular, $A_0 = \kk$. We write furthermore $A_\eps = \kk[\eps]/(\eps^2)$ with the $\Lambda$-algebra structure induced by $\Lambda \to A_0 \to A_\eps$.
We denote the category of Noetherian local complete $\Lambda$-algebras with residue field $\kk$ by $\mathbf{Alg}_\Lambda$; inside it, we denote the category of Artinian local $\Lambda$-algebras with residue field $\kk$ by $\mathbf{Art}_\Lambda$. Morphisms are local homomorphisms of $\Lambda$-algebras. A \emph{small extension} is a surjection $B' \to B$ in $\mathbf{Art}_\Lambda$ such that the kernel $I \subset B'$ of $B' \to B$ is a principal ideal with $I \cdot \m_{B'} = 0$; then $I$ is one-dimensional. Every surjection in $\mathbf{Art}_\Lambda$ factors into a sequence of small extensions. 

A \emph{functor of Artin rings} (sic) is a covariant functor $F: \mathbf{Art}_\Lambda \to \mathbf{Set}$ with $F(A_0) = \{*\}$. We denote the terminal object $T$ in the functor category, defined by $T(A) = \{*\}$ for all $A \in \mathbf{Art}_\Lambda$, by $\{*\}$. We say, a functor of Artin rings $F$ is \emph{unobstructed} if $F(B') \to F(B)$ is surjective for every surjection $B' \to B$ in $\mathbf{Art}_\Lambda$. We say a natural transformation $F \to G$ of functors of Artin rings is \emph{smooth} if, for every surjection $B' \to B$ in $\mathbf{Art}_\Lambda$, the induced map $F(B') \to F(B) \times_{G(B)} G(B')$ is surjective. A functor of Artin rings $F$ is unobstructed if and only if the natural map $F \to \{*\}$ is smooth.
For a functor of Artin rings $F: \mathbf{Art}_\Lambda \to \mathbf{Set}$, we have the following conditions:
\begin{itemize}
 \item[$(H_0)$]: $F(A_0) = \{*\}$ has only one element.
 \item[$(H_1)$]: The natural map $F(A' \times_A A'') \to F(A') \times_{F(A)} F(A'')$ is surjective for every small extension $A'' \to A$.
 \item[$(H_2)$]: The map in $(H_1)$ is bijective for $A'' = A_\eps$ and $A = A_0$.
\end{itemize}
A functor of Artin rings $F$ satisfying these conditions is called a \emph{deformation functor}. In this case, the \emph{tangent space} $t_F := F(A_\eps)$ carries a natural structure of a $\kk$-vector space. Also, if $p: A' \to A$ is a small extension with kernel $I \subset A'$, and if $\eta' \in F(A')$ with $\eta := F(p)(\eta') \in F(A)$, then there is a natural transitive action of $t_F \otimes_\kk I$ on $F(p)^{-1}(\eta)$.
\begin{itemize}
 \item[$(H_3)$]: The $\kk$-vector space $t_F$ has finite dimension.
 \item[$(H_4)$]: If $p: A' \to A$ is a small extension with kernel $I$, and $\eta' \in F(A')$ with $\eta := F(p)(\eta')$, then the natural action of $t_F \otimes_\kk I$ on $F(p)^{-1}(\eta)$ is bijective.
\end{itemize}
A deformation functor has a hull if and only if it satisfies $(H_3)$; it is pro-representable if and only if it satisfies $(H_3)$ and $(H_4)$.
A general reference for functors of Artin rings is \cite{Schlessinger1968}.

In another direction, following e.g.\ \cite{ManettiLieMethods2022}, a functor of Artin rings is \emph{homogeneous} if the map in $(H_1)$ is bijective whenever $A'' \to A$ is surjective. Every homogeneous functor of Artin rings is a deformation functor, but conversely, most deformation functors are not homogeneous.

\subsection{Logarithmic deformation functors}

Here we fix notations for the study of the deformation theory of log schemes. Let $Q$ be a sharp toric monoid. We denote the associated monoid $\kk$-algebra by $\kk[Q]$, and the elements of the distinguished $\kk$-basis by $z^q$ for $q \in Q$. Since $Q$ is sharp, $Q^+ = Q \setminus \{0\}$ is a monoid ideal; we denote the associated maximal ideal of $\kk[Q]$ by $\kk[Q^+]$. When we complete $\kk[Q]$ in $\kk[Q^+]$, then we obtain the Noetherian complete local $\kk$-algebra $\kk\llbracket Q\rrbracket$ with residue field $\kk$. We denote its maximal ideal by $\m_Q$. We write $\mathbf{Art}_Q$ for the category of Artinian local $\kk\llbracket Q\rrbracket$-algebras with residue field $\kk$. As above, we have specific objects $A_0 = \kk$, $A_k = \kk\llbracket Q\rrbracket/(\m_Q^{k + 1})$, and $A_\eps = \kk[\eps]/(\eps^2)$ in $\mathbf{Art}_Q$. We denote the associated log schemes by $S_0 = \Spec (Q \to A_0)$, $S_k = \Spec(Q \to A_k)$, and $S_\eps = \Spec (Q \to A_\eps)$. For a general object $A \in \mathbf{Art}_Q$, we write $S_A = \Spec(Q \to A)$.

If $f_0: X_0 \to S_0$ is a log smooth morphism, then we write $\mathrm{LD}_{X_0/S_0}: \mathbf{Art}_Q \to \mathbf{Set}$ for its associated log smooth deformation functor. A general reference for basic log smooth deformation theory is \cite{Kato1996}.


\part{Abstract unobstructedness theorems}


\chapter{Algebraic structures}\label{alg-str-sec}\note{alg-str-sec}

The purpose of this chapter is to fix the definitions of several \emph{algebraic structures} that will play an important role in our study of logarithmic deformation theory. We will study essentially seven different algebraic structures:
\begin{enumerate}[label=(\arabic*)]
 \item Lie algebras;
 \item Lie--Rinehart algebras and pairs;
 \item Gerstenhaber algebras and calculi;
 \item Batalin--Vilkovisky algebras and calculi.\footnote{Gerstenhaber calculi and Batalin--Vilkovisky calculi will have two variants, a one-sided and a two-sided one, which differ by an additional operation whose relevance we discovered only toward the end of this project. This operation is relevant for some of our applications but not for all.}
\end{enumerate}
Each of them comes in four variants. The singly graded original will usually appear as a collection of sheaves $(\F^i)_i$ on a space $X$ together with a collection of operations between them; to each $\F^i$ will be assigned a degree $|\F^i| \in \ZZ$. While one may omit this degree in the case of Lie algebras and Lie--Rinehart algebras, it is essential for Gerstenhaber or Batalin--Vilkovisky algebras and calculi. The \emph{differential bigraded} version shows up once we take an acyclic resolution $\F^i \to \F^{i,\bullet}$ of each $\F^i$ which is compatible with all operations. The (plainly) \emph{bigraded} version is obtained by forgetting the differential $\bar\partial: \F^{i,j} \to \F^{i,j + 1}$. Finally, the \emph{curved} version arises when we relax the condition $\bar\partial^2 = 0$. In this case, $\bar\partial: \F^{i,j} \to \F^{i,j + 1}$ is called a \emph{predifferential}, and there will be a section $\ell$ of some sheaf $\F^i$, the \emph{curvature}, which measures the defect of $\bar\partial$ to be a differential.

\vspace{\baselineskip}

Many of our arguments in later chapters are equally valid for several of the above algebraic structures. To talk about all of them at once, we introduce an adhoc formalism of algebraic structures, which is essentially a simplified (and somewhat adapted) version of colored operads.\footnote{Since we work with constants, we have to allow operations of arity $0$, which not all introductory texts on colored operads do. It is then often possible to treat (bi-)graded and differential (bi-)graded versions at once by changing the target of a representation but not the operad. Even a theory for the curved case has been recently worked out in \cite{Lucio2022}. We will do all this by hand, defining three algebraic structures $\cP$, $\cP^{bg}$, and $\cP^{crv}$ in each case. The inclined reader familiar with operadic calculus may want to compare the result of the canonical procedure with our explicit definitions below.} For the purpose of our discussion, an \emph{algebraic structure} $\cP$ comprises the following data:
\begin{enumerate}[label=(\arabic*)]
 \item a non-empty set $D = D(\cP)$, called the set of \emph{domains};
 \item a \emph{grading} $|\cdot|: D \to \ZZ$, assigning to every domain its degree;
 \item for every domain $P \in D$, a set $\cP(P)$, the set of \emph{constants};
 \item for all domains $P_1,...,P_n \in D$ ($n \geq 1$) and $Q \in D$, a set $\cP(P_1,...,P_n;Q)$, the set of \emph{operations};
 \item \emph{axioms} between the constants and operations that should be satisfied.
\end{enumerate}
Then in a \emph{representation}\index{algebraic structure!representation} of $\cP$ (or \emph{$\cP$-algebra}), say in the category of abelian sheaves on a topological space $X$, there is an abelian sheaf $\E^P$ associated with every $P \in D$, a section $C^\gamma \in \E^P$ associated with every $\gamma \in \cP(P)$, and a multilinear map 
$$A^\mu: \E^{P_1} \times ... \times \E^{P_n} \to \E^Q$$
associated with every operation $\mu \in \cP(P_1,...,P_n;Q)$; they must satisfy the axioms, which are stated as equations for sums of compositions of operations and constants. 

\vspace{\baselineskip}

We will be concerned with representations of the seven algebraic structures in several contexts---as modules over a ring, as sheaves on a topological space etc. We unify these different setups as follows: 

\begin{defn}
 A \emph{context}\index{context}\index{algebraic structure!context} is a quadruple $(\cS,\C,\cO,\bC)$ where $\cS$ is a site, $\C$ is a sheaf of rings on $\cS$, the sheaf of \emph{constants}, $\cO$ is a sheaf of $\C$-algebras, and $\bC$ is a class of $\cO$-modules. In a representation of an algebraic structure $\cP$ in the context $(\cS,\C,\cO,\bC)$, the domains are $\cO$-modules in the class $\bC$, the constants are global sections of $\E^P$, and the operations are $\C$-multilinear maps. They are not required to be $\cO$-multilinear.
\end{defn}

\begin{ex}
 Typical examples of contexts include the following:
 \begin{enumerate}[label=(\alph*)]
  \item Let $f: X \to S$ be a flat morphism of Noetherian schemes. Then we can take for $\cS$ the small Zariski site of $X$ and $\C = f^{-1}(\cO_S)$, $\cO = \cO_X$. For $\bC$, we take coherent $\cO_X$-modules which are flat over $S$. We denote this context by $\mathfrak{Coh}(X/S)$. When we take all quasi-coherent modules flat over $S$ respective all $\cO_X$-modules flat over $S$, we denote the context by $\mathfrak{QCoh}(X/S)$ respective $\mathfrak{Mod}(X/S)$. When we wish to include also non-flat modules, we write $\mathfrak{Coh}'(X/S)$, $\mathfrak{QCoh}'(X/S)$, respective $\mathfrak{Mod}'(X/S)$.
  \item Let $A \to R$ be a flat ring homomorphism. Then we can take $\cS$ to be the site of a topological space with one point, $\C = A$, and $\cO = R$. For $\bC$, we take finitely generated $R$-modules which are flat over $A$. We denote this context by $\mathfrak{Mod}_{fg}(R/A)$.
  \item Let $\Lambda$ be a complete local Noetherian $\kk$-algebra with residue field $\kk$. For $\cS$, we take the site of a topological space with one point; we take $\C = \cO = \Lambda$. For $\bC$, we take flat and complete $\Lambda$-modules (which are not necessarily finitely generated). We denote this context by $\mathfrak{Comp}(\Lambda)$.
 \end{enumerate}
\end{ex}

\section{Lie algebras}\label{Lie-alg-sec}\note{Lie-alg-sec}

Here are our three variants of Lie algebras.

\begin{defn}\index{Lie algebra}
 Let $(\cS,\C,\cO,\bC)$ be a context.
 \begin{enumerate}[label=(\alph*)]
  \item A \emph{Lie algebra} $(\cL,[-,-])$ consists of an $\cO$-module $\cL \in \bC$ and a $\C$-bilinear map $[-,-]: \cL \times \cL \to \cL$ such that $[\theta,\xi] = -[\xi,\theta]$ and the \emph{Jacobi identity}
  $$[\theta,[\xi,\eta]] = [[\theta,\xi],\eta] + [\xi,[\theta,\eta]]$$
  holds. By convention, we assign to $\cL$ the \emph{degree} $|\cL| = -1$; this is in order to fit later conventions.
  \item A \emph{graded Lie algebra}\index{Lie algebra!bigraded}\footnote{Or \emph{bigraded} Lie algebra to put it more systematically, consistent with the fact that we have already assigned the degree $|\cL| = -1$ to a Lie algebra.} $(\cL^\bullet,[-,-])$ consists of an $\cO$-module $\cL^i \in \bC$ of \emph{bidegree} $(-1,i)$ and \emph{total degree} $i - 1$ for every $i \geq 0$ and a $\C$-bilinear map $[-,-]: \cL^i \times \cL^j \to \cL^{i + j}$ for all indices $i,j \geq 0$ such that 
  $$[\theta,\xi] = -(-1)^{(|\theta| + 1)(|\xi| + 1)} [\xi,\theta]$$
  (with $|\theta| = i - 1$ for $i \in \cL^i$) and the \emph{Jacobi identity}
  $$[\theta,[\xi,\eta]] = [[\theta,\xi],\eta] + (-1)^{(|\theta| + 1)(|\xi| + 1)}[\xi,[\theta,\eta]]$$
  holds.
  \item A \emph{curved Lie algebra}\index{Lie algebra!curved} $(\cL^\bullet,[-,-],\bar\partial,\ell)$ consists of a graded Lie algebra $(\cL^\bullet,[-,-])$, a $\C$-linear map $\bar\partial: \cL^i \to \cL^{i + 1}$ such that 
  $$\bar\partial[\theta,\xi] = [\bar\partial\theta,\xi] + (-1)^{|\theta| + 1}[\theta,\bar\partial\xi],$$
  and a global section $\ell \in \cL^2$ such that $\bar\partial^2\theta = [\ell,\theta]$ and $\bar\partial\ell = 0$. The map $\bar\partial: \cL^i \to \cL^{i + 1}$ is called the \emph{predifferential}, and the global section $\ell \in \cL^2$ is called the \emph{curvature}.
  \item A \emph{differential graded Lie algebra} is a curved Lie algebra with $\ell = 0$. In this case, $\bar\partial^2 = 0$, but this condition is not sufficient in general (unless the curved Lie algebra is faithful as defined below).
  \item A Lie algebra is \emph{faithful}\index{Lie algebra!faithful} or \emph{centerless}\index{Lie algebra!centerless} if the map 
 $$\mathrm{ad}_L: \cL \to \cH om_\C(\cL,\cL), \quad \theta \mapsto (\xi \mapsto [\theta,\xi]),$$
 is injective.
 \item 
 A graded/dg/curved Lie algebra is \emph{faithful} or \emph{centerless} if the map 
 $$\mathrm{ad}_{L}^i: \cL^i \to \cH om_\C(\cL^0,\cL^i), \quad \theta \mapsto (\xi \mapsto [\theta,\xi]),$$
 is injective for every $i \geq 0$. Note that the source of the Hom space in the target is always $\cL^0$.
 \item A graded/dg/curved Lie algebra is \emph{bounded}\index{Lie algebra!bounded} if $\cL^i = 0$ for $i >> 0$.
 \end{enumerate}
  
\end{defn}

\begin{rem}
 A couple of remarks are in order.
 \begin{enumerate}[label=(\arabic*)]
  \item What we call a (differential) graded Lie algebra is often called an \emph{odd} differential graded Lie algebra to clarify the grading convention, which differs from the more obvious \emph{even} grading convention with $[\theta,\xi] = -(-1)^{|\theta||\xi|}[\xi,\theta]$.
  \item The operator $\bar\partial: \cL^i \to \cL^{i + 1}$ is called a \emph{predifferential} to distinguish it from the \emph{differentials}, which satisfy $\bar\partial^2 = 0$. We will later modify $\bar\partial$ to $\bar\partial_\phi = \bar\partial + [\phi,-]$ for $\phi \in \cL^1$, and it will be useful to distinguish actual differentials from predifferentials also terminologically. 
  
  \item The term \emph{predifferential} seems to have arisen relatively recently in \cite{Maunder2017,Bellier2020, Lucio2022} as well as independently in the author's own previous article \cite{Felten2022}. According to \cite{Bellier2020, Lucio2022}, a predifferential graded object is a graded object $\A^\bullet$ together with a map $\bar\partial: \A^\bullet \to \A^{\bullet - 1}$ (in homological grading convention) which may not satisfy $\bar\partial^2 = 0$. Thus, one may argue that a \emph{predifferential graded Lie algebra} should be what we have defined as a curved Lie algebra without the conditions $\bar\partial^2\theta = [\ell,\theta]$ and $\bar\partial\ell = 0$.
  \item In \cite[Defn.~10.4]{Felten2020}, we have defined a predifferential graded Lie algebra in such a way that we have an element $\ell \in \cL^2$ with $\bar\partial^2\theta = [\ell,\theta]$ but without the condition $\bar\partial\ell = 0$, halfway between a predifferential graded Lie algebra as discussed in the previous remark and a curved Lie algebra as defined above. This definition has been modelled on the structures used by Chan, Leung, and Ma in \cite{ChanLeungMa2023}, and there, the condition $\bar\partial\ell = 0$ is not explicit. Early in the course of writing this book, we have added the condition to the definition because we need it (and its variants) later in the proofs of Lemma~\ref{d-is-diff} and Lemma~\ref{check-d-is-diff}. After making the manuscript available on the arXiv, the author has learned from Bruno Vallette that the condition $\bar\partial\ell = 0$ is (or should be) indeed part of the standard definition of a curved Lie algebra, for example as defined in \cite[Defn.~4.1]{Maunder2017} or \cite[Defn.~2.15]{Lucio2022}, or in \cite{Positselski2011} for curved dg algebras\footnote{Positselski refers to curved algebras as curved dg algebras to emphasize their graded and predifferential nature. His definition requires $\bar\partial(h) = 0$ for the curvature $h$.} (although $\bar\partial \ell = 0$ is not required in \cite{Chuang2016}, which also refers to predifferentials as differentials for lack of a better name as late as 2016). Thus, we have abandoned our original terminology \emph{predifferential graded Lie algebra}. This terminology seemed most natural to us, given that we consider $\bar\partial$ as a predifferential, even consistently with \cite{Bellier2020,Lucio2022}, but it was chosen in ignorance of the concept of curved Lie algebras. We have replaced it with the already established standard name. This also prevents confusion with the concept of differential graded pre-Lie algebras, discussed for example in \cite[\S 6.4.8]{LodayVallette2012}.
  \item The name \emph{faithful} refers to the fact that the adjoint representation is faithful. The name \emph{centerless} is borrowed from group theory, where a group is called \emph{centerless} if the identity is the only element which commutes with all other elements. For other algebraic structures, the two notions will differ, see for example Lie--Rinehart pairs.
 \end{enumerate}

\end{rem}

\begin{ex}
 Let $f: X \to S$ be a smooth morphism of (Noetherian) schemes. Then the relative tangent sheaf $\T^1_{X/S}$ with its Lie bracket forms a Lie algebra in the context $\mathfrak{Coh}(X/S)$. The Lie bracket is $f^{-1}(\cO_S)$-linear but not $\cO_X$-linear.
\end{ex}

\section{Lie--Rinehart algebras and pairs}\label{LR-alg-sec}\note{LR-alg-sec}

A \emph{Lie--Rinehart algebra} consists of a Lie algebra acting on a ring of functions. The notion goes back to Rinehart's thesis \cite{Rinehart1963}, where it is called a \emph{$(K,R)$-Lie algebra}. The name Lie--Rinehart algebra seems to arise first in \cite{Huebschmann1990} and is also mentioned in \cite{KosmannSchwarzbach1995}; it seems firmly established in \cite{Huebschmann1998}. A Lie--Rinehart algebra can be considered as an algebraic analog of a Lie algebroid, but when we apply our definition of a Lie--Rinehart algebra in the context of a smooth variety $X/\kk$, then we obtain a more general notion than a Lie algebroid.

\begin{defn}
 Let $(\cS,\C,\cO,\bC)$ be a context.
 \begin{enumerate}[label=(\alph*)]
  \item A \emph{Lie--Rinehart algebra}\index{Lie--Rinehart algebra} 
  $$(\F,\T,\ast^F,\ast^T,\nabla^F,\nabla^T,1_F)$$
  consists of two $\cO$-modules $\F \in \bC$ and $\T \in \bC$ with degrees $|\F| = 0$ and $|\T| = -1$, a global section $1_F \in \F$, two $\cO$-bilinear maps 
  $$\ast^F: \F \times \F \to \F, \quad \ast^T: \F \times \T \to \T,$$
  and two $\C$-bilinear maps 
  $$\nabla^F: \T \times \F \to \F, \quad \nabla^T: \T \times \T \to \T.$$
  We write $\ast = \ast^F$, $\nabla = \nabla^F$, and $[-,-] = \nabla^T$ for short. They must satisfy the relations 
  $$(a \ast b) \ast^P p = a \ast^P( b \ast^P p), \quad a \ast b = b \ast a, \quad  1_F \ast^P p = p$$
  for the two multiplications $\ast^F, \ast^T$, the anti-commutativity 
  $$[\theta,\xi] := \nabla^T_\theta(\xi) = - \nabla^T_\xi(\theta) = - [\xi,\theta],$$ 
  the $F$-linearity condition $a \ast^T \nabla_\theta(b) = \nabla_{a \ast^T \theta}(b)$, the two Jacobi identities 
  $$\nabla^P_{[\theta,\xi]}(p) = \nabla^P_\theta \nabla^P_\xi(p) - \nabla^P_\xi\nabla^P_\theta(p),$$
  and the two odd Poisson identities 
  $$\nabla^P_\theta(a \ast^P p) = \nabla_\theta(a) \ast^P p + a \ast^P \nabla^P_\theta(p).$$
  \item A \emph{bigraded Lie--Rinehart algebra}\index{Lie--Rinehart algebra!bigraded} 
  $$(\F^\bullet,\T^\bullet,\ast^F,\ast^T,\nabla^F,\nabla^T,1_F)$$
  consists of $\cO$-modules $\F^i \in \bC$ (of bidegree $(0,i)$ and total degree $i$) and $\T^i \in \bC$ (of bidegree $(-1,i)$ and total degree $i - 1$) for every $i \geq 0$, the $\cO$-bilinear maps 
  $$\ast^F: \F^i \times \F^j \to \F^{i + j}, \quad \ast^T: \F^i \times \T^j \to \T^{i + j},$$
  the $\C$-bilinear maps 
  $$\nabla^F: \T^i \times \F^j \to \F^{i + j}, \quad \nabla^T: \T^i \times \T^j \to \T^{i + j},$$
  and the global section $1_F \in \F^0$. 
  They must satisfy the relations 
  $$(a \ast b) \ast^P p = a \ast^P( b \ast^P p), \quad a \ast b = (-1)^{|a||b|}b \ast a, \quad  1_F \ast^P p = p$$
  for the two multiplications $\ast^F, \ast^T$, the anti-commutativity 
  $$[\theta,\xi] = - (-1)^{(|\theta| + 1)(|\xi| + 1)}[\xi,\theta],$$ 
  the $F$-linearity condition $a \ast^T \nabla_\theta(b) = \nabla_{a \ast^T \theta}(b)$, the two Jacobi identities 
  $$\nabla^P_{[\theta,\xi]}(p) = \nabla^P_\theta \nabla^P_\xi(p) - (-1)^{(|\theta| + 1)(|\xi| + 1)} \nabla^P_\xi\nabla^P_\theta(p),$$
  and the two odd Poisson identities 
  $$\nabla^P_\theta(a \ast^P p) = \nabla_\theta(a) \ast^P p + (-1)^{(|\theta| + 1)|a|}a \ast^P \nabla^P_\theta(p).$$
  \item A \emph{curved Lie--Rinehart algebra}\index{Lie--Rinehart algebra!curved}
  $$(\F^\bullet,\T^\bullet,\ast^F,\ast^T,\nabla^F,\nabla^T,1_F,\bar\partial^F,\bar\partial^T,\ell)$$
  consists of a bigraded Lie--Rinehart algebra, the $\C$-linear maps $\bar\partial^F: \F^i \to \F^{i + 1}$ and $\bar\partial^T: \T^i \to \T^{i + 1}$ satisfying the four derivation rules 
  $$\bar\partial^P\nabla^P_\theta(p) = \nabla^P_{\bar\partial^T\theta}(p) + (-1)^{|\theta| + 1}\nabla^P_\theta(\bar\partial^P p)$$
  and 
  $$\bar\partial^P(a \ast^P p) = \bar\partial^F(a) \ast^P p + (-1)^{|a|} a \ast^P \bar\partial^P p,$$
  and a global section $\ell \in \T^2$ with $(\bar\partial^P)^2(p) = \nabla^P_\ell(p)$ and $\bar\partial^T(\ell) = 0$. We usually write $\bar\partial = \bar\partial^P$ for short; this map is called the \emph{predifferential}. The section $\ell \in \T^2$ is called the \emph{curvature}.
  \item A \emph{differential bigraded Lie--Rinehart algebra} is a curved Lie--Rinehart algebra with $\ell = 0$.
  \item A Lie--Rinehart algebra is \emph{centerless}\index{Lie--Rinehart algebra!centerless} if the map 
  $$\mathrm{ad}_T: \T \to \cH om_\C(\T,\T), \quad \theta \mapsto (\xi \mapsto \nabla^F_\theta(\xi)),$$
  is injective; it is \emph{strictly faithful}\index{Lie--Rinehart algebra!(strictly) faithful} if the map 
  $$\mathrm{ad}_F: \T \to \cH om_\C(\F,\F), \quad \theta \mapsto (a \mapsto \nabla^F_\theta(a)),$$
  is injective; it is \emph{faithful} if the intersection of the two kernels is zero.
  \item A bigraded/dbg/curved Lie--Rinehart algebra is \emph{centerless} if the map 
  $$\mathrm{ad}_{T}^i: \T^i \to \cH om_\C(\T^0,\T^i), \quad \theta \mapsto (\xi \mapsto \nabla^F_\theta(\xi)),$$
  is injective for all $i \geq 0$; it is \emph{strictly faithful} if the map 
  $$\mathrm{ad}_{F}^i: \T^i \to \cH om_\C(\F^0,\F^i), \quad \theta \mapsto (a \mapsto \nabla^F_\theta(a)),$$
  is injective for all $i \geq 0$; it is \emph{faithful} if the intersection of the two kernels is zero for all $i \geq 0$.
  \item A bigraded/dbg/curved Lie--Rinehart algebra is \emph{bounded}\index{Lie--Rinehart algebra!bounded} if $\F^i = 0$ and $\T^i = 0$ for $i >> 0$.
 \end{enumerate}

\end{defn}
\begin{ex}
 Let $f: X \to S$ be a smooth morphism. Then we obtain a Lie--Rinehart algebra in the context $\mathfrak{Coh}(X/S)$ by setting $\T = \T^1_{X/S}$ and $\F = \cO_X$. The two products $\ast^F$ and $\ast^T$ are given by the $\cO_X$-module structure, $\nabla^T$ is the Lie bracket, and $\nabla^F$ is the action of $\T^1_{X/S}$ on $\cO_X$ by derivations. The latter two operations are only $f^{-1}(\cO_S)$-linear but not $\cO_X$-linear. This Lie--Rinehart algebra is strictly faithful because a derivation in $\T^1_{X/S}$ is uniquely determined by its action on $\cO_X$. It is also centerless because a derivation which commutes with all other derivations must be trivial (look at $[D,fD](g) = D(f) \cdot D(g)$ for a derivation $D$). In particular, it is faithful.
\end{ex}

A \emph{Lie--Rinehart module} over a Lie--Rinehart algebra $(\F,\T)$, as defined, for example, in \cite{Casas2004}, is an $\cO$-module $\E$ together with a $\C$-bilinear operation $\nabla^E: \T \times \E \to \E$ which satisfies $\nabla_{a\theta}^E(e) = a\nabla_\theta^E(e)$, the Jacobi identity 
$$\nabla_{[\theta,\xi]}^E(e) = \nabla_\theta^E\nabla_\xi^E(e) - \nabla_\xi^E\nabla_\theta^E(e),$$
and the derivation rule 
$$\nabla_{\theta}^E(ae) = [\theta,a] \cdot e + a \cdot \nabla_\theta^E(e).$$
A \emph{Lie--Rinehart pair} is then a Lie--Rinehart algebra together with a Lie--Rinehart module.

\begin{defn}
 Let $(\cS,\C,\cO,\bC)$ be a context.
 \begin{enumerate}[label=(\alph*)]
  \item A \emph{Lie--Rinehart pair}\index{Lie--Rinehart pair}
  $$(\F,\T,\E,\ast^F,\ast^T,\ast^E,\nabla^F,\nabla^T,\nabla^E,1_F)$$
  consists of a Lie--Rinehart algebra, an $\cO$-module $\E \in \bC$ of degree $|\E| = 0$, an $\cO$-bilinear map $\ast^E: \F \times \E \to \E$, and a $\C$-bilinear map $\nabla^E: \T \times \E \to \E$ such that 
  $$(a \ast^F b) \ast^E e = a \ast^E (b \ast^E e), \quad 1_F \ast^E e = e,$$
  the $F$-linearity
  $$a \ast^E \nabla^E_\theta(e) = \nabla^E_{a \ast^T \theta}(e),$$
  the Jacobi identity 
  $$\nabla_{[\theta,\xi]}^E(e) = \nabla_\theta^E\nabla_\xi^E(e) - \nabla_\xi^E\nabla_\theta^E(e),$$
  and the derivation rule 
  $$\nabla_{\theta}^E(a \ast^E e) = \nabla^F_\theta(a) \ast^E e + a \ast^E \nabla_\theta^E(e)$$
  hold.
  \item A \emph{bigraded Lie--Rinehart pair}\index{Lie--Rinehart pair!bigraded} 
  $$(\F^\bullet,\T^\bullet,\E^\bullet,\ast^F,\ast^T,\ast^E,\nabla^F,\nabla^T,\nabla^E,1_F)$$
  consists of a bigraded Lie--Rinehart algebra, $\cO$-modules $\E^i \in \bC$ (of bidegree $(0,i)$ and total degree $i$) for all $i \geq 0$, the $\cO$-bilinear maps $\ast^E: \F^i \times \E^j \to \E^{i + j}$, and the $\C$-bilinear maps $\nabla^E: \T^i \times \E^j \to \E^{i + j}$ such that
  $$(a \ast^F b) \ast^E e = a \ast^E (b \ast^E e), \quad 1_F \ast^E e = e,$$
  the $F$-linearity
  $$a \ast^E \nabla^E_\theta(e) = \nabla^E_{a \ast^T \theta}(e),$$
  the Jacobi identity 
  $$\nabla_{[\theta,\xi]}^E(e) = \nabla_\theta^E\nabla_\xi^E(e) - (-1)^{(|\theta| + 1)(|\xi| + 1)}\nabla_\xi^E\nabla_\theta^E(e),$$
  and the derivation rule 
  $$\nabla_{\theta}^E(a \ast^E e) = \nabla^F_\theta(a) \ast^E e + (-1)^{(|\theta| + 1)|a|}a \ast^E \nabla_\theta^E(e)$$
  hold.
  \item A \emph{curved Lie--Rinehart pair}\index{Lie--Rinehart pair!curved}
  $$(\F^\bullet,\T^\bullet,\E^\bullet,(\ast^P)_{P \in D},(\nabla^P)_{P \in D},(\bar\partial^P)_{P \in D},1_F,\ell)$$
  consists of a curved Lie--Rinehart algebra and a $\C$-linear map $\bar\partial^E: \E^i \to \E^{i + 1}$ satisfying the derivation rules 
  $$\bar\partial^E\nabla^E_\theta(e) = \nabla^T_{\bar\partial^T\theta}(e) + (-1)^{|\theta| + 1}\nabla^E_\theta(\bar\partial^Ee)$$
  and 
  $$\bar\partial^E(a \ast^E e) = \bar\partial^F(a) \ast^E e + (-1)^{|a|} a \ast^E \bar\partial^E(e)$$
  as well as $(\bar\partial^E)^2(e) = \nabla^E_\ell(e)$.
  \item A \emph{differential bigraded Lie--Rinehart pair} is a curved Lie--Rinehart pair with $\ell = 0$.
  \item A Lie--Rinehart pair is \emph{centerless}\index{Lie--Rinehart pair!centerless} respective \emph{strictly faithful}\index{Lie--Rinehart pair!(strictly) faithful} if the underlying Lie--Rinehart algebra is; it is \emph{faithful} if the intersection of the kernels of the three maps 
  $$\mathrm{ad}_P: \T \to \cH om_\C(P,P), \quad \theta \mapsto (p \mapsto \nabla^P_\theta(p)),$$
  for $P \in \{\F,\T,\E\}$ is zero.
  \item A bigraded/dbg/curved Lie--Rinehart pair is \emph{centerless} respective \emph{strictly faithful} if the underlying (predifferential) bigraded Lie--Rinehart algebra is; it is \emph{faithful} if, for each $i \geq 0$, the intersection of the kernels of the three maps 
  $$\mathrm{ad}^i_P: \T^i \to \cH om_\C(P^0,P^i), \quad \theta \mapsto (p \mapsto \nabla^P_\theta(p)),$$
  is zero.
  \item A bigraded/dbg/curved Lie--Rinehart pair is \emph{bounded}\index{Lie--Rinehart pair!bounded} if $\F^i = 0$, $\T^i = 0$, and $\E^i = 0$ for $i >> 0$.
 \end{enumerate}
\end{defn}
\begin{ex}
 Let $X/\kk$ be a smooth variety, let $\F = \cO_X$, $\T = \T^1_{X/\kk}$, and let $\E$ be a vector bundle with an (algebraic) connection $\nabla: \E \to \E \otimes \Omega^1_{X/\kk}$. Then we can consider $(\F,\T,\E)$ as a Lie--Rinehart pair in the context $\mathfrak{Coh}(X/\kk)$ with the induced map $\nabla: \T \times \E \to \E$.
\end{ex}

\section{Gerstenhaber algebras and calculi}\label{G-alg-sec}\note{G-alg-sec}

When we endow the polyvector fields $\bigwedge^p\T^1_{X/S}$ with the Schouten--Nijenhuis bracket, the result is a \emph{Gerstenhaber algebra}. The notion goes back to \cite{Gerstenhaber1963}. We fix an integer $d \geq 1$, which we call the \emph{dimension} of the Gerstenhaber algebra. For the polyvector fields, it corresponds to the relative dimension of $f: X \to S$. We put Gerstenhaber algebras in negative degrees, i.e., they are in the range $[-d,0]$. This makes the formulae for its modules more natural, and this convention has been used before in \cite{ChanLeungMa2023,Felten2022}.

\begin{defn}
 Let $(\cS,\C,\cO,\bC)$ be a context.
 \begin{enumerate}[label=(\alph*)]
  \item A \emph{Gerstenhaber algebra}\index{Gerstenhaber algebra} of dimension $d$
 $$(\G^\bullet,\wedge,1_G,[-,-])$$
 consists of:
 \begin{itemize}
  \item an $\cO$-module $\G^p \in \bC$ of degree $|\G^p| = p$ for every $-d \leq p \leq 0$; for $p$ outside this range, we set $\G^p = 0$ whenever necessary;
  \item an $\cO$-bilinear product $-\wedge-: \G^p \times \G^{p'} \to \G^{p + p'}$ and a global section $1_G \in \G^0$
  such that 
  $$\theta \wedge (\xi \wedge \eta) = (\theta \wedge \xi) \wedge \eta; \quad \theta \wedge \xi = (-1)^{|\theta||\xi|} \xi \wedge \theta; \quad 1_G \wedge \theta = \theta;$$
 \item a $\C$-bilinear product $[-,-]: \G^{p} \times \G^{p'} \to \G^{p + p' + 1}$
 such that
  $$[\theta,\xi] = - (-1)^{(|\theta| + 1)(|\xi| + 1)}[\xi,\theta]$$ and both the Jacobi identity
  $$[\theta,[\xi,\eta]] = [[\theta,\xi],\eta] + (-1)^{(|\theta| + 1)(|\xi| + 1)}[\xi,[\theta,\eta]]$$
  and the odd Poisson identity 
  $$[\theta,\xi \wedge \eta] = [\theta,\xi] \wedge \eta + (-1)^{(|\theta| + 1)|\xi|}\xi \wedge [\theta,\eta]$$
  hold.
  \end{itemize}
  From these axioms, we also find $[\theta,1_G] = 0$, the formula $[\varphi,[\varphi,\xi]] = [\frac{1}{2}[\varphi,\varphi],\xi]$ for $|\varphi|$ even, a second odd Poisson identity 
  $$[\theta \wedge \xi, \eta] = \theta \wedge [\xi,\eta] + (-1)^{(|\eta| + 1)|\xi|}[\theta,\eta] \wedge \xi,$$
  and $[\varphi^n,\theta] = n[\varphi,\theta] \wedge \varphi^{n - 1}$ for $n \geq 1$ and $|\varphi|$ even, where we write $\varphi^n := \varphi \wedge ... \wedge \varphi$.
  \item A \emph{bigraded Gerstenhaber algebra}\index{Gerstenhaber algebra!bigraded} of dimension $d$
 $$(\G^{\bullet,\bullet},\wedge,1_G,[-,-])$$
 consists of:
 \begin{itemize}
  \item the $\cO$-modules $\G^{p,q} \in \bC$ of bidegree $(p,q)$ and total degree $p + q$ for every $-d \leq p \leq 0$ and every $q \geq 0$; for $(p,q)$ outside this range, we set $\G^{p,q} = 0$ whenever necessary;
  \item an $\cO$-bilinear product $-\wedge-: \G^{p,q} \times \G^{p',q'} \to \G^{p + p',q + q'}$ and a global section $1_G \in \G^{0,0}$
  such that 
  $$\theta \wedge (\xi \wedge \eta) = (\theta \wedge \xi) \wedge \eta; \quad \theta \wedge \xi = (-1)^{|\theta||\xi|} \xi \wedge \theta; \quad 1_G \wedge \theta = \theta;$$
 \item a $\C$-bilinear product $[-,-]: \G^{p,q} \times \G^{p',q'} \to \G^{p + p' + 1,q + q'}$
 such that
  $$[\theta,\xi] = - (-1)^{(|\theta| + 1)(|\xi| + 1)}[\xi,\theta]$$ and both the Jacobi identity
  $$[\theta,[\xi,\eta]] = [[\theta,\xi],\eta] + (-1)^{(|\theta| + 1)(|\xi| + 1)}[\xi,[\theta,\eta]]$$
  and the odd Poisson identity 
  $$[\theta,\xi \wedge \eta] = [\theta,\xi] \wedge \eta + (-1)^{(|\theta| + 1)|\xi|}\xi \wedge [\theta,\eta]$$
  hold.
  \end{itemize}
  From these axioms, we find the same additional formulae as above, namely $[\theta,1_G] = 0$, the formula $[\varphi,[\varphi,\xi]] = [\frac{1}{2}[\varphi,\varphi],\xi]$ for $|\varphi|$ even, a second odd Poisson identity 
  $$[\theta \wedge \xi, \eta] = \theta \wedge [\xi,\eta] + (-1)^{(|\eta| + 1)|\xi|}[\theta,\eta] \wedge \xi,$$
  and $[\varphi^n,\theta] = n[\varphi,\theta] \wedge \varphi^{n - 1}$ for $n \geq 1$ and $|\varphi|$ even, where we write $\varphi^n := \varphi \wedge ... \wedge \varphi$.
  \item A \emph{curved Gerstenhaber algebra}\footnote{Note that for us, \emph{curved} always refers to the graded Lie algebra $\G^{-1,\bullet}$ in a bigraded Gerstenhaber algebra $\G^{\bullet,\bullet}$. We never use the term \emph{curved Gerstenhaber algebra} to denote the structure of a curved Lie algebra on a singly graded Gerstenhaber algebra $\G^\bullet$, considered as a graded Lie algebra.}\index{Gerstenhaber algebra!curved} of dimension $d$
  $$(\G^{\bullet,\bullet},\wedge,1,[-,-],\bar\partial,\ell)$$
  consists of a bigraded Gerstenhaber algebra of dimension $d$, a $\C$-linear map $\bar\partial: \G^{p,q} \to \G^{p,q + 1}$ satisfying the two derivation rules 
  $$\bar\partial[\theta,\xi] = [\bar\partial\theta,\xi] + (-1)^{|\theta| + 1}[\theta,\bar\partial\xi], \quad \bar\partial(\theta \wedge \xi) = \bar\partial\theta \wedge \xi + (-1)^{|\theta|}\theta \wedge \bar\partial\xi,$$
  and a global section $\ell \in \G^{-1,2}$ with $\bar\partial^2(\theta) = [\ell,\theta]$ and $\bar\partial\ell = 0$. Then, it also satisfies $\bar\partial(1) = 0$ and $\ell \wedge \ell = 0$ as well as 
  $$\bar\partial(\varphi^n) = n\cdot \bar\partial\varphi \wedge \varphi^{n - 1}$$ for $n \geq 1$ and $|\varphi|$ even, where we write $\varphi^n := \varphi \wedge ... \wedge \varphi$.
  \item A \emph{differential bigraded Gerstenhaber algebra} of dimension $d$ is a curved Gerstenhaber algebra of dimension $d$ with $\ell = 0$.
  \item A Gerstenhaber algebra is \emph{centerless} if the map 
  $$\mathrm{ad}_T: \G^{-1} \to \cH om_\C(\G^{-1},\G^{-1}), \quad \theta \mapsto (\xi \mapsto [\theta,\xi]),$$
  is injective; it is called \emph{strictly faithful} if the map 
  $$\mathrm{ad}_F: \G^{-1} \to \cH om_\C(\G^0,\G^0), \quad \theta \mapsto (\xi \mapsto [\theta,\xi]),$$
  is injective; it is \emph{faithful} if the intersection of the kernels of all maps 
  $$\mathrm{ad}_{G^p}: \G^{-1} \to \cH om(\G^p,\G^p), \quad \theta \mapsto (\xi \mapsto [\theta,\xi]),$$
  is zero.
  \item A bigraded/dbg/curved Gerstenhaber algebra is \emph{centerless}\index{Gerstenhaber algebra!centerless} if the map 
  $$\mathrm{ad}_T^q: \G^{-1,q} \to \cH om_\C(\G^{-1,0},\G^{-1,q}), \quad \theta \mapsto (\xi \mapsto [\theta,\xi]),$$
  is injective for all $q \geq 0$; it is called \emph{strictly faithful}\index{Gerstenhaber algebra!(strictly) faithful} if the map 
  $$\mathrm{ad}_F^q: \G^{-1,q} \to \cH om_\C(\G^{0,0},\G^{0,q}), \quad \theta \mapsto (\xi \mapsto [\theta,\xi]),$$
  is injective for all $q \geq 0$; it is \emph{faithful} if the intersection of the kernels of all maps 
  $$\mathrm{ad}_{G^p}^q: \G^{-1,q} \to \cH om(\G^{p,0},\G^{p,q}), \quad \theta \mapsto (\xi \mapsto [\theta,\xi]),$$
  is zero for each $q \geq 0$.
  \item A bigraded/dbg/curved Gerstenhaber algebra is \emph{bounded}\index{Gerstenhaber algebra!bounded} if $\G^{p,q} = 0$ for $q >> 0$.
 \end{enumerate}
\end{defn}
\begin{rem}
 When we take $\F = \G^0$, $\T = \G^{-1}$, then we obtain a Lie--Rinehart algebra with $\ast^P = \wedge$ and $\nabla^P = [-,-]$.
\end{rem}
\begin{ex}\label{smooth-polyvector-exa}\note{smooth-polyvector-exa}
 Let $f: X \to S$ be a smooth morphism of relative dimension $d$. When we set $\G^p := \bigwedge^{-p} \T^1_{X/S}$ for $-d \leq p \leq 0$, then we obtain a Gerstenhaber algebra in the context $\mathfrak{Coh}(X/S)$ by taking for $[-,-]$ the Schouten--Nijenhuis bracket $[-,-]_{sn}$. When we take instead $[-,-] = -[-,-]_{sn}$---the \emph{negative} of the Schouten--Nijenhuis bracket---, then we obtain a Gerstenhaber algebra as well. We shall follow this latter convention.
\end{ex}

There are several notions of modules over a Gerstenhaber algebra possible, cf.~for example \cite{FeltenThesis}. The only notion interesting for us at this point is the one that fits to the de Rham complex as a module over the polyvector fields via contraction. We call a pair of a Gerstenhaber algebra and such a module (a \emph{de Rham module}) a \emph{Gerstenhaber calculus}.

\begin{defn}\label{Gerstenhaber-calc-def}\note{Gerstenhaber-calc-def}
 Let $(\cS,\C,\cO,\bC)$ be a context.
 \begin{enumerate}[label=(\alph*)]
  \item A \emph{Gerstenhaber calculus}\index{Gerstenhaber calculus} of dimension $d$
 $$\G\C^\bullet = (\G^\bullet,\wedge,1_G,[-,-],\A^{\bullet},\wedge,1_A,\partial,\,\invneg\, ,\cL)$$
 consists of a Gerstenhaber algebra of dimension $d$ and:
 \begin{itemize}
  \item an $\cO$-module $\A^i \in \bC$ of degree $i$ for every $0 \leq i \leq d$; for $i$ outside this range, we set $\A^i = 0$ whenever necessary;
  \item an $\cO$-bilinear product 
  $-\wedge - : \A^{i} \times \A^{i'} \to \A^{i+i'}$
  and a global section $1_A \in \A^{0}$ such that 
  $$(\alpha \wedge \beta) \wedge \gamma = \alpha \wedge (\beta \wedge \gamma); \quad \alpha \wedge \beta = (-1)^{|\alpha||\beta|}\beta \wedge \alpha; \quad 1_A \wedge \alpha = \alpha; $$
  \item a $\C$-linear map $\partial: \A^{i} \to \A^{i + 1}$ with $\partial^2 = 0$, satisfying the derivation rule
  $$\partial(\alpha \wedge \beta) = \partial(\alpha) \wedge \beta + (-1)^{|\alpha|} \alpha \wedge \partial(\beta),$$
  the \emph{de Rham differential};
  \item an $\cO$-bilinear map 
  $\invneg \ : \G^{p} \times \A^{i} \to \A^{p + i}$
  such that 
  $$1_G \ \invneg \ \alpha = \alpha \quad \mathrm{and} \quad (\theta \wedge \xi) \ \invneg \ \alpha = \theta \ \invneg \ (\xi \ \invneg \ \alpha);$$
  it is called the \emph{contraction map};
  \item a $\C$-bilinear map $\cL_{-}(-): \G^{p} \times \A^{i} \to \A^{p + i + 1}$ such that 
  $$\cL_{[\theta,\xi]}(\alpha) = \cL_\theta(\cL_\xi(\alpha)) - (-1)^{(|\theta| + 1)(|\xi| + 1)}\cL_\xi(\cL_\theta(\alpha));$$
  it is called the \emph{Lie derivative}.
  \end{itemize}
  They must satisfy additionally:
  \begin{itemize}
  \item the mixed Leibniz rule\index{mixed Leibniz rule}
  $$\theta \ \invneg \ \cL_\xi(\alpha) = (-1)^{|\xi| + 1}([\theta,\xi] \ \invneg \ \alpha) + (-1)^{|\theta|(|\xi| + 1)}\cL_\xi(\theta \ \invneg \ \alpha);$$
  \item the Lie--Rinehart homotopy formula\index{Lie--Rinehart homotopy formula}
  $$(-1)^{|\theta|}\cL_\theta(\alpha) = \partial(\theta \ \invneg \ \alpha) - (-1)^{|\theta|}(\theta \ \invneg \ \partial\alpha);$$
  consequently, we have 
  $$\cL_{\theta \wedge \xi}(\alpha) = (-1)^{|\xi|}\cL_\theta(\xi\ \invneg\ \alpha) + \theta\ \invneg\ \cL_\xi(\alpha);$$
  \item for $\theta \in \G^{0}$, the identity $\theta \ \invneg\ \alpha = (\theta \ \invneg \ 1_A) \wedge \alpha$;
  \item for $\theta \in \G^{-1}$, the identity
  $$\theta\ \invneg\ (\alpha \wedge \beta) = (\theta\ \invneg\ \alpha) \wedge \beta + (-1)^{|\alpha||\theta|}\alpha \wedge (\theta\ \invneg\ \beta);$$
  consequently, we have for $\theta \in \G^{-1}$
  $$\cL_\theta(\alpha \wedge \beta) = \cL_\theta(\alpha) \wedge \beta + (-1)^{|\alpha|(|\theta| + 1)} \alpha \wedge \cL_\theta(\beta);$$
  \item the map $\lambda: \G^0 \to \A^0, \ \theta \mapsto \theta \ \invneg \ 1_A,$ is an $\cO$-linear isomorphism with $\lambda(\theta \wedge \xi) = \lambda(\theta) \wedge \lambda(\xi)$.
 \end{itemize}
 Then we also have $\cL_{1_G}(\alpha) = 0$ as well as $\cL_\theta(1_A) = \partial(\theta \ \invneg \ 1_A)$ for $\theta \in \G^0$ and $\cL_\theta(1_A) = 0$ for $\theta \in \G^p$ with $p \leq - 1$. For $\varphi$ with $|\varphi|$ even and $n \geq 2$, we have 
 $$\cL_{\varphi^n}(\alpha) = n\varphi^{n - 1} \ \invneg \ \cL_\varphi(\alpha) + \frac{n(n - 1)}{2}([\varphi,\varphi] \wedge \varphi^{n - 2}) \ \invneg \ \alpha.$$
 The Lie--Rinehart homotopy formula yields $\partial\cL_\theta(\alpha) = (-1)^{|\theta| + 1} \cL_\theta(\partial\alpha)$.
 \item A \emph{bigraded Gerstenhaber calculus}\index{Gerstenhaber calculus!bigraded} of dimension $d$
 $$\G\C^{\bullet,\bullet} = (\G^{\bullet,\bullet},\wedge,1_G,[-,-],\A^{\bullet,\bullet},\wedge,1_A,\partial,\,\invneg\, ,\cL)$$
 consists of a bigraded Gerstenhaber algebra of dimension $d$ and:
 \begin{itemize}
  \item an $\cO$-module $\A^{i,j} \in \bC$ of bidegree $(i,j)$ and total degree $i + j$ for every $0 \leq i \leq d$ and every $j \geq 0$; for $(i,j)$ outside this range, we set $\A^{i,j} = 0$ whenever necessary;
  \item an $\cO$-bilinear product 
  $-\wedge - : \A^{i,j} \times \A^{i',j'} \to \A^{i+i',j + j'}$
  and a global section $1_A \in \A^{0,0}$ such that 
  $$(\alpha \wedge \beta) \wedge \gamma = \alpha \wedge (\beta \wedge \gamma); \quad \alpha \wedge \beta = (-1)^{|\alpha||\beta|}\beta \wedge \alpha; \quad 1_A \wedge \alpha = \alpha; $$
  \item a $\C$-linear map $\partial: \A^{i,j} \to \A^{i + 1,j}$ with $\partial^2 = 0$, satisfying the derivation rule
  $$\partial(\alpha \wedge \beta) = \partial(\alpha) \wedge \beta + (-1)^{|\alpha|} \alpha \wedge \partial(\beta),$$
  the \emph{de Rham differential};
  \item an $\cO$-bilinear map 
  $\invneg \ : \G^{p,q} \times \A^{i,j} \to \A^{p + i,q + j}$
  such that 
  $$1_G \ \invneg \ \alpha = \alpha \quad \mathrm{and} \quad (\theta \wedge \xi) \ \invneg \ \alpha = \theta \ \invneg \ (\xi \ \invneg \ \alpha);$$
  it is called the \emph{contraction map};
  \item a $\C$-bilinear map $\cL_{-}(-): \G^{p,q} \times \A^{i,j} \to \A^{p + i + 1,q + j}$ such that 
  $$\cL_{[\theta,\xi]}(\alpha) = \cL_\theta(\cL_\xi(\alpha)) - (-1)^{(|\theta| + 1)(|\xi| + 1)}\cL_\xi(\cL_\theta(\alpha));$$
  it is called the \emph{Lie derivative}.
  \end{itemize}
  They must satisfy additionally:
  \begin{itemize}
  \item the mixed Leibniz rule
  $$\theta \ \invneg \ \cL_\xi(\alpha) = (-1)^{|\xi| + 1}([\theta,\xi] \ \invneg \ \alpha) + (-1)^{|\theta|(|\xi| + 1)}\cL_\xi(\theta \ \invneg \ \alpha);$$
  \item the Lie--Rinehart homotopy formula
  $$(-1)^{|\theta|}\cL_\theta(\alpha) = \partial(\theta \ \invneg \ \alpha) - (-1)^{|\theta|}(\theta \ \invneg \ \partial\alpha);$$
  consequently, we have 
  $$\cL_{\theta \wedge \xi}(\alpha) = (-1)^{|\xi|}\cL_\theta(\xi\ \invneg\ \alpha) + \theta\ \invneg\ \cL_\xi(\alpha);$$
  \item for $\theta \in \G^{0,q}$, the identity $\theta \ \invneg\ \alpha = (\theta \ \invneg \ 1_A) \wedge \alpha$;
  \item for $\theta \in \G^{-1,q}$, the identity
  $$\theta\ \invneg\ (\alpha \wedge \beta) = (\theta\ \invneg\ \alpha) \wedge \beta + (-1)^{|\alpha||\theta|}\alpha \wedge (\theta\ \invneg\ \beta);$$
  consequently, we have for $\theta \in \G^{-1,q}$
  $$\cL_\theta(\alpha \wedge \beta) = \cL_\theta(\alpha) \wedge \beta + (-1)^{|\alpha|(|\theta| + 1)} \alpha \wedge \cL_\theta(\beta);$$
  \item the map $\lambda: \G^{0,q} \to \A^{0,q}, \ \theta \mapsto \theta \ \invneg \ 1_A,$ is an $\cO$-linear isomorphism with $\lambda(\theta \wedge \xi) = \lambda(\theta) \wedge \lambda(\xi)$.
 \end{itemize}
 As above, we also have $\cL_{1_G}(\alpha) = 0$ as well as $(-1)^{|\theta|}\cL_\theta(1_A) = \partial(\theta \ \invneg \ 1_A)$ for $\theta \in \G^{0,q}$ and $\cL_\theta(1_A) = 0$ for $\theta \in \G^{p,q}$ with $p \leq - 1$. For $\varphi$ with $|\varphi|$ even and $n \geq 2$, we have 
 $$\cL_{\varphi^n}(\alpha) = n\varphi^{n - 1} \ \invneg \ \cL_\varphi(\alpha) + \frac{n(n - 1)}{2}([\varphi,\varphi] \wedge \varphi^{n - 2}) \ \invneg \ \alpha.$$
 The Lie--Rinehart homotopy formula yields $\partial\cL_\theta(\alpha) = (-1)^{|\theta| + 1} \cL_\theta(\partial\alpha)$.
 \item A \emph{curved Gerstenhaber calculus}\index{Gerstenhaber calculus!curved} of dimension $d$
 $$\G\C^{\bullet,\bullet} = (\G^{\bullet,\bullet},\wedge,1_G,[-,-], \bar\partial,\ell,\A^{\bullet,\bullet},\wedge,1_A,\partial,\,\invneg\, ,\cL,\bar\partial)$$
 is simultaneously a curved Gerstenhaber algebra of dimension $d$ and a bigraded Gerstenhaber calculus of dimension $d$
 together with a $\C$-linear map $\bar\partial: \A^{i,j} \to \A^{i,j + 1}$ satisfying the derivation rules
 $$\bar\partial(\alpha \wedge \beta) = \bar\partial\alpha \wedge\beta + (-1)^{|\alpha|}\alpha \wedge\bar\partial\beta, \quad \bar\partial(\theta \ \invneg\ \alpha) = (\bar\partial\theta) \ \invneg \ \alpha + (-1)^{|\theta|} \theta \ \invneg\ \bar\partial\alpha$$
 as well as $\bar\partial^2(\alpha) = \cL_\ell(\alpha)$ and $\partial\bar\partial + \bar\partial\partial = 0$. The Lie--Rinehart homotopy formula yields 
 $$\bar\partial\cL_\theta(\alpha) = \cL_{\bar\partial\theta}(\alpha) + (-1)^{|\theta| + 1}\cL_\theta(\bar\partial\alpha).$$
 \item A \emph{differential bigraded Gerstenhaber calculus} of dimension $d$ is a curved Gerstenhaber calculus of dimension $d$ with $\ell = 0$.
 \item A Gerstenhaber calculus of dimension $d$ is \emph{centerless}\index{Gerstenhaber calculus!centerless} respective \emph{strictly faithful} if the underlying Gerstenhaber algebra is; \emph{faithfulness} is analogous to the previous definitions, taking into account all maps into $\cH om_\C(\G^{p,0},\G^{p,q})$ and $\cH om_\C(\A^{i,0},\A^{i,q})$.
 \item A bigraded/dbg/curved Gerstenhaber calculus of dimension $d$ is \emph{centerless} respective \emph{strictly faithful} if the underlying bigraded/dbg/curved Gerstenhaber algebra is; \emph{faithfulness} is analogous to the previous definitions.
 \item A bigraded/dbg/curved Gerstenhaber calculus of dimension $d$ is \emph{bounded}\index{Gerstenhaber calculus!bounded} if $\G^{p,q} = 0$ and $\A^{i,q} = 0$ for all $-d \leq p \leq 0$, $0 \leq i \leq d$ for $q >> 0$.
 \end{enumerate}
 
\end{defn}

\begin{rem}
 The symbol $\G\C^\bullet$ which we use above as a short-hand notation for a Gerstenhaber calculus is in analogy with our notation for general algebraic structures; the bullet is meant to run over the index set $D = [-d,0] \sqcup [0,d]$ although we will never use the notation $\G\C^P$ for specific values $P \in D$.
\end{rem}

\begin{ex}\label{smooth-Gerstenhaber-calc-exa}\note{smooth-Gerstenhaber-calc-exa}
 Let $f: X \to S$ be a smooth morphism of relative dimension $d$. Continuing Example~\ref{smooth-polyvector-exa}, we set $\A^i = \Omega^i_{X/S}$. Then we obtain a Gerstenhaber calculus of dimension $d$ in the context $\mathfrak{Coh}(X/S)$ with the de Rham differential $\partial$, the contraction map $\invneg$ as defined, for example, in \cite{FeltenThesis}, and the Lie derivative $\cL$ as defined by the Lie--Rinehart homotopy formula in the definition.
\end{ex}

\subsection{The local Batalin--Vilkovisky condition}

In a Gerstenhaber calculus $\G\C^\bullet = (\G^\bullet,\A^\bullet)$, the $\wedge$-product turns $\G^0$ and $\A^0$ into sheaves of commutative rings, which are isomorphic via $\lambda: \G^0 \to \A^0$. We denote this sheaf of rings occasionally by $\F$, consistent with our notation for Lie--Rinehart pairs. Every $\G^p$ and every $\A^i$ comes with a structure of $\F$-module, and both $\wedge$-products as well as the contraction $\invneg$ are $\F$-bilinear. 

\begin{defn}\label{G-calc-Gorenstein-defn}\note{G-calc-Gorenstein-defn}
 A Gerstenhaber calculus $(\G^\bullet,\A^\bullet)$ of dimension $d$ is \emph{Gorenstein}\footnote{We have chosen this name because, in applications, $\A^d$ is usually some canonical bundle of a space.}\index{Gerstenhaber calculus!Gorenstein} if $\A^d$ is a locally free $\F$-module of rank $1$.
\end{defn}

When $\G\C^\bullet$ is Gorenstein, then every local generator $\omega$ of $\A^d$ defines locally an $\F$-linear map 
$$\kappa_\omega: \enspace \G^p \to \A^{p + d}, \quad \theta \mapsto \theta \ \invneg \ \omega.$$
In geometric applications, for example on a smooth or, more generally, log smooth morphism $f: X \to S$, this map is usually an isomorphism of $\F$-modules. Moreover, when $\G\C^\bullet$ is in fact a (part of a) Batalin--Vilkovisky calculus as defined below, then $\A^d \cong \A^0$, and $\kappa_\omega$ is a global isomorphism. We consider the condition that $\kappa_\omega$, defined locally, is an isomorphism as a property of Gerstenhaber calculi.

\begin{defn}\label{loc-BV-defn}\note{loc-BV-defn}
 A Gerstenhaber calculus $(\G^\bullet,\A^\bullet)$ is \emph{locally Batalin--Vilkovisky}\index{Gerstenhaber calculus!locally Batalin--Vilkovisky} if it is Gorenstein, and the map $\kappa_\omega: \G^p \to \A^{p + d}$ is an isomorphism of $\F$-modules for every local generator $\omega$ of $\A^d$ as an $\F$-module.
\end{defn}

Note that it is sufficient to check this condition on some local generator in a neighborhood of each point.

We also introduce a bigraded variant of the two definitions. We write $\F^0 = \G^{0,0} = \A^{0,0}$.

\begin{defn}\label{loc-BV-bg-defn}\note{loc-BV-bg-defn}
 A bigraded Gerstenhaber calculus $(\G^{\bullet,\bullet},\A^{\bullet,\bullet})$ is \emph{Gorenstein} if we can find everywhere locally an element $\omega \in \A^{d,0}$ such that $v_\omega: \A^{0,j} \to \A^{d,j}, \: \alpha \mapsto \alpha \wedge \omega,$ is an isomorphism of $\F^0$-modules. It is \emph{locally Batalin--Vilkovisky} if, for every such $\omega$, the map $\kappa_\omega: \G^{p,q} \to \A^{p + d,q}, \: \theta \mapsto \theta \ \invneg \ \omega$, is an isomorphism of $\F^0$-modules.
\end{defn}

\begin{rem}
 We have to be careful with the notion of being locally Batalin--Vilkovisky for a curved Gerstenhaber calculus since $\kappa_\omega$ and $v_\omega$ are only compatible with $\bar\partial$ if $\bar\partial\omega = 0$. If we add this condition to the definition of being locally Batalin--Vilkovisky in the curved case, then we might end up with a Batalin--Vilkovisky calculus, as defined below, which is not locally Batalin--Vilkovisky.
\end{rem}

\subsection{The operator $d = \partial + \bar\partial + \ell \ \invneg \ (-)$ and the hypercohomology}

Let $(\G^{\bullet,\bullet},\A^{\bullet,\bullet})$ be a curved Gerstenhaber calculus in the context $(\cS,\C,\cO,\bC)$. Assume $\ell = 0$ for a moment. Then $(\A^{\bullet,\bullet},\partial,\bar\partial)$ is a double complex, and we can form its \emph{total complex}
$\mathrm{Tot}^\bullet(\A^{\bullet,\bullet},\partial,\bar\partial) = (\A^\bullet, d)$ with
$$\A^m := \bigoplus_{i + j = m} \A^{i,j}, \quad d := \partial + \bar\partial,$$
where $\A^m$ is a finite direct sum since both variables $i$ and $j$ are bounded below.
However, when $\ell \not= 0$, then possibly $d^2 \not= 0$.
Nonetheless, as was apparently first observed in \cite[Defn.~4.10]{ChanLeungMa2023}, we can define a differential on $\A^\bullet$ as the $\C$-linear operator
$$d: \A^m \to \A^{m + 1}, \quad \alpha \mapsto \partial\,\alpha + \bar\partial\, \alpha + (\ell \ \invneg\ \alpha),$$
called the \emph{total de Rham differential};\index{total de Rham differential} it still preserves the total degree, albeit that $\A^{i,j}$ is mapped into the direct sum of three pieces $\A^{i + 1,j}$, $\A^{i,j + 1}$, and $\A^{i - 1,j + 2}$ instead of two as in the case $\bar\partial^2 = 0$.

\begin{lemdef}\label{d-is-diff}\note{d-is-diff}
 The operator $d: \A^m \to \A^{m + 1}$ is a differential, i.e., $d^2 = 0$, and the derivation rule 
 $$d(\alpha \wedge \beta) = d\alpha \wedge \beta + (-1)^{|\alpha|}\alpha \wedge d\beta$$
 holds. We denote its cohomology by 
 $\HH^m(\A^{\bullet,\bullet}) := H^m(\A^\bullet, d)$ and call it the \emph{hypercohomology} of $(\G^{\bullet,\bullet},\A^{\bullet,\bullet})$.
\end{lemdef}
\begin{proof}
 The first assertion is a direct computation using $\partial(\ell \ \invneg\ \alpha) = -\cL_\ell(\alpha) - \ell \ \invneg\ \partial\alpha$. The second assertion follows from the derivation rules for $\partial$, $\bar\partial$, and $\ell \ \invneg\ (-)$.
\end{proof}

\begin{rem}
 The name comes from the following situation: Let $f: X \to S$ be a (log) smooth morphism of relative dimension $d$, and consider the Gerstenhaber calculus of Example~\ref{smooth-Gerstenhaber-calc-exa}. Then there is a differential bigraded Gerstenhaber calculus in $\mathfrak{QCoh}(X/S)$ which forms a flasque resolution of the original singly graded Gerstenhaber calculus. Its hypercohomology in the sense of the above Definition is exactly the hypercohomology of the de Rham complex.
\end{rem}

\subsection{Two-sided Gerstenhaber calculi}

A Gerstenhaber calculus $\G\C^\bullet$ as defined above comes with a contraction map $\invneg: \G^p \times \A^i \to \A^{p + i}$ which turns $\A^\bullet$ into a module over $\G^\bullet$. However, usually we can define a second analogous contraction map $\vdash\,: \G^p \times \A^i \to \G^{p + i}$, which turns $\G^\bullet$ into a module over $\A^\bullet$. This \emph{left contraction}, as we shall call it to distinguish it from $\invneg$, which we then call the \emph{right contraction}, has not been used in Chan, Leung, and Ma's article \cite{ChanLeungMa2023}, and we have become aware of it only lately. Most of the theory does not need the left contraction, but it is occasionally convenient to have it at our disposal, especially in the setting of enhanced generically log smooth families, which carry a Gerstenhaber calculus which is not determined by the geometry of the underlying generically log smooth family. We have decided to include the left contraction not in our original definition of a Gerstenhaber calculus but in a separate definition, that of a \emph{two-sided} Gerstenhaber calculus. In practice, many one-sided Gerstenhaber calculi carry a unique left contraction; we will see below that this is the case, for example, for Gerstenhaber calculi which are locally Batalin--Vilkovisky.\footnote{The local Batalin--Vilkovisky condition is broken by taking global sections. Thus, while we can construct the left contraction on the sheaf level in this case, the global sections carry a left contraction as well which is not easily obtained from the right contraction on the level of global sections.} One may desire that $\lambda(\theta \vdash \alpha) = \theta \ \invneg \ \alpha$ for $|\theta| + |\alpha| = 0$ so that $\vdash$ would be an extension of $\invneg$ to the range $|\theta| + |\alpha| \leq 0$, but this is not compatible with the other natural relations that we want to impose, in particular \eqref{left-contraction-def} below. We will have $\lambda(\theta \vdash \alpha) = (-1)^{|\theta|} \theta \ \invneg \ \alpha$ for $|\theta| + |\alpha| = 0$ instead, and, given the other relations, this cannot be rectified by writing $\alpha \dashv \theta$ since this would also change the natural signs in the other relations, for example \eqref{left-contraction-def}. At least, this convention yields the very satisfactory formula $[\theta,g] = \theta \vdash \partial(\lambda(g))$ for our conventions for the bracket $[-,-]$, which is the negative of the Schouten--Nijenhuis bracket in practice.

\begin{defn}\label{two-sided-G-calc-defn}\note{two-sided-G-calc-defn}
 Let $(\cS,\C,\cO,\bC)$ be a context.
 \begin{enumerate}[label=(\alph*)]
  \item A \emph{two-sided Gerstenhaber calculus}\index{Gerstenhaber calculus, two-sided} of dimension $d$ 
  $$\G\C^\bullet = (\G^\bullet,\wedge,1_G,[-,-],\A^{\bullet},\wedge,1_A,\partial,\,\invneg\, ,\cL, \vdash)$$
  is a Gerstenhaber calculus of dimension $d$ together with an $\cO$-bilinear operation 
  $$\vdash\,: \G^p \times \A^i \to \G^{p + i},$$
  the \emph{left contraction},\index{left contraction} satisfying the following relations:
  \begin{itemize}
   \item $\theta \vdash 1_A = \theta$ and $\theta \vdash (\alpha \wedge \beta) = (\theta \vdash \alpha) \vdash \beta$;
   \item for $\alpha \in \A^0$, the identity $\theta \vdash \alpha = \theta \wedge (1_G \vdash \alpha)$;
  \item for $\alpha \in \A^1$, the identity 
  $$(\theta \wedge \xi) \vdash \alpha = (-1)^{|\xi||\alpha|} (\theta \vdash \alpha) \wedge \xi + \theta \wedge (\xi \vdash \alpha);$$
  \item for $\theta \in \G^{-i}$ and $\alpha \in \A^i$, the identity $\lambda(\theta \vdash \alpha) = (-1)^{i}\,\theta \ \invneg \ \alpha$;
  \item for $\omega \in \A^d$, the identity
  \begin{equation}\label{left-contraction-def}
   (\theta \vdash \alpha) \ \invneg \ \omega = (-1)^{|\alpha||\omega|} (\theta \ \invneg \ \omega) \wedge \alpha;
  \end{equation}
  \item for $\omega^\vee \in \G^{-d}$, the identity
  \begin{equation}
   \omega^\vee \vdash (\theta \ \invneg \ \alpha) = (-1)^{|\omega^\vee||\theta|}\, \theta \wedge (\omega^\vee \vdash \alpha);
  \end{equation}
  \item for $\theta \in \G^{-1}$, the \emph{special left mixed Leibniz rule}\index{special left mixed Leibniz rule}\index{mixed Leibniz rule!special left}\footnote{This relation is dictated by the compatibility of $\vdash$ with gauge transforms, cf.~also Definition~\ref{Cartan-str-defn}. As stated, the relation does not hold for $\theta$ of other degrees than $-1$ in those Gerstenhaber calculi in which we are interested. Unfortunately, we could not find a general relation of which this is a special case.}
  $$[\theta,\xi \vdash \alpha] = [\theta,\xi] \vdash \alpha + \xi \vdash \cL_\theta(\alpha).$$
  \end{itemize}
  \item A \emph{bigraded two-sided Gerstenhaber calculus}\index{Gerstenhaber calculus, two-sided!bigraded} of dimension $d$
 $$\G\C^{\bullet,\bullet} = (\G^{\bullet,\bullet},\wedge,1_G,[-,-],\A^{\bullet,\bullet},\wedge,1_A,\partial,\,\invneg\, ,\cL, \vdash)$$
  is a bigraded Gerstenhaber calculus of dimension $d$ together with an $\cO$-bilinear map 
  $$\vdash\,: \G^{p,q} \times \A^{i,j} \to \G^{p + i,q + j},$$
  the \emph{left contraction}, satisfying the following relations:
  \begin{itemize}
   \item $\theta \vdash 1_A = \theta$ and $\theta \vdash (\alpha \wedge \beta) = (\theta \vdash \alpha) \vdash \beta$;
   \item for $\alpha \in \A^{0,j}$, the identity $\theta \vdash \alpha = \theta \wedge (1_G \vdash \alpha)$;
  \item for $\alpha \in \A^{1,j}$, the identity 
  $$(\theta \wedge \xi) \vdash \alpha = (-1)^{|\xi||\alpha|} (\theta \vdash \alpha) \wedge \xi + \theta \wedge (\xi \vdash \alpha);$$
  \item for $\theta \in \G^{-i,q}$ and $\alpha \in \A^{i,j}$, the identity $\lambda(\theta \vdash \alpha) = (-1)^{i}\,\theta \ \invneg \ \alpha$;\footnote{This relation does not acquire the usual sign from the total degree when we obtain a bigraded two-sided Gerstenhaber calculus as the Thom--Whitney resolution of a singly graded one.}
  \item for $\omega \in \A^{d,j}$, the identity $(\theta \vdash \alpha) \ \invneg \ \omega = (-1)^{|\alpha||\omega|} (\theta \ \invneg \ \omega) \wedge \alpha$;
  \item for $\omega^\vee \in \G^{-d,q}$, the identity $\omega^\vee \vdash (\theta \ \invneg \ \alpha) = (-1)^{|\omega^\vee||\theta|}\, \theta \wedge (\omega^\vee \vdash \alpha)$;
  \item for $\theta \in \G^{-1,q}$, the \emph{special left mixed Leibniz rule}
  $$[\theta,\xi \vdash \alpha] = [\theta,\xi] \vdash \alpha + (-1)^{(|\theta| + 1)|\xi|}\,\xi \vdash \cL_\theta(\alpha).$$
  \end{itemize}
  \item A \emph{curved two-sided Gerstenhaber calculus}\index{Gerstenhaber calculus, two-sided!curved} of dimension $d$
 $$\G\C^{\bullet,\bullet} = (\G^{\bullet,\bullet},\wedge,1_G,[-,-], \bar\partial,\ell,\A^{\bullet,\bullet},\wedge,1_A,\partial,\,\invneg\, ,\cL,\vdash,\bar\partial)$$
 is simultaneously a bigraded two-sided Gerstenhaber calculus of dimension $d$ and a curved Gerstenhaber calculus of dimension $d$ such that 
 $$\bar\partial(\theta \vdash \alpha) = \bar\partial\theta \vdash \alpha + (-1)^{|\theta|}\, \theta \vdash \bar\partial\alpha.$$
 \item A \emph{differential bigraded two-sided Gerstenhaber calculus} of dimension $d$ is a curved one with $\ell = 0$.
 \item The notions of faithfulness, strict faithfulness, centerlessness, and boundedness apply to the underlying one-sided Gerstenhaber calculus.
 \end{enumerate}
\end{defn}

 When $\G\C^\bullet$ is a one-sided Gerstenhaber calculus which is locally Batalin--Vilkovisky, then we can use the relation 
 $$(\theta \vdash \alpha) \ \invneg \ \omega = (-1)^{|\alpha||\omega|} (\theta \ \invneg \ \omega) \wedge \alpha$$
to \emph{define} the left contraction. Since this relation is $\A^0$-linear in $\omega$ on both sides, this definition is independent of the choice of local generator $\omega$ and thus globally consistent. Similarly, we can define the left contraction on a bigraded one-sided Gerstenhaber calculus which is locally Batalin--Vilkovisky. With this definition, most relations for $\vdash$ follow from the relations in a one-sided Gerstenhaber calculus, but not all.

\begin{lemma}\label{one-sided-two-sided-G-calc}\note{one-sided-two-sided-G-calc}
 Let $(\cS,\C,\cO,\bC)$ be a context.
 \begin{enumerate}[label=\emph{(\alph*)}]
  \item Let $\G\C^\bullet$ be a one-sided Gerstenhaber calculus of dimension $d$ which is locally Batalin--Vilkovisky. Define $\vdash$ as above. Then $\vdash\, : \G^p \times \A^i \to \G^{p + i}$ is $\cO$-bilinear and satisfies the following relations:
  \begin{itemize}
   \item $\theta \vdash 1_A = \theta$ and $\theta \vdash (\alpha \wedge \beta) = (\theta \vdash \alpha) \vdash \beta$;
   \item for $\alpha \in \A^0$, the identity $\theta \vdash \alpha = \theta \wedge (1_G \vdash \alpha)$;
  \item for $\theta \in \G^0$ and $\alpha \in \A^0$, the identity $\lambda(\theta \vdash \alpha) = \theta \ \invneg \ \alpha$;
  \item for $\theta \in \G^{-1}$ and $\alpha \in \A^1$, the identity $\lambda(\theta \vdash \alpha) = - \, \theta \ \invneg \ \alpha$;
  \item for $\omega^\vee \in \G^{-d}$, the identity $\omega^\vee \vdash (\theta \ \invneg \ \alpha) = (-1)^{|\omega^\vee||\theta|}\, \theta \wedge (\omega^\vee \vdash \alpha)$;
  \item for $\theta \in \G^{-1}$, the special left mixed Leibniz rule
  $$[\theta,\xi \vdash \alpha] = [\theta,\xi] \vdash \alpha + \xi \vdash \cL_\theta(\alpha).$$
  \end{itemize}
  In other words, only the identity 
  $$(\theta \wedge \xi) \vdash \alpha = (-1)^{|\xi||\alpha|} (\theta \vdash \alpha) \wedge \xi + \theta \wedge (\xi \vdash \alpha),$$
  for $\alpha \in \A^1$ as well as $\lambda(\theta \vdash \alpha) = (-1)^i \, \theta \ \invneg \ \alpha$ for $\theta \in \G^{-i}$ and $\alpha \in \A^i$ with $i \geq 2$ are missing.
  \item Let $\G\C^{\bullet,\bullet}$ be a bigraded one-sided Gerstenhaber calculus of dimension $d$ which is locally Batalin--Vilkovisky. Define $\vdash$ as above. Then $\vdash: \G^{p,q} \times \A^{i,j} \to \G^{p + i, q + j}$ is an $\cO$-bilinear map satisfying all relations from the definition of a bigraded two-sided Gerstenhaber calculus except for possibly the relation for $(\theta \wedge \xi) \vdash \alpha$ and $\lambda(\theta \vdash \alpha) = (-1)^i \, \theta \ \invneg \ \alpha$ for $i \geq 2$.\footnote{As above, in the curved case, if $\bar\partial\omega\not= 0$, we cannot conclude that the relation for $\bar\partial(\theta \vdash \alpha)$ holds automatically.}
 \end{enumerate}
\end{lemma}
\begin{proof}
 We start with the singly graded case. That $\vdash$ respects the indicated degrees follows from the degrees in the defining relation of $\vdash$. Similarly, since both $\wedge$ and $\invneg$ are $\cO$-bilinear, we find that $\vdash$ is $\cO$-bilinear. The claimed relations now all follow from applying the local isomorphism $\kappa_\omega(-) = (-) \ \invneg \ \omega$ to the relation. To obtain the special left mixed Leibniz rule, one uses the usual mixed Leibniz rule. The proof of $\lambda(\theta \vdash \alpha) = -\, \theta \ \invneg \ \alpha$ for $i = 1$ uses the relation for $\theta \ \invneg \ (\alpha \wedge \beta)$, and that we have no such relation for $|\theta| \not= -1$ is the reason why we cannot show $\lambda(\theta \vdash \alpha) = (-1)^i \, \theta \ \invneg \ \alpha$ for $i \geq 2$ at this point. The proof in the bigraded case is similar. The computation for the bigraded special left mixed Leibniz rule is very tedious due to many signs that need to be tracked.
\end{proof}

\section{Batalin--Vilkovisky algebras and calculi}\label{BV-alg-sec}\note{BV-alg-sec}

A \emph{Batalin--Vilkovisky algebra} is a Gerstenhaber algebra with an additional operator $\Delta: \G^p \to \G^{p + 1}$. This operator is usually obtained inside a Gerstenhaber \emph{calculus} from the de Rham differential $\partial$ by means of an isomorphism $\G^p \cong \A^{p + d}$ induced by a chosen global volume form $\omega \in \A^d$, which exists on Calabi--Yau spaces. For theoretical purposes, it is also worthwhile to study Batalin--Vilkovisky algebras without a de Rham module. We fix an integer $d \geq 1$.

\begin{defn}
 Let $(\cS,\C,\cO,\bC)$ be a context.
 \begin{enumerate}[label=(\alph*)]
  \item A \emph{Batalin--Vilkovisky algebra}\index{Batalin--Vilkovisky algebra} of dimension $d$
  $$(\G^\bullet,\wedge,1_G,[-,-],\Delta)$$
  is a Gerstenhaber algebra of dimension $d$ together with a $\C$-linear \emph{Batalin--Vilkovisky operator}\index{Batalin--Vilkovisky operator} $\Delta: \G^p \to \G^{p + 1}$ which satisfies 
  $$\Delta(1_G) = 0, \quad \Delta^2 = 0, \quad \Delta[\theta,\xi] = [\Delta(\theta),\xi] + (-1)^{|\theta| + 1}[\theta,\Delta(\xi)],$$
  and the \emph{Bogomolov--Tian--Todorov formula}\index{Bogomolov--Tian--Todorov formula}
  $$(-1)^{|\theta|}[\theta,\xi] = \Delta(\theta \wedge \xi) - \Delta(\theta) \wedge \xi - (-1)^{|\theta|}\theta \wedge \Delta(\xi).$$
  \item A \emph{bigraded Batalin--Vilkovisky algebra}\index{Batalin--Vilkovisky algebra!bigraded} of dimension $d$
  $$(\G^{\bullet,\bullet},\wedge,1_G,[-,-],\Delta)$$
  is a bigraded Gerstenhaber algebra of dimension $d$ together with a $\C$-linear \emph{Batalin--Vilkovisky operator} $\Delta: \G^{p,q} \to \G^{p + 1,q}$ which satisfies 
  $$\Delta(1_G) = 0, \quad \Delta^2 = 0, \quad \Delta[\theta,\xi] = [\Delta(\theta),\xi] + (-1)^{|\theta| + 1}[\theta,\Delta(\xi)],$$
  and the \emph{Bogomolov--Tian--Todorov formula}
  $$(-1)^{|\theta|}[\theta,\xi] = \Delta(\theta \wedge \xi) - \Delta(\theta) \wedge \xi - (-1)^{|\theta|}\theta \wedge \Delta(\xi).$$
  \item A \emph{curved Batalin--Vilkovisky algebra}\index{Batalin--Vilkovisky algebra!curved} of dimension $d$
  $$(\G^{\bullet,\bullet},\wedge,1_G,[-,-],\Delta,\bar\partial,\ell,y)$$
  is a bigraded Batalin--Vilkovisky algebra of dimension $d$ together with a $\C$-linear operator $\bar\partial: \G^{p,q} \to \G^{p,q + 1}$ satisfying the two derivation rules 
   $$\bar\partial[\theta,\xi] = [\bar\partial\theta,\xi] + (-1)^{|\theta| + 1}[\theta,\bar\partial\xi], \quad \bar\partial(\theta \wedge \xi) = \bar\partial\theta \wedge \xi + (-1)^{|\theta|}\theta \wedge \bar\partial\xi,$$
  and global sections $\ell \in \G^{-1,2}$ and $y \in \G^{0,1}$ which satisfy
  $$\bar\partial^2(\theta) = [\ell,\theta], \quad \bar\partial(\ell) = 0, \quad \bar\partial\Delta(\theta) + \Delta\bar\partial(\theta) = [y,\theta], \quad \bar\partial(y) + \Delta(\ell) = 0.$$
  We then have 
  $$\bar\partial(1_G) = 0,\quad \ell \wedge \ell = 0, \quad y \wedge y = 0, \quad y \wedge \ell + \ell \wedge y = 0, \quad [\ell,y] + [y,\ell] = 0.$$
  \item A \emph{differential bigraded Batalin--Vilkovisky algebra} of dimension $d$ is a curved Batalin--Vilkovisky algebra of dimension $d$ with $\ell = 0$ and $y = 0$.
  \item The notions of \emph{centerless}\index{Batalin--Vilkovisky algebra!centerless}, \emph{strictly faithful}, \emph{faithful}, and \emph{bounded}\index{Batalin--Vilkovisky algebra!bounded} apply to the underlying bigraded/dbg/curved Gerstenhaber algebra.
 \end{enumerate}

\end{defn}

\begin{rem}
 The condition $\bar\partial(\ell) + \Delta(y) = 0$ is not explicitly discussed in \cite{ChanLeungMa2023}; however, the condition is satisfied in their context, and we need it to show Lemma~\ref{check-d-is-diff}. Moreover, the proof of Proposition~\ref{BV-calc-construction} suggests that the condition is natural.
\end{rem}

A \emph{Batalin--Vilkovisky calculus} is the analog of a Batalin--Vilkovisky algebra for a Gerstenhaber calculus. 

\begin{defn}
 Let $(\cS,\C,\cO,\bC)$ be a context.
 \begin{enumerate}[label=(\alph*)]
  \item A \emph{Batalin--Vilkovisky calculus}\index{Batalin--Vilkovisky calculus} of dimension $d$
  $$(\G^{\bullet},\wedge,1_G,[-,-],\Delta,\A^{\bullet},\wedge,1_A,\partial,\invneg\:,\cL,\omega)$$
  is simultaneously a bigraded Gerstenhaber calculus of dimension $d$ and a bigraded Batalin--Vilkovisky algebra of dimension $d$ together with a global section $\omega \in \A^{d}$ such that 
  $$\kappa: \G^{p} \to \A^{p + d}, \quad \theta \mapsto \theta \ \invneg\ \omega,$$
  is an isomorphism of $\cO$-modules, 
  $$v: \A^{0} \to \A^{d}, \quad \alpha \mapsto \alpha \wedge \omega,$$
  is an isomorphism of $\cO$-modules, and 
  $$\Delta(\theta) \ \invneg \ \omega = \partial(\theta \ \invneg \ \omega)$$
  holds.
  \item A \emph{bigraded Batalin--Vilkovisky calculus}\index{Batalin--Vilkovisky calculus!bigraded} of dimension $d$
  $$(\G^{\bullet,\bullet},\wedge,1_G,[-,-],\Delta,\A^{\bullet,\bullet},\wedge,1_A,\partial,\invneg\:,\cL,\omega)$$
  is simultaneously a Gerstenhaber calculus of dimension $d$ and a Batalin--Vilkovisky algebra of dimension $d$ together with a global section $\omega \in \A^{d,0}$ such that 
  $$\kappa: \G^{p,q} \to \A^{p + d,q}, \quad \theta \mapsto \theta \ \invneg\ \omega,$$
  is an isomorphism of $\cO$-modules, 
  $$v: \A^{0,q} \to \A^{d,q}, \quad \alpha \mapsto \alpha \wedge \omega,$$
  is an isomorphism of $\cO$-modules, and 
  $$\Delta(\theta) \ \invneg \ \omega = \partial(\theta \ \invneg \ \omega)$$
  holds.
  \item A \emph{curved Batalin--Vilkovisky calculus}\index{Batalin--Vilkovisky calculus!curved} of dimension $d$
  $$(\G^{\bullet,\bullet},\wedge,1_G,[-,-],\Delta,\bar\partial,\ell,y,\A^{\bullet,\bullet},\wedge,1_A,\partial,\invneg\:,\cL,\bar\partial,\omega)$$
  is simultaneously a curved Gerstenhaber calculus of dimension $d$ and a bigraded Batalin--Vilkovisky calculus of dimension $d$ such that $y \ \invneg \ \omega = \bar\partial \omega$.
  \item A \emph{differential bigraded Batalin--Vilkovisky calculus} of dimension $d$ is a curved Batalin--Vilkovisky calculus of dimension $d$ with $\ell = 0$ and $y = 0$.
  \item The notions of \emph{centerless},\index{Batalin--Vilkovisky calculus!centerless} \emph{strictly faithful}, \emph{faithful}, and \emph{bounded}\index{Batalin--Vilkovisky calculus!bounded} apply to the underlying bigraded/dbg/curved Gerstenhaber calculus.
 \end{enumerate}
\end{defn}
\begin{ex}
 Let $f: X \to S$ be a smooth morphism of Noetherian schemes of relative dimension $d$. Assume that $\omega_{X/S} \cong \cO_X$. Then the polyvector fields together with the de Rham complex and a chosen volume form $\omega \in \omega_{X/S}$ form a Batalin--Vilkovisky calculus of dimension $d$ in the context $\mathfrak{Coh}(X/S)$. Namely, we use the Gerstenhaber calculus of Example~\ref{smooth-Gerstenhaber-calc-exa}, and then we apply Proposition~\ref{BV-calc-construction} below.
\end{ex}

\subsection{Batalin--Vilkovisky calculi from Gerstenhaber calculi}

Batalin--Vilkovisky calculi arise naturally from Gerstenhaber calculi in the Calabi--Yau setting.

\begin{prop}\label{BV-calc-construction}\note{BV-calc-construction}
 Let $(\cS,\C,\cO,\bC)$ be a context.
 \begin{enumerate}[label=\emph{(\alph*)}]
  \item Let $(\G^\bullet,\A^\bullet)$ be a Gerstenhaber calculus of dimension $d$, and let $\omega \in \A^d$ be a global section such that 
  $$\kappa: \G^p \to \A^{p + d}, \quad \theta \mapsto (\theta \ \invneg \ \omega),$$
  is an isomorphism of $\cO$-modules for every $-d \leq p \leq 0$. Then $(\G^\bullet, \A^\bullet)$ is a Batalin--Vilkovisky calculus with $\Delta := \kappa^{-1} \circ \partial \circ \kappa$.
  \item Let $(\G^{\bullet,\bullet},\A^{\bullet,\bullet})$ be a bigraded Gerstenhaber calculus of dimension $d$, and let $\omega \in \A^{d,0}$ be a global section such that 
  $$\kappa: \G^{p,q} \to \A^{p + d,q}, \quad \theta \mapsto (\theta \ \invneg \ \omega),$$
  is an isomorphism of $\cO$-modules for all $-d \leq p \leq 0$ and $q \geq 0$. Then $(\G^{\bullet,\bullet},\A^{\bullet,\bullet})$ is a bigraded Batalin--Vilkovisky calculus with $\Delta := \kappa^{-1} \circ \partial \circ \kappa$.
  \item Let $(\G^{\bullet,\bullet},\A^{\bullet,\bullet})$ be a curved Gerstenhaber calculus of dimension $d$, and let $\omega \in \A^{d,0}$ be a global section such that 
  $$\kappa: \G^{p,q} \to \A^{p + d,q}, \quad \theta \mapsto (\theta \ \invneg \ \omega),$$
  is an isomorphism of $\cO$-modules for all $-d \leq p \leq 0$ and $q \geq 0$. Then $(\G^{\bullet,\bullet},\A^{\bullet,\bullet})$ is a curved Batalin--Vilkovisky calculus with $\Delta := \kappa^{-1} \circ \partial \circ \kappa$ and $y := \kappa^{-1}(\bar\partial\omega) \in \G^{0,1}$.
 \end{enumerate}
\end{prop}
\begin{proof}
 Since $v \circ \lambda = \kappa$, the map $v: \A^0 \to \A^d, \, \alpha \mapsto \alpha \wedge \omega,$ must be an $\cO$-linear isomorphism.  With our definition, $\Delta$ is a $\C$-linear operator with $\Delta(1) = 0$ (by degree reasons) and $\Delta^2 = 0$. The Bogomolov--Tian--Todorov formula can be shown by evaluating
 $$\left(\Delta(\theta \wedge \xi) - \Delta(\theta) \wedge \xi - (-1)^{|\theta|}\theta \wedge \Delta(\xi)\right) \ \invneg \ \omega$$
 with the Lie--Rinehart homotopy formula and then applying the mixed Leibniz rule. The derivation rule for $\Delta[\theta,\xi]$ then follows by applying $\Delta$ to the Bogomolov--Tian--Todorov formula. This concludes the proof in the singly graded case. The bigraded case is similar.
 
 To show $\bar\partial\Delta + \Delta\bar\partial = [y,-]$ in the curved case, note that
 $$(\bar\partial\Delta(\theta) + \Delta\bar\partial(\theta)) \ \invneg \ \omega = \left((-1)^{|\theta|}\Delta\theta\ \invneg\ \bar\partial\omega\right) - \cL_\theta(\bar\partial\omega).$$
 The Bogomolov--Tian--Todorov formula yields
 $$-[y, \theta] \ \invneg\ \omega = (-1)^{|\theta|}\partial(\theta \ \invneg\ \bar\partial\omega) + (-1)^{|\theta| + 1}\Delta\theta \ \invneg \ \bar\partial\omega.$$
 Then the Lie--Rinehart homotopy formula applied to the left summand on the right yields
 $$(\bar\partial\Delta(\theta) + \Delta\bar\partial(\theta)) \ \invneg \ \omega = [y, \theta] \ \invneg \ \omega.$$
 Next, by a direct computation
 $$(\bar\partial y + \Delta\ell) \ \invneg\ \omega = \bar\partial^2(\omega) - \cL_\ell(\omega) - (y \wedge y) \ \invneg\ \omega = 0$$
 since $y \wedge y = 0$ (because $y \in \G^{0,1}$), so $\bar\partial (y) + \Delta(\ell) = 0$. The remaining condition $y \ \invneg \ \omega = \bar\partial \omega$ is exactly the construction of $y$.
\end{proof}

\subsection{Two-sided Batalin--Vilkovisky calculi}

Similar to the case of Gerstenhaber calculi, we define a two-sided Batalin--Vilkovisky calculus as one endowed with a left contraction $\vdash$. Since every Batalin--Vilkovisky calculus $\G\C^\bullet$ comes already equipped with a volume form $\omega \in \A^d$ which induces isomorphisms $\kappa: \G^p \cong \A^{p + d}$, there is only one possible $\vdash$ compatible with \eqref{left-contraction-def}. By Lemma~\ref{one-sided-two-sided-G-calc}, it satisfies already most of the relations imposed in a two-sided Gerstenhaber calculus.

\begin{defn}\label{two-sided-BV-calc-def}\note{two-sided-BV-calc-def}
 A Batalin--Vilkovisky calculus $\G\C^\bullet$ of dimension $d$ is \emph{two-sided}\index{Batalin--Vilkovisky calculus!two-sided} if the left contraction $\vdash\,: \G^p \times \A^i \to \G^{p + i}$ turns it into a two-sided Gerstenhaber calculus. Bigraded two-sided Batalin--Vilkovisky calculi and curved two-sided Batalin--Vilkovisky calculi are defined analogously. 
\end{defn}

We do not impose any additional relations. In particular, the obvious analog of Proposition~\ref{BV-calc-construction} holds.

\subsection{The operator $\check d = \Delta + \bar\partial + (\ell + y) \wedge (-)$}

Let $\G^{\bullet,\bullet}$ be a curved Batalin--Vilkovisky \emph{algebra}. Assume $\ell = 0$ and $y = 0$ for a moment. Then $(\G^{\bullet,\bullet},\Delta,\bar\partial)$ is a double complex, and we can form its \emph{total complex} $(\G^\bullet,\check d)$ as 
$$\G^m := \bigoplus_{p + q = m} \G^{p,q}, \quad \check d := \Delta + \bar\partial;$$
here, $\G^m$ is a finite direct sum since both $p$ and $q$ are bounded below. If $\ell \not= 0$ or $y \not= 0$, then we may have $\check d^2 \not= 0$. To correct this, we set 
$$\check d: \G^m \to \G^{m + 1}, \quad \theta \mapsto \Delta\theta + \bar\partial\theta + (\ell + y) \wedge \theta;$$
this is a $\C$-linear operator compatible with the total degree but not the bidegree.

\begin{lemma}\label{check-d-is-diff}\note{check-d-is-diff}
 The operator $\check d$ is a $\C$-linear differential, i.e., we have $\check d^2 = 0$.
 We denote the cohomology of $(\G^\bullet,\check d)$ by $\HH^\bullet(\G^{\bullet,\bullet})$.
\end{lemma}
\begin{proof}
 A direct computation yields 
 $$\check d^2(\theta) = \Delta (y) \wedge \theta + \bar\partial (\ell) \wedge \theta + \Delta (\ell)\wedge \theta + \bar\partial (y) \wedge \theta;$$
 it is here that we use our assumptions $\bar\partial(\ell) = 0$ and $\bar\partial (y) + \Delta(\ell) = 0$ in order to conclude $\check d^2(\theta) = 0$ (note that $\Delta y = 0$ by degree reasons).
\end{proof}
\begin{rem}
 The derivation rules for $\check d(\theta \wedge \xi)$ and $\check d[\theta,\xi]$ which one may naively expect do not hold.
\end{rem}

In a curved Batalin--Vilkovisky \emph{calculus} $(\G^{\bullet,\bullet},\A^{\bullet,\bullet})$, the operator $\check d$ is nothing but the translation of $d$ to $\G^{\bullet,\bullet}$ using the isomorphism $\kappa$.

\begin{lemma}\label{d-check-d-comp}\note{d-check-d-comp}
  Let $(\G^{\bullet,\bullet},\A^{\bullet,\bullet})$ be a curved Batalin--Vilkovisky calculus. Then $\kappa \circ \check d = d \circ \kappa$.
\end{lemma}
\begin{proof}
 This is a short and direct computation using $(-1)^{|\theta|}\theta \ \invneg\ \bar\partial\omega = (y \wedge \theta) \ \invneg\ \omega$.
\end{proof}

\section{Semi-Cartan structures and Cartan structures}\label{Cartan-str-sec}\note{Cartan-str-sec}

We introduce two notions of additional structures on an algebraic structure $\cP$ to capture the similarities between our above definitions: \emph{semi-Cartan structures} and \emph{Cartan structures}. The name is meant to suggest that an algebraic structure with a (semi-)Cartan structure on it is similar to the Cartan calculus in differential geometry.

When discussing (semi-)Cartan structures on an algebraic structure $\cP$ abstractly, the axioms of $\cP$ play no role. If all the data of an algebraic structure except axioms are given, we call this a \emph{pre-algebraic structure}. If $\cP$ is a pre-algebraic structure, then we call a representation of $\cP$ in a context a $\cP$-\emph{pre-algebra}.

Our notion of a context $(\cS,\C,\cO,\bC)$ allows us to distinguish between $\cO$-linear maps and $\C$-linear maps. The notion of a \emph{multilinear differential operator of finite order} interpolates between the two. The classical notion of a linear differential operator can be found e.g.~in \cite[0G3Q]{stacks}. Here is our multilinear variant.

\begin{defn}
 A $\C$-multilinear map 
 $$\mu: P_1 \times ... \times P_n \to Q$$
 is a \emph{differential operator of order $0$} if it is $\cO$-multilinear. Then, inductively, for $n \geq 1$, it is a \emph{differential operator of order $N + 1$} if, for all indices $1 \leq \ell \leq n$ and all local sections $a \in \Gamma(U,\cO_X)$, the induced map 
 $$D_{\ell;a}(\mu): P_1|_U \times ... \times P_n|_U \to Q|_U, \quad (p_1,...,p_n) \mapsto \mu(p_1,...,a \cdot p_\ell,...,p_n) - a \cdot \mu(p_1,...,p_n),$$
 is a differential operator of order $N$. The map $\mu$ is a \emph{differential operator of finite order} if it is a differential operator of order $N$ for some $N \geq 0$. Differential operators of order $N$ form a sheaf $\D iff_{\cO/\C}^N(P_1,...,P_n;Q)$.
\end{defn}
\begin{rem} 
 \hspace{1cm}
 \begin{enumerate}[label=(\arabic*)]
  \item If a map is a differential operator of order $N$, it is also one of any order $M \geq N$. If $\mu$ is a differential operator of order $N$, and $\nu_1,...,\nu_n$ are differential operators of order $M$, then the composition $\mu(\nu_1,...,\nu_n)$ is a differential operator of order $M + N$.
  \item Differential operators of finite order play an important role in the context $\mathfrak{Coh}(X/S)$ since they admit base change along morphisms $T \to S$, which we do not have for general $\C$-multilinear maps.
  \item If $\F$ is a sheaf of $\cO$-algebras, and all sheaves $P_1,...,P_n,Q$ are sheaves of $\F$-modules, then the same definition of multilinear differential operators makes sense with $\cO$ replaced by $\F$.
 \end{enumerate}

\end{rem}

A (pre-)algebraic structure $\cP$ itself does not keep track of the $\cO$-module structure and possible differential operator properties of the maps. We use \emph{semi-Cartan structures} for this purpose.

\begin{defn}
 Let $\cP$ be a pre-algebraic structure. Then a \emph{semi-Cartan structure}\index{semi-Cartan structure} on $\cP$ 
 $$(F,\:1_F,\:(\ast^P)_{P \in D},\:\cP_\bullet)$$
 consists of a domain $F \in D(\cP)$ of degree $|F| = 0$, a  constant $1_F \in \cP(F)$, an operation $\ast^P \in \cP(F,P;P)$ for every $P \in D(\cP)$, and an exhaustive increasing filtration 
 $$\cP_0(P_1,...,P_n;Q) \subseteq \cP_1(P_1,...,P_n;Q) \subseteq ... \subseteq \cP(P_1,...,P_n;Q)$$
 with $\ast^P \in \cP_0(F,P;P)$. A $\cP$-pre-algebra $\E^\bullet$ is \emph{semi-Cartanian}\index{semi-Cartanian} if $\ast^P: \E^F \times \E^P \to \E^P$ is $\cO$-bilinear, we have
 $$a \ast^F b = b \ast^F a, \quad a \ast^P( b \ast^P p) = (a \ast^F b) \ast^P p, \quad 1_F \ast^P p = p,$$
 i.e., $\F := \E^F$ is an $\cO$-algebra with unit $1_F$ and $\E^P$ is an $\F$-module via $\ast^P$, and, for every $\mu \in \cP_N(P_1,...,P_n;Q)$, the map 
 $$\mu: \E^{P_1} \times ... \times \E^{P_n} \to \E^Q$$
 is a multilinear differential operator of order $N$ with respect to $\F/\C$. In particular, it is a multilinear differential operator of order $N$ with respect to $\cO/\C$. 
\end{defn}

In the context $\mathfrak{Coh}(X/S)$, we will usually have $\E^F = \cO_X$, but in the context $\mathfrak{Comp}(\Lambda)$, we have $\cO = \C$, and $\E^F$ will be often a non-trivial $\cO$-algebra.

\begin{ex}
 The following algebraic structures carry a natural semi-Cartan structure in their singly graded version:
 \begin{enumerate}[label=(\alph*)]
 \item Lie--Rinehart algebras; we have $\ast^P \in \cP_0(F,P;P)$ and $\nabla^P \in \cP_1(T,P;P)$;
 \item Lie--Rinehart pairs;
 \item Gerstenhaber algebras of dimension $d$; we take $F = \G^0$; the operation $\ast^P$ is given by the $\wedge$-product with $F$; for the filtration, we have $\wedge \in \cP_0$ and $[-,-] \in \cP_1$;
 \item Gerstenhaber calculi of dimension $d$; we take $F = \G^0$; for $\A^i$, the operation $\ast^P$ is given by the $\wedge$-product with $\A^0$ via the isomorphism $\lambda$; for the filtration, we have $\invneg \in \cP_0$, $\partial \in \cP_1$, $\cL \in \cP_1$;
 \item Batalin--Vilkovisky algebras and calculi of dimension $d$ with the same semi-Cartan structure as the underlying Gerstenhaber algebras and calculi; for the filtration, we have $\Delta \in \cP_1$.
 \end{enumerate}
 They carry a natural semi-Cartan structure in their bigraded and curved versions as well, but for these versions, we will introduce separate conventions below. For each of these algebraic structures, a $\cP$-algebra as defined above is automatically semi-Cartanian. Note that the condition that a map should be not only $\C$-multilinear but $\cO$-multilinear, which recurs often in our above definitions, fits nicely into the framework of semi-Cartan structures, while an algebraic structure itself has no direct means of distinguishing $\C$-multilinear from $\cO$-multilinear maps. Also note that the algebraic structure of Lie algebras does not naturally carry a semi-Cartan structure since it has no object $F$.
\end{ex}

A \emph{Cartan structure} also carries information about the tangent sheaf and the action by derivations.

\begin{defn}\label{Cartan-str-defn}\note{Cartan-str-defn}
 Let $\cP$ be a pre-algebraic structure. Then a \emph{Cartan structure}\index{Cartan structure} on $\cP$ 
 $$(F,\:T,\: Z,\:1_F,\:(\ast^P)_{P \in D},\: (\nabla^P)_{P \in D},\:\cP_\bullet)$$
 is a semi-Cartan structure together with a domain $T \in D(\cP)$ of degree $|T| = -1$ (in particular, $F \not= T$), a subset $Z \subseteq D(\cP)$ with $F \in Z$,\footnote{See the end of the section for the purpose of $Z$.} and an operation $\nabla^P \in \cP_1(T,P;P)$ for every domain $P \in D(\cP)$. We assume furthermore that 
 $$\cP(P_1,...,P_n;Q) = \emptyset \enspace \mathrm{if} \ n \geq 3.$$
 A $\cP$-pre-algebra is \emph{Cartanian}\index{Cartanian} if it is semi-Cartanian, satisfies $\nabla^T_{\theta}(\xi) = - \nabla_\xi^T(\theta)$ for $\theta,\xi \in \T$ and
 $$a \ast^M \nabla^M_\theta(m) = \nabla^M_{a \ast^T \theta}(m)$$
 for every $M \in Z$ (in particular for $M = F$),
 has $\nabla_\theta^P(\gamma) = 0$ for all constants $\gamma \in \cP(P)$, and for all $\mu \in \cP(P_1,...,P_n;Q)$ (for $n \geq 1$), the relation 
 \begin{equation}\label{partial-unifying-eqn}
  \nabla^Q_\theta(\mu(p_1,...,p_n)) = \sum_{s = 1}^n \mu(p_1,...,\nabla^{P_s}_\theta(p_s),...,p_n)
 \end{equation}
  holds.
\end{defn}

Lie--Rinehart algebras, Gerstenhaber algebras, and (one- and two-sided) Gerstenhaber calculi (and, technically, also their bigraded counterparts) carry a natural Cartan structure with $Z = \{F\}$. We have $\nabla^P_\theta(\xi) = [\theta,\xi]$ respective $\nabla^P_\theta(\alpha) = \cL_\theta(\alpha)$. Lie--Rinehart \emph{pairs} carry a natural Cartan structure with $Z = \{F,E\}$. In all these cases, the $\cP$-algebras are automatically Cartanian. Batalin--Vilkovisky algebras and calculi also carry a natural Cartan structure---the same as the underlying Gerstenhaber algebras or calculi---, but the $\cP$-algebras are not Cartanian since we have $\nabla_\theta(\Delta(\xi)) = [\theta,\Delta(\xi)] \not= \Delta[\theta,\xi] = \Delta(\nabla_\theta(\xi))$. (Similarly, the curved algebras are not Cartanian since $\nabla^P$ is not compatible with the predifferential $\bar\partial$, and since $\nabla_\theta(\ell) \not= 0$ in general; however, we will introduce a modified version of Cartanianity below to account for this.)

\begin{rem}
 The relation \eqref{partial-unifying-eqn} can be considered as a partial unification of many relations that we have required in the definitions. For example, it yields the Poisson identities, the definition of the Lie bracket, and the Jacobi identity in a Lie--Rinehart algebra. It yields with the Jacobi identity and the derivation rule two of the three relations of a Lie--Rinehart module. Ultimately, a Cartan structure on an algebraic structure determines a Lie--Rinehart algebra inside it and exhibits every domain in $D(\cP)$ as a sort of weak Lie--Rinehart module over it (not necessarily satisfying the third condition of a Lie--Rinehart module unless $P \in Z$).
\end{rem}

\begin{rem}
 The condition $\cP(P_1,...,P_n;Q) = \emptyset$ for $n \geq 3$ plays an important role when we construct the Thom--Whitney resolution of a Cartanian $\cP$-algebra in $\mathfrak{Coh}(X/S)$. In the general case, it is not clear (to the author) how to define the signs of the operations in the resolution correctly. Also, by giving the definitions of bigraded and curved versions explicitly for all algebraic structures we are interested in, we have circumvented the problem of defining the correct relations of the operations in the Thom--Whitney resolution in general.
\end{rem}

The purpose of the subset $Z \subseteq D(\cP)$ is to give a modified variant of faithfulness which takes only $P \in Z$ into account. For $Z = \{F\}$, we recover strict faithfulness, but Lie--Rinehart pairs are often only faithful when we take $\E^E$ into account as well.

\begin{defn}\label{Cartan-faithful-def}\note{Cartan-faithful-def}
 Let $\cP$ be a pre-algebraic structure carrying a Cartan structure.
 \begin{enumerate}[label=(\alph*)]
  \item A Cartanian $\cP$-pre-algebra is \emph{strictly faithful} if the map 
  $$\mathrm{ad}_F: \T \to \cH om_\C(\F,\F), \quad \theta \mapsto (a \mapsto \nabla^F_\theta(a)),$$
  is injective. 
  \item A Cartanian $\cP$-pre-algebra is \emph{$Z$-faithful} if the intersection of the kernels of the maps
  $$\mathrm{ad}_M: \T \to \cH om_\C(\E^M,\E^M), \quad \theta \mapsto (m \mapsto \nabla^M_\theta(m)),$$
  for $M \in Z$ is zero. 
 \end{enumerate}
\end{defn}

Due to the condition $a \ast^M \nabla^M_\theta(m) = \nabla^M_{a \ast^T \theta}(m)$, the notion of $Z$-faithfulness is easier to study, but it is sufficient for our purposes.

\section{(Semi-)Cartan structures in the bigraded/curved case}\label{bg-version-sec}\note{bg-version-sec}

Given an algebraic structure $\cP$ such as Lie--Rinehart algebras, it is difficult to give a general rule to construct the relations of the bigraded version $\cP^{bg}$ out of the relations of $\cP$ such that we get back our explicit definitions of the bigraded versions above. However, on the level of the pre-algebraic structures, i.e., omitting the relations, the construction of $\cP^{bg}$ out of $\cP$ is straightforward.

\begin{defn}
 Let $\cP$ be a pre-algebraic structure. Then we define a new pre-algebraic structure $\cP^{bg}$ as follows: The set of domains is $D(\cP^{bg}) := \{P^j \ | \ P \in D(\cP), j \geq 0\}$, and the grading is $|P^j| := |P| + j$. The set of constants is $\cP^{bg}(P^0) = \cP(P)$ and $\cP^{bg}(P^j) = \emptyset$ for $j \geq 1$. For the operations, we have 
 $$\cP^{bg}(P_1^{j_1},...,P_n^{j_n};Q^j) = \cP(P_1,...,P_n;Q)$$
 if $j_1 + ... + j_n = j$, and 
  $$\cP^{bg}(P_1^{j_1},...,P_n^{j_n};Q^j) = \emptyset$$
  otherwise. If $\gamma \in \cP(P)$ or $\mu \in \cP(P_1,...,P_n;Q)$, then we keep the same notation for the constant in $\cP^{bg}(P^0)$ and the operation in $\cP^{bg}(P^{j_1},...,P^{j_n};Q^j)$ if confusion is unlikely.
  \end{defn}
  
  If $\cP$ carries a semi-Cartan structure or a Cartan structure, then $\cP^{bg}$ carries a variant of the respective structure as well. However, for a $\cP^{bg}$-pre-algebra to be semi-Cartanian respective Cartanian, we wish to impose additional conditions that do not come from the general definition of (semi-)Cartanian pre-algebras, and which reflect the bigraded nature of $\cP^{bg}$.
  
  \begin{defn}\label{Cartanian-Pbg-def}\note{Cartanian-Pbg-def}
  If $\cP$ carries a semi-Cartan structure, then we say that a $\cP^{bg}$-pre-algebra $\E^{\bullet,\bullet}$ is \emph{semi-Cartanian}\index{semi-Cartanian} if
  \begin{itemize}
   \item $\ast^P: \F^{j_1} \times \E^{P,j_2} \to \E^{P,j_1 + j_2}$ is $\cO$-bilinear;
   \item $a \ast^F b = (-1)^{|a||b|} b \ast^F a$ for $a \in \F^{j_1}$, $b \in \F^{j_2}$ and the total degree $|a| = 0 + j_1$, $|b| = 0 + j_2$;
   \item $a \ast^P (b \ast^P p) = (a \ast^F b) \ast^P p$ for $a \in \F^{j_1}$, $b \in \F^{j_2}$, $p \in \E^{P,j_3}$;
   \item $1_F \ast^P p = p$ for $p \in \E^{P,j}$;
   \item every operation $\mu \in \cP^{bg}_N(P_1^{j_1},...,P_n^{j_n};Q^j)$ (for the induced filtration from $\cP$) is a multilinear differential operator of order $N$ with respect to $\F^0/\C$.
  \end{itemize}
  If $\cP$ carries a Cartan structure, then we say a $\cP^{bg}$-pre-algebra $\E^{\bullet,\bullet}$ is \emph{Cartanian}\index{Cartanian} if it is semi-Cartanian and:
  \begin{itemize}
   \item $\nabla_\theta^T(\xi) = - (-1)^{(|\theta| + 1)(|\xi| + 1)}\nabla_\xi^T(\theta)$ for $\theta \in \T^{j_1}$, $\xi \in \T^{j_2}$;
   \item $a \ast^M \nabla_\theta^M(m) = \nabla_{a \ast^T \theta}^M(m)$ for $M \in Z$ and $a \in \F^{j_1}$, $\theta \in \T^{j_2}$, $m \in \E^{M,j_3}$;
   \item $\nabla^P_\theta(\gamma) = 0$ for $\theta \in \T^j$ and constants $\gamma \in \E^{P,0}$;
   \item $\nabla_\theta^Q(\mu(p)) = (-1)^{|\mu|(|\theta| + 1)} \mu(\nabla_\theta^P(p))$ for $\theta \in \T^{j_1}$, $p \in \E^{P,j_2}$ for unary operations $\mu \in \cP(P;Q)$, where $|\mu| := |Q| - |P|$ is the degree of $\mu$;
   \item $\nabla_\theta^R\mu(p,q) = \mu(\nabla^{P}_\theta(p),q) + (-1)^{(|\theta| + 1)(|\mu| + |p|)}\mu(p,\nabla^Q_\theta(q))$ for $p \in \E^{P,j_1}$,$q \in \E^{Q,j_2}$, $\theta \in \T^j$ for binary operations $\mu \in \cP(P,Q;R)$, where $|\mu| := |R| - |P| - |Q|$ is the degree of $\mu$.
  \end{itemize}
\end{defn}
\begin{rem}
 These sign conventions are compatible with the explicit definitions of bigraded structures that we have made.
 We do not know the correct signs in the formulae if $\mu$ is a ternary or higher operation. Calculations with various compositions of binary operations in the Thom--Whitney resolution (see Chapter~\ref{TW-reso-sec}) suggest that there may not be a universal sign correct for all ternary operations. This is essentially the reason why we exclude higher operations from the definition of a Cartan structure. All sign conventions that we make are chosen in a way to fit our definition of Thom--Whitney resolutions in Chapter~\ref{TW-reso-sec}.
\end{rem}

Similarly, we can define the curved version on the level of pre-algebraic structures. 

\begin{defn}\label{Cartan-curved-def}\note{Cartan-curved-def}
 Let $\cP$ be a pre-algebraic structure which carries a Cartan structure. Then we define a new pre-algebraic structure $\cP^{crv}$ as follows: First, we take the pre-algebraic structure $\cP^{bg}$. Then we add a constant $\ell \in \cP(T^2)$ and an operation $\bar\partial \in \cP_1(P^j;P^{j + 1})$ for every $P^j \in D(\cP^{bg})$. We say that a $\cP^{crv}$-pre-algebra is \emph{Cartanian}\index{Cartanian} if it is Cartanian for $\cP^{bg}$ and:
 \begin{itemize}
  \item $\bar\partial(\gamma) = 0$ for constants $\gamma \in \cP(P)$;
  \item $\bar\partial(\mu(p)) = (-1)^{|\mu|}\mu(\bar\partial(p))$ for unary operations $\mu \in \cP(P;Q)$ and $p \in \E^{P,j}$;
  \item $\bar\partial\mu(p,q) = \mu(\bar\partial(p),q) + (-1)^{|\mu| + |p|}\mu(p,\bar\partial(q))$ for binary operations $\mu \in \cP(P,Q;R)$ and $p \in \E^{P,j_1}$, $q \in \E^{Q,j_2}$;
  \item $\bar\partial^2(p) = \nabla_\ell(p)$ for all $P^j \in D(\cP^{bg})$;
  \item  $\bar\partial(\ell) = 0$.
 \end{itemize}
\end{defn}

\begin{defn}
 Let $\cP$ be a pre-algebraic structure carrying a Cartan structure.
 \begin{enumerate}[label=(\alph*)]
  \item A Cartanian $\cP^{bg}$-pre-algebra is \emph{strictly faithful} if, for every $j \geq 0$, the map 
  $$\mathrm{ad}_F^j: \T^j \to \cH om_\C(\F^0,\F^j), \quad \theta \mapsto (a \mapsto \nabla^{F^0}_\theta(a)),$$
  is injective.
   \item A Cartanian $\cP^{bg}$-pre-algebra is \emph{$Z$-faithful} if, for every $j \geq 0$, the intersection of the kernels of the maps
  $$\mathrm{ad}_M^j: \T^j \to \cH om_\C(\E^{M,0},\E^{M,j}), \quad \theta \mapsto (m \mapsto \nabla^{M^0}_\theta(m)),$$
  for $M \in Z$ is zero.
 \end{enumerate}
 A Cartanian $\cP^{crv}$-pre-algebra is strictly faithful respective $Z$-faithful if the underlying $\cP^{bg}$-pre-algebra is.
\end{defn}

\chapter{Gauge transforms}\label{gauge-trafo-sec}\note{gauge-trafo-sec}

Let us work in the following setup:

\begin{sitn}
 Let $\Lambda$ be a complete local Noetherian $\kk$-algebra with residue field $\kk$, let $X$ be a topological space (with associated site $\cS$), and let $\C$ be the constant sheaf of rings with stalk $\Lambda$. We form a context $\mathfrak{Flat}(X;\Lambda) := (\cS,\C,\cO,\bC)$ by setting $\cO = \C$ and by taking for $\bC$ the class of flat sheaves of $\Lambda$-modules, i.e., sheaves of $\Lambda$-modules $\E$ with flat stalks $\E_x$. 
\end{sitn}

For every $A \in \mathbf{Art}_\Lambda$, we can form the tensor product $\E_A := \E \otimes_\Lambda A$ in the category of sheaves of $\Lambda$-modules; it is a flat sheaf of $A$-modules. We consider it as an infinitesimal thickening of $\E_0 := \E \otimes_\Lambda \kk$. If $B' \to B$ is a surjection in $\mathbf{Art}_\Lambda$ with kernel $I \subset B'$, then we have an exact sequence 
$$0 \to I \otimes_\Lambda \E_x \to \E_{B',x} \to \E_{B,x} \to 0$$
of stalks at $x \in X$ due to flatness. If $B' \to B$ is a small extension, then $I \otimes_\Lambda \E_x \cong I \otimes_\kk \E_{0,x}$.

Let $\cP$ be a pre-algebraic structure---we shall be primarily interested in our seven algebraic structures carrying a Cartan structure and in their bigraded counterparts---, and let $\E^\bullet$ be a $\cP$-pre-algebra in $\mathfrak{Flat}(X;\Lambda)$. For every $A \in \mathbf{Art}_\Lambda$, the tensor product $\E_A^\bullet$ is a $\cP$-pre-algebra in $\mathfrak{Flat}(X;A)$; indeed, all operations allow a base change along $\Lambda \to A$ because they are $\Lambda$-multilinear, and the underlying topological space remains the same. 

For a surjection $B' \to B$ in $\mathbf{Art}_\Lambda$, an \emph{infinitesimal automorphism} of $\E_{B'}^\bullet$ over $\E_B^\bullet$ consists of an $B'$-linear automorphism $\phi^P: \E_{B'}^P \to \E_{B'}^P$ for every domain $P \in D(\cP)$ which restricts to the identity on $\E_B^P$ along the canonical map $\E_{B'}^P \to \E_B^P$, preserves all constants in $\cP(P)$, and commutes with all operations 
$$\mu: \E^{P_1}_{B'} \times ... \times \E_{B'}^{P_n} \to \E_{B'}^Q.$$
Infinitesimal automorphisms form a sheaf of groups $\A ut_{B'/B}(\E^\bullet_{B'})$ on $X$. In general, it is not so easy to determine this sheaf, but if $\cP$ carries a Cartan structure and $\E^\bullet$ is Cartanian, then we can easily describe specific infinitesimal automorphisms, the \emph{gauge transforms}. They are controlled by elements $\theta \in \T_{B'} = \E_{B'}^T$ and constructed from the operations $\nabla^P$---which is the reason why we require these operations as part of a Cartan structure. In good situations, $\theta$ can be recovered from its gauge transform, and the gauge transforms are precisely those infinitesimal automorphisms which ``arise from geometry''. Before we define gauge transforms, we need an easy lemma.

\begin{lemma}
 Let $\cP$ be a pre-algebraic structure carrying a Cartan structure, and let $\E^\bullet$ be a Cartanian $\cP$-pre-algebra in $\mathfrak{Flat}(X;\Lambda)$. Then $\E^\bullet_A$ is a Cartanian $\cP$-pre-algebra in $\mathfrak{Flat}(X;A)$.
\end{lemma}
\begin{proof}
 All requirements can be expressed as relations of (compositions of) operations (possibly with inserted constants). This holds in particular for the differential operator property, which can be expressed as relations for maps 
 $$\F \times ... \times \F \times \E^{P_1} \times ... \times \E^{P_n} \to \E^Q.$$
 The relations can be checked on the stalks, and there we have $\E^P_{A,x} = \E^P \otimes_\Lambda A$. By $A$-multilinearity of all relations, it is sufficient to check them on entries which are in the image of the natural maps $\E^P_x \to \E^P_{A,x}$, and there, the relations hold by assumption. Also note that all $\cO$-linearity requirements are automatic since $\C = \cO$.
\end{proof}

\begin{defn}\label{gauge-trafo-def}\note{gauge-trafo-def}
 Let $B' \to B$ be a surjection in $\mathbf{Art}_\Lambda$ with kernel $I \subset B'$, and let $\E^\bullet$ be a Cartanian $\cP$-pre-algebra. Let $\theta \in \Gamma(U,I \cdot \T_{B'})$ for some open $U \subseteq X$. Then the \emph{gauge transform}\index{gauge transform} is the map 
 $$\mathrm{exp}_\theta: \enspace \E_{B'}^P|_U \to \E_{B'}^P|_U, \quad p \mapsto \sum_{n = 0}^\infty \frac{(\nabla^P_\theta)^n(p)}{n!} = p + \nabla^P_\theta(p) + \frac{1}{2}\nabla^P_\theta\nabla^P_\theta(p) + ... \ .$$
\end{defn}

This definition is one of the key reasons why we assume $\mathrm{char}(\kk) = 0$. Let us write $[\theta,\xi] := \nabla^T_\theta(\xi)$ for $\theta,\xi \in \T$ as usual. This turns $\T$ into a sheaf of Lie algebras.

\begin{lemma}\label{gauge-trafo-prop}\note{gauge-trafo-prop}
 The gauge transform has the following properties:
 \begin{enumerate}[label=\emph{(\alph*)}]
  \item $\mathrm{exp}_\theta: \E_{B'}^P \to \E_{B'}^P$ is a well-defined $B'$-linear sheaf homomorphism.
  \item $\mathrm{exp}_\theta: \E_{B'}^P \to \E_{B'}^P$ restricts to the identity on $\E_B^P$.
  \item For a constant $\gamma \in \cP(P)$, we have $\mathrm{exp}_\theta(\gamma) = \gamma$.
  \item For an operation $\mu \in \cP(P_1,...,P_n;Q)$, we have 
  $$\mu(\mathrm{exp}_\theta(p_1),..., \mathrm{exp}_\theta(p_n)) = \mathrm{exp}_\theta(\mu(p_1,...,p_n)).$$
  \item For $\theta = 0 \in I \cdot \T_{B'}$, we have $\mathrm{exp}_0(p) = p$.
  \item For $\theta,\xi \in I \cdot \T_{B'}$, the composition is given by $\mathrm{exp}_\theta \circ \mathrm{exp}_\xi(p) = \mathrm{exp}_{\theta \odot \xi}(p)$ for the \emph{Baker--Campbell--Hausdorff formula}\index{Baker--Campbell--Hausdorff formula}
  \begin{align}
  \theta \odot \xi &:= \sum_{n = 1}^\infty\frac{(-1)^{n + 1}}{n} \sum_{r_1 + s_1 > 0}\dots \sum_{r_n + s_n > 0}\frac{1}{\sum_{i = 1}^n(r_i + s_i)}\frac{[\theta^{(r_1)},\xi^{(s_1)},...,\theta^{(r_n)},\xi^{(s_n)}]}{\prod_{i = 1}^nr_i!s_i!} \label{BCH}\tag{BCH} \\
  &= \theta + \xi + \frac{1}{2}[\theta,\xi] + \frac{1}{12}[\theta,[\theta,\xi]] - \frac{1}{12}[\xi,[\theta,\xi]] + ... \nonumber
 \end{align}
 where $[\theta] := \theta$ and $[\theta_1,\theta_2,...,\theta_{n + 1}] := [\theta_1,[\theta_2,...,\theta_n]]$ inductively for $n \geq 2$, and where $\theta^{(r)}$ denotes the repetition of $r$ entries of $\theta$.
 \item The inverse of $\mathrm{exp}_\theta$ is given by $\mathrm{exp}_{-\theta}$.
 \end{enumerate}
 This shows that $\mathrm{exp}_\theta$ is an infinitesimal automorphism of the $\cP$-pre-algebra $\E_{B'}^\bullet$ over $\E_B^\bullet$. When we endow $I \cdot \T_{B'}$ with the Baker--Campbell--Hausdorff product $\odot$, it becomes a sheaf of groups with neutral element $0$. Then
 \begin{equation}\label{exp-homom}
  \mathrm{exp}: \enspace  (I \cdot \T_{B'},\odot) \to (\A ut_{B'/B}(\E_{B'}^\bullet),\circ)
 \end{equation}
 is a homomorphism of sheaves of groups. If $B' \to B$ is a small extension, then $\theta \odot \xi = \theta + \xi$, and $(I \cdot \T_{B'},\odot)$ is abelian.
\end{lemma}
\begin{proof}
Since $\nabla^P_\theta$ is $B'$-linear, it follows from the nilpotency of $I$ that $\mathrm{exp}_\theta$ is a well-defined $B'$-linear operator. The induced map on $\E_B^P$ is the identity since $\theta \in I \cdot \T_{B'}$. For the constants $\gamma \in \cP(P)$, we find $\mathrm{exp}_\theta(\gamma) = \gamma$ since $\nabla_\theta^P(\gamma) = 0$. For unary operations $\mu \in \cP(P;Q)$, we have 
$$\mathrm{exp}_\theta(\mu(p)) = \sum_{n = 0}^\infty \frac{\nabla_\theta^n(\mu(p))}{n!} = \sum_{n = 0}^\infty \frac{\mu(\nabla^n_\theta(p))}{n!} = \mu(\mathrm{exp}_\theta(p)).$$
For binary operations $\mu \in \cP(P_1,P_2;Q)$, we find 
$$\nabla^n_\theta(\mu(p_1,p_2)) = \sum_{k = 0}^n{\binom{n}{k}}\mu(\nabla^k_\theta(p_1),\nabla^{n - k}_\theta(p_2))$$
by induction on $n$, and then $\mathrm{exp}_\theta(\mu(p_1,p_2)) = \mu(\mathrm{exp}_\theta(p_1),\mathrm{exp}_\theta(p_2))$. By assumption on a Cartan structure, there are no higher operations, but a similar proof would show the corresponding formula. If $\theta = 0$, then $\nabla_0(p) = 0$, so $\mathrm{exp}_\theta(p) = p$. 

We can reduce the Baker--Campbell--Hausdorff formula in our case to a general Baker--Campbell--Hausdorff formula for operators on $\kk$-vector spaces. Note that
 $$\mathrm{exp}_\theta = \mathrm{exp}(\nabla_\theta) := \sum_{n = 0}^\infty \frac{\nabla_\theta^n}{n!}$$
 as operators. Now $\mathrm{exp}(\nabla_\theta) \circ \mathrm{exp}(\nabla_\xi) - \mathrm{Id}$ is a nilpotent $B'$-linear operator; when we apply the formula for the logarithm, then we obtain an operator 
 $$C := \sum_{n = 1}^\infty \frac{(-1)^{n + 1}}{n} (\mathrm{exp}(B_\theta) \circ \mathrm{exp}(B_\xi) - \mathrm{Id})^n: \E^P_{B'} \to \E^P_{B'}$$
 with the property that $\mathrm{exp}(C) = \mathrm{exp}_\theta \circ \mathrm{exp}_\xi$. When we expand this formula, we get 
 $$C = \sum_{n = 1}^\infty \frac{(-1)^{n + 1}}{n}\sum_{r_1 + s_1 > 0} \dots \sum_{r_n + s_n > 0} \frac{\nabla_\theta^{r_1}\nabla_\xi^{s_1} ... \nabla_\theta^{r_n}\nabla_\xi^{s_n}}{\prod_{i = 1}^n r_i!s_i!}.$$
 For general operators $P_1, P_2, ...$, we set $[P_1] := P_1$, $[P_1,P_2] := P_1P_2 - P_2P_1$, and $[P_1,...,P_{n + 1}] := [P_1,[P_2,...,P_{n + 1}]]$ inductively, i.e., they are the commutators. Let us introduce notation 
 $$I(n;e) := \left\{(r_i,s_i)_i \in \prod_{i = 1}^n (\NN \times \NN) \ | \ r_i + s_i > 0,\ \sum_{i = 1}^n(r_i + s_i) = e\right\}$$
 for the part of the index set where the sum of all indices is $e$. Then the classical Baker--Campbell--Hausdorff formula for operators\index{Baker--Campbell--Hausdorff formula!for operators} yields
 \begin{align}
  &\qquad \:\,\sum_{n = 1}^e \frac{(-1)^{n + 1}}{n} \sum_{(r_i,s_i)_i \in I(n;e)} \frac{\nabla_\theta^{r_1}\nabla_\xi^{s_1} ... \nabla_\theta^{r_n}\nabla_\xi^{s_n}}{\prod_{i = 1}^n r_i!s_i!} \nonumber \\
  &= \frac{1}{e} \cdot \sum_{n = 1}^e \frac{(-1)^{n + 1}}{n} \sum_{(r_i,s_i)_i \in I(n;e)} \frac{[\nabla_\theta^{(r_1)},\nabla_\xi^{(s_1)}, ..., \nabla_\theta^{(r_n)},\nabla_\xi^{(s_n)}]}{\prod_{i = 1}^n r_i!s_i!} \nonumber
 \end{align}
 where the sum over $n$ stops at $e$ because $I(n;e) = \emptyset$ for $n > e$. The formula does only hold for the sum but not for each individual value of $n$. The Jacobi identity yields via induction over $m$ the formula
 $$[\nabla_{\theta_1},...,\nabla_{\theta_m}](p) = \nabla_{[\theta_1,...,\theta_m]}(p),$$
 where we have on the left the iterated commutator of operators, and on the right the iterated Lie bracket in $\T_{B'}$. The formula shows that $C = \nabla_{\theta \odot \xi}$ with the definition of $\theta \odot \xi$ in the statement; thus $\mathrm{exp}_\theta \circ \mathrm{exp}_\xi = \mathrm{exp}_{\theta \odot \xi}$.
 
 Evaluating the Baker--Campbell--Hausdorff formula, we find $\theta \odot (-\theta) = 0$ since $[\theta,\theta] = -[\theta,\theta]$ and hence $[\theta,\theta] = 0$. Thus, $\mathrm{exp}_\theta$ and $\mathrm{exp}_{-\theta}$ are inverse to each other.
 
 The sheaf of magmas (sets with a binary operation) $(I \cdot \T_{B'},\odot)$ is a sheaf of groups because $\odot$ defines a group structure on every nilpotent $\kk$-Lie algebra, see e.g.~\cite[V,\,\S\S1-3]{ManettiComplexDefo2004}. That $\mathrm{exp}$ is a group homomorphism is now clear. If $B' \to B$ is a small extension, we have $I \cdot \m_{B'} = 0$ and hence $I^2 = 0$, so $[I \cdot \T_{B'},I\cdot \T_{B'}] = 0$, and hence $\theta \odot \xi = \theta + \xi$.
\end{proof}

\begin{ex}
 Let $f_A: X_A \to S_A$ be a separated, log smooth, and saturated morphism to the log point $S_A = \Spec (Q \to A)$. Let $\G^\bullet_{X_A/S_A}$ be its Gerstenhaber algebra of polyvector fields (with the conventions of \cite{Felten2022}). Then the map 
 $$(I \cdot \G^{-1}_{X_A/S_A},\odot) \to (\A ut_{A/\kk}(\G^\bullet_{X_A/S_A}),\circ)$$
 is injective. By \cite[\S\S2-3]{Felten2022}, the gauge transforms of the Gerstenhaber algebra $\G^\bullet_{X_A/S_A}$ are precisely those automorphisms which are induced from infinitesimal automorphisms of the log smooth deformation $f_A: X_A \to S_A$ of its central fiber $f_0: X_0 \to S_0$.
\end{ex}

A priori, the exponential map in \eqref{exp-homom} is not injective in general. However, under mild conditions, it is. For our purposes, this condition is (virtually) always satisfied.

\begin{lemma}\label{gauge-inject}\note{gauge-inject}
 Let $\E^\bullet$ be a Cartanian $\cP$-pre-algebra in the context $\mathfrak{Flat}(X;\Lambda)$. Assume that $\E_0^\bullet$ is faithful, i.e., the intersection of the kernels of the maps 
 $$\mathrm{ad}_{P,0}: \T_0 \to \cH om_\kk(\E_0^P,\E_0^P), \quad \theta \mapsto (p \mapsto \nabla^P_\theta(p)),$$
 is zero. Then, for all $A \in \mathbf{Art}_\Lambda$, the intersection of the kernels of the maps $\mathrm{ad}_{P,A}: \T_A \to \cH om_A(\E_A^P,\E_A^P)$ is zero. Moreover, for a surjection $B' \to B$ in $\mathbf{Art}_\Lambda$ with kernel $I \subset B'$, the exponential map 
 $$\mathrm{exp}: (I \cdot \T_{B'},\odot) \to (\A ut_{B'/B}(\E_{B'}^\bullet),\circ)$$
 is injective. In particular, it is injective if $\E_0^\bullet$ is centerless or strictly faithful.
\end{lemma}
\begin{proof}
 We prove the first statement by induction over small extensions $B' \to B$ with kernel $I$. Let $\theta \in \T_{B'}$ with $\nabla_\theta^P = 0$ for all $P \in D(\cP)$. Then, in particular, $\nabla_\theta^P|_B = 0$. By the induction hypothesis, we find $\theta|_B = 0$, so $\theta \in I \cdot \T_{B'}$. Since $B' \to B$ is a small extension, $I$ is a one-dimensional $\kk$-vector space, so we can choose a generator $i \in I$ with $I = i \cdot \kk$. Then, at least locally, $\theta = i \cdot \xi$ for some $\xi \in \T_{B'}$. Now the kernel of $\mu_i: B' \to B', \: b \mapsto ib,$ is $\m_{B'}$; thus, since $\E_{B'}^P$ is flat over $B'$, the kernel of $\mu_i: \E_{B'}^P \to \E_{B'}^P, \: p \mapsto ip,$ is $\m_{B'} \cdot \E_{B'}^P$. We have 
 $$0 = \nabla^P_\theta(p) = i \cdot \nabla^P_\xi(p),$$
 hence $\nabla_\xi^P(p) \in \m_{B'} \cdot \E_{B'}^P$. This shows that $\nabla_{\xi_0}^P: \E_0^P \to \E_0^P$ is the zero map, where $\xi_0 := \xi|_0$. Since this holds for all $P \in D(\cP)$, the faithfulness assumption on $\E_0^\bullet$ shows $\xi_0 = 0$. However, in this case, $\theta = i \cdot \xi = 0$.
 
 By decomposing $B' \to B$ into a sequence of small extensions, we see that it is sufficient to prove the second statement in the case of a small extension $B' \to B$. In this case, $\mathrm{exp}_\theta(p) = p + \nabla_\theta^P(p)$, so $\mathrm{exp}_\theta = \mathrm{id}$ implies $\nabla_\theta^P(p) = 0$ for all $P \in D(\cP)$ and $p \in \E_{B'}^P$. Then the claim follows from our first statement.
\end{proof}
\begin{rem}
 If $\E^P$ is \emph{complete}, i.e., $\E^P$ is the limit of the system $\E^p_k := \E^P \otimes_\Lambda \Lambda/\m_\Lambda^{k + 1}$, and $\E_0^\bullet$ is faithful, then $\E^\bullet$ is faithful as well. More generally, it suffices that the map from $\E^P$ to the limit is injective to conclude faithfulness of $\E^\bullet$. The converse is false: Let $\Lambda = \kk\llbracket t\rrbracket$, and let $X = \{*\}$ be a point. Consider $F = \Lambda[s]$ and $T = \Lambda[s] \cdot D$. These flat modules are not complete, but they inject into their $t$-adic completions. We turn them into a Lie--Rinehart algebra by setting $D = t \cdot \partial_s$, i.e., $D(s) = t\partial_s(s) = t$ etc., and 
 $$[q(t,s)\cdot D,r(t,s)\cdot D] := \big(q(t,s) \cdot D(r(t,s)) - r(t,s) \cdot D(q(t,s))\big) \cdot D.$$
 This Lie--Rinehart algebra is both strictly faithful and centerless. However, in the base change $(F_0,T_0)$ to $\kk$, the generator $D$ of $T_0$ acts as the zero operator, and the Lie bracket vanishes as well. Thus, $(F_0,T_0)$ is neither strictly faithful nor centerless nor faithful.
\end{rem}
\begin{rem}
 In \cite[\S 10]{Felten2022}, a curved Lie algebra $L^\bullet$ (in the context $\mathfrak{Comp}(\Lambda)$ of flat and complete $\Lambda$-modules) is defined to be \emph{faithful} if, in our language, $L^\bullet_A$ is centerless for all $A \in \mathbf{Art}_\Lambda$. The proof of the lemma shows that it would have been enough to require that $L_0^\bullet$ is centerless.
\end{rem}



\chapter{The extended Maurer--Cartan equations}\label{extended-MC-eqn-sec}\note{extended-MC-eqn-sec}

As is well known, every differential graded Lie algebra $L^\bullet$ (in the context $\mathfrak{Comp}(\Lambda)$) gives rise to a deformation functor $\mathrm{Def}(L^\bullet,-)$ by taking gauge equivalence classes of Maurer--Cartan solutions. This is the reason why dg Lie algebras are ubiquitous in deformation theory: For many deformation problems, one can find a dg Lie algebra $L^\bullet$ such that $\mathrm{Def}(L^\bullet,-)$ classifies isomorphism classes of deformations. Less well known is that every \emph{curved} Lie algebra gives rise to a deformation functor $\mathrm{Def}(L^\bullet,-)$ as well, by considering solutions of the \emph{extended} Maurer--Cartan equation. This is absolutely necessary in logarithmic deformation theory because these deformation problems are not controlled by any dg Lie algebra.

We work in the context $\mathfrak{Comp}(\Lambda)$ throughout. Recall that objects in $\bC$ are flat and complete $\Lambda$-modules. We work with the following structures in this chapter.

\begin{defn}\label{L-curved-def}\note{L-curved-def}\index{$\Lambda$-linear curved algebras}
 \hspace{1cm}
 \begin{enumerate}[label=(\alph*)]
  \item A \emph{$\Lambda$-linear curved Lie algebra} is a curved Lie algebra $L^\bullet$ in the context $\mathfrak{Comp}(\Lambda)$ with $\ell \in \m_\Lambda \cdot L^2$.
  \item Let $\cP$ be a pre-algebraic structure carrying a Cartan structure. Then a \emph{$\Lambda$-linear curved $\cP$-algebra} is a Cartanian $\cP^{crv}$-pre-algebra in $\mathfrak{Comp}(\Lambda)$ in the sense of Definition~\ref{Cartan-curved-def} with $\ell \in \m_\Lambda \cdot \T^2$.
  \item A \emph{$\Lambda$-linear curved Lie--Rinehart algebra} respective \emph{Lie--Rinehart pair} respective \emph{Gerstenhaber algebra} respective \emph{Gerstenhaber calculus} is a curved Lie--Rinehart algebra (respective any of the other notions) in the context $\mathfrak{Comp}(\Lambda)$ with $\ell \in \m_\Lambda \cdot \T^2$.
  \item A \emph{$\Lambda$-linear Batalin--Vilkovisky curved algebra} respective \emph{Batalin--Vilkovisky calculus} is a curved Batalin--Vilkovisky algebra respective calculus in $\mathfrak{Comp}(\Lambda)$ with $\ell \in \m_\Lambda \cdot \G^{-1,2}$ and $y \in \m_\Lambda \cdot \G^{0,1}$.
 \end{enumerate}
 Note that, for example, if $\G_d$ is the algebraic structure of Gerstenhaber algebras of dimension $d$, then every $\Lambda$-linear curved Gerstenhaber algebra is a $\Lambda$-linear curved $\G_d$-algebra but not conversely since we assume fewer relations for the latter notion.
\end{defn}

\section{The classical extended Maurer--Cartan equation}

We introduce the classical extended Maurer--Cartan equation. After reviewing the most basic case of curved Lie algebras, we briefly discuss the case of curved Lie--Rinehart algebras which is important for deformations of pairs, and then go to the main case of a curved Gerstenhaber algebra.

\subsection{Curved Lie algebras}\label{curved-Lie-alg-sec}\note{curved-Lie-alg-sec}

Let $L^\bullet$ be a $\Lambda$-linear curved Lie algebra. Here, we recall how it gives rise to a deformation functor. By definition, it comes with a predifferential $\bar\partial$, inducing a predifferential $\bar\partial: L^q_A \to L^{q + 1}_A$ for every $A \in \mathbf{Art}_\Lambda$. However, for every $\phi \in \m_A \cdot L_A^1$, we can define a map 
$$\bar\partial_\phi: L_A^q \to L_A^{q + 1}, \quad \xi \mapsto \bar\partial \xi + [\phi,\xi].$$
It is an $A$-linear predifferential as well; in particular, we have 
$$\bar\partial_\phi^2(\xi) = [\bar\partial(\phi) + \frac{1}{2}[\phi,\phi] + \ell, \xi]$$
so that $(L^\bullet_A,[-,-],\bar\partial_\phi,\ell_\phi)$ with 
$$\ell_\phi := \bar\partial(\phi) + \frac{1}{2}[\phi,\phi] + \ell$$
is another curved Lie algebra. In particular, we have $\bar\partial_\phi(\ell_\phi) = 0$ since $[\phi,[\phi,\phi]] = 0$. We can consider 
$\mathrm{PDiff}_A(L^\bullet) := \m_A \cdot L_A^1$
as the space of predifferentials. Since $\phi$ can always be lifted along a surjection $B' \to B$ in $\mathbf{Art}_\Lambda$, there is not really a distinguished predifferential; we may consider a curved Lie algebra as a graded Lie algebra which comes with an affine space of predifferentials, none of which is intrinsically distinguished.

If the central fiber $L_0^\bullet$ is centerless, then $\bar\partial_\phi$ is a differential if and only if $\ell_\phi = 0$ (by the argument in the proof of Lemma~\ref{gauge-inject}); in general, this is a sufficient condition. If $\ell = 0$, then this becomes the classical \emph{Maurer--Cartan equation}\note{Maurer--Cartan equation!classical}
$$\bar\partial(\phi) + \frac{1}{2}[\phi,\phi] = 0.$$
In our case, we have an additional term $\ell$ so that we talk about an \emph{extended} Maurer--Cartan equation.

\begin{defn}
 Let $L^\bullet$ be a $\Lambda$-linear curved Lie algebra, and let $A \in \mathbf{Art}_\Lambda$. Then the \emph{classical extended Maurer--Cartan equation}\index{Maurer--Cartan equation!classical extended} is the equation
 $$\ell_\phi = \bar\partial(\phi) + \frac{1}{2}[\phi,\phi] + \ell = 0$$
 which is to be satisfied by elements $\phi \in \m_A \cdot L_A^1$. Its solutions form a homogeneous\footnote{This is standard. See e.g.\ \cite[10.5]{Felten2022} for some hints how to show this.} functor 
 $$\mathrm{MC}(L^\bullet,-): \mathbf{Art}_\Lambda \to \mathbf{Set}$$
 of Artin rings.
\end{defn}

For a fixed predifferential $\bar\partial$, the solutions of the classical extended Maurer--Cartan equation form the \emph{differentials} on $L_A^\bullet$ in the sense that $(L_A^\bullet,[-,-],\bar\partial_\phi,\ell_\phi)$ becomes a differential graded Lie algebra in our sense. 

Observe that the Definition~\ref{gauge-trafo-def} of gauge transforms makes sense also for curved Lie algebras, although they do not come with a Cartan structure. The predifferential $\bar\partial_\phi$ is not compatible with gauge transforms. Nonetheless, we can formulate precisely how it behaves under gauge transforms. In order to do this, we need the power series expansion 
$$T(x) := \frac{e^x - 1}{x} = \sum_{n = 0}^\infty \frac{x^n}{(n + 1)!}  = 1 + \frac{x}{2} + \frac{x^2}{6} + \frac{x^3}{24} + ... \ .$$
\begin{lemdef}\label{gauge-action}\note{gauge-action}\index{Baker--Campbell--Hausdorff formula}
 Let $L^\bullet$ be a $\Lambda$-linear curved Lie algebra, and let $B' \to B$ be a surjection in $\mathbf{Art}_\Lambda$ with kernel $I \subset B'$. Let $\theta \in I \cdot L_{B'}^0$ and $\phi \in \m_{B'} \cdot L_{B'}^1$. Then we have 
 $$\mathrm{exp}_\theta(\bar\partial_\phi(\xi)) = \mathrm{exp}_\theta(\bar\partial\xi  + [\phi,\xi]) = \left(\bar\partial(-) + [\mathrm{exp}_\theta(\phi) - T(\nabla_\theta)(\bar\partial\theta),-]\right)(\mathrm{exp}_\theta(\xi))$$
 with $\nabla_\theta = [\theta,-]$ and 
 $$T(\nabla_\theta)(\bar\partial\theta) = \bar\partial\theta + \frac{1}{2}[\theta,\bar\partial\theta] + \frac{1}{6}[\theta,[\theta,\bar\partial\theta]] + ... \ .$$
 This gives rise to the \emph{gauge action}\index{gauge action}
 $$\mathrm{exp}_\theta * \phi := \mathrm{exp}_\theta(\phi) - T(\nabla_\theta)(\bar\partial\theta) = \phi + T(\nabla_\theta)([\theta,\phi] - \bar\partial\theta),$$
 which satisfies $\mathrm{exp}_\theta * (\mathrm{exp}_\xi * \phi) = \mathrm{exp}_{\theta \odot \xi} * \phi$ with the Baker--Campbell--Hausdorff product $\theta \odot \xi$.
 We say that $\phi,\phi' \in \m_{B'} \cdot L_{B'}^1$ are \emph{gauge equivalent}\index{gauge equivalent} relative to $B$ if there is some $\theta \in I \cdot L_{B'}^0$ with $\mathrm{exp}_\theta * \phi = \phi'$. This is an equivalence relation.
\end{lemdef}
\begin{proof}
 See \cite[\S6.3]{ManettiLieMethods2022} for a general discussion about the gauge action. The formula for the behavior of $\bar\partial_\phi$ under gauge transforms is taken from \cite[Lemma~2.5]{ChanLeungMa2023}. The statements remain true for a curved Lie algebra since the square $\bar\partial^2$ does not show up anywhere. To see the first statement explicitly, note that 
 $$\bar\partial\nabla_\theta^n(\xi) = \nabla_\theta^n(\bar\partial\xi) + \sum_{s = 1}^n{\binom{n}{s}}\nabla_{\nabla_\theta^{s - 1}(\bar\partial \theta)}\nabla_\theta^{n - s}(\xi)$$
 for $n \geq 1$. Then we obtain 
 $$\bar\partial \mathrm{exp}_\theta(\xi) = \mathrm{exp}_\theta(\bar\partial\xi) + [T(\nabla_\theta)(\bar\partial \theta),\mathrm{exp}_\theta(\xi)]$$
 by a direct computation, from which the version for $\bar\partial_\phi$ follows easily.
\end{proof}

We are mostly interested in gauge equivalence relative to $A_0 = \kk$. If $L_0^\bullet$ is centerless, then $\phi,\phi' \in \m_A\cdot L_A^1$ are gauge equivalent if and only if there is some $\theta \in \m_A \cdot L_A^0$ with 
$$\mathrm{exp}_\theta \circ \bar\partial_\phi = \bar\partial_{\phi'} \circ \mathrm{exp}_\theta.$$
\begin{rem}
 It is tempting to use this as a definition of gauge equivalence. In fact, I have stated this wrongly in \cite[\S 10.2]{Felten2022}, where gauge equivalence has been defined by this equation without assuming that $L_0^\bullet$ is centerless. The problem with the latter definition is that it can depend on which space $\bar\partial_\phi$ acts, i.e., when we require the above equation in a bigraded Gerstenhaber calculus then it might give a different notion of gauge equivalence than when we consider it only in the graded Lie algebra which is the $(-1)$-part of the bigraded Gerstenhaber algebra. However, the definition of gauge equivalence in \cite[Defn.~7.1]{Felten2022} is still correct because of strict faithfulness (i.e., $(-1)$-injectivity)---there, $\mathrm{exp}_\theta \circ \bar\partial_\phi = \bar\partial_{\phi'} \circ \mathrm{exp}_\theta$ is required on the level of the bigraded Gerstenhaber algebra.
\end{rem}

If $\phi,\phi' \in \m_A \cdot L_A^1$ are gauge equivalent, then $\phi$ is a Maurer--Cartan solution if and only if $\phi'$ is a Maurer--Cartan solution. If $L_0^\bullet$ is centerless, then this is easy to see, but it is true in general. For the general case, first observe that 
$$\bar\partial T(\nabla_\theta)(\bar\partial\theta) = T(\nabla_\theta)(\bar\partial^2\theta)  + \frac{1}{2}[T(\nabla_\theta)(\bar\partial\theta),T(\nabla_\theta)(\bar\partial\theta)],$$
hence we have 
$$\mathrm{exp}_\theta(\ell) - \ell = \frac{1}{2}[T(\nabla_\theta)(\bar\partial\theta),T(\nabla_\theta)(\bar\partial\theta)] - \bar\partial T(\nabla_\theta)(\bar\partial\theta).$$
With this, we find 
$$\bar\partial\phi' + \frac{1}{2}[\phi',\phi'] + \ell = \mathrm{exp}_\theta(\bar\partial\phi + \frac{1}{2}[\phi,\phi] + \ell)$$
by using Lemma~\ref{gauge-action} in a straightforward computation of the left hand side since $\mathrm{exp}_\theta$ commutes with $[-,-]$. Thus $\phi$ is a Maurer--Cartan solution if and only if $\phi' = \mathrm{exp}_\theta * \phi$ is a Maurer--Cartan solution.

Gauge equivalence classes of Maurer--Cartan solutions form a deformation functor 
$$\mathrm{Def}(L^\bullet,-): \mathbf{Art}_\Lambda \to \mathbf{Set},$$
which is in general not homogeneous.
The key insight of \cite{ChanLeungMa2023} and, subsequently, of \cite{Felten2022}, is that deformation functors in log geometry are isomorphic to this deformation functor for an appropriate choice of $L^\bullet$.

\begin{ex}
 Let $f_0: X_0 \to S_0$ be proper, log smooth, and saturated over the punctual log scheme $S_0 = \Spec (Q \to \kk)$. Then for the curved Lie algebra $L^\bullet_{X_0/S_0}$ constructed in \cite[\S 8]{Felten2022}, the log smooth deformation functor $\mathrm{LD}_{X_0/S_0}$ is isomorphic to $\mathrm{Def}(L^\bullet_{X_0/S_0},-)$.
\end{ex}

\subsection{Predifferential graded $\cP$-algebras}

Let $\cP$ be an algebraic structure carrying a Cartan structure, and let $E^{\bullet,\bullet}$ be a $\Lambda$-linear curved $\cP$-algebra, i.e., a Cartanian $\cP^{crv}$-\emph{pre}-algebra in $\mathfrak{Comp}(\Lambda)$ with $\ell \in \m_\Lambda \cdot T^2$. Then $T^\bullet$ is a $\Lambda$-linear curved Lie algebra, so we obtain a deformation functor 
$$\mathrm{Def}(E^{\bullet,\bullet},-) := \mathrm{Def}(T^\bullet,-): \mathbf{Art}_\Lambda \to \mathbf{Set}.$$
A general $\phi \in \m_A \cdot T_A^1$ gives a new predifferential
$$\bar\partial_\phi: \E^{P,j} \to \E^{P,j + 1}, \quad p \mapsto \bar\partial p + \nabla_\phi(p),$$
in the sense that $(E_A^{\bullet,\bullet},\bar\partial_\phi,\ell_\phi)$ is again a Cartanian curved $\cP$-algebra with 
$$\ell_\phi := \bar\partial\phi + \frac{1}{2}[\phi,\phi] + \ell \in T_A^2.$$ 
Indeed, if $\bar\partial$ satisfies the five conditions in Definition~\ref{Cartan-curved-def}, then $\bar\partial_\phi$ satisfies them as well.

\begin{rem}
 If $E^{\bullet,\bullet}$ is in fact a Lie--Rinehart algebra or pair or a Gerstenhaber algebra or calculus, then it still is with the new predifferential $\bar\partial_\phi$ and the new $\ell_\phi$. This is because the axioms we are missing with a Cartanian curved $\cP$-algebra do not involve $\bar\partial$. On the other side, this discussion does not apply to Batalin--Vilkovisky algebras or calculi since already their plain, singly graded version does not admit a Cartan structure.
\end{rem}

If $\phi$ is a solution of the classical extended Maurer--Cartan equation, then $\ell_\phi = 0$ and $\bar\partial_\phi$ is a differential on $E_A^{\bullet,\bullet}$. If two solutions $\phi,\phi'$ of the classical extended Maurer--Cartan equation are gauge equivalent, $\phi' = \mathrm{exp}_\theta * \phi$, then we have 
$$\mathrm{exp}_\theta \circ \bar\partial_\phi = \bar\partial_{\phi'} \circ \mathrm{exp}_\theta$$
on all $E_A^{P,j}$ by the proof of Lemma~\ref{gauge-action}. If $E_0^{\bullet,\bullet}$ is faithful, e.g.~centerless or strictly faithful, then we can identify the set of Maurer--Cartan solutions with those $\phi \in \m_A \cdot T_A^1$ such that $\bar\partial_\phi^2 = 0$, and two differentials $\bar\partial_\phi$ and $\bar\partial_{\phi'}$ are gauge equivalent if and only if they are transformed into each other by a gauge transform. Here, it is crucial that, in our definition of faithfulness in the bigraded case, we have required the condition not only for $T^0$ but for all $T^j$.

\begin{ex}
 Let $f_0: X_0 \to S_0$ be a separated, log smooth, and saturated morphism to  a log point $S_0$. Then the construction in \cite{Felten2022} does not only yield a curved Lie algebra which controls the log smooth deformation functor, but in fact, it yields a curved Gerstenhaber algebra. We will see later that we really get a curved Gerstenhaber calculus even more generally. Since the underlying bigraded Gerstenhaber calculus is strictly faithful, we can interpret Maurer--Cartan solutions as differentials on a Gerstenhaber calculus, not only on a Lie algebra.
\end{ex}

\begin{ex}
 When we study deformations of pairs $f_0: X_0 \to S_0$ and a line bundle $\cL_0$ on $X_0$, then we apply this discussion to $\Lambda$-linear curved Lie--Rinehart pairs.
\end{ex}

\begin{ex}
 In a $\Lambda$-linear pbdg Batalin--Vilkovisky algebra or calculus $\E^{\bullet,\bullet}$, we consider the classical extended Maurer--Cartan equation and deformation functor of the underlying $\Lambda$-linear Gerstenhaber algebra or calculus as the ones of $\E^{\bullet,\bullet}$.
\end{ex}

When showing unobstructedness, we can work equivalently with the Maurer--Cartan functor $\mathrm{MC}(E^{\bullet,\bullet},-)$.

\begin{lemma}\label{unobstr-equiv}\note{unobstr-equiv}
 Let $E^{\bullet,\bullet}$ be a $\Lambda$-linear curved $\cP$-algebra. Then $\mathrm{Def}(E^{\bullet,\bullet},-)$ is unobstructed if and only if $\mathrm{MC}(E^{\bullet,\bullet},-)$ is unobstructed.
\end{lemma}
\begin{proof}
 If $\mathrm{MC}(G^{\bullet,\bullet},-)$ is unobstructed, then $\mathrm{Def}(G^{\bullet,\bullet},-)$ is obviously unobstructed as well. Conversely, let $B' \to B$ be a small extension with kernel $I$. Let $\phi$ be a Maurer--Cartan solution over $B$, and let $\psi'$ be a Maurer--Cartan solution over $B'$ such that $\psi := \psi'|_B$ is gauge equivalent to $\phi$. Let $\theta \in \m_B \cdot T_B^0$ be such that $\mathrm{exp}_\theta * \psi = \phi$, and let $\theta' \in \m_{B'} \cdot T_{B'}^0$ be a lift of $\theta$. Then define $\phi' := \mathrm{exp}_{\theta'} * \psi'$. This is a Maurer--Cartan solution as well. Since the gauge action commutes with base change along $B' \to B$, we find $\phi'|_B = \phi$. Thus $\mathrm{MC}(E^{\bullet,\bullet},-)$ is unobstructed.
\end{proof}

\section{The semi-classical extended Maurer--Cartan equation}

Now let $G^{\bullet,\bullet}$ be a $\Lambda$-linear curved \emph{Batalin--Vilkovisky} algebra. Let us assume (for simplicity) that $G_0^{\bullet,\bullet}$ is strictly faithful. We have seen above that a classical Maurer--Cartan solution $\phi$ gives rise to a differential $\bar\partial_\phi$ on $G^{\bullet,\bullet}$ by modifying the predifferential $\bar\partial$ with $[\phi,-]$.
However, on a curved Batalin--Vilkovisky algebra, we also have the Batalin--Vilkovisky operator $\Delta$.
The anti-commutation rule becomes
$$\bar\partial_\phi\Delta + \Delta\bar\partial_\phi = [y + \Delta\phi,-],$$
which is in general not zero. To achieve that this is also zero, i.e., that we do not only have $\bar\partial_\phi^2 = 0$ but also an anti-commuting Batalin--Vilkovisky operator, we also want to modify $\Delta$ with an element $f \in \m_A \cdot G_A^{0,0}$, i.e., $\Delta_f := \Delta + [f,-]$. Then $\Delta_f^2 = 0$ is automatic; moreover, we have 
$$\bar\partial_\phi\Delta_f + \Delta_f\bar\partial_\phi = [\bar\partial f + [\phi,f] + y + \Delta\phi,-].$$
The term in the bracket on the right hand side is in $G^{0,2}_A$, so we cannot conclude from strict faithfulness\footnote{The fact that $[1,-] = 0$ suggests that we cannot expect an equivalent of faithfulness for elements that are in $G_A^{0,\bullet}$ rather than $G_A^{-1,\bullet}$.} that it is necessary to have 
$$\bar\partial f + [\phi,f] + y + \Delta\phi = 0$$
in order to have $\bar\partial_\phi\Delta_f + \Delta_f\bar\partial_\phi = 0$, but it is certainly sufficient.

\begin{defn}
 The \emph{semi-classical\footnote{This notion is intermediate between the \emph{classical} extended Maurer--Cartan equation and the \emph{quantum} extended Maurer--Cartan equation that we consider below.} extended Maurer--Cartan equation}\index{Maurer--Cartan equation!semi-classical extended} is the system of equations
 \begin{align}
  \bar\partial\phi + \frac{1}{2}[\phi,\phi] + \ell &= 0 \nonumber \\
  \bar\partial f + [\phi,f] + y + \Delta\phi &= 0 \nonumber
 \end{align}
 which is to be satisfied by pairs $(\phi,f)$ with $\phi \in \m_A \cdot G_A^{-1,1}$ and $f \in \m_A \cdot G_A^{0,0}$. Its solutions form a homogeneous\footnote{The proof is the same as for $\mathrm{MC}(L^\bullet,-)$.} functor 
 $$\mathrm{SMC}(G^{\bullet,\bullet},-): \mathbf{Art}_\Lambda \to \mathbf{Set}$$
 of Artin rings.
\end{defn}

\begin{lemma}
 If $(\phi,f) \in \mathrm{SMC}(G^{\bullet,\bullet},A)$, then $G_A^{\bullet,\bullet}$ forms an $A$-linear differential bg Batalin--Vilkovisky algebra with the differentials $\Delta_f$ and $\bar\partial_\phi$ . More precisely:
 \begin{enumerate}[label=\emph{(\roman*)}]
  \item $\bar\partial_\phi^2(\theta) = 0$ and $\bar\partial_\phi\Delta_f(\theta) + \Delta_f\bar\partial_\phi(\theta) = 0$;
  \item the two derivation rules for $\bar\partial_\phi(\theta \wedge \xi)$ and $\bar\partial_\phi[\theta,\xi]$ hold;
  \item $\Delta_f^2(\theta) = 0$, $\Delta_f(1) = 0$, and the derivation rule for $\Delta_f[\theta,\xi]$ holds;
  \item the Bogomolov--Tian--Todorov formula for $\Delta_f(\theta \wedge \xi)$ holds.
 \end{enumerate}
\end{lemma}
\begin{proof}
 All statements
 follow from easy and straightforward computations.
\end{proof}

For a $\Lambda$-linear curved Gerstenhaber algebra, or, more generally, a $\Lambda$-linear curved $\cP$-algebra as in Definition~\ref{L-curved-def}, the gauge transform $\mathrm{exp}_\theta$ for an element $\theta \in \m_A \cdot T_A^0$ gives an isomorphism between $(\E_A^{\bullet,\bullet},\bar\partial_\phi)$ and $(\E_A^{\bullet,\bullet},\bar\partial_{\phi'})$ for $\phi' = \mathrm{exp}_\theta * \phi$. The same is true for $\Lambda$-linear curved Batalin--Vilkovisky algebras as soon as we extend the gauge action to $f \in \m_A \cdot G_A^{0,0}$ appropriately.

\begin{lemdef}\label{gauge-action-BV}\note{gauge-action-BV}
 Let $G^{\bullet,\bullet}$ be a $\Lambda$-linear curved Batalin--Vilkovisky algebra, and let $B' \to B$ be a surjection in $\mathbf{Art}_\Lambda$ with kernel $I \subset B'$. Let $\theta \in I \cdot G_{B'}^{-1,0}$ and $f \in \m_{B'} \cdot G_{B'}^{0,0}$. Then we have 
 $$\mathrm{exp}_\theta \circ \Delta_f(\xi) = \Delta \circ \mathrm{exp}_\theta(\xi) + [\mathrm{exp}_\theta(f) - T(\nabla_\theta)(\Delta \theta),\: \mathrm{exp}_\theta(\xi)].$$
 This gives rise to the \emph{gauge action}
 $$\mathrm{exp}_\theta * f := \mathrm{exp}_\theta(f) - T(\nabla_\theta)(\Delta\theta) = f + T(\nabla_\theta)([\theta,f] - \Delta\theta),$$
 which satisfies $\mathrm{exp}_\theta * (\mathrm{exp}_\xi * f) = \mathrm{exp}_{\theta \odot \xi} * f$ with the Baker--Campbell--Hausdorff product $\theta \odot \xi$.
\end{lemdef}
\begin{proof}
 The proof of the first statement is the same as the one of Lemma~\ref{gauge-action} with $\bar\partial$ replaced by $\Delta$ and $\phi$ replaced by $f$. The formula $\mathrm{exp}_\theta * (\mathrm{exp}_\xi * f) = \mathrm{exp}_{\theta \odot \xi} * f$ comes from the fact that $\mathrm{exp}_\theta * f$ is the gauge action in the differential graded Lie algebra $M^\bullet := (G_{B'}^{\bullet + 1,0},[-,-],\Delta)$. Note that the grading is such that $M^0 = G_{B'}^{-1,0}$ and $M^1 = G_{B'}^{0,0}$.
\end{proof}

We say that two pairs $(\phi,f)$ and $(\phi',f')$ are \emph{gauge equivalent} if there is some $\theta \in \m_A \cdot G_A^{-1,0}$ with $\mathrm{exp}_\theta * \phi = \phi'$ and $\mathrm{exp}_\theta * f = f'$. In this case, $\mathrm{exp}_\theta$ induces an isomorphism between the $\Lambda$-linear curved Batalin--Vilkovisky algebras $(G_A^{\bullet,\bullet},\Delta_f,\bar\partial_\phi)$ and $(G_A^{\bullet,\bullet},\Delta_{f'},\bar\partial_{\phi'})$. Just as in the case of $\phi$, the pair $(\phi,f)$ is a solution of the semi-classical Maurer--Cartan equation if and only if $(\phi',f')$ is. To see the second part, observe that 
\begin{align}
 \bar\partial T(\nabla_\theta)&(\Delta\theta) + \Delta T(\nabla_\theta)(\bar\partial\theta) \nonumber \\
 &= T(\nabla_\theta)(\bar\partial\Delta\theta) + T(\nabla_\theta)(\Delta\bar\partial\theta) + [T(\nabla_\theta)(\Delta\theta),T(\nabla_\theta)(\bar\partial\theta)]. \nonumber
\end{align}
This implies 
$$\mathrm{exp}_\theta(y) - y = [T(\nabla_\theta)(\Delta\theta),T(\nabla_\theta)(\bar\partial\theta)] - \bar\partial T(\nabla_\theta)(\Delta\theta) - \Delta T(\nabla_\theta)(\bar\partial\theta).$$
From here, we obtain 
$$\bar\partial f' + [\phi',f'] + y + \Delta\phi' = \mathrm{exp}_\theta(\bar\partial f + [\phi,f] + y + \Delta\phi)$$
by a straightforward computation using Lemma~\ref{gauge-action} and Lemma~\ref{gauge-action-BV}. Thus, $(\phi,f)$ is a semi-classical Maurer--Cartan solution if and only if $(\phi',f')$ is. Gauge equivalence classes of elements $(\phi,f) \in \mathrm{SMC}(G^{\bullet,\bullet},A)$ form a deformation functor
$$\mathrm{SDef}(G^{\bullet,\bullet},-): \mathbf{Art}_\Lambda \to \mathbf{Set}.$$
In fact, $(H_1)$ is easy since $\mathrm{SMC}(G^{\bullet,\bullet},-)$ is homogeneous, and $(H_2)$ follows (as in the classical case) from a careful analysis of lifts of gauge transforms in the diagram in \cite[10.5]{Felten2022}. As in the classical case, $\mathrm{SDef}(G^{\bullet,\bullet},-)$ is unobstructed if and only if $\mathrm{SMC}(G^{\bullet,\bullet},-)$ is unobstructed.

\par\vspace{\baselineskip}

Now suppose that $G^{\bullet,\bullet}$ is part of a curved Batalin--Vilkovisky calculus $(G^{\bullet,\bullet},A^{\bullet,\bullet})$. Recall that, for $\phi \in \m_A \cdot G_A^{-1,1}$, we have on $A_A^{\bullet,\bullet}$ the new predifferential
$$\bar\partial_\phi(\alpha) := \bar\partial\alpha + \cL_\phi(\alpha)$$
since $\nabla_\theta = \cL_\theta$ on $A^{\bullet,\bullet}$. We wish to find a modification $\omega_{(\phi,f)}$ of $\omega$ such that:
\begin{enumerate}[label=(\roman*)]
 \item $(G_A^{\bullet,\bullet},A_A^{\bullet,\bullet},\Delta_f,\bar\partial_\phi,\omega_{(\phi,f)})$ is a $\Lambda$-linear curved Batalin--Vilkovisky calculus for all pairs $(\phi,f)$ with $\ell_\phi$ and $y_{(\phi,f)} := \bar\partial f + [\phi,f] + y + \Delta\phi$;
 \item if $(\phi,f)$ and $(\phi',f')$ are gauge equivalent via $\theta$, then $(G_A^{\bullet,\bullet},A_A^{\bullet,\bullet},\Delta_f,\bar\partial_\phi,\omega_{(\phi,f)})$ and $(G_A^{\bullet,\bullet},A_A^{\bullet,\bullet},\Delta_{f'},\bar\partial_{\phi'},\omega_{(\phi',f')})$ are isomorphic via $\mathrm{exp}_\theta$.
\end{enumerate}
If $(\phi,f)$ is a semi-classical Maurer--Cartan solution, then $(G_A^{\bullet,\bullet},A_A^{\bullet,\bullet},\Delta_f,\bar\partial_\phi,\omega_{(\phi,f)})$ is automatically a \emph{differential} bg Batalin--Vilkovisky calculus.

Since $\omega_{(\phi,f)} \in A_A^{d,0}$ would be a constant, we expect that 
\begin{equation}\label{omega-cond}
 \omega_{(\phi',f')} = \mathrm{exp}_\theta(\omega_{(\phi,f)}).
\end{equation}
\begin{lemma}\label{omega-compat}\note{omega-compat}
 The choice
 \begin{align}
 \omega_{(\phi,f)} := \omega_f := e^{f \ \invneg }\omega :=  \sum_{n = 0}^\infty \frac{(f \ \invneg \ )^n(\omega)}{n!} \:
 &= \omega + (f \ \invneg \ \omega) + \frac{f \ \invneg \ (f \ \invneg\ \omega)}{2} + ... \nonumber \\
 &= \left(\sum_{n = 0}^\infty \frac{f^n}{n!}\right) \ \invneg\ \omega \:= (e^f) \  \invneg\ \omega, \nonumber
\end{align}
satisfies \eqref{omega-cond}.
\end{lemma}
\begin{proof}
 We have to show $\mathrm{exp}_\theta(e^f \ \invneg \ \omega) = e^{f'} \ \invneg \ \omega$ with $f' = \mathrm{exp}_\theta(f) - T(\nabla_\theta)(\Delta\theta)$. It is straightforward to see that it is sufficient to show 
 $$\mathrm{exp}(-T(\nabla_\theta)(\Delta\theta)) \ \invneg \ \omega = \mathrm{exp}_\theta(\omega)$$
 where the $\mathrm{exp}$ on the left hand side is the exponential of the nilpotent element $-T(\nabla_\theta)(\Delta\theta)$ in the ring $G_A^{0,0}$. We can expand the left hand side to 
 $$\left(1 + \sum_{t = 1}^\infty \sum_{n = 1}^t \frac{(-1)^n}{n!} \sum_{\substack{k_1 + ... + k_n = t \\ k_i \geq 1}}  \prod_{i = 1}^n \frac{1}{k_i!} \nabla_\theta^{k_i - 1}\Delta(\theta)\right) \ \invneg \ \omega.$$
 Let us denote the summand for $t \geq 1$ in this formula by $S(t)$. The right hand side of the equation is 
 $$\mathrm{exp}_\theta(\omega) = \omega + \sum_{t = 1}^\infty \frac{1}{t!}\cL^t_\theta(\omega).$$
 To conclude the proof, we show 
 $$(S(t) \ \invneg \ 1) \wedge \omega =  \frac{1}{t!}\cL^t_\theta(\omega)$$
 by induction on $t \geq 1$. We have $S(1) = -\Delta\theta$ and $\cL_\theta(\omega) = -\Delta\theta \ \invneg \ \omega$ for the base case. On the right hand side, the induction step yields 
 $$\frac{1}{(t + 1)!}\cL_\theta^{t + 1}(\omega) = \frac{1}{t + 1}\Big(\big([\theta,S(t)] - S(t) \wedge \Delta\theta\big) \ \invneg \ 1\Big) \wedge \omega$$
 so that it remains to show 
 $$S(t + 1) = \frac{1}{t + 1}\big([\theta,S(t)] - S(t) \wedge \Delta\theta\big).$$
 This is slightly tricky combinatorics. Let $I(t)$ be the set of tuples $(k_1,...,k_n)$ of natural numbers $\geq 1$ of variable length $\geq 1$ with $k_1 + ... + k_n = t$, and let $s(k_1,...,k_n) = \frac{(-1)^n}{n!} \prod_{i = 1}^n \frac{1}{k_i!}\nabla_\theta^{k_i - 1}\Delta\theta$ so that $S(t) = \sum_{\underline k \in I(t)} s(\underline k)$. Then 
 \begin{align}
  [\theta,s(k_1,...,k_n)] &- s(k_1,...,k_n) \wedge \Delta\theta \nonumber \\ &= \sum_{j = 1}^n (k_j + 1) s(k_1,...,k_j + 1,...,k_n) + \sum_{j = 0}^n s(k_1,...,k_{j},1,k_{j + 1},...,k_n). \nonumber
 \end{align}
 Collecting all summands $s(\ell_1,...,\ell_n)$ for $\ell_1 + ... + \ell_n = t + 1$, we find the assertion.
\end{proof}

The volume form $\omega_f$ does what we want.

\begin{lemma}\label{MC-sol-modified-BV-calculus}\note{MC-sol-modified-BV-calculus}
 Let $\phi \in \m_A \cdot G_A^{-1,1}$ and $f \in \m_A \cdot G_A^{0,0}$.
 Then $(G_A^{\bullet,\bullet},A_A^{\bullet,\bullet},\Delta_f,\bar\partial_\phi,\omega_f)$ is a $\Lambda$-linear curved Batalin--Vilkovisky calculus with $\ell_\phi$ and $y_{(\phi,f)}$. If $(\phi,f) \in \mathrm{SMC}(G^{\bullet,\bullet},A)$, then it is a differential bg Batalin--Vilkovisky calculus. Concretely, we have:
 \begin{enumerate}[label=\emph{(\roman*)}]
  \item $\bar\partial_\phi^2(\alpha) = \cL_{\ell_\phi}(\alpha)$ and $\bar\partial_\phi\partial(\alpha) + \partial\bar\partial_\phi(\alpha) = 0$;
  \item the two derivation rules hold for $\bar\partial_\phi(\alpha \wedge \beta)$ and $\bar\partial_\phi(\theta \ \invneg\ \alpha)$;
  \item $\bar\partial_\phi \omega_f = y_{(\phi,f)} \ \invneg \ \omega_f$; in particular, if $(\phi,f)$ is a semi-classical Maurer--Cartan solution, then 
  $$\bar\partial_\phi\omega_f = 0$$
  and $\bar\partial_\phi(\theta) \ \invneg\ \omega_f = \bar\partial_\phi(\theta\ \invneg\ \omega_f)$;
  \item $\kappa_f: G_A^{p,q} \to G_A^{d + p,q}, \: \theta \mapsto \theta \ \invneg\ \omega_f,$ is an isomorphism of $A$-modules;
  \item $\upsilon_f: A_A^{0,q} \to A_A^{d,q}, \: \alpha \mapsto \alpha \wedge \omega_f,$ is an isomorphism of $A$-modules;
  \item $\Delta_f(\theta) \ \invneg\ \omega_f = \partial(\theta\ \invneg\ \omega_f)$, i.e., the modified Batalin--Vilkovisky operator $\Delta_f$ arises from the de Rham differential $\partial$ via the new isomorphism $\kappa_f$.
 \end{enumerate}
\end{lemma}
\begin{proof}
 The first two statements follow from the general theory of $\Lambda$-linear curved $\cP$-algebras. For $\bar\partial_\phi\omega_f = y_{(\phi,f)} \ \invneg \ \omega_f$, first note that $\bar\partial_\phi(e^f) = \bar\partial_\phi(f) \wedge e^f$. Then we have 
 \begin{align}
  \bar\partial_\phi(\omega_f) &= \bar\partial_\phi(e^f) \ \invneg \ \omega + e^f \ \invneg\ \bar\partial_\phi\omega = e^f \ \invneg\ (\bar\partial_\phi f\ \invneg \ \omega + \bar\partial_\phi\omega) \nonumber \\
  &= e^f \ \invneg \ (\bar\partial_\phi f \ \invneg \ \omega + y \ \invneg \ \omega + \partial(\phi \ \invneg \ \omega)) \nonumber \\
  &= e^f \ \invneg \ (\bar\partial f + [\phi,f] + y + \Delta\phi) \ \invneg \ \omega = y_{(\phi,f)} \ \invneg \ \omega_f \nonumber
 \end{align}
 since $y \ \invneg\ \omega = \bar\partial\omega$ and $\cL_\phi(\omega) = \partial(\phi \ \invneg \ \omega) = \Delta\phi \ \invneg \ \omega$.
 Next, let $\mu_f(\theta) := e^f \wedge \theta$; it is an isomorphism of $A$-modules with inverse $\mu_{-f}$. Since $\kappa_f = \kappa \circ \mu_f$ and $\kappa$ is an isomorphism, $\kappa_f$ is an isomorphism as well. The proof that $\upsilon_f$ is an isomorphism is similar.
 The last statement follows from an application of the Bogomolov--Tian--Todorov formula when we also use $[e^f,\theta] = [f,\theta] \wedge e^f$.
\end{proof}

If $(\phi,f)$ and $(\phi',f')$ are gauge equivalent via $\theta$, then the two $\Lambda$-linear curved Batalin--Vilkovisky calculi $(G_A^{\bullet,\bullet},A_A^{\bullet,\bullet},\Delta_f,\bar\partial_\phi,\omega_f)$ and $(G_A^{\bullet,\bullet},A_A^{\bullet,\bullet},\Delta_{f'},\bar\partial_{\phi'},\omega_{f'})$ are isomorphic via $\mathrm{exp}_\theta$. Indeed, the compatibilities with most operations follow from generalities on (Cartanian) $\Lambda$-linear curved $\cP$-algebras, the compatibility with $\Delta$ is Lemma~\ref{gauge-action-BV}, the compatibility with $\omega$ is Lemma~\ref{omega-compat}, and we have checked the compatibility with $\ell$ and $y$ when we have shown that solving the semi-classical Maurer--Cartan equation is a gauge invariant property.

\section{Perfect curved Batalin--Vilkovisky algebras and calculi}

In this section, we present conditions under which we can show in Chapter~\ref{two-abstract-unobstructed-sec} unobstructedness of our Maurer--Cartan functors using the argument of Chan--Leung--Ma in \cite{ChanLeungMa2023}.

If a $\Lambda$-linear curved \emph{Gerstenhaber} calculus $(G^{\bullet,\bullet},A^{\bullet,\bullet})$ arises from a \emph{proper}, saturated, and log smooth morphism, log Calabi--Yau or not, then it is perfect in the sense of the following definition:

\begin{defn}\label{G-calc-q-perf-def}\note{G-calc-q-perf-def}
 Let $(G^{\bullet,\bullet},A^{\bullet,\bullet})$ be a $\Lambda$-linear curved Gerstenhaber calculus. We say that it is \emph{quasi-perfect}\index{Gerstenhaber calculus!quasi-perfect}\index{quasi-perfect} if it is strictly faithful and the following conditions hold:
 \begin{enumerate}[label=(\roman*)]
  \item The first spectral sequence $'\!E_{0;r}$ associated with the double complex $(A_0^{\bullet,\bullet},\partial,\bar\partial)$ degenerates at $E_1$.
  \item For every $A \in \mathbf{Art}_\Lambda$ and every $k \geq 0$, the induced map 
  $$H^k(A_A^\bullet,d) \to H^k(A_0^\bullet, d)$$
  is surjective, where $d = \partial + \bar\partial + \ell \ \invneg \ (-)$.
 \end{enumerate}
  We say that it is \emph{perfect}\index{Gerstenhaber calculus!perfect}\index{perfect!Gerstenhaber calculus} if, furthermore, for every $A \in \mathbf{Art}_\Lambda$, the $A$-module $H^k(A_A^\bullet,d)$ is finitely generated.
\end{defn}
\begin{rem}
 Several remarks are in order.
 \begin{enumerate}[label=(\arabic*)]
  \item The first spectral sequence is the one with $E_1$-page $'\!E_{0;1}^{p,q} = H^q(G_0^{p,\bullet},\bar\partial)$ and boundary maps $d_r^{p,q}:\  '\!E_{0;r}^{p,q} \to\ '\!E_{0;r}^{p + r,q + 1 - r}$. The index $0$ in $E_{0;r}$ is to indicate that this spectral sequence is constructed from $A_0^{\bullet,\bullet}$.
  \item If $(G^{\bullet,\bullet},A^{\bullet,\bullet})$ is perfect, and if $\phi \in \m_A \cdot G_A^{-1,1}$ is a Maurer--Cartan solution, then the first spectral sequence of the double complex $(A_A^{\bullet,\bullet},\partial,\bar\partial_\phi)$ degenerates at $E_1$ by Proposition~\ref{spectral-sequence-Artin-ring}. The cohomology $H^q(A_A^{p,\bullet},\bar\partial_\phi)$ is a flat $A$-module whose formation commutes with base change, as is the total cohomology $H^k(A_A^\bullet,\partial + \bar\partial_\phi)$. This is possible without any finiteness conditions by the very general nature of Proposition~\ref{spectral-sequence-Artin-ring}, and will be discussed in more detail toward the end of this section, especially Lemma~\ref{hypercohom-comparison}.
  \item Geometrically, the $E_1$-degeneration of $'\!E_{0;r}$ corresponds to the Hodge--de Rham degeneration over the point, the surjectivity of $H^k(A_A^\bullet,d) \to H^k(A_0^\bullet,d)$ corresponds to the freeness of the relative algebraic de Rham cohomology $R^kf_*\Omega_{X_A/S_A}^\bullet$, and the finiteness condition corresponds to properness of $f: X_A \to S_A$. However, especially the surjectivity condition is stronger than its geometric interpretation because it even has a content if there is no Maurer--Cartan element which would correspond to a geometric map $f: X_A \to S_A$.
 \end{enumerate}
\end{rem}

If we have a $\Lambda$-linear curved \emph{Batalin--Vilkovisky} calculus, then we can use the isomorphisms $\kappa$ to state (quasi-)perfectness in terms of the Batalin--Vilkovisky \emph{algebra}.

\begin{lemma}\label{perfectness-stated-on-G}\note{perfectness-stated-on-G}
 Let $(G^{\bullet,\bullet},A^{\bullet,\bullet})$ be a $\Lambda$-linear curved Batalin--Vilkovisky calculus. Assume that $G^{\bullet,\bullet}$ is strictly faithful. Then the calculus is quasi-perfect if and only if the following conditions hold:
 \begin{enumerate}[label=\emph{(\roman*)}]
  \item The first spectral sequence $'\!E_{0;r}$ associated with the double complex $(G_0^{\bullet,\bullet},\Delta,\bar\partial)$ over $A_0 = \kk$ degenerates at $E_1$.
  \item For every $A \in \mathbf{Art}_\Lambda$, the induced map $H^k(G_A^\bullet,\check d) \to H^k(G_0^\bullet,\check d)$ with 
  $$\check d = \Delta + \bar\partial + (\ell + y) \wedge (-)$$
  is surjective.
 \end{enumerate}
 In this case, $(G^{\bullet,\bullet},A^{\bullet,\bullet})$ is perfect if and only if, for every $A \in \mathbf{Art}_\Lambda$, the $A$-module $H^k(G_A^\bullet,\check d)$ is finitely generated.
\end{lemma}
\begin{proof}
 Over $A_0 = \kk$, the two double complexes $(G_0^{\bullet,\bullet},\Delta,\bar\partial)$ and $(A_0^{\bullet,\bullet},\partial,\bar\partial)$ are isomorphic via $\kappa = (-) \ \invneg\ \omega$. This is because $\bar\partial\omega \in \m_\Lambda \cdot A^{d,1}$. In particular, their spectral sequences are isomorphic. Since $\kappa \circ \check d = d \circ \kappa$, the isomorphism $\kappa$ also induces an isomorphism $H^k(G_A^\bullet,\check d) \cong H^k(A_A^{\bullet},d)$, which is compatible with base change.
\end{proof}

For a $\Lambda$-linear curved Batalin--Vilkovisky algebra, we turn this into a definition:

\begin{defn}
 Let $G^{\bullet,\bullet}$ be a $\Lambda$-linear curved Batalin--Vilkovisky algebra.  We say that $G^{\bullet,\bullet}$ is \emph{quasi-perfect}\index{Batalin--Vilkovisky algebra!quasi-perfect}\index{quasi-perfect!Batalin--Vilkovisky algebra} respective \emph{perfect}\index{perfect!Batalin--Vilkovisky algebra}\index{Batalin--Vilkovisky algebra!perfect} if it is strictly faithful and satisfies the corresponding conditions in Lemma~\ref{perfectness-stated-on-G}.
\end{defn}

\begin{rem}
 The third condition might seem a bit odd. In view of the $T^1$-lifting criterion of Z.~Ran in \cite{Ran1992}, one might expect a condition about the freeness of, for example, $H^j(G_A^{-1,\bullet},\bar\partial_\phi)$ for every Maurer--Cartan solution $\phi$. We will see in \ref{freeness-section} that, what is required here, is in fact stronger, and, in a sense, more elegant. We do not know if the weaker condition would be sufficient to prove unobstructedness theorems, but if the weaker condition (for all $H^j(G_A^{i,\bullet},\bar\partial_\phi)$) holds and the functor is unobstructed, then also the stronger condition as stated above holds. 
\end{rem}

The following is the abstract unobstructedness theorem of \cite[Thm.~5.6]{ChanLeungMa2023}, stated in its proper clarity and generality.\footnote{In fact, quasi-perfectness can further be weakened to semi-perfectness, which we discuss below when we prove the theorem.}

\begin{thm}[First abstract unobstructedness theorem]\index{Abstract unobstructedness theorem!first}
 Let $G^{\bullet,\bullet}$ be a quasi-perfect $\Lambda$-linear curved Batalin--Vilkovisky algebra. Then $\mathrm{SMC}(G^{\bullet,\bullet},-)$ is unobstructed, i.e., for every surjection $B' \to B$ in $\mathbf{Art}_\Lambda$, the induced map $\mathrm{SMC}(G^{\bullet,\bullet},B') \to \mathrm{SMC}(G^{\bullet,\bullet},B)$ is surjective.
\end{thm}

This is sufficient as long as we only want to know that a solution of the (classical extended) Maurer--Cartan equation exists. Namely, we can start with the pair $(0,0)$ over $A_0$, and then we can lift it from order to order and obtain a solution $(\phi_k,f_k)$ of the semi-classical extended Maurer--Cartan equation. Then $\phi_k$ is also a solution of the classical extended Maurer--Cartan equation. This is sufficient e.g.~for the situation in our previous article \cite{FFR2021} (with Filip and Ruddat), where we were interested in the existence of a deformation in the geometric setup. When we want to show that \emph{every} classical Maurer--Cartan solution can be lifted independently of an $f$ that makes $(\phi,f)$ into a semi-classical Maurer--Cartan solution, i.e., when we want to show that the deformation functor in a geometric setup is unobstructed, we have to work again with $\Lambda$-linear curved Batalin--Vilkovisky \emph{calculi}.

\begin{thm}[Second abstract unobstructedness theorem]\label{second-abstract-unob-thm}\note{second-abstract-unob-thm}\index{Abstract unobstructedness theorem!second}
 Let $(G^{\bullet,\bullet},A^{\bullet,\bullet})$ be a quasi-perfect $\Lambda$-linear curved Batalin--Vilkovisky calculus. Then $\mathrm{MC}(G^{\bullet,\bullet},-)$ is unobstructed.
\end{thm}

Note that we require a perfect curved Batalin--Vilkovisky calculus. It is not sufficient to have a perfect curved Gerstenhaber calculus, as they arise from any geometry, i.e., not necessarily Calabi--Yau. On the other hand, the finiteness condition is usually satisfied in geometry, but it is not necessary for the abstract theorem. This second abstract unobstructedness theorem has, to the author's knowledge, not been addressed in \cite{ChanLeungMa2023}.

\section{Freeness of $H^j(G_A^{i,\bullet},\bar\partial_\phi)$ and $H^j(A_A^{i,\bullet},\bar\partial_\phi)$}\label{freeness-section}\note{freeness-section}

Let, for a moment, $Q$ be a sharp toric monoid, and let $f_A: X_A \to S_A$ be a proper, log smooth, and saturated morphism of fs log schemes, where $S_A = \Spec(Q \to A)$ for an Artinian local $\kk\llbracket Q\rrbracket$-algebra $A$ with residue field $\kk$. Then, by the Hodge--de Rham degeneration theorem, 
$$R^q(f_A)_*\Omega_{X_A/S_A}^p$$ is a finitely generated free $A$-module whose formation commutes with base change. If $f_A: X_A \to S_A$ is log Calabi--Yau, i.e., we have $\omega_{X_A/S_A} \cong \cO_{X_A}$, then also 
$$R^q(f_A)_*\Theta^p_{X_A/S_A}$$ is a finitely generated free $A$-module whose formation commutes with base change. This suggests that, for a Maurer--Cartan solution $\phi$, the $A$-modules $H^j(A_A^{i,\bullet},\bar\partial_\phi)$ and $H^j(G_A^{i,\bullet},\bar\partial_\phi)$ might be free and commute with base change. This is in fact true under the quasi-perfectness assumption, not even assuming finite generation.

\subsection{Freeness of $H^j(G_A^{i,\bullet},\bar\partial_\phi)$ and the relevance of $\check d = \bar\partial + \Delta + (\ell + y) \wedge (-)$}

Let $G^{\bullet,\bullet}$ be a quasi-perfect $\Lambda$-linear curved Batalin--Vilkovisky algebra. Suppose we have a semi-classical Maurer--Cartan solution $(\phi,f)$ over $A \in \mathbf{Art}_\Lambda$. We have seen above that then
$$(C^{\bullet,\bullet},d,d') := (G_A^{\bullet,\bullet},\Delta_f,\bar\partial_\phi)$$
is a double complex. For the moment, we are interested in the cohomology $H^k(G_A^\bullet,\bar\partial_\phi + \Delta_f)$ of its total complex. By our quasi-perfectness assumption, we know that
$$H^k(G_A^\bullet,\check d) \to H^k(G_0^\bullet,\check d)$$
is surjective, 
where $\check d\theta = \bar\partial\theta + \Delta\theta + (\ell + y) \wedge \theta$.
Since $\phi$ and $f$ are nilpotent for the $\wedge$-product, we can define $E := \sum_{n = 0}^\infty\frac{(f + \phi)^n}{n!}$; then we have a map
$$\Phi_E: G_A^m \to G_A^m, \quad \theta \mapsto E \wedge \theta.$$
This map is an isomorphism of $A$-modules, for its inverse is $\Psi_E(\theta) = \sum_{n = 0}^\infty\frac{(-1)^n(f + \phi)^n}{n!} \wedge \theta$.

\begin{lemma}
 We have 
 $$\Phi_E \circ (\bar\partial_\phi + \Delta_f) = \check d \circ \Phi_E.$$
 In particular, we have 
 $$H^k(G_A^\bullet,\bar\partial_\phi + \Delta_f) \cong H^k(G_A^\bullet,\check d),$$
 so that $H^k(G_A^\bullet,\Delta_f + \bar\partial_\phi) \to H^k(G_0^\bullet,\Delta + \bar\partial)$ is surjective. 
\end{lemma}
\begin{proof}
 Recall the formulae
 $$\bar\partial(\theta^n) = n\bar\partial(\theta) \wedge \theta^{n - 1} \quad \mathrm{and} \quad \Delta(\theta^n) = n\Delta(\theta) \wedge \theta^{n - 1} + \frac{n(n -1)}{2}\theta^{n - 2} \wedge [\theta,\theta]$$
 for $\theta \in G_A^\bullet$ of even degree and $n \geq 2$. By computation, the assertion $\Phi_E \circ (\bar\partial_\phi + \Delta_f)(\xi) = \check d \circ \Phi_E(\xi)$ is equivalent to 
 $$E \wedge [f + \phi,\xi] = [E,\xi] + \left(\bar\partial(f + \phi) + \Delta(f + \phi) + \frac{1}{2}[f + \phi,f + \phi] + \ell + y\right) \wedge E \wedge \xi.$$
 Since $E \wedge [f + \phi,\xi] = [E,\xi]$ and $\check d(f + \phi) = 0$, this holds.
\end{proof}

\begin{prop}\label{free-G-prop}\note{free-G-prop}
 Let $G^{\bullet,\bullet}$ be a quasi-perfect $\Lambda$-linear curved Batalin--Vilkovisky algebra. Let $(\phi,f)$ be a semi-classical Maurer--Cartan solution over $A \in \mathbf{Art}_\Lambda$. Then 
  $$H^j(G_A^{i,\bullet},\bar\partial_\phi)$$
  is a free $A$-module, and its formation commutes with base change in $\mathbf{Art}_A$. If $G^{\bullet,\bullet}$ is perfect, then this cohomology module is finitely generated.
\end{prop}
\begin{proof}
 We apply Proposition~\ref{spectral-sequence-Artin-ring} and Corollary~\ref{spectral-sequence-Artin-ring-base-change} to the double complex $(C^{\bullet,\bullet},d,d') = (G_A^{\bullet,\bullet},\Delta_f,\bar\partial_\phi)$. If $H^k(G_A^\bullet,\Delta_f + \bar\partial_\phi)$ is finitely generated, so are its subquotients.
\end{proof}
\begin{rem}
 Whenever $G^{\bullet,\bullet}$ is part of a perfect $\Lambda$-linear curved Batalin--Vilkovisky calculus $(G^{\bullet,\bullet},A^{\bullet,\bullet})$, then, by Theorem~\ref{SMC-MC-smooth} below, a classical Maurer--Cartan solution $\phi$ can be extended to a semi-classical Maurer--Cartan solution $(\phi,f)$; thus, in this case, the proposition also holds for classical Maurer--Cartan solutions $\phi$.
\end{rem}

\subsection{Freeness of $H^j(A_A^{i,\bullet},\bar\partial_\phi)$ and the relevance of $d = \partial + \bar\partial + \ell \ \invneg \ (-)$}

Let $(G^{\bullet,\bullet},A^{\bullet,\bullet})$ be a quasi-perfect $\Lambda$-linear curved Gerstenhaber calculus. It may be actually a Batalin--Vilkovisky calculus, but we do not require this extra structure here. Suppose we have a classical Maurer--Cartan solution $\phi$ over $A \in \mathbf{Art}_\Lambda$, not necessarily part of a semi-classical Maurer--Cartan solution $(\phi,f)$. Then we can form a double complex 
$$(C^{\bullet,\bullet},d,d') := (A_A^{\bullet,\bullet},\partial,\bar\partial_\phi).$$
Here, we know that $H^k(A_A^\bullet,d) \to H^k(A_0^\bullet,d)$ is surjective with $d\alpha := \partial\alpha + \bar\partial\alpha + (\ell \ \invneg\ \alpha).$ 
After setting $E := e^\phi = \sum_{n = 0}^\infty \frac{\phi^n}{n!}$, we have an isomorphism
$$\Phi_E: A_A^m \to A_A^m,  \quad \alpha \mapsto E \ \invneg \ \alpha,$$
of $A$-modules.

\begin{lemma}\label{hypercohom-comparison}\note{hypercohom-comparison}
 We have 
 $$\Phi_E \circ (\partial + \bar\partial_\phi)(\alpha) = d \circ \Phi_E(\alpha).$$
 In particular, we have 
 $$H^k(A_A^\bullet,\partial + \bar\partial_\phi) \cong H^k(A_A^\bullet,d),$$
 so that $H^k(A_A^\bullet,\partial + \bar\partial_\phi) \to H^k(A_0^\bullet,d)$ is surjective.
\end{lemma}
\begin{proof}
 First, we have
 $\cL_{e^\phi}(\alpha) = e^\phi \ \invneg\ \cL_\phi(\alpha) + \frac{1}{2}[\phi,\phi] \ \invneg\ e^\phi \ \invneg\ \alpha.$
 Then we compute
 \begin{align}
  &\quad\: d \circ \Phi_E(\alpha) - \Phi_E \circ (\partial + \bar\partial_\phi)(\alpha) \nonumber \\
  &= \partial(e^\phi \ \invneg\ \alpha) + \bar\partial(e^\phi \ \invneg\ \alpha) + (\ell \ \invneg\ e^\phi \ \invneg\ \alpha) - (e^\phi \ \invneg\ \partial\alpha) - (e^\phi \ \invneg\ \bar\partial\alpha) - (e^\phi \ \invneg\ \cL_\phi(\alpha)) \nonumber \\
  &= \cL_{e^\phi}(\alpha) + \bar\partial(e^\phi) \ \invneg\ \alpha + (\ell \ \invneg\ e^\phi \ \invneg\ \alpha) - (e^\phi \ \invneg\ \cL_\phi(\alpha)) \nonumber \\
  &= (\bar\partial\phi \ \invneg\ e^\phi \ \invneg\ \alpha) + (\frac{1}{2}[\phi,\phi] \ \invneg\ e^\phi \ \invneg\ \alpha) + (\ell \ \invneg\ e^\phi \ \invneg\ \alpha) = 0 \nonumber
 \end{align}
 because $\phi$ is a classical Maurer--Cartan solution.
\end{proof}

Then we have:

\begin{prop}\label{free-A-prop}\note{free-A-prop}
 Let $(G^{\bullet,\bullet},A^{\bullet,\bullet})$ be a quasi-perfect $\Lambda$-linear curved Gerstenhaber calculus. Let $\phi$ be a classical Maurer--Cartan solution over $A \in \mathbf{Art}_\Lambda$. Then 
  $$H^j(A_A^{i,\bullet},\bar\partial_\phi)$$
  is a free $A$-module, and its formation commutes with base change in $\mathbf{Art}_A$. If $(G^{\bullet,\bullet},A^{\bullet,\bullet})$ is perfect, then the cohomology module is finitely generated.
\end{prop}
\begin{proof}
 We apply Proposition~\ref{spectral-sequence-Artin-ring} and Corollary~\ref{spectral-sequence-Artin-ring-base-change} to the double complex $(A_A^{\bullet,\bullet},\partial,\bar\partial_\phi)$.
\end{proof}



\chapter{The two abstract unobstructedness theorems}\label{two-abstract-unobstructed-sec}\note{two-abstract-unobstructed-sec}

We prove the two abstract unobstructedness theorems. We achieve this by solving the \emph{quantum extended Maurer--Cartan equation}, hence the names classical and semi-classical for the above Maurer--Cartan equations.

\section{Formal power series and formal Laurent series in $G^{p,q}_A$}

Let $G^{\bullet,\bullet}$ be a $\Lambda$-linear curved Batalin--Vilkovisky algebra. Let us write $G_\Lambda^{p,q} := G^{p,q}$.
Then we set
$$G^{p,q}_A\fpsh := \{\sum_{i = 0}^\infty g_i \hb^i \ | \ g_i \in G_A^{p,q} \}$$
for either $A \in \mathbf{Art}_\Lambda$ or $A = \Lambda$ itself, a convention which we follow in this section.
Compared with \cite{ChanLeungMa2023}, we have changed the notation from $t$ to $\hb$ in order to avoid confusion with the frequent use of $t$ as a variable in $\Lambda$. We also set
$$G^{p,q}_A\fpsu := \{\sum_{i = 0}^\infty g_i u^i \ | \ g_i \in G_A^{p,q}\},$$
where we think of $u$ as a variable with $u^2 = \hb$, i.e., 
$G_A^{p,q}\fpsh \subset G_A^{p,q}\fpsu$
is the subspace of formal power series with entries in even degrees.\footnote{In \cite{ChanLeungMa2023}, the notation $t^{\frac{1}{2}}$ is used instead of $u$, which I find somewhat cumbersome.} 
Note that, as an $A$-module, we have
$$G^{p,q}_A\fpsu = \varprojlim G^{p,q}_A[u]/(u^m).$$

\begin{lemma}\label{power-series-prop}
 \note{power-series-prop}
 Let $A \in \mathbf{Art}_\Lambda$ or $A = \Lambda$. Also, let $B' = \Lambda$ or $B' \in \mathbf{Art}_\Lambda$, let $B \in \mathbf{Art}_\Lambda$, and let $B' \to B$ be a (local) homomorphism of $\Lambda$-algebras.
 \begin{enumerate}[label=\emph{(\roman*)}]
  \item We have isomorphisms
 $$G^{p,q}_{B'}\fpsh \otimes_{B'} B \xrightarrow{\cong} G^{p,q}_B\fpsh \quad \mathrm{and} \quad G^{p,q}_{B'}\fpsu \otimes_{B'} B \xrightarrow{\cong} G^{p,q}_B\fpsu.$$
 \item The $A$-modules $G^{p,q}_A\fpsh$ and $G_A^{p,q}\fpsu$ are flat. 
 \item The $\Lambda$-modules $G^{p,q}\fpsh$ and $G^{p,q}\fpsu$ are complete with respect to $\m_\Lambda$.
 \end{enumerate}
\end{lemma}

Before we embark into the proof, we need the following lemma.

\begin{lemma}
 Let $R$ be a (commutative and unital) ring. Let $(M_i)_{i \in \NN}$ be an inverse system of flat $R$-modules with surjective restriction maps. Let $N$ be a finitely presented $R$-module. Then we have an isomorphism
 $$\Phi_N: \ \ \left(\varprojlim M_i \right) \otimes_R N \to \varprojlim \left(M_i \otimes_R N \right).$$
\end{lemma}
\begin{proof}
 First note that the statement holds for $N = R^n$. If $N$ is finitely \emph{generated}, we have an exact sequence 
 $$0 \to K \to R^n \to N \to 0.$$
 The system $(M_i \otimes_R K)_i$ has surjective restriction maps, so, together with the flatness of $M_i$, the upper line in the following diagram is exact:
 \[
  \xymatrix{
   0 \ar[r] & \varprojlim (M_i \otimes_R K) \ar[r] & \varprojlim (M_i \otimes_R R^n) \ar[r] & \varprojlim (M_i \otimes_R N) \ar[r] & 0 \\
   & \left(\varprojlim M_i\right) \otimes_R K \ar[r] \ar[u]^{\Phi_K} & \left(\varprojlim M_i\right) \otimes_R R^n \ar[r] \ar[u]^{\cong} & \left(\varprojlim M_i\right) \otimes_R N \ar[r] \ar[u]^{\Phi_N} & 0 \\
  }
 \]
 The lower row is exact since taking the tensor product is right exact. From the diagram, we find that $\Phi_N$ is surjective. Since our $N$ is assumed to be finitely \emph{presented}, the kernel $K$ is actually finitely \emph{generated}, so $\Phi_K$ is surjective. Now a diagram chase shows that $\Phi_N$ is also injective.
\end{proof}

\begin{proof}[Proof of Lemma~\ref{power-series-prop}]
 It suffices to show the Lemma for $u$; then the proof for $\hb$ is the same.

 (i): Note that $B$ is finitely presented as an $B'$-module because it is actually a $\kk$-vector space of finite dimension. Note also that $G_{B'}^{p,q}[u]/(u^m)$ is a flat $B'$-module, and that $G_{B'}^{p,q}[u]/(u^m) \otimes_{B'} B \cong G_B^{p,q}[u]/(u^m)$.
 
 (ii): Note that $G_A^{p,q}[u]/(u^m)$ is a flat $A[u]/(u^m)$-module because a tensor product preserves flatness. Thus, by \cite[0912]{stacks}, the limit $G_A^{p,q}\fpsu$ is a flat $A[u]$-module, hence a flat $A$-module.
 
 (iii): We have 
 $$G^{p,q}\fpsu = \varprojlim_m \varprojlim_k G^{p,q}_k[u]/(u^m) =  \varprojlim_k \varprojlim_m G^{p,q}_k[u]/(u^m) = \varprojlim_k G^{p,q}_k\fpsu$$
 since inverse limits commute with each other.
\end{proof}

Next, we set, for either $A \in \mathbf{Art}_\Lambda$ or $A = \Lambda$,
$$G_A^{p,q}\fLsu := \{\sum_{i\, =\, -N}^\infty g_iu^i \ | \ N \geq 0, g_i \in G_A^{p,q}\}$$
and also
$$G_A^{p,q}\fLsh := \{\sum_{i\, =\, -N}^\infty g_i\hb^i \ | \ N \geq 0, g_i \in G_A^{p,q}\}.$$
We have
$$G^{p,q}_A\fLsu = G^{p,q}_A\fpsu _u,$$
i.e., it is the localization of the $A[u]$-module $G^{p,q}_A\fpsu$ in the non-zero divisor $u$. Thus, we also have
$$G_A^{p,q}\fLsu = G_A^{p,q}\fpsu \otimes_{A[u]}A[u,u^{-1}].$$
In particular, $G_A^{p,q}\fLsu$ is a flat $A[u,u^{-1}]$-module. Similarly, we have
$$G^{p,q}_A\fLsu = G^{p,q}_A\fpsu \otimes_{\kk[u]} \kk[u,u^{-1}].$$
Note also that
$$G_{B'}^{p,q}\fLsu \otimes_{B'} B = G_{B'}^{p,q}\fpsu \otimes_{B'[u]} B[u,u^{-1}] = (G_{B'}^{p,q}\fpsu \otimes_{B'} B) \otimes_{B[u]} B[u,u^{-1}] = G_B^{p,q}\fLsu$$
for $B' = \Lambda$ or $B' \in \mathbf{Art}_\Lambda$ and $B \in \mathbf{Art}_\Lambda$, along a map $B' \to B$.

\begin{rem}
 Unlike the case of $G^{p,q}\fpsu$, the natural map
 $$G^{p,q}\fLsu \to \varprojlim_k G_k^{p,q}\fLsu$$
 is \emph{not} an isomorphism. Namely, an element in the limit on the right might have components in arbitrarily negative degree while the degree of the non-zero components of an element on the left is bounded below. The map is, however, injective.
\end{rem}

\section{The quantum extended Maurer--Cartan equation}\index{Maurer--Cartan equation!quantum extended}

Instead of the original extended Maurer--Cartan equations, we study the \emph{quantum extended Maurer--Cartan equation}
$$(\bar\partial + \hb\Delta)\varphi_A + \frac{1}{2}[\varphi_A,\varphi_A] + (\ell + \hb y) = 0 \: \in \: G^1_A\fpsh$$
which is to be satisfied by elements
$$\varphi_A \: =\: \sum_{i = 0}^\infty C^{0,0}_i(\varphi_A)\hb^i + \sum_{i = 0}^\infty C^{-1,1}_i(\varphi_A)\hb^i + ... + \sum_{i = 0}^\infty C^{-d,d}_i(\varphi_A)\hb^i \: \in \: \m_A \cdot G^0_A\fpsh.$$
Here, we write 
$$G^m_A\fpsh := \bigoplus_{p + q = m}G^{p,q}_A\fpsh;$$
since $p$ is bounded below by $-d$, and $q$ is bounded below by $0$, this direct sum is actually finite. We denote the set of solutions by $\mathrm{QMC}(G^{\bullet,\bullet},A)$.

\begin{lemma}
 Assume that $\varphi_A \in \m_A \cdot G_A^0\fpsh$ satisfies the quantum extended Maurer--Cartan equation. Then $(\bar\partial + \hb \Delta + [\varphi_A,-])^2 = 0$ as an operator $G^m_A\fpsh \to G^{m + 1}_A\fpsh$.
\end{lemma}
\begin{proof}
 For $\theta \in G_A^m\fpsh$, we have
 $$(\bar\partial + \hb \Delta + [\varphi_A,-])^2(\theta) = [\bar\partial\varphi_A + \hb\Delta(\varphi_A) + \frac{1}{2}[\varphi_A,\varphi_A] + \ell + \hb y,\theta] = 0$$
 by direct computation.
\end{proof}
\begin{rem}
 In \cite{ChanLeungMa2023}, the authors claim that the two conditions would be equivalent. While this is the case for the classical extended Maurer--Cartan equation, I don't see why it should be true here. Namely, not for every $(p,q)$ it is true that, for $\theta \in G^{p,q}_A$, the condition $[\theta,\xi] = 0$ for all $\xi$ implies $\theta = 0$. For example, we have $[1,\xi] = 0$ for all $\xi$ (proven by using the odd Poisson identity) but $1 \not= 0$.
\end{rem}

We define a \emph{scaling operator}\index{scaling operator}
$$\sigma: \: G_A^{p,q}\fLsu \to G_A^{p,q}\fLsu, \quad \theta \mapsto u^{q - p - 2} \theta;$$
it is an automorphism. Since $\varphi_A \in \m_A \cdot G^0_A\fpsh$, the sum 
$$e^{\sigma(\varphi_A)} = \sum_{n = 0}^\infty \frac{\sigma(\varphi_A)^n}{n!} \: \in \: G^0_A\fLsu$$
is actually finite and hence well-defined. Recall that we have $\check d = \bar\partial + \Delta + (\ell + y) \wedge - $.

\note{q-ext-MC-var-1}
\begin{lemma}\label{q-ext-MC-var-1}
 An element $\varphi_A \in  \m_A \cdot G^0_A\fpsh$ satisfies the quantum extended Maurer--Cartan equation if and only if 
 $$\check d(e^{\sigma(\varphi_A)}) = 0.$$
\end{lemma}
\begin{proof}
 We write $\varphi = \varphi_A$ for short. Let $\varphi = \varphi^{0,0} + \varphi^{-1,1} + ... + \varphi^{-d,d}$ be the bidegree decomposition with $\varphi^{-i,i} \in G_k^{-i,i}\fpsh$. Then the quantum extended Maurer--Cartan equation is equivalent to the following system of equations:
 \begin{align}
  (0,1)&: \quad \bar\partial\varphi^{0,0} + \hb\Delta\varphi^{-1,1} + \frac{1}{2}[\varphi^{0,0},\varphi^{-1,1}] + \frac{1}{2}[\varphi^{-1,1},\varphi^{0,0}] + \hb y = 0 \\
  (-1,2)&: \quad \bar\partial\varphi^{-1,1} + \hb\Delta\varphi^{-2,2} + \frac{1}{2}[\varphi^{0,0},\varphi^{-2,2}] + \frac{1}{2}[\varphi^{-1,1},\varphi^{-1,1}] + \frac{1}{2}[\varphi^{-2,2},\varphi^{0,0}] + \ell = 0 \\
  (-j,j + 1)&: \quad \bar\partial\varphi^{-j,j} + \hb\Delta\varphi^{-j - 1,j + 1} + \sum_{i = 0}^{j + 1} \frac{1}{2}[\varphi^{-i,i},\varphi^{-(j + 1 - i),j + 1 - i}] = 0 
 \end{align}
 where we have one equation for every $2 \leq j \leq d$ (or only two in total if $d = 1$). By multiplication with powers of $\hb$, this system is equivalent to the system 
 \begin{align}
  (0,1)&: \quad \hb^{-1}\bar\partial\varphi^{0,0} + \Delta\varphi^{-1,1} + \frac{1}{2}\hb^{-1}[\varphi^{0,0},\varphi^{-1,1}] + \frac{1}{2}\hb^{-1}[\varphi^{-1,1},\varphi^{0,0}] + y = 0 \\
  (-1,2)&: \quad \bar\partial\varphi^{-1,1} + \hb\Delta\varphi^{-2,2} + \frac{1}{2}[\varphi^{0,0},\varphi^{-2,2}] + \frac{1}{2}[\varphi^{-1,1},\varphi^{-1,1}] + \frac{1}{2}[\varphi^{-2,2},\varphi^{0,0}] + \ell = 0 \\
  (-j,j + 1)&: \quad \hb^{j - 1}\bar\partial\varphi^{-j,j} + \hb^j\Delta\varphi^{-j - 1,j + 1} + \sum_{i = 0}^{j + 1} \frac{1}{2}\hb^{j - 1}[\varphi^{-i,i},\varphi^{-(j + 1 - i),j + 1 - i}] = 0. 
 \end{align}
 Since $\sigma(\varphi) = \hb^{-1}\varphi^{0,0} + \varphi^{-1,1} + \hb\varphi^{-2,2} + ...$, putting everything together again, this system is equivalent to 
 $$\bar\partial(\sigma(\varphi)) + \Delta(\sigma(\varphi)) + \frac{1}{2}[\sigma(\varphi),\sigma(\varphi)] + \ell + y = 0 \: \in \: G_A^1\fpsh.$$
 A direct computation yields
 $$\check d(e^{\sigma(\varphi)}) = \left(\bar\partial(\sigma(\varphi)) + \Delta(\sigma(\varphi)) + \frac{1}{2}[\sigma(\varphi),\sigma(\varphi)] + \ell + y\right) \wedge e^{\sigma(\varphi)}.$$
 Thus, if $\varphi$ satisfies the quantum extended Maurer--Cartan equation, then $\check d(e^{\sigma(\varphi)}) = 0$. Conversely, assume that $\check d(e^{\sigma(\varphi)}) = 0$. Let $\mu$ be the above term such that $\check d(e^{\sigma(\varphi)}) = \mu \wedge e^{\sigma(\varphi)} = 0$. Assume that $\mu \not= 0$. Let $m$ be the number such that $\mu \in \m_A^m \cdot G_A^0\fLsu$ but $\mu \notin \m_A^{m + 1} \cdot G_A^0\fLsu$. Now
 $$\mu \wedge e^{\sigma(\varphi)} = \mu + \mu \wedge \sigma(\varphi) + ... = 0.$$
 Since $\mu \wedge \sigma(\varphi) + ... \in \m_A^{m + 1}\cdot G_A^0\fLsu$, we find that $\mu \in \m_A^{m + 1} \cdot G_A^0\fLsu$, a contradiction. Thus, we have $\mu = 0$, and $\varphi$ satisfies the quantum extended Maurer--Cartan equation.
\end{proof}

\section{Solving the quantum extended Maurer--Cartan equation}

In this section, we show that the quantum extended Maurer--Cartan functor $\mathrm{QMC}(G^{\bullet,\bullet},-)$ is unobstructed if $G^{\bullet,\bullet}$ is quasi-perfect. In fact, an even weaker assumption is sufficient.

\begin{defn}\label{semi-perf-defn}\note{semi-perf-defn}\index{semi-perfect}\index{Batalin--Vilkovisky algebra!semi-perfect}
 Let $G^{\bullet,\bullet}$ be a strictly faithful $\Lambda$-linear Batalin--Vilkovisky algebra of dimension $d$. Then $G^{\bullet,\bullet}$ is \emph{semi-perfect} if:
 \begin{enumerate}[label=(\roman*)]
  \item the first spectral sequence $'\!E_{0;r}$ associated with the double complex $(G_0^{\bullet,\bullet},\Delta,\bar\partial)$ over $A_0 = \kk$ degenerates at $E_1$ in $[0,1]$, as defined in Definition~\ref{degeneration-level-k};
  \item for every $A \in \mathbf{Art}_\Lambda$, the induced map $H^k(G_A^\bullet,\check d) \to H^k(G_0^\bullet,\check d)$ is surjective for $k \in \{-1,0\}$.
 \end{enumerate}
\end{defn}

Clearly, every quasi-perfect $G^{\bullet,\bullet}$ is also semi-perfect. If the cohomology groups over $A_0 = \kk$ are finitely generated as $\kk$-vector spaces, in order to satisfy the first condition, it is sufficient to have that the dimensions of $'\!E_{0;1}^{p,q}$ sum up to the dimension of $H^k(G_0^\bullet,\check d)$ for either $k = 0$ or $k = 1$. The second condition for both $k = -1$ and $k = 0$ ensures that $H^0(G_A^\bullet,\check d)$ is a flat $A$-module.

\begin{thm}\label{QMC-unobstructed}\note{QMC-unobstructed}\index{unobstructedness theorem!quantum extended MC eqn.}
 Let $G^{\bullet,\bullet}$ be a $\Lambda$-linear curved Batalin--Vilkovisky algebra of dimension $d$. Assume that $G^{\bullet,\bullet}$ is semi-perfect.
 Then the quantum extended Maurer--Cartan functor 
 $$\mathrm{QMC}(G^{\bullet,\bullet},-): \mathbf{Art}_\Lambda \to \mathbf{Set}$$
 is unobstructed. In particular, if $G^{\bullet,\bullet}$ is quasi-perfect, then the functor is unobstructed.
\end{thm}
\begin{rem}
 This is the most fundamental unobstructedness theorem which we consider here. We have learned the proof from Chan--Leung--Ma in \cite[Thm.~5.6]{ChanLeungMa2023}, where a more general statement in a more restrictive context is given. Here, we give the least restrictive conditions possible for the argument to work. The proof in \cite[Thm.~5.6]{ChanLeungMa2023} in turn has been inspired by \cite{Barannikov2002,KKP2008,Terilla2008}, which all deal with the classical, not logarithmic, case. In particular, \cite{Terilla2008} contains a very general unobstructedness theorem, which is proven by studying $V\fpsh$ with a differential $Q + \hb\Delta$ instead of the original BV algebra $V$ with differential $Q$ and BV operator $\Delta$. However, Theorem~\ref{QMC-unobstructed} is not a mere consequence of this unobstructedness theorem due to the predifferential nature of our situation. 
\end{rem}

Let $\varphi_B \in \m_B \cdot G_B^0\fpsh$ be a solution of the quantum extended Maurer--Cartan equation over $B \in \mathbf{Art}_\Lambda$, and let $B' \to B$ be a small extension in $\mathbf{Art}_\Lambda$ with kernel $I \subset B'$. We have an exact sequence 
$$0 \to I \cdot G_{B'}^{p,q}\fpsh \to G_{B'}^{p,q}\fpsh \to G_{B}^{p,q}\fpsh \to 0$$
where we have additionally $I \cdot G_{B'}^{p,q}\fpsh = G_0^{p,q} \fpsh \otimes_\kk I$ when we consider $I$ as a $\kk$-vector space. 

Let us write $\varphi := \varphi_B$. In the rest of the section, we show the existence of a solution $\varphi' := \varphi_{B'}$ of the quantum extended Maurer--Cartan equation over $B'$ with $\varphi'|_B = \varphi_B = \varphi$, where $(-)|_B$ is the restriction map in the above exact sequence.

As the first step, we show another variant of Lemma~\ref{q-ext-MC-var-1}.

\begin{lemma}
 Let $\varphi_A \in \m_A \cdot G_A^0\fpsh$. Then $\varphi_A$ is a solution of the quantum extended Maurer--Cartan equation if and only if 
 $$\check d_{\hb}(e^{\varphi_A/\hb}) = 0$$
 in $G_A^\bullet\fLsh$ where $\check d_{\hb} := \bar\partial + \hb\Delta + \hb^{-1}(\ell + \hb y) \wedge (-)$. Moreover, this map satisfies $\check d_{\hb}^2 = 0$.
\end{lemma}
\begin{proof}
 The element $e^{\varphi_A/\hb} \in G_A^0\fLsh$ is well-defined because $\m_A$ is nilpotent, so exponents of $\hb$ are bounded below when computing the power series.
 
 By direct computation, we have $\check d_{\hb}(\theta) = u \cdot \sigma^{-1} \circ \check d \circ \sigma(\theta)$ in $G_A^\bullet\fLsu$, so $\check d_{\hb}(e^{\varphi_A/\hb}) = 0$ if and only if $\check d \circ \sigma(e^{\varphi_A/\hb}) = 0$.
 
 Since $\sigma(\theta \wedge \xi) = \hb \cdot \sigma(\theta) \wedge \sigma(\xi)$, we find $\sigma((\varphi_A/\hb)^n) = \sigma(\varphi_A)^n \cdot \hb^{-1}$ and thus 
 $$\sigma(e^{\varphi_A/\hb}) = \hb^{-1}\cdot e^{\sigma(\varphi_A)}.$$
 Consequently, $\check d_{\hb}(e^{\varphi_A/\hb}) = 0$ is equivalent to $\check d(e^{\sigma(\varphi_A)}) = 0$, the statement of Lemma~\ref{q-ext-MC-var-1}.
\end{proof}

Since we assume $\varphi$ to be a solution of the quantum extended Maurer--Cartan equation, we obtain $\check d_{\hb}(e^{\varphi/\hb}) = 0$. By a direct computation, we have
$$\check d_{\hb}(e^{\varphi/\hb}) = \left(\bar\partial(\varphi/\hb) + \hb\Delta(\varphi/\hb) + \frac{1}{2}\hb[\varphi/\hb,\varphi/\hb] + \hb^{-1}(\ell + \hb y)\right) \wedge e^{\varphi/\hb}.$$
Let $\mu$ be the term in the left bracket on the right hand side. As in the proof of Lemma~\ref{q-ext-MC-var-1}, we have $\mu = 0$.

Now we choose an arbitrary lift $\tilde\varphi$ of $\varphi$ in $\m_{B'} \cdot G_{B'}^0\fpsh$, i.e., $\tilde\varphi|_B = \varphi$. Let 
$$\tilde\mu := \bar\partial(\tilde\varphi/\hb) + \hb\Delta(\tilde\varphi/\hb) + \frac{1}{2}\hb[\tilde\varphi/\hb,\tilde\varphi/\hb] + \hb^{-1}(\ell + \hb y);$$
then $\tilde\mu \in \hb^{-1}\cdot\m_{B'}\cdot G_{B'}^1\fpsh$, i.e., the lowest $\hb$-term is at $\hb^{-1}$. Since $\tilde\mu|_B = \mu = 0$, we have $\tilde\mu \in I \cdot G_{B'}^1\fLsh$. Because $I \cdot \m_{B'} = 0$, we have moreover
$$\check d_{\hb}(e^{\tilde\varphi/\hb}) = \tilde\mu \wedge e^{\tilde\varphi/\hb} = \tilde\mu.$$

\begin{lemma}
 Let $G^{\bullet,\bullet}$ be as above. Then we have an exact sequence 
 $$0 \to (G_0^\bullet\fLsh,\check d_{\hb}) \otimes_\kk I \to (G_{B'}^\bullet\fLsh,\check d_{\hb}) \to (G_B^\bullet\fLsh,\check d_{\hb}) \to 0$$
 of complexes; it induces an exact  sequence
 $$0 \to H^1(G_0^\bullet\fLsh,\check d_{\hb}) \otimes_\kk I \to H^1(G_{B'}^\bullet\fLsh,\check d_{\hb}) \to H^1(G_B^\bullet\fLsh,\check d_{\hb})$$
 in cohomology.\footnote{Here we use the second condition of Theorem~\ref{QMC-unobstructed}.}
\end{lemma}
\begin{proof}
 The first sequence is exact on the level of graded pieces, and the horizontal maps are compatible with the differentials, so the first sequence of complexes is exact.
 
 For $A \in \mathbf{Art}_\Lambda$, we have a series of functorial isomorphisms. Denote by $SG_A^\bullet\fLsh$ the complex with graded pieces $SG_A^m\fLsh = u^mG_A^m\fLsh \subseteq G_A^m\fLsu$. Then we have an isomorphism
 $$\sigma: (G_A^\bullet\fLsh,\check d_{\hb}) \xrightarrow{\cong} (SG_A^\bullet\fLsh,u\check d)$$
 of complexes since $\check d_{\hb} = u \cdot \sigma^{-1} \circ \check d \circ \sigma$. Next, multiplication with $u^m$ in degree $m$ gives an isomorphism
 $$u^\bullet: (G_A^\bullet\fLsh,\check d) \xrightarrow{\cong} (SG_A^\bullet\fLsh,u\check d)$$
 of complexes. They induce an isomorphism $H^k(G_A^\bullet\fLsh,\check d_{\hb}) \cong H^k(G_A^\bullet\fLsh,\check d)$ of $A$-modules, functorial in $A \in \mathbf{Art}_\Lambda$. Further, we have a functorial isomorphism
 $$H^k(G_A^\bullet\fLsh,\check d) \xrightarrow{\cong} H^k(G_A^\bullet,\check d) \fLsh.$$
 Since $H^k(G_{B}^\bullet,\check d) \to H^k(G_0^\bullet,\check d)$ is surjective for both $k = -1$ and $k = 0$, we find that $H^0(G_B^\bullet,\check d)$ is flat over $B$ by Proposition~\ref{Artin-complex-base-change}. Then the proof of Corollary~\ref{spectral-sequence-Artin-ring-base-change} shows that $H^0(G_{B'}^\bullet,\check d) \to H^0(G_{B}^\bullet,\check d)$ is surjective so that
 $$H^0(G_{B'}^\bullet\fLsh,\check d_{\hb}) \to H^0(G_B^\bullet\fLsh,\check d_{\hb})$$
 is surjective as well. This proves the second part of the assertion.
\end{proof}

Since $\check d_{\hb}(e^{\tilde\varphi/\hb}) = \tilde\mu$, we have $\check d_{\hb}(\tilde\mu) = 0$. Thus, we can find an element $a \in G_0^1\fLsh \otimes_\kk I$ with $a \mapsto \tilde\mu$ and $\check d_{\hb}(a) = 0$; in particular, we have $[a] \mapsto [\tilde\mu] = 0$ on cohomology classes of $\check d_{\hb}$. Since the map is injective on cohomology classes by the Lemma, we have $[a] = 0$; thus, there is $b \in G_0^0\fLsh \otimes_\kk I$ with $\check d_{\hb}(b) = a$. Expanding the definition of $\check d_{\hb}$, we find $(\bar\partial + \hb\Delta)(b) = a$. Note that in fact $a \in \hb^{-1}G_0^1\fpsh \otimes_\kk I$, i.e., $a$ has no $\hb$-terms below $\hb^{-1}$, because this is true for $\tilde\mu$. We want to choose $b$ in such a way that $b \in \hb^{-1}G_0^0\fpsh \otimes_\kk I$ as well. 

\begin{lemma}\label{QMC-unob-lemma}\note{QMC-unob-lemma}
 In the above situation, let $\alpha \in G_0^1\fpsh \subset G_0^1\fLsh$, and assume that we have $\beta \in G_0^0\fLsh$ with $(\bar\partial + \hb\Delta)(\beta) = \alpha$. Then we can find $\gamma \in G_0^0\fpsh$ with $(\bar\partial + \hb\Delta)(\gamma) = \alpha$.\footnote{Here we use the first assumption of Theorem~\ref{QMC-unobstructed}.}
\end{lemma}
\begin{proof}
 We have $(G_0^\bullet\fLsh,\bar\partial + \hb\Delta) = (G_0^\bullet\fpsh,\bar\partial + \hb\Delta) \otimes_{\kk\fpsh} \kk\fLsh$. Since $\kk\fpsh \to \kk\fpsh_{\hb} = \kk\fLsh$ is flat, we find $$H^k(G_0^\bullet\fLsh,\bar\partial + \hb\Delta) = H^k(G_0^\bullet\fpsh,\bar\partial + \hb\Delta) \otimes_{\kk\fpsh} \kk\fLsh.$$ Since $G_0^\bullet\fpsh \to G_0^\bullet\fLsh$ is injective, $\alpha$ defines a class $[\alpha] \in H^1(G_0^\bullet\fpsh,\bar\partial + \hb\Delta)$; by assumption, its image in $H^1(G_0^\bullet\fLsh,\bar\partial + \hb\Delta)$ is $0$. The kernel of 
 $$H^1(G_0^\bullet\fpsh,\bar\partial + \hb\Delta) \to H^1(G_0^\bullet\fpsh,\bar\partial + \hb\Delta) \otimes_{\kk\fpsh} \kk\fLsh$$
 consists of those classes $\delta$ on the left with $\hb^m\delta = 0$ for some $m > 0$. By assumption, the first spectral sequence degenerates at $E_1$ in $[0,1]$ so that $H^1(G_0^\bullet\fpsh,\bar\partial + \hb\Delta)$ is flat, hence torsion-free, by Proposition~\ref{spec-seq-alternative-formulation}. Thus $[\alpha] = 0$ as a cohomology class in $H^1(G_0^\bullet\fpsh,\bar\partial + \hb\Delta)$, and the assertion follows.
\end{proof}

The lemma holds as well for $G_0^\bullet\fpsh \otimes_\kk I$ because $I$ is a $\kk$-vector space of dimension $1$. Replacing $a$ with $\hb a$ for a moment, we can find $c \in G_0^0\fpsh \otimes_\kk I$ with $(\bar\partial + \hb\Delta)(c) = \hb a$; then $b = \hb^{-1}c$ is our desired $b$.
With this $b$, we set $\varphi' := \tilde\varphi - \hb b \in \m_{B'}\cdot G_{B'}^0\fpsh$. Then we have $\varphi'|_B = \varphi$ and 
 $$\check d_{\hb}(e^{\varphi'/\hb}) = \check d_{\hb}(e^{\tilde\varphi/\hb - b}) = \check d_{\hb}(1 + (\tilde\varphi/\hb - b) + ...) = \check d_{\hb}(e^{\tilde\varphi/\hb} - b) = \tilde\mu - a = 0$$
 since $b \wedge \tilde\varphi = 0$ and $b \wedge b = 0$. Thus $\varphi'$ is a solution of the quantum extended Maurer--Cartan equation, which is a lift of $\varphi = \varphi_B$, concluding the proof of Theorem~\ref{QMC-unobstructed}.

\section{The first abstract unobstructedness theorem}

In this section, we show that the semi-classical Maurer--Cartan functor $\mathrm{SMC}(G^{\bullet,\bullet},-)$ is unobstructed if $G^{\bullet,\bullet}$ is quasi-perfect or, more generally, semi-perfect. To do so, we introduce yet another variant $\mathrm{QMC}_0(G^{\bullet,\bullet},-)$.

\begin{defn}
 Let $G^{\bullet,\bullet}$ be a $\Lambda$-linear curved Batalin--Vilkovisky algebra. A solution 
 $$\varphi_A = \sum_{i = 0}^\infty \left(\sum_{p = 0}^d C_i^{-p,p}(\varphi_A)\right)\hb^i \in \m_A \cdot G_A^0\fpsh$$
 with $C_i^{-p,p}(\varphi_A) \in \m_A\cdot G_A^{-p,p}$ of the quantum extended Maurer--Cartan equation 
 $$(\bar\partial + \hb\Delta)\varphi_A + \frac{1}{2}[\varphi_A,\varphi_A] + (\ell + \hb y) = 0$$
 is called \emph{obtuse} if $C_0^{0,0}(\varphi_A) = 0$. We denote the set of obtuse solutions of the quantum extended Maurer--Cartan equation by $\mathrm{QMC}_0(G^{\bullet,\bullet},A)$.
\end{defn}
\begin{rem}
 The name ``obtuse'', usually used to refer to an angle of more than $90^\circ$, refers to the missing apex in $(0,0)$ when we plot the non-zero coefficients $C_i^{-p,p}$ at $(i,p)$ into the $\NN^2$-plane.
\end{rem}

\note{QMC0-unobstructed}
\begin{prop}[cf.~5.12 in \cite{ChanLeungMa2023}]\label{QMC0-unobstructed}
 Let $G^{\bullet,\bullet}$ be a semi-perfect $\Lambda$-linear curved Batalin--Vilkovisky algebra.  Then the functor $\mathrm{QMC}_0(G^{\bullet,\bullet},-)$ is unobstructed.
\end{prop}
\begin{proof}
 Let $B' \to B$ be a small extension in $\mathbf{Art}_\Lambda$ with kernel $I \subset B'$. Let $\varphi \in \mathrm{QMC}_0(G^{\bullet,\bullet},B)$ be an obtuse solution of the quantum extended Maurer--Cartan equation, and let $\varphi'$ be a (possibly not obtuse) lift to $B'$, which exists by Theorem~\ref{QMC-unobstructed}. By assumption, we have
 $$(\bar\partial + \hb\Delta)\varphi' + \frac{1}{2}[\varphi',\varphi'] + (\ell + \hb y) = 0;$$
 considering only the coefficient of $\hb^0$, we then find
 $$\bar\partial\left(\sum_{p = 0}^dC_0^{-p,p}(\varphi')\right) + \frac{1}{2}\left[\sum_{p = 0}^dC_0^{-p,p}(\varphi'),\sum_{p = 0}^dC_0^{-p,p}(\varphi')\right] + \ell = 0 \quad \in \: G_{B'}^1 = \bigoplus_{p = 0}^d G_{B'}^{-p,p + 1}.$$
 Considering the pieces in the various bidegrees separately, we find, for $G_{B'}^{0,0}$, that
 $$\bar\partial C_0^{0,0}(\varphi') + [C_0^{0,0}(\varphi'),C_0^{-1,1}(\varphi')] = 0.$$
 Since $\varphi$ is obtuse, we have $C_0^{0,0}(\varphi') \in I \cdot G_{B'}^{0,0}$; because $I \cdot \m_{B'} = 0$, we find $\bar\partial C_0^{0,0}(\varphi') = 0$. We set $\hat\varphi' := \varphi' - C_0^{0,0}(\varphi')$. Again using $I \cdot \m_{B'} = 0$, we find that $\hat\varphi'$ is a solution of the quantum extended Maurer--Cartan equation; it is obtuse by construction.
\end{proof}

With this preparation, it is now easy to prove the first abstract unobstructedness theorem.

\begin{thm}[First abstract unobstructedness theorem]\index{unobstructedness theorem!first abstract}\index{Maurer--Cartan equation!semi-classical extended}
 Let $G^{\bullet,\bullet}$ be a quasi-perfect or just semi-perfect $\Lambda$-linear curved Batalin--Vilkovisky algebra. Then the semi-classical Maurer--Cartan functor $\mathrm{SMC}(G^{\bullet,\bullet},-)$ is unobstructed.
\end{thm}
\begin{proof}
 Let $B' \to B$ be a small extension in $\mathbf{Art}_\Lambda$ with kernel $I \subset B'$. Let $(\phi,f)$ be a semi-classical Maurer--Cartan solution over $B$. Then $\varphi := \phi + \hb f$ is an obtuse solution of the quantum extended Maurer--Cartan equation. By Proposition~\ref{QMC0-unobstructed}, we find an obtuse lift $\varphi'$. We set $\phi' := C_0^{-1,1}(\varphi')$ and $f' := C_1^{0,0}(\varphi')$. Then the $(-1,2)$-part of the $\hb^0$-piece of the quantum extended Maurer--Cartan equation for $\varphi'$ yields 
 $$\bar\partial\phi' + \frac{1}{2}[\phi',\phi'] + \ell = 0,$$
 and the $(0,1)$-part of the $\hb^1$-piece yields
 $$\bar\partial f' + [\phi',f'] + y + \Delta\phi' = 0.$$
 Thus, $(\phi',f')$ is a semi-classical Maurer--Cartan solution lifting $(\phi,f)$ to $B'$.
\end{proof}

\section{The second abstract unobstructedness theorem}\index{unobstructedness theorem!second abstract}\index{Maurer--Cartan equation!classical extended}

We deduce the second abstract unobstructedness Theorem~\ref{second-abstract-unob-thm} from the first one. In fact, it is sufficient to show that the forgetful map
$$\mathrm{SMC}(G^{\bullet,\bullet},-) \to \mathrm{MC}(G^{\bullet,\bullet},-)$$
is a smooth morphism of functors of Artin rings. Namely, if it is smooth, then it is also surjective, i.e., surjective on every $A \in \mathbf{Art}_\Lambda$; then smoothness of $\mathrm{MC}(G^{\bullet,\bullet},-)$ follows from smoothness of $\mathrm{SMC}(G^{\bullet,\bullet},-)$ by \cite[2.5(iii)]{Schlessinger1968}. For the second abstract unobstructedness theorem, semi-perfectness alone is not sufficient, and we assume quasi-perfectness for simplicity.

\note{SMC-MC-smooth}
\begin{thm}\label{SMC-MC-smooth}
 Let $(G^{\bullet,\bullet},A^{\bullet,\bullet})$ be a quasi-perfect $\Lambda$-linear curved Batalin--Vilkovisky calculus. Then 
 $$\mathrm{SMC}(G^{\bullet,\bullet},-) \to \mathrm{MC}(G^{\bullet,\bullet},-)$$
 is smooth.
\end{thm}
\begin{proof}
 Let $B' \to B$ be a small extension in $\mathbf{Art}_\Lambda$ with kernel $I \subset B'$. Let $(\phi,f)$ be a semi-classical Maurer--Cartan solution over $B$, and let $\phi'$ be a classical Maurer--Cartan solution over $B'$ with $\phi'|_B = \phi$. We have to show that it can be extended to a semi-classical Maurer--Cartan solution $(\phi',f')$ over $B'$ with $f'|_B = f$.
 
 We write $\omega_f = e^f \ \invneg\ \omega$; then we have $\bar\partial_\phi\omega_f = 0$ by Lemma~\ref{MC-sol-modified-BV-calculus}. Thus, it defines a class $[\omega_f] \in H^0(A_B^{0,\bullet},\bar\partial_\phi)$. By Proposition~\ref{free-A-prop}, the formation of $H^0(A_{B'}^{0,\bullet},\bar\partial_{\phi'})$ commutes with base change; in particular, $H^0(A_{B'}^{0,\bullet},\bar\partial_{\phi'}) \to H^0(A_B^{0,\bullet},\bar\partial_\phi)$ is surjective, and we find $\omega' \in A_{B'}^{0,0}$ representing a class $[\omega']$---i.e., $\bar\partial_{\phi'}\omega' = 0$---with $[\omega']|_B = [\omega_f]$. Since
 $$H^0(A_{B'}^{0,\bullet},\bar\partial_{\phi'}) = \mathrm{ker}(\bar\partial_{\phi'}: A_{B'}^{0,0} \to A_{B'}^{0,1})$$
 and analogously for $B$, we have in fact $\omega'|_B = \omega_f$.
 
 Let $\tilde f \in \m_{B'} \cdot G_{B'}^{0,0}$ be an arbitrary lift of $f$, and set $\omega_{\tilde f} := e^{\tilde f} \ \invneg\ \omega$. Then $\omega_{\tilde f}|_B = \omega_f$. Let $g \in G_{B'}^{0,0}$ be the unique element with $g \ \invneg\ \omega = \omega'$, and let $h := g - e^{\tilde f}$. Then we have 
 $$h \ \invneg\ \omega = \omega' - \omega_{\tilde f} \in I \cdot A_{B'}^{d,0},$$
 so $h \in I \cdot G_{B'}^{0,0}$. Set $f' := \tilde f + h$. Then we have $f'|_B = f$ and
 $$e^{f'} \ \invneg\ \omega = e^{\tilde f + h} \ \invneg\ \omega = (e^{\tilde f} + h) \ \invneg \ \omega = \omega'$$
 since $I \cdot \m_{B'} = 0$, so $e^{\tilde f + h} = e^{\tilde f} + h$.
 
 We have $e^{f'} \wedge e^{-f'} = 1$, so multiplication with $e^{f'}$ is an isomorphism of $B'$-modules. Consequently, the contraction map
 $$\kappa' : G_{B'}^{p,q} \to A_{B'}^{d + p,q}, \quad \theta \mapsto \theta \ \invneg\ \omega' = (\theta \wedge e^{f'}) \ \invneg \ \omega,$$
 is an isomorphism of $B'$-modules.
 
 Although we do not yet know if $(\phi',f')$ is a semi-classical Maurer--Cartan solution, we define $\Delta_{f'}(\theta) := \Delta(\theta) + [f',\theta]$ as in that case. Then, using $e^{f'} \wedge [f',\theta] = [e^{f'},\theta]$ and the Bogomolov--Tian--Todorov formula for $\Delta(e^{f'} \wedge \theta)$, we find 
 $$\partial(\theta \ \invneg\ \omega') = \Delta_{f'}(\theta) \ \invneg \ \omega'.$$
 Now we have
 \begin{align}
  (\bar\partial f' + [\phi',f'] + y + \Delta\phi') \ \invneg \ \omega' &= \bar\partial f' \ \invneg\ \omega' + \Delta_{f'}(\phi') \ \invneg\ \omega' + y \ \invneg\ \omega' \nonumber \\
  &= (\bar\partial e^{f'}) \ \invneg\ \omega + \partial(\phi' \ \invneg\ \omega') + e^{f'} \ \invneg\ \bar\partial\omega \nonumber \\
  &= \bar\partial(e^{f'} \ \invneg\ \omega) + \cL_{\phi'}(\omega') = \bar\partial_{\phi'}(\omega') = 0 \nonumber
 \end{align}
 where we have $y \ \invneg\ \omega' = e^{f'} \ \invneg\ \bar\partial\omega$ due to $y \ \invneg\ \omega = \bar\partial\omega$, and $\cL_{\phi'}(\omega') = \partial(\phi' \ \invneg\ \omega')$ by the Lie--Rinehart homotopy formula. Since contraction with $\omega'$ is an isomorphism, $(\phi',f')$ is a semi-classical Maurer--Cartan solution.
\end{proof}


\part{Geometry}


\chapter{Generically log smooth families}\label{gen-log-smooth-sec}\note{gen-log-smooth-sec}

In this chapter,\index{generically log smooth family} we review the basic theory of generically log smooth families and sketch a deformation theory thereof. Additionally, we introduce the notion of an \emph{enhanced generically log smooth family}, which has better base change properties. A \emph{generically log smooth family}, as defined in \cite{FFR2021}, is---on the level of schemes---a separated flat family $f: X \to S$ of finite type of Noetherian schemes with reduced Cohen--Macaulay fibers of some fixed dimension $d \geq 1$, the \emph{relative dimension} of $f: X \to S$. It is endowed with an open subset $j: U \subseteq X$ on which we have the structure of a log smooth and saturated log morphism $f: U \to S$ of fs log schemes. The complement $Z := X \setminus U$, on which we have defined no log structure, is required to satisfy the \emph{codimension condition}\index{codimension condition}
\begin{equation}\label{CC}
 \mathrm{codim}(Z_s,X_s) \geq 2 \tag{CC}
\end{equation}
for the fibers over every point $s \in S$. If $g: Y \to T$ is a second generically log smooth family with log smooth locus $V$, then a \emph{morphism} of generically log smooth families consists of a log morphism $b: T \to S$ and a morphism $c: Y \to X$ with $b \circ g = f \circ c$ and $V \subseteq c^{-1}(U)$. On $V \to U$, the map $c$ must be also a morphism of log schemes, compatible with the log part of $b: T \to S$. If $f: X \to S$ is a generically log smooth family and $b: T \to S$ a log morphism, then the fiber product is defined in the obvious way. This construction of the fiber product works well because we insist that $f: U \to S$ is saturated.

\section{Generically log smooth deformations}

As a general framework to study deformations of generically log smooth families, we introduce the following definition.

\begin{defn}\index{generically log smooth deformation}
 Let $S_0 = \Spec (Q \to \kk)$ be a log point, and let $f_0: (X_0, U_0) \to S_0$ be a generically log smooth family. Let $A \in \mathbf{Art}_Q$. Then a \emph{generically log smooth deformation} over $S_A$ is a generically log smooth family $f_A: (X_A,U_A) \to S_A$ together with a morphism $i: (X_0,U_0) \to (X_A,U_A)$ of generically log smooth families which induces an isomorphism $f_A \times_{S_A} S_0 \cong f_0$. In particular, $i^{-1}(U_A) = U_0$. 
 
 An \emph{isomorphism} of generically log smooth deformations over the same base $S_A$ is an isomorphism of generically log smooth families over $S_A$ which is compatible with the maps from $X_0$. If $S_B \to S_{B'}$ is a morphism, then a morphism of generically log smooth deformations over $S_B \to S_{B'}$ is defined similarly (inducing an isomorphism after base change along $S_B \to S_{B'}$).
\end{defn}

Isomorphism classes of generically log smooth deformations form a deformation functor.

\begin{lemma}
 The functor 
 $$\mathrm{LD}_{X_0/S_0}^{gen}: \mathbf{Art}_Q \to \mathbf{Set}.$$
 of isomorphism classes of generically log smooth deformations is a deformation functor, i.e., it satisfies $(H_0)$, $(H_1)$, and $(H_2)$.
\end{lemma}
\begin{proof}
 This is completely analogous to the proof that the log smooth deformation functor $\mathrm{LD}$ is a deformation functor, see \cite[Thm.~8.7]{Kato1996}.
\end{proof}

To give an impression of this functor, let us provide the following fact.

\begin{lemma}
 For every $A \in \mathbf{Art}_Q$, we have a Cartesian square
\[
 \xymatrix{
  \mathrm{LD}_{X_0/S_0}^{gen}(A) \ar[r] \ar[d] & \mathrm{LD}_{U_0/S_0}(A) \ar[d] \\
  \mathrm{Def}_{X_0}(A) \ar[r] & \mathrm{Def}_{U_0}(A) \\
 }
\]
of sets with the log smooth deformation functor $\mathrm{LD}_{U_0/S_0}$ and the flat deformation functor $\mathrm{Def}_{X_0}$.
\end{lemma}
\begin{proof}
 Let $P(A)$ be the fiber product; its elements consist of a flat deformation $X_A \to S_A$ and a log smooth deformation $U_A \to S_A$ such that there is an isomorphism $\cO_{X_A}|_{U_0} \cong \cO_{U_A}$. By choosing one such isomorphism, we see that the canonical map $\mathrm{LD}^{gen}_{X_0/S_0}(A) \to P(A)$ is surjective. Note in particular that $X_A \to S_A$ is Cohen--Macaulay. For injectivity, note that the canonical map 
 $$\rho: \mathrm{LD}_{X_0/S_0}^{gen}(A) \to \mathrm{LD}_{U_0/S_0}(A)$$
 is injective because we can extend an isomorphism of two structure sheaves from $U_0$ to $X_0$. 
\end{proof}

\section{The reflexive de Rham complex and polyvector fields}\label{W-sec}\note{W-sec}
\index{reflexive de Rham complex}

Let $f: X \to S$ be a generically log smooth family with log smooth locus $U$. Then we have the log de Rham complex $\Omega^\bullet_{U/S}$. As in \cite{FFR2021}, we define the \emph{reflexive de Rham complex} by 
$$\W^\bullet_{X/S} := j_*\Omega^\bullet_{U/S}$$
where $j: U \subseteq X$ is the open inclusion.\footnote{The log structure can always be extended to $X$ by taking $\M_X := j_*\M_U$. The reason why Gross and Siebert started to work with $\W^1_{X/S}$ instead of $\Omega^1_{X/S}$, the functorially defined sheaf of log differential forms, in the first place is that $\Omega^1_{X/S}$ is not quasi-coherent already in the simplest examples, see \cite[1.11]{GrossSiebertII}. Also, it is unclear if this log structure commutes with base change.} Since $X$ is Cohen--Macaulay and $Z = X \setminus U$ has codimension $\geq 2$, the pieces $\W^i_{X/S}$ are reflexive coherent sheaves. If $d$ is the relative dimension of $f: X \to S$, then 
$$\omega_{X/S} := \W^d_{X/S}$$
is the reflexive log canonical sheaf. In general, it is only a reflexive sheaf of rank one, but in many applications, it will be a line bundle, even if the other pieces are not locally free; in this case, we say that $f: X \to S$ is \emph{log Gorenstein}.\index{log Gorenstein} If $\omega_{X/S} \cong \cO_X$, then we say $f: X \to S$ is \emph{log Calabi--Yau}.\index{log Calabi--Yau} Similarly to $\Omega^1_{U/S}$, we have the \emph{log tangent sheaf} $\Theta^1_{U/S} = \cH om(\Omega^1_{U/S},\cO_U)$, which is locally free on $U$ and classifies relative log derivations $(D,\Delta)$ with values in $\cO_X$. Since, in general, the sheaf of log derivations is naturally reflexive, we also write 
$$\Theta^1_{X/S} := j_*\Theta^1_{U/S};$$
we consider this as a sheaf of Lie algebras with the natural Lie bracket. Its sections $(D,\Delta)$ on an open subset $W \subseteq X$ act as derivations $D: \cO_W \to \cO_W$ with log part $\Delta: \M_{U \cap W} \to \cO_{U \cap W}$.
We denote the exterior powers of $\Theta^1_{U/S}$ by $\Theta^p_{U/S} := \Lambda^p \Theta^1_{U/S}$. Then we define 
$$\V^{-1}_{X/S} := j_*\Theta^1_{U/S} \quad \mathrm{and} \quad \V^{-p}_{X/S} := j_*\Theta^p_{U/S}$$
as well as $\V^0_{X/S} := \cO_X$. We consider $\V^\bullet_{X/S}$ as a Gerstenhaber algebra endowed with the \emph{negative} of the usual Schouten--Nijenhuis bracket, see also Proposition~\ref{G-A-construction} below. In particular, the difference between $\V^{-1}_{X/S}$ and $\Theta^1_{X/S}$ is the sign of the Lie bracket as well as the sign of the action on $\cO_X$.  We call $\V^\bullet_{X/S}$ the \emph{reflexive Gerstenhaber algebra of log polyvector fields}. Each $\V^p_{X/S}$ is a reflexive coherent sheaf.\footnote{We put $\V^\bullet_{X/S}$ in negative degrees since this convention is more natural from the perspective of Gerstenhaber calculi, in particular the two contraction maps $\vdash$ and $\invneg$, consistent with our definition of a Gerstenhaber calculus. The notation $\W^\bullet_{X/S}$ originally grew out of denoting $\Omega$ as $W$ in emails, for its vague similarity of shapes. The notation $\V^\bullet_{X/S}$ is intended to suggest \emph{vector fields}. The symbol $\Theta^\bullet$ is reserved for actual log morphisms, and the symbol $\T$ is used for the sheaves computed from the classical, non-logarithmic cotangent complex.}

\subsubsection*{The base change property}

Let $f: X \to S$ be a generically log smooth family, and let $b: T \to S$ be a strict morphism. Let $g: Y \to T$ be the generically log smooth family obtained by base change, and let $c: Y \to X$ be the induced map. On $V = c^{-1}(U)$, we have an isomorphism $c^*\Omega^i_{U/S} \cong \Omega^i_{V/T}$, which can be extended to a map $c^*\W^i_{X/S} \to \W^i_{Y/T}$ since the target is reflexive. In general, this map is not an isomorphism, so we have attached a name to the situation where it is in \cite{FFR2021}:
\begin{defn}\label{bc-prop}\note{bc-prop}\index{base change property}
 Let $f: X \to S$ be a generically log smooth family. Then $f: X \to S$ has the \emph{base change property} if, for every strict morphism $b: T \to S$ of Noetherian fs log schemes with base change $g: Y \to T$ and map $c: Y \to X$, the map $c^*\W^i_{X/S} \to \W^i_{Y/T}$ is an isomorphism.
\end{defn}
\begin{rem}
 If $f: X \to S$ is log Gorenstein, then we have isomorphisms $c^*\V^p_{X/S} \cong \V^p_{Y/T}$ as well. For a more comprehensive discussion of this fact and an interpretation of the map, see \cite[\S 6.3]{FeltenThesis} in the author's thesis.
\end{rem}

\begin{ex}
 By \cite[Thm.~8.2]{FFR2021}, a generically log smooth deformation which is at the same time a log toroidal family has the base change property. This is sufficient for one of our main applications---namely unisingular deformations of log toroidal families of elementary Gross--Siebert type.
\end{ex}

For generically log smooth deformations $f_A: X_A \to S_A$ of $f_0: X_0 \to S_0$, this property has a particularly simple behavior. 

\begin{lemma}\label{base-change-prop-criterion}\note{base-change-prop-criterion}
 Let $f_A: X_A \to S_A$ be a generically log smooth deformation of $f_0: X_0 \to S_0$. Let $i: X_0 \to X_A$ be the inclusion. Then $f_A$ has the base change property if and only if $i^*\W^i_{X_A/S_A} \cong \W^i_{X_0/S_0}$ for all $0 \leq i \leq d$; this is the case if and only if $\W^i_{X_A/S_A}$ is flat over $S_A$ for all $0 \leq i \leq d$.
\end{lemma}
\begin{proof}
 The equivalence of the latter two statements follows directly from \cite[Thm.~0.4]{Wahl1976}. If $f_A$ has the base change property, we obviously have the isomorphism for $S_0 \to S_A$. For the converse, it is sufficient to assume that $X_A = \Spec S$ is affine, and that $T = \Spec R$ is affine as well. If $R = \ell$ is a field, then $A \to \ell$ factors through a flat map $\kk \to \ell$, so the isomorphism holds. Then, in the general case, we already know that we have an isomorphism $H^0(Y_t,(c^*\W^i_{X_A/S_A})_t) \to H^0(V_t,(c^*\W^i_{X_A/S_A})_t)$ in the fibers of $Y \to T$, so the claim follows from Lemma~\ref{bijective-in-fibers}.
\end{proof}

This allows us to give a characterization of the base change property in general.

\begin{cor}\label{general-base-change-char}\note{general-base-change-char}
 Let $f: X \to S$ be a generically log smooth family of relative dimension $d$. Then $f$ has the base change property if and only if the following two statements hold for each $0 \leq i \leq d$:
 \begin{enumerate}[label=\emph{(\roman*)}]
  \item $\W^i_{X/S}$ is flat over $S$;
  \item for every point $s \in S$, the map $\W^i_{X/S}|_{X_s} \to \W^i_{X_s/\kappa(s)}$ is an isomorphism.
 \end{enumerate}
\end{cor}
\begin{proof}
 If $f$ has the base change property, then (ii) holds by assumption. To show that $\W^i_{X/S}$ is flat, it is sufficient to assume that $S$ is the spectrum of a local Noetherian ring, and that $X$ is affine. By Lemma~\ref{base-change-prop-criterion}, the base change of $\W^i_{X/S}$ to every Artinian ring is flat. Then a combination of \cite[00HP,\, 00MC,\,0584]{stacks} respective their methods allows to conclude that $\W^i_{X/S}$ is flat over $S$. For the converse, first observe that the base change property holds for a base change to a field $\ell$ mapping to $S$ since it factors through a field extension $\kappa(s) \to \ell$ for some $s \in S$. Then the assertion follows from Lemma~\ref{bijective-in-fibers}.
\end{proof}

With this criterion, we can give examples of generically log smooth families which are log Gorenstein and have the base change property, but which are not log toroidal.

\begin{figure}
 \begin{mdframed}
  \begin{center}
   \includegraphics[scale=0.8]{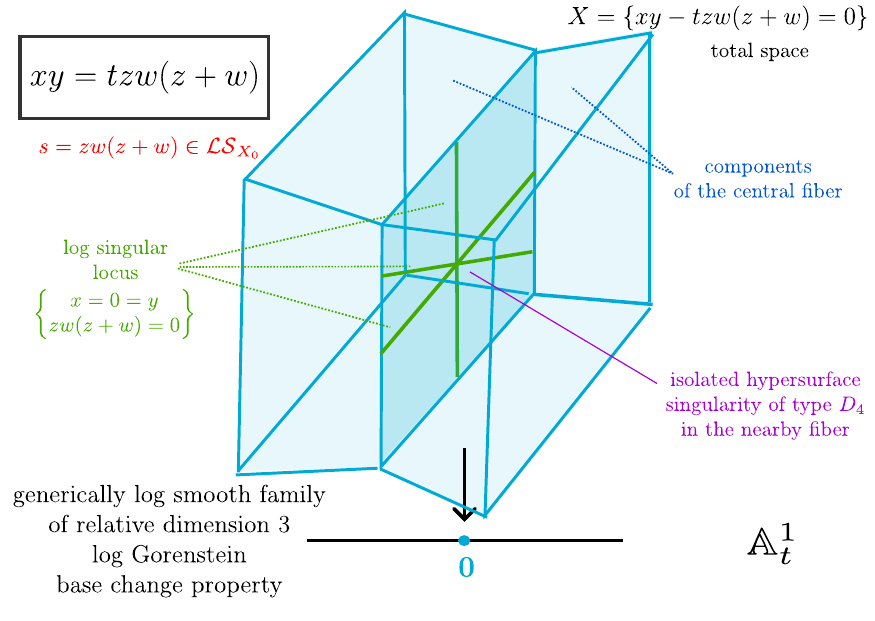}
  \end{center}
  \caption{Example~\ref{3-line-log-sing-ex}}
 \end{mdframed}
\end{figure}

\begin{ex}\label{3-line-log-sing-ex}\note{3-line-log-sing-ex}\index{$D_4$-singularity}
 Let 
 $$X = \Spec \kk[x,y,z,w,t]/(xy - tzw(z + w)),$$
 and let $f: X \to S = \bAA^1_t$ be the obvious map. Endow both the source and target with the divisorial log structure defined by $t = 0$. The nearby fiber $X_s = f^{-1}(s)$ is a three-dimensional normal variety with an isolated hypersurface singularity at $0$. In fact, $xy = zw(z + w)$ defines a threefold singularity of type $D_4$, in particular, a terminal Gorenstein singularity. Since $xy = zw$ is the only toric terminal Gorenstein threefold singularity, $X_s$ is not toroidal, and $f: X \to \bAA^1_t$ is not log toroidal. The central fiber $V = \Spec \kk[x,y,z,w]/(xy)$ is a normal crossing space with two components intersecting in $\bAA^2_{z,w}$. Inside $V$, the log singular locus is $Z = \{zw(z + w) = 0\}$, i.e., it consists of three lines intersecting in $0$. The induced log structure is given by $s = zw(z + w) \in \cL\cS_V \subseteq \T^1_V$ outside $Z$. Outside the very singular point $0$, the local model for the log singularity is $xy = tz$, so the family is log toroidal outside $0$. To show that $f: X \to S$ has the base change property and is log Gorenstein, we compute first (in Macaulay2) the relative classical tangent sheaf $\T^0_{X/S}$ as the kernel of the map 
 $$\begin{pmatrix}
    \frac{\partial f}{\partial x} & \frac{\partial f}{\partial y} & \frac{\partial f}{\partial z} & \frac{\partial f}{\partial w}\\
   \end{pmatrix}: \quad \cO_X^{\oplus 4} \to \cO_X,$$
  where $f = xy - tzw(z + w)$. Then $\V^{-1}_{X/S} = \T^0_{X/S}$, and we can compute $\W^1_{X/S}$ as the dual of $\V^{-1}_{X/S}$. It turns out that $\W^3_{X/S} \cong \cO_X$, so $f: X \to S$ is log Gorenstein. Flatness of $\W^1_{X/S}$ and $\W^2_{X/S}$ follows from the fact that they have no $\kk[t]$-torsion. Finally, a direct computation shows that both $\W^1_{X/S}|_{X_s}$ and $\W^2_{X/S}|_{X_s}$ are reflexive for all $s \in \kk$. For this, it suffices to check reflexivity for $s = 0$ and $s = 1$. Thus, $f: X \to S$ has the base change property.
\end{ex}

\begin{figure}
 \begin{mdframed}
  \begin{center}
   \includegraphics[scale=0.8]{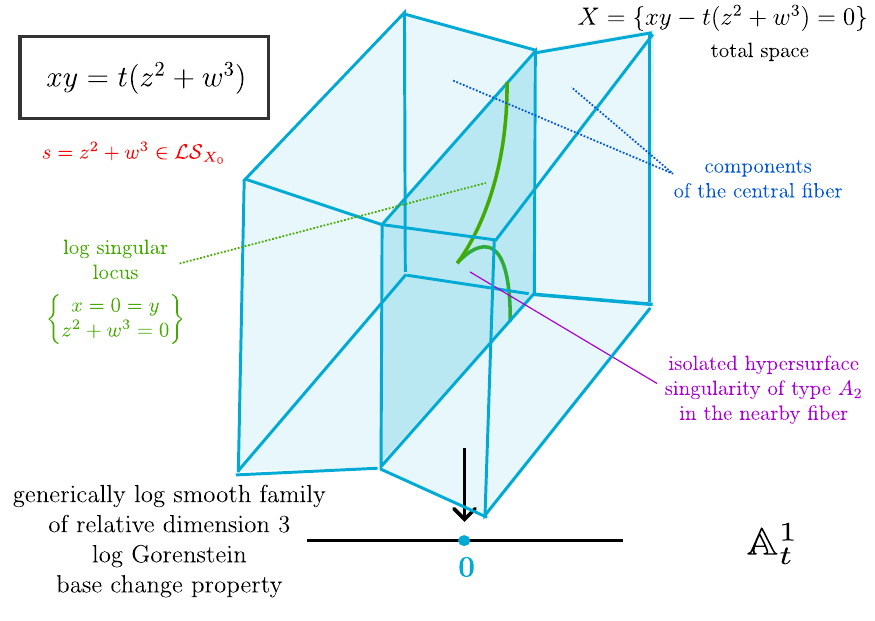}
  \end{center}
  \caption{Example~\ref{cuspidal-log-sing-ex}}
 \end{mdframed}
\end{figure}

\begin{ex}\label{cuspidal-log-sing-ex}\note{cuspidal-log-sing-ex}\index{$A_2$-singularity}
 Let 
 $$X = \Spec \kk[x,y,z,w,t]/(xy - t(z^2 + w^3)),$$
 and let $f: X \to S = \bAA^1_t$ be the obvious map. Endow both the source and target with the divisorial log structure defined by $t = 0$. Computations in Macaulay2 analogous to Example~\ref{3-line-log-sing-ex} show that $f: X \to S$ is a log Gorenstein generically log smooth family with the base change property as well. Now the log singular locus inside the central fiber $V = \Spec \kk[x,y,z,w]/(xy)$ is a cuspidal cubic with cusp at $0$. The local model for the log singularity outside the very singular point $0$ is again $xy = tz$. The nearby fiber $X_s = f^{-1}(s)$ is a normal threefold with isolated hypersurface singularity at $0$. More precisely, $X_s$ has a Kleinian singularity of type $A_2$, i.e., again a terminal Gorenstein singularity which is not toroidal. Thus, also this $f: X \to \bAA^1_t$ is not log toroidal.
\end{ex}

\subsubsection*{The log Gorenstein property}\index{log Gorenstein}

On the one hand side, every log Calabi--Yau generically log smooth family is log Gorenstein; on the other hand side, if a generically log smooth family $f: X \to S$ has the base change property, and we know that it is log Gorenstein, then also $\V^p_{X/S}$ is flat over $S$, and its formation commutes with base change. For these reasons, we will virtually always assume that our generically log smooth families are log Gorenstein. Since we have not yet done so elsewhere, we investigate briefly the log Gorenstein notion.

\begin{lemma}\label{log-Gor-basic}\note{log-Gor-basic}
 Let $f: X \to S$ be a generically log smooth family of relative dimension $d$. If $f: X \to S$ is log Gorenstein, so is any base change $g: Y \to T$. If $f: X \to S$ has the base change property for $\W^d_{X/S}$, i.e, $c^*\W^d_{X/S} \to \W^d_{Y/T}$ is an isomorphism for every base change, then $f: X \to S$ is log Gorenstein if and only if every fiber $X_s \to \kappa(s)$ is log Gorenstein.
\end{lemma}
\begin{proof}
 If $\W^d_{X/S}$ is a line bundle, so is $c^*\W^d_{X/S}$, where $c: Y \to X$ is the base change map. Then $c^*\W^d_{X/S}$ is reflexive, and $\W^d_{Y/T}$ is reflexive as well; both are isomorphic on $V = c^{-1}(U)$, so they are isomorphic, and $\W^d_{Y/T}$ is a line bundle. 
 
 If $f: X \to S$ has the base change property for $\W^d_{X/S}$, then $\W^d_{X/S}|_{X_s} \cong \W^d_{X_s/\kappa(s)}$. The latter is a line bundle by assumption. In particular, we have $\mathrm{dim}_{\kappa(x)}(\W^d_{X/S} \otimes \kappa(x)) = 1$ for all points $x \in X$. Since $\W^d_{X/S}$ is a line bundle on $U$, this implies that it is a line bundle on $X$. Namely, we find locally a surjection $\cO_X \to \W^d_{X/S}$, and it must be injective because it is injective on $U$ and $\cO_X$ is $Z$-closed.
\end{proof}

Recall from \cite[I, Defn.~4.3.1]{LoAG2018} that a homomorphism $\theta: Q \to P$ of integral monoids is \emph{vertical} if its image is not contained in any proper face of $P$. If $\theta: \NN \to P$ is a vertical saturated injection of sharp toric monoids, then $P$ is a Gorenstein monoid. Namely, let $\rho = \theta(1)$, and let $E \subset P$ be the unique basis of $P$ as an $\NN$-set. $E$ is the union of those faces of $P$ which do not contain $\rho$, i.e., since $\theta$ is vertical, we have $E = \partial P$. Every $p \in P$ decomposes uniquely as $p = e + n\rho$ with $e \in E$ and $n \in \NN$. Thus, we have $\mathrm{int}(P) = \rho + P$, hence $P$ is Gorenstein with Gorenstein degree $\rho$ by the criterion in \cite[p.~126]{Oda1988}. In particular, $A_\theta: A_P \to A_\NN$ is a flat morphism with Gorenstein fibers. Since every vertical saturated log smooth morphism $f: U_0 \to S_0$ over $S_0 = \Spec (\NN \to \kk)$ admits a local model by some vertical saturated injection $\theta: \NN \to P$, we find that $U_0$ is a Gorenstein scheme. With this preparation, we prove:

\begin{lemma}\label{log-Gor-vertical}\note{log-Gor-vertical}
 Let $f: X \to S$ be a vertical generically log smooth family, i.e., for every $x \in U$, the map $\overline\M_{S,f(x)} \to \overline\M_{U,x}$ is vertical. Assume that $f: X \to S$ is log Gorenstein. Then every fiber is a Gorenstein scheme. Conversely, if every fiber is a Gorenstein scheme, then $f: X \to S$ is log Gorenstein provided that either $S$ is a log point or $f: X \to S$ has the base change property for $\W^d_{X/S}$.\footnote{Presumably the converse holds in general. For example, when we know that, for a vertical saturated log smooth morphism $f: X \to S$, $\Omega^d_{U/S}$ is isomorphic to the relative dualizing sheaf $\omega_{U/S}$, then the converse holds in general. Unfortunately, this is only known for $S = \Spec(\NN \to \kk)$, see \cite[Thm.~2.21]{Tsuji1999dual}.}
\end{lemma}
\begin{proof}
 First, we show that every fiber of $f: U \to S$ is a Gorenstein scheme. We may assume that $S = \Spec (Q \to \kk)$ for some sharp toric monoid $Q$ and field $\kk$. If $Q = 0$, then verticality implies that $f: U \to S$ is strict, hence smooth, and $U$ is a Gorenstein scheme. If $Q = \NN$, then our claim follows from the above discussion. In the general case, we find a local homomorphism $h: Q \to \NN$, giving rise to a map of log schemes $b: T = \Spec (\NN \to \kk) \to \Spec (Q \to \kk)$. After base change along $b: T \to S$, the underlying scheme of $U$ remains unchanged, and we are in the case $Q = \NN$. Thus, here $U$ is a Gorenstein scheme as well. Now assume that $f: X \to S$ is log Gorenstein. Then every fiber is log Gorenstein, and after base change induced by $h: Q \to \NN$ as above, we can assume that the base is $S = \Spec (\NN \to \kk)$. Let $\omega_X$ be the dualizing sheaf of $X$. As in \cite[Prop.~2.21]{FeltenThesis}, we have an isomorphism $\omega_X|_U \cong \W^d_{U/S}$. By \cite[0AWN]{stacks}, the dualizing sheaf $\omega_X$ has the property $(S_2)$, so we have $\omega_X = j_*\omega_X|_U$. Since $\W^d_{X/S}$ is a line bundle by assumption, this implies $\omega_X \cong \W^d_{X/S}$. But then the dualizing sheaf is a line bundle, so $X$ is Gorenstein. Conversely, if each fiber of $f: X \to S$ is Gorenstein, then we can apply \cite[Cor.~2.22]{FeltenThesis} after base change along $h: Q \to \NN$ and obtain that $\W^d_{X_s/\kappa(s)}$ is a line bundle on each fiber. If $S$ is a log point itself, then we are finished. If $f: X \to S$ has the base change property for $\W^d_{X/S}$, then the result follows from Lemma~\ref{log-Gor-basic}.
\end{proof}

If $f: X \to S$ is not vertical, then it may be log Gorenstein even if the fibers are not Gorenstein schemes.

\begin{ex}
 Let $P \subseteq \ZZ^2$ be the sharp toric monoid generated by ray generators $(1,0)$ and $(-2,3)$ as well as interior generators $(0,1)$ and $(-1,2)$. Then $P$ is not a Gorenstein monoid, hence $A_P$ is not a Gorenstein scheme. Nonetheless, $A_P \to A_0$ is of course log Gorenstein because it is log smooth. Next, let $\tilde P \subseteq \ZZ^3$ be the sharp toric  monoid consisting of elements $(x,y,z)$ with either $x < 0$, $z \geq 0$, and $(x,y) \in P$ or $x \geq 0$, $z \geq x$, and $(x,y) \in P$. Let $\rho = (0,0,1)$. Then $\theta: \NN \to \tilde P, \: 1 \mapsto \rho$, is a saturated injection, so $A_\theta: A_P \to A_\NN$ is saturated and log smooth, in particular log Gorenstein, but neither the special fiber nor the generic fiber of $A_\theta$ is a Gorenstein scheme.
\end{ex}

\subsubsection*{The Gerstenhaber calculus}

Polyvector fields and the de Rham complex carry more structure which is very relevant to our deformation theory. The obvious $\wedge$-product, the unit element $1 \in \V^0_{X/S}$, and the \emph{negative} Schouten--Nijenhuis bracket $[-,-] = -[-,-]_{sn}$ (already mentioned above) turn $\V_{X/S}^\bullet$ into a Gerstenhaber algebra in the context $\mathfrak{Coh}'(X/S)$.\footnote{The context $\mathfrak{Coh}'(X/S)$, as opposed to $\mathfrak{Coh}(X/S)$, allows also modules which are not flat over $S$.}\footnote{The sign convention for the bracket is used in various places, including \cite{ChanLeungMa2023}, \cite{Felten2022}, and \cite{FeltenThesis}.}
Similarly, $\W^\bullet_{X/S}$ naturally comes with a $\wedge$-product, $1 \in \W^0_{X/S}$, and a de Rham differential $\partial: \W^i_{X/S} \to \W^{i + 1}_{X/S}$ which is only $f^{-1}(\cO_S)$-linear. As is shown e.g.~in \cite[Lemma~2.32]{FeltenThesis} in the author's thesis, we can construct a natural (right) contraction map $\invneg$. Then we can define the Lie derivative $\cL$ via the Lie--Rinehart homotopy formula. As is checked in \cite[Lemma~2.36]{FeltenThesis}, this turns $(\V^\bullet_{X/S},\W^\bullet_{X/S})$ into a (one-sided) Gerstenhaber calculus in the context $\mathfrak{Coh}'(X/S)$. We can also endow it with a left contraction. Summarizing the discussion, we have:

\begin{prop}\label{G-A-construction}\note{G-A-construction}\index{generically log smooth family!Gerstenhaber calculus}
 Let $f: X \to S$ be a generically log smooth family of relative dimension $d$.
 \begin{enumerate}[label=\emph{(\alph*)}]
  \item We have a two-sided Gerstenhaber calculus 
 $$\V\,\W_{X/S}^\bullet := (\V^\bullet_{X/S},\wedge,1,[-,-],\: \W_{X/S}^\bullet,\wedge,1,\partial,\,\invneg\,,\cL,\vdash)$$ 
 of dimension $d$ in the context $\mathfrak{Coh}'(X/S)$. The graded pieces are reflexive coherent sheaves. The $\wedge$-products and the two contractions $\vdash$ and $\invneg$ are $\cO_X$-linear whereas $[-,-]$, $\partial$, and $\cL$ are only $f^{-1}(\cO_S)$-linear. The Gerstenhaber calculus is strictly faithful.
   \item If $f: X \to S$ is log Gorenstein, then $\V\,\W^\bullet_{X/S}$ is locally Batalin--Vilkovisky.
   \item If $f: X \to S$ is log Calabi--Yau, then $\V\,\W^\bullet_{X/S}$ carries the structure of a two-sided Batalin--Vilkovisky calculus, depending on the choice of a global trivializing section $\omega \in \W^d_{X/S}$. 
   \item If $f: X \to S$ has the base change property, then $\W^i_{X/S}$ is flat over $S$. If $f: X \to S$ is moreover log Gorenstein, then $\V^p_{X/S}$ is flat over $S$ as well. In this case, we have a two-sided Gerstenhaber calculus in the context $\mathfrak{Coh}(X/S)$.
 \end{enumerate}
\end{prop}
\begin{proof}
 We have a one-sided Gerstenhaber calculus by \cite[Lemma~2.32]{FeltenThesis}. On $U$, we have that $\W^d_{U/S}$ is a line bundle, so we can construct the left contraction $\vdash$ on $U$ as discussed before Lemma~\ref{one-sided-two-sided-G-calc}. Then the left contraction on $X$ is obtained as the direct image from $U$. Most relations then follow from Lemma~\ref{one-sided-two-sided-G-calc}, first on $U$, and then on $X$. By induction on $r$, we see that 
 \begin{align}
  0 =\ &(\theta_r \wedge ... \wedge \theta_1) \ \invneg \ (\omega \wedge \alpha)  \nonumber \\
  =\ &((\theta_r \wedge ... \wedge \theta_1)\ \invneg \ \omega) \wedge \alpha + \sum_{i = 1}^r (-1)^{d + 1 - i}((\theta_r \wedge ... \wedge \hat\theta_i \wedge ... \wedge \theta_1) \ \invneg \ \omega) \wedge (\theta_i \ \invneg \ \alpha) \nonumber
 \end{align}
 for $r \geq 1$ and $\theta_i \in \V^{-1}_{X/S}$ and $\alpha \in \W^1_{X/S}$. Thus, we have 
 $$(\theta_r \wedge ... \wedge \theta_1) \vdash \alpha = \sum_{i = 1}^r (-1)^{i + 1}(\theta_r \wedge ... \wedge \hat\theta_i \wedge ... \wedge \theta_1)  \wedge (\theta_i \vdash \alpha).$$
 From this, a direct computation yields the relation for $(\theta \wedge \xi) \vdash \alpha$ when $\theta = \theta_r \wedge ... \wedge \theta_1$ and $\xi = \xi_s \wedge ... \wedge \xi_1$. However, on $U$, every section of $\V^p_{X/S}$ is locally of this form, so the relation follows. As usual, the above formula yields
 $$(\theta_r \wedge ... \wedge \theta_1) \vdash (\alpha_1 \wedge ... \wedge \alpha_r) = \sum_{\tau \in S_r} (-1)^\tau \prod_{i = 1}^r \theta_i \vdash \alpha_{\tau(i)},$$
 and since we also have 
 $$(\theta_r \wedge ... \wedge \theta_1) \ \invneg \ (\alpha_1 \wedge ... \wedge \alpha_r) = \sum_{\tau \in S_r} (-1)^\tau \prod_{i = 1}^r \theta_i \ \invneg \ \alpha_{\tau(i)},$$
 we obtain $\lambda(\theta \vdash \alpha) = (-1)^r \, \theta \ \invneg \ \alpha$ for $r \geq 2$ from the case $r = 1$, using that, on $U$, every element of $\V^{-r}_{X/S}$ respective $\W^r_{X/S}$ is locally of the product form.
 
 On the strict locus $U^{str}$, strict faithfulness is clear from the construction. Since the strict locus is scheme-theoretically dense, the restriction $\V^{-1}_{X/S} \to j_*\V^{-1}_{U^{str}/S}$ is injective; hence, $\V^\bullet_{X/S}$ is strictly faithful.
 
 If $f: X \to S$ is log Gorenstein, then a local generator $\omega \in \W^d_{X/S}$ on an open subset $W \subseteq X$ induces an isomorphism $\V^p_{(U \cap W)/S} \cong \W^{p + d}_{(U \cap W)/S}$. Since both sides are reflexive, this is an isomorphism on $W$. In the log Calabi--Yau case, we use Proposition~\ref{BV-calc-construction}. The remaining statements are clear.
\end{proof}

When we keep only the data of a Lie--Rinehart algebra in the context $\mathfrak{Coh}'(X/S)$, then we denote it by $\cL\R^\bullet_{X/S}$; it has two $\cO_X$-modules 
$$\cL\R^F_{X/S} = \cO_X, \quad \cL\R^T_{X/S} = \Theta^1_{X/S}.$$
As in the case of the Gerstenhaber algebra, we consider $\cL\R^T_{X/S}$ endowed with the negative of the usual Lie bracket on the log tangent sheaf, and it acts on $\cL\R^F_{X/S}$ via the negative of the evaluation of the log derivation on functions. However, when we consider $\Theta^1_{X/S}$ as a Lie algebra, then we use the usual Lie bracket, as explained above.

\section{Infinitesimal automorphisms}\label{inf-auto}\note{inf-auto}\index{generically log smooth family!infinitesimal automorphisms}

We generalize the treatment of infinitesimal automorphisms in \cite{Felten2022} from log smooth deformations to generically log smooth deformations. Let $B' \to B$ be a surjection in $\mathbf{Art}_Q$ with kernel $I \subset B'$, write $S \to S'$ for $S_B \to S_{B'}$, and let
\[
 \xymatrix{
 X_0 \ar[r] \ar[d]^{f_0} & X \ar[r] \ar[d]^f & X' \ar[d]^{f'} \\
 S_0 \ar[r] & S \ar[r] & S' \\
 }
\]
be a Cartesian diagram of generically log smooth deformations with log smooth loci $U_0$,  $U$, and $U'$. The ideal sheaf of the closed subscheme $X \subset X'$ is $\I = I \cdot \cO_{X'}$. An automorphism of $f': X' \to S'$ over $f: X \to S$ consists of an automorphism $\phi: \cO_{X'} \to \cO_{X'}$ of sheaves of rings, compatible with $f': X' \to S'$ and $\cO_{X'} \to \cO_X$, and an automorphism $\Phi: \M_{U'} \to \M_{U'}$ of sheaves of monoids, compatible with the log part of $f': X' \to S'$ and $\M_{U'} \to \M_U$; on $U'$, they must be compatible with $\alpha: \M_{U'} \to \cO_{U'}$ so that $(\phi,\Phi)$ is an automorphism of a log scheme. They form a sheaf of groups $\A ut_{X'/X}$ on $X'$.

\begin{lemma}\label{sheaf-of-autom}\note{sheaf-of-autom}
 Let $V$ be an open subset of $X_0$. Then the natural restriction map
 $$\Gamma(V,\A ut_{X'/X}) \to \Gamma(V \cap U_0,\A ut_{X'/X})$$
 is an isomorphism. Thus, $\A ut_{X'/X} = j_* \A ut_{U'/U}$ where $j: U' \to X'$ is the inclusion.
\end{lemma}
\begin{proof}
 Since $X'$ is Cohen--Macaulay and $Z' = X' \setminus U'$ is of codimension $\geq 2$, we have $\cO_{X'} = j_*\cO_{U'}$ for the inclusion $j: U' \to X'$. Thus, given $\phi$ on $U' \cap V$, it can be extended in a unique way to an automorphism of sheaves of groups on $X' \cap V$. It satisfies all compatibilities automatically.
\end{proof}

Since $U' \to S'$ is log smooth, the results of \cite{Felten2022} hold for $\A ut_{U'/U}$. Let us briefly recall them. The sheaf of groups $\A ut_{U'/U}$ is isomorphic to the sheaf $\D er_{U'/S'}(\I)$ of relative log derivations $(D, \Delta)$ with values in $\I$, i.e., $D: \cO_{U'} \to \I$ is a derivation and $\Delta: \M_{U'} \to \I$ is its log part.
This is a sheaf of Lie algebras as a subalgebra of $\Theta^1_{U'/S'}$ which
is filtered by 
$$F^k := F^k\D er_{U'/S'}(\I) := \D er_{U'/S'}(\I^k) \subset \D er_{U'/S'}(\I)$$
for $k \geq 1$, the sheaf of derivations with values in $\I^k \subset \I$. We have $[F^k,F^\ell] \subset F^{k + \ell}$, so the 
lower central series of $\D er_{U'/S'}(\I)$ is eventually 
$0$, and it is a sheaf of nilpotent Lie algebras. In particular, the Baker--Campbell--Hausdorff formula
turns $\D er_{U'/S'}(\I)$ into a sheaf of groups.

More precisely, in \cite{Felten2022}, we find two explicit isomorphisms
\[
 \xymatrix{
  {\D er}_{U'/S'}(\I)
  \ar@<0.5ex>[r]^-{\mathrm{Exp}} &
  {\A ut}_{U'/U} \ar@<0.5ex>[l]^-{\mathrm{Log}}. \\
 }
\]
Given $(D,\Delta) \in \D er_{U'/S'}(\I)$, the automorphism
$(\phi,\Phi) = \mathrm{Exp}_{U'/U}(D,\Delta)$ is defined by the formulae
\begin{align}
 \phi: \cO_{U'} \to \cO_{U'}, \quad \phi(a) &= \sum_{n = 0}^\infty \frac{D^n(a)}{n!}\nonumber \\
 \Phi: \M_{U'} \to \M_{U'}, \quad \Phi(m) &= m + \alpha^{-1}\left(\sum_{n = 0}^\infty \frac{[\Delta(m) + D]^n(1)}{n!}\right) \nonumber 
\end{align}
where the sums are actually finite and the symbol $\Delta(m)$ denotes the multiplication operator with this element. Conversely, given $(\phi,\Phi) \in \A ut_{U'/U}$, the classical part of $\mathrm{Log}_{U'/U}(\phi,\Phi)$ is
$$ D: \cO_{U'} \to \I, \quad D(a) = \sum_{n = 1}^\infty \frac{(-1)^{n - 1}[\phi - \mathrm{Id}]^n(a)}{n},$$
and the log part is
$$\Delta: \M_{U'} \to \I, \quad \Delta(m) = \sum_{n = 1}^\infty \left(\sum_{k = 0}^n {\binom{n}{k}} \frac{(-1)^{k + 1}\alpha(\Phi^k(m) - m)}{n}\right),$$
which is a finite sum as well.

On $U'$, we have $\D er_{U'/S'}(\I) = I \cdot \Theta^1_{U'/S'}$, which is also the kernel of the restriction map $\Theta^1_{U'/S'} \to \Theta^1_{U/S}$. Thus, we have $\A ut_{X'/X} \cong j_*(I \cdot\Theta^1_{U'/S'})$ and moreover, 
$$\mathrm{Log}:\quad  \A ut_{X'/X} \xrightarrow{\cong} \mathrm{ker}(\Theta^1_{X'/S'} \to \Theta^1_{X/S})$$
where the map $\Theta^1_{X'/S'} \to \Theta^1_{X/S}$ is induced via push-forward from the map on $U'$, where we have a canonical restriction of relative log derivations, and need not be surjective if the deformation does not have the base change property defined in Definition~\ref{bc-prop} (or possibly if it is not log Gorenstein). If $f': X' \to S'$ has the base change property and is log Gorenstein, then the kernel is equal to $I \cdot \Theta^1_{X'/S'}$.

If $B' \to B$ is a \emph{small} extension so that $I \cdot \m_{B'} = 0$, then we obtain 
$$\A ut_{X'/X} = \Theta^1_{X_0/S_0} \otimes_\kk I.$$

\subsubsection*{Restricting automorphisms}

Let $B' \to B$ and $B'' \to B'$ be two surjections in $\mathbf{Art}_Q$, and let $f: X \to S$, $f': X' \to S'$, $f'': X'' \to S''$ be (compatible) deformations of $f_0: X_0 \to S_0$. Then there is a natural restriction map
$$\rho: \A ut_{X''/X} \to \A ut_{X'/X},$$
which fits into a diagram
\[
 \xymatrix{
  1 \ar[r] & \A ut_{X''/X} \ar[d]^{\rho} \ar[r]^{\mathrm{Log}} & \Theta^1_{X''/S''} \ar[d] \ar[r] & \Theta^1_{X/S} \ar@{=}[d] \\
  1 \ar[r] & \A ut_{X'/X} \ar[r]^{\mathrm{Log}} & \Theta^1_{X'/S'} \ar[r] & \Theta^1_{X/S} \\
 }
\]
of (left) exact sequences. If $f'': X'' \to S''$ has the base change property and $f_0: X_0 \to S_0$ is log Gorenstein, then the middle vertical arrow is surjective \footnote{Since $W^d_{X''/S''}|_{X_0} \cong W^d_{X_0/S_0}$ is a line bundle and $W^d_{X''/S''}$ is flat over $S''$, it is a line bundle as well, cf.~Remark~\ref{geom-fam-Gerstenhaber-calc-rem}. Thus, $f'': X'' \to S''$ is log Gorenstein.}. In this case, also $\rho$ is surjective. In particular, it is surjective over affine open subsets---every automorphism of $f': X' \to S'$ defined on an affine open subset can be lifted to $f'': X'' \to S''$.

\begin{figure}
 \begin{mdframed}
  \begin{center}
   \includegraphics[scale=0.8]{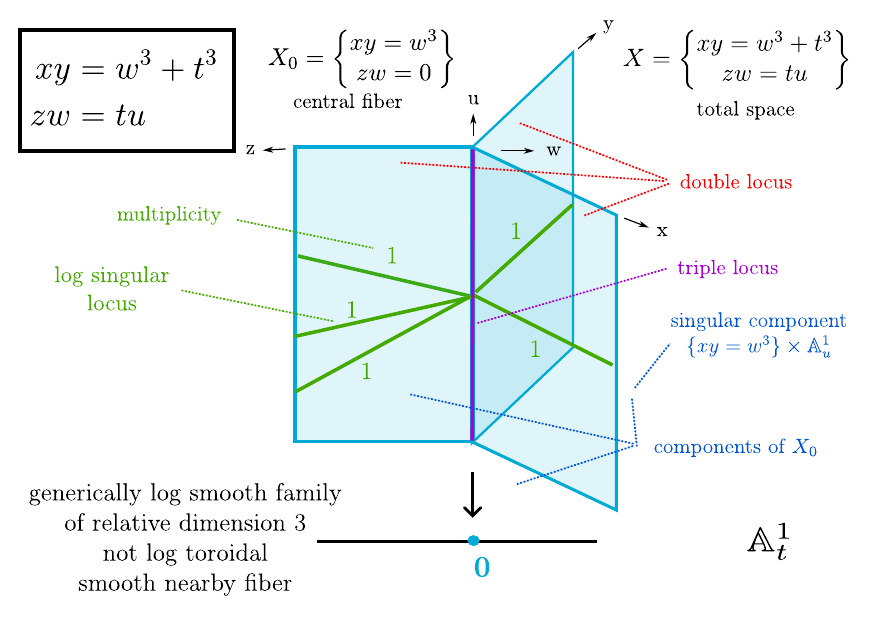}
  \end{center}
  \caption{Example~\ref{base-change-violation-t3-w3}. Compare also with Example~\ref{xy-w2-t2-zw-tu-example}.}
 \end{mdframed}
\end{figure}

\begin{ex}\label{base-change-violation-t3-w3}\note{base-change-violation-t3-w3}
 If $f'': X'' \to S''$ has \emph{not} the property that $\Theta^1_{X''/S''} \otimes_{\cO_{X''}} \cO_{X_0} \cong \Theta^1_{X_0/S_0}$, then surjectivity may fail. For an example, let $S = \bAA^1_t$, let 
 $$X = \Spec \kk[x,y,z,w,t,u]/(xy - w^3 - t^3,\: zw - tu),$$
 and let $f: X \to S$ be the obvious map. Endow both spaces with the compactifying log structures coming from $\{t = 0\}$. This is a generically log smooth family. The central fiber $X_0$ has three irreducible components intersecting in three strata $\bAA^2_{u,z}$, $\bAA^2_{u,x}$, and $\bAA^2_{u,y}$ of codimension $1$. The part of the log singular locus inside the central fiber is given by three lines $\{u^3 + z^3 = 0\}$ inside $\bAA^2_{u,z}$ and one line $\{u = 0\}$ each inside $\bAA^2_{u,x}$ and $\bAA^2_{u,y}$. The family is unisingular in the sense of Definition~\ref{unisingular-defn}.
 
 The sheaf of relative classical derivations on $f: X \to S$ equals the sheaf of relative log derivations on $f: X \to S$. This allows us to compute the restriction maps 
 $$\Theta^1_{X_2/S_2} \to \Theta^1_{X_1/S_1} \to \Theta^1_{X_0/S_0}$$
 explicitly, and it turns out that $\A ut_{X_2/X_0} \to \A ut_{X_1/X_0}$ is \emph{not} surjective. In fact, the cokernel of the map is a $\kk$-vector space of dimension $1$. We have carried out the computation in Macaulay2, and an annotated script for this computation is included in Chapter~\ref{t3-w3-script-sec} in the appendix. The computation is done over $\QQ$, but everything should commute with flat base change to $\kk$.
\end{ex}

\subsubsection*{The action of automorphisms on the Lie--Rinehart algebra and the Gerstenhaber calculus}

Let us assume that $f_0: X_0 \to S_0$ is log Gorenstein, and that our deformations have the base change property. Let $\varphi: X' \to X'$ be an automorphism of the deformation $f': X' \to S'$ over $f: X \to S$. Let us write $\varphi: X_1' \to X_2'$ to distinguish the roles of source and target. As part of the definition, we have a map 
$\phi: \W^0_{X'_2/S'} \to \W^0_{X'_1/S'}$. Pull-back of differential forms gives us 
$$d\varphi^*:\enspace \W^1_{X_2'/S'} \to \W^1_{X'_1/S'},$$
and the exterior powers of this map give maps $d^i\varphi^*$ on $\W^i_{X'/S'}$ after push-forward. Considering $\V^{-1}_{X_2'/S'}$ as the dual of $\W^1_{X_2'/S'}$ (via the universal property of the log differential forms, or equivalently, via the contraction map) gives us 
$$T\varphi^*: \enspace \V^{-1}_{X'_2/S'} \to \V^{-1}_{X_1'/S'}$$
via pre-composition with $(d\varphi^*)^{-1}$ and post-composition with $\phi$, i.e., if $\xi \in \V^{-1}_{X_2'/S'}$ and $\alpha \in \W^1_{X_1'/S'}$, then $T\varphi^*(\xi) \ \invneg \ \alpha = \phi(\xi \ \invneg \ (d\varphi^*)^{-1}(\alpha))$. Finally, we obtain a map $T^p\varphi^*$ on $\V^p_{X'/S'}$ via exterior powers and push-forward of $T\varphi^*$.

\begin{lemma}\label{geom-auto-gauge-trafo-corr}\note{geom-auto-gauge-trafo-corr}
 Let $f': X' \to S'$ have the base change property and be log Gorenstein. Let $\theta = \mathrm{Log}(\varphi)$. The above map is equal to the gauge transform $\mathrm{exp}_{-\theta}$ on the Gerstenhaber calculus $\V\,\W^\bullet_{X'/S'}$. In particular, it is an automorphism of the Gerstenhaber calculus. This induces a one-to-one correspondence between automorphisms of $f': X' \to S'$ over $f: X \to S$ and gauge transforms induced by elements $\theta \in I \cdot \V^{-1}_{X'/S'}$.
\end{lemma}
\begin{proof}
 This is essentially \cite[Lemma~3.2]{Felten2022}, but the proof there is not completely correct. First, we work on $U'$. Then $\phi: \cO_{X'} \to \cO_{X'}$ is equal to $\mathrm{exp}_{-\theta}$ by definition; namely, if $\theta = (D,\Delta)$, then $[\theta,g] = -D(g)$ for $g \in \cO_{X'}$. From the axioms of a Gerstenhaber calculus, we deduce $\theta \ \invneg \ \partial g = [-\theta,g] \ \invneg \ 1 = D(g)$.\footnote{We see here a nice example of how our choice $[-,-] = -[-,-]_{sn}$ interacts with our chosen axioms for a Gerstenhaber calculus. Since we do not switch signs in the contraction map, we get $D(a) = \theta \ \invneg\ \partial a$. Then, due to the sign in the Lie--Rinehart homotopy formula, we have $\partial D(a) = -\cL_\theta(a)$, and this is what we need because we are using $\mathrm{exp}_{-\theta}$.} Let us write $\square_\theta: \W^1_{U'/S'} \to \W^1_{U'/S'}, \: \alpha \mapsto \partial(\theta \ \invneg \ \alpha)$. Since $d\varphi^* \circ \partial = \partial \circ \phi$, we have 
 $$d\varphi^*(\partial g) = \sum_{n = 0}^\infty \frac{\partial D^n(g)}{n!} = \sum_{n = 0}^\infty \frac{\square_\theta^n(\partial g)}{n!} = \sum_{n = 0}^\infty \frac{\cL_{-\theta}^n(\partial g)}{n!} = \mathrm{exp}_{-\theta}(\partial g).$$
 The second last equality holds due to $\cL_{-\theta}(\partial g) = \nabla_\theta(\partial g) + \theta \ \invneg \ \partial^2g$. Since we have both $d\varphi^*(g \cdot \alpha) = \phi(g) \cdot d\varphi^*(\alpha)$ and $\mathrm{exp}_{-\theta}(g \cdot \alpha) = \phi(g) \cdot \mathrm{exp}_{-\theta}(\alpha)$, we have $d\varphi^* = \mathrm{exp}_{-\theta}$ on the submodule of $\W^1_{U'/S'}$ generated by elements of the form $\partial g$. On the strict locus of $U' \to S'$, the log differentials agree with the classical differentials, so this is everything. Since the strict locus is scheme-theoretically dense, we get $d\varphi^* = \mathrm{exp}_{-\theta}$ everywhere on $U'$. 
 
 There is a second, more conceptual proof of $d\varphi^* = \mathrm{exp}_{-\theta}$, which does not rely on the density of the strict locus.\footnote{If one works with \emph{integral} log smooth morphisms instead of saturated ones, then the strict locus may be empty, for example in the central fiber of $\bAA^2 \to \bAA^1, \: t \mapsto x^2y^2$. See also Chapter~\ref{char-abs-G-calc-sec} for another situation which needs the more general proof.} In general, $\W^1_{U'/S'}$ is locally generated by elements of the form $\delta(m)$, where $\delta: \M_{U'} \to \W^1_{U'/S'}$ is the log part of the universal derivation, see \cite[IV, Prop.~1.2.11]{LoAG2018}. This map satisfies $d\varphi^* \circ \delta = \delta \circ \Phi$, i.e., we have 
 \begin{align}
  d\varphi^*(\delta(m)) &= \delta\left(m + \alpha^{-1}\left(\sum_{n = 0}^\infty\frac{[\Delta(m) + D]^n(1)}{n!}\right)\right) \nonumber \\
  &= \delta(m) + \left(\sum_{n = 0}^\infty\frac{[\Delta(m) + D]^n(1)}{n!}\right)^{-1} \cdot d\left(\sum_{n = 0}^\infty\frac{[\Delta(m) + D]^n(1)}{n!}\right)  \nonumber
 \end{align}
 where $d: \cO_{U'} \to \W^1_{U'/S'}$ is the classical part of the universal derivation, and where $\Delta(m)$ denotes the operator which multiplies the argument with $\Delta(m) \in \cO_{U'}$. Using the formula in the proof of \cite[Lemma~2.2]{Felten2022}, we can show that 
 $$\left(\sum_{n = 0}^\infty\frac{[\Delta(m) + D]^n(1)}{n!}\right)^{-1} = \sum_{n = 0}^\infty\frac{[-\Delta(m) + D]^n(1)}{n!}.$$
 Since $\partial \delta(m) = 0$, we have $\cL_{-\theta}(\delta(m)) = \square_\theta(\delta(m))$, and then we obtain inductively $\partial\square_\theta^n(\delta(m)) = 0$ and $\cL_{-\theta}^n(\delta(m)) = \square_\theta^n(\delta(m))$. Thus, in order to show $d\varphi^*(\delta(m)) = \mathrm{exp}_{-\theta}(\delta(m))$, it is sufficient to show 
 $$\square_\theta^n(\delta(m)) = \sum_{k = 0}^n {\binom{n}{k}} [-\Delta(m) + D]^{n - k}(1) \cdot \partial[\Delta(m) + D]^k(1) =: A_n(m).$$
 Unfortunately, this seems to be surprisingly difficult. As the first step, we wish to show 
 $$\partial[\Delta(m) + D]^n(1) = \sum_{k = 1}^n{\binom{n}{k}}[\Delta(m) + D]^{n - k}(1)\cdot \square_\theta^k(\delta(m))$$
 for $n \geq 1$ (the sum really starts at $k = 1$). This is easy to verify for the first few values of $n$, but a direct inductive argument does not seem to be available. Assume that we know this formula up to $n$. Then we compute, for $1 \leq p \leq n + 1$, 
 \begin{align}
  P_p(m) &:= \sum_{k = 1}^{p} {\binom{p}{k}} \left(\partial [\Delta(m) + D]^{p - k}(1)\right) \wedge \square_\theta^k(\delta(m)) \quad \in \W^2_{U'/S'} \nonumber \\
  &= \sum_{\substack{k,\ell \geq 1 \\ k + \ell \leq p}} \frac{p!}{k!\ell!(p - k - \ell)!} [\Delta(m) + D]^{p - k - \ell} \wedge \square_\theta^\ell(\delta(m)) \wedge \square_\theta^k(\delta(m)) = 0. \nonumber
 \end{align}
 The sum is zero since $\square_\theta^\ell(\delta(m)) \wedge \square_\theta^k(\delta(m)) = - \square_\theta^k(\delta(m)) \wedge \square_\theta^\ell(\delta(m))$, and the sum is symmetric in $k$ and $\ell$. This shows 
 \begin{align}
  \sum_{k = 1}^{p}&{\binom{p}{k}}\ \partial[\Delta(m) + D]^{p - k}(1)\, \cdot \,\Big(\theta \ \invneg \ \square_\theta^k(\delta(m))\Big) \nonumber \\
  &= \sum_{k = 1}^{p}{\binom{p}{k}} \ D[\Delta(m) + D]^{p - k} \, \cdot \, \square_\theta^k(\delta(m)) \nonumber
 \end{align}
 via the formula for $\theta \ \invneg \ (\alpha \wedge \beta)$ since $D(g) = \theta \ \invneg\  \partial g$ for $g \in \cO_{U'}$. With this preparation, the induction step is now a straightforward computation of $\partial [\Delta(m) + D]^{n + 1}(1)$. As soon as we have this formula, we can compute 
 \begin{align}
  A_n(m) &= \sum_{k = 0}^n {\binom{n}{k}}[-\Delta(m) + D]^{n - k}(1) \cdot \sum_{\ell = 0}^{k - 1}{\binom{k}{\ell}}[\Delta(m) + D]^{\ell}(1) \cdot \square_\theta^{k - \ell}(\delta(m)) \nonumber \\
  &= \sum_{p = 1}^n \sum_{q = 0}^{n - p}\frac{n!}{(n - p - q)!p!q!}[-\Delta(m) + D]^q(1) \cdot [\Delta(m) + D]^{n - p - q}(1) \cdot \square_\theta^p(\delta(m)) \nonumber \\ 
  &= \sum_{p = 1}^n {\binom{n}{p}} [\Delta(0) + D]^{n - p}(1) \cdot \square_\theta^p(\delta(m)) = \square_\theta^n(\delta(m)) \nonumber
 \end{align}
 by setting $p = k - \ell$ and $q = n - k$.
 
 Since both $d^\bullet\varphi^*$ and $\mathrm{exp}_{-\theta}$ are compatible with the $\wedge$-product, we obtain $d^i\varphi^* = \mathrm{exp}_{-\theta}$ on $\W^i_{U'/S'}$. A direct computation with the definition yields 
 $$T\varphi^*(\xi) \ \invneg \ \alpha = \mathrm{exp}_{-\theta}(\xi) \ \invneg \ \alpha$$
 for $\xi \in \Theta^1_{U'/S'}$ and $\alpha \in \W^1_{U'/S'}$, so $T\varphi^*(\xi) = \mathrm{exp}_{-\theta}(\xi)$. Then both $T^p\varphi^*$ and $\mathrm{exp}_{-\theta}$ are compatible with the $\wedge$-product, so we have $T^p\varphi^* = \mathrm{exp}_{-\theta}$ as well. Since all sheaves $\F$ involved satisfy $j_*\F|_U = \F$, we obtain the assertion on $X'$ from the one on $U'$ by push-forward. Finally, infinitesimal automorphisms are in one-to-one correspondence with elements of $I \cdot \Theta^1_{X'/S'}$, and they are in turn in one-to-one correspondence with gauge transforms of the Gerstenhaber calculus by Lemma~\ref{gauge-inject} since $\V^\bullet_{X_0/S_0}$ is strictly faithful.
\end{proof}

\begin{rem}
 We write $\mathrm{Exp}(\theta)$ with a capital E for the automorphism induced by $\theta \in I \cdot \Theta^1_{X'/S'}$, and we write $\mathrm{exp}_{-\theta}$ with a lower case e for the gauge transform. In forming $\mathrm{Exp}(\theta)$, we use the positive Lie bracket of $\Theta^1_{X'/S'}$, and in forming $\mathrm{exp}_{-\theta}$, we use the negative Lie bracket of $\V^{-1}_{X'/S'}$.
\end{rem}

When we restrict this construction to $\V^0_{X'/S'}$ and $\V^{-1}_{X'/S'}$, then we obtain an induced gauge transform of the Lie--Rinehart algebra $\cL\R^\bullet_{X'/S'}$. Exactly as in the case of the Gerstenhaber calculus, automorphisms of $f': X' \to S'$ are in one-to-one correspondence with gauge transforms induced by $\theta \in I \cdot \cL\R_{X'/S'}^T$.

\section{Enhanced generically log smooth families}

Unless a (log Gorenstein) generically log smooth family $f: X \to S$ has the base change property, its Gerstenhaber calculus $\V\,\W_{X/S}^\bullet$ is not stable under base change. While the generically log smooth families which we studied in earlier works have the base change property, this is no longer true in the generality we are interested here, as shown by Example~\ref{base-change-violation-t3-w3}. We stabilize the Gerstenhaber calculus with the notion of an \emph{enhanced generically log smooth family}.

\begin{defn}\label{enhanced-gen-log-defn}\note{enhanced-gen-log-defn}
 An \emph{enhanced generically log smooth family of relative dimension $d \geq 1$}\index{enhanced generically log smooth family} is a tuple 
 $$(f: X \to S,\, U \subseteq X, \,\G^\bullet_{X/S},\,\A^\bullet_{X/S},\,\varpi^\bullet)$$
 where:
 \begin{itemize}
  \item $f: X \to S$ is a generically log smooth family of relative dimension $d$ with  log smooth locus $j: U \subseteq X$;
  \item $\W^d_{X/S}$ is a line bundle, i.e., $f: X \to S$ is log Gorenstein;
  \item $\G\C^\bullet_{X/S} = (\G^\bullet_{X/S},\A^\bullet_{X/S})$ is a two-sided Gerstenhaber calculus of dimension $d$ in the context $\mathfrak{Coh}(X/S)$; in particular, its pieces are flat over $S$;
  \item $\varpi^\bullet: \G^\bullet_{X/S} \to \V^\bullet_{X/S}$ and $\varpi^\bullet: \A^\bullet_{X/S} \to \W^\bullet_{X/S}$ form a map of two-sided Gerstenhaber calculi, i.e., $\varpi^p$ and $\varpi^i$ are $\cO_X$-linear maps which are compatible with the two $\wedge$-products, with the left contraction $\vdash$ and the right contraction $\invneg$, with the bracket $[-,-]$, with the de Rham differential $\partial$, with the Lie derivative $\cL_{-}(-)$, and with the constants;
  \item $\varpi^0: \G^0_{X/S} \to \V^0_{X/S}$ and $\varpi^0: \A^0_{X/S} \to \W^0_{X/S}$ are isomorphisms;
  \item $\varpi^d: \A^d_{X/S} \to \W^d_{X/S}$ is an isomorphism; in particular, $\A^d_{X/S}$ is a line bundle;
  \item the Gerstenhaber calculus $\G\C^\bullet_{X/S}$ is locally Batalin--Vilkovisky, i.e., every local volume form $\omega \in \A^d_{X/S}$ gives rise to a local isomorphism $\kappa_\omega: \G^p_{X/S} \cong \A^{p + d}_{X/S}$;
  \item $\varpi^\bullet$ is an isomorphism on $U$ for all $\G^p_{X/S}$ and $\A^i_{X/S}$.
 \end{itemize}
\end{defn}

Given a map $b: T \to S$, we define the fiber product $g: Y = X \times_S T$ in the obvious way. The Gerstenhaber calculus is given by $\G^p_{Y/T} := c^*\G^p_{X/S}$ and $\A^i_{Y/T} := c^*\A^i_{X/S}$, with the obvious maps $\varpi^\bullet$ to $\V^p_{Y/T}$ and $\W^i_{Y/T}$. Every operation on $\G\C^\bullet_{X/S}$ is a multilinear differential operator, so they pull-back to operations on $\G\C^\bullet_{Y/T}$ by Proposition~\ref{multilin-diff-op-shvs}. Since multi-compositions and linear combinations of multilinear differential operators are multilinear differential operators, we can consider every relation in $\G\C^\bullet_{X/S}$ as a multilinear differential operator which vanishes. Now the pull-back is compatible with multi-compositions and linear combinations, so the relation also holds in $\G\C^\bullet_{Y/T}$. Similarly, when $g: X' \to X$ is \'etale, separated, and of finite type,\footnote{\'Etale alone is not sufficient as the resulting family is required to be separated and of finite type.} then we have an induced two-sided Gerstenhaber calculus $\G\C^\bullet_{X'/S}$ which turns $f \circ g: X' \to S$ into an enhanced generically log smooth family. In both cases, the log Gorenstein and local Batalin--Vilkovisky condition as well as the isomorphy on $U$ are preserved due to their $\cO_X$-linear nature.

\begin{defn}\label{enhanced-morphism-defn}\note{enhanced-morphism-defn}
 Let $f: X \to S$ and $g: Y \to T$ be enhanced generically log smooth families of relative dimension $d$. Then a \emph{morphism} from $g$ to $f$ is a tuple $\varphi = (b,c,d^\bullet\varphi^*,T^\bullet\varphi^*)$ where:
 \begin{itemize}
  \item $b: T \to S$ is a morphism of log schemes;
  \item $c: Y \to X$ is a morphism of schemes with $V \subseteq c^{-1}(U)$ and $f \circ c = b \circ g$; it carries also the structure of a morphism of log schemes $c: V \to U$ with $f \circ c = b \circ g$ as morphisms of log schemes;
  \item the induced morphism $Y \to X \times_S T$ is strict and \'etale;
  \item $(d^\bullet\varphi^*,T^\bullet\varphi^*): c^*\G\C^\bullet_{X/S} \to \G\C^\bullet_{Y/T}$ is an isomorphism of two-sided Gerstenhaber calculi which is compatible with the two maps $\varpi^\bullet$ to $c^*\V\,\W^\bullet_{X/S}$ and $\V\,\W^\bullet_{Y/T}$; here, $c^*$ of a two-sided Gerstenhaber calculus is formed as indicated above.
 \end{itemize}
\end{defn}

The assumption that $Y \to X \times_S T$ is strict and \'etale is necessary in order to make sense of the map $T^\bullet\varphi^*$ being part of a morphism of two-sided Gerstenhaber calculi.

Often, the map $\varpi^\bullet$ is injective, for example when $\G\C^\bullet_{X/S}$ is obtained as the pull-back of $\V\,\W^\bullet_{X/S}$ in a family over a one-dimensional base, or when $\G\C^\bullet_{X/S}$ is the direct image from a resolution of log singularities with no contracted components. Within this text, we will usually assume that $\varpi^\bullet$ is injective and remains so after any base change since this simplifies the arguments---now $\G\C^\bullet_{X/S}$ is a Gerstenhaber subcalculus of $\V\,\W^\bullet_{X/S}$---and is often satisfied.

\begin{defn}
 We say that an enhanced generically log smooth family $f: X \to S$ is \emph{torsionless} if $\G^p_{X/S}$ and $\A^i_{X/S}$ are torsionless and remain so after any base change. Equivalently, $\varpi^\bullet$ is injective and remains so after any base change. By Lemma~\ref{injective-in-fibers}, it is sufficient to require this condition in the fibers of $f: X \to S$.
\end{defn}

For a generically log smooth families, we use the notation $\Theta^\bullet_{X/S}$ for the Gerstenhaber algebra of reflexive polyvector fields in positive degrees, endowed with the positive Schouten--Nijenhuis bracket. Analogously, we write $\Gamma^\bullet_{X/S}$ to denote $\G^\bullet_{X/S}$ in positive degrees and with the positive bracket, i.e., the negative of the bracket on $\G^\bullet_{X/S}$. In particular, we will use $\Gamma^1_{X/S}$ when studying infinitesimal automorphisms.

Examples of enhanced generically log smooth families which are not mere generically log smooth families, i.e., $\varpi^\bullet$ is not an isomorphism, can be found in Chapter~\ref{enhanced-sys-of-defo-sec}.

\section{Systems of deformations}

One should think of generically log smooth deformations rather as the correct language to talk about deformations of generically log smooth spaces than of the correct notion of deformations. For example, in \cite[Ex.~2.9]{GrossSiebertII}, they construct an infinite series of presumably pairwise non-isomorphic generically log smooth deformations of the central fiber of $\Spec \kk[x,y,z,t]/(xy - zt) \to \Spec \kk[t]$ endowed with the log structures defined by $\{t = 0\}$. Since the deformation theory that we present in this work relies on locally unique deformations, as are given for example in the case of log smooth deformations, we have to fix what we call a \emph{system of deformations}. This notion might seem ad hoc at first glance, and it probably is, but we will demonstrate in Corollary~\ref{sys-defo-exists} in the appendix that natural systems of deformations exist in a substantial range of applications. In this section, we study the plain generically log smooth case, and in the next section, we turn to the enhanced case.

\begin{defn}\label{admissible-def}\note{admissible-def}
 Let $f_0: (X_0,U_0) \to S_0$ be a generically log smooth family over a log point $S_0 = \Spec(Q \to \kk)$. Assume that $f_0: X_0 \to S_0$ is log Gorenstein. We say that a Zariski open cover $\V = \{V_\alpha\}_\alpha$ of $X_0$ is:
 \begin{enumerate}[label=(\alph*)]\index{open cover!pre-admissible}
  \item \emph{pre-admissible} if it consists of (only) finitely many affine open immersions $j_\alpha: V_\alpha \to X_0$, i.e., for every affine subset $W \subseteq X_0$, the intersection $W \cap V_\alpha$ is affine as well;
  \item \emph{weakly admissible}\index{open cover!weakly admissible} if it is pre-admissible and
  $$H^n(V_{\alpha_1} \cap ... \cap V_{\alpha_r},\cO_{X_0}|_{\alpha_1...\alpha_r}) = 0, \quad H^n(V_{\alpha_1} \cap ... \cap V_{\alpha_r},\Theta^1_{X_0/S_0}|_{\alpha_1...\alpha_r}) = 0$$  
  for all $n \geq 1$ and indices $\alpha_1,...,\alpha_r$;
  \item \emph{admissible}\index{open cover!admissible} if it is pre-admissible and 
  $$H^n(V_{\alpha_1} \cap ... \cap V_{\alpha_r},\F|_{\alpha_1...\alpha_r}) = 0$$
  for $\F = \G^p_{X_0/S_0}, \A^i_{X_0/S_0}$, all $n \geq 1$, and indices $\alpha_1,...,\alpha_r$;
  \item \emph{affine} if every $V_\alpha$ is an affine scheme.
 \end{enumerate}
 Every affine open cover is admissible, and every admissible open cover is weakly admissible.
\end{defn}

\begin{rem}
 Under these conditions, we can compute the cohomology of the respective sheaves from the \v{C}ech complex.
\end{rem}

\begin{defn}\label{sys-of-defo}\note{sys-of-defo}\index{generically log smooth family!system of deformations}\index{system of deformations}
 Let $f_0: (X_0,U_0) \to S_0$ be a generically log smooth family over a log point $S_0 = \Spec(Q \to \kk)$. Assume that $f_0: X_0 \to S_0$ is log Gorenstein. Let $\V = \{V_\alpha\}_\alpha$ be a weakly admissible open cover. A \emph{system of deformations} $\D$ of $f_0: X_0 \to S_0$, subordinate to $\V$, is a tuple 
 $$(V_{\alpha;A} \to S_A,\: \rho_{\alpha;BB'},\: \psi_{\alpha\beta;A})$$
 where:
 \begin{itemize}
  \item for every $A \in \mathbf{Art}_Q$, the  object $V_{\alpha;A} \to S_A$ is a generically log smooth deformation of $X_0|_\alpha := X_0|_{V_\alpha}$ over $S_A$; we assume that this deformation has the base change property, i.e., $\W^k_{V_{\alpha;A}/S_A}$ is flat over $S_A$;
  \item for every $B' \to B$ in $\mathbf{Art}_Q$, we have a map $\rho_{\alpha;BB'}: V_{\alpha;B} \to V_{\alpha;B'}$ of generically log smooth deformations; they are compatible with compositions and, by definition, with the maps from $X_0|_\alpha$ that are part of a generically log smooth deformation;
  \item on intersections $V_\alpha \cap V_\beta$, we have isomorphisms $\psi_{\alpha\beta;A}: V_{\alpha;A}|_{\alpha\beta} \xrightarrow{\cong} V_{\beta;A}|_{\alpha\beta}$ of generically log smooth deformations; they commute with the restrictions $\rho_{BB'}$, and, by definition, with the maps from $X_0|_{\alpha\beta}$; there is no cocycle condition on triple intersections $V_\alpha \cap V_\beta \cap V_\gamma$.
 \end{itemize}
 A generically log smooth deformation $f_A: X_A \to S_A$ of $f_0: X_0 \to S_0$ is \emph{of type $\D$} if we can find isomorphisms $\chi_\alpha: X_A|_\alpha \cong V_{\alpha;A}$ of generically log smooth deformations for each $\alpha$. Isomorphism classes of generically log smooth deformations of type $\D$ form a functor of Artin rings
 $$\mathrm{LD}_{X_0/S_0}^\D: \mathbf{Art}_Q \to \mathbf{Set}.$$
\end{defn}
\begin{rem}\label{etale-local-isom-rem}\note{etale-local-isom-rem}
 The condition $X_{A}|_\alpha \cong V_{\alpha;A}$ is equivalent to that there is an (\'etale) open cover $\{U_i\}_i$ by affines subordinate to $\V$ such that we can find isomorphisms $X_A|_i \cong V_{\alpha(i);A}|_i$ of generically log smooth deformations. The proof is by induction over small extensions, where $H^1(V_\alpha,\A ut_{X'/X}) = H^1(V_\alpha,\Theta^1_{X_0/S_0}\otimes_\kk I) = 0$, and then we have to arrange the automorphisms on $U_i$ to be the given one on $V_\alpha$ over the smaller ring, which we achieve by liftability of all automorphisms on the affines $U_i$.
\end{rem}
\begin{ex}
 Let $f_0: X_0 \to S_0$ be separated, log smooth, and saturated. Let $\V = \{V_\alpha\}_\alpha$ be a finite open affine cover. The log smooth deformations give rise to a system of deformations $\D$. Generically log smooth deformations of type $\D$ are the same as log smooth deformations.
\end{ex}

\begin{ex}
 Let $(V,Z,s)$ be a well-adjusted triple as defined in Definition~\ref{well-adj-def}, and let $f_0: X_0 \to S_0$ be the associated generically log smooth family. Assume that it is of elementary Gross--Siebert type. Then Corollary~\ref{sys-defo-exists} gives a system of deformations $\D$.
\end{ex}

\begin{lemma}\label{defo-type-D-is-defo-functor}\note{defo-type-D-is-defo-functor}
 Let $\D$ be a system of deformations. Then $\mathrm{LD}_{X_0/S_0}^\D$ is a deformation functor.
\end{lemma}
\begin{proof}
 The condition $(H_0)$ is clear. To check $(H_1)$, let $A' \to A$ be arbitrary and $A'' \to A$ be a surjection in $\mathbf{Art}_Q$, and let $B := A' \times_A A''$. Let $f: X \to S_A$ be of type $\D$, and let $f': X' \to S_{A'}$ and $f'': X'' \to S_{A''}$ be two liftings of type $\D$. We show that 
 $$Y := (|X_0|,\: \cO_{X'} \times_{\cO_X} \cO_{X''},\: \M' \times_\M \M'')$$
 is a common lifting of type $\D$ over $S_B$. After restricting to $V_\alpha$, we have an isomorphism $X'|_\alpha \cong V_{\alpha;A'}$ by assumption. By restriction, it induces an isomorphism $X|_\alpha \cong V_{\alpha;A}$. Since all automorphisms of $X|_\alpha$ can be lifted along $X''|_\alpha \to X|_\alpha$ \footnote{This is because $H^1(V_\alpha, I'' \cdot \Theta^1_{X''/S_{A''}}) = 0$ for $I''$ the kernel of $A'' \to A$.}, we can modify the existent isomorphism $X''|_\alpha \cong V_{\alpha;A''}$ to be compatible with the restriction $X''|_\alpha \to X|_\alpha$ and the given isomorphism $X|_\alpha \cong V_{\alpha;A}$. Then the universal property of $Y$ yields a map $Y|_\alpha \to V_{\alpha;B}$ of generically log smooth families which is compatible with everything. The base change of this map to $S_{A'}$ is an isomorphism, so it must be a closed immersion on underlying schemes. Let $J' \subset B$ be the kernel of $B \to A'$, and let $J''$ be the kernel of $B \to A''$. If $g \in \cO_{V_{\alpha;B}}$ is in the kernel of $Y|_\alpha \to V_{\alpha;B}$, it must be in both $J' \cdot \cO_{V_{\alpha;B}}$ and $J'' \cdot \cO_{V_{\alpha;B}}$. Since $V_{\alpha;B} \to S_B$ is flat, extension of ideals from $S_B$ commutes with intersection; we have $J' \cap J'' = 0$, hence $g = 0$. Because all other maps are strict, $Y|_\alpha \to V_{\alpha;B}$ is strict, too; in particular, it is an isomorphism of generically log smooth families. Finally, given $(H_1)$, the condition $(H_2)$ is inherited from $\mathrm{LD}_{X_0/S_0}^{gen}$.
\end{proof}
\begin{rem}
 Note that, as usual, the proof does not show bijectivity of the map in $(H_1)$ in general because $Y$ depends on the choice of the two maps $X \to X'$ and $X \to X''$, which are not encoded in the deformation functor.
\end{rem}

\section{Enhanced systems of deformations}\label{enhanced-sys-of-defo-sec}\note{enhanced-sys-of-defo-sec}

In Definition~\ref{sys-of-defo}, we have required that, in a system of deformations $\D$, every local model $V_{\alpha;A} \to S_A$ should have the base change property. This is necessary because we use this flatness in many places. Working with enhanced generically log smooth families allows us to use a modified Gerstenhaber calculus $\G\C^\bullet_{X/S}$, flat over $S$, instead of the sometimes non-flat $\V\,\W^\bullet_{X/S}$, and thus study the deformation theory of many generically log smooth families which do not have the base change property. For simplicity, we assume that our enhanced generically log smooth families are torsionless as defined above.\footnote{As long as $\varpi^{-1}: \G^{-1}_{X/S} \to \V^{-1}_{X/S}$ is injective and remains so after any base change, the theory should work with not too many modifications. However, if we drop the injectivity of $\varpi^{-1}$, then new complications arise due to the presence of gauge transforms which act trivially on the Gerstenhaber calculus.}

When it comes to infinitesimal automorphisms of enhanced generically log smooth families, then we have to distinguish between \emph{inner automorphisms} and \emph{outer automorphisms}. The latter are precisely the automorphisms in the sense of Definition~\ref{enhanced-morphism-defn}. The former are those outer automorphisms which are of the form $\mathrm{Exp}(\theta)$ for some $\theta \in I \cdot \Gamma^{1}_{X'/S'} = I \cdot \G^{-1}_{X'/S'}$ under the map 
$$\mathrm{Exp}:\enspace I \cdot \Gamma^{1}_{X'/S'} \xrightarrow{\varpi} I \cdot \Theta^1_{X'/S'} \subseteq \mathrm{ker}(\Theta^1_{X'/S'} \to \Theta^1_{X/S}) \xrightarrow{\mathrm{Exp}} \A ut_{X'/X},$$
i.e., the gauge transform on $\G\C^\bullet_{X'/S'} \subseteq \V\,\W^\bullet_{X'/S'}$ is given by $\mathrm{exp}_{-\theta}$. This obviously defines an (outer) automorphism in the sense of Definition~\ref{enhanced-morphism-defn} (cf.~Lemma~\ref{geom-auto-gauge-trafo-corr}). Since $\mathrm{Exp}$ is injective, every inner automorphism is induced by a unique $\theta \in I \cdot \Gamma^{1}_{X'/S'}$. The outer automorphisms are those induced by elements $\theta \in \mathrm{ker}(\Theta^{1}_{X'/S'} \to \Theta^{1}_{X/S})$ with the property that $\mathrm{exp}_{-\theta}$ preserves $\G\C^\bullet_{X'/S'} \subseteq \V\,\W^\bullet_{X'/S'}$.

\begin{defn}\label{enhanced-sys-of-defo}\note{enhanced-sys-of-defo}\index{enhanced generically log smooth family!system of deformations}
 Let $f_0: X_0 \to S_0$ be a torsionless enhanced generically log smooth family over a log point $S_0 = \Spec (Q \to \kk)$. Let $\V = \{V_\alpha\}_\alpha$ be a pre-admissible open cover of $X_0$. We write $\V\,\W^\bullet_0 := \V\,\W^\bullet_{X_0/S_0}$ as well as $\G\C^\bullet_0 := \G\C^\bullet_{X_0/S_0}$ for short.
 \begin{enumerate}[label=(\alph*)]
  \item An \emph{enhanced system of deformations} $\D$ subordinate to $\V$ is a tuple 
 $$(V_{\alpha;A} \to S_A,\: \rho_{\alpha;BB'},\: \psi_{\alpha\beta;A}, \: \G\C^\bullet_{\alpha;A})$$
 where:
 \begin{itemize}
  \item for every $A \in \mathbf{Art}_Q$, the  object $V_{\alpha;A} \to S_A$ is a generically log smooth deformation of $X_0|_\alpha := X_0|_{V_\alpha}$ over $S_A$; the two-sided Gerstenhaber calculus 
  $$\G\C^\bullet_{\alpha;A} \subseteq \V\,\W^\bullet_{V_{\alpha;A}/S_A} =: \V\,\W^\bullet_{\alpha;A}$$ in the context $\mathfrak{Coh}(V_{\alpha;A}/S_A)$ turns $V_{\alpha;A} \to S_A$ into an enhanced generically log smooth family; the restriction map $\V\,\W^\bullet_{\alpha;A} \to \V\,\W^\bullet_{0}|_\alpha$ maps $\G\C^\bullet_{\alpha;A} $ to $\G\C^\bullet_0|_\alpha$, and the induced map after base change is an isomorphism;
  \item for every $B' \to B$ in $\mathbf{Art}_Q$, we have a map $\rho_{\alpha;BB'}: V_{\alpha;B} \to V_{\alpha;B'}$ of generically log smooth deformations; they are compatible with compositions and, by definition, with the maps from $X_0|_\alpha$ that are part of a generically log smooth deformation; the restriction map $\V\,\W^\bullet_{\alpha;B'} \to \V\,\W^\bullet_{\alpha;B}$ maps $\G\C^\bullet_{\alpha;B'}$ to $\G\C_{\alpha;B}^\bullet$ and induces an isomorphism after base change;
  \item on intersections $V_\alpha \cap V_\beta$, we have isomorphisms $\psi_{\alpha\beta;A}: V_{\alpha;A}|_{\alpha\beta} \xrightarrow{\cong} V_{\beta;A}|_{\alpha\beta}$ of generically log smooth deformations; they commute with the restrictions $\rho_{BB'}$, and, by definition, with the maps from $X_0|_{\alpha\beta}$; the induced isomorphism 
  $$\psi_{\alpha\beta;A}^*:\: \V\,\W^\bullet_{\beta;A}|_{\alpha\beta} \to \V\,\W^\bullet_{\alpha;A}|_{\alpha\beta}$$
  maps $\G\C^\bullet_{\beta;A}|_{\alpha\beta}$ isomorphically onto $\G\C^\bullet_{\alpha;A}|_{\alpha\beta}$;
  \item applying Lemma~\ref{geom-auto-gauge-trafo-corr} on the log smooth locus yields a unique element 
  $$o_{\alpha\beta\gamma;A} \in \Gamma(V_\alpha \cap V_\beta \cap V_\gamma \cap U_0, \m_A \cdot \V^{-1}_{\alpha;A})$$
  such that the cocycle $\psi_{\gamma\alpha;A} \circ \psi_{\beta\gamma;A} \circ \psi_{\alpha\beta;A}$ induces $\mathrm{exp}_{-o_{\alpha\beta\gamma;A}}$ on the Gerstenhaber calculus $\V\,\W_{\alpha;A}^\bullet$ on $V_{\alpha} \cap V_\beta \cap V_\gamma \cap U_0$; since this Gerstenhaber calculus is reflexive, this is actually a section on $V_\alpha \cap V_\beta \cap V_\gamma$ of the kernel of $\V^{-1}_{\alpha;A} \to \V^{-1}_{0}|_\alpha$, and $\mathrm{exp}_{-o_{\alpha\beta\gamma;A}}$ is the induced action of the cocycle; we assume that
  \begin{equation*}
   o_{\alpha\beta\gamma;A} \in \Gamma(V_\alpha \cap V_\beta \cap V_\gamma, \m_A \cdot \G^{-1}_{\alpha;A});
  \end{equation*}
  in other words, the cocycles must be inner automorphisms;
  \item the open cover $\V$ is admissible, i.e., $H^n(V_{\alpha_1} \cap ... \cap V_{\alpha_r},\F) = 0$ for $\F = \G^p_0,\A^i_0$, for all $n \geq 1$, and for all indices $\alpha_1,...,\alpha_r$.
 \end{itemize}
 \item An \emph{enhanced generically log smooth deformation of type $\D$} of $f_0: X_0 \to S_0$ is a generically log smooth deformation $f_A: X_A \to S_A$ of $f_0: X_0 \to S_0$ together with isomorphisms $\chi_\alpha: X_A|_\alpha \cong V_{\alpha;A}$ of generically log smooth deformations such that 
 $$\psi_{\beta\alpha;A} \circ \chi_\beta|_{\alpha\beta} \circ \chi_\alpha^{-1}|_{\alpha\beta}: \: V_{\alpha;A}|_{\alpha\beta} \to V_{\alpha;A}|_{\alpha\beta}$$
 is an inner automorphism of the enhanced generically log smooth family $V_{\alpha;A} \to S_A$.
 Transporting $\G\C_{\alpha;A}^\bullet \subseteq \V\,\W^\bullet_{\alpha;A}$ along $\chi_\alpha$ gives rise to a Gerstenhaber calculus $\G\C_{X_A/S_A}^\bullet \subseteq \V\,\W_{X_A/S_A}^\bullet$, which is independent of $\alpha$ on overlaps, hence well-defined, due to our condition on the automorphisms on overlaps. This turns $f_A: X_A \to S_A$ into an enhanced generically log smooth family.
  \item Let $B' \to B$ be a morphism in $\mathbf{Art}_\Lambda$, and let $f': (X_{B'}, \{\chi_\alpha'\}) \to S_{B'}$ and $f: (X_B, \{\chi_\alpha\}) \to S_B$ be two enhanced generically log smooth deformations of type $\D$. Then a \emph{morphism} $c: X_B \to X_{B'}$ is a morphism of generically log smooth deformations such that ${\chi}_\alpha'|_B \circ \chi_\alpha^{-1}$ is an inner automorphism of the enhanced generically log smooth family $V_{\alpha;B} \to S_B$.
  \item Two enhanced generically log smooth deformations $f: X_A \to S_A$ and $f': X_A' \to S_A$ are \emph{equivalent} if there is an isomorphism between them. This means that we have an isomorphism of generically log smooth deformations such that $\chi_\alpha' \circ \chi_\alpha^{-1}$ is an inner automorphism. Equivalence classes of enhanced generically log smooth deformations of type $\D$ form a functor\index{enhanced generically log smooth family!deformation functor} of Artin rings
 $$\mathrm{ELD}_{X_0/S_0}^\D: \mathbf{Art}_Q \to \mathbf{Set}.$$
 \end{enumerate}
\end{defn}

\begin{lemma}\label{enhanced-LD-defo-functor}\note{enhanced-LD-defo-functor}
 This is a deformation functor, i.e., it satisfies $(H_0)$, $(H_1)$, and $(H_2)$.
\end{lemma}
\begin{proof}
 $(H_0)$ is clear. For $(H_1)$, set $B = A' \times_A A''$ and note that 
 $$V_{\alpha;A'} \times_{V_{\alpha;A}} V_{\alpha;A''} = V_{\alpha;B}, \quad \G\C_{\alpha;A'} \times_{\G\C_{\alpha;A}} \G\C_{\alpha;A''} = \G\C_{\alpha;B}$$
 by the methods of Lemma~\ref{defo-type-D-is-defo-functor} and Lemma~\ref{P-def-defo-functor} below. Then $(H_1)$ follows essentially like in the proof of Lemma~\ref{defo-type-D-is-defo-functor}. To see $(H_2)$, we use the method also used in the proof of Lemma~\ref{P-def-defo-functor} below.
\end{proof}

Of course, if $\G\C_0^\bullet = \V\,\W_{0}^\bullet$ and $\G\C_{\alpha;A}^\bullet = \V\,\W_{\alpha;A}^\bullet$, then we have defined nothing new and just obtain generically log smooth deformations of type $\D$.

\begin{rem}
 Although the local models for the deformations do not have the base change property as required in Definition~\ref{sys-of-defo}, the functor $\mathrm{LD}^\D_{X_0/S_0}$ makes sense for an enhanced system of deformations as well, albeit that it is in general only a functor of Artin rings and not a deformation functor since the proof of condition $(H_1)$ in Lemma~\ref{defo-type-D-is-defo-functor} may fail. Since we do not require $H^n(V_\alpha,\V^{-1}_{X_0/S_0}) =0 $ for $n \geq 1$, and since $H^2(V_\alpha,\V^{-1}_{X_0/S_0})$ is no longer an obstruction space for liftings along small extensions in the situation where not all local automorphisms lift, the isomorphisms $\chi_\alpha: X_A|_\alpha \cong V_{\alpha;A}$ must be relaxed to $X_A$ being locally isomorphic to $V_{\alpha;A}$, say on a cover $\U = \{U_i\}_i$ which refines $\V$. Moreover, the change of comparison map between $\chi_i$ and $\chi_j$ is no longer an inner automorphism of $V_{\alpha;A}|_{ij}$ as an enhanced generically log smooth family but only an automorphism of $V_{\alpha;A}|_{ij}$ as a generically log smooth family. In particular, these deformations do not necessarily carry a structure of enhanced generically log smooth family as dictated by $\D$. We obtain a comparison map 
 $$\mathrm{ELD}^\D_{X_0/S_0} \to \mathrm{LD}^\D_{X_0/S_0}$$
 which, in general, may be neither injective nor surjective. Namely, deformations of first order can still be classified as usual, and the map is given by
 $$H^1(X_0,\G^{-1}_{X_0/S_0}) \to H^1(X_0,\V^{-1}_{X_0/S_0})$$
 on the level of tangent spaces.
 Note that $\kk \cdot \eps\, \otimes  \V^{-1}_{X_0/S_0} = \mathrm{ker}(\V^{-1}_{X_\eps/S_\eps} \to \V^{-1}_{X_0/S_0})$ is the sheaf of infinitesimal first order automorphisms of $f_0 \times \kk[\eps]/(\eps^2)$ as a generically log smooth family.
\end{rem}

\par\vspace{\baselineskip}

The following example shows that, in a sense, every generically log smooth deformation to an integral scheme comes from an enhanced system of deformations. Before we go to the example, we need an easy lemma.

\begin{lemma}\label{central-fiber-injection}\note{central-fiber-injection}
 Let $R$ be one of the following rings: $\kk[t]$, $\kk[t]_{(t)}$, $\kk\llbracket t\rrbracket$. Let $f: X \to S = \Spec R$ be a flat morphism whose fibers are pure of dimension $d$ and $S_2$. Let $U \subseteq X$ be an open subset such that $Z = X \setminus U$ has relative codimension $\geq 2$, and let $\F$ be a reflexive sheaf on $X$ which is locally free on $U$. Let $S_0 \subseteq S$ be the closed subscheme defined by $t = 0$, and let $f_0: X_0 \to S_0$ be the base change of $f: X \to S$. Let $\F_0 := \F|_{X_0}$. Then the natural map $\rho: \F_0 \to j_*\F_0|_{U_0}$ is injective.
\end{lemma}
\begin{proof}
 Inside $U$, every stalk $\F_x$ is flat over $\cO_{S,f(x)}$ and hence over $R$. Since $R$ is a Dedekind domain, flatness is equivalent to being torsion-free. In particular, $\F_x$ has no $t$-torsion for $x \in U$, and hence $\F = j_*\F|_U$ has no $t$-torsion at any stalk. Thus, the left map in the sequence 
 $$0 \to \F \xrightarrow{\cdot t} \F \to \F_0 \to 0$$
 is injective, and since $X_0 \subseteq X$ is defined by the ideal $(t)$, the sequence is exact. Let $\theta \in \F$ be such that $\theta|_{X_0}$ is in the kernel of $\rho: \F_0 \to j_*\F_0|_{U_0}$. Then $\theta|_U$ is in the image of the multiplication with $t$. Let $g \in \Gamma(U,\F)$ be such that $tg = \theta|_U$. Since $\F$ is reflexive, our element $g$ can be extended to $X$, and we have $tg = \theta$ for this extension.
\end{proof}

\begin{ex}
 Let $S = \Spec(\NN \to \CC\llbracket t\rrbracket)$, and let $f: X \to S$ be a log Gorenstein generically log smooth family of relative dimension $d$. Let us assume that $f: X \to S$ is vertical, which ensures that the generic fiber $X_\eta$ carries the trivial log structure, hence we have a genuine deformation to a scheme. Let us first study the geometry of $f: X \to S$. For simplicity, we assume that $X$ is connected. Using \cite[055J]{stacks}, also the generic fiber $X_\eta$ is connected. Since $X_\eta$ carries the trivial log structure, it is smooth on $U_\eta$, i.e., regular in codimension $1$. Since it is also Cohen--Macaulay, we find that $X_\eta$ is normal. Because it is connected, it is then integral as well. We may show similarly that the total space $X$ is normal; namely, the generic points of the irreducible components of $X_0$ are regular. Since $X$ is connected, it must be integral. By Lemma~\ref{log-Gor-vertical}, our log Gorenstein assumption ensures that both fibers $X_0$ and $X_\eta$ are Gorenstein. Then also the total space $X$ is Gorenstein.
 
 On $X$, we have a locally Batalin--Vilkovisky Gerstenhaber calculus $\V\,\W_{X/S}^\bullet$ from Proposition~\ref{G-A-construction}. Its pieces are locally free on $U$, so they are flat over $\CC\llbracket t\rrbracket$, and hence they have no $t$-torsion, neither on $U$ nor on $X$. Since $\CC\llbracket t \rrbracket$ is a Dedekind domain, this implies that these sheaves are flat over $S$. Lemma~\ref{central-fiber-injection} shows that we have injections $\V\,\W_{X/S}^P|_{X_0} \subseteq \V\,\W_{X_0/S_0}^P$, and on the fiber $X_\eta$ this holds because $X_\eta \subseteq X$ is an open subset. Then Lemma~\ref{injective-in-fibers} shows that this map is injective for any base change $g: Y \to T$ of $f: X \to S$. Given a, say affine, cover $\V = \{\hat V_\alpha\}_\alpha$ of $X$, we obtain an enhanced system of deformations $\D$ for $f_0: X_0 \to S_0$ subordinate to $\V$ by setting $V_{\alpha;A} := (f \times_S S_A)|_{\hat V_\alpha}$ and $\G\C_{\alpha;A}^\bullet := \V\,\W_{X/S}^\bullet|_{V_{\alpha;A}}$. The restriction maps and comparison isomorphisms come from the global nature of $f \times_S S_A$. Then the cocycles have the required form because they are the identity. Obviously, $f \times_S S_A$ is an enhanced generically log smooth deformation of $f_0: X_0 \to S_0$.
\end{ex}

We discuss this construction explicitly for the family of Example~\ref{base-change-violation-t3-w3}.

\begin{ex}
 Let 
 $$f: X = \Spec \kk[x,y,z,w,u,t]/(xy - t^3 - w^3,\, zw - tu) \to \Spec \kk[t] = S,$$
 be the obvious map $t \mapsto t$. It is smooth over $t \not= 0$, and generically log smooth once we endow it with the compactifying log structure coming from $t = 0$. Outside 
 $$\{0\} = \{x = y = z = w = u = t = 0\},$$
 it is unisingular. Let us denote its Gerstenhaber calculus by $\G\C^\bullet_S := \V\,\W^\bullet_S = (\V^\bullet_{X/S},\W^\bullet_{X/S})$. Then we obtain an enhanced system of deformations with a single affine open $V_\alpha = X_0$ by setting $\G\C^\bullet_A := \G\C^\bullet_S \otimes_S S_A$. The map $f: X \to S$ is flat because $\cO_X$ is torsion-free and $\kk[t]$ is a Dedekind domain. Similarly, $\G_S^p$ and $\A_S^i$ are flat over $S$. The central fiber is reduced and pure of dimension $3$. It is the product of the boundary of the Gorenstein affine toric variety $\Spec \kk[x,y,z,w]/(xy - w^3)$ with $\bAA^1_u$ and hence Gorenstein itself. In particular, $f: X \to S$ is a Gorenstein and hence Cohen--Macaulay morphism. As in Example~\ref{base-change-violation-t3-w3}, we can compute $\G_S^{-1}$ as the relative classical derivations. Then $\G_S^{-3} = (\bigwedge^3 \G_S^{-1})^{\vee\vee} = j_*\G_S^{-3}|_U$ is isomorphic to $\cO_X$, where $U \subseteq X$ is the log smooth locus---this seems to be the hardest part of the computation. Then $\A_S^3 \cong (\G_S^{-1})^\vee \cong \cO_X$ is a line bundle, and this still holds, of course, after any base change, i.e., the family is log Gorenstein, which was to be expected in view of Lemma~\ref{log-Gor-vertical} since $f: X \to S$ is vertical.
 Another computation in Macaulay2 shows that 
 \begin{equation}\label{t3-w3-restriction}
  \G^{-1}_S \otimes_S S_0 = \G_0^{-1} \to j_*\G_0^{-1}|_{U_0}
 \end{equation}
 is injective, showing purity at $G^{-1}$ (in the sense of Definition~\ref{geom-fam-P-alg-def}). Similarly, one can show purity at $G^{-2}$. Of course, we already know this by Lemma~\ref{central-fiber-injection}.  Since the map \eqref{t3-w3-restriction} is not surjective, not every infinitesimal automorphism of generically log smooth families corresponds to a gauge transform for $\G_0^{-1}$. This is possible because $(\G_0^\bullet,\A_0^\bullet)$ is not equal to the reflexive Gerstenhaber calculus $(\V_{X_0/S_0}^\bullet,\W_{X_0/S_0}^\bullet)$ and its variants over infinitesimal deformations who show up in Lemma~\ref{geom-auto-gauge-trafo-corr}.
\end{ex}



\chapter{Families with a vector bundle}\label{gen-log-sm-vector-bundle}\note{gen-log-sm-vector-bundle}

A generically log smooth family with a vector bundle $f: (X,U,\E) \to S$ is, as the name says, a generically log smooth family $f:  (X,U) \to S$ together with a vector bundle $\E$ on $X$.\footnote{We will be mainly interested in the case of line bundles because then the deformation functor is unobstructed in the Calabi--Yau case. However, the basic theory works for vector bundles as well. We restrict to vector bundles as opposed to more general modules because they have an obvious local deformation once we have a deformation of the underlying space.} Let us assume that $\E$ is of constant rank $r \geq 1$. A morphism from a generically log smooth family with a vector bundle $g: (Y,V,\E_Y) \to T$ to another one $f: (X,U,\E_X) \to S$ is a morphism of generically log smooth families---consisting of two maps $c: Y \to X$ and $b: T \to S$ with $V \subseteq c^{-1}(U)$---together with a homomorphism $\E_X \to c_*\E_Y$ of $\cO_X$-modules which induces an isomorphism $c^*\E_X \cong \E_Y$.

 Similarly, an enhanced generically log smooth family with a vector bundle $f: (X,U,\E) \to S$ consists of an enhanced generically log smooth family $f: (X,U) \to S$ and a vector bundle $\E$ on $X$, which we assume to be of some constant rank $r \geq 1$. A morphism is a morphism of enhanced generically log smooth families together with a compatible map $\E_X \to c_*\E_Y$ which induces an isomorphism $c^*\E_X \cong \E_Y$.

\section{Generically log smooth deformations with a vector bundle}

Let $S_0 = \Spec (Q \to \kk)$ for some sharp toric monoid $Q$, and let $f_0: (X_0,U_0,\E_0) \to S_0$ be a generically log smooth family of relative dimension $d$ with a vector bundle $\E_0$ of rank $r$.

\begin{defn}
 Let $A \in \mathbf{Art}_Q$. Then a \emph{generically log smooth deformation with a vector bundle}\index{generically log smooth family!with a vector bundle} is a generically log smooth family with a vector bundle $f_A: (X_A,U_A,\E_A) \to S_A$ together with a morphism $i:(X_0,U_0,\E_0) \to (X_A,U_A,\E_A)$ of generically log smooth families with a vector bundle---especially, we have a map $\E_A \to \E_0$---which induces an isomorphism $f_A \times_{S_A} S_0 \cong f_0$. In particular, $i^{-1}(U_A) = U_0$. Morphisms and isomorphisms of generically log smooth deformations with a vector bundle are defined in analogy with the case of generically log smooth deformations, in particular, they are compatible with the map from the central fiber. Isomorphism classes of generically log smooth deformations with a vector bundle form a functor of Artin rings 
 $$\mathrm{LD}^{gen}_{X_0/S_0}(\E_0): \mathbf{Art}_Q \to \mathbf{Set}.$$
\end{defn}

\begin{lemma}\label{gen-log-sm-pair-defo-functor}\note{gen-log-sm-pair-defo-functor}
 This is a deformation functor, i.e., it satisfies $(H_0)$, $(H_1)$, and $(H_2)$.
\end{lemma}
\begin{proof}
 The condition $(H_0)$ is clear. For $(H_1)$, let $B = A' \times_A A''$, let $S = S_A$, $S' = S_{A'}$, $S'' = S_{A''}$, $T = S_B$, and let $f': (X',\E') \to S'$ and $f'': (X'',\E'') \to S''$ be two deformations which restrict to $f: (X,\E) \to S$ via a choice of the restriction map. For the lift $g: Y \to T$, we take as underlying generically log smooth deformation
 $$Y := (|X_0|,\:\cO_{X'} \times_{\cO_X} \cO_{X''},\: \M_{U'} \times_{\M_U} \M_{U''}).$$
 Then we set $\F := \E' \times_\E \E''$ along the chosen maps $\E' \to \E$ and $\E'' \to \E$; this is an $\cO_Y$-module because the restrictions on the level of the vector bundles are compatible with the restrictions on the level of the structure sheaves, and it is quasi-coherent because the construction commutes with localization (in local sections $f \in \cO_Y$ over affines). Then $\F$ is locally free by \cite[Thm.~2.2]{Ferrand2003}. For $(H_2)$, let $A = \kk$ and $A'' = A_\eps$. Let $g: (Z,\cH) \to T$ be a simultaneous lift of $f': (X',\E') \to S'$ and $f'': (X'',\E'') \to S''$. Let $g: (Y, \F) \to T$ be the lift constructed above. The universal property of the fiber products gives a map $Z \to Y$ of generically log smooth families together with a compatible map $\F \to \cH$ of modules. Since generically log smooth deformations form a deformation functor, $Z \to Y$ is an isomorphism (cf.~\cite[Lemma~9.2]{Kato1996}). But then also $\F \to \cH$ is an isomorphism because both coherent sheaves are flat over $T$, and the pull-back of the map to $S_0$ is an isomorphism.
\end{proof}

\section{Derivations of $f: (X,\E) \to S$}\label{derivations-gen-log-sm-vec-bdl-sec}\note{derivations-gen-log-sm-vec-bdl-sec}

We have seen in Chapter~\ref{inf-auto} that infinitesimal automorphisms of deformations of generically log smooth families are controlled by log derivations. Here, we introduce the analogous notion of derivations for generically log smooth families with a vector bundle $f: (X,\E) \to S$.

\begin{defn}\label{derivation-vec-bdl}\note{derivation-vec-bdl}
 A \emph{derivation}\index{derivation!family with a vector bundle}u of $f: (X,\E) \to S$ on an open subset $V \subseteq X$ is a triple $(D,\Delta,u)$ where $D: \cO_X|_V \to \cO_X|_V$ is a relative derivation of $f: X \to S$, the map 
 $$\Delta: \M_U|_{V \cap U} \to \cO_U|_{V \cap U}$$
 is a sheaf homomorphism which turns $(D,\Delta)$ into a relative log derivation of $f: U \cap V \to S$, and $u: \E|_V \to \E|_V$ is a homomorphism of abelian sheaves such that 
 $$u(a\cdot e) = D(a) \cdot e + a \cdot u(e)$$
 for $a \in \cO_X$ and $e \in \E$. In particular, $u$ is $f^{-1}(\cO_S)$-linear. Derivations of $f: (X,\E) \to S$ form a sheaf $\Theta^1_{X/S}(\E)$ of $\cO_X$-modules.
\end{defn}

At this point, we do not yet know that $\Theta^1_{X/S}(\E)$ is coherent. However, we have an $\cO_X$-linear forgetful map $\Theta^1_{X/S}(\E) \to \Theta^1_{X/S}$ forgetting $u$, and we have an $\cO_X$-linear inclusion 
$$\E nd(\E) \to \Theta^1_{X/S}(\E), \quad A \mapsto (D(a) = 0, \ \Delta(m) = 0,\  u(e) = A(e)).$$
This gives rise to a sequence 
\begin{equation}\label{AE}
 0 \to \E nd(\E) \to \Theta^1_{X/S}(\E) \to \Theta^1_{X/S} \to 0 \tag{AE}
\end{equation}
of $\cO_X$-modules, which we call the \emph{Atiyah extension}\index{Atiyah extension} in analogy with the classical deformation theory of smooth manifolds with a line bundle.

\begin{lemma}\label{Atiyah-ext-lemma}\note{Atiyah-ext-lemma}
 The sequence \eqref{AE} is exact and locally split. In particular, $\Theta^1_{X/S}(\E)$ is a coherent reflexive sheaf, locally free of rank $d + r^2$ on $U$.
\end{lemma}
\begin{proof}
 As is easily checked, this sequence is exact on the left and in the middle for the sections over any open subset. If $V \subseteq X$ is an open on which $\cL|_V \cong \cO_X|_V$ via the local basis $e_1,...,e_r$, and if $(D,\Delta) \in \Gamma(V,\Theta^1_{X/S})$, then we form a splitting by setting $u(\sum_i a_ie_i) := \sum_iD(a_i)e_i$.
\end{proof}

We turn $\Theta^1_{X/S}(\E)$ into a sheaf of Lie algebras by setting 
$$[(D_1,\Delta_1,u_1),(D_2,\Delta_2,u_2)] := (D_1 \circ D_2 - D_2 \circ D_1, \: D_1 \circ \Delta_2 - D_2 \circ \Delta_1, \: u_1 \circ u_2 - u_2 \circ u_1).$$
This is a natural extension of the Lie algebra structure on $\Theta^1_{X/S}$; the Lie bracket is $f^{-1}(\cO_S)$-linear but not $\cO_X$-linear. The sheaf $\Theta^1_{X/S}(\E)$ also fits into a Lie--Rinehart pair 
$$\cL\R\cP^\bullet_{X/S}(\E)$$
as follows: We set 
$$\cL\R\cP^F = \cO_X, \quad \cL\R\cP^T = \Theta^1_{X/S}(\E), \quad \cL\R\cP^E = \E$$ as modules over $\cO_X = \cL\R\cP^F$. Then we set $\nabla^F_{(D,\Delta,u)}(a) := - D(a)$ and $ \nabla^E_{(D,\Delta,u)}(e) := - u(e)$ as well as
$$\nabla^T_{(D_1,\Delta_1,u_1)}(D_2,\Delta_2,u_2) := -[(D_1,\Delta_1,u_1),(D_2,\Delta_2,u_2)],$$
where the $(-)$-sign is in accordance with our convention of taking the negative of the Schouten--Nijenhuis bracket.

\begin{lemma}\label{LRP-gen-log-vector-bundle}\note{LRP-gen-log-vector-bundle}
 We have:
 \begin{enumerate}[label=\emph{(\alph*)}]
  \item If $\Theta^1_{X/S}$ is flat over $S$, then $\cL\R\cP^\bullet_{X/S}(\E)$ is a Lie--Rinehart pair in $\mathfrak{Coh}(X/S)$.
  \item If $f: X \to S$ has the base change property and is log Gorenstein\footnote{In this case, $\Theta^1_{X/S}$ is flat over $S$ by Proposition~\ref{G-A-construction}.}, then the formation of $\cL\R\cP^\bullet_{X/S}(\E)$ commutes with base change.
  \item $\cL\R\cP^\bullet_{X/S}(\E)$ is $Z$-faithful for $Z = \{F,E\}$.
 \end{enumerate}
\end{lemma}
\begin{proof}
 Part (a) is mostly computational. Flatness is required because of our definition of the context $\mathfrak{Coh}(X/S)$.
 For part (b), let $b: T \to S$, let $c: Y = X \times_S T \to X$, and let $g: (Y,c^*\E) \to T$ be the generically log smooth family with a vector bundle obtained by base change. Then we have obvious maps $\cO_X \to c_*\cO_Y$ and $\E \to c_*c^*\E$, which induce isomorphisms on $Y$. To construct the map 
 $$\Theta^1_{X/S}(\E) \to c_*\Theta^1_{Y/T}(c^*\E),$$
 let $\theta = (D,\Delta,u) \in \Theta^1_{X/S}(\E)$. Then $(D,\Delta)$ corresponds to a homomorphism $k: \W^1_{X/S} \to \cO_X$, which can be pulled back to $\W^1_{Y/T} \to \cO_Y$. The map $u: \E \to \E$ is a differential operator of order $1$, so there is a pull-back $c^*\E \to c^*\E$ as well. This gives the desired section of $c_*\Theta^1_{Y/T}(c^*\E)$. Since this map is compatible with the two Atiyah extensions and $c^*\Theta^1_{X/S} \cong \Theta^1_{Y/T}$ by the base change property and log Gorenstein assumption, we obtain an isomorphism on $Y$. It remains to show that the maps $\nabla^P$ on $\cL\R\cP^\bullet_{Y/T}(c^*\E)$ are the ones induced via pull-back of multilinear differential operators from $\cL\R\cP^\bullet_{X/S}(\E)$. We can assume that $X$, $S$, $T$, $Y$ are all affine. Then, on global sections of $Y$, the claim follows from the compatibility of the maps on $X$ and $\cO_T$-(multi-)linearity. On open subsets of $Y$, the claim follows from the unique extendability of multilinear differential operators to localizations.
 For part (c), let 
 $$\theta = (D,\Delta,u) \in \Gamma(V,\Theta^1_{X/S}(\E))$$
 be such that $D(a) = 0$ for all $a \in \Gamma(V',\cO_X)$ and $u(e) = 0$ for all $e \in \Gamma(V',\E)$ for $V' \subseteq V$. Since $\Theta^1_{X/S}$ is strictly faithful, we have $(D,\Delta) = 0$. Since $u = 0$, we find $\theta = 0$, so $\cL\R\cP^\bullet_{X/S}(\E)$ is $Z$-faithful.
\end{proof}

\section{Infinitesimal automorphisms}

Let $f_0: (X_0,\E_0) \to S_0$ be a generically log smooth family with a vector bundle. Let $B' \to B$ be a surjection in $\mathbf{Art}_Q$ with kernel $I \subset B'$, and let $f: (X,\E) \to S = S_B$ and $f': (X',\E') \to S' = S_{B'}$ be two generically log smooth deformations with a vector bundle. Assume we have a morphism $f \to f'$. We denote the sheaf of automorphisms of $f': (X',\E') \to S'$ over $f: (X,\E) \to S$ by $\A ut_{(X',\E')/(X,\E)}$. A section over $V \subseteq X'$ consists of an automorphism $\phi: \cO_{X'} \to \cO_{X'}$ of sheaves of rings on $V$, a map $\Phi: \M_{X'} \to \M_{X'}$ on $V \cap U'$ that yields an automorphism of log schemes, and a map $\psi: \E' \to \E'$ on $V$ which is compatible with $\phi$. As in Lemma~\ref{sheaf-of-autom}, for every open $V \subseteq X_0$, the restriction 
$$\Gamma(V,\A ut_{(X',\E')/(X,\E)}) \to \Gamma(V \cap U_0, \A ut_{(X',\E')/(X,\E)})$$
is an isomorphism. Let 
$$\Theta^{1}_{X'/S'}(\E',I)$$
be the sheaf of those derivations $(D,\Delta,u)$ of $f: (X',\E') \to S'$ with $D$ and $\Delta$ mapping into $I \cdot \cO_{X'}$ and $u$ mapping into $I \cdot \E'$.

\begin{lemma}\label{vector-bdl-autom}\note{vector-bdl-autom}
 We have two inverse isomorphisms 
 \[
 \xymatrix{
  \Theta^1_{X'/S'}(\E',I)
  \ar@<0.5ex>[r]^-{\mathrm{Exp}} &
  {\A ut}_{(X',\E')/(X,\E)} \ar@<0.5ex>[l]^-{\mathrm{Log}} \\
 }
\]
 of sheaves of groups, where the left hand side is endowed with the Baker--Campbell--Hausdorff product induced from the original\footnote{As in the discussion of automorphisms of generically log smooth families above, we use the original (not negative) Lie bracket on $\Theta^1(X/S,\E)$ when describing automorphisms.} Lie bracket $[-,-] = -\nabla^T$. If $\theta = (D,\Delta,u) \in \Theta^1_{X'/S'}(\E',I)$, then $(\phi,\Phi,\psi) = \mathrm{Exp}(\theta)$ is given by the formulae in Chapter~\ref{inf-auto} for $\phi$, $\Phi$, and by 
 $$\psi: \enspace \E' \to \E', \quad e \mapsto \sum_{n = 0}^\infty \frac{u^n(e)}{n!} = e + u(e) + \frac{1}{2}u^2(e) + ...;$$
 if $(\phi,\Phi,\psi) \in \A ut_{(X',\E')/(X,\E)}$, then $(D,\Delta,u) =  \mathrm{Log}(\phi,\Phi,\psi)$ is given by the formulae in Chapter~\ref{inf-auto}, and by 
 $$u: \enspace \E' \to \E', \quad e \mapsto \sum_{n = 1}^\infty \frac{(-1)^{n - 1}[\psi - \mathrm{Id}]^n(e)}{n}.$$
\end{lemma}
\begin{proof}
 First, let us work on $U_0$. 
 Given $(D,\Delta,u) \in \Theta^1_I(X'/S',\E')$, a rather easy direct computation shows that $(\phi,\Phi,\psi) = \mathrm{Exp}(D,\Delta,u)$ is an automorphism. For the converse, the hardest part is to show that $u = \mathrm{Log}(\psi)$ satisfies $u(ae) = D(a)e + au(e)$. This can be achieved by an easy adaptation of the corresponding part in the proof of \cite[Lemma~2.3]{Felten2022}. Because $\mathrm{Exp}(u)$ and $\mathrm{Log}(\psi)$ are given by the formulae for the exponential respective the logarithm, they are inverse to each other. Thus, we have two isomorphisms of sheaves of \emph{sets} on $U_0$. Now let $Z_0 = X_0 \setminus U_0$. Since $\cO_{X'}$ and $\E'$ are $Z_0$-closed, i.e., $j_*\cO_{U'} = \cO_{X'}$ and $j_*\E'|_{U'} = \E'$, it is easy to show that the automorphism sheaf is $Z_0$-closed. Similarly, $\Theta^1_{X'/S'}(\E',I)$ is $Z_0$-closed; in fact, both $I \cdot \cO_{X'}$ and $I \cdot \E'$ are $Z_0$-closed as well because they are the kernels of $\cO_{X'} \to \cO_X$ respective $\E' \to \E$, and both $\cO_{X'}$ and $\E'$ are flat over $S'$. Thus, we have two isomorphisms of sheaves of sets on $X'$. The proof that $\mathrm{Exp}$ is a group homomorphism for the Baker--Campbell--Hausdorff formula is similar to the proof of Lemma~\ref{gauge-trafo-prop}. Then its inverse $\mathrm{Log}$ is a group homomorphism as well.
\end{proof}

\begin{rem}
 If $\Theta^1_{X'/S'}$ is flat over $S'$, then the canonical map 
 $$I \cdot \Theta^1_{X'/S'}(\E') \to \Theta^1_{X'/S'}(\E',I)$$
 is an isomorphism. Namely, the right hand side is the kernel of $\Theta^1_{X'/S'}(\E') \to \Theta^1_{X/S}(\E)$, which can be computed as the left hand side if $\Theta^1_{X'/S'}$, hence $\Theta^1_{X'/S'}(\E')$, is flat over $S'$.
\end{rem}

\subsection{Automorphisms of the Lie--Rinehart pair}

Let us assume that $f_0: (X_0,\E_0) \to S_0$ is log Gorenstein, and that every deformation which we consider has the base change property. Let $f': (X',\E') \to S'$ be a deformation over $S'$ with base change $f: (X,\E) \to S$, and let $\varphi = (\phi,\Phi,\psi)$ be an automorphism of $f'$ over $f$. We wish to construct an induced action on $\cL\R\cP^\bullet_{X'/S'}(\E')$. On $\cL\R\cP^F_{X'/S'}(\E') = \cO_{X'}$ and $\cL\R\cP^E_{X'/S'}(\E') = \E'$, the action should be given by $\phi$ respective $\psi$. As we have seen above, we have $\phi = \mathrm{exp}_{-\theta}$ and $\psi = \mathrm{exp}_{-\theta}$ for $\theta = \mathrm{Log}(\varphi)$. Then we just \emph{define} the induced action of $\varphi$ on $\cL\R\cP^T_{X'/S'}(\E')$ by $\mathrm{exp}_{-\theta}$. Namely, this is automatically a gauge transform and hence an automorphism of Lie--Rinehart pairs, it induces automatically the identity on $f: (X,\E) \to S$, and it is automatically compatible with $T\varphi^*$ along the projection in the Atiyah extension \eqref{AE} since this projection is a morphism of Lie algebras and $T\varphi^*$ is equal to $\mathrm{exp}_{-\theta}$ in $\cL\R^T_{X'/S'}$. This construction yields a one-to-one correspondence between automorphisms in $\A ut_{(X',\E')/(X,\E)}$ and gauge transforms of $\cL\R\cP^\bullet_{X'/S'}(\E')$ induced from $\theta \in I \cdot \cL\R\cP^T_{X'/S'}(\E')$ because $\cL\R\cP^\bullet_{X'/S'}(\E')$ is $Z$-faithful by Lemma~\ref{LRP-gen-log-vector-bundle}.

\section{Systems of deformations}

Similar to the case of plain generically log smooth deformations, we wish to devise a method to define a subfunctor of $\mathrm{LD}^{gen}_{X_0/S_0}(\E_0)$ which classifies specific generically log smooth deformations with a vector bundle which are locally unique. Let us fix an open cover of $X_0$ which is weakly admissible for $f_0: X_0 \to S_0$ in the sense of Definition~\ref{admissible-def}. This specifies how deformations of the underlying generically log smooth family should look like locally. If we had allowed arbitrary coherent sheaves $\E_0$, then we would also need to specify the local deformations of $\E_0$, but since we insist that $\E_0$ is a vector bundle, we just require the deformation to be a vector bundle as well. In fact, every flat deformation $\E$ of $\E_0$, i.e., $f: X \to S$ is flat and $\E$ is flat over $S$, is locally free. Namely, a local trivialization can be lifted to $\E$.

\begin{defn}
 Let $\V = \{V_\alpha\}_\alpha$ be an open cover of $X_0$ by finitely many open immersions. Then it is \emph{weakly $\E_0$-admissible} if it is weakly admissible and $\E_0|_\alpha \cong \cO_{X_0}^{\oplus r}|_\alpha$ for every $\alpha$.
\end{defn}

In this case, a flat deformation $\E$ of $\E_0$ is not only locally trivial but trivial on each $V_\alpha$. Namely, since $H^1(V_\alpha,\cO_{X_0}|_\alpha) = 0$, we can lift the trivialization of $\E_0|_\alpha$ to $\E|_\alpha$.

\begin{defn}
 Let $f_0: X_0 \to S_0$ be a log Gorenstein generically log smooth family of relative dimension $d$, and let $\E_0$ be a vector bundle of rank $r$ on $X_0$. Let $\V = \{V_\alpha\}_\alpha$ be a weakly $\E_0$-admissible open cover of $X_0$. Then a \emph{system of deformations} $\D$ for $f_0: (X_0,\E_0) \to S_0$ subordinate to $\V$ is a system of deformations for $f_0: X_0 \to S_0$ subordinate to $\V$. A generically log smooth deformation with a vector bundle $f: (X,\E) \to S$ is \emph{of type $\D$} if $f: X \to S$ is of type $\D$. Isomorphism classes of these deformations form a functor of Artin rings 
 $$\mathrm{LD}^\D_{X_0/S_0}(\E_0): \mathbf{Art}_Q \to \mathbf{Set}.$$
\end{defn}

Deformations of type $\D$ are now such deformations which are, on $V_\alpha$, isomorphic to $(V_{\alpha;A},\cO_{V_{\alpha;A}}^{\oplus r}) \to S_A$.

\begin{lemma}
 This is a deformation functor, i.e., it satisfies $(H_0)$, $(H_1)$, and $(H_2)$.
\end{lemma}
\begin{proof}
 This is a combination of the proofs of Lemma~\ref{defo-type-D-is-defo-functor} and Lemma~\ref{gen-log-sm-pair-defo-functor}.
\end{proof}

\section{Enhanced systems of deformations}

Now we present the version for enhanced generically log smooth families. As usual, we restrict to the simpler torsionless case.

\begin{defn}\label{enhanced-sys-defo-vector-bdl}\note{enhanced-sys-defo-vector-bdl}
 Let $f_0: X_0 \to S_0$ be a torsionless enhanced generically log smooth family of relative dimension $d$, and let $\E_0$ be a vector bundle on $X_0$ of constant rank $r \geq 1$. Let $\V = \{V_\alpha\}_\alpha$ be a weakly $\E_0$-admissible open cover of $X_0$. Then an \emph{enhanced system of deformations}\index{system of deformations!enhanced} $\D$ for $f_0: (X_0,\E_0) \to S_0$ subordinate to $\V$ is an enhanced system of deformations $\D$ subordinate to $\V$. Enhanced deformations of type $\D$ are given by enhanced generically log smooth deformations $f: X_A \to S_A$ of type $\D$, i.e., generically log smooth deformations $f_A: X_A \to S_A$ together with isomorphisms $\chi_\alpha: X_A|_\alpha \cong V_{\alpha;A}$, together with a vector bundle $\E_A$ on $X_A$ of constant rank $r$. Two enhanced generically log smooth deformations with a vector bundle \emph{of type} $\D$ are \emph{equivalent} if there is an isomorphism between them together with a compatible isomorphism of the vector bundles. Equivalence classes of enhanced deformations form a deformation functor 
 $$\mathrm{ELD}^\D_{X_0/S_0}(\E_0): \enspace \mathbf{Art}_Q \to \mathbf{Set}.$$
\end{defn}

Since we require $H^1(V_\alpha,\cO_{X_0}) = 0$, we have $\E_A|_\alpha \cong \cO_{V_\alpha;A}^{\oplus r}$.

Let $f: (X_A,\E_A) \to S_A$ be an enhanced generically log smooth deformation with a vector bundle of type $\D$. Through the comparison isomorphisms $\chi_\alpha: X_A|_\alpha \cong V_{\alpha;A}$, we have a distinguished subsheaf $\Gamma^{1}_{X_A/S_A} \subseteq \Theta^{1}_{X_A/S_A}$. We obtain an Atiyah extension 
$$0 \to \E nd(\E_A) \to \Gamma^{1}_{X_A/S_A}(\E_A) \to \Gamma^{1}_{X_A/S_A} \to 0$$
by considering triples $(D,\Delta,u)$ with $(D,\Delta) \in \Gamma^{1}_{X_A/S_A}$ and an additive map $u: \E_A \to \E_A$ satisfying $u(a \cdot e) = D(a) \cdot e + a \cdot u(e)$. A Lie bracket is given on $\Gamma^{1}_{X_A/S_A}(\E_A)$ with the same formula as for $\Theta^1_{X_A/S_A}(\E_A)$.
The kernel of 
$$\Gamma^{1}_{X_{B'}/S_{B'}}(\E_{B'}) \to \Gamma^{1}_{X_B/S_B}(\E_B)$$
is given by $I \cdot \Gamma^{1}_{X_{B'}/S_{B'}}(\E_{B'})$ (since the sheaves are flat over the base), and every $\theta = (D,\Delta,u) \in I \cdot \Gamma^{1}_{X_{B'}/S_{B'}}(\E_{B'})$ induces an (inner) automorphism $\mathrm{Exp}(\theta)$ of the enhanced deformation which acts as $\mathrm{exp}_{-(D,\Delta)}$ on the two-sided Gerstenhaber calculus, and by 
$$\psi(e) = \sum_{n = 0}^\infty \frac{u^n(e)}{n!} = e + u(e) + \frac{1}{2}u^2(e) + ...$$
on the vector bundle $\E_{B'}$. Every automorphism of the enhanced deformation is of this form because it induces an inner automorphism of the underlying enhanced generically log smooth family by definition.

We form a Lie--Rinehart pair consisting of $\F = \cO_{X_A}$, $\T = \Gamma^{1}_{X_A/S_A}(\E_A)$, and $\E = \E_A$ by endowing it with the negative brackets, i.e., $\nabla^F_\theta(a) = -D(a)$, $\nabla^T_\theta(\xi) = - [\theta,\xi]_{Lie}$, and $\nabla^E_\theta(e) = -u(e)$; we denote it by $\cL\R\cP^\bullet_{X_A/S_A}(\E_A)$. 

We shall occasionally also write $\G^{-1}_{X/S}(\E)$ for $\Gamma^1_{X/S}(\E)$ when considered endowed with the negative Lie bracket and the negative action on $\E$.

\section{From a vector bundle to its determinant}

In the above situation, we can form the determinant $\mathrm{det}(\E_A)$ of $\E_A$, and we can forget $\E_A$. This gives rise to a diagram 
\[
 \xymatrix{
  D_r := \mathrm{ELD}^\D_{X_0/S_0}(\E_0) \ar[rr]^\pi \ar[dr]^\rho & & D_0 := \mathrm{ELD}_{X_0/S_0}^\D \\
  & D_1 := \mathrm{ELD}_{X_0/S_0}^\D(\mathrm{det}(\E_0)) \ar[ur]^\sigma & \\
 }
\]
of deformation functors. In this section, we examine these maps briefly.

The main results of this article show that both $D_0$ and $D_1$ are smooth deformation functors if $\A^d_0 \cong \cO_{X_0}$ and certain technical conditions are met. Since the obstructions to extending a line bundle $\cL_B$ from $X_B$ to a given thickening $X_{B'}$ lie in $H^2(X_0,\cO_{X_0}) \otimes_\kk I$, the map $\sigma: D_1 \to D_0$ is smooth if $H^2(X_0,\cO_{X_0}) = 0$. However, if $H^2(X_0,\cO_{X_0}) \not= 0$, then the map is not always smooth.

\begin{ex}
 This is \cite[Exer.~6.7(c)-(d)]{Hartshorne2010}. Let $X_0 \subset \PP^3$ be the smooth quartic surface defined by $f_0 = x^4 + y^4 + xz^3 + yw^3$, and let $\cL_0$ be the line bundle associated with the divisor $Y_0 = \{x = y = 0\}$. Since $X_0$ is a K3 surface, we have $H^2(X_0,\cO_{X_0}) \cong \kk \not= 0$. Consider the deformation $X$ of $X_0$ over $\kk[t]/(t^2)$ defined by $f = f_0 + tz^2w^2$. Then there is no line bundle $\cL$ on $X$ with $\cL|_0 \cong \cL_0$. Thus, the map $\sigma: D_1 \to D_0$ is not a smooth morphism of deformation functors.
\end{ex}

Similarly, if $H^2(X_0,\E nd(\E_0)) = \mathrm{Ext}^2(\E_0,\E_0) = 0$, then the map $\pi: D_r \to D_0$ is smooth by \cite[Thm.~7.1]{Hartshorne2010}.

\begin{ex}\label{R-Thomas-ex-1}\note{R-Thomas-ex-1}
 This example is due to R.~Thomas. Let $X_0 \subseteq \PP^2 \times \PP^2$ be the $(3,3)$-divisor of \cite{Thomas1999},\footnote{The published version of this article contains a mistake which is corrected in a newer version on the arXiv.} and let $\E_0$ be the vector bundle of rank $2$ of \cite[Thm.~1.1]{Thomas1999}. Let $X_1 = X_0 \times \Spec \CC[t]/(t^2)$, and let $X_2 = X_0 \times \Spec \CC[t]/(t^3)$. Let $\E_1$ be a lift of $\E_0$ to $X_1$ which is not trivial. Then $\E_1$ cannot be lifted to a vector bundle $\E_2$ on $X_2$. Thus, the map $\pi: D_2 \to D_0$ of deformation functors is not smooth.
\end{ex}
\begin{rem}
 It is currently an open question if, in this example, the deformation functor $D_2$ itself is smooth. Related to this, if $D_2$ is smooth here, it is open to find some Calabi--Yau space (say smooth projective) $X_0$ together with a vector bundle $\E_0$ of rank $r \geq 2$ such that $D_r$ is not smooth.
\end{rem}

The trace map in linear algebra gives rise to an $\cO_{X_0}$-linear trace map 
$$\mathrm{tr}: \E nd(\E_0) \to \cO_{X_0}.$$
It satisfies $\mathrm{tr}(\mathrm{Id}) = r$, i.e., the composition with the diagonal embedding $L: \cO_{X_0} \to \E nd(\E_0), \: \lambda \mapsto (e \mapsto \lambda \cdot e)$, is the multiplication with $r$. Since $\mathrm{char}(\kk) = 0$, the trace map is split by $\frac{1}{r}L$. Thus, the induced map 
$$H^2(\mathrm{tr}): \mathrm{Ext}^2(\E_0,\E_0) \to H^2(X_0,\cO_{X_0})$$
is split and surjective; we denote its kernel by $\mathrm{Ext}^2(\E_0,\E_0)_0$.

\begin{prop}\label{rho-smoothness-criterion}\note{rho-smoothness-criterion}
 Assume that $\mathrm{Ext}^2(\E_0,\E_0)_0 = 0$. Then $\rho: D_r \to D_1$ is a smooth map of deformation functors.
\end{prop}
\begin{proof}
 Let $B' \to B$ be a small extension in $\mathbf{Art}_\Lambda$ with kernel $I \subset B'$, let $f: (X,\E) \to S_{B}$ be an enhanced deformation of generically log smooth families with a vector bundle, let $\cL = \mathrm{det}(\E)$, and let $f': (X',\cL') \to S_{B'}$ be a lift to $S_{B'}$. We have to show that there is a vector bundle $\E'$ on $X'$ which lifts $\E$ and satisfies $\mathrm{det}(\E') \cong \cL'$ in a compatible way. By \cite[Thm.~7.1]{Hartshorne2010}, the obstruction for lifting $\E$ lies in $H^2(X_0,I \otimes_\kk \E nd(\E_0))$. Similarly, the obstruction for lifting $\cL$ to $X'$ lies in $H^2(X_0,I \otimes_\kk \E nd(\cL_0))$, where $\cL_0 := \mathrm{det}(\E_0)$. Given some local lift $\E'$ of $\E$ and a local automorphism of $\E'$ over $\E$ determined by $\theta \in I \otimes_\kk \E nd(\E_0)$, the induced local automorphism of $\cL' = \mathrm{det}(\E')$ is determined by $\mathrm{tr}(\theta) \in I \otimes_\kk \E nd(\cL_0)$. Thus, if $o \in I \otimes_\kk \mathrm{Ext}^2(\E_0,\E_0)$ is the obstruction to lifting $\E$ to $X'$, then $H^2(\mathrm{tr})(o) \in I \otimes_\kk H^2(X_0,\cO_{X_0})$ is the obstruction to lifting $\cL$ to $X'$. However, we have $H^2(\mathrm{tr})(o) = 0$ in our situation, and $H^2(\mathrm{tr})$ is injective by assumption. Thus, $o = 0$, and a lift $\E'$ exists.
\end{proof}

\begin{ex}
 We continue Example~\ref{R-Thomas-ex-1}. Since $H^2(X_0,\cO_{X_0}) = 0$, the map $\sigma: D_1 \to D_0$ is smooth. Thus, given the non-trivial $\E_1$, we can find a lift $\cL_2$ of $\cL_1 := \mathrm{det}(\E_1)$ to $X_2$. However, a lift $\E_2$ still does not exist, so $\rho: D_2 \to D_1$ is not smooth.
\end{ex}



\chapter{Geometric families of $\cP$-algebras}\label{geom-fam-P-alg-sec}\note{geom-fam-P-alg-sec}

When we start with a log Gorenstein generically log smooth family $f_0: X_0 \to S_0$, we can form a sheaf of Gerstenhaber calculi $(\V^\bullet_{X_0/S_0},\W^\bullet_{X_0/S_0})$ on $X_0$ by Proposition~\ref{G-A-construction}. When $f_A: X_A \to S_A$ is a deformation of $f_0$ with the base change property, then its sheaf of Gerstenhaber calculi is a deformation of $(\V^\bullet_{X_0/S_0},\W^\bullet_{X_0/S_0})$. In this chapter, we define and study sheaves of $\cP$-algebras for algebraic structures $\cP$ carrying a Cartan structure, as well as their infinitesimal deformations. When we want to capture as much information as possible, we can work with two-sided Gerstenhaber calculi. When we want only a minimal structure that is enough for basic deformation theory, we can work with Lie--Rinehart algebras. In the case of deformations with a vector bundle, we work with Lie--Rinehart pairs.

While we are currently mostly interested in families which come from log geometry, this theory is of independent interest as it can also be applied beyond log geometry. We have weakened the Cohen--Macaulay assumption, which is usually satisfied in log geometry, to an $S_2$-assumption to allow potential applications in the deformation theory of normal singularities.

\begin{defn}\label{geom-fam-P-alg-def}\note{geom-fam-P-alg-def}
 Let $\cP$ be an algebraic structure carrying a Cartan structure. Let $S$ be a Noetherian scheme. A \emph{geometric family of $\cP$-algebras}\index{geometric family of $\cP$-algebras} of dimension $d$ is a tuple
 $$(f: X \to S,\: U^{sm} \subseteq  U \subseteq X, \: \E^\bullet)$$
 where:
 \begin{enumerate}[label=(\alph*)]
  \item 
  \begin{itemize}
    \item $f: X \to S$ is a separated and flat morphism of finite type of Noetherian schemes whose geometric fibers are reduced, have Serre's property $(S_2)$, and are pure of dimension $d$;
  \item $U \subseteq X$ is a Zariski open subset satisfying \eqref{CC};
  \item $U^{sm} \subseteq U$ is a Zariski open subset such that $U^{sm} \to S$ is smooth, and $U^{sm}_s \subseteq X_s$ is scheme-theoretically dense in the fiber for every $s \in S$;
  \end{itemize}
  \item 
   \begin{itemize}
    \item $\E^\bullet$ is a Cartanian $\cP$-algebra in $\mathfrak{Coh}(X/S)$;
   \end{itemize}
   \item
    \begin{itemize}
     \item the map $\cO_X \to \F := \E^F, \: 1 \mapsto 1_F,$ is an isomorphism of $\cO_X$-modules;
     \item for a morphism $b: T \to S$ of Noetherian schemes, let $c: Y := X \times_S T \to X$ be the induced map via pullback; then for all such $b: T \to S$, we assume that the $\cP$-algebra $c^*\E^\bullet$ in $\mathfrak{Coh}(Y/T)$ is $Z$-faithful (as defined in Definition~\ref{Cartan-faithful-def}).
    \end{itemize}
 \end{enumerate}
 When no confusion is likely, we write a geometric family of $\cP$-algebras as $f: (X,\E^\bullet) \to S$ for short.
 A \emph{morphism} of geometric families of $\cP$-algebras from $g: (Y,V,V^{sm},\E_Y^\bullet) \to T$ to $f: (X,U,U^{sm},\E_X^\bullet) \to S$ consists of two morphisms $b: T \to S$ and $c: Y \to X$ of schemes, satisfying $c^{-1}(U) = V$ and $c^{-1}(U^{sm}) = V^{sm}$ and giving a Cartesian diagram, and maps 
 $$\E_X^P \to c_*\E_Y^P$$
 of $\cO_X$-modules that induce isomorphisms $c^*\E_X^P \cong \E_Y^P$  and are compatible with all constants in $\cP(P)$ and all operations in $\cP(P_1,...,P_n;Q)$.
 A geometric family of $\cP$-algebras $f:(X, \E^\bullet) \to S$ is called:
 \begin{itemize}
  \item \emph{strictly faithful}\index{geometric family of $\cP$-algebras!strictly faithful} if $c^*\E^\bullet$ is strictly faithful for all $b: T \to S$;
  \item \emph{tame at $P \in D(\cP)$}\index{geometric family of $\cP$-algebras!tame} if $\E^P|_U$ is locally free;
  \item \emph{pure at $P \in D(\cP)$}\index{geometric family of $\cP$-algebras!pure} if $c^*\E^P \to j_*(c^*\E^P)|_V$ is injective after every base change $b: T \to S$;
  \item \emph{closed at $P \in D(\cP)$}\index{geometric family of $\cP$-algebras!closed} if $c^*\E^P \to j_*(c^*\E^P)|_V$ is an isomorphism after every base change $b: T \to S$;\footnote{The names \emph{pure} and \emph{closed} refer to the terminology in \cite[Defn.~5.9.9]{EGAIV-2}.}
  \item \emph{reflexive at $P \in D(\cP)$}\index{geometric family of $\cP$-algebras!reflexive} if it is tame at $P$ and closed at $P$;
  \item \emph{locally free at $P \in D(\cP)$}\index{geometric family of $\cP$-algebras!locally free} if $\E^P$ is a locally free sheaf.
 \end{itemize}

\end{defn}

\begin{rem}
 Several remarks are in order. 
 \begin{enumerate}[label=(\arabic*)]
  \item  Concretely, $\E^P$ consists of the following structures: For each $P \in D(\cP)$, we have a coherent sheaf $\E^P$ of $\cO_X$-modules, flat over $S$; for each $\gamma \in \cP(P)$, we have a global section $\gamma \in \E^P$; for each $\mu \in \cP(P;Q)$, we have an $f^{-1}(\cO_S)$-linear map $\mu: \E^P \to \E^Q$; for each $\mu \in \cP(P,Q;R)$, we have an $f^{-1}(\cO_S)$-bilinear map $\mu: \E^P \times \E^Q \to \E^R$. They must satisfy all relations required by the algebraic structure $\cP$---we do not just assume that $\E^P$ is a pre-algebra.
  \item We write $\F := \E^F$ and $\T := \E^T$. Since $\cO_X \to \F$ is an isomorphism and $\ast^P: \F \times \E^P \to \E^P$ is $\cO_X$-bilinear, the $\F$-module structure coincides with the $\cO_X$-module structure on every $\E^P$, and in particular, the product on $\F$ coincides with the product on $\cO_X$. Thus, every $\mu \in \cP_N(P_1,...,P_n;Q)$ is a differential operator of order $N$ with respect to $\cO_X/f^{-1}(\cO_S)$.
  \item Let $b: T \to S$ be a morphism of Noetherian schemes with base change $c: Y \to X$. Then each $c^*\E^P$ is a coherent sheaf on $Y$, flat over $T$. By Proposition~\ref{multilin-diff-op-shvs}, every multilinear differential operator $\mu: \E^{P_1} \times ... \times \E^{P_n} \to \E^Q$ pulls back to a well-defined multilinear differential operator $c^*\mu: c^*\E^{P_1} \times ... \times c^*\E^{P_n} \to c^*\E^Q$. Since this construction preserves compositions and the identity on objects $\E^P$, and is compatible with constants, the relations of $\cP$ hold in $c^*\E^\bullet$ as well. Thus, $c^*\E^\bullet$ is a Cartanian $\cP$-algebra in $\mathfrak{Coh}(Y/T)$.
  \item Since $c^*\E^\bullet$ is a Cartanian $\cP$-algebra in $\mathfrak{Coh}(Y/T)$, the requirement that $c^*\E^\bullet$ should be $Z$-faithful is defined. A fortiori, if $\E^\bullet$ remains $Z$-faithful after any base change, so does $c^*\E^\bullet$. Thus, if $(f: X \to S, \: U \subseteq X,\: \E^\bullet)$ is a geometric family of $\cP$-algebras, so is $(g: Y \to T, \: V \subseteq Y, \: c^*\E^\bullet)$. Concretely, $Z$-faithfulness means that, if $\theta \in c^*\T$ is a local section with $\nabla_\theta^P = 0$ as a map $\nabla_\theta^P: c^*\E^P \to c^*\E^P$ for all $P \in Z$, then $\theta = 0$.
  \item Since $U \subseteq X$ satisfies \eqref{CC} and the fibers of $f: X \to S$ have Serre's property $(S_2)$, the natural map $\cO_X \to j_*\cO_U$ is an isomorphism by \cite[Prop.~3.5]{Hassett2004}. This holds after any base change. Thus, every geometric family of $\cP$-algebras is closed at $F$.
  \item If $\E^\bullet$ is strictly faithful, then the natural map $\T \to j_*\T|_U$ is injective. Namely, if $\theta|_U = 0$, then $\nabla_\theta^F|_U = 0$, but $\cO_X = j_*\cO_U$, so $\nabla_\theta^F = 0$. Conversely, if $\E^\bullet$ is pure at $T$ and $\E^\bullet|_U$ is strictly faithful, then $\E^\bullet$ is strictly faithful.
  \item Let $f: (X,\E^\bullet) \to S$ be pure or closed at $P \in D(\cP)$. Then, for every $s \in S$, the natural map $\E^P_s \to j_*\E^P_s|_{U_s}$ is injective respective bijective. Conversely, if this holds, then it holds after any base change to the spectrum of a field, and then $f: (X,\E^\bullet) \to S$ is pure respective closed at $P \in D(\cP)$ by Lemma~\ref{injective-in-fibers} and Lemma~\ref{bijective-in-fibers}.
  \item In practice, we usually know that $\E^P|_U$ is locally free, that $\E^\bullet|_{U^{sm}}$ is $Z$-faithful or even strictly faithful after any base change, and that $\T_s \to j_*\T_s|_{U_s}$ is injective for $s \in S$. Then we can conclude\footnote{The map $\E^P|_U \to j^{sm}_*\E^P|_{U^{sm}}$ is injective if $\E^P$ is locally free on $U$ because $U^{sm} \to U$ is scheme-theoretically dense in every fiber and thus after any base change.} that the family is tame at every $P \in D(\cP)$, pure at $T$, and $Z$-faithful respective strictly faithful. This is one reason why we carry the two open subsets $U$ and $U^{sm}$ as part of the structure. 
 \end{enumerate}
\end{rem}
\begin{ex}
 Here is an example of a geometric family of Gerstenhaber calculi which arises from a classical construction.
 Let 
 $$f: \bAA^4 \to \bAA^1_t, \quad t \mapsto x^2 + y^2 + z^2 + w^3.$$
 Then the central fiber is the normal $A_2$-threefold singularity, and $f$ is a smoothing. We take $\Lambda = \kk\llbracket t\rrbracket$ and obtain a Gerstenhaber calculus as the base change of the classical reflexive relative Gerstenhaber calculus of $f: \bAA^4 \to \bAA^1_t$. Since $f: \bAA^4 \to \bAA^1_t$ is smooth outside the isolated singularity in the central fiber, $\E^P|_U$ is locally free when we set $U := \bAA^4 \setminus \{0\}$. A direct calculation using Macaulay2 shows that $\T_0 \to j_*\T_0|_{U_0}$ is injective. Since $\E^\bullet$ is strictly faithful on $U$, we find that $\E^\bullet_A$ is tame and strictly faithful for every $A \in \mathbf{Art}_\Lambda$.
\end{ex}

For practical applications, let us show the following lemma. The proof is a little more involved than one may expect because of the non-$\cO_X$-linear nature of the involved maps.

\begin{lemma}\label{Z-faithful-check-on-fiber}\note{Z-faithful-check-on-fiber}
 Assume that $Z \subseteq D(\cP)$ is finite. Then, to show that $c^*\E^\bullet$ is $Z$-faithful (respective strictly faithful) for all $b: T \to S$, it is sufficient to show that the fiber $\E^\bullet_s$ is $Z$-faithful (respective strictly faithful) for every point $s \in S$.
\end{lemma}
\begin{proof}
 Let $f: (X,\E^\bullet) \to S$ be a candidate for a geometric family of $\cP$-algebras which satisfies everything except possibly the last condition on $Z$-faithfulness. Assume furthermore that $\E_s^\bullet$ is $Z$-faithful (respective strictly faithful) for every point $s \in S$. First, let $K$ be a field, and let $\Spec K \to S$ be a morphism. Let $s \in S$ be the image. Then we have a factorization 
 $$\Spec K \to \Spec \kappa(s) \to S.$$
 Let $\{\lambda_i\}$ be a $\kappa(s)$-basis of $K$. Let $V \subseteq X_s$ be an affine open subset, and assume $\theta = \sum_i \theta(i) \otimes \lambda_i \in \Gamma(V_K,\T_K)$ with $\theta(i) \in \Gamma(V,\T_{\kappa(s)})$ satisfies $\nabla^M_\theta = 0$ for all $M \in Z$. Then, for $m \in \E^M_{\kappa(s)}$, we have $\nabla_\theta^M(m) = \sum_i \nabla_{\theta(i)}^M(m) \otimes \lambda_i = 0$, so $\nabla^M_{\theta(i)}(m) = 0$ and hence $\theta(i) = 0$ by assumption. Thus, $\theta = 0$. 
 
 Next, let $W \subseteq V_K$ be a standard open subset defined by a function $g$, and let $\theta \in \Gamma(W,\T_K)$ with $\nabla^M_\theta = 0$ for all $M \in Z$. We have $g^n\nabla^M_\theta = \nabla^M_{g^n\theta}$. Thus, we can assume without loss of generality that $\theta$ is the restriction of a section $\theta \in \Gamma(V_K,\T_K)$. Let $W' \subseteq V_K$ be another standard open, and let $m \in \Gamma(W',\E^M_K)$ be a section. Then $\nabla_\theta^M(m)|_{W \cap W'} = 0$, so $g^n \cdot \nabla_\theta^M(m) = 0$ in $\Gamma(W',\E_K^M)$ for some $n \geq 0$. The coherent subsheaves $\Z^M(n) := \{m \in \E_K^M \ | \ g^n \cdot m = 0\}$ for an ascending sequence in the coherent sheaf $\E^M_K$; since $X$ is Noetherian, the sequence becomes stationary. Thus, there is some $n$ with $g^n \cdot \nabla_\theta^M(m) = 0$ for all local sections $m \in \E^M_K$. Since $Z$ is finite, we can find an $n$ which works for all $M \in Z$ simultaneously. Then we find $g^n \theta = 0$ because we already know $Z$-faithfulness (respective strict faithfulness) for sections in $\Gamma(V_K,\T_K)$; in particular, $\theta|_W = 0$ so that $\E^\bullet_K$ is $Z$-faithful (respective strictly faithful).
 
 For the next step, let $A$ be an Artinian local ring with residue field $K$, and let $\Spec A \to S$ be a morphism. Then the proof of Lemma~\ref{gauge-inject} shows that $\E_A^\bullet$ is $Z$-faithful (respective strictly faithful) if $\E_K^\bullet$ is.
 
 Finally, let $T = \Spec \cO_T$ be Noetherian affine, and let $b: T \to S$ be a morphism. Let $\theta \in \Gamma(V,\T_T)$ be a section over an affine open $V \subseteq Y$ with $\nabla_\theta^M = 0$ for $M \in Z$. Then, for a map $\cO_T \to A$ to an Artinian local ring $A$, the induced differential operator $\nabla_\theta^M$ on $\E^M_A$ is zero as well. Since $\E_A^\bullet$ is $Z$-faithful (respective strictly faithful), we find $\theta|_A = 0$. Then $\theta = 0$ by Lemma~\ref{zero-in-thick-fibers}.
\end{proof}

\subsubsection*{Notions of good geometric families of $\cP$-algebras}

In Definition~\ref{geom-fam-P-alg-def}, we have given a general definition to capture geometric families of $\cP$-algebras for all algebraic structures $\cP$ carrying a Cartan structure at once. However, for specific algebraic structures $\cP$, we wish to impose additional hypotheses unique to that $\cP$. It is here that $U^{sm}$ comes into play to tie the family of $\cP$-algebras to what happens classically on the smooth locus. These definitions are not essential but rather illustrate the concept of geometric families of $\cP$-algebras.

Let $\T^0_{X/S}$ be the (classical) relative tangent sheaf, i.e., the dual of $\Omega^1_{\underline X/\underline S}$. For every geometric family of $\cP$-algebras, we have a map 
$$a: \T \to \T^0_{X/S}, \quad \theta \mapsto (g \mapsto -\nabla_\theta^F(g)),$$
of Lie algebras (when endowing the right hand side with the negative Lie bracket)\footnote{With the negative sign we account for our construction of the bracket on $\T = \G^{-1}$.} and $\cO_X$-modules, which corresponds to the \emph{anchor map} of a Lie algebroid.

\begin{defn}
 A geometric family of Lie--Rinehart algebras $f: (X,\F,\T) \to S$ is \emph{good}\index{geometric family of Lie--Rinehart algebras!good} if it is strictly faithful, tame at $F$ and $T$, pure at $T$, and if the anchor map $a: \T|_{U^{sm}} \to \T^0_{U^{sm}/S}$ is an isomorphism.
\end{defn}

If we have a Lie--Rinehart pair $(\F,\T,\E)$, then we have an extended anchor map 
$$a^E: \T \to \T^0_{X/S}(\E), \quad \theta \mapsto (g \mapsto -\nabla^F_\theta(g), \: e \mapsto -\nabla^E_\theta(e)),$$
where $\T^0_{X/S}(\E)$ denotes derivations of $f: (X,\E) \to S$ in the sense of Definition~\ref{derivation-vec-bdl}.

\begin{defn}
 A geometric family of Lie--Rinehart pairs $f: (X,\F,\T,\E) \to S$ is \emph{good}\index{geometric family of Lie--Rinehart pairs!good} if it is $Z$-faithful, tame and pure at $T$, if $\E$ is a vector bundle of constant rank $r$, and if the extended anchor map $a^E: \T|_{U^{sm}} \to \T^0_{U^{sm}/S}(\E)$ is an isomorphism.
\end{defn}

\begin{defn}\label{good-G-alg-defn}\note{good-G-alg-defn}
 A geometric family of Gerstenhaber algebras $f: (X,\G^\bullet) \to S$ is \emph{good}\index{geometric family of Gerstenhaber algebras!good} if:
 \begin{itemize}
  \item the relative dimension $d$ of $f: X \to S$ equals the dimension $d$ of the Gerstenhaber algebra $\G^\bullet$, which is  in degrees $-d \leq p \leq 0$;
  \item it is strictly faithful as well as tame and pure at all $\G^{p}$;
  \item  the anchor map $a: \T|_{U^{sm}} \to \T^0_{U^{sm}/S}$ is an isomorphism;
  \item $\G^\bullet|_U$ is isomorphic to the exterior algebra of $\G^{-1}|_U$ via the $\wedge$-product;
  \item $\G^{-d}$ is a line bundle.
 \end{itemize}
\end{defn}

\begin{defn}\label{good-G-calc-defn}\note{good-G-calc-defn}
 A geometric family of two-sided Gerstenhaber calculi $f: (X,\G^\bullet,\A^\bullet)\to S$ is \emph{good}\index{geometric family of Gerstenhaber calculi!good} if:
 \begin{itemize}
  \item the relative dimension $d$ of $f: X \to S$ equals the dimension $d$ of the two-sided Gerstenhaber calculus $(\G^\bullet,\A^\bullet)$;
  \item it is strictly faithful as well as tame and pure at all $\G^p$ and $\A^i$;
  \item it is \emph{locally Batalin--Vilkovisky},\index{locally Batalin--Vilkovisky} i.e., $\A^d$ is a line bundle, and a local generator $\omega \in \A^d$ induces a local isomorphism $\G^p \cong \A^{p + d}$ for every $-d \leq p \leq 0$; if this property holds for one local generator, it holds for every local generator (cf.~Proposition~\ref{BV-calc-construction});
  \item the anchor map $a: \G^{-1}|_{U^{sm}} \to \T^0_{U^{sm}/S}$ is an isomorphism;
  \item $\G^\bullet|_U$ is isomorphic to the exterior algebra of $\G^{-1}|_U$ via the $\wedge$-product;
  \item $\A^\bullet|_U$ is isomorphic to the exterior algebra of $\A^1|_U$ via the $\wedge$-product;
  \item the induced map 
  $$a^\vee: \A^1|_{U^{sm}} \to \cH om(\T^0_{U^{sm}/S},\cO_{U^{sm}}) \cong \Omega^1_{\underline U^{sm}/\underline S}, \quad \alpha \mapsto (a^{-1}(\theta) \mapsto (\theta\ \ \invneg\ \alpha)),$$
  is an isomorphism of $\cO_X$-modules.
 \end{itemize}
\end{defn}

\begin{rem}
 If $f: (X,\G^\bullet) \to S$ is a good geometric family of Gerstenhaber algebras, then $\G^\bullet|_{U^{sm}}$ is isomorphic to the Gerstenhaber algebra $\G^\bullet_{U^{sm}/S}$ of the strict log smooth family $f: U^{sm} \to S$. Similarly, if $f: (X,\G^\bullet,\A^\bullet) \to S$ is a good geometric family of two-sided Gerstenhaber calculi, then $(\G^\bullet,\A^\bullet)|_{U^{sm}}$ is isomorphic to the Gerstenhaber calculus of $f: U^{sm} \to S$. For example, one can show that 
 $$a^\vee \circ \partial = \partial_{dR}:\: \cO_{U^{sm}} \to \Omega^1_{\underline U^{sm}/\underline S},$$
 where $\partial_{dR}$ is the usual de Rham differential, by exploiting the definition of the anchor map $a$ and the axioms of a Gerstenhaber calculus.
\end{rem}

\begin{ex}\label{log-Gorenstein-good-geom-fam-ex}\note{log-Gorenstein-good-geom-fam-ex}
 Let $f: X \to S$ be a log Gorenstein generically log smooth family. Then its Gerstenhaber calculus in the sense of Proposition~\ref{G-A-construction} forms a good geometric family of two-sided Gerstenhaber calculi when we take for $U^{sm}$ the strict locus of $f: U \to S$, which is scheme-theoretically dense in every fiber by \cite[Prop.~10.1]{Felten2022}.
\end{ex}

\section{Geometric deformations of $\cP$-algebras}

Let $\cP$ be an algebraic structure carrying a Cartan structure. Let $\Lambda$ be a complete local Noetherian $\kk$-algebra\footnote{We do \emph{not} assume that $\Lambda$ must be of the form $\kk\llbracket Q\rrbracket$.} with residue field $\kk$, and let 
$$f_0: (X_0,U_0,U_0^{sm},\E^\bullet_0) \to S_0 = \Spec \kk$$
be a geometric family of $\cP$-algebras. 

\begin{defn}
 Let $f_0: X_0 \to S_0$ be as above, and let $A \in \mathbf{Art}_\Lambda$. Then a \emph{geometric deformation of $\cP$-algebras}\index{geometric deformation of $\cP$-algebras} of $f_0: X_0 \to S_0$ over $A$ is a geometric family of $\cP$-algebras $$f_A: (X_A,U_A,U_A^{sm},\E_A^\bullet) \to S_A$$
 together with a morphism from $f_0: X_0 \to S_0$ (which, by definition, induces an isomorphism after base change). If $f_0: X_0 \to S_0$ is a good geometric family of $\cP$-algebras, then we say that a geometric deformation $f_A: X_A \to S_A$ is \emph{good}\index{geometric deformation of $\cP$-algebras!good} if it is good as a geometric family of $\cP$-algebras.
\end{defn}
\begin{rem}\label{geom-fam-Gerstenhaber-calc-rem}\note{geom-fam-Gerstenhaber-calc-rem}
 When we already know that $f_0: (X_0,\E_0^\bullet) \to S_0$ is a geometric family of $\cP$-algebras, checking the conditions for $f_A: (X_A,\E_A^\bullet) \to S_A$ becomes easier. Assume that $f_A: X_A \to S_A$ is a separated and flat thickening of $f_0: X_0 \to S_0$ of finite type, and let $\E_A^\bullet$ be a Cartanian $\cP$-algebra in $\mathfrak{Coh}(X_A/S_A)$ together with a map $\E_A^\bullet \to \E_0^\bullet$ of $\cO_{X_A}$-modules which is compatible with all constants and operations and induces an isomorphism after base change to $X_0$. This is enough to conclude that $f_A: (X_A,\E_A^\bullet) \to S_A$ is a geometric deformation of $f_0: (X_0,\E_0^\bullet) \to S_0$ when we take $U_A$ and $U_A^{sm}$ such that $U_A = U_0$ and $U_A^{sm} = U_0^{sm}$ on underlying schemes. The most interesting part of the proof is $Z$-faithfulness, which follows from Lemma~\ref{Z-faithful-check-on-fiber}, and that $\cO_{X_A} \to \F_A$ is an isomorphism. The latter morphism is surjective because its pull-back to $X_0$ is surjective. The target $\F_A$ is flat over $A$ so that the embedding of the kernel is universally injective (over $A$). Then the pull-back to $X_0$ is the kernel of $\cO_{X_0} \to \F_0$, hence $0$. Thus, the kernel vanishes itself, and $\cO_{X_A} \to \F_A$ is an isomorphism.
 
 Similarly, if $f_0: (X_0,\E_0^\bullet)$ is tame or locally free at $P \in D(\cP)$, then $f_A$ is tame respective locally free at $P \in D(\cP)$. Since $\E_A^P$ is flat over $S_A$, it follows from Lemma~\ref{injective-in-fibers} respective Lemma~\ref{bijective-in-fibers} that $f_A$ is pure respective closed at $P \in D(\cP)$ if $f_0$ is. Then, if $f_0$ is reflexive at $P \in D(\cP)$, the deformation $f_A$ is reflexive at $P \in D(\cP)$ as well. If $f_0$ is strictly faithful, so is $f_A$ by Lemma~\ref{Z-faithful-check-on-fiber}.
\end{rem}

\begin{rem}\label{good-defo-rem}\note{good-defo-rem}
 When we already know that $f_0:(X_0,\E_0^\bullet) \to S_0$ is a \emph{good} geometric family of Gerstenhaber calculi, a deformation $f_A: (X_A,\E_A^\bullet) \to S_A$ is good as well. A local generator of $\A^d_0$ can be lifted to a local section of $\A^d_A$, which induces a local isomorphism $\cO_{X_A} \cong \A_A^d$ similarly to the case of $\F_A$ above. Then, the local generator gives locally a homomorphism $\kappa_\omega: \G^p_A \to \A^{p + d}_A$ whose pull-back to $X_0$ is an isomorphism; since $\A^{p + d}_A$ is flat over $A$, we find that this $\kappa_\omega$ is an isomorphism. Similar arguments complete the proof.
\end{rem}

This definition is only intended to be language. Be aware of the curious fact that, if $A, A' \in \mathbf{Art}_\Lambda$ are isomorphic as $\kk$-algebras but not as $\Lambda$-algebras, then there is no real difference in the notions of deformations over $A$ or $A'$ but nonetheless we distinguish them.

\par\vspace{\baselineskip}

We do not wish to form isomorphism classes and a functor of Artin rings at this point as we do not yet have defined a useful notion of isomorphism of geometric deformations of $\cP$-algebras. To this end, we need gauge transforms in the sense of Definition~\ref{gauge-trafo-def}.

Let $B' \to B$ be a surjection in $\mathbf{Art}_\Lambda$ with kernel $I \subset B'$; let $f': X_{B'} \to S_{B'}$ be a geometric deformation of $\cP$-algebras of $f_0: X_0 \to S_0$ over $B$, and let $f: X_B \to S_B$ be the base change to $B$. For $\theta \in \Gamma(X_0,I \cdot \T_{B'})$, the \emph{gauge transform} $\mathrm{exp}_\theta$ consists of sheaf maps
$$\mathrm{exp}_\theta: \quad \E^P_{B'} \to \E^P_{B'}, \quad p \mapsto \sum_{n = 0}^\infty \frac{(\nabla^P_\theta)^n(p)}{n!}
,$$
which give rise to an automorphism $\mathrm{exp}_\theta: X_{B'} \to X_{B'}$ of geometric deformations $\cP$-algebras.\footnote{Use $\cO_{X_{B'}} = \F_{B'}$ for the automorphism of underlying schemes.} 
Let us denote by $\G auge_{X_{B'}/X_B}$ the sheaf of gauge transforms, considered as a sheaf of groups which acts by automorphisms of $f': X_{B'} \to S_{B'}$.
If $B' \to B$ and $B'' \to B'$ are two surjections in $\mathbf{Art}_\Lambda$, then we obtain a restriction map
$$\G auge_{X_{B''}/X_B} \to \G auge_{X_{B'}/X_B},$$
which coincides with the restriction map $I_{B''/B} \cdot \G_{B''}^{-1} \to I_{B'/B}\cdot  \G_{B'}^{-1}$; it is surjective\footnote{Recall from Example~\ref{base-change-violation-t3-w3} that a generically log smooth deformation does not need to have this property if we take all automorphisms into account.} because our deformations satisfy, by definition, $\T_{B'} \otimes_{B'} B = \T_B$.

If $\theta,\xi \in \Gamma(X_0,I \cdot \G_{B'}^{-1})$ with $\mathrm{exp}_\theta = \mathrm{exp}_\xi$, then $\theta = \xi$ by induction over small extensions, using $Z$-faithfulness on each small extension. Thus, gauge transforms form a subsheaf 
$$\G auge_{X_{B'}/X_B} \subseteq \A ut_{X_{B'}/X_B}$$
of all automorphisms of geometric deformations of $\cP$-algebras.\footnote{This is the reason why we require $Z$-faithfulness; if we do not have it, then there are ``ghost'' gauge transforms which act by the identity.}

\section{Systems of deformations}

We define a device to control geometric deformations of $\cP$-algebras locally, analogous to the case of generically log smooth families.
Let $\Lambda$ and $f_0: (X_0,\E_0^\bullet) \to S_0$ be as above.

\begin{defn}\label{P-admissible-def}\note{P-admissible-def}
 A pre-admissible covering $\V = \{V_\alpha\}_\alpha$ of $X_0$ is \emph{admissible}\index{cover!admissible} with respect to $\E_0^\bullet$ if, for every $P \in D(\cP)$, we have 
 $$H^i(V_{\alpha_1} \cap ... \cap V_{\alpha_r}, \E_0^P|_{\alpha_1...\alpha_r}) = 0$$
 for all $i \geq 1$ and indices $\alpha_1,...,\alpha_r$, where we write $\F|_{\alpha_1...\alpha_r} := \F|_{V_{\alpha_1} \cap ... \cap V_{\alpha_r}}$.
\end{defn}

We require this condition to compute the cohomology of $\E_0^P$ from the \v{C}ech complex of $\{V_\alpha\}_\alpha$, and to make deformations on $V_\alpha$ unique up to isomorphism. The following definition essentially goes back to \cite[\S 2]{ChanLeungMa2023}. The key point is that the cocycles must be gauge transforms.

\begin{defn}
 Let $\cP$ be an algebraic structure carrying a Cartan structure, let $f_0: (X_0,U_0,U_0^{sm},\E_0^\bullet) \to S_0$ be a geometric family of $\cP$-algebras, and let $\{V_\alpha\}_\alpha$ be an admissible open cover.
 \begin{enumerate}[label=(\alph*)]
  \item A \emph{system of deformations}\index{system of deformations} $\D$ for $f_0: (X_0,\E_0^\bullet) \to S_0$ is a tuple
 $$\D = (f_{\alpha;A}: (V_{\alpha;A},\E_{\alpha;A}^\bullet) \to S_A, \: \rho_{\alpha;BB'},\: \psi_{\alpha\beta;A},\: o_{\alpha\beta\gamma;A}) $$
 where:
 \begin{itemize}
  \item for every $A \in \mathbf{Art}_\Lambda$, the object $f_{\alpha;A}: (V_{\alpha;A},\E_{\alpha;A}^\bullet) \to S_A$ is a geometric deformation of $\cP$-algebras of $f_0: (X_0,\E_0^\bullet) \to S_0$; in particular, $(V_{\alpha;0},\E_{\alpha;A}^\bullet) = (X_0,\E_0^\bullet)|_\alpha$;
  \item for every map $B' \to B$ in $\mathbf{Art}_\Lambda$, the object $\rho_{\alpha;BB'}: (V_{\alpha;B},\E_{\alpha;B}^\bullet) \to (V_{\alpha;B'},\E_{\alpha;B'}^\bullet)$ is a map of geometric families of $\cP$-algebras, compatible with the map from $V_{\alpha;0}$; they induce isomorphisms $\E^P_{\alpha;B'} \otimes_{B'} B \cong \E_{\alpha;B}^P$ by the definition of a morphism of geometric families of $\cP$-algebras;
  \item for every $A \in \mathbf{Art}_\Lambda$, the object $\psi_{\alpha\beta;A}: (V_{\alpha;A},\E_{\alpha;A}^\bullet)|_{\alpha\beta} \to (V_{\beta;A},\E_{\beta;A}^\bullet)|_{\alpha\beta}$ is an isomorphism of geometric deformations of $\cP$-algebras; we assume them to be compatible with the restriction maps $\rho_{\alpha;BB'}$; in particular, $\psi_{\alpha\beta;0}$ is the identity; they satisfy $\psi_{\alpha\beta;A} ^{-1} = \psi_{\beta\alpha;A}$;
  \item the cocycles $\psi_{\gamma\alpha;A} \circ \psi_{\beta\gamma;A} \circ \psi_{\alpha\beta;A}$ are equal to (unique) gauge transforms $\mathrm{exp}_{o_{\alpha\beta\gamma;A}}$ for $o_{\alpha\beta\gamma;A} \in \Gamma(V_\alpha \cap V_\beta \cap V_\gamma,\m_A \cdot \T_{\alpha;A})$.
 \end{itemize}
 \item A \emph{geometric deformation of $\cP$-algebras of type $\D$} is a geometric deformation $f_A: (X_A,\E_A^\bullet) \to S_A$ of $\cP$-algebras of $f_0: (X_0,\E_0^\bullet) \to S_0$ 
together with isomorphisms 
 $$\chi_\alpha: (X_A,\E_A^\bullet)|_{\alpha} \cong (V_{\alpha;A},\E_{\alpha;A}^\bullet)$$
 of geometric deformations of $\cP$-algebras such that 
 $$\psi_{\beta\alpha;A} \circ \chi_\beta|_{\alpha\beta} \circ \chi_\alpha^{-1}|_{\alpha\beta}$$
 is a gauge transform of $(V_{\alpha;A},\E_{\alpha;A}^\bullet)|_{\alpha\beta}$ for all indices $\alpha,\beta$.
 \item 
 If $B' \to B$ is a map in $\mathbf{Art}_\Lambda$, and if $f: (X_B,\E_B^\bullet) \to S_B$ and $f': (X_{B'},\E_{B'}^\bullet) \to S_{B'}$ are two geometric deformations of $\cP$-algebras of type $\D$ with local isomorphisms $\chi_\alpha$ and $\chi'_\alpha$, then a \emph{morphism} is a morphism of geometric deformations of $\cP$-algebras---consisting of maps $\cO_{X_{B'}} \to \cO_{X_B}$ and $\E_{B'}^\bullet \to \E_B^\bullet$ inducing an isomorphism in pull-backs, and compatible with the maps to the central fiber---with the following property: The local isomorphism $\chi'_\alpha$ induces, via the Cartesian squares, a local isomorphism $\chi'_\alpha|_B: (X_{B},\E_B^\bullet)|_\alpha \cong (V_{\alpha;B},\E_{\alpha;B}^\bullet)$. Then $\chi'_\alpha|_B \circ \chi_\alpha^{-1}$ must be a gauge transform of $(V_{\alpha;B},\E_{\alpha;B}^\bullet)$.
  \item 
 
 Two geometric deformations of type $\D$ are \emph{equivalent} if there is an isomorphism $\sigma: (X_A,\E_A^\bullet) \to (X'_A,\E_A'^\bullet)$ of geometric deformations of $\cP$-algebras such that $\chi'_\alpha \circ \sigma|_\alpha \circ \chi_\alpha^{-1}$ is a gauge transform of $(V_{\alpha;A},\E_{\alpha;A}^\bullet)$. This is precisely an isomorphism of geometric deformations of $\cP$-algebras of type $\D$ because we have $\chi'_\alpha \circ \sigma_\alpha = \chi'_\alpha|_A$ by the construction of $\chi'_\alpha|_A$. Equivalence classes of geometric deformations of $\cP$-algebras of type $\D$ form a functor of Artin rings
 $$\mathrm{GDef}^\D(\E_0^\bullet,-): \mathbf{Art}_\Lambda \to \mathbf{Set}.$$
 \end{enumerate}

\end{defn}
\begin{rem}
 At this point, we see why it makes sense to work over $\mathbf{Art}_\Lambda$. Namely, for $A,A' \in \mathbf{Art}_\Lambda$, we may choose completely different local models $V_{\alpha;A}$ even if $A \cong A'$ as $\kk$-algebras. For example, in the case of log smooth deformations over $S_0 = \Spec(\NN \to \kk)$, we obtain different local models over $A = \kk[t]/(t^2)$ and $A' = \kk[s]/(s^2)$ with $1 \mapsto t$ respective $1 \mapsto 0$.
\end{rem}

\begin{lemma}\label{P-def-defo-functor}\note{P-def-defo-functor}
 The functor $\mathrm{GDef}^\D(\E_0^\bullet,-)$
 is a deformation functor.
\end{lemma}
\begin{proof}
 Condition $(H_0)$ is clear. For condition $(H_1)$, let $A' \to A$ be arbitrary and $A'' \to A$ be surjective in $\mathbf{Art}_\Lambda$, and let $B := A' \times_A A''$. In the case of the local models $V_{\alpha;C}$, we form 
 $$Y := (|V_\alpha|,\: \cO_{V_{\alpha;A'}} \times_{\cO_{V_{\alpha;A}}}\cO_{V_{\alpha;A''}},\: \E^P_{\alpha;A'} \times_{\E^P_{\alpha;A}} \E^P_{\alpha;A''});$$
 due to the universal property, we obtain a map $Y \to V_{\alpha;B}$. As in the case of $\mathrm{LD}_{X_0/S_0}^\D$ in Lemma~\ref{defo-type-D-is-defo-functor}, we first obtain that $\cO_Y \cong \cO_{V_{\alpha;B}}$ is an isomorphism---for the surjectivity, we use $\cO_Y \otimes_B A' \cong \cO_{V_{\alpha;A'}}$ and that we can apply \cite[09ZW]{stacks}. With the same injectivity argument as for $\cO_Y$, we find that $\E_{\alpha;B}^P \to \E_Y^P$ is injective. However, it is also surjective because $\E_Y^P \otimes_B A' \cong \E_{\alpha;A'}^P$ by \cite[Thm.~2.2]{Ferrand2003}. Thus, $Y = V_{\alpha;B}$. The general case is then very similar to the case of $\mathrm{LD}_{X_0/S_0}^{\D}$ in Lemma~\ref{defo-type-D-is-defo-functor} because we can lift gauge transforms from $X_A|_{\alpha}$ to $X_{A''}|_\alpha$. The proof of condition $(H_2)$ is analogous to the log smooth case in \cite[Thm.~8.7]{Kato1996}, using the above-constructed push-out of geometric deformations of $\cP$-algebras.
\end{proof}

If $B' \to B$ is a surjection in $\mathbf{Art}_\Lambda$ and $f: (X_B,\E_B^\bullet) \to S_B$ and $f': (X_{B'},\E_{B'}^\bullet) \to S_{B'}$ are two compatible geometric deformations of $\cP$-algebras of type $\D$, then we have the sheaf 
$$\A ut_{X_{B'}/X_B}^\D$$
of relative automorphisms. It consists precisely of the gauge transforms, i.e.,
$$\A ut_{X_{B'}/X_B}^\D = \G auge_{X_{B'}/X_B}$$
since gauge transforms correspond to gauge transforms under isomorphisms of $\cP$-algebras---and since gauge transforms inject into all automorphisms. In particular, if $B' \to B$ is a small extension with kernel $I \subset B'$, then we get
$$\A ut_{X_{B'}/X_B}^\D = \T_0 \otimes_\kk I.$$
Now let $f: (X_B,\E_B^\bullet) \to S_B$ a geometric deformation of $\cP$-algebras of type $\D$. On $V_\alpha$, the local models allow a lift to $S_{B'}$. Since all automorphisms can be lifted, we can compare them with maps that become the identity over $S_B$. This yields the following standard result.

\begin{prop}
 Let $f_0: X_0 \to S_0$ be a geometric family of $\cP$-algebras, and let $\D$ be a system of deformations over $\Lambda$. Let $B' \to B$ be a small extension with kernel $I$, and assume that we have a geometric deformation $f: X_B \to S_B$ of type $\D$. Then:
 \begin{enumerate}[label=\emph{(\roman*)}]
  \item The automorphisms of a given lifting $f': X_{B'} \to S_{B'}$ lie in $H^0(X_0,\T_0) \otimes_\kk I$.
  \item Given one lifting, the set of all liftings is a torsor under $H^1(X_0,\T_0) \otimes_\kk I$.
  \item The obstruction to the existence of a lifting is in $H^2(X_0,\T_0) \otimes_\kk I$.
 \end{enumerate}
\end{prop}

\section{Systems of deformations from geometry}

\subsection{Systems of deformations from log geometry: generically log smooth families}

Let $f_0: X_0 \to S_0$ be a log Gorenstein generically log smooth family over $S_0 = \Spec( Q \to \kk)$, and let $\V = \{V_\alpha\}_\alpha$ be a weakly admissible open cover. When we have a system of deformations $\D$ for $f_0$ subordinate to $\V$, then we can form an associated system of good geometric deformations of Lie--Rinehart algebras $\D^{lr}$ such that geometric deformations of Lie--Rinehart algebras of type $\D^{lr}$ correspond one-to-one to generically log smooth deformations of type $\D$. On the central fiber, we take for $U_0^{sm}$ the strict locus $U_0^{str}$ of $f: U_0 \to S_0$, which is smooth because $f_0: U_0 \to S_0$ is log smooth, and scheme-theoretically dense by \cite[Prop.~10.1]{Felten2022}.
We have a geometric family of Lie--Rinehart algebras $\cL\R^\bullet_{X_0/S_0}$. This is a \emph{good} geometric family of Lie--Rinehart algebras essentially by construction. Since $\V$ is weakly admissible with respect to $f_0: X_0 \to S_0$, it is also admissible with respect to $\cL\R_{X_0/S_0}^\bullet$. The local model $V_{\alpha;A} \to S_A$ from $\D$ gives rise to a geometric deformation of Lie--Rinehart algebras $(V_{\alpha;A},\cL\R_{V_{\alpha;A}/S_A}^\bullet)$; the pieces are flat because we assume each local deformation to have the base change property. The $(V_{\alpha;A},\cL\R_{V_{\alpha;A}/S_A}^\bullet)$ come with restriction maps and comparison isomorphisms induced from those of $\D$. The cocycles are gauge transforms by Lemma~\ref{geom-auto-gauge-trafo-corr}. 

Analogously, we can also form a good geometric family of Gerstenhaber algebras $\V_{X_0/S_0}^\bullet$ on $X_0$, or a good geometric family $\V\,\W^\bullet_{X_0/S_0}$ of two-sided Gerstenhaber calculi. If $\V$ is not only weakly admissible but admissible, then we can form an analogous system of geometric deformations of Gerstenhaber algebras $\D^g$ respective a system of geometric deformations of Gerstenhaber calculi $\D^{gc}$. Again, the log Gorenstein assumption and the base change property are crucial to make $\V\,\W^\bullet_{X_A/S_A}$ flat over $S_A$ and compatible with base change. Again, Lemma~\ref{geom-auto-gauge-trafo-corr} is crucial to make the cocycles gauge transforms.

In all three cases, we have the following result, which we formulate only for the most important case of two-sided Gerstenhaber calculi.

\begin{prop}\label{LDD-GDefD-iso}\note{LDD-GDefD-iso}
 Let $\Lambda = \kk\llbracket Q\rrbracket$, let $f_0: X_0 \to S_0$ be a generically log smooth family, and let $\D$ be a system of generically log smooth deformations subordinate to an admissible cover $\V = \{V_\alpha\}_\alpha$. Let $\D^{gc}$ be the associated system of geometric deformations of two-sided Gerstenhaber calculi. Then the natural map 
 $$\mathrm{LD}_{X_0/S_0}^\D \to \mathrm{GDef}^{\D^{gc}}(\V\,\W^\bullet_{X_0/S_0},-), \qquad (f: X_A \to S_A) \mapsto (X_A,\,U_A,\,U_A^{str},\: \V^\bullet_{X_A/S_A},\W^\bullet_{X_A/S_A})$$
 is an isomorphism of functors of Artin rings.
\end{prop}
\begin{proof}
 Let $f_A : X_A \to S_A$ be a generically log smooth deformation of type $\D$. Write $(\V^\bullet,\W^\bullet)$ for its two-sided Gerstenhaber calculus. First note that the image of the map is in fact a (good) geometric family of Gerstenhaber calculi. Next, note that we have a geometric \emph{deformation} of two-sided Gerstenhaber calculi because $U_A^{str}|_{X_0} = U_0^{str}$ and $f_A: X_A \to S_A$ has the base change property. It is of type $\D^{gc}$ because the (existent but not fixed) isomorphisms $(X_A,\V^\bullet,\W^\bullet)|_\alpha \cong (V_{\alpha;A},\V_{\alpha;A}^\bullet,\W_{\alpha;A}^\bullet)$ of generically log smooth families give rise to a choice of the desired isomorphisms of two-sided Gerstenhaber calculi. On overlaps, the comparison maps must be gauge transforms by Lemma~\ref{geom-auto-gauge-trafo-corr}. If we choose other isomorphisms, then they differ, again by Lemma~\ref{geom-auto-gauge-trafo-corr}, by a gauge transform so that the resulting geometric deformations of two-sided Gerstenhaber calculi of type $\D^{gc}$ are equivalent. That the map is an isomorphism follows from tracking gluings of the pieces on $V_\alpha$ as in the proof of \cite[Prop.~4.2]{Felten2022}.
\end{proof}

\subsection{Systems of deformations from log geometry: enhanced generically log smooth families}\label{sys-defo-enhanced-gen-log-fam}\note{sys-defo-enhanced-gen-log-fam}

Let $f_0: X_0 \to S_0$ be an enhanced generically log smooth family with an enhanced system of deformations $\D$ subordinate to $\V = \{V_\alpha\}_\alpha$. Then we obtain a system of geometric deformations of Gerstenhaber calculi $\D^{gc}$ essentially by forgetting the log structures and only keeping the underlying schemes as well as the Gerstenhaber calculi. It is clear from the definitions that we obtain an isomorphism 
$$\mathrm{ELD}^\D_{X_0/S_0} \to \mathrm{GDef}^{\D^{gc}}(\G_0^\bullet,\A_0^\bullet,-)$$
of functors of Artin rings.

\subsection{Systems of deformations from log geometry: vector bundles}\label{sys-defo-vector-bundles}\note{sys-defo-vector-bundles}

Let $f_0: (X_0,\E_0) \to S_0$ be an enhanced generically log smooth family with a vector bundle $\E_0$ of rank $r$, and let $\V = \{V_\alpha\}_\alpha$ be an $\E_0$-admissible open cover of $X_0$. Let $\D$ be a system of deformations subordinate to $\V$. Now $\cL\R\cP_0^\bullet := \cL\R\cP^\bullet_{X_0/S_0}(\E_0)$ is a good geometric family of Lie--Rinehart pairs when we take for $U_0^{sm}$ the strict locus of $f_0: U_0 \to S_0$ as above. We construct a system of geometric deformations of Lie--Rinehart pairs $\D^{lrp}$ as follows: When $V_{\alpha;A} \to S_A$ is a local model from $\D$, then we take 
$$\cL\R\cP_{\alpha;A}^\bullet = \cL\R\cP^\bullet_{V_{\alpha;A}/S_A}(\cO_{V_{\alpha;A}}^{\oplus r}).$$
It makes sense to define $\cL\R\cP_{\alpha;A}^E$ as above because $\E_0$ is trivial on $V_\alpha$ by assumption.
We take the restriction maps $\rho_{\alpha;BB'}$ induced from $\D$. Over $S_0$, we have to \emph{choose} a trivialization $\E_0|_\alpha \cong \cO_{V_{\alpha;\kk}}^{\oplus r}$ to obtain $\cL\R\cP_0^E|_\alpha \cong \cL\R\cP_{\alpha;\kk}^E$. The definition of the comparison isomorphisms $\psi_{\alpha\beta;A}$ needs special care. Namely, if we just took the comparison isomorphisms from $\cO_{V_{\alpha;A}}$ for $\cO_{V_{\alpha;A}}^{\oplus r}$, then they would not give the correct map over $S_0$. Instead, we have to lift the isomorphism 
$$\cO_{V_{\alpha;\kk}}^{\oplus r}|_{\alpha\beta} \cong \E_0|_{\alpha\beta} \cong \cO_{V_{\beta;\kk}}^{\oplus r}|_{\alpha\beta}$$
order by order along $A_k = \kk[Q]/\m_Q^{k + 1}$ and then take the base change to an arbitrary $A$.
The cocycles are gauge transforms because we have gauge transforms on the level of the underlying enhanced generically log smooth family, and any compatible automorphism of the vector bundle then gives rise to a gauge transform in $\G^{-1}(\E)$.

\begin{prop}
 Let $f_0: X_0 \to S_0$ be an enhanced generically log smooth family, let $\E_0$ be a vector bundle on $X_0$, and let $\D$ be a system of deformations subordinate to an $\E_0$-admissible open cover $\V = \{V_\alpha\}_\alpha$ of $X_0$. Let $\D^{lrp}$ be the associated system of geometric deformations of Lie--Rinehart pairs. Then the natural map 
 $$\mathrm{ELD}^\D_{X_0/S_0}(\E_0) \to \mathrm{GDef}^{\D^{lrp}}(\cL\R\cP_0^\bullet,-), \quad (f_A: (X_A,\E_A) \to S_A) \mapsto \cL\R\cP^\bullet_{X_A/S_A}(\E_A)$$
 is an isomorphism of functors of Artin rings.
\end{prop}
\begin{proof}
 The map $f_A: X_A \to S_A$ together with $|U_A| = |U_0|$ and $|U_A^{sm}| = |U_0^{sm}|$ as topological spaces, and with $\cL\R\cP^\bullet_{X_A/S_A}(\E_A)$ is a good geometric family of Lie--Rinehart pairs, essentially by Lemma~\ref{LRP-gen-log-vector-bundle}. The map $f_0 \to f_A$ induces a map $\cL\R\cP^\bullet_{X_A/S_A}(\E_A) \to \cL\R\cP^\bullet(X_0/S_0,\E_0)$ which exhibits $\cL\R\cP^\bullet_{X_A/S_A}(\E_A)$ as a geometric \emph{deformation} of Lie--Rinehart pairs. Now let $\chi_\alpha: X_A|_\alpha \cong V_{\alpha;A}$ be the local isomorphism which turns $f_A$ into a deformation of type $\D$. By assumption, we have $\E_0|_\alpha \cong \cO_{X_0}^{\oplus r}|_\alpha$. This isomorphism can be lifted to an isomorphism $\E_A|_\alpha \cong \cO_{X_A}^{\oplus r}|_\alpha$ because $H^0(V_\alpha,\cO_{X_A}) \to H^0(V_\alpha,\cO_{X_0})$ and hence $H^0(V_\alpha,\E_A) \to H^0(V_\alpha,\E_0)$ is surjective. This induces an isomorphism $\cL\R\cP^\bullet_{X_A/S_A}(\E_A)|_\alpha \cong \cL\R\cP^\bullet_{\alpha;A}$ of geometric deformations of Lie--Rinehart pairs. The comparison maps between $\chi_\alpha$ and $\chi_\beta$ are gauge transforms because they are induced from comparisons of enhanced generically log smooth deformations with a vector bundle. Thus we have an element in $\mathrm{GDef}^{\D^{rlp}}(\cL\R\cP_0^\bullet,A)$. When we choose different trivializations of $\E_0$, different lifts of the trivialization to $A$, or different isomorphisms $\chi_\alpha: X_A|_\alpha \cong V_{\alpha;A}$, then the geometric deformations of Lie--Rinehart pairs of type $\D^{rlp}$ are all equivalent. Similarly, if $f_A: (X_A,\E_A) \to S_A$ is isomorphic to $f'_A: (X_A',\E_A') \to S_A$, we obtain isomorphic geometric deformations of Lie--Rinehart pairs of type $\D^{rlp}$. Thus, our map is well-defined on the level of objects. It is then easy to see that it is in fact a natural transformation. The proof of bijectivity is analogous to \cite[Prop.~4.2]{Felten2022}.
\end{proof}

\subsection{Systems of deformations for isolated normal hypersurface singularities}\label{sys-defo-hypersurface-sing}\note{sys-defo-hypersurface-sing}

In this section, we briefly indicate how our notion of system of Gerstenhaber deformations can also, in principle, be applied to classical deformation problems.

\begin{ex}
 Let $0 \not= F(x_0,...,x_d) \in \kk[x_0,...,x_d]$ be a polynomial such that $\{F = t\}$ is smooth for $t \not = 0$ and is normal with a single isolated singularity in $\{0\}$ for $t = 0$. Let $f: X = \bAA^{d + 1} \to \bAA^1_t = S, \: t \mapsto F$, be the associated family; it is flat, and, by assumption, smooth on $U = X \setminus \{0\}$. Since $X$ is Cohen--Macaulay, all fibers are Cohen--Macaulay. They are reduced by assumption. Let $j: U \to X$ be the inclusion, and let $(\G_S^\bullet,\A_S^\bullet) = (j_*\Theta^\bullet_{U/S},j_*\Omega^\bullet_{U/S})$ be the direct image of the Gerstenhaber calculus of a smooth morphism (of course with $[-,-] = -[-,-]_{sn}$). The sheaves $\G_S^p$ and $\A_S^i$ are locally free of rank $d$ on $U$, so they are torsion-free and hence flat over the Dedekind domain $\kk[t]$. Since $\A_S^d$ is a reflexive sheaf of rank $1$ on the regular scheme $X$, it is a line bundle. Then Lemma~\ref{central-fiber-injection} shows that we obtain a system of deformations over $\Lambda = \kk\llbracket t\rrbracket$ by setting $(\G_A^\bullet,\A_A^\bullet) := (\G_S^\bullet \otimes_S S_A,\, \A_S^\bullet \otimes_S S_A)$.
\end{ex}



\chapter{The characteristic algebra $E^{\bullet,\bullet}_{X_0/\Lambda}$}\label{char-alg-constr-sec}\note{char-alg-constr-sec}
\index{characteristic algebra $E^{\bullet,\bullet}_{X_0/\Lambda}$}

In this chapter, we work in the following situation.

\begin{sitn}\label{geom-P-defo-sitn}\note{geom-P-defo-sitn}
 We fix a complete local Noetherian $\kk$-algebra $\Lambda$ with residue field $\kk$. We have an algebraic structure $\cP$  carrying a Cartan structure, and $f_0: (X_0,\E_0^\bullet) \to S_0$ is a geometric family of $\cP$-algebras. We have an open cover $\V = \{V_\alpha\}_\alpha$ of $X_0$ which is admissible with respect to $\E_0^\bullet$, and $\D$ is a system of deformations for $f_0$ subordinate to $\V$ over $\Lambda$. We have another open cover $\U = \{U_i\}_i$ of $X_0$ such that the union $\U \cup \V$ is admissible with respect to $\E_0^\bullet$.
\end{sitn}

In this situation, we construct a $\Lambda$-linear (Cartanian) $\cP^{crv}$-pre-algebra (in the sense of Definition~\ref{L-curved-def})
$$E^{\bullet,\bullet}_{X_0/\Lambda} := E^{\bullet,\bullet}(X_0/S_0,\E_0^\bullet,\V,\U,\D),$$
the \emph{characteristic algebra} of $f_0: (X_0,\E_0^\bullet) \to S_0$ with respect to $\V,\U,\D$. It depends on further choices which are suppressed in the notation. Then 
$$L^\bullet_{X_0/\Lambda} := E^{T,\bullet}_{X_0/\Lambda}$$
is a $\Lambda$-linear curved Lie algebra which controls the deformation functor $\mathrm{GDef}^\D(\E_0^\bullet,-)$. If $f_A: (X_A,\E_A^\bullet) \to S_A$ is the deformation corresponding to $\phi \in \m_A \cdot (L^1_{X_0/\Lambda} \otimes_\Lambda A)$, then 
$$H^k(X_A,\E_A^P) = H^k(E_{X_0/\Lambda}^{P,\bullet} \otimes_\Lambda A,\: \bar\partial + \nabla^P_\phi).$$
The most important special case of Situation~\ref{geom-P-defo-sitn} is the case of (two-sided) Gerstenhaber calculi.

\begin{sitn}\label{geom-G-calc-sitn}\note{geom-G-calc-sitn}
 We fix a complete local Noetherian $\kk$-algebra $\Lambda$ with residue field $\kk$. We have a good geometric family $f_0: (X_0,\G_0^\bullet,\A_0^\bullet) \to S_0$ of two-sided Gerstenhaber calculi of relative dimension $d \geq 1$, as defined in Definition~\ref{good-G-calc-defn}. We have an admissible open cover $\V = \{V_\alpha\}_\alpha$  of $X_0$, and $\D$ is a system of deformations for $f_0$ subordinate to $\V$. Each deformation in $\D$ is good by Remark~\ref{good-defo-rem}. We have another open cover $\U = \{U_i\}_i$ of $X_0$ such that $\U \cup \V$ is admissible with respect to $(\G_0^\bullet,\A_0^\bullet)$.
\end{sitn}

In this situation, the characteristic algebra is actually a $\Lambda$-linear curved two-sided Gerstenhaber calculus; we denote it by 
\begin{equation}\label{PV-DR-line}
 (PV_{X_0/\Lambda}^{\bullet,\bullet},DR_{X_0/\Lambda}^{\bullet,\bullet}).
\end{equation}
In Situation~\ref{geom-G-calc-sitn}, $\A^d_0$ is a line bundle. If $\A_0^d \cong \cO_{X_0}$, then we say $f_0$ is \emph{virtually Calabi--Yau}.\index{virtually Calabi--Yau} In this case, \eqref{PV-DR-line} is a $\Lambda$-linear curved two-sided Batalin--Vilkovisky calculus by Lemma~\ref{PV-DR-BV-struc}. Consequently, if \eqref{PV-DR-line} is quasi-perfect, then $\mathrm{GDef}^\D(\G_0^\bullet,\A_0^\bullet,-)$ is unobstructed by Theorem~\ref{second-abstract-unob-thm}. In particular, this applies in our situation of main interest.

\begin{sitn}\label{enh-gen-log-sm-sitn}\note{enh-gen-log-sm-sitn}
 We fix a sharp toric monoid $Q$ and set $\Lambda := \kk\llbracket Q\rrbracket$; we set $S_0 := \Spec (Q \to \kk)$. We have a torsionless enhanced generically log smooth family $f_0: X_0 \to S_0$, an admissible open cover $\V = \{V_\alpha\}_\alpha$ of $X_0$ with admissibility as specified in Definition~\ref{enhanced-sys-of-defo}, and a system of deformations $\D$ subordinate to $\V$ in the sense of Definition~\ref{enhanced-sys-of-defo}. Furthermore, we have another open cover $\U = \{U_i\}_i$ of $X_0$ such that $\U \cup \V$ is admissible. Similar to Example~\ref{log-Gorenstein-good-geom-fam-ex}, when taking the associated geometric families of two-sided Gerstenhaber calculi $\G\C^\bullet_{X_A/S_A}$ and the system of deformations $\D^{gc}$, we are in Situation~\ref{geom-G-calc-sitn}. We have an isomorphism 
 $$\mathrm{ELD}^\D_{X_0/S_0} \cong \mathrm{GDef}^{\D^{gc}}(\G_0^\bullet,\A_0^\bullet,-)$$
 of functors of Artin rings.
\end{sitn}

In this situation, if $f_0: X_0 \to S_0$ is log Calabi--Yau, then \eqref{PV-DR-line} is a $\Lambda$-linear curved two-sided Batalin--Vilkovisky calculus which controls the deformation functor $\mathrm{ELD}^\D_{X_0/S_0}$.

In another direction, we can also construct a $\Lambda$-linear curved Lie--Rinehart pair which controls enhanced generically log smooth deformations with a vector bundle.

\begin{sitn}\label{enh-vec-bdl-defo-sitn}\note{enh-vec-bdl-defo-sitn}
 We fix a sharp toric monoid $Q$ and set $\Lambda := \kk\llbracket Q\rrbracket$; we set $S_0 := \Spec (Q \to \kk)$. We have a torsionless enhanced generically log smooth family $f_0: X_0 \to S_0$, a vector bundle $\E_0$ on $X_0$ of rank $r$, an $\E_0$-admissible open cover $\V = \{V_\alpha\}_\alpha$ of $X_0$,\footnote{Here, $\E_0$-admissibility means that, on the one hand side, $\V$ is weakly $\E_0$-admissible as used in Definition~\ref{enhanced-sys-defo-vector-bdl}, and, on the other hand side, $\V$ is admissible as specified in Definition~\ref{enhanced-sys-of-defo}.}\index{$\E_0$-admissible} and a system of deformations $\D$ subordinate to $\V$ in the sense of Definition~\ref{enhanced-sys-defo-vector-bdl}. Furthermore, we have another open cover $\U = \{U_i\}_i$ of $X_0$ such that $\U \cup \V$ is $\E_0$-admissible. When taking the Lie--Rinehart pairs $\cL\R\cP^\bullet_{X_A/S_A}(\E_A)$ and the system of deformations $\D^{rlp}$, we are in Situation~\ref{geom-P-defo-sitn} since the cohomology of $\G^{-1}_{X_0/S_0}(\E_0)$ vanishes on each $V_{\alpha_1} \cap ... \cap V_{\alpha_s}$. We have an isomorphism 
 $$\mathrm{ELD}^\D_{X_0/S_0}(\E_0) \cong \mathrm{GDef}^{\D^{rlp}}(\cL\R\cP^\bullet_0,-)$$
 of functors of Artin rings.
\end{sitn}

The main part of this chapter is an adaptation of \cite[\S\S 5 - 8]{Felten2022}, generalizing it from the case of Gerstenhaber algebras to $\cP$-algebras, in particular two-sided Gerstenhaber calculi.

\section{The Thom--Whitney resolution}\label{TW-reso-sec}\note{TW-reso-sec}
\index{Thom--Whitney resolution}

The Thom--Whitney resolution is an acyclic resolution of complexes which is well-suited to preserve additional algebraic structures on these complexes---exactly what we want to have when going from a sheaf of $\cP$-algebras to a curved $\cP$-algebra. In its present form, they go back to \cite{AznarHodgeDeligne1987}; it seems that their first use in algebraic infinitesimal deformation theory is in  Iacono--Manetti's algebraic proof of the Bogomolov--Tian--Todorov theorem in \cite{AlgebraicBTT2010}, where the reader can find extensive references to their prior use in topology. Of course, they feature prominently in Chan--Leung--Ma's method of deforming log spaces \cite{ChanLeungMa2023}. In \cite{Felten2022}, we gave a recollection of its basic properties, which we summarize again for convenience and to fix notations.

We denote the category of sets $[n] = \{0,...,n\}$ with order-preserving injections as morphisms by $\Delta_\mathrm{mon}$; then a \emph{semi\underline{co}simplicial}\index{semisimplicial} object in a category $\C$ is a covariant functor $\Delta_\mathrm{mon} \to \C$. Explicitly, a semicosimplicial object is a list of objects $A_0, A_1, ...$ together with $n + 1$ maps $\partial_{k,n}: A_{n - 1} \to A_n$ for each $n \geq 1$ satisfying $\partial_{\ell, n + 1}\partial_{k,n} = \partial_{k + 1,n + 1}\partial_{\ell,n}$.

A \emph{semisimplicial} object is a contravariant functor $\Delta_\mathrm{mon} \to \C$. For us, the most important semisimplicial object is the semisimplicial differential graded commutative algebra $(A_{PL})_n$ formed by the dgca of (global algebraic) differential forms on $H_n = \{t_0 + ... + t_n = 1\} \subset \bAA^{n + 1}_\kk$ together with maps $\delta^{k,n}: (A_{PL})_n \to (A_{PL})_{n - 1}$ induced by pull-back along inclusions of coordinate hyperplanes in $\bAA^{n + 1}$.

If $V^\Delta$ is a semicosimplicial complex of vector spaces---here, each $V_n$ is a complex $(V_n^\bullet,d)$ of vector spaces---then we can form the \emph{Thom--Whitney bicomplex}\index{Thom--Whitney resolution!Thom--Whitney bicomplex}
$$C_\mathrm{TW}^{i,j}(V^\Delta) = \left\{(x_n)_{n \in \NN} \in \prod_{n \in \NN} (A_{PL})_n^i \otimes_\kk V_n^j \enspace \Big| \enspace \forall k,n:\: (\delta^{k,n} \otimes \mathrm{Id})(x_n) = (\mathrm{Id} \otimes \partial_{k,n})(x_{n - 1})  \right\}.$$
It comes with two natural differentials
$$\delta_1((a_n \otimes v_n)_n) = (da_n \otimes v_n)_n \quad \mathrm{for} \quad (a_n \otimes v_n)_n \in C_\mathrm{TW}^{i,j}(V^\Delta),$$ induced from the differential of $(A_{PL})_n$, and 
$$\delta_2((a_n \otimes v_n)_n) = (-1)^i(a_n \otimes dv_n)_n, \quad \mathrm{for} \quad (a_n \otimes v_n)_n \in C_\mathrm{TW}^{i,j}(V^\Delta)$$ induced from the differential of $V_n^\bullet$. The Thom--Whitney bicomplex is an exact functor, i.e., if 
$$0 \to U^\Delta \to V^\Delta \to W^\Delta \to 0$$
is exact, then 
$$0 \to C^{i,j}_\TW(U^\Delta) \to C_\TW^{i,j}(V^\Delta) \to C_\TW^{i,j}(W^\Delta) \to 0$$
is exact as well for all $i,j$. We shall also need the following result.

\begin{lemma}\label{TW-constr-bounded-filtered-colimit}\note{TW-constr-bounded-filtered-colimit}
 Let $V^\Delta_{m}$ be a directed system of semicosimplicial complexes of vector spaces, and let $V^\Delta$ be the colimit in the sense that each $V^j_n$ is the colimit of the $V^j_{m;n}$. Assume that the system $V^\Delta_m$ is bounded in the sense that there is $N > 0$ such that $V^j_{m;n} = 0$ for all $n > N$ and all $m$ and $j$. Then the canonical map 
 $$\varinjlim_m C^{i,j}_\TW(V_m^\Delta) \to C^{i,j}_\TW(V^\Delta)$$
 is an isomorphism of complexes of vector spaces.
\end{lemma}
\begin{proof}
 It is easy to show that $V^\Delta$ is a semicosimplicial complex of vector spaces as well, so the statement makes sense. The proof of the main statement is straightforward once we use that each element either in $C^{i,j}_\TW(V^\Delta_m)$ or $C^{i,j}_\TW(V^\Delta)$ is represented by a \emph{finite} sequence $(x_n)_n$ with $x_n \in (A_{PL})_n^i \otimes V_{m;n}^j$ respective $(A_{PL})_n^i \otimes V_n^j$, and that $(A_{PL})_n^i \otimes V_n^j$ is the colimit of the system $(A_{PL})_n^i \otimes V^j_{m;n}$.
\end{proof}

\par\vspace{\baselineskip}

Let $S/\kk$ be a Noetherian scheme defined over $\kk$, and let $f: X \to S$ be a morphism of Noetherian schemes. Let $\U = \{U_i\}_i$ be a pre-admissible open cover of $X$ in the sense of Definition~\ref{admissible-def}, and let $\F$ be a quasi-coherent sheaf on $X$.
Then we can form the \emph{\v{C}ech semicosimplicial sheaf} $\F(\U)$ as in \cite[Ex.~5.1]{Felten2022}---its objects are the quasi-coherent sheaves 
$$\F(\U)_n = \prod_{i_0 < ... < i_n} j_*\F|_{U_{i_0} \cap ... \cap U_{i_n}},$$
and the maps are the usual \v{C}ech maps. This semicosimplicial object is \emph{bounded} in the sense that $\F(\U)_n = 0$ for $n >> 0$. Since $S$ is a $\kk$-scheme, the sections $\Gamma(V,\F(\U)_n)$ form a $\kk$-vector space on each open $V \subseteq X$. By considering $\F(\U)$ as a semicosimplicial complex concentrated in degree $0$, we can form its Thom--Whitney bicomplex $C_\mathrm{TW}^{\bullet,\bullet}(\Gamma(V,\F(\U)))$. They form a presheaf of $\kk$-vector spaces on $X$ which is actually a sheaf, as explained e.g.~in \cite[Constr.~5.3]{Felten2022}. Moreover, we have a natural $\cO_{X}$-action which turns each $C_\mathrm{TW}^{i,0}(\F(\U))$ into a quasi-coherent sheaf.
The following fact is not discussed in \cite{Felten2022}.

\begin{lemma}\label{TW-flat-sheaf-lemma}\note{TW-flat-sheaf-lemma}
 If $\F$ is flat over $S$, then $C_\TW^{i,0}(\F(\U))$ is flat over $S$ as well.
\end{lemma}
\begin{proof}
Assume that $S = \Spec(R)$ is affine. Using that $U_i \to X$ is an affine morphism, we find that each $\F(\U)_n$ is flat over $S$. Then, if $V \subseteq X$ is affine, we get that the (actually finite) product 
$$\prod_{n \in \NN} (A_{PL})_n^i \otimes_\kk \Gamma(V,\F(\U)_n)$$
is flat over $S$. Similar to the proof of \cite[Lemma~5.2]{Felten2022}, we find 
$$\Gamma(V,\,C_\TW^{i,0}(\F(\U))\,) \otimes_R Q = \Gamma(V,\,C_\TW^{i,0}((\F \otimes_R Q)(\U))\,)$$
for every finitely presented $R$-module $Q$. This allows us to conclude that the inclusion 
$$\Gamma(V,\,C_\TW^{i,0}(\F(\U))\,) \to \prod_{n \in \NN} (A_{PL})_n^i \otimes_\kk \Gamma(V,\F(\U)_n)$$
is universally injective so that the left term is flat over $R$.
\end{proof}

The first differential 
$$\delta_1: C_\mathrm{TW}^{i,0}(\F(\U)) \to C_\mathrm{TW}^{i + 1,0}(\F(\U))$$
is $\cO_X$-linear, and this, together with the canonical map 
$$\F \to C_\mathrm{TW}^{0,0}(\F(\U)), \quad f \mapsto (1 \otimes (f|_{U_{i_0} \cap ... \cap U_{i_n}})_{i_0 < ... < i_n})_n,$$
turns $C_\mathrm{TW}^{\bullet,0}(\F(\U))$ into a resolution of $\F$. Namely, this complex is quasi-isomorphic (on the level of sections over any open subset) to the \v{C}ech complex $\check \C^\bullet(\U,\F)$ by \cite[Thm.~2.14]{AznarHodgeDeligne1987}.

\begin{defn}
 Fix a pre-admissible cover $\U$ of $X$. Then the \emph{Thom--Whitney resolution} of a quasi-coherent sheaf $\F$ is the complex $\TW^\bullet(\F) := C^{\bullet,0}_\TW(\F(\U))$ of quasi-coherent sheaves with $\cO_X$-linear differentials, which is quasi-isomorphic to $\F$ via the canonical map $\F \to \TW^0(\F)$. This defines a functor.
\end{defn}
\begin{rem}
 The Thom--Whitney resolution is actually defined on the level of sheaves of $\CC$-vector spaces. The $\cO_X$-module structure comes in addition to that; the Thom--Whitney resolution commutes with the forgetful functor from (quasi-coherent) $\cO_X$-modules to abelian sheaves.
\end{rem}
\begin{rem}
 When applying the Thom--Whitney resolution to a complex $(\F^\bullet,\partial)$ of sheaves, we have to apply our above sign convention: The map $\TW^i(\partial)$ acquires a sign $(-1)^i$. Similarly, if $h: \F^\bullet \to \G^\bullet$ is a map of degree $|h|$ between complexes of sheaves, then our convention will be that $\TW^i(h)$ acquires a sign $(-1)^{|h|\cdot i}$.
\end{rem}

\begin{lemma}\label{TW-exact}\note{TW-exact}\index{Thom--Whitney resolution!exactness}
 The functor $\TW^i(-)$ is exact on quasi-coherent sheaves\footnote{The lemma is probably not true in general for arbitrary abelian sheaves.} for all $i \geq 0$.
\end{lemma}
\begin{proof}
 Each embedding $U_{i_0} \cap ... \cap U_{i_n} \to X$ is an affine open immersion, so the direct image is an exact functor on quasi-coherent sheaves. Thus, a short exact sequence 
 $$0 \to \E \to \F \to \G \to 0$$
 of quasi-coherent sheaves gives rise to a short exact sequence 
 $$0 \to \E(\U)_n \to \F(\U)_n \to \G(\U)_n \to 0,$$
 so we have a short exact sequence of semicosimplicial $\cO_X$-modules. Then, on each affine open subset of $X$, the above mentioned exactness of the Thom--Whitney construction on semicosimplicial vector spaces yields the exactness of 
 $$0 \to \TW^i(\E) \to \TW^i(\F) \to \TW^i(\G) \to 0.$$
\end{proof}

\begin{lemma}\label{TW-filtered-colimit}\note{TW-filtered-colimit}\index{Thom--Whitney resolution!filtered colimits}
 The functor $\TW^i(-)$ preserves filtered colimits of quasi-coherent sheaves for all $i \geq 0$.
\end{lemma}
\begin{proof}
 On a Noetherian scheme, a filtered colimit $\F = \varinjlim_m \F_m$ of quasi-coherent sheaves has the property that $\varinjlim_m \Gamma(W,\F_m) = \Gamma(W,\F)$ for all open subsets $W \subseteq X$. Since $\U$ is a finite open cover, there is an $N$ such that $\F_m(\U)_n = 0$ for all $n > N$ and all $m$. Then we have $\varinjlim_m \Gamma(W,\TW^i(\F_m)) = \Gamma(W,\TW^i(\F))$ for all open subsets $W \subseteq X$ by Lemma~\ref{TW-constr-bounded-filtered-colimit}. In particular, $\varinjlim_m\TW^i(\F_m) = \TW^i(\F)$ as sheaves.
\end{proof}

If $\U$ is admissible\footnote{In analogy with our prior definitions, we say that an open cover $\U = \{U_i\}_i$ of $X$ by finitely many affine open immersions $j_i: U_i \to X$ is \emph{admissible} with respect to a family $\{\F_k\}_k$ of quasi-coherent sheaves if $H^\ell(U_{i_0} \cap ... \cap U_{i_n},\F_k) = 0$ for all $\ell \geq 1$, all indices $i_0,...,i_n$, and all $\F_k$.} with respect to $\F$, then $\TW^\bullet(\F)$ computes the cohomology of $\F$ by \cite[0FLH]{stacks}. Also, in this case, each $\TW^i(\F)$ is acyclic for $\Gamma(X,-)$. Unfortunately, the proof of acyclicity in \cite[Lemma~5.6]{Felten2022} cannot be easily generalized to a non-affine open cover $\U$; instead, we have the following more conceptual proof.

\begin{lemma}\label{gen-acyc-TW-lemma}\note{gen-acyc-TW-lemma}
 Assume that $\U$ is an admissible open cover of $X$ with respect to $\F$. Let $V \subseteq X$ be an open subset with $H^\ell(V \cap U_{i_0} \cap .. \cap U_{i_n}, \F) = 0$ for all indices $i_0,...,i_n$. Then 
 $H^\ell(V,\TW^k(\F)) = 0$
 for $\ell \geq 1$; hence, we have 
 $$H^\ell(V,\F) = H^\ell\left(\, \Gamma\big(V,\TW^\bullet(\F)\big),\,\delta_1\right).$$
 In particular, this holds for $V = X$ and $V = U_{j_0} \cap ... \cap U_{j_m}$.
\end{lemma}
\begin{proof}
 Since $X$ is (locally) Noetherian, the injective objects of $\mathrm{QCoh}(X)$ are precisely those injective objects of $\mathrm{Mod}(X)$ which are quasi-coherent. This is a consequence of \cite[Ch.~II,\, Thm.~7.18]{Hartshorne1966}, as pointed out by A.~Preygel on math\textit{overflow}. Now let $\F \to \I^\bullet$ be a quasi-coherent injective resolution of $\F$. Then each 
 $$\I^m(\U)_n = \prod_{i_0 < ... < i_n}j_*\I^m|_{U_{i_0} \cap ... \cap U_{i_n}}$$
 is a quasi-coherent and injective sheaf as well; indeed, the restriction is injective, and then the direct image is injective by \cite[02N5]{stacks}.  In particular, each $\I^m(\U)_n$ is a flasque sheaf by \cite[Lemma~2.4]{Hartshorne1977}. Since the Thom--Whitney resolution is an exact functor, this implies that $\TW^k(\I^m)$ is flasque as well because it is constructed by application of the functor to each open subset. Since $\TW^k(-)$ is an exact functor on quasi-coherent sheaves, we have an exact sequence 
 $$0 \to \TW^k(\F) \to \TW^k(\I^0) \to \TW^k(\I^1) \to ...$$
 of quasi-coherent sheaves. Thus, $\TW^k(\I^\bullet)$ computes the cohomology of $\TW^k(\F)$.
 
 Since each $U_i \to X$ is an affine morphism, we have
 \begin{equation}\label{cohom-FU-comp}
  H^\ell(V,\F(\U)_n) = \bigoplus_{i_0 < ... < i_n} H^\ell(V \cap U_{i_0} \cap ... \cap U_{i_n},\F) = 0
 \end{equation}
 for $\ell \geq 1$. Since $\I^\bullet(\U)_n$ computes the cohomology of $\F(\U)_n$, this shows that 
 $$0 \to H^0(V,\F(\U)_n) \to H^0(V,\I^0(\U)_n) \to H^0(V,\I^1(\U)_n) \to ...$$
 is exact. But then 
 $$0 \to H^0(V,\TW^k(\F)) \to H^0(V,\TW^k(\I^0)) \to H^0(V,\TW^k(\I^1)) \to ...$$
 is exact as well due to the exactness of the Thom--Whitney construction on semicosimplicial vector spaces, and we find $H^\ell(V,\TW^k(\F)) = 0$ for $\ell \geq 1$.
\end{proof}
\begin{rem}
 This is not true for the cohomology on a general open subset $V \subseteq X$, even if each $U_i$ is affine.
\end{rem}

If $b: T \to S$ is an affine morphism of finite type, and if $c: Y = X \times_T S \to X$, then $c^{-1}\U$ is an open cover of $Y$, and we have a natural isomorphism
\begin{equation}\label{TW-base-change-formula}
 c^*C_\mathrm{TW}^{i,0}(\F(\U)) \cong C_\mathrm{TW}^{i,0}((c^*\F)(c^{-1}\U))
\end{equation}
of quasi-coherent sheaves by \cite[Lemma~5.2]{Felten2022} and \cite[02KG]{stacks}. However, in general, $c^{-1}(\U)$ may not be admissible with respect to $c^*\F$ so that $c^*C_\TW^{\bullet,0}(\F(\U))$ may not compute the cohomology of $c^*\F$. Of course, if $\U$ is an affine cover, then $c^{-1}\U$ is affine as well and hence admissible. The following case is also of interest for us.

\begin{lemma}
 Assume that $S = \Spec A$ is the spectrum of an Artinian local $\kk$-algebra $A$ with residue field $\kk$, and let $b: S_0 \to S$ be the inclusion induced by $A \to \kk$. Let $c: X_0 \to X$  be the inclusion of the central fiber. Assume that $f: X \to S$ is flat and separated, and assume that $\F$ is flat over $S$. Then an open cover $\U$ of $X$ by finitely many affine open immersions is admissible with respect to $\F$ if and only if $c^{-1}\U$ is admissible with respect to $\F_0 := c^*\F$.
\end{lemma}
\begin{proof}
 Use \cite[Thm.~0.4]{Wahl1976} to compare $H^\ell(U_{i_0} \cap ... \cap U_{i_n},\F)$ with $H^\ell(U_{i_0} \cap ... \cap U_{i_n},\F_0)$.
\end{proof}

\subsection*{The Thom--Whitney resolution of a geometric deformation of $\cP$-algebras}
\index{Thom--Whitney resolution!geometric deformations of $\cP$-algebras}

Let $\cP$ be an algebraic structure carrying a Cartan structure, and let $f_0: (X_0,\E_0^\bullet) \to S_0$ be a geometric family of $\cP$-algebras. Let $\U = \{U_i\}_i$ be an open cover of $X_0$ which is admissible with respect to $\E_0^\bullet$. Let $W_0 \subseteq X_0$ be some Zariski open subset, and let $f: (W,\E^\bullet) \to S$ be a geometric deformation of $f_0|_{W_0}$ over $S = \Spec A$ for $A \in \mathbf{Art}_\Lambda$. We define 
 $$\TW^{P,q}(\E^\bullet) := C^{q,0}_\TW(\E^P(\U))$$
 for $q \geq 0$ and $P \in D(\cP)$; this is a quasi-coherent sheaf on $X$, flat over $S$. The natural map 
 $$\E^P \to \TW^{P,\bullet}(\E^\bullet)$$
 is a resolution of $\E^P$, but it is not necessarily acyclic and does not necessarily compute the cohomology of $\E^P$ unless $W_0 = X_0$.\footnote{This is because we may have $H^\ell(W \cap U_{i_0} \cap ... \cap U_{i_n},\E^P) \not= 0$.} We denote the differential by 
 $$\bar\partial: \TW^{P,q}(\E^\bullet) \to \TW^{P,q + 1}(\E^\bullet);$$
 it is $\cO_W$-linear. Every constant $\gamma \in \cP(P)$ gives rise to a constant $\gamma \in \TW^{P,0}(\E^\bullet)$ via $\E^P \to \TW^{P,0}(\E^\bullet)$. Every unary operation $\mu \in \cP_N(P;Q)$ gives rise to an unary operation 
 $$\mu: \TW^{P,q}(\E^\bullet) \to \TW^{Q,q}(\E^\bullet), \quad (a_n \otimes p_n)_n \mapsto (-1)^{q\cdot |\mu|} \cdot (a_n \otimes \mu(p_n))_n,$$
 with $|\mu| := |Q| - |P|$.
 The choice of the sign is compatible with our prior convention for the sign of the second differential in the Thom--Whitney bicomplex of a semicosimplicial complex of vector spaces. It can be interpreted as swapping $a_n$ and $\mu$. Every binary operation $\mu \in \cP_N(P,Q;R)$ gives rise to a binary operation
 \begin{align}
  \mu: \: &\TW^{P,q}(\E^\bullet) \times \TW^{Q,q'}(\E^\bullet) \to \TW^{R,q + q'}(\E^\bullet), \nonumber \\ 
  &((a_n \otimes p_n)_n, (b_n \otimes q_n)_n) \mapsto (-1)^{(|P| + |\mu|) \cdot q'} \cdot ((a_n \wedge b_n) \otimes \mu(p_n,q_n))_n \ , \nonumber
 \end{align}
 with $|\mu| := |R| - |Q| - |P|$.
 This formula can be interpreted as the sign which we obtain from swapping $b_n$ and the operator $\mu(p_n,-)$, which is of degree $|R| - |Q|$. The sign convention is abstracted from \cite[Defn.~3.9]{ChanLeungMa2023} and the subsequent discussion.
 
 \begin{rem}\label{no-ternary-sign-rem}\note{no-ternary-sign-rem}
  Unfortunately, we do not know a reasonable sign convention for operations of higher arity. In a Gerstenhaber algebra, different compositions of $\wedge$ and $[-,-]$ suggest different rules for ternary operations. This is the main reason why we require $\cP(P_1,...,P_n;Q) = \emptyset$ for $n \geq 3$ in the definition of a Cartan structure.
 \end{rem}
 
 \begin{prop}
  $\TW^{\bullet,\bullet}(\E^\bullet)$ is a $Z$-faithful bounded Cartanian $\cP^{crv}$-pre-algebra in the context $\mathfrak{QCoh}(W/S)$ with $\ell = 0$. Its formation commutes with base change in $\mathbf{Art}_\Lambda$.
 \end{prop}
 \begin{proof}
  Since each $\TW^{P,q}(\E^\bullet)$ is a quasi-coherent sheaf, flat over $S$, and each operation is $A$-multilinear, we have indeed a $\cP^{crv}$-pre-algebra when setting $\ell = 0$. Since $\F(\U)_n = 0$ for $n$ exceeding the number $M$ of opens in $\U$, we have $\TW^{P,q}(\E^\bullet) = 0$ for $q > M$ since then $(A_{PL})^q_n = 0$ for all $n \leq M$; thus, $\TW^{\bullet,\bullet}(\E^\bullet)$ is (uniformly) bounded.
  
  If $\mu \in \cP_0(P;Q)$ or $\mu \in \cP_0(P,Q;R)$, then the induced map is $\cO_W$-(bi-)linear; in particular, the induced products 
  $$\ast^P: \TW^{F,q}(\E^\bullet) \times \TW^{P,q'}(\E^\bullet) \to \TW^{P,q + q'}(\E^\bullet)$$
  are $\cO_W$-bilinear. Graded commutativity (for $P = F$) and unitality of this product are easy; it is also associative because the signs of the two ways to evaluate 
  $$(a_n \otimes f_n)_n \ast^F (b_n \otimes g_n)_n \ast^P (c_n \otimes p_n)_n$$
  turn out to agree. If $\mu \in \cP_N(P;P)$, then, for $a = (a_n \otimes g_n)_n \in \TW^{F,0}(\E^\bullet)$, we find 
  $$D_{1;a}((b_n \otimes p_n)_n) = (-1)^{|b||\mu|} \big((a_n \wedge b_n) \otimes (D_{1;g_n}(p_n))\big)_n;$$
  repeated application of this formula shows that, on $\TW^{\bullet,\bullet}(\E^\bullet)$, $\mu$ is a differential operator of order $N$ with respect to $\TW^{F,0}(\E^\bullet)/A$. Similar formulae hold in case $\mu \in \cP_N(P,Q;R)$ with the factor $(-1)^{(|P| + |\mu|)|c|}$ for $\mu$ applied to $(b_n \otimes p_n)_n$ and $(c_n \otimes q_n)_n$ for both $D_{1;a}$ and $D_{2;a}$. Thus, $\mu$ is a bilinear differential operator with respect to $\TW^{F,0}(\E^\bullet)/A$ in this case. In particular, $\TW^{\bullet,\bullet}(\E^\bullet)$ is semi-Cartanian as a $\cP^{bg}$-pre-algebra in the sense of Definition~\ref{Cartanian-Pbg-def}. That it is a Cartanian $\cP^{bg}$-pre-algebra (in the sense of Definition~\ref{Cartanian-Pbg-def}) follows from an easy but tedious computation comparing signs.
  
  For $\gamma \in \cP(P)$, we have $\bar\partial(1 \otimes \gamma) = d(1) \otimes \gamma = 0$. The two formulae for $\bar\partial\mu(p)$ and $\bar\partial\mu(p,q)$ are straightforward. Finally, we have $\bar\partial^2(p) = \nabla^P_\ell(p)$ and $\bar\partial(\ell) = 0$ since $\bar\partial^2 = 0$ and $\ell = 0$. Thus, $\TW^{\bullet,\bullet}(\E^\bullet)$ is a Cartanian $\cP^{crv}$-pre-algebra in the sense of Definition~\ref{Cartan-curved-def}.
  
  The proof of $Z$-faithfulness is very similar to \cite[Lemma~6.3]{Felten2022}; we just have to choose instead of a function $f \in \cO(V \cap U_{i_0} \cap ... \cap U_{i_n})$ an appropriate section of some $\E^M$ for $M \in Z$.
  
  That the formation of the quasi-coherent sheaf $\TW^{P,q}(\E^\bullet)$ commutes with base change along $B' \to B$ in $\mathbf{Art}_\Lambda$ follows from \eqref{TW-base-change-formula}. Obviously, the formation of constants commutes with base change. Finally, on affine opens, the formation of an operation $\mu$ commutes with base change; since $\mu$ is a multilinear differential operator over both $B'$ and $B$, the operation over $B$ is the induced one from the one over $B'$ via our base change construction.
 \end{proof}
 
 \begin{rem}
  In special cases, additional relations are satisfied.
  \begin{enumerate}[label=(\arabic*)]
   \item If $\E^\bullet$ is a Lie--Rinehart algebra (respective pair), then $\TW^{\bullet,\bullet}(\E^\bullet)$ is a differential bigraded Lie--Rinehart algebra (respective pair). 
   \item If $\E^\bullet$ is a Gerstenhaber algebra (respective calculus), then $\TW^{\bullet,\bullet}(\E^\bullet)$ is a differential bigraded Gerstenhaber algebra (respective calculus). In this situation, we also write $\TW^{p,q}(\G^\bullet)$ and $\TW^{i,j}(\A^\bullet)$ for the pieces of $\TW^{\bullet,\bullet}(\E^\bullet)$. Explicitly, the bigraded $\wedge$-product and the bigraded Lie bracket on $\TW^{\bullet,\bullet}(\G^\bullet)$ are given by
 $$(a_n \otimes \theta_n)_n \wedge (b_n \otimes \xi_n)_n = (-1)^{|b||\theta|} ((a_n \wedge b_n) \otimes (\theta_n \wedge \xi_n))_n$$
 and
 $$[(a_n \otimes \theta_n)_n,(b_n \otimes \xi_n)_n] = (-1)^{(|\theta| + 1)|b|}((a_n \wedge b_n) \otimes [\theta_n,\xi_n])_n.$$
 The bigraded $\wedge$-product on $\TW^{\bullet,\bullet}(\A^\bullet)$ is given by
 $$(a_n \otimes \alpha_n)_n \wedge (b_n \otimes \beta_n)_n = (-1)^{|b||\alpha|} ((a_n \wedge b_n) \otimes (\alpha_n \wedge \beta_n))_n;$$
 the $\cO_W$-bilinear contraction map is
 $$(a_n \otimes \theta_n)_n \ \invneg\ (b_n \otimes \alpha_n)_n = (-1)^{|\theta||b|} ((a_n \wedge b_n) \otimes (\theta_n \ \invneg \ \alpha_n))_n,$$
 and the $f^{-1}(\cO_S)$-bilinear Lie derivative is
 $$\cL_{(a_n \otimes \theta_n)_n}((b_n \otimes \alpha_n)_n) = (-1)^{(|\theta| + 1)|b|} ((a_n \wedge b_n) \otimes \cL_{\theta_n}(\alpha_n))_n.$$
  \item \index{Thom--Whitney resolution!Gerstenhaber calculus} If $\E^\bullet$ is a two-sided Gerstenhaber calculus, then $\TW^{\bullet,\bullet}(\E^\bullet)$ is a differential bigraded two-sided Gerstenhaber calculus. The bigraded left contraction on $\TW^{\bullet,\bullet}(\A^\bullet)$ is given by 
  $$(a_n \otimes \theta_n) \vdash (b_n \otimes \alpha_n)_n = (-1)^{|\theta||b|} ((a_n \wedge b_n) \otimes (\theta_n \vdash \alpha_n))_n.$$
  A direct computation shows that all the relations required in Definition~\ref{two-sided-G-calc-defn} are satisfied.
   \item If $\E^\bullet = (\G^\bullet,\A^\bullet)$ is a Gerstenhaber calculus, and $\omega \in \Gamma(W,\A^d)$ is a global section which turns it into a Batalin--Vilkovisky calculus by Proposition~\ref{BV-calc-construction}, then $(1 \otimes \omega)_n \in \TW^{d,0}(\A^\bullet)$ turns $\TW^{\bullet,\bullet}(\E^\bullet)$ into a differential bigraded Batalin--Vilkovisky calculus. The operations $\kappa, v$, and $\Delta$ on $\TW^{\bullet,\bullet}(\E^\bullet)$ are induced from these operations on $\E^\bullet$. Obviously, if $\E^\bullet$ is two-sided, so is $\TW^{\bullet,\bullet}(\E^\bullet)$ since this holds for the Gerstenhaber calculi.
  \end{enumerate}
 \end{rem}

\section{Bigraded Thom--Whitney  deformations of $\cP$-algebras}\label{bg-TW-defo-P-alg-sec}\note{bg-TW-defo-P-alg-sec}

In this section, we study geometric deformations of $f_0: (X_0,\E_0^\bullet) \to S_0$ via their Thom--Whitney resolutions. So let us assume that we are in Situation~\ref{geom-P-defo-sitn}.\footnote{Under Lemma~\ref{gen-acyc-TW-lemma}, the admissibility of the \emph{union} $\U \cup \V$ ensures that a Thom--Whitney resolution with respect to $\U$ computes the cohomology on $V_\alpha$. If both covers consist of affines, this is automatic.} 
First, we can apply the Thom--Whitney resolution (with respect to $\U$) to $\E_0^\bullet$ and obtain 
$$\TW^{\bullet,\bullet}_0 := \TW^{\bullet,\bullet}(\E_0^\bullet)$$
as a Cartanian $\cP^{crv}$-pre-algebra on the central fiber. 
Next, we apply the Thom--Whitney resolution to each local model $\E_{\alpha;A}^\bullet$ and obtain the Cartanian $\cP^{crv}$-pre-algebra 
$$\TW^{\bullet,\bullet}_{\alpha;A} := \TW^{\bullet,\bullet}(\E_{\alpha;A}^\bullet)$$
on $V_{\alpha;A}$.
They come with restriction maps
$$\TW(\rho_{\alpha;BB'})^*: \TW_{\alpha;B'}^{\bullet,\bullet} \to \TW_{\alpha;B}^{\bullet,\bullet}$$
and comparison isomorphisms
$$\TW(\psi_{\alpha\beta;A})^*: \TW_{\beta;A}^{\bullet,\bullet}|_{\alpha\beta} \xrightarrow{\cong} \TW_{\alpha;A}^{\bullet,\bullet}|_{\alpha\beta}$$
on overlaps. The symbol ${}^*$ indicates that now the arrow direction is reversed compared with the original geometric arrow directions of these maps.

The Thom--Whitney resolution preserves gauge transforms. If $\theta \in I_{B'/B}\cdot \E^{T}_{B'}$ induces a gauge transform of some $\cP$-algebra $\E_{B'}^\bullet$ over $B'$, then the induced automorphism of the Thom--Whitney resolution $\TW^{\bullet,\bullet}(\E_{B'}^\bullet)$ is the gauge transform of 
$$(1 \otimes \theta)_n \in I_{B'/B} \cdot \TW^{T,0}(\E_{B'}^\bullet).$$
Compared to a general gauge transform of $\TW^{\bullet,\bullet}(\E_{B'}^\bullet)$, this gauge transform is special because it is compatible with $\bar\partial$ due to $\bar\partial((1 \otimes \theta)_n) = 0$; a general gauge transform interacts with $\bar\partial$ according to Lemma~\ref{gauge-action}. We find that the cocycle 
$$\TW(\psi_{\gamma\alpha;A})^* \circ \TW(\psi_{\beta\gamma;A})^* \circ \TW(\psi_{\alpha\beta;A})^*$$
is a gauge transform and moreover compatible with $\bar\partial$. 

\par\vspace{\baselineskip}

In analogy with \cite[Defn.~6.4]{Felten2022}, we define \emph{bigraded Thom--Whitney deformations} of $f_0:(X_0,\E_0^\bullet) \to S_0$. This gives a notion of the underlying deformation of \emph{bigraded} $\cP$-algebras when we forget the horizontal differential $\bar\partial$ of $\TW^{\bullet,\bullet}(\E^\bullet)$. Studying this object is easier than studying the actual Thom--Whitney resolutions and gives additional insight into their structures---in fact, while there may be many geometric deformations of Gerstenhaber calculi, their Thom--Whitney resolutions all have the same underlying deformation of bigraded $\cP$-algebras once we forget the horizontal differential $\bar\partial$. In \cite{Felten2022}, we called such a deformation a \emph{Thom--Whitney--Gerstenhaber deformation}.

\begin{defn}\label{bg-TW-defo-defn}\note{bg-TW-defo-defn}\index{Thom--Whitney deformation, bigraded}
 In the above situation, a \emph{bigraded Thom--Whitney deformation} over $A \in \mathbf{Art}_\Lambda$ is a Cartanian $\cP^{bg}$-pre-algebra $\cH^{\bullet,\bullet}$ in the context $\mathfrak{Flat}(|X_0|;A)$ together with a map  
 $$\rho^*: \cH^{\bullet,\bullet} \to \TW^{\bullet,\bullet}(\E_0^\bullet)$$
 of $\cP^{bg}$-pre-algebras and isomorphisms 
 $$\chi_\alpha^*: \TW_{\alpha;A}^{\bullet,\bullet}\xrightarrow{\cong}  \cH^{\bullet,\bullet}|_\alpha$$
 of $\cP^{bg}$-pre-algebras in $\mathfrak{Flat}(|V_\alpha|;A)$ which are compatible with the two restrictions $\rho^*|_\alpha$ and $\TW(\rho_{\alpha;A\kk})^*$. The composition 
 $$\TW(\psi_{\beta\alpha;A})^* \circ (\chi_\alpha^*)^{-1} \circ \chi_\beta^*: \: \TW^{\bullet,\bullet}_{\beta;A}|_{\alpha\beta} \to \TW^{\bullet,\bullet}_{\beta;A}|_{\alpha\beta}$$
 must be a gauge transform of $\TW^{\bullet,\bullet}_{\beta;A}|_{\alpha\beta}$ by some element $\theta \in I_{B'} \cdot \TW^{T,0}_{\beta;A}|_{\alpha\beta}$ which \emph{may not} satisfy $\bar\partial\theta = 0$.
 
 If $B' \to B$ is a map in $\mathbf{Art}_\Lambda$, and $(\cH_{B}^{\bullet,\bullet},{\rho}^*,(\chi_\alpha^*)_\alpha)$ and $(\cH_{B'}^{\bullet,\bullet},{\rho'}^*,({\chi'}_\alpha^*)_\alpha)$ are two bigraded Thom--Whitney deformations, then a \emph{morphism} is a map 
 $$r^*: \cH_{B'}^{\bullet,\bullet} \to \cH_B^{\bullet,\bullet}$$
 of $\cP^{bg}$-pre-algebras which is compatible with $\rho^*$ and ${\rho'}^*$, which induces an isomorphism $\cH^{P,q}_{B'} \otimes_{B'} B \cong \cH^{P,q}_B$ of sheaves of $B$-modules, and such that 
 $$({\chi'}_\alpha^*|_B)^{-1} \circ \chi_\alpha^*$$
 is a gauge transform of $\TW_{\alpha;A}^{\bullet,\bullet}$. Every morphism over the identity $A \to A$ is an isomorphism.
\end{defn}
\begin{rem}
 There is no horizontal differential $\bar\partial$ on $\cH^{\bullet,\bullet}$. Correspondingly, we do not require the comparison gauge transforms on overlaps to be compatible with $\bar\partial$. Also, this definition does not fix a scheme structure of a deformation of $X_0$ over $S_A$. The closest that we get to a structure sheaf is $\cH^{0,0}$, but this is, of course, not locally isomorphic to $\cO_{V_\alpha}$.
\end{rem}
\begin{ex}\label{TW-resolution-geom-P-def}\note{TW-resolution-geom-P-def}
 Let $f: (X_A,\E_A^\bullet) \to S_A$ be a geometric deformation of $\cP$-algebras of type $\D$. Then $\TW^{\bullet,\bullet}(\E_A^\bullet)$ has a canonical structure of a bigraded Thom--Whitney deformation induced by the comparison isomorphisms $\chi_\alpha$ that are part of a geometric deformation of type $\D$.
\end{ex}

\begin{lemma}\label{bg-TW-defo-autom-lemma}\note{bg-TW-defo-autom-lemma}
 Let $\cH^{\bullet,\bullet}$ be a bigraded Thom--Whitney deformation over $A$. Then 
 $$(\m_A \cdot \cH^{T,0}, \odot) \to (\A ut(\cH^{\bullet,\bullet}),\circ), \quad \theta \mapsto \mathrm{exp}_\theta,$$
 is an isomorphism of sheaves of groups, i.e., the automorphisms of $\cH^{\bullet,\bullet}$ as a bigraded Thom--Whitney deformation are precisely the gauge transforms.
\end{lemma}
\begin{proof}
 Every gauge transform is indeed an automorphism of $\cH^{\bullet,\bullet}$, so the map is well-defined. It is injective by $Z$-faithfulness. Locally on $V_\alpha$, the map is an isomorphism essentially due to our definition of an isomorphism of bigraded Thom--Whitney deformations.
\end{proof}
\begin{cor}
 Let $\cH_{B'}^{\bullet,\bullet}$ be a bigraded Thom--Whitney deformation over $B'$, with base change $\cH_B^{\bullet,\bullet}$ to $B$. Then every global automorphism of $\cH_B^{\bullet,\bullet}$ can be lifted to $\cH_{B'}^{\bullet,\bullet}$.
\end{cor}
\begin{proof}
 It is sufficient to assume that $B' \to B$ is a small extension. Then we have an exact sequence 
 $$0 \to \TW_0^{T,0} \otimes_\kk I \to \m_{B'} \cdot \cH_{B'}^{T,0} \to \m_B \cdot \cH_B^{T,0} \to 0.$$
 By Lemma~\ref{gen-acyc-TW-lemma}, the right hand map is surjective on global sections.
\end{proof}

As in \cite[Lemma~6.6]{Felten2022}, the we have the following result.

\begin{lemma}
 Let $B' \to B$ be a small extension in $\mathbf{Art}_\Lambda$ with kernel $I \subset B'$. Let $\cH^{\bullet,\bullet}$ be a bigraded Thom--Whitney deformation over $B$.
 \begin{enumerate}[label=\emph{(\alph*)}]
  \item Given a lifting ${\cH'}^{\bullet,\bullet}$ to $B'$, the relative automorphisms are in 
  $$H^0(X_0,\TW_0^{T,0}) \otimes_\kk I.$$
  \item Given a lifting ${\cH'}^{\bullet,\bullet}$ to $B'$, the isomorphism classes of liftings are in
  $$H^1(X_0,\TW_0^{T,0}) \otimes_\kk I.$$
  \item The obstructions to the existence of a lifting are in 
  $$H^2(X_0,\TW_0^{T,0}) \otimes_\kk I.$$
 \end{enumerate}
\end{lemma}

Since $H^\ell(X_0,\TW_0^{T,0}) = 0$ for $\ell \geq 1$ by Lemma~\ref{gen-acyc-TW-lemma}---we are now working on the whole space $X_0$, not on an open subspace where this may fail---, the obstructions vanish, and there is a single isomorphism class of a lifting.

\begin{defn}\label{char-sheaf-def}\note{char-sheaf-def}
 For $A = A_0 = \kk$, we define the \emph{characteristic bigraded sheaf} $\E_A^{\bullet,\bullet}$ as the Cartanian $\cP^{bg}$-pre-algebra $$\E_0^{\bullet,\bullet} := \E_{A_0}^{\bullet,\bullet} := \TW^{\bullet,\bullet}_0$$
 (by forgetting the differential $\bar\partial$). For $A = A_k = \Lambda/\m_\Lambda^{k + 1}$, we define 
 $\E_A^{\bullet,\bullet}$ by inductively choosing a lift $\E_{k + 1}^{\bullet,\bullet} := \E_{A_{k + 1}}^{\bullet,\bullet}$ of $\E_k^{\bullet,\bullet} = \E_{A_k}^{\bullet,\bullet}$. Then, for a general $A$, we define 
 $$\E_A^{\bullet,\bullet} := \E_{A_k}^{\bullet,\bullet} \otimes_{A_k} A$$
 as the pull-back along the canonical map $A_k \to A$ for sufficiently large $k$. 
\end{defn}
\begin{rem} \hspace{1cm}
 \begin{enumerate}[label=(\arabic*)]
  \item Since $\E_{k + 1}^{\bullet,\bullet}$ is defined as a lift of $\E_k^{\bullet,\bullet}$, we have isomorphisms $$\E_{k + 1}^{\bullet,\bullet} \otimes_{A_{k + 1}} A_k \cong \E_k^{\bullet,\bullet}$$
 as part of the choice of a lift. Then, for a general $A$, we have an induced isomorphism 
 $$\E_{k + 1}^{\bullet,\bullet} \otimes_{A_{k + 1}} A \cong \E_k^{\bullet,\bullet} \otimes_{A_k} A$$
 which does not depend on any further choice. In this sense, $\E_A^{\bullet,\bullet}$ is well-defined.
 \item If $\E_k^{\bullet,\bullet}$ and $\tilde\E_k^{\bullet,\bullet}$ are two systems of chosen lifts, then we can lift the isomorphism $\Phi_0: \E_0^{\bullet,\bullet} = \tilde\E_0^{\bullet,\bullet}$ in a non-canonical and non-unique way order by order to obtain a compatible system of isomorphisms $\Phi_k:  \E_k^{\bullet,\bullet} \cong \tilde\E_k^{\bullet,\bullet}$. After this choice, the induced isomorphisms $\Phi_A: \E_A^{\bullet,\bullet} \cong \tilde\E_A^{\bullet,\bullet}$ are unique (separately for every possible choice of $A_k \to A$).
 \item When $f_A: (X_A,\E_A^\bullet) \to S_A$ is a geometric deformation of $\cP$-algebras, then $\TW^{\bullet,\bullet}(\E_A^\bullet)$ is non-uniquely isomorphic to $\E_A^{\bullet,\bullet}$. The isomorphism can be constructed by lifting along the steps of a decomposition of $A \to \kk$ into small extensions.
 \item In the case of a Gerstenhaber algebra, we write $\cP\V^{\bullet,\bullet}_A$, and in the case of a Gerstenhaber calculus, we write additionally $\D\R^{\bullet,\bullet}_A$.
 \end{enumerate}
\end{rem}

\section{Predifferentials on bigraded Thom--Whitney deformations}
\index{Thom--Whitney deformation, bigraded!predifferentials}

In this section, we generalize \cite[\S 7]{Felten2022} from Gerstenhaber algebras to $\cP$-algebras. The local model $\TW_{\alpha;A}^{\bullet,\bullet}$ of a bigraded Thom--Whitney deformation comes with a differential 
$$\bar\partial_{\alpha;A}:\: \TW_{\alpha;A}^{P,q} \to \TW_{\alpha;A}^{P,q + 1}.$$
Every $\phi \in \m_A \cdot \TW_{\alpha;A}^{T,1}$ defines a modification $\bar\partial_{\alpha;A} + \nabla_\phi^P(-)$, which turns $\TW_{\alpha;A}^{\bullet,\bullet}$ into another Cartanian $\cP^{crv}$-pre-algebra with $\ell_{\alpha;A}(\phi) = \bar\partial_{\alpha;A}(\phi) +\frac{1}{2} \nabla^T_\phi(\phi)$. 
If $\theta \in \m_A \cdot \TW_{\alpha;A}^{T,0}$, then Lemma~\ref{gauge-action} shows that 
$$\mathrm{exp}_\theta(\bar\partial_{\alpha;A}(p) + \nabla_\phi^P(p)) = \bar\partial_{\alpha;A}(\mathrm{exp}_\theta(p)) + \nabla_{\mathrm{exp}_\theta * \phi}^P(\mathrm{exp}_\theta(p)),$$
i.e., maps of the form $\bar\partial_{\alpha;A} + \nabla^P_\phi(-)$ remain of this form under gauge transforms, albeit with a new value of $\phi$.

\begin{defn}\label{bg-TW-defo-prediff}\note{bg-TW-defo-prediff}
 Let $\cH^{\bullet,\bullet}$ be a bigraded Thom--Whitney deformation over $A \in \mathbf{Art}_\Lambda$. Then a \emph{predifferential} on $\cH^{\bullet,\bullet}$ is a (global) map $\bar\partial:\: \cH^{P,q} \to \cH^{P,q + 1}$
 with 
 $$(\chi_\alpha^*)^{-1} \circ \bar\partial|_\alpha \circ \chi_\alpha^* = \bar\partial_{\alpha;A} + \nabla^P_{\phi_\alpha}(-)$$
 for some $\phi_\alpha \in \m_A \cdot \TW_{\alpha;A}^{T,0}$. Two predifferentials $\bar\partial_1$ and $\bar\partial_2$ are \emph{gauge equivalent} if there is an automorphism $\psi^*: \cH^{\bullet,\bullet} \to \cH^{\bullet,\bullet}$ with $(\psi^*)^{-1} \circ \bar\partial_1 \circ \psi^* = \bar\partial_2$. 
\end{defn}

The element $\phi_\alpha$ is unique since $\TW_{\alpha;A}^{\bullet,\bullet}$ is $Z$-faithful. Since 
$$(\bar\partial_{\alpha;A} + \nabla^P_{\phi_\alpha}(-))^2 = \nabla^P_{\ell_{\alpha;A}(\phi_\alpha)}(-),$$
we also have $\bar\partial^2 = \nabla_{\ell_\alpha}^P(-)$ for $\ell_\alpha := \chi_\alpha^*(\ell_{\alpha;A}(\phi_\alpha)) \in \m_A \cdot \cH^{T,2}|_\alpha$. By $Z$-faithfulness, they glue to a global section $\ell \in \m_A \cdot \cH^{T,2}$ with $\bar\partial^2 = \nabla_\ell^P(-)$. This turns the triple $(\cH^{\bullet,\bullet},\bar\partial,\ell)$ into a global Cartanian $\cP^{crv}$-pre-algebra.

A predifferential is a \emph{differential} if $\bar\partial^2 = 0$. This is the case if and only if the corresponding $\ell$ vanishes, again due to $Z$-faithfulness.

We say that two pairs $(\cH^{\bullet,\bullet}_1,\bar\partial_1)$ and $({\cH}_2^{\bullet,\bullet},\bar\partial_2)$ consisting of a bigraded Thom--Whitney deformation together with a predifferential are \emph{isomorphic} if there is an isomorphism $r^*: {\cH}_2^{\bullet,\bullet} \cong \cH^{\bullet,\bullet}_1$ of bigraded Thom--Whitney deformations with $\bar\partial_2 = (r^*)^{-1} \circ \bar\partial_1 \circ r^*$. We collect a few obvious results which ensure that the functor of isomorphism classes of pairs $(\cH^{\bullet,\bullet},\bar\partial)$ is well-behaved.

\begin{lemma}\label{basic-prediff-lemma}\note{basic-prediff-lemma}
 The following hold:
 \begin{enumerate}[label=\emph{(\alph*)}]
  \item If $(\cH^{\bullet,\bullet}_1,\bar\partial_1)$ and $(\cH^{\bullet,\bullet}_2,\bar\partial_2)$ are isomorphic, then $\bar\partial_1$ is a differential if and only if $\bar\partial_2$ is a differential.
  \item On the same $\cH^{\bullet,\bullet}$, the two pairs $(\cH^{\bullet,\bullet},\bar\partial_1)$ and $(\cH^{\bullet,\bullet},\bar\partial_2)$ are isomorphic if and only if $\bar\partial_1$ and $\bar\partial_2$ are gauge equivalent.
  \item If $\cH^{\bullet,\bullet}_1$ and ${\cH}_2^{\bullet,\bullet}$ are two isomorphic bigraded Thom--Whitney deformations, and if $\bar\partial_1$ is a predifferential on $\cH_1^{\bullet,\bullet}$, then there is some predifferential $\bar\partial_2$ on ${\cH}_2^{\bullet,\bullet}$ such that the pairs become isomorphic.
  \item If $r^*: \cH_{B'}^{\bullet,\bullet} \to \cH_B^{\bullet,\bullet}$ is a morphism of bigraded Thom--Whitney deformations over $B' \to B$, and $\bar\partial'$ is a predifferential on $\cH_{B'}^{\bullet,\bullet}$, then there is a unique induced predifferential $\bar\partial$ on $\cH_B^{\bullet,\bullet}$ which commutes with $r^*$. If $\bar\partial'$ is a differential, so is $\bar\partial$.\footnote{The converse may fail. Also note that $B' \to B$ does not need to be a surjection.}
  \item If $(\cH_1^{\bullet,\bullet},\bar\partial_1)$ and $(\cH_2^{\bullet,\bullet},\bar\partial_2)$ are isomorphic pairs over $B'$, so are $\cH_1^{\bullet,\bullet} \otimes_{B'} B$ and $\cH_2^{\bullet,\bullet} \otimes_{B'} B$ with their induced predifferentials.
 \end{enumerate}

\end{lemma}

\begin{defn}
 \index{Thom--Whitney deformation, bigraded!$\mathrm{TWD}^\D(\E_0^\bullet,-)$}
 Isomorphism classes of pairs $(\cH^{\bullet,\bullet},\bar\partial)$ consisting of a bigraded Thom--Whitney deformation $\cH^{\bullet,\bullet}$ and a differential (not just a predifferential) $\bar\partial$ form a functor of Artin rings 
 $$\mathrm{TWD}^\D(\E_0^\bullet,-): \mathbf{Art}_\Lambda \to \mathbf{Set}.$$
\end{defn}

Example~\ref{TW-resolution-geom-P-def} yields a natural transformation 
$$\mathbf{tw}: \mathrm{GDef}^\D(\E_0^\bullet,-) \Rightarrow \mathrm{TWD}^\D(\E_0^\bullet,-)$$
of functors of Artin rings. Conversely, when we have a pair $(\cH^{\bullet,\bullet},\bar\partial)$ with $\bar\partial^2 = 0$, then we can take the cohomology sheaf $\E_A^P := H^0(\cH^{P,\bullet},\bar\partial)$ for each $P \in D(\cP)$. All constants and operations descend to $\E_A^\bullet$, so that it is a $\cP$-pre-algebra (in some context, say sheaves of $A$-modules on $|X_0|$). Since \cite[Lemma~7.3]{Felten2022} holds here as well---namely, using our admissibility assumption with respect to $\U \cup \V$ and Lemma~\ref{gen-acyc-TW-lemma}, we have $H^1(V_\alpha,\E_{\alpha;A}^T) = 0$, and this cohomology is computed by $(\TW_{\alpha;A}^{T,\bullet},\bar\partial)$---, the $\cP$-pre-algebra $\E_A^\bullet$ is locally isomorphic to the local model $\E_{\alpha;A}^\bullet$. Comparing two of these isomorphisms on the level of bigraded Thom--Whitney deformations with a differential shows that the isomorphisms with $(\TW_{\alpha;A}^{\bullet,\bullet},\bar\partial_{\alpha;A})$ respective $(\TW_{\beta;A}^{\bullet,\bullet},\bar\partial_{\beta;A})$ differ from $\TW(\psi_{\alpha\beta;A})^*$ by a gauge transform $\mathrm{exp}_\theta$ which respects differentials. However, such a gauge transform must satisfy $\bar\partial\theta = 0$, i.e., it comes from a gauge transform of $\E_{\alpha;A}^\bullet$. Thus, $\cO_{X_A} := \E_A^F$ turns $X_A \to S_A$ into a flat deformation of $X_0 \to S_0$, and $\E_A^\bullet$ is a geometric deformation of $\cP$-algebras of $\E_0^\bullet$. This defines a natural transformation 
$$\mathbf{h}: \mathrm{TWD}^\D(\E_0^\bullet,-) \Rightarrow \mathrm{GDef}^\D(\E_0^\bullet,-)$$
of functors of Artin rings. As in \cite[Lemma~7.4]{Felten2022}, we have the following result:

\begin{prop}\label{GDefD-TWDD-iso}\note{GDefD-TWDD-iso}
 The two natural transformations $\mathbf{tw}$ and $\mathbf{h}$ are inverse to each other, so that $\mathrm{GDef}^\D(\E_0^\bullet,-) \cong \mathrm{TWD}^\D(\E_0^\bullet,-)$.
\end{prop}

\section{Construction of $E_{X_0/\Lambda}^{\bullet,\bullet}$}

We generalize \cite[\S 8]{Felten2022} to the case of $\cP$-algebras. Recall from Definition~\ref{char-sheaf-def} that we have a system $\E_A^{\bullet,\bullet}$ of bigraded Thom--Whitney deformations over $A \in \mathbf{Art}_\Lambda$ together with restriction maps between them along $B' \to B$. We will construct a compatible system of predifferentials on $\E_A^{\bullet,\bullet}$. Then $E_{X_0/\Lambda}^{\bullet,\bullet}$ is a Cartanian $\cP^{crv}$-pre-algebra in the context $\mathfrak{Comp}(\Lambda)$ obtained as a limit of global sections of $\E_A^{\bullet,\bullet}$.

Let $B' \to B$ be a map in $\mathbf{Art}_\Lambda$. Every predifferential (locally) on $\E_{B'}^{\bullet,\bullet}$ restricts uniquely to $\E_B^{\bullet,\bullet}$ as we have seen in Lemma~\ref{basic-prediff-lemma}. Conversely, if $B' \to B$ is a surjection, and if we have a predifferential $\bar\partial_B$ on $\E_B^{\bullet,\bullet}$, it is always the restriction of some predifferential $\bar\partial_{B'}$ on $\E_{B'}^{\bullet,\bullet}$. To see this, it is sufficient to assume that $B' \to B$ is a small extension with kernel $I \subset B'$. Then, as in \cite[\S 8]{Felten2022}, the predifferentials on $\E_{B'}^{\bullet,\bullet}$ which restrict to $\bar\partial_B$ form a torsor under $I \cdot \E_{B'}^{T,1} \cong \E_0^{T,1} \otimes_\kk I$. Since $H^1(X_0,\E_0^{T,1}) = 0$ by Lemma~\ref{gen-acyc-TW-lemma}, every torsor is trivial, so it has a global section $\bar\partial_{B'}$.

\begin{defn}\label{char-curved-sheaf-defn}\note{char-curved-sheaf-defn}
 \index{characteristic curved sheaf}
 The \emph{characteristic curved sheaf} $(\E_A^{\bullet,\bullet},\bar\partial_A)$ is obtained from the characteristic bigraded sheaf $\E_A^{\bullet,\bullet}$ as follows: On $A_0 = \kk$, we have the differential $\bar\partial_0$ of $\E_0^{\bullet,\bullet}$. Then, on $A_{k + 1} = \Lambda/\m_\Lambda^{k + 2}$, the predifferential $\bar\partial_{k + 1}$ is obtained by choosing a lift of $\bar\partial_k$. Then, for a general $A \in \mathbf{Art}_\Lambda$, the predifferential $\bar\partial_A$ is the unique restriction of $\bar\partial_k$ along $A_k \to A$ for $k$ sufficiently large. We obtain a unique element $\ell_A \in H^0(X_0,\m_A \cdot \E_A^{T,2})$ with $\bar\partial_A^2 = \nabla_\ell(-)$ and $\bar\partial_A(\ell_A) = 0$ as explained after Definition~\ref{bg-TW-defo-prediff}; they satisfy $\ell_{B'}|_B = \ell_B$.
\end{defn}

Over $A \in \mathbf{Art}_\Lambda$, we define the characteristic algebra as 
$$E_{X_0/A}^{\bullet,\bullet} := H^0(X_0,(\E_A^{\bullet,\bullet},\bar\partial_A,\ell_A)).$$
\begin{lemma}\label{char-curved-sheaf-prop}\note{char-curved-sheaf-prop}
 We have the following:
 \begin{enumerate}[label=\emph{(\alph*)}]
  \item $\E_A^{P,q}$ is acyclic, i.e., $H^\ell(X_0,\E_A^{P,q}) = 0$ for $\ell \geq 1$.
  \item $E_{X_0/A}^{P,q}$ is a flat $A$-module.\footnote{Unfortunately, also the question of flatness of the global sections is not discussed in \cite{Felten2022}.}
  \item $E_{X_0/A}^{\bullet,\bullet}$ is a Cartanian $\cP^{crv}$-pre-algebra in the context $\mathfrak{Comp}(A)$.
  \item For each $B' \to B$, we have a restriction map $E_{X_0/B'}^{P,q} \to E_{X_0/B}^{P,q}$ which induces an isomorphism $E_{X_0/B'}^{P,q} \otimes_{B'} B \cong E_{X_0/B}^{P,q}$.
 \end{enumerate}
\end{lemma}
\begin{proof}
 If $B' \to B$ is a small extension, then we have an exact sequence 
 $$0 \to I \otimes_\kk \E_0^{P,q} \to \E_{B'}^{P,q} \to \E_B^{P,q} \to 0.$$
 The left hand term is acyclic by Lemma~\ref{gen-acyc-TW-lemma}. Thus, if the right hand term is acyclic, the middle term is acyclic as well, so the statement follows from induction over small extensions.
 
 For flatness of $E_{X_0/A}^{P,q}$, let $\check\C^\bullet(\W,\E_{X_0/A}^{P,q})$ be the \v{C}ech resolution of $\E_{X_0/A}^{P,q}$ in sheaves for some finite affine open cover $\W$ of $X_0$. We can assume that every open subset in $\W$ is contained in some $V_\alpha$. This computes the cohomology of $\E_{X_0/A}^{P,q}$, so we have an exact sequence 
 $$0 \to E_{X_0/A}^{P,q} \to H^0(X_0,\check\C^0(\W,\E_{X_0/A}^{P,q})) \to ... \to H^0(X_0,\check\C^{N - 1}(\W,\E_{X_0/A}^{P,q})) \to 0$$
 where $N$ is the number of open subsets in $\W$. From the second term on, each term is a flat $A$-module because $\E_{X_0/A}^{P,q}|_\alpha \cong \TW_{\alpha;A}^{P,q}$ is an $\cO_{V_{\alpha;A}}$-module which is flat over $S_A$ by Lemma~\ref{TW-flat-sheaf-lemma}, so sections over an intersection of opens in $\W$ form a flat $A$-module. Then $E_{X_0/A}^{P,q}$ is a flat $A$-module by repeated application of \cite[Lemma~0.2.1]{Wahl1976}.
 
 Once we know flatness, $E_{X_0/A}^{\bullet,\bullet}$ is a Cartanian $\cP^{crv}$-pre-algebra in $\mathfrak{Comp}(A)$ because $\E_A^{\bullet,\bullet}$ is one in $\mathfrak{Flat}(|X_0|;A)$. The last statement is similar to \cite[Cor.~5.7]{Felten2022}.
\end{proof}

\begin{defn}\label{char-algebra-def}\note{char-algebra-def}
 The \emph{characteristic algebra} is the limit\footnote{Recall that the formation of inverse limits of sheaves commutes with taking sections. Thus, we obtain the same characteristic algebra either by taking first global sections of $\E^{P,q}_A$ and then the limit, or vice versa. This also means that the characteristic algebra behaves well with respect to gluing the characteristic algebras of open subsets.} 
 $$E_{X_0/\Lambda}^{\bullet,\bullet} := \varprojlim_{k \to \infty} E_{X_0/A_k}^{\bullet,\bullet}.$$
\end{defn}

Each $E_{X_0/\Lambda}^{P,q}$ is a flat $\Lambda$-module by \cite[0912]{stacks}, and it is complete by construction. Thus, $E_{X_0/\Lambda}^{\bullet,\bullet}$ is a Cartanian $\cP^{crv}$-pre-algebra in the context $\mathfrak{Comp}(\Lambda)$. For each $A \in \mathbf{Art}_\Lambda$, we have $E_{X_0/A}^{\bullet,\bullet} \cong E_{X_0/\Lambda}^{\bullet,\bullet} \otimes_\Lambda A$.

\begin{rem}
 Since $E_{X_0/\Lambda}^{\bullet,\bullet}$ depends on many choices, let us collect them in one place; it depends:
 \begin{enumerate}[label=(\arabic*)]
  \item on the open cover $\V$ and the system of deformations $\D$;
  \item on the open cover $\U$;
  \item on the liftings $\E_{k + 1}^{\bullet,\bullet}$ of the bigraded Thom--Whitney deformation $\E_k^{\bullet,\bullet}$;
  \item on the liftings of the predifferentials along $\E_{k + 1}^{\bullet,\bullet} \to \E_k^{\bullet,\bullet}$.
 \end{enumerate}
\end{rem}

\begin{rem}
 In the case of Gerstenhaber algebras, we write $PV_{X_0/\Lambda}^{\bullet,\bullet}$ for $E_{X_0/\Lambda}^{\bullet,\bullet}$. It is a $\Lambda$-linear curved Gerstenhaber algebra in this case. In the case of (two-sided) Gerstenhaber calculi, we write additionally $DR_{X_0/\Lambda}^{\bullet,\bullet}$ for the part corresponding to the de Rham complex. Then $(PV_{X_0/\Lambda}^{\bullet,\bullet},DR_{X_0/\Lambda}^{\bullet,\bullet})$ is a $\Lambda$-linear curved (two-sided) Gerstenhaber calculus.
\end{rem}

\section{Maurer--Cartan elements in $E_{X_0/\Lambda}^{\bullet,\bullet}$}

The key reason to study $E_{X_0/\Lambda}^{\bullet,\bullet}$ is because it controls $\mathrm{TWD}^\D(\E_0^\bullet,-)$ and thus $\mathrm{GDef}^\D(\E_0^\bullet,-)$. In fact, the $\Lambda$-linear curved Lie algebra $L^\bullet := E_{X_0/\Lambda}^{T,\bullet}$ is sufficient for this purpose. If $\phi_A \in \m_A \cdot L_A^1$ is a solution of the extended Maurer--Cartan equation
$$\bar\partial_A\phi_A + \frac{1}{2}[\phi_A,\phi_A] + \ell_A = 0 \in L_A^2,$$
then $\bar\partial_A + \nabla_{\phi_A}(-)$ is a differential on $\E_A^{\bullet,\bullet}$, so the pair $(\E_A^{\bullet,\bullet},\bar\partial_A + \nabla_{\phi_A}(-))$ gives an element of $\mathrm{TWD}^\D(\E_0^\bullet,A)$, defining a natural transformation 
$$\mathbf{mc}: \mathrm{Def}(L^\bullet,-) \Rightarrow \mathrm{TWD}^\D(\E_0^\bullet,-)$$
of functors of Artin rings from the deformation functor of $L^\bullet$, introduced in Chapter~\ref{curved-Lie-alg-sec}.

\begin{lemma}\label{DefL-TWDD-iso}\note{DefL-TWDD-iso}
 The map $\mathbf{mc}$ is a well-defined natural isomorphism. In particular, $L^\bullet$ controls the deformation functor $\mathrm{GDef}^\D(\E_0^\bullet,-)$.
\end{lemma}
\begin{proof}
 If $\phi$ and $\phi'$ are gauge equivalent Maurer--Cartan elements in $\m_A \cdot L_A^1$, then an element $\theta \in \m_A \cdot L_A^0$ with $\mathrm{exp}_\theta * \phi = \phi'$ defines a gauge transform of $\E_A^{\bullet,\bullet}$ which transforms $\bar\partial$ into $\bar\partial'$. Thus, $\mathbf{mc}$ is well-defined.
 For injectivity, let $\phi$ and $\phi'$ be two Maurer--Cartan elements with $\mathbf{mc}(\phi) = \mathbf{mc}(\phi')$. Then there is some automorphism of $\E_A^{\bullet,\bullet}$ which transforms $\bar\partial$ into $\bar\partial'$. By Lemma~\ref{bg-TW-defo-autom-lemma}, this automorphism must be a gauge transform $\mathrm{exp}_\theta$ for some $\theta \in \m_A \cdot L_A^0 = H^0(X_0,\m_A \cdot \E_A^{T,0})$. By $Z$-faithfulness, we find $\phi' = \mathrm{exp}_\theta * \phi$ from Lemma~\ref{gauge-action} (its variant for $\cP$-algebras), so $\phi$ and $\phi'$ are gauge equivalent as elements of $\m_A \cdot L_A^1$ and thus define the same element in $\mathrm{Def}(L^\bullet,A)$. For surjectivity, we can assume that $\cH^{\bullet,\bullet} = \E_A^{\bullet,\bullet}$ in any pair $(\cH^{\bullet,\bullet},\bar\partial)$ by Lemma~\ref{basic-prediff-lemma}. Since predifferentials form a torsor under $\m_A \cdot \E_A^{T,1}$, we have an element $\phi \in \m_A \cdot H^0(X_0,\E_A^{T,1}) = \m_A \cdot L_A^1$ with $\bar\partial = \bar\partial_A + \nabla_\phi(-)$. Since $\bar\partial$ is a differential by assumption, $\phi$ must satisfy the extended Maurer--Cartan equation due to $Z$-faithfulness. But then $(\cH^{\bullet,\bullet},\bar\partial)$ is the image of $\phi$ under $\mathbf{mc}$.
\end{proof}

\begin{cor}
 In Situation~\ref{geom-G-calc-sitn}, we have a $\Lambda$-linear curved two-sided Gerstenhaber calculus 
 $$(PV_{X_0/\Lambda}^{\bullet,\bullet},DR_{X_0/\Lambda}^{\bullet,\bullet})$$
 with $\mathrm{GDef}^\D(\G_0^\bullet,\A_0^\bullet,-) \cong \mathrm{Def}(L^\bullet,-)$ for $L^\bullet = PV^{-1,\bullet}_{X_0/\Lambda}$.
\end{cor}

\begin{cor}
 In Situation~\ref{enh-gen-log-sm-sitn}, we have a $\Lambda$-linear curved two-sided Gerstenhaber calculus 
 $$(PV_{X_0/\Lambda}^{\bullet,\bullet},DR_{X_0/\Lambda}^{\bullet,\bullet})$$
 with controls enhanced generically log smooth deformations of $f_0: X_0 \to S_0$, i.e., $\mathrm{ELD}^\D_{X_0/S_0} \cong \mathrm{Def}(L^\bullet,-)$ for $L^\bullet = PV^{-1,\bullet}_{X_0/\Lambda}$.
\end{cor}

 To obtain the second open cover $\U$ in Situation~\ref{enh-gen-log-sm-sitn}, we may choose any affine open cover $\U$. To see this, we have to show admissibility with respect to $\U \cup \V$. Since $V_{\alpha_0} \cap ... \cap V_{\alpha_m} \to X_0$ is an affine morphism and $U_{i_0} \cap ... \cap U_{i_n}$ is affine, their intersection is affine as well, so all coherent sheaves are acyclic on this intersection.

\begin{cor}
 In Situation~\ref{enh-vec-bdl-defo-sitn}, we have a $\Lambda$-linear curved Lie--Rinehart pair 
 $$LRP^{\bullet,\bullet}_{X_0/\Lambda}(\E_0)$$
 which controls enhanced generically log smooth deformations of $f_0: X_0 \to S_0$ together with a vector bundle, i.e., we have $\mathrm{ELD}^\D_{X_0/S_0}(\E_0) \cong \mathrm{Def}(L^\bullet,-)$ for $L^\bullet = LRP^{T,\bullet}_{X_0/\Lambda}(\E_0)$.
\end{cor}

As above, an affine open cover $\U$ satisfies the condition that $\U \cup \V$ must be $\E_0$-admissible.

\section{The Calabi--Yau case}

Suppose we are in Situation~\ref{geom-G-calc-sitn}, and assume that $f_0$ is virtually Calabi--Yau, i.e., $\A_0^d \cong \cO_{X_0}$ via the choice of an appropriate global section $\omega_0 \in \Gamma(X_0,\A_0^d)$. Let $(\cP\V_k^{\bullet,\bullet},\D\R_k^{\bullet,\bullet})$ be the characteristic sheaf over $A_k = \Lambda/\m_\Lambda^{k + 1}$. By the canonical map $\A_0^d \to \D\R_0^{d,0}$, we consider $\omega_0$ as an element $\hat\omega_0$ of $\D\R_0^{d,0}$. By assumption, 
$$\kappa_0: \: \G^p_0 \to \A_0^{p + d}, \quad \theta \mapsto (\theta \ \invneg \ \omega_0),$$
is an isomorphism of $\cO_{X_0}$-modules. Applying the Thom--Whitney functor $\TW^q(-)$ to this map yields the contraction $\hat\kappa_0$ with $\hat\omega_0$ (when we use positive signs instead of $(-1)^q$). Thus, each $\hat\kappa_0: \cP\V_0^{p,q} \to \D\R_0^{p + d,q}$ is an isomorphism. Since each $\D\R_{k + 1}^{d,0} \to \D\R_k^{d,0}$ is surjective, we can lift $\hat\omega_0$ order by order and obtain a compatible system of elements $\hat\omega_k \in \D\R_k^{d,0}$. It is now easy to see that the contraction maps 
$$\hat\kappa_k: \: \cP\V_k^{p,q} \to \D\R_k^{p + d,q}, \quad \theta \mapsto (\theta \ \invneg\ \hat\omega_k),$$
are isomorphisms of sheaves of $A_k$-modules as well, using the description of the kernel as $I \otimes_\kk \D\R_0^{p,q}$. Now Proposition~\ref{BV-calc-construction} turns each $(\cP\V_k^{\bullet,\bullet},\D\R_k^{\bullet,\bullet})$ into a sheaf of curved two-sided Batalin--Vilkovisky calculi. This shows the following result.

\begin{lemma}\label{PV-DR-BV-struc}\note{PV-DR-BV-struc}
 In Situation~\ref{geom-G-calc-sitn}, assume furthermore that $f_0$ is virtually Calabi--Yau. Then 
 $$(PV_{X_0/\Lambda}^{\bullet,\bullet},DR_{X_0/\Lambda}^{\bullet,\bullet})$$
 can be endowed with a (non-canonical) structure of $\Lambda$-linear curved two-sided Batalin--Vilkovisky calculus.
\end{lemma}
\begin{proof}
 Take global sections and then the inverse limit along the $A_k$. The structure is not canonical because it depends both on the choice of $\omega_0$ and on the choice of the liftings $\hat\omega_k$ of $\hat\omega_0$.
\end{proof}

\begin{cor}\label{gen-log-sm-BV-struc}\note{gen-log-sm-BV-struc}
 In Situation~\ref{enh-gen-log-sm-sitn}, if $f_0: X_0 \to S_0$ is log Calabi--Yau, then $$(PV_{X_0/\Lambda}^{\bullet,\bullet},DR_{X_0/\Lambda}^{\bullet,\bullet})$$ can be endowed with a (non-canonical) structure of $\Lambda$-linear curved two-sided Batalin--Vilkovisky calculus. 
\end{cor}

\section{Summary for generically log smooth families}

For the reader's convenience, we summarize the theory in the case of a generically log smooth family. Let $Q$ be a sharp toric monoid, let $\Lambda = \kk\llbracket Q\rrbracket$, and let $f_0: X_0 \to S_0$ be a log Gorenstein generically log smooth family. Let $\V$ be an admissible open cover of $X_0$ in the sense of Definition~\ref{admissible-def}, and let $\D$ be a system of deformations subordinate to $\V$. In particular, all local deformations $V_{\alpha;A} \to S_A$ have the base change property. Let $\U$ be another admissible open cover of $X_0$ with the property that $\U \cup \V$ is admissible as well; for example this is satisfied if $\U$ consists of affine open subsets. We have a $\Lambda$-linear curved two-sided Gerstenhaber calculus
\begin{equation}\label{char-G-calc-summary}
 (PV_{X_0/\Lambda}^{\bullet,\bullet},DR_{X_0/\Lambda}^{\bullet,\bullet})
\end{equation}
as characteristic algebra, and with the $\Lambda$-linear curved Lie algebra 
$$L_{X_0/\Lambda}^\bullet := PV_{X_0/\Lambda}^{-1,\bullet},$$
we have a sequence 
$$\mathrm{LD}^\D_{X_0/S_0} \cong \mathrm{GDef}^\D(\E_0^\bullet,-) \cong \mathrm{TWD}^\D(\E_0^\bullet,-) \cong \mathrm{Def}(L^\bullet,-)$$
of isomorphisms of deformation functors, i.e., $L^\bullet$ controls the deformation functor $\mathrm{LD}^\D_{X_0/S_0}$. If $\phi \in \m_A \cdot L_A^1$ is a Maurer--Cartan element corresponding to a generically log smooth family $f_A: X_A \to S_A$ of type $\D$, then we have acyclic resolutions 
$$\V_{X_A/S_A}^p \to (\cP\V_A^{p,\bullet},\bar\partial_A + \nabla_\phi(-)), \quad \W_{X_A/S_A}^i \to (\D\R_A^{i,\bullet},\bar\partial_A + \nabla_\phi(-))$$
so that we find 
$$H^q(X_A,\Theta_{X_A/S_A}^p) = H^q(X_A,\V_{X_A/S_A}^p) \cong H^q(PV_{X_0/A}^{p,\bullet},\bar\partial_A + \nabla_\phi(-))$$
and 
$$H^j(X_A,\W_{X_A/S_A}^i) \cong H^j(DR_{X_0/A}^{i,\bullet},\bar\partial_A + \nabla_\phi(-)).$$
Similarly, we have a resolution 
$$(\W_{X_A/S_A}^\bullet,\partial) \to (\bigoplus_{i + j = \bullet}\D\R_A^{i,j},\: \partial + \bar\partial_A + \nabla_\phi(-))$$
with acyclic pieces so that we have 
$$\HH^k(X_A,\W_{X_A/S_A}^\bullet) \cong H^k(\bigoplus_{i + j = \bullet} DR_{X_0/A}^{i,j}, \: \partial + \bar\partial_A + \nabla_\phi(-)).$$
By the computation in Lemma~\ref{hypercohom-comparison}, we have an isomorphism 
$$H^k(\bigoplus_{i + j = \bullet} DR_{X_0/A}^{i,j}, \: \partial + \bar\partial_A + \nabla_\phi(-)) \cong H^k(\bigoplus_{i + j = \bullet} DR_{X_0/A}^{i,j}, \: \partial + \bar\partial_A + \ell_A \ \invneg \ (-))$$
so that the hypercohomology is, to a certain extent, independent of the Maurer--Cartan element $\phi$.

If \eqref{char-G-calc-summary} is quasi-perfect, then the cohomologies $H^k(X_A,\W^i_{X_A/S_A})$ and the hypercohomology $\HH^k(X_A,\W^\bullet_{X_A/S_A})$ are flat $A$-modules, and their formation commutes with base change. The Hodge--de Rham spectral sequence 
$$E_1^{pq} = H^q(X_A,\W^p_{X_A/S_A}) \Rightarrow \HH^{p + q}(X_A,\W^\bullet_{X_A/S_A})$$
degenerates at $E_1$.

If \eqref{char-G-calc-summary} is quasi-perfect, and $f_0: X_0 \to S_0$ is log Calabi--Yau, then also the cohomology $H^k(X_A,\V^p_{X_A/S_A})$ is a flat $A$-module, and its formation commutes with base change. Furthermore, the deformation functor $\mathrm{LD}^\D_{X_0/S_0}$ is unobstructed.

If $f_0: X_0 \to S_0$ is proper, then all cohomologies and hypercohomologies are finitely generated $A$-modules.



\chapter[The Gerstenhaber calculus of log toroidal families][Gerstenhaber calculus of log toroidal families]{The Gerstenhaber calculus of log toroidal families}\label{perfect-G-calc-sec}\note{perfect-G-calc-sec}

Let $Q$ be a sharp toric monoid, let $\Lambda = \CC\llbracket Q\rrbracket$, and let $f_0: X_0 \to S_0$ be a generically log smooth family. Let $A_Q = \Spec (Q \to \CC[Q])$, and let $a_A: S_A \to A_Q$ be the canonical chart of $S_A$ for $A \in \mathbf{Art}_\Lambda$. In \cite[Defn.~4.1]{FFR2021}, we have defined that $f_0: X_0 \to S_0$ is \emph{log toroidal}\index{log toroidal family} with respect to $a_0: S_0 \to A_Q$ if it admits a local model $A_{P,\F} \to A_Q$ as specified in \cite{FFR2021}.\footnote{Pay attention to the difference between being log toroidal and being log toroidal with respect to $a_0: S_0 \to A_Q$.} Similarly, when $f_A: X_A \to S_A$ is a generically log smooth deformation of $f_0: X_0 \to S_0$, we say that it is log toroidal with respect to $a_A: S_A \to A_Q$ if it has such a local model. 
In this chapter, we work in the following situation.

\begin{sitn}\label{proper-log-toroidal-sitn}\note{proper-log-toroidal-sitn}
  We have a proper generically log smooth family $f_0: X_0 \to S_0$ of relative dimension $d \geq 1$ which is log toroidal with respect to $a_0: S_0 \to A_Q$ and log Gorenstein. We fix an admissible open cover $\V = \{V_\alpha\}_\alpha$ of $X_0$ in the sense of Definition~\ref{admissible-def}, and $\D$ is a system of deformations subordinate to $\V$. We assume that all deformations in $\D$ are log toroidal with respect to $a_A: S_A \to A_Q$.\footnote{Neither do we require that at every point, there is a unique local model, nor do we require that every local model on some deformation extends to a thickening of the deformation. We just require that around every point on each deformation, there is some local model.} In particular, they have automatically the base change property by \cite[Cor.~7.9]{FFR2021}. We fix another open cover $\U = \{U_i\}_i$ of $X_0$ such that $\U \cup \V$ is admissible. 
\end{sitn}

The goal of this chapter is to prove the following result.\footnote{This section has been written before we became aware of the left contraction $\vdash$ that characterizes a two-sided Gerstenhaber calculus. The left contraction $\vdash$ is not relevant for this chapter, so all Gerstenhaber calculi are one-sided.}

\begin{thm}\label{perfect-G-calc-log-toroidal}\note{perfect-G-calc-log-toroidal}
 In the above situation, the characteristic $\Lambda$-linear curved Gerstenhaber calculus $(PV_{X_0/\Lambda}^{\bullet,\bullet},DR_{X_0/\Lambda}^{\bullet,\bullet})$ is perfect.\index{Gerstenhaber calculus!perfect} In particular, if $f_0: X_0 \to S_0$ is log Calabi--Yau, then $\mathrm{LD}^\D_{X_0/S_0}$ is unobstructed.
\end{thm}
\begin{proof}
 The degeneration of the Hodge--de Rham spectral sequence of $f_0: X_0 \to S_0$ is \cite[Thm.~1.9]{FFR2021}. The (hyper-)cohomology $H^k(DR_{X_0/\CC}^\bullet,d) \cong \HH^k(X_0,W_{X_0/S_0}^\bullet)$ is finitely generated because $X_0$ is proper. If we assume for a moment that $H^k(DR_{X_0/A}^\bullet,d) \to H^k(DR_{X_0/\CC}^\bullet,d)$ is surjective, then Proposition~\ref{Artin-complex-base-change} yields that $H^k(DR_{X_0/A}^\bullet,d) \otimes_A \CC \to H^k(DR_{X_0/\CC}^\bullet,d)$ is an isomorphism; thus, Nakayama's Lemma, which holds for modules over the Artinian local ring $A$ without any finiteness assumption,\footnote{See Chapter~\ref{spectral-sequence-A-sec}.} yields that $H^k(DR_{X_0/A}^\bullet,d)$ is finitely generated---just take a finitely generated submodule which surjects onto $H^k(DR_{X_0/\CC}^\bullet,d)$. The difficult part is to show the surjectivity of $H^k(DR_{X_0/A}^\bullet,d) \to H^k(DR_{X_0/\CC}^\bullet,d)$. This is the content of Chapter~\ref{surj-proof-sec} below. For the unobstructedness, we use Corollary~\ref{gen-log-sm-BV-struc}.
\end{proof}

\begin{ex}
 Let $(V,Z,s)$ be a well-adjusted triple such that the associated generically log smooth family $f_0: X_0 \to S_0$ is a log toroidal family of elementary Gross--Siebert type as defined in Definition~\ref{log-tor-GS-type}. Let $\V$ and $\U$ be affine open covers. Let $\D$ be a system of unisingular deformations obtained from Corollary~\ref{sys-defo-exists}. Assume that $f_0: X_0 \to S_0$ is log Calabi--Yau. Then $\mathrm{LD}_{X_0/S_0}^\D = \mathrm{LD}_{X_0/S_0}^{uni}$ is unobstructed.
\end{ex}

In the remainder of this chapter, we prove the surjectivity of 
\begin{equation}\label{DR-surj-eqn}
 H^k(DR_{X_0/A}^\bullet,d) \to H^k(DR_{X_0/\CC}^\bullet,d).
\end{equation}
If a Maurer--Cartan solution exists, this is a consequence of \cite[Thm.~1.10]{FFR2021}, respective a variant thereof. However, in the non--Calabi--Yau case, the statement is true even if no Maurer--Cartan solution exists, and in the Calabi--Yau case, we need to know the surjectivity before we can conclude that a Maurer--Cartan solution exists. Essentially, this result is \cite[Lemma~4.18]{ChanLeungMa2023}. Our proof will be essentially the same, slightly more conceptual; the point is that our construction of $(PV_{X_0/\Lambda}^{\bullet,\bullet},DR_{X_0/\Lambda}^{\bullet,\bullet})$ is very similar to the one in \cite{ChanLeungMa2023} but technically not quite the same, and we want to make sure that our algebraic version is actually perfect in the sense of Definition~\ref{G-calc-q-perf-def}, allowing a proof of unobstructedness in the Calabi--Yau case.

\par\vspace{\baselineskip}

A \emph{monoid ideal} $K \subseteq Q$ is a subset such that $k \in K$ and $q \in Q$ implies $k + q \in K$. If $K \subset Q$ is a proper\footnote{Proper means $K \not= Q$.} monoid ideal with $Q \setminus K$ finite, then $A_K := \CC[Q]/\CC[K]$ is in $\mathbf{Art}_\Lambda$. Because the Artinian local rings $A_k = \CC\llbracket Q\rrbracket/\m_Q^{k + 1}$ are of this form for $(Q \setminus \{0\})^{k + 1}$, and every $A \in \mathbf{Art}_\Lambda$ admits a map $A_k \to A$ for some $k \geq 0$, it is sufficient to prove the surjectivity in \eqref{DR-surj-eqn} for $A = A_k$ and in particular for $A = A_K$.

\section{Absolute differential forms}

Let $\mathbf{0} := \Spec(0 \to \CC)$. Then $A_Q \to \mathbf{0}$ is a saturated log smooth morphism of relative dimension $r := \mathrm{rk}(Q^{gp})$. The log de Rham complex is 
$$\Omega^\bullet_{A_Q/\bnull} = \bigoplus_{k = 0}^r \CC[Q] \otimes_\CC \bigwedge^k_\CC (Q^{gp} \otimes_\ZZ \CC)$$
with de Rham differential $\partial(z^q \otimes \alpha) = z^q \otimes (q \wedge \alpha)$ on global sections. For every $ A \in \mathbf{Art}_\Lambda$, we have a log de Rham complex $\Omega^\bullet_{S_A/\bnull}$ of the composite $S_A \to A_Q \to \bnull$ as well. When $K \subset Q$ is a monoid ideal with $Q \setminus K$ finite and $A = A_K := \CC[Q]/\CC[K]$, then this complex has a particularly simple form. Namely, in this case, \cite[IV, Cor.~2.3.3]{LoAG2018} yields $\Omega^1_{A_Q/\bnull} \otimes_{\CC[Q]} A_K \cong \Omega_{S_K/\bnull}^1$, where $S_K := \Spec (Q \to A_K)$, i.e., we have 
$$\Omega_{S_K/\bnull}^\bullet = \bigoplus_{k = 0}^r A_K \otimes_\CC \bigwedge_\CC^k(Q^{gp} \otimes_\ZZ \CC)$$
as graded $A_K$-module. Since the de Rham differential must be compatible, we find $\partial(z^q \otimes \alpha) = z^q \otimes (q \wedge \alpha)$; now $S_K$ is a punctual scheme, so this actually describes the full de Rham complex of sheaves.\footnote{The example $Q = 0$ and $A_\eps = \CC[\eps]/(\eps^2)$ shows that this description of $\Omega^\bullet_{S_A/\bnull}$ is not true for a general $A$ which is not of the form $\CC[Q]/\CC[K]$. Namely, in this case, we have $\Omega^1_{A_Q/\bnull} = 0$ but $\Omega^1_{S_\eps/\bnull} = \CC \cdot \partial \eps \not= 0$. In particular, $\Omega_{S_A/\bnull}^1$ is not locally free in general.} We say that $\Omega^\bullet_{S_K/\bnull}$ is the \emph{absolute de Rham complex} on $S_K$.

Let $f: X_K \to S_K$ be a generically log smooth deformation of type $\D$. On $U_K$, we have the de Rham complex $\Omega^\bullet_{U_K/\bnull}$ of the composite $X_K \to S_K \to \bnull$. Then we define the \emph{absolute de Rham complex} of $X_K \to S_K$ by 
$$W^\bullet_{X_K/\bnull} := j_*\Omega^\bullet_{U_K/\bnull}.$$
Of course, the de Rham differential $\partial$ is $\CC$-linear but not $A_K$-linear. On $U_K$, we have the sequence
$$0 \to f^*\Omega^1_{S_K/\bnull} \to \Omega^1_{U_K/\bnull} \to \Omega^1_{U_K/S_K} \to 0$$
which is exact and locally split by \cite[IV, Thm.~3.2.3]{LoAG2018}. In particular, $\Omega^1_{U_K/\bnull}$ is locally free of rank $d + r$. Taking exterior powers and the direct image under $j: U_K \to X_K$, we obtain an exact sequence 
$$0 \to f^*\Omega^1_{S_K/\bnull} \wedge W^{k - 1}_{X_K/\bnull} \to W^k_{X_K/\bnull} \to W^k_{X_K/S_K} \to 0,$$
where the right hand map is surjective by the local computation in \cite[Cor.~7.11, Cor.~7.12]{FFR2021}, and where the left hand side is \emph{defined} as 
$$j_*\mathrm{im}(f^*\Omega^1_{S_K/\bnull} \otimes \Omega^{k - 1}_{U_K/\bnull} \to \Omega^k_{U_K/\bnull}, \enspace f^*\mu \otimes \alpha \mapsto f^*\mu \wedge \alpha) \:\subseteq\: W^k_{X_K/\bnull}.$$
A careful local computation using  \cite[Cor.~7.11, Cor.~7.12]{FFR2021} shows that this is actually equal to the image of 
$$f^*\Omega^1_{S_K/\bnull} \otimes W^{k - 1}_{X_K/\bnull} \to W^k_{X_K/\bnull}.$$
Then we define a decreasing filtration of $W^\bullet_{X_K/\bnull}$ by 
$$F^sW^k_{X_K/\bnull} := \mathrm{im}(f^*\Omega^s_{S_K/\bnull} \otimes W^{k - s}_{X_K/\bnull} \to W^k_{X_K/\bnull}).$$
We have $F^0W^\bullet_{X_K/\bnull} = W^\bullet_{X_K/\bnull}$ and $F^{r + 1}W^\bullet_{X_K/\bnull} = 0$. Each $F^sW^\bullet_{X_K/\bnull}$ is closed under the de Rham differential $\partial$.

When $\mathrm{Gr}^sW^k_{X_K/\bnull} := F^sW^k_{X_K/\bnull}/F^{s + 1}W^k_{X_K/\bnull}$ is the quotient of the filtration, then we obtain a commutative diagram 
\[
 \xymatrixcolsep{1.5em}\xymatrix{
   0 \ar[r] & f^*\Omega^s_{S_K/\bnull} \otimes F^1W^{k - s}_{X_K/\bnull} \ar[r] \ar@{->>}[d] & f^*\Omega^s_{S_K/\bnull} \otimes W^{k - s}_{X_K/\bnull} \ar[r] \ar@{->>}[d] & f^*\Omega^s_{S_K/\bnull} \otimes W^{k - s}_{X_K/S_K} \ar[r] \ar@{->>}[d]^{\tau^k_s} & 0 \\
   0 \ar[r] & F^{s + 1}W^k_{X_K/\bnull} \ar[r] & F^s W^k_{X_K/\bnull} \ar[r] & \mathrm{Gr}^sW^k_{X_K/\bnull} \ar[r] & 0. \\
 }
\]
Using the local direct sum decomposition of $W^k_{X_K/\bnull}$ on $U_K$, we can show that $\tau_s^k$ (the vertical map on the right hand side) is an isomorphism on $U_K$, and then it must be an isomorphism on $X_K$ because the source is $Z$-closed. Then also the target $\mathrm{Gr}^sW^k_{X_K/\bnull}$ must be $Z$-closed. Furthermore, a diagram chase shows that 
$$(-1)^s \cdot \tau_s^{k + 1}(f^*\mu \otimes \partial(\alpha)) = \partial\,\tau_s^k(f^*\mu \otimes \alpha)$$
because the second summand of $\partial (f^*\mu \wedge \alpha)$ is inside $F^{s + 1}W^k_{X_K/\bnull}$. Thus, 
$$\tau_s: f^{-1}\Omega^s_{S_K/\bnull} \otimes_{A_K} W^\bullet_{X_K/S_K}[-s] \to \mathrm{Gr}^sW^\bullet_{X_K/\bnull}$$
is an isomorphism of complexes, where the convention for $[-s]$ is that the differential acquires a factor $(-1)^{-s}$.

The log morphism $U_K \to \bnull$ has an $\cO_{X_K}$-bilinear contraction map 
$$\invneg\ : \Theta^p_{U_K/\bnull} \times \Omega^i_{U_K/\bnull} \to \Omega^{i - p}_{U_K/\bnull}.$$
Via the inclusion $\Theta^p_{U_K/S_K} \subseteq \Theta^p_{U_K/\bnull}$, this gives rise to the contraction map 
$$\invneg\ : \G^p_{X_K/S_K} \times W^i_{X_K/\bnull} \to W^{p + i}_{X_K/\bnull}.$$
Then we can define a Lie derivative 
$$\cL: \G^p_{X_K/S_K} \times W^i_{X_K/\bnull} \to W^{p + i + 1}_{X_K/\bnull}$$
via the Lie--Rinehart homotopy formula. Together with the obvious $\wedge$-product, they turn $(\G^\bullet_{X_K/S_K},W^\bullet_{X_K/\bnull})$ into an \emph{absolute Gerstenhaber calculus}\index{Gerstenhaber calculus!absolute} in the sense of the following definition. We include directly also the bigraded and curved versions in this definition, which we use below. Beyond the scope of this article, absolute Gerstenhaber calculi are useful for a deeper Hodge-theoretic study of the spaces carrying them, for example via logarithmic semi-infinite variations of Hodge structures as in \cite[\S 6]{ChanLeungMa2023}.

\begin{defn}\label{abs-G-calc-defn}\note{abs-G-calc-defn}
 Let $Q$ be a sharp toric monoid with $r := \mathrm{rk}(Q^{pg})$, and let $K \subset Q$ be a proper monoid ideal with $Q \setminus K$ finite. Let $(\cS,\C,\cO,\bC)$ be a context where $\C$ is the constant sheaf on the site $\cS$ with stalk $A_K = \CC[Q]/\CC[K]$.
 \begin{enumerate}[label=(\alph*)]
  \item An \emph{absolute Gerstenhaber calculus} of dimension $d \geq 1$
 $$(\G^\bullet,\wedge,1_G,[-,-],\: \K^\bullet,\wedge,1_K,\partial,\:\invneg\:,\cL,\:\epsilon, F^\bullet\K^\bullet)$$
 consists of:
 \begin{itemize}
  \item a Gerstenhaber algebra $(\G^\bullet,\wedge,1_G,[-,-])$ of dimension $d$;
  \item an $\cO$-module $\K^i \in \bC$ of degree $i$ for every $0 \leq i \leq d + r$; for $i$ outside this range, we set $\K^i = 0$ whenever necessary;
  \item an $\cO$-bilinear product 
  $-\wedge - : \K^{i} \times \K^{i'} \to \K^{i+i'}$
  and a global section $1_K \in \K^{0}$ such that 
  $$(\alpha \wedge \beta) \wedge \gamma = \alpha \wedge (\beta \wedge \gamma); \quad \alpha \wedge \beta = (-1)^{|\alpha||\beta|}\beta \wedge \alpha; \quad 1_K \wedge \alpha = \alpha; $$
  \item a $\CC$-linear\footnote{It is not necessarily $A_K$-linear.} map $\partial: \K^{i} \to \K^{i + 1}$ with $\partial^2 = 0$, satisfying the derivation rule
  $$\partial(\alpha \wedge \beta) = \partial(\alpha) \wedge \beta + (-1)^{|\alpha|} \alpha \wedge \partial(\beta),$$
  the \emph{absolute de Rham differential};\index{de Rham differential!absolute}
  \item an $\cO$-bilinear map 
  $\invneg \ : \G^{p} \times \K^{i} \to \K^{p + i}$
  such that 
  $$1_G \ \invneg \ \alpha = \alpha \quad \mathrm{and} \quad (\theta \wedge \xi) \ \invneg \ \alpha = \theta \ \invneg \ (\xi \ \invneg \ \alpha);$$
  it is called the \emph{absolute contraction map};
  \item a $\CC$-bilinear map $\cL_{-}(-): \G^{p} \times \K^{i} \to \K^{p + i + 1}$ such that 
  $$\cL_{[\theta,\xi]}(\alpha) = \cL_\theta(\cL_\xi(\alpha)) - (-1)^{(|\theta| + 1)(|\xi| + 1)}\cL_\xi(\cL_\theta(\alpha));$$
  it is called the \emph{absolute Lie derivative};
  \item an $A_K$-linear injective map $\epsilon: f^{-1}\Omega^\bullet_{S_K/\bnull} \to (\K^\bullet,1_K,\wedge,\partial)$ of differential graded algebras with $\epsilon(1) = 1_K$; here, $f^{-1}\Omega^i_{S_K/\bnull}$ is the constant sheaf on $\cS$ with stalk $\Omega^i_{S_K/\bnull}$; we assume that $f^{-1}\Omega^i_{S_K/\bnull} \otimes_{A_K} \cO \to \K^i$ is injective as well and consider $f^{-1}\Omega^i_{S_K/\bnull} \otimes_{A_K} \cO$ as a subsheaf of $\K^i$ via this map; we write $f^*\Omega^i_{S_K/\bnull} := f^{-1}\Omega^i_{S_K/\bnull} \otimes_{A_K} \cO$ for short; these notations are justified by considering $f: (\cS,\C) \to (\{*\},A_K)$ as a map of ringed sites;
  \item a decreasing filtration $F^s\K^\bullet$ of $\K^\bullet$ by $\cO$-submodules, where $F^s\K^i$ is defined as the image of $f^*\Omega^s_{S_K/\bnull} \otimes \K^{i - s} \to \K^i, \: f^*\mu \otimes \alpha \mapsto f^*\mu \wedge \alpha$.
 \end{itemize}
 They need to have the following properties:
 \begin{itemize}
   \item the mixed Leibniz rule, the Lie--Rinehart homotopy formula, and the formulae for contraction with $\theta \in \G^0$ respective $\theta \in \G^{-1}$ hold as stated in Definition~\ref{Gerstenhaber-calc-def}; the map $\lambda: \G^0 \to \K^0, \: \theta \mapsto \theta \ \invneg\ 1_K$, is an isomorphism of $\cO$-modules with $\lambda(\theta \wedge \xi) = \lambda(\theta) \wedge \lambda(\xi)$;
   \item for $\theta \in \G^p$, $\mu \in \Omega^s_{S_K/\bnull}$, and $\alpha \in \K^i$, we have the commutation relation 
   \begin{equation}\label{comm-rel}
    [\theta \ \invneg\ ,\: \epsilon(\mu) \wedge(-)](\alpha) := \theta \ \invneg\ (\epsilon(\mu) \wedge \alpha) - (-1)^{|\theta||\mu|} \epsilon(\mu) \wedge (\theta \ \invneg \ \alpha) = 0;
   \end{equation}
   in particular, $\cL$ is $A_K$-linear in the second variable but not necessarily in the first one;
  \item the filtration $F^s\K^\bullet$ is compatible with $\wedge$,  $\partial, \ \invneg\ ,$ and $\cL$ in that we have induced maps 
  $$\wedge: F^s\K^i \times F^{s'}\K^{i'} \to F^{s + s'}\K^{i + i'}, \quad \partial: F^s\K^i \to F^s\K^{i + 1},$$
  $$\invneg\ : \G^p \times F^s\K^i \to F^s\K^{p + i}, \quad \cL: \G^p \times F^s\K^i \to F^s\K^{p + i + 1};$$
  \item let $\A^\bullet := F^0\K^\bullet/F^1\K^\bullet$; we assume that $\A^\bullet$ is concentrated in degrees $[0,d]$, and that $\A^i \in \bC$; for $a \in A_K$, we have automatically $\partial\,\epsilon(a) \in F^1\K^1$; thus, the differential on $\A^\bullet$ is $A_K$-linear;
  \item the induced map
  $$\tau_s: \Omega^s_{S_K/\bnull} \otimes_{A_K} \A^\bullet[-s] \to \mathrm{Gr}^s\K^\bullet$$ 
  is an isomorphism of complexes for all $s$.
 \end{itemize}
 The contraction map and the Lie derivative descend to $\A^\bullet$; since all necessary relations are satisfied, the pair $(\G^\bullet,\A^\bullet)$ then becomes an ordinary Gerstenhaber calculus of dimension $d$.
 
  \item A \emph{bigraded absolute Gerstenhaber calculus} of dimension $d \geq 1$
   $$(\G^{\bullet,\bullet},\wedge,1_G,[-,-],\: \K^{\bullet,\bullet},\wedge,1_K,\partial,\:\invneg\:,\cL,\:\epsilon, F^\bullet\K^{\bullet,\bullet})$$
   consists of:
 \begin{itemize}
  \item a bigraded Gerstenhaber algebra $(\G^{\bullet,\bullet},\wedge,1_G,[-,-])$ dimension $d$ in $(\cS,\C,\cO,\bC)$;
  \item for every $0 \leq i \leq d + r$ and every $j \geq 0$, an $\cO$-module $\K^{i,j} \in \bC$ of bidegree $(i,j)$ and total degree $i + j$; for $(i,j)$ outside this range, we set $\K^{i,j} = 0$;
  \item an $A_K$-bilinear product $-\wedge - : \K^{i,j} \times \K^{i',j'} \to \K^{i+i',j + j'}$
  and an element $1_K \in \K^{0,0}$ such that 
  $$(\alpha \wedge \beta) \wedge \gamma = \alpha \wedge (\beta \wedge \gamma); \quad \alpha \wedge \beta = (-1)^{|\alpha||\beta|}\beta \wedge \alpha; \quad 1_K \wedge \alpha = \alpha; $$
  \item a $\CC$-linear map $\partial: \K^{i,j} \to \K^{i + 1,j}$ with $\partial^2 = 0$ and satisfying the derivation rule 
  $$\partial(\alpha \wedge \beta) = \partial(\alpha) \wedge \beta + (-1)^{|\alpha|} \alpha \wedge \partial(\beta),$$
  the \emph{absolute de Rham differential};
  \item an injective degree-preserving $A_K$-linear  homomorphism $\epsilon: f^{-1}\Omega_{S_K/\bnull}^\bullet \to \K^{\bullet,0}$ of differential graded-commutative algebras (with $\epsilon(1) = 1$); furthermore, we assume that the induced map $f^*\Omega^i_{S_K/\bnull} \to \K^{i,0}$ is injective;
  \item an $\cO$-bilinear map 
  $\invneg \ : \G^{p,q} \times \K^{i,j} \to \K^{p + i,q + j}$
  such that $1 \ \invneg \ \alpha = \alpha$ and $(\theta \wedge \xi) \ \invneg \ \alpha = \theta \ \invneg \ (\xi \ \invneg \ \alpha)$,
  the \emph{absolute contraction map};
  \item an $A_K$-bilinear map $\cL_{-}(-): \G^{p,q} \times \K^{i,j} \to \K^{p + i + 1,q + j}$ such that 
  $$\cL_{[\theta,\xi]}(\alpha) = \cL_\theta(\cL_\xi(\alpha)) - (-1)^{(|\theta| + 1)(|\xi| + 1)}\cL_\xi(\cL_\theta(\alpha)),$$
  the \emph{absolute Lie derivative};
  \item a decreasing filtration $F^s\K^{\bullet,\bullet}$ of $\K^{\bullet,\bullet}$ by $\cO$-submodules, where $F^s\K^{i,j}$ is the image of $f^*\Omega^s_{S_K/\bnull} \otimes \K^{i - s,j} \to \K^{i,j}, \: f^*\mu \otimes \alpha \mapsto f^*\mu \wedge \alpha$.
  \end{itemize}
  They need to have the following properties:
  \begin{itemize}
  \item the mixed Leibniz rule, the Lie--Rinehart homotopy formula, and the formulae for contraction with $\theta \in \G^{0,q}$ and $\theta \in \G^{-1,q}$ hold as stated in Definition~\ref{Gerstenhaber-calc-def}; the map $\lambda: \G^{0,q} \to \K^{0,q}, \theta \mapsto \theta \ \invneg \ 1_K$, is an $\cO$-linear isomorphism with $\lambda(\theta \wedge \xi) = \lambda(\theta) \wedge \lambda(\xi)$;
  \item for $\theta \in \G^{p,q}$, $\mu \in \Omega^s_{S_K/\bnull}$, and $\alpha \in \K^{i,j}$, we have the commutation relation 
   \begin{equation}\label{comm-rel-bg}
    [\theta \ \invneg\ ,\: \epsilon(\mu) \wedge(-)](\alpha) := \theta \ \invneg\ (\epsilon(\mu) \wedge \alpha) - (-1)^{|\theta||\mu|} \epsilon(\mu) \wedge (\theta \ \invneg \ \alpha) = 0;
   \end{equation}
   \item the filtration $F^s\K^{\bullet,\bullet}$ is compatible with $\wedge$, $\partial, \ \invneg\ ,$ and $\cL$ in that we have induced maps 
   $$\wedge: F^s\K^{i,j} \times F^{s'} \K^{i',j'} \to F^{s + s'}\K^{i + i',j + j'},\quad \partial: F^s\K^{i,j} \to F^s\K^{i + 1,j},$$ 
   $$\invneg\ :\: \G^{p,q} \times F^s\K^{i,j} \to F^s\K^{p + i,q + j}, \quad \cL:\: \G^{p,q} \times F^s\K^{i,j} \to F^s\K^{p + i + 1, q + j};$$
  \item let $\A^{\bullet,\bullet} := F^0\K^{\bullet,\bullet}/F^1\K^{\bullet,\bullet}$; we assume that $\A^{\bullet,\bullet}$ is concentrated in degrees $[0,d]$ for the first variable, and that $\A^{i,j} \in \bC$; for $a \in A_K$, we have automatically $\partial\,\epsilon(a) \in F^1\K^{1,0}$; thus, the differential on $\A^{\bullet,\bullet}$ is $A_K$-linear;
  \item the induced map
  $$\tau_{s,j}: f^{-1}\Omega^s_{S_K/\bnull} \otimes_{A_K} \A^{\bullet,j}[-s] \to \mathrm{Gr}^s\K^{\bullet,j}$$ 
  is an isomorphism of complexes for all $s$.
 \end{itemize}
 It is called \emph{bounded} if $\G^{p,q} = 0$ for $q >> 0$ and $\K^{i,j} = 0$ for $j >> 0$.
 \item A \emph{curved absolute Gerstenhaber calculus} of dimension $d \geq 1$ is a bigraded absolute Gerstenhaber calculus together with an $A_K$-linear map $\bar\partial: \G^{p,q} \to \G^{p,q + 1}$, an element $\ell \in I_K \cdot \G^{-1,2}$, and an $A_K$-linear map $\bar\partial: \K^{i,j} \to \K^{i,j + 1}$ such that $(\G^{\bullet,\bullet},\bar\partial,\ell)$ is a curved Gerstenhaber algebra and the map $\bar\partial: \K^{i,j} \to \K^{i,j + 1}$ satisfies the derivation rules
 $$\bar\partial(\alpha \wedge \beta) = \bar\partial\alpha \wedge\beta + (-1)^{|\alpha|}\alpha \wedge\bar\partial\beta, \quad \bar\partial(\theta \ \invneg\ \alpha) = (\bar\partial\theta) \ \invneg \ \alpha + (-1)^{|\theta|} \theta \ \invneg\ \bar\partial\alpha$$
 as well as $\bar\partial^2(\alpha) = \cL_\ell(\alpha)$ and $\partial\bar\partial + \bar\partial\partial = 0$. Moreover, we assume that $\bar\partial \epsilon(\mu) = 0$ for $\mu \in f^{-1}\Omega^s_{S_K/\bnull}$.
 \item A \emph{differential bigraded absolute Gerstenhaber calculus} is a curved absolute Gerstenhaber calculus with $\ell = 0$.
 \end{enumerate}

\end{defn}

In the above situation, we write $\K^i_{X_K/S_K} := W^i_{X_K/\bnull}$.

\begin{lemma}
 The pair $(\G^\bullet_{X_K/S_K},\K^\bullet_{X_K/S_K})$ is an absolute Gerstenhaber calculus in the context $\mathfrak{Coh}(X_K/S_K)$.
\end{lemma}
\begin{proof}
 Since the absolute differential forms $\Omega^1_{U_K/\bnull}$ on $U_K$ are locally free of rank $d + r$, the pieces $\K^i_{X_K/S_K}$ are in the range $0 \leq i \leq d + r$. From \cite[Cor.~7.12]{FFR2021}, we find that $W^i_{X_K/\bnull}$ is flat over $S_K$ because, in the local model, it is a free $A_K$-module.\footnote{Note that, in \cite[Cor.~7.12]{FFR2021}, we have $E_K = E + (Q \setminus K)$, and that $e \in H$ if and only if $e + q \in H$ for $e \in E$ and $q \in Q \setminus K$.} The $\wedge$-product and the de Rham differential on $\K^\bullet_{X_K/S_K}$ have the required properties because $\K^\bullet_{X_K/S_K}$ is the direct image of the de Rham complex of $U_K \to \bnull$. Similarly, the properties of the contraction $\invneg$ and the Lie derivative $\cL$ are inherited from $U_K \to \bnull$. The map $\epsilon: f^{-1}\Omega^\bullet_{S_K/\bnull} \to \K^\bullet_{X_K/S_K}$ comes via the direct image from $U_K$ from the adjunction of $\Omega^\bullet_{S_K/\bnull} \to f_*j_*\Omega^\bullet_{U_K/\bnull}$. Here, we use that $j_*(f^{-1}\Omega^i_{S_K/\bnull})|_{U_K} = f^{-1}\Omega^i_{S_K/\bnull}$. Since $f^*\Omega^1_{S_K/\bnull} \to \Omega^1_{U_K/\bnull}$ is locally split injective, $f^*\Omega^i_{S_K/\bnull} \to W^i_{X_K/\bnull}$ is injective as well. The inclusion $A_K \to \cO_{X_K}$ is injective because every stalk of $\cO_{X_K}$ is an $A_K$-algebra which is flat and hence free as an $A_K$-module. Since $f^{-1}\Omega^i_{S_K/\bnull}$ is a free $A_K$-module, also the induced map $f^{-1}\Omega^i_{S_K/\bnull} \to f^*\Omega^i_{S_K/\bnull}$ is injective; hence, $\epsilon$ is injective. The filtration $F^s\K^\bullet_{X_K/S_K}$ is defined as specified in the definition. The mixed Leibniz rule and the formulae for the contraction with $\theta \in \G^0_{X_K/S_K}$ and $\theta \in \G^{-1}_{X_K/S_K}$ are inherited from the Gerstenhaber calculus of $U_K \to \bnull$. The Lie--Rinehart homotopy formula holds by the definition of $\cL$. The map $\lambda: \G^0_{X_K/S_K} \to \K^0_{X_K/S_K}$ is obviously an isomorphism of $\cO_{X_K}$-algebras because both are equal to $\cO_{X_K}$. To show the commutation relation \eqref{comm-rel}, first note that it holds for $p = 0$. For $p = -1$ and $s = 0$, we have $\theta \ \invneg\ \epsilon(\mu) = 0$ by degree reasons. For $s = 1$, we have $\theta \ \invneg\ \epsilon(\mu) = 0$ because $\theta$ is a relative derivation of $X_K \to S_K$. Since $\Omega^\bullet_{S_K/\bnull}$ is generated in degree $1$, the relation $\theta \ \invneg \ \epsilon(\mu) = 0$ follows for all $s$ from the formula for $\theta \ \invneg \ (\alpha \wedge \beta)$. Then the same formula yields the commutation relation \eqref{comm-rel} for $p = -1$ and all $s$. Then, on $U_K$, the commutation relation follows for all $p$ from $(\theta \wedge \xi) \ \invneg \ \alpha = \theta \ \invneg \ (\xi \ \invneg \ \alpha)$ because, there, $\G^{p}_{X_K/S_K}$ is generated by $\G^{-1}_{X_K/S_K}$. Finally, on $X_K$, the relation is inherited from $U_K$ via direct image. Every element of $F^s\K^i_{X_K/S_K}$ is locally a sum of elements of the form $\epsilon(\mu) \wedge \alpha$ for $\mu \in \Omega^s_{S_K/\bnull}$ and $\alpha \in \K^{i - s}_{X_K/S_K}$. Then direct computations show that $F^s\K^\bullet_{X_K/S_K}$ is closed under $\partial$, $\invneg$, and $\cL$. We have already seen above that the remaining statements hold.
\end{proof}

\section{The characteristic absolute Gerstenhaber calculus}\label{char-abs-G-calc-sec}\note{char-abs-G-calc-sec}

In this section, we study Thom--Whitney resolutions of absolute Gerstenhaber calculi, and we construct an analog of the characteristic sheaf of curved Gerstenhaber calculi $(\cP\V_A^{\bullet,\bullet},\D\R_A^{\bullet,\bullet})$ in absolute Gerstenhaber calculi. 

Let $f: X_K \to S_K$ be a generically log smooth deformation of $f_0: X_0 \to S_0$ of type $\D$. Since the Thom--Whitney resolution is a $\CC$-linear construction, we can apply it to the absolute Gerstenhaber calculus $(\G^\bullet_{X_K/S_K},\K^\bullet_{X_K/S_K})$ as well.

\begin{lemma}
 Let $(\G_{X_K/S_K}^{\bullet,\bullet},\K_{X_K/S_K}^{\bullet,\bullet})$ be the Thom--Whitney resolution of the absolute Gerstenhaber calculus $(\G^\bullet_{X_K/S_K},\K^\bullet_{X_K/S_K})$ with operations according to the sign conventions in Chapter~\ref{TW-reso-sec}.
 \begin{enumerate}[label=\emph{(\alph*)}]
  \item $(\G_{X_K/S_K}^{\bullet,\bullet},\K_{X_K/S_K}^{\bullet,\bullet})$ is a differential bg absolute Gerstenhaber calculus in $\mathfrak{Coh}(X_K/S_K)$.
  \item $F^s\K^{\bullet,\bullet}_{X_K/S_K}$ is the Thom--Whitney resolution of $F^s\K^\bullet_{X_K/S_K}$.
  \item $\A^{\bullet,\bullet}_{X_K/S_K} := F^0\K^{\bullet,\bullet}_{X_K/S_K}/F^1\K^{\bullet,\bullet}_{X_K/S_K}$ is the Thom--Whitney resolution of $\A_{X_K/S_K}^\bullet$. 
 \end{enumerate}
\end{lemma}
\begin{proof}
 We just give some brief indication. The map $\epsilon: f^{-1}\Omega^\bullet_{S_K/\bnull} \to \K^{\bullet,0}_{X_K/S_K}$ is the composition of the original $\epsilon$ with the injection $\K^\bullet_{X_K/S_K} \to \K^{\bullet,0}_{X_K/S_K}$. The filtered piece $F^s\K^i$ is the image of the map 
 $$\bigwedge^s(Q^{gp} \otimes_\ZZ \CC) \otimes_\CC \K^{i - s} \to \K^i, \quad \mu \otimes \alpha \mapsto \epsilon(\mu) \wedge \alpha.$$
 On the one hand, we can consider this as an $\cO_{X_K}$-linear map between two (plain, non-graded) coherent sheaves; then we can apply the Thom--Whitney resolution, which is an exact functor for each $j$, and obtain that the Thom--Whitney resolution $\TW^j(F^s\K^i) := C^{j,0}_\TW(F^s\K^i(\U))$ is the image of the map 
 $$\bigwedge^s(Q^{gp} \otimes_\ZZ \CC) \otimes_\CC \TW^j(\K^{i-s}) \to \TW^j(\K^s),\: \mu \otimes (a_n \otimes \alpha_n)_n \mapsto (a_n \otimes (\epsilon(\mu) \wedge \alpha_n))_n.$$
 On the other hand, $F^s\K^{i,j}$ is, by definition, the image of the map 
 $$\bigwedge^s(Q^{gp} \otimes_\ZZ \CC) \otimes_\CC \K^{i - s,j} \to \K^{i,j},$$
 $$\mu \otimes (a_n \otimes \alpha_n)_n \mapsto (1 \otimes \epsilon(\mu))_n \wedge (a_n \otimes \alpha_n)_n  = (-1)^{js} ((1 \wedge a_n) \otimes (\epsilon(\mu) \wedge \alpha_n))_n,$$
 where the $\wedge$-product on the right has to be formed according to the conventions for bilinear maps, as they hold for the construction of the $\wedge$-product on $\K_{X_K/S_K}^{\bullet,\bullet}$. Although the signs differ, the two maps have the same image, so $F^s\K^{i,j}_{X_K/S_K}$ is the ($j$-th piece of the) Thom--Whitney resolution of $F^s\K^i_{X_K/S_K}$, embedded via $\TW^j(-)$ applied to the embedding $F^s\K^i \to \K^i$ sitting in degree $i$. Then, by exactness, $\TW^j(\A^i) \cong F^0\K^{i,j}/F^1\K^{i,j}$. This isomorphism is compatible with all operations (that are compatible with $\K^i \to \A^i$) because $\TW^j(\A^i)$ and $\TW^j(\K^i)$ sit in the same degrees and hence have the same sign conventions. With this preparation, we find that the map $\tau_{s,j}$ in the definition of a bg absolute Gerstenhaber calculus is an isomorphism of complexes because of the analogous condition on $\K^\bullet_{X_K/S_K}$. We define $\ell = 0$. The remaining conditions are more or less straightforward.
\end{proof}

Every infinitesimal automorphism $\varphi$ of $f: X_K \to S_K$ induces an automorphism $(T_\varphi^\bullet,d_\varphi^\bullet)$ of the absolute Gerstenhaber calculus $(\G^\bullet_{X_K/S_K},\K^\bullet_{X_K/S_K})$. Since $\G^\bullet_{X_K/S_K}$ is the usual Gerstenhaber algebra of $f: X_K \to S_K$, the induced automorphism there is $T^\bullet_\varphi = \mathrm{exp}_{-\theta}$ for a unique $\theta \in I_K \cdot \G^{-1}_{X_K/S_K}$, where $I_K \subset A_K$ is the kernel of $A_K \to \CC$. By the second proof of Lemma~\ref{geom-auto-gauge-trafo-corr}, we find that $d^\bullet\varphi = \mathrm{exp}_{-\theta}$ as well, where $\mathrm{exp}_{-\theta}$ is constructed via the absolute Lie derivative on $\K^\bullet_{X_K/S_K}$.\footnote{The first proof of Lemma~\ref{geom-auto-gauge-trafo-corr} is not applicable here because $\K^i_{X_K/S_K}$ is even on the strict and smooth locus locally not generated by elements of the form $\partial g$ for $g \in \cO_{X_K}$.}

A gauge transform $\mathrm{exp}_{-\theta}$ of $(\G^\bullet_{X_K/S_K},\K^\bullet_{X_K/S_K})$ induces a gauge transform of the Thom--Whitney resolution $(\G^{\bullet,\bullet}_{X_K/S_K},\K_{X_K/S_K}^{\bullet,\bullet})$ given by 
$(1 \otimes (-\theta))_n \in I_K \cdot \G^{-1,0}_{X_K/S_K}$.

In Situation~\ref{proper-log-toroidal-sitn}, we can apply this construction, using the open cover $\U$ for the Thom--Whitney resolution, to the local deformations $V_{\alpha;A_K} \to S_K$. Since this will be sufficient, we restrict to the powers of the maximal ideal $K = (Q \setminus \{0\})^k$ and write $A_k$, $S_k$, and $V_{\alpha;k}$. We obtain differential bigraded absolute Gerstenhaber calculi $(\G_{\alpha;k}^{\bullet,\bullet},\K_{\alpha;k}^{\bullet,\bullet})$ together with comparison isomorphisms
$$\TW(\psi_{\alpha\beta;k})^*: \G_{\beta;k}^{\bullet,\bullet}|_{\alpha\beta} \to \G_{\alpha;k}^{\bullet,\bullet}|_{\alpha\beta}, \quad \K_{\beta;k}^{\bullet,\bullet}|_{\alpha\beta} \to \K_{\alpha;k}^{\bullet,\bullet}|_{\alpha\beta}$$
as well as restriction maps as in Chapter~\ref{bg-TW-defo-P-alg-sec}. The cocycles of $\TW(\psi_{\alpha\beta;k})^*$ are gauge transforms.

We have a characteristic curved Gerstenhaber algebra $\cP\V_k^{\bullet,\bullet}$ in the sense of Definition~\ref{char-curved-sheaf-defn}. By definition, it comes with local isomorphisms 
$$\chi_{\alpha;k}^*:\: \G_{\alpha;k}^{\bullet,\bullet} \to \cP\V_k^{\bullet,\bullet}|_\alpha$$
of bigraded Gerstenhaber algebras, compatible with restriction maps. On overlaps, we have that 
$$\TW(\psi_{\beta\alpha;k})^* \circ (\chi_{\alpha;k}^*)^{-1} \circ \chi_{\beta;k}^*: \: \G_{\beta;k}^{\bullet,\bullet}|_{\alpha\beta} \to \G_{\beta;k}^{\bullet,\bullet}|_{\alpha\beta}$$
is the gauge transform $\mathrm{exp}_{\theta_{\alpha\beta;k}}$ for a---since we have strict faithfulness---unique element $\theta_{\alpha\beta;k} \in I_K \cdot \G_{\beta;k}^{-1,0}|_{\alpha\beta}$. The predifferential $\bar\partial_k$ on $\cP\V_k^{\bullet,\bullet}$ satisfies 
$$(\chi_{\alpha;k}^*)^{-1} \circ \bar\partial_k|_\alpha \circ \chi_{\alpha;k}^* = \bar\partial_{\alpha;k} + \nabla_{\phi_{\alpha;k}}(-)$$
for unique elements $\phi_{\alpha;k} \in I_K \cdot \G_{\alpha;k}^{-1,1}$. All these data are compatible with restriction maps because they are constructed via an order-by-order lift along $A_{k + 1} \to A_k$.

We use these data to construct the characteristic curved sheaf of absolute Gerstenhaber calculi.

\begin{defn}
 In the above situation, the sheaf $\D\A_k^{i,j}$ is given by $$\Gamma(W,\D\A_k^{i,j}) = \{s_\alpha \in \Gamma(W \cap V_\alpha,\K_{\alpha;k}^{i,j}) \ | \  \TW(\psi_{\alpha\beta;k})^*(s_\alpha|_{\alpha\beta}) = \mathrm{exp}_{\theta_{\alpha\beta;k}}(s_\beta|_{\alpha\beta})\},$$
 i.e., by the gluing of the pieces $\K_{\alpha;k}^{i,j}$, identified along $\TW(\psi_{\alpha\beta;k})^* \circ \mathrm{exp}_{\theta_{\alpha\beta;k}}^{-1}$.
\end{defn}

The gluings satisfy the cocycle condition because, after transporting the gauge transforms along the maps $\TW(\psi_{\alpha\beta;k})^*$, which commute with gauge transforms, the questions comes down to comparing the $\theta_{\alpha\beta;k}$ with the elements $o_{\alpha\beta\gamma;k}$ which control the cocycles of $\TW(\psi_{\alpha\beta;k})^*$. Now the cocycle condition is satisfied for $\cP\V_k^{\bullet,\bullet}$, so it must be satisfied for $\D\A_k^{i,j}$ as well. We also have induced restriction maps $\D\A_{k + 1}^{i,j} \to \D\A_k^{i,j}$.

We endow $\D\A_k^{\bullet,\bullet}$ with the (bigraded) operations $\wedge$, $\invneg$, $\partial$, and $\cL$. They are defined locally and commute with both $\TW(\psi_{\alpha\beta;k})^*$ and gauge transforms, so they are well-defined on $\D\A_k^{\bullet,\bullet}$. We also endow $\D\A_k^{\bullet,\bullet}$ with a predifferential $\bar\partial_k$ of bidegree $(0,1)$, which is given on $V_\alpha$ by $\bar\partial_{\alpha;k} + \nabla_{\phi_{\alpha;k}}(-)$. This is well-defined because it is so on the Gerstenhaber algebra $\cP\V_k^{\bullet,\bullet}$. Namely, $\TW(\psi_{\alpha\beta;k})^*$ commutes with all predifferentials, and for the gauge transform part of the comparison, the equations for the transform of $\phi_{\alpha;k}$ into $\phi_{\beta;k}$ (in Lemma~\ref{gauge-action-BV}) are the same on $\D\A_k^{\bullet,\bullet}$ and $\cP\V_k^{\bullet,\bullet}$. Moreover, all operations are compatible with restriction maps. This shows:

\begin{lemma}
 The pair $(\cP\V_k^{\bullet,\bullet},\D\A_k^{\bullet,\bullet})$ forms a sheaf of curved absolute Gerstenhaber calculi.
\end{lemma}
\begin{proof}
 As discussed above, we have globally defined sheaves, operations, and constants. For $\mu \in \Omega^s_{S_K/\bnull}$, the commutation relation \eqref{comm-rel-bg} shows that $\theta \ \invneg \ \epsilon(\mu) = 0$ on $\K_{\alpha;k}^{s,1}$ for $\theta \in \G^{-1,1}_{\alpha;k}$. Thus, we have $\cL_\theta(\epsilon(\mu)) = 0$ by the Lie--Rinehart homotopy formula, and this implies $\epsilon(\mu) \in \D\A_k^{s,0}$, giving the map $\epsilon$. The constant $\ell$ already exists globally as part of $\cP\V_k^{\bullet,\bullet}$. Now we have all the data required by Definition~\ref{abs-G-calc-defn}.
 
 Each $(\G_{\alpha;k}^{\bullet,\bullet},\K_{\alpha;k}^{\bullet,\bullet})$ is a curved absolute Gerstenhaber calculus when endowed with the predifferential $\bar\partial_{\alpha;k} + \nabla_{\phi_{\alpha;k}}(-)$ and the constant $\ell_{\alpha;k} := \bar\partial_{\alpha;k}(\phi_{\alpha;k}) +  \frac{1}{2}[\phi_{\alpha;k},\phi_{\alpha;k}]$. The condition $\bar\partial\epsilon(\mu) = 0$ is preserved under changing the predifferential by $\nabla_{\phi}(-)$ because of the commutation relation \eqref{comm-rel-bg}.
 Since $(\cP\V_k^{\bullet,\bullet},\D\A_k^{\bullet,\bullet})$ is locally isomorphic to this curved absolute Gerstenhaber calculus, and since all properties required in Definition~\ref{abs-G-calc-defn} are local, it is a curved absolute Gerstenhaber calculus as well.
\end{proof}

\section{The proof of surjectivity}\label{surj-proof-sec}\note{surj-proof-sec}

In this section, we show that the map in \eqref{DR-surj-eqn} is surjective for $A = A_k$ when we are in Situation~\ref{proper-log-toroidal-sitn}. For $Q = 0$, nothing is to show. For $Q = \NN$, we employ an old trick which Steenbrink ascribes in \cite{Steenbrink1976} to Katz, referring to \cite{SGA7-2}. It also has been used in \cite[Thm.~4.1]{GrossSiebertII}, and subsequently in \cite[Ass.~4.15]{ChanLeungMa2023} and in \cite[Thm.~1.10]{FFR2021}, where we prove the surjectivity in the case of an actual log toroidal deformation. Finally, the case of a general $Q$ can be reduced to $Q = \NN$. 

So let us first assume that $Q = \NN$. By Lemma~\ref{gen-acyc-TW-lemma}, the sheaves $\D\R_0^{i,j}$ are $\Gamma$-acyclic since $\U$ is admissible. Then each $\D\R_k^{i,j}$ must be $\Gamma$-acyclic as well because of the exact sequence
\begin{equation}\label{DR-ext-alg}
 0 \to I \otimes_\CC \D\R_0^{i,j} \to \D\R_{k + 1}^{i,j} \to \D\R_{k}^{i,j} \to 0.
\end{equation}
Since $DR_{X_0/A_k}^{i,j} = \Gamma(X_0,\D\R_k^{i,j})$, we have $H^n(DR_{X_0/A_k}^\bullet,d) = \HH^n(X_0,(\D\R_k^\bullet,d))$, so it is sufficient to prove the surjectivity in (actual) hypercohomology. Next, we form a complex 
$$\D\A_k^n[u] := \bigoplus_{s = 0}^\infty \D\A_k^n \cdot u^s$$
with differential $e_d(\alpha \cdot u^s) := d(\alpha) \cdot u^s + s\bar\rho \wedge \alpha \cdot u^{s - 1}$. Here, $\bar\rho \in \Gamma(X_0,\D\A_k^1)$ is the following element: Since $Q = \NN$, we have an element $1$ in the sheaf of monoids $\M_{S_k}$ on $S_k$. Applying the universal log derivation, we obtain $[1] \in \Omega^1_{S_k/\bnull} = \CC[\NN] \otimes_\CC (\NN^{gp} \otimes_\ZZ \CC)$, and then $\bar\rho := \epsilon([1])$. The operator $d$ in the formula is our differential $d = \partial + \bar\partial + \ell \ \invneg \ (-)$. Then we consider the composition 
$$\pi: (\D\A_k^\bullet[u],e_d) \to (\D\A_k^\bullet,d) \to (\D\R_k^\bullet,d) \to (\D\R_0^\bullet,d).$$
Each of them is compatible with the differential, and each of them is surjective (the first map is the projection onto the $u^0$-summand). Thus, we can define a complex $(\R^\bullet,e_d)$ by the exact sequence 
$$0 \to (\R^\bullet,e_d) \to (\D\A_k^\bullet[u],e_d) \to (\D\R_0^\bullet,d) \to 0.$$
It is then sufficient to prove $\HH^n(X_0,\R^\bullet) = 0$ to conclude the surjectivity in \eqref{DR-surj-eqn}. This is the main goal of the remainder of this section. We have to overcome three main difficulties: First, the main argument works only for complex analytic spaces, not schemes, secondly, we have no global deformation $X_k$ for which $\D\A_k^\bullet[u]$ is quasi-coherent, so analytification needs special care, and thirdly, we have to analytify the differential operator $e_d$, so we cannot just take a pull-back along $\varphi: X^{an} \to X$, even locally.

\subsubsection*{The local model}

Let $(\NN \subset P,\F)$ be an elementary log toroidal datum in the sense of \cite[Defn.~3.1]{FFR2021}; it gives rise to a log morphism $f: A_{P,\F} \to A_\NN$, which serves as an \'etale local model for a generically log smooth deformation of type $\D$ in Situation~\ref{proper-log-toroidal-sitn}. Let $k \geq 0$, and let $f: L_k \to S_k$ be the base change of the local model to $S_k = \Spec (\NN \to \CC[t]/(t^{k + 1})$. Let $K = (\NN^+)^{k + 1}$ be the corresponding monoid ideal. 

Let $(\G^\bullet_{k},\K^\bullet_{k})$ be the absolute Gerstenhaber calculus of $f: L_k \to S_k$. Then we define a complex $(\K_{k}^\bullet[u],e)$ with $\K_{k}^n[u] := \bigoplus_{s = 0}^\infty \K_{k}^n \cdot u^s$ and $e(\alpha \cdot u^s) := \partial(\alpha) \cdot u^s + s\bar\rho \wedge \alpha \cdot u^{s - 1}$ similar to the above one with $\bar\rho = \epsilon([1])$. We define $(\R_{k}^\bullet,e)$ as the kernel in the sequence 
\begin{equation}\label{local-model-R}
 0 \to (\R_{k}^\bullet,e) \to (\K_{k}^\bullet[u],e) \to (\A^\bullet_{0},\partial) \to 0. 
\end{equation}
Let $\tilde f: \tilde L_k \to \tilde S_k$ be the analytification of $f: L_k \to S_k$, where we denote the analytification with $\tilde{-}$ in order to not load our notation even more with the heavy $(-)^{an}$. The pieces of the complexes in \eqref{local-model-R} are (quasi-)coherent sheaves, and the maps between them are $\cO_{L_k}$-linear. Thus, we can analytify them. The de Rham differential $\partial$ on $\A_0^\bullet$ is a first-order differential operator, so we can analytify it as well in a unique way by Proposition~\ref{analytify-diff-op-1} and obtain $\tilde \partial$. But $e$ is a first-order differential operator as well, so we find a unique analytification $\tilde e$. The uniqueness properties in Proposition~\ref{analytify-diff-op-1} show that 
$$0 \to (\tilde\R_{k}^\bullet,\tilde e) \to (\tilde\K_{k}^\bullet[u],\tilde e) \to (\tilde \A^\bullet_{0},\tilde\partial) \to 0$$
is a short exact sequence of complexes (!) whose pieces are quasi-coherent analytic sheaves, and whose differentials are analytic differential operators of first order. 

\begin{prop}
 The complex $(\tilde \R_k^\bullet,\tilde e)$ has no cohomology at the stalk at $0 \in \tilde L_k$.
\end{prop}
\begin{proof}
 In principle, this comes down to a calculation which is already in \cite[Thm.~4.1]{GrossSiebertII}. There, the computation is done on the global sections on $L_k$ of the algebraic counterparts, and it is wrongly concluded that this would yield $\Gamma$-acyclicity of the global algebraic complex. In \cite[Lemma~12.1]{FFR2021}, we have remedied this situation by working on the stalks of the analytification, which needs some additional consideration of convergence questions. However, strictly speaking, the proof of \cite[Lemma~12.1]{FFR2021} is still incomplete because we made an implicit assumption about the form of $\tilde\partial$ and $\tilde e$ on the analytic stalks, which we did not prove back then. So here comes what is still missing.
 
 The remaining gap can be closed by using the sequence topology on analytic stalks, which we review in Chapter~\ref{analyt-preliminaries-sec}. In the discussion after \cite[Lemma~7.13]{FFR2021}, we have shown that the local analytic ring at $0 \in \tilde L_k$ is 
$$\cO_{\tilde L_k,0} = \left\{ \sum_{e \in E_K} \alpha_e z^e \in \CC\llbracket E_K\rrbracket \ \middle| \ \mathrm{sup}_{e \in E_K \setminus 0}\left\{\frac{\mathrm{log}|\alpha_e|}{h(e)}\right\} < \infty \right\}$$
with the notations from there, i.e., $E_K = P \setminus (P + K)$ and $h: P \to \NN$ is a local homomorphism of sharp toric monoids. It is now quite easy to show that any sequence of finite truncations of $\alpha = \sum_{e \in E_K}\alpha_ez^e$ which exhausts $E_K$ actually converges to $\alpha$ in the sequence topology. Namely, this can be shown directly for the stalk of $(\Spec \CC[\NN^r])^{an}$ when described as in \cite[Lemma~7.13]{FFR2021}; then this property descends along a local surjection $\phi: \NN^r \to P$ to our sharp toric monoid $P$ and the stalk of $\tilde A_P$; and finally, this property descends along the closed embedding $\tilde L_k \subset \tilde A_P$ of complex analytic spaces. Furthermore, for a stalk of a coherent analytic sheaf as in \cite[Lemma~7.14]{FFR2021}, the same is true. Namely, each such stalk is, in the sequence topology, a closed subset of $\cO_{\tilde L_k,0}^{\oplus r}$ carrying the product topology. Then it is also true for the quasi-coherent analytic sheaf $\tilde\K_k^n[u]$. By Proposition~\ref{analytify-diff-op-1} and Lemma~\ref{diff-op-continuity}, the analytifications $\tilde\partial$ and $\tilde e$ are continuous for the sequence topologies; thus, we can evaluate them on the stalks by evaluating them on finite truncations, and then taking the limit. However, the finite truncations come from sections of the algebraic quasi-coherent sheaves on $L_k$; thus, the analytic differential operators $\tilde \partial$ and $\tilde e$ actually have the form used in our proof of \cite[Lemma~12.1]{FFR2021}.
\end{proof}

\subsubsection*{The transfer to $f: V_{\alpha;k} \to S_k$}

Consider $f: V_{\alpha;k} \to S_k$, and let $(\G^\bullet_{\alpha;k},\K^\bullet_{\alpha;k})$ be the absolute Gerstenhaber calculus. Then we define a complex $(\K_{\alpha;k}^\bullet[u],e)$ with $\K_{\alpha;k}^n[u] := \bigoplus_{s = 0}^\infty \K_{\alpha;k}^n \cdot u^s$ and $e(\alpha \cdot u^s) := \partial(\alpha) \cdot u^s + s\bar\rho \wedge \alpha \cdot u^{s - 1}$ as for the local model. We define $(\R_{\alpha;k}^\bullet,e)$ as the kernel of $(\K_{\alpha;k}^\bullet[u],e) \to (\A^\bullet_{\alpha;0},\partial)$.
Let $\tilde f: \tilde V_{\alpha;k} \to \tilde S_k$ be the analytification of $f: V_{\alpha;k} \to S_k$. As above, we obtain an analytification 
\begin{equation}\label{an-log-toroidal-R}
 0 \to (\tilde\R_{\alpha;k}^\bullet,\tilde e) \to (\tilde \K_{\alpha;k}^\bullet[u],\tilde e) \to (\tilde \A^\bullet_{\alpha;0},\tilde\partial) \to 0.
\end{equation}
By assumption, $f: V_{\alpha;k} \to S_k$ is, locally in the \'etale topology, isomorphic to $f: L_k \to S_k$ for some elementary log toroidal datum $(\NN \subset P,\F)$, which, in general, depends on the point $v \in V_{\alpha;k}$. By \cite[Cor.~4.16]{FeltenThesis}, if $v \in V_{\alpha;k}(\CC)$ is a $\CC$-valued point, we can arrange the \'etale roof such that $v = 0$ in $L_k(\CC)$.

If $f: X_k \to S_k$ is a generically log smooth family, and $g: Y_k \to X_k$ is a strict \'etale morphism, then we have $g^*\A^i_{X_k/S_k} \cong \A^i_{Y_k/S_k}$, $g^*\G^p_{X_k/S_k} \cong \G^p_{Y_k/S_k}$, and $g^*\K^i_{X_k/S_k} \cong \K^i_{Y_k/S_k}$ canonically. These maps are compatible with $g^{-1}$ applied to all operations in the (absolute or relative) Gerstenhaber calculus. Using this together with Proposition~\ref{analytify-diff-op-1}\footnote{The point is that the two analytifications of $\partial$ and $e$ either on $\tilde X_k$ or on $\tilde Y_k$ agree via $\tilde g$.} and the fact that $\tilde g: \tilde Y_k \to \tilde X_k$ is a local isomorphism in the Euclidean topology, we find that the stalk of the exact sequence \eqref{an-log-toroidal-R} for $Y_k \to S_k$ at $y \in \tilde Y_k$ is the same as the stalk of the exact sequence \eqref{an-log-toroidal-R} for $X_k \to S_k$ at $\tilde g(y) \in \tilde X_k$. Applying this along the \'etale roof of the local model at $v \in \tilde V_{\alpha;k}$, we find that $(\tilde\R_{\alpha;k}^\bullet, \tilde e)$ is an acyclic complex.\footnote{Note that it does not matter that we may change the log smooth open subset of the generically log smooth family in comparing $V_{\alpha;k}$ with its local models, as is allowed in \cite[Defn.~4.1]{FFR2021}. Namely, the absolute and relative Gerstenhaber calculi are independent of the choice of open subset of log smoothness.}

Applying the Thom--Whitney resolution (with respect to $\U$) to the original algebraic version of \eqref{an-log-toroidal-R} yields a short exact sequence 
\begin{equation}\label{TW-reso-ses}
 0 \to (\TW^\bullet(\R^\bullet_{\alpha;k}),e,\bar\partial) \to (\TW^\bullet(\K_{\alpha;k}^\bullet[u]),e,\bar\partial) \to (\TW^\bullet(\A_{\alpha;0}^\bullet),\partial,\bar\partial) \to 0
\end{equation}
of sheaves of double complexes, where we apply our sign conventions of Chapter~\ref{TW-reso-sec} in the formation of $\partial$ and $e$ on the resolution. We can analytify this short exact sequence as well because it consists of quasi-coherent sheaves with first-order differential operators between them.\footnote{The Thom--Whitney construction does not commute with analytification, i.e., when we apply the analytic analog of the Thom--Whitney construction to $\F^{an}$, this is in general not the same as the analytification of $\TW^\bullet(\F)$. The reason for this is that direct images from open subsets do not commute with analytification. Consider for example the structure sheaf on $\bAA^1 \subseteq \PP^1$.} Each horizontal row
$$\R_{\alpha;k}^i \to \TW^\bullet(\R_{\alpha;k}^i,\bar\partial)$$
is a resolution in the category of quasi-coherent sheaves. Thus, also the analytification 
$$\tilde\R_{\alpha;k}^i \to \widetilde\TW^\bullet(\R_{\alpha;k}^i,\tilde{\bar\partial})$$
is a resolution in the category of quasi-coherent analytic sheaves. In the vertical direction, we have analytic first-order differential operators $\tilde e$ such that $\tilde{\bar\partial} \tilde e + \tilde e\tilde{\bar\partial} = 0$. In particular, we have a quasi-isomorphism 
$$(\tilde\R_{\alpha;k}^\bullet,\tilde e) \to (\mathrm{Tot}^\bullet\widetilde\TW^\bullet(\R^\bullet_{\alpha;k}),\tilde e + \tilde\db) =: \widetilde{\mathrm{Tot}}^\bullet(\R^\bullet_{\alpha;k})$$
of complexes of sheaves so that the latter complex is acyclic.

When $\phi \in \m_k \cdot \G^{-1,1}_{\alpha;k}$, then we can modify the (pre-)differential $\bar\partial$ to obtain a predifferential $\db_\phi := \db + \cL_\phi(-)$. In the short exact sequence \eqref{TW-reso-ses}, only the horizontal differentials change from $\db$ to $\db_\phi$. Setting $E := e^\phi \in \mathrm{Tot}^0(\TW^\bullet(\G^\bullet_{\alpha;k}))$, we obtain a diagram of exact sequences of total complexes 
\[
 \resizebox{12cm}{!}{
 \xymatrix{
  {\begin{array}{c}
    (\mathrm{Tot}^\bullet\TW^\bullet(\R_{\alpha;k}^\bullet),  e + \db)
   \end{array}}
 \ar@{^{(}->}[r] & (\mathrm{Tot}^\bullet\TW^\bullet(\K_{\alpha;k}^\bullet[u]), e + \bar\partial) \ar@{->>}[r] & (\mathrm{Tot}^\bullet\TW^\bullet(\A_{\alpha;0}^\bullet),\partial + \db)\\
  {\begin{array}{c}
    (\mathrm{Tot}^\bullet\TW^\bullet(\R_{\alpha;k}^\bullet), \\  e + \db_\phi + \ell_\phi \ \invneg \ (-))
   \end{array}}
 \ar@{^{(}->}[r] \ar[u]^{\Phi_E}_\cong & {\begin{array}{c}
            (\mathrm{Tot}^\bullet\TW^\bullet(\K_{\alpha;k}^\bullet[u]),\\ e + \db_\phi + \ell_\phi \ \invneg \ (-))
           \end{array}}
 \ar@{->>}[r] \ar[u]^{\Phi_E}_\cong & (\mathrm{Tot}^\bullet\TW^\bullet(\A_{\alpha;0}^\bullet),\partial + \db) \ar@{=}[u]\\
 }
 }
\]
similar to Lemma~\ref{hypercohom-comparison}. Thus, the analytification of the complex on the lower left side is acyclic. Furthermore, the restriction of our complex $(\R^\bullet,e_d) \subseteq (\D\A_k^\bullet[u],e_d)$ above to $V_{\alpha;k}$ is isomorphic to the lower left complex once we take $\phi = \phi_{\alpha;k}$.

\subsubsection*{From local to global analytification}

We wish to find a global analytification $(\tilde\R^\bullet,\tilde e_d)$ of $(\R^\bullet,e_d)$ which is locally the one coming from $\tilde V_{\alpha;k} \to V_{\alpha;k}$. Then $(\tilde\R^\bullet,\tilde e_d)$ must be $\Gamma$-acyclic because it is acyclic, and we wish to conclude from this that $(\R^\bullet,e_d)$ is $\Gamma$-acyclic as well. The first part is quite easy. On underlying spaces, we have a global analytification map $\varphi: |\tilde X_0| \to |X_0|$. Then, on each $V_\alpha$, we can analytify separately and obtain a map 
$$\varphi^{-1}(\R^\bullet,e_d)|_\alpha \to (\tilde\R_\alpha^\bullet,\tilde e_d),$$
where the right hand side is the analytification of $(\R^\bullet,e_d)|_\alpha$ along $\tilde V_{\alpha;k} \to V_{\alpha;k}$ by definition. On overlaps $V_\alpha \cap V_\beta$, we obtain comparison isomorphisms which are compatible with the differentials. Since these comparison isomorphisms come from the tensor product representation, they fit into a commutative diagram 
\[
 \xymatrix{
   \varphi^{-1}(\R^\bullet,e_d)|_{\alpha\beta} \ar@{=}[d] \ar[r] & \varphi^{-1}(\R^\bullet,e_d)|_{\alpha\beta} \otimes_{\varphi^{-1}\cO_{\alpha;k}|_{\alpha\beta}} \varphi^{-1}\tilde\cO_{\alpha;k}|_{\alpha\beta} \ar[d]^\cong \\
   \varphi^{-1}(\R^\bullet,e_d)|_{\alpha\beta} \ar[r] & \varphi^{-1}(\R^\bullet,e_d)|_{\alpha\beta} \otimes_{\varphi^{-1}\cO_{\beta;k}|_{\alpha\beta}} \varphi^{-1}\tilde\cO_{\beta;k}|_{\alpha\beta}
 }
\]
 with the identity on the left. Thus, also the cocycles of the comparison isomorphism fit into such a diagram with the identity on the left. However, on the analytic stalks, the cocycles are locally bounded and hence  continuous for the sequence topology; since the image of $\varphi^{-1}\R^\bullet|_\alpha$ is dense in $\tilde\R^\bullet_\alpha$, this implies that the cocycles are the identity because finitely generated subsheaves of $\R^\bullet_\alpha$ are coherent, and the sequence topology on stalks of coherent sheaves is Hausdorff, cf.~the proof of Lemma~\ref{locally-bounded-uniqueness}.\footnote{Note that the cocycles are not $\tilde\cO_{k;\alpha}|_{\alpha\beta\gamma}$-linear because the cocycles are, in general, not the identity on the structure sheaf. Nonetheless, they are locally bounded operators in the sense of Chapter~\ref{analyt-preliminaries-sec}, and they are continuous on finitely generated subsheaves.} Thus, we have a global analytification $(\tilde R^\bullet, \tilde e_d)$ of $(\R^\bullet, e_d)$, and $(\tilde \R^\bullet,e_d)$ is acyclic, hence also $\Gamma$-acyclic.
 
 \subsubsection*{Comparison of cohomologies}
 
 If we have a global scheme structure, i.e., a gluing of the pieces $V_{\alpha;k}$, then we can conclude that $(\R^\bullet,e_d)$ is acyclic by comparing the two spectral sequences of the stupid filtrations and applying the classical GAGA comparison of sheaf cohomologies of proper schemes $X$, which holds not only for coherent but also for quasi-coherent sheaves $\F$ by Lemma~\ref{qcoh-gaga}.
 Since we do not have a global scheme structure a priori, more care is needed. Nonetheless, we have a map of spectral sequences\footnote{Cf.~also \cite[Lemma~2.46]{FeltenThesis} for a discussion of the construction.}
 $$(E_1^{p,q} := H^q(X_0,\R^p) \Rightarrow \HH^{p + q}(X_0,\R^\bullet)) \to (\tilde E_r^{p,q} := H^q(\tilde X_0,\tilde\R^p) \Rightarrow \HH^{p + q}(\tilde X_0,\tilde\R^\bullet)),$$
 and we know already that the abutment on the right hand side is $0$.
 
 First, by Lemma~\ref{char-curved-sheaf-prop}, we know that $\D\R_k^{i,j}$ is $\Gamma$-acyclic. Now $\D\A_k^{i,j}$ has a decreasing filtration $F^s\D\A_k^{i,j}$ whose subquotients are isomorphic to $f^{-1}\Omega^s_{S_k/\bnull} \otimes_{A_k} \D\R_k^{i,j}$. Thus, $\D\A_k^{i,j}$ is $\Gamma$-acyclic as well. Since $X_0$ is Noetherian, direct limits commute with cohomology; thus, $\D\A_k^{i,j}[u]$ is $\Gamma$-acyclic, and then $\R^n$ must be $\Gamma$-acyclic as well. For $\ell \geq 2$, we have directly $H^\ell(X_0,\R^n) = 0$, and for $\ell = 1$, this follows from the surjectivity of 
 $$H^0(X_0,\D\A_k^{i,j}[u]) \to H^0(X_0,\D\R_0^{i,j}).$$
 Analogously to $\tilde\R^\bullet$, we can also analytify $\D\R_k^{\bullet,\bullet}$ and $\D\A_k^{\bullet,\bullet}$ globally, yielding $\widetilde{\D\R}_k^{\bullet,\bullet}$ and $\widetilde{\D\A}_k^{\bullet,\bullet}$. The sheaf $\widetilde{\D\R}_0^{i,j}$ is the usual analytification of the $\Gamma$-acyclic quasi-coherent sheaf $\D\R_0^{i,j}$ on $X_0$, so it is $\Gamma$-acyclic by Lemma~\ref{qcoh-gaga}. Similar to the algebraic case, this implies that first $\widetilde{\D\R}_k^{i,j}$ and then also $\widetilde{\D\A}_k^{i,j}$ is $\Gamma$-acyclic. Since cohomology commutes with filtered colimits on the compact Hausdorff space $\tilde X_0$ by \cite[Thm.~4.12.1]{Godement1973}, also $\widetilde{\D\A}_k^{i,j}[u]$ is $\Gamma$-acyclic. Then also $\tilde\R^p$ must be $\Gamma$-acyclic because $\widetilde{\D\A}_k^p[u] \to \widetilde{\D\R}_0^p$ is surjective on the level of global sections. This shows that the maps $E_1^{p,q} \to \tilde E_1^{p,q}$ of the above map of spectral sequences is an isomorphism for $q \geq 1$. For the case $q = 0$, first observe that 
 $$H^0(X_0,\D\R_0^{i,j}) \to H^0(\tilde X_0,\widetilde{\D\R}_0^{i,j})$$
 is an isomorphism by Lemma~\ref{qcoh-gaga}. Using \eqref{DR-ext-alg} and its analytic counterpart, we can show that 
 $$H^0(X_0,\D\R_k^{i,j}) \to H^0(\tilde X_0,\widetilde{\D\R}_k^{i,j})$$
 is an isomorphism by induction on $k$. Using the filtration, we can show similarly that 
 $$H^0(X_0,\D\A_k^{i,j}) \to H^0(\tilde X_0,\widetilde{\D\A}_k^{i,j})$$
 is an isomorphism, and then $E_1^{p,q} \to \tilde E_1^{p,q}$ must be an isomorphism for $q = 0$ as well, where we use again that, on both spaces $X_0$ and $\tilde X_0$, cohomology commutes with filtered colimits, in order to obtain the isomorphism on the level of $\D\A_k^p[u]$. This shows $\HH^n(X_0,(\R^\bullet,e_d)) = 0$ for all $n \geq 0$, concluding the proof in the case $Q = \NN$.

 \subsubsection*{The case of a general $Q$}
 
 To reduce this case to $Q = \NN$, we use an idea of Chan--Leung--Ma, see \cite[\S 4.3.2]{ChanLeungMa2023}, which we also used in the case of a globally given log toroidal family in \cite{FeltenThesis}. 
 
 We have to show the surjectivity of $\HH^n(X_0,(\D\R_k^\bullet,d)) \to \HH^n(X_0,(\D\R_0^\bullet,d))$. Let $K = (Q^+)^{k + 1}$ be the monoid ideal corresponding to $A_k$. Then we can find a finite decreasing filtration 
 $$Q^+ = I_0 \supset I_1 \supset ... \supset K$$
 as in \cite[Lemma~8.11]{FeltenThesis}, i.e., each $I_n$ is a monoid ideal, we have $\mathrm{dim}_\CC(\CC[I_n]/\CC[I_{n + 1}]) = 1$, and for each $n$, there is a monoid homomorphism $h_n: Q \to \NN$ with $I_n = h_n^{-1}((i,\infty))$ and $I_{n + 1} = h_n^{-1}((i + 1,\infty))$ for some $i \geq 0$. Then it is sufficient to show the surjectivity for the restriction along $\CC[Q]/\CC[I_{n + 1}] \to \CC[Q]/\CC[I_{n}]$. From now on, let $I = I_{n + 1}$ and $J = I_n$ for some $n$, let $A' = \CC[Q]/\CC[I]$ and $A = \CC[Q]/\CC[J]$, and let $S' = \Spec (Q \to A')$ and $S = \Spec (Q \to A)$. Let $h = h_n$, let $B' = \CC[t]/(t^{i + 1})$ and $B = \CC[t]/(t^{i})$ for the value of $i$ corresponding to $n$, and let $T = \Spec (\NN \to \CC[t]/(t^i))$ and $T' = \Spec (\NN \to \CC[t]/(t^{i + 1}))$. Furthermore, let $T_0 = \Spec (\NN \to \CC)$. Then $h$ induces a commutative diagram 
 \[
  \xymatrix{
  T_0 \ar[r] \ar[d]^{b_0} & T \ar[r] \ar[d]^b & T' \ar[d]^{b'} \\
  S_0 \ar[r] & S \ar[r] & S' \\
  }
 \]
 of log schemes. By \cite[Prop.~4.7]{FeltenThesis}, the base change $g_0: Y_0 \to T_0$ of $f_0: X_0 \to S_0$ along $T_0 \to S_0$ is a log toroidal family with respect to $T_0 \to A_\NN$. These two log toroidal families have the same underlying space, and the induced maps $\A^\bullet_{X_0/S_0} \to \A^\bullet_{Y_0/T_0}$ and $\G^\bullet_{X_0/S_0} \to \G^\bullet_{Y_0/T_0}$ (via pull-back of maps $\A^1_{X_0/S_0} \to \cO_{X_0}$) are isomorphisms.
 
 Let $V_{\alpha;I}$ and $V_{\alpha;J}$ be the local models over $S'$ respective $S$ of the system of deformations $\D$. They induce a kind of system of deformations on $T$ and $T'$ via pull-back along $b$ and $b'$. We denote the pull-backs by $V_{\alpha;i - 1}$ and $V_{\alpha;i}$. Again by \cite[Prop.~4.7]{FeltenThesis}, they are log toroidal families. The induced maps $\A^\bullet_{\alpha;J} \otimes_A B \to \A^\bullet_{\alpha;i - 1}$ and $\A^\bullet_{\alpha;I} \otimes_{A'} B' \to \A^\bullet_{\alpha;i}$ are isomorphisms by applying Lemma~\ref{bijective-in-fibers} to the sources. Then we also have induced isomorphisms $\G^\bullet_{\alpha;J} \otimes_A B \to \G^\bullet_{\alpha;i - 1}$ and $\G^\bullet_{\alpha;I} \otimes_{A'} B' \to \G^\bullet_{\alpha;i}$. 
 Let $(\cP\V_I^{\bullet,\bullet},\D\R_I^{\bullet,\bullet})$ be a choice of a characteristic sheaf of curved Gerstenhaber calculi over $S'$. Then 
 $$(\cP\V_i^{\bullet,\bullet},\D\R_i^{\bullet,\bullet}) := (\cP\V_I^{\bullet,\bullet} \otimes_{A'} B', \D\R_I^{\bullet,\bullet} \otimes_{A'} B')$$
 is a sheaf of curved Gerstenhaber calculi as well, and we define $(\cP\V_{i - 1}^{\bullet,\bullet},\D\R_{i - 1}^{\bullet,\bullet})$ similarly as the base change along $b: T \to S$. When denoting the Thom--Whitney resolution of $(\G_{\alpha;i}^\bullet,\A_{\alpha;i}^\bullet)$ by $(\G_{\alpha;i}^{\bullet,\bullet},\A_{\alpha;i}^{\bullet,\bullet})$ as usual, we have isomorphisms 
 $$(\G_{\alpha;i}^{\bullet,\bullet},\A_{\alpha;i}^{\bullet,\bullet}) \cong (\G_{\alpha;I}^{\bullet,\bullet},\A_{\alpha;I}^{\bullet,\bullet}) \otimes_{A'} B' \xrightarrow{\chi_{\alpha;I}^* \otimes_{A'} B'} (\cP\V_i^{\bullet,\bullet},\D\R_i^{\bullet,\bullet})|_\alpha $$
 of bg (!) Gerstenhaber calculi since the Thom--Whitney resolution commutes with affine base change. Ignoring the induced predifferential $\bar\partial_i$, $(\cP\V_i^{\bullet,\bullet},\D\R_i^{\bullet,\bullet})$ is a bigraded Thom--Whitney deformation in the sense of Definition~\ref{bg-TW-defo-defn}. However, $\bar\partial_i$ is also a predifferential on it in the sense of Definition~\ref{bg-TW-defo-prediff}. The same is true for $(\cP\V_{i - 1}^{\bullet,\bullet},\D\R_{i - 1}^{\bullet,\bullet})$, which is a bigraded Thom--Whitney deformation with predifferential $\bar\partial_{i - 1}$. Furthermore, we have a restriction map between them which is compatible with all data.
 
 The above discussion for $Q = \NN$ applies to $(\cP\V_i^{\bullet,\bullet},\D\R_i^{\bullet,\bullet})$. This shows that the map 
 $$\HH^n(X_0,(\D\R_i^\bullet,d)) \to \HH^n(X_0,(\D\R_0^\bullet,d))$$
 is surjective, where $d = \partial + \bar\partial + \ell \ \invneg \ (-)$, and where $\D\R_0^\bullet$ is the total complex of the Thom--Whitney resolution of either $\A_{X_0/S_0}^\bullet$ or $\A_{Y_0/T_0}^\bullet$, which coincide. Then also 
 $$\HH^n(X_0,(\D\R_i^\bullet,d)) \to \HH^n(X_0,(\D\R_{i - 1}^\bullet,d))$$
 must be surjective by Nakayama's lemma since both are free $B'$ respective $B$-modules with the same base change to $\CC$. After this preparation, a diagram chase as in the proof of \cite[Lemma~8.9]{FeltenThesis} completes the proof of surjectivity along the base change $\CC[Q]/\CC[I] \to \CC[Q]/\CC[J]$, and thus the proof of Theorem~\ref{perfect-G-calc-log-toroidal}.



\chapter{Deforming line bundles}\label{fiber-bundle-constr-sec}\note{fiber-bundle-constr-sec}

Suppose we are in the following situation:

\begin{sitn}\label{gen-log-sm-fiber-bundle-sitn}\note{gen-log-sm-fiber-bundle-sitn}
 We have a sharp toric monoid $Q$, and we set $\Lambda = \kk\llbracket Q\rrbracket$. We have a torsionless enhanced generically log smooth family $f_0: X_0 \to S_0$ of relative dimension $d \geq 1$ together with a line bundle $\cL_0$ on $X_0$, and $\V = \{V_\alpha\}_\alpha$ is an admissible open cover such that $\cL_0|_\alpha$ is isomorphic to $\cO_{X_0}|_\alpha$ via a trivializing section $s_{0;\alpha} \in \cL_0|_\alpha$. Then we have coordinate transformation functions $u_{0;\alpha\beta} \in \Gamma(V_\alpha \cap V_\beta,\cO_{X_0}^*)$ with $s_{0;\alpha} = u_{0;\alpha\beta} \cdot s_{0;\beta}$. We fix an enhanced system of deformations $\D$ for $f_0$ subordinate to $\V$.
\end{sitn}

 Below, we construct a $\PP^1$-bundle $p_0: P_0(\cL_0) \to X_0$ together with an enhanced system of deformations $\D(\cL_0)$ of the enhanced generically log smooth family $g_0: P_0(\cL_0) \to S_0$ subordinate to $\V(\cL_0) := \{p_0^{-1}(V_\alpha)\}_\alpha$ and an isomorphism 
 $$\mathrm{ELD}^\D_{X_0/S_0}(\cL_0) \Rightarrow \mathrm{ELD}^{\D(\cL_0)}_{P_0(\cL_0)/S_0}$$
 of deformation functors. The map $p_0: P_0(\cL_0) \to X_0$ is obviously projective, and over $U_0 \subseteq X_0$, it is log smooth and satisfies $\Omega^1_{P_0(\cL_0)/X_0} \cong \cO_{P_0(\cL_0)}$. Thus, if $f_0: X_0 \to S_0$ is proper respective log Calabi--Yau, then $g_0: P_0(\cL_0) \to S_0$ is proper respective log Calabi--Yau. Moreover, we will see in Lemma~\ref{log-toroidal-fiber-bundle} below that $g_0: P_0(\cL_0) \to S_0$ as well as its deformations are log toroidal families if this is the case for $f_0: X_0 \to S_0$ and its deformations.
 
 \begin{thm}\label{proper-log-CY-log-toroidal-unobstr-pairs}\note{proper-log-CY-log-toroidal-unobstr-pairs}
  In Situation~\ref{gen-log-sm-fiber-bundle-sitn}, assume that $\kk = \CC$, that $f_0: X_0 \to S_0$ is proper and log Calabi--Yau, that $f_0: X_0 \to S_0$ is log toroidal with respect to $a_0: S_0 \to A_Q$, and that each local model $V_{\alpha;A} \to S_A$ is log toroidal with respect to $a_A: S_A \to A_Q$. Then the deformation functor $\mathrm{LD}^\D_{X_0/S_0}(\cL_0)$ is unobstructed. 
 \end{thm}
 \begin{proof}
  We apply Theorem~\ref{perfect-G-calc-log-toroidal} to the log toroidal family $g_0: P_0(\cL_0) \to S_0$.
 \end{proof}
 
 \begin{rem}
  The idea for the construction of $P_0(\cL_0) \to S_0$ is taken from \cite{Iacono2021defPairs}, where the classical case of a smooth algebraic variety $X_0$ is treated. While unobstructedness of pairs was already known in the classical case, the primary goal of that article is to show that the dg Lie algebra controlling the deformations of the pair $(X_0,\cL_0)$ is homotopy abelian, a stronger result which implies the unobstructedness. Here, we content ourselves with proving the unobstructedness.
 \end{rem}
 
 \begin{rem}
  Deformations of pairs $(X_0,\cL_0)$ for $Q = \NN$ have been also studied in \cite{CLM2022pairs}. Our study of the curved Lie--Rinehart pair controlling the deformations of $(X_0,\cL_0)$ has been inspired by and is essentially contained in that work. However, this work does not study the $\PP^1$-bundle $P_0(\cL_0)$, and hence they do not obtain the above unobstructedness result. Instead, their main result \cite[Thm.~1.1]{CLM2022pairs} is a relation between the deformation theory of a pair $(X_0,\F_0^\bullet)$ with a vector bundle $\F_0$, and the  deformation theory of $(X_0,\mathrm{det}(\F_0))$. With Theorem~\ref{proper-log-CY-log-toroidal-unobstr-pairs}, we can remove their condition that $(X_0,\mathrm{det}(\F_0))$ must be (more or less) unobstructed in \cite[Thm.~1.1]{CLM2022pairs} because we know it holds. Presumably, this also simplifies \cite{ChanMaSuen2022}, where they apply \cite[Thm.~1.1]{CLM2022pairs}.
 \end{rem}

\section{The construction of $P_A(\cL_A)$}

In Situation~\ref{gen-log-sm-fiber-bundle-sitn}, let $Y_0 \subseteq X_0$ be an open subset of the form $V_{\alpha_1} \cap ... \cap V_{\alpha_r}$ (possibly $Y_0 = X_0$), and let $f_A: (Y_A,\cL_A) \to S_A$ be an enhanced  generically log smooth deformation with a line bundle of type $\D$. We have a $\PP^1$-bundle 
$$\PP(\cO_{Y_A} \oplus \cL_A) := \mathrm{Proj}\ \mathrm{Sym}^\bullet (\cO_{Y_A} \oplus \cL_A)$$
over $Y_A$. On $Y_{A;\alpha} := Y_A \cap V_\alpha$, we choose a lift $s_{A;\alpha}$ of $s_{0;\alpha}$, which is possible due to our assumption on the form of $Y_0$. This lift trivializes $\cL_A|_\alpha$, and we have coordinate transformation functions $u_{A;\alpha\beta}$ satisfying $s_{A;\alpha} = u_{A;\alpha\beta} \cdot s_{A;\beta}$, which are a lift of $u_{0;\alpha\beta}$. In particular, we have 
$$\mathrm{Sym}^\bullet(\cO_{Y_A} \oplus \cL_A)|_\alpha \cong \cO_{Y_{A;\alpha}}[R_\alpha,S_\alpha]$$
where $R_\alpha$ corresponds to $1 \in \cO_{Y_A}$ and $S_\alpha$ corresponds to $s_{A;\alpha} \in \cL_A$. Thus, 
$$W_\alpha := \PP(\cO_{Y_A} \oplus \cL_A)|_\alpha \cong Y_{A;\alpha} \times \PP^1,$$
and we decompose $W_\alpha = W'_\alpha \cup W''_\alpha$ into two affine patches $W'_\alpha = \{S_\alpha \not= 0\}$ with coordinate $x_\alpha = R_\alpha/S_\alpha$ and $W''_\alpha = \{R_\alpha \not= 0\}$ with coordinate $y_\alpha = S_\alpha/R_\alpha$. The coordinate transformation on $W'_\alpha \cap W'_\beta$ is given by $x_\alpha = u_{A;\alpha\beta}^{-1} \cdot x_\beta$, and the coordinate transformation on $W''_\alpha \cap W''_\beta$ is given by $y_\alpha = u_{A;\alpha\beta} \cdot y_\beta$; since $W'_\alpha \cap W_\beta = W'_\beta$ and $W''_\alpha \cap W_\beta = W''_\beta$, this together with $y_\alpha = x_\alpha^{-1}$ on $W'_\alpha \cap W''_\alpha$ is enough to describe all coordinate transformations.

We see from these coordinate transformations that 
$$\Delta_{0,A} := \{x_\alpha = 0\}, \quad \Delta_{\infty,A} := \{y_\alpha = 0\}$$
are well-defined closed subschemes of $\PP(\cO_{Y_A} \oplus \cL_A)$; they are sections of the projection, and they are the complements of $W''_\alpha$ respective $W'_\alpha$. We denote their ideal sheaves by $\I_{0,A}$ and $\I_{\infty,A}$. They are line bundles, so they give rise to a Deligne--Faltings log structure $\gamma_1: \I_{0,A} \to \cO_\PP,\: \gamma_2: \I_{\infty,A} \to \cO_\PP$ in the sense of \cite[III, Defn.~1.7.1]{LoAG2018}. This gives rise to a log scheme over $\underline{Y}_A$, the underlying scheme of $Y_A$ endowed with the trivial log structure, which we denote by 
$$p'_A: P'_A(\cL_A) \to \underline{ Y}_A.$$

\begin{lemma}\label{pA-log-smooth}\note{pA-log-smooth}
 The morphism $p_A'$ is smooth and log smooth of relative dimension $1$, and we have a global isomorphism $\Omega^1_{P_A'(\cL_A)/\underline{Y}_A} \cong \cO_{P_A'(\cL_A)}$.
\end{lemma}
\begin{proof}
 The formation of the associated log structure from a Deligne--Faltings log structure commutes with pull-back of both types of log structures. Thus, for each $\alpha$, we have a Cartesian diagram
 \[
  \xymatrix{
   P'_A(\cL_A)|_\alpha \ar[r]^{\pi} \ar[d] & \PP^1(0 + \infty) \ar[d] \\
   \underline{Y}_{A;\alpha} \ar[r] & \Spec \kk \\
  }
 \]
 of log schemes, where $\PP^1(0 + \infty)$ denotes $\PP^1$ endowed with the log structure coming from the Deligne--Faltings log structure of the ideals of the two points $0$ and $\infty$. By \cite[III, Prop.~1.7.3]{LoAG2018}, the latter log schemes is equal to $\PP^1$ endowed with the divisorial log structure of the two points. Thus, the morphism on the right is smooth and log smooth, and so is its (local) base change $p_A'$.
 
 Let $x$ and $y$ be the two coordinates on $\PP^1$. Then we have $\pi^*x = x_\alpha$ and $\pi^*y = y_\alpha$. On $\bAA^1_x(0) \subseteq \PP^1(0 + \infty)$, we have an element $\hat x \in \M$ in the monoid sheaf mapping to $x \in \cO$. When $\delta_{\PP^1}$ is the log part of the universal derivation, we have $\Omega^1_{\bAA^1_x(0)/\kk} \cong \cO \cdot \delta_P(\hat x)$. In particular, when denoting by $\hat x_\alpha$ the element in $\M_{P'_A(\cL_A)}|_{W'_\alpha}$ induced from $\hat x$, we have 
 $$\Omega^1_{P'_A(\cL_A)/\underline{Y}_A}|_{W'_\alpha} \cong \cO_{W'_\alpha} \cdot \delta_{P/Y}(\hat x_\alpha).$$
 Similarly, we have 
 $$\Omega^1_{P'_A(\cL_A)/\underline{Y}_A}|_{W''_\alpha} \cong \cO_{W''_\alpha} \cdot \delta_{P/Y}(\hat y_\alpha).$$
 Since $y = x^{-1}$ on $\bAA^1_x(0) \cap \bAA^1_y(\infty)$, we have $\hat y_\alpha = \hat x_\alpha^{-1}$ on $W'_\alpha \cap W''_\alpha$; thus, $\delta_{P/Y}(\hat x_\alpha) = - \delta_{P/Y}(\hat y_\alpha)$. On overlaps $W'_\alpha \cap W'_\beta$, we have $x_\alpha = u_{A;\alpha\beta}^{-1} \cdot x_\beta$; it follows from the construction of the log structure associated with a Deligne--Faltings log structure that we have $\hat x_\alpha = u_{A;\alpha\beta}^{-1} \cdot \hat x_\beta$ as well. Thus, $\delta_{P/Y}(\hat x_\alpha) = \delta_{P/Y}(u_{A;\alpha\beta}^{-1}) + \delta_{P/Y}(\hat x_\beta) = \delta_{P/Y}(\hat x_\beta)$; the latter equality follows from $(p_A')^{-1}\cO^*_{Y_A}$-linearity of the universal relative derivation. Similarly, we have $\delta_{P/Y}(\hat y_\alpha) = \delta_{P/Y}(\hat y_\beta)$. In particular, there is a global section $\delta$ of $\Omega^1_{P'_A(\cL_A)/\underline{Y}_A}$ with $\delta|_{W'_\alpha} = \delta_{P/Y}(\hat x_\alpha)$ and $\delta|_{W''_\alpha} = -\delta_{P/Y}(\hat y_\alpha)$. This is the trivialization in the last assertion of the lemma. Note also that $\delta$ is independent of the choice of $s_{A;\alpha}$.
\end{proof}

We define the map $p_A: P_A(\cL_A) \to Y_A$ of (non-enhanced) generically log smooth families by the Cartesian diagram 
\[
 \xymatrix{
  P_A(\cL_A) \ar[r] \ar[d]^{p_A} & P'_A(\cL_A) \ar[d]^{p'_A} \\
  Y_A \ar[r] & \underline{Y}_A. \\
 }
\]
 The space $P_A(\cL_A)$ carries a log structure on $p_A^{-1}(U_A)$; there, the map $p_A$ is smooth and log smooth. The composition with $f_A: Y_A \to S_A$ turns $g_A: P_A(\cL_A) \to S_A$ into a generically log smooth family over $S_A$. We write $\Omega^1_{P_A(\cL_A)/Y_A}$ for the relative differential forms, rather than $\W^1_{P_A(\cL_A)/Y_A}$, since $p_A$ is, as a base change of $p_A'$, in some sense log smooth. We denote the global  generator constructed in Lemma~\ref{pA-log-smooth} by $\gamma = \delta_{P/Y}(\hat x_\alpha) = -\delta_{P/Y}(\hat y_\alpha)$.
 
 The direct image from $p_A^{-1}(U_A)$ gives an exact sequence 
 \begin{equation}\label{fiber-bundle-extension}
  0 \to p_A^*\W^1_{Y_A/S_A} \to \W^1_{P_A(\cL_A)/S_A} \to \Omega^1_{P_A(\cL_A)/Y_A} \to 0.
 \end{equation}
 Namely, $p_A$ is flat, so $p_A^*\W^1_{Y_A/S_A}$ is already $Z$-closed. Except for surjectivity, the sequence is exact because $j_*$ is left exact. The map on the right is surjective because $\delta_{P/S}(\hat x_\alpha)$ is a preimage of $\gamma = \delta_{P/Y}(\hat x_\alpha)$ on $W'_\alpha \cap p_A^{-1}(U_A)$, and similarly for $W''_\alpha$. We also find 
 $$\W^{d + 1}_{P_A(\cL_A)/S_A} \cong p_A^*\W^d_{Y_A/S_A}$$ because this holds on $p_A^{-1}(U_A)$ since $\Omega^1_{P_A(\cL_A)/Y_A} \cong \cO_{P_A(\cL_A)}$, and both sides are $Z$-closed; in particular, due to our log Gorenstein assumption, they are line bundles.
 
 Next, we turn $g_A: P_A(\cL_A) \to S_A$ into an \emph{enhanced} generically log smooth family. The exact sequence \eqref{fiber-bundle-extension} is locally split, and any two splittings differ by a section of 
 $$\cH om(\Omega^1_{P_A(\cL_A)/Y_A},p_A^*\W^1_{Y_A/S_A}).$$
 Inside $p_A^*\W^1_{Y_A/S_A}$, we have $p_A^*\A^1_{Y_A/S_A}$, and the local splittings constructed above and given by $\gamma \mapsto \delta_{P/S}(\hat x_\alpha) = -\delta_{P/S}(\hat y_\alpha)$ all differ by a section of $\cH om(\Omega^1_{P_A(\cL_A)/Y_A},p_A^*\A^1_{Y_A/S_A})$. Thus, we have a class of \emph{distinguished splittings}, which differ by a section in that $\cO_{P_A(\cL_A)}$-module. Given a local distinguished splitting $B_1: \Omega^1_{P_A(\cL_A)/Y_A} \to \W^1_{P_A(\cL_A)/S_A}$, we define 
 $$\A^i_{P_A(\cL_A)/S_A} := p_A^*\A^i_{Y_A/S_A} + p_A^*\A^{i - 1}_{Y_A/S_A} \wedge B_1(\Omega^1_{P_A(\cL_A)/Y_A}) \subseteq \W^i_{P_A(\cL_A)/S_A}.$$
 This is independent of the choice of distinguished splitting $B_1$. At this point, the discussion is analogous to the one given later in Chapter~\ref{log-prereg-sec} in more detail, so we just summarize quickly the situation. We have a diagram 
 \[
 \xymatrix{
  0 \ar[r] & p_A^*\W^i_{Y_A/S_A} \ar[r]^-{E_i} & \W^i_{P_A(\cL_A)/S_A} \ar[r]^-{Q_i} & p_A^*\W^{i - 1}_{Y_A/S_A} \otimes \Omega^1_{P_A(\cL_A)/Y_A} \ar[r] & 0 \\
  0 \ar[r] & p_A^*\A^i_{Y_A/S_A} \ar[r]^-{E_i} \ar@{^{(}->}[u] & \A^i_{P_A(\cL_A)/S_A} \ar[r]^-{Q_i} \ar@{^(->}[u] & p_A^*\A^{i - 1}_{Y_A/S_A} \otimes \Omega^1_{P_A(\cL_A)/Y_A} \ar[r] \ar@{^(->}[u] & 0 \\
 }
\]
 where both rows are locally split exact sequences. A distinguished local splitting $B_1$ induces a splitting 
 $$T_i: \A^i_{P_A(\cL_A)/S_A} \to p_A^*\A^i_{Y_A/S_A}$$
 which satisfies $T_i(\alpha) \wedge T_j(\beta) = T_{i + j}(\alpha \wedge \beta)$. Dualizing the diagram, we obtain a locally split exact sequence 
 $$0 \to p_A^*\V^{p + 1} \otimes \Theta^1_{P_A(\cL_A)/Y_A} \xrightarrow{I_p} \V^p_{P_A(\cL_A)/S_A} \xrightarrow{F_p} p_A^*\V^p_{Y_A/S_A} \to 0;$$
 when dualizing the splitting $T_{-p}$ of $E_{-p}$, we obtain a splitting 
 $$S_p: p_A^*\V^p_{Y_A/S_A} \to \V^p_{P_A(\cL_A)/S_A}$$
 of $F_p$. These functions satisfy a large number of identities given in Lemma~\ref{formulae-LXS}. Then we define 
 $$\G^p_{P_A(\cL_A)/S_A} := I_p(p_A^*\G^{p + 1}_{Y_A/S_A} \otimes \Theta^1_{P_A(\cL_A)/Y_A}) + S_p(p_A^*\G^p_{Y_A/S_A}) \subseteq \V^p_{P_A(\cL_A)/S_A},$$
 which is independent of the original choice of distinguished splitting $B_1$. The proof of Proposition~\ref{modif-is-enhanced} shows that we obtain an enhanced generically log smooth family by setting 
 $$\G\C^\bullet_{P_A(\cL_A)/S_A} := (\G^\bullet_{P_A(\cL_A)/S_A}, \, \A^\bullet_{P_A(\cL_A)/S_A}).$$
 
 \vspace{\baselineskip}
 
 The construction of $p_A: P_A(\cL_A) \to Y_A$ commutes with base change along $B' \to B$. Namely, the formation of the $\PP^1$-bundle $\PP(\cO_{Y_A} \oplus \cL_A)$ commutes with base change, the formation of the ideals $\I_{0,A}$ and $\I_{\infty,A}$ commutes with base change, the formation of the Deligne--Faltings log structure and its associated log structure commutes with base change, and the pull-back along $Y_A \to \underline{Y}_A$ commutes with base change. From the above diagram with $E_i$ and $Q_i$, we see that the formation of $\G\C^\bullet_{P_A(\cL_A)/S_A}$ commutes with base change as well. Thus, the formation of $g_A: P_A(\cL_A) \to S_A$ as an enhanced generically log smooth family commutes with base change.
 
 In particular, we have an enhanced generically log smooth family $g_0: P_0(\cL_0) \to S_0$, and $g_A: P_A(\cL_A) \to S_A$ is a deformation thereof over $p_0^{-1}(Y_0) \subseteq P_0(\cL_0)$. 
 
 When $\G\C^\bullet_{Y_A/S_A} = \V\,\W^\bullet_{Y_A/S_A}$, then $\G\C^\bullet_{P_A(\cL_A)/S_A} = \V\,\W^\bullet_{P_A(\cL_A)/S_A}$ as well so that we do not leave the setting of generically log smooth families in the case where we have started in this setting.

 \section{Isomorphisms and automorphisms}
 
 Let $f: (Y_A,\cL_A) \to S_A$ and $\tilde f: (\tilde Y_A,\tilde\cL_A) \to S_A$ be two enhanced generically log smooth deformations with a line bundle of type $\D$, and assume that $\varphi = (\phi,\Phi,\psi): (Y_A,\cL_A) \to (\tilde Y_A,\tilde\cL_A)$ is an isomorphism over $f_0: (Y_0,\cL_0) \to S_0$. Then we have an induced isomorphism 
 $$P_A(\varphi):\: P_A(\cL_A) \xrightarrow{\cong} \tilde P_A(\tilde\cL_A)$$
 of enhanced generically log smooth families which is compatible with $(\phi,\Phi): Y_A \cong \tilde Y_A$ and with the identity on $(Y_0,\cL_0)$. This construction is functorial for compositions of isomorphisms, and it is compatible with base change along $B' \to B$.
 
 More generally, we also have this construction for isomorphisms of the underlying (non-enhanced) generically log smooth families. Before going to the enhanced case, we study automorphisms in the non-enhanced case. So let $\varphi = (\phi,\Phi,\psi)$ be an automorphism of $f: (Y_A,\cL_A) \to S_A$ as a non-enhanced generically log smooth family with a line bundle. Let us write $\psi(s_{A;\alpha}) = v_\alpha \cdot s_{A;\alpha}$. This determines $v_\alpha \in \Gamma(Y_{A;\alpha},\cO_{Y_A}^*)$ uniquely, and we have $v_\alpha|_0 = 1$. First, $\varphi$ induces an automorphism of $\mathrm{Sym}^\bullet(\cO_{Y_A} \oplus \cL_A)$, and this in turn induces an automorphism of $\PP(\cO_{Y_A} \oplus \cL_A)$, which we denote by $\hat\phi$. We have $\hat\phi(x_\alpha) = v_\alpha^{-1} \cdot x_\alpha$ on $W'_\alpha$ and $\hat\phi(y_\alpha) = v_\alpha \cdot y_\alpha$ on $W''_\alpha$. We also have an induced automorphism $\hat\phi_0$ of the ideal $\I_{0,A}$ and an induced automorphism $\hat\phi_\infty$ of the ideal $\I_{\infty,A}$. Together, they define an automorphism of the Deligne--Faltings log structure, and hence an automorphism $(\hat\phi,\hat\Phi')$ of the log scheme $P_A'(\cL_A)$. By construction, we have a commutative diagram 
 \[
  \xymatrix{
   \I_{0,A}^* \ar[r] \ar[d]^{\hat\phi_0} & \M_{P'_A(\cL_A)} \ar[d]^{\hat\Phi'} \ar[r] & \cO_{P'_A(\cL_A)} \ar[d]^{\hat\phi} \\
  \I_{0,A}^* \ar[r] & \M_{P'_A(\cL_A)} \ar[r] & \cO_{P'_A(\cL_A)}  \\
  }
 \]
 and a similar one for $\I_{\infty,A}^*$. On $W'_\alpha$, the element $x_\alpha$ is a generator of $\I_{0,A}$, so we have $x_\alpha \in \I_{0,A}^*$. Thus, we find $\hat\Phi'(\hat x_\alpha) = v_\alpha^{-1} + \hat x_\alpha$ on $W'_\alpha$, and similarly $\hat\Phi'(\hat y_\alpha) = v_\alpha + \hat y_\alpha$ on $W''_\alpha$. On $p_A^{-1}(U_0) \cap W'_\alpha$, we obtain $\hat\Phi(\hat x_\alpha) = v_\alpha^{-1} + \hat x_\alpha$, and on $p_A^{-1}(U_0) \cap W''_\alpha$, we obtain $\hat\Phi(\hat y_\alpha) = v_\alpha + \hat y_\alpha$, both after base change along $Y_A \to \underline{Y}_A$, where $\hat\Phi$ is the log part of the induced automorphism on $P_A(\cL_A)$.
 
 We have a canonical isomorphism 
 $$p_A^*\cL_A \xrightarrow{\cong} \I_{\infty,A} \otimes \I_{0,A}^{-1}, \quad s_{A;\alpha} \mapsto \begin{cases}
               1 \otimes x_\alpha^{-1}, & W'_\alpha \\ 
               y_\alpha \otimes 1, & W''_\alpha                                                                                      
                                                                                    \end{cases}.
$$
 When we apply $\hat\phi_\infty \otimes \hat\phi_0^{-1}$ on the right, then we obtain an induced automorphism $\hat\psi: p_A^*\cL_A \to p_A^*\cL_A$ with $\hat\psi(a \cdot s) = \hat\phi(a) \cdot \hat\psi(s)$ and $\hat\psi(s_{A;\alpha}) = v_\alpha \cdot s_{A;\alpha}$. We can reconstruct $\hat\psi$ from $(\hat\phi,\hat\Phi)$ since $\hat\Phi(\hat y_\alpha) = v_\alpha + \hat y_\alpha$.
 
 We construct a map 
 $$\rho: (p_A)_*\A ut_{P_A(\cL_A)/P_0(\cL_0)} \to \A ut_{(Y_A,\cL_A)/(Y_0,\cL_0)}$$
 in the other direction. Let $(\hat\phi,\hat\Phi)$ be an automorphism of $g_A: P_A(\cL_A) \to S_A$. Since $(p_A)_*\cO_{P_A(\cL_A)} = \cO_{Y_A}$, we can define $\phi := (p_A)_*\hat\phi$. There is a derivation $(\hat D,\hat\Delta)$ such that $(\hat\phi,\hat\Phi) = \mathrm{exp}(\hat D,\hat\Delta)$. Thus, $\hat\Phi$ is of the form 
 $$\hat\Phi(m) = m + \alpha^{-1}\left(\sum_{n = 0}^\infty\frac{[\hat\Delta(m) + \hat D]^n(1)}{n!}\right).$$
 In particular, this also holds for 
 $$(p_A)_*\hat\Phi: (p_A)_*\M_{P_A(\cL_A)} \to (p_A)_*\M_{P_A(\cL_A)},$$
 and when we use the same formula with $\hat\Delta$ replaced by $\M_{Y_A} \to (p_A)_*\M_{P_A(\cL_A)} \to \cO_{Y_A}$ and $\hat D$ replaced by $(p_A)_*\hat D$, we obtain a map $\Phi: \M_{Y_A} \to \M_{Y_A}$ which is compatible with $(p_A)_*\hat\Phi$ along $\M_{Y_A} \to (p_A)_*\M_{P_A(\cL_A)}$. In particular, the pair $(\phi,\Phi)$ is a log morphism, and it is in fact an automorphism because we can apply the same construction to the inverse of $(\hat\phi,\hat\Phi)$, and because this construction preserves compositions and the identity.
 
 Given $(\hat\phi,\hat\Phi)$, there is a unique $v_\alpha \in \Gamma(W''_\alpha \cap p_A^{-1}(U_0), \cO_{P_A(\cL_A)}^*)$ with $\hat\Phi(\hat y_\alpha) = \hat y_\alpha + v_\alpha$. It can in fact be extended to $W_\alpha \cap p_A^{-1}(U_A)$ such that $\hat\Phi(\hat x_\alpha) = \hat x_\alpha + v_\alpha^{-1}$ on $W'_\alpha$. Since $(p_A)_*\cO_{P_A(\cL_A)} = \cO_{Y_A}$, we have indeed $v_\alpha \in \Gamma(Y_{A;\alpha},\cO_{Y_A}^*)$. A direct computation based on the transformation behavior of $\hat y_\alpha$ yields $v_\alpha = \hat\phi(u_{A;\alpha\beta}) \cdot u_{A;\alpha\beta}^{-1} \cdot v_\beta$. This shows that 
 $$\hat\psi: p_A^*\cL_A \to p_A^*\cL_A, \quad a \cdot s_{A;\alpha} \mapsto \hat\phi(a) \cdot v_\alpha \cdot s_{A;\alpha},$$
 is well-defined. Moreover, we have $\hat\psi(a \cdot s) = \hat\phi(a) \cdot \hat\psi(s)$. As a consequence of the projection formula, we have $(p_A)_*p_A^*\cL_A = \cL_A$. Then we set $\psi := (p_A)_*\hat\psi$. Since this $\psi$ forms, together with $\phi$ and $\Phi$, in fact an automorphism of $f_A: (Y_A,\cL_A) \to S_A$ over $f_0: (Y_0,\cL_0) \to S_0$, this completes the construction of $\rho$.

 We observe that $\rho(P_A(\phi,\Phi,\psi)) = (\phi,\Phi,\psi)$. For $\phi$ and $\psi$, this is straightforward. For $\Phi$, the easiest way to see this is by noting that the forgetful map $(\phi,\Phi) \to \phi$ is injective because of the density of the strict locus; namely, on the strict locus, the forgetful map is bijective, and the automorphism sheaf has injective restrictions to the strict locus. This shows that 
 $$P_A: \A ut_{(Y_A,\cL_A)/(Y_0,\cL_0)} \to (p_A)_*\A ut_{P_A(\cL_A)/P_0(\cL_0)}$$
 is injective, and that $\rho$ is surjective.

 \begin{prop}\label{PA-rho-isom}\note{PA-rho-isom}
  The maps $P_A(-)$ and $\rho$ are inverse isomorphisms of groups.
 \end{prop}

 In order to prove this, we study derivations first. The exact sequence \eqref{fiber-bundle-extension} above is locally split. Thus, its dual sequence 
 $$0 \to \cO_{P_A(\cL_A)} \to \Theta^1_{P_A(\cL_A)/S_A} \to \cH om(p_A^*W^1_{Y_A/S_A},\cO_{P_A(\cL_A)}) \to 0$$
 is locally split exact as well. The right hand side is canonically isomorphic to $p_A^*\Theta^1_{Y_A/S_A}$.
 
 In the $\PP^1$-bundle $p_A: P_A(\cL_A) \to Y_A$, we have 
 $$(p_A)_*\cO_{P_A(\cL_A)} = \cO_{Y_A}, \quad R^q(p_A)_*\cO_{P_A(\cL_A)} = 0 \enspace \mathrm{for} \enspace q \geq 1.$$
 Furthermore, for \emph{every} coherent sheaf $\F$ on $Y_A$, we have the projection formula 
 $$(p_A)_*p_A^*\F = \F, \quad R^q(p_A)_*p_A^*\F = 0 \enspace \mathrm{for} \enspace q \geq 1.$$
 This is a consequence of the very general projection formula \cite[Prop.~3.9.4]{Lipman2009}. Thus, we obtain an exact sequence 
 $$0 \to \cO_{Y_A} \xrightarrow{\iota} (p_A)_*\Theta^1_{P_A(\cL_A)/S_A} \xrightarrow{\pi} \Theta^1_{Y_A/S_A} \to 0.$$
 \begin{prop}\label{Atiyah-comp-isom}\note{Atiyah-comp-isom}
  This exact sequence is isomorphic to the Atiyah sequence \eqref{AE} before Lemma~\ref{Atiyah-ext-lemma}.
 \end{prop}
 \begin{proof}
  Let $\UU \subseteq Y_A$ be an open subset, and let $(D,\Delta) \in \Gamma(p_A^{-1}(\UU),\Theta^1_{P_A(\cL_A)/S_A})$ be a derivation. Then $\pi(D,\Delta)$ is given by the adjoint maps of $p_A^{-1}\cO_{Y_A} \to \cO_{P_A}$ and $p_A^{-1}\M_{Y_A} \to \cO_{P_A}$, using that $(p_A)_*\cO_{P_A} = \cO_{Y_A}$. Given such $(D,\Delta)$, we can furthermore form a map $E: p_A^*\cL_A \to p_A^*\cL_A$ given by 
 $$E(a \cdot s_{A;\alpha}) := D(a) \cdot s_{A;\alpha} + a \cdot \Delta(\hat y_\alpha) \cdot s_{A;\alpha}$$
 on $p_A^{-1}(\UU) \cap W_\alpha$. Although $\hat y_\alpha$ is only defined on $W''_\alpha$, the element $\Delta(\hat y_\alpha)$ is well-defined on $W_\alpha$ because we have $\Delta(\hat y_\alpha) = - \Delta(\hat x_\alpha)$ on $W'_\alpha \cap W''_\alpha$. The map $E$ is well-defined because it turns out that 
 $$E(a \cdot s_{A;\alpha}) = E(a \cdot u_{A;\alpha\beta} \cdot s_{A;\beta})$$
 using the formula for $W_\alpha$ on the left and the one for $W_\beta$ on the right. Furthermore, for $s \in p_A^*\cL_A$ and $a \in \cO_{P_A(\cL_A)}$, we have $E(a \cdot s) = D(a) \cdot s + a \cdot E(s)$. Since $(p_A)_*p_A^*\cL_A = \cL_A$, we find that sections on $p_A^{-1}(\UU)$ give rise to a derivation of $f: (Y_A,\cL_A) \to S_A$. In other words, we have defined a map 
 $$\eta: \enspace (p_A)_*\Theta^1_{P_A(\cL_A)/S_A} \to \Theta^1_{Y_A/S_A}(\cL_A).$$
 It is easy to see that this map is $\cO_{Y_A}$-linear. Furthermore, a direct computation shows that the diagram 
 \[
  \xymatrix{
   0 \ar[r] & \cO_{Y_A} \ar[r] \ar@{=}[d] & (p_A)_*\Theta^1_{P_A(\cL_A)/S_A} \ar[r] \ar[d]^\eta & \Theta^1_{Y_A/S_A} \ar[r] \ar@{=}[d] & 0 \\
   0 \ar[r] & \cO_{Y_A} \ar[r] & \Theta^1_{Y_A/S_A}(\cL_A) \ar[r] & \Theta^1_{Y_A/S_A} \ar[r]  & 0 \\
  }
 \]
 is not only commutative on the right but also commutative on the left. Hence $\eta$ is an isomorphism.
 \end{proof}

 \begin{proof}[Proof of Proposition~\ref{PA-rho-isom}]
  It is quite easy to show that 
  $$\m_A \cdot (p_A)_*\Theta^1_{P_A(\cL_A)/S_A} = (p_A)_*(\m_A \cdot \Theta^1_{P_A(\cL_A)/S_A}).$$
  Thus, we have a diagram 
  \[
   \xymatrix{
    (p_A)_*\A ut_{P_A(\cL_A)/P_0(\cL_0)} \ar[r]^\rho & \A ut_{(Y_A,\cL_A)/(Y_0,\cL_0)} \\
    \m_A \cdot (p_A)_*\Theta^1_{P_A(\cL_A)/S_A} \ar[u]^{(p_A)_*\mathrm{exp}}_{\cong} \ar[r]^\eta_\cong & \m_A \cdot \Theta^1_{Y_A/S_A}(\cL_A) \ar[u]^{\mathrm{exp}}_{\cong}. \\
   }
  \]
  This diagram is commutative. To see this, we have to show that 
  $$(\phi_1,\Phi_1,\psi_1) := \rho(\mathrm{exp}(\hat D,\hat \Delta)) = \mathrm{exp}(\eta(\hat D,\hat\Delta)) =: (\phi_2,\Phi_2,\psi_2).$$
  The equality $\phi_1 = \phi_2$ is straightforward. Again, by the density of the strict locus, we also find $\Phi_1 = \Phi_2$. For $\psi_1 = \psi_2$, it suffices to show $\mathrm{exp}(\hat E) = \hat \psi$, where $\hat E$ is defined by the formula in the proof of Proposition~\ref{Atiyah-comp-isom}, and $\hat\psi$ is constructed from $(\hat\phi,\hat\Phi)$ as described above. Thus, we have to show $\mathrm{exp}(\hat E)(s_{A;\alpha}) = v_\alpha \cdot s_{A;\alpha}$. On the one hand side, we have 
  $$v_\alpha = \sum_{n = 0}^\infty \frac{[\hat\Delta(\hat y_\alpha) + \hat D]^n(1)}{n!}.$$
  On the other hand side, we have 
  $$\hat E(a \cdot s_{A;\alpha}) = (\hat\Delta(\hat y_\alpha) \cdot a + D(a)) \cdot s_{A;\alpha},$$
  proving the claim by induction. It is clear that the two vertical arrows are isomorphisms. Since the formation of $\eta$ commutes with base change along $A \to \kk$, also the lower horizontal map is an isomorphism. Thus, $\rho$ is an isomorphism, and so is $P_A(-)$ since $\rho \circ P_A(-) = \mathrm{id}$. Finally, to see that the maps are group homomorphisms, it is sufficient to note that $P_A(-)$ is a group homomorphism by the functorial nature of its construction.
 \end{proof}

 Now we turn to the enhanced case. From the definitions, we find that 
 $$\Gamma^{1}_{P_A(\cL_A)/S_A} = (F_{-1})^{-1}(p_A^*\Gamma^{1}_{Y_A/S_A})$$
 so that we have 
 $$(p_A)_*\Gamma^{1}_{P_A(\cL_A)/S_A} = \pi^{-1}(\Gamma^{1}_{Y_A/S_A}).$$ 
 Similarly, we have 
 $$\Gamma^{1}_{Y_A/S_A}(\cL_A) = q^{-1}(\Gamma^{1}_{Y_A/S_A})$$
 where $q: \Theta^1_{Y_A/S_A}(\cL_A) \to \Theta^1_{Y_A/S_A}$ is the surjection in the Atiyah extension. Thus, we obtain an induced isomorphism 
 $$\eta: (p_A)_*\Gamma^{1}_{P_A(\cL_A)/S_A} \xrightarrow{\cong} \Gamma^{1}_{Y_A/S_A}(\cL_A).$$
 From the proof of Proposition~\ref{PA-rho-isom}, we see that $\rho$ gives rise to a one-to-one correspondence between gauge transforms of $g_A: P_A(\cL_A) \to S_A$ and gauge transforms of $f_A: (Y_A,\cL_A) \to S_A$. In particular, if $\varphi$ is a gauge transform, then $P_A(\varphi)$ is a gauge transform as well.

 \section{The enhanced system of deformations $\D(\cL_0)$}
 
 We use the construction of $P_A(\cL_A)$ to form an enhanced system of deformations $\D(\cL_0)$ for $g_0: P_0(\cL_0) \to S_0$, subordinate to $\V(\cL_0) := \{p_0^{-1}(V_\alpha)\}_\alpha$. We endow each local model $V_{\alpha;A} \to S_A$ of $\D$ with the line bundle $\cL_{\alpha;A} := \cO_{V_{\alpha;A}}$. For the restriction maps along some $B' \to B$, we take the obvious map $\cL_{\alpha;B'} \to \cL_{\alpha;B}$. For the restriction to $B = \kk$, we use the identification $\cL_0|_\alpha \cong \cO_{X_0}|_\alpha$ via $s_{0;\alpha} \in \cL_0|_\alpha$. For the comparison maps, we start from $\cL_{\alpha;0}|_{\alpha\beta} \cong \cL_0|_{\alpha\beta} \cong \cL_{\beta;0}|_{\alpha\beta}$ and lift this order by order to an isomorphism $\cL_{\alpha;k}|_{\alpha\beta} \cong \cL_{\beta;k}|_{\alpha\beta}$.\footnote{In general, the identification $\cL_{\alpha;k}|_{\alpha\beta} = \cO_{V_{\alpha;k}}|_{\alpha\beta} \cong \cO_{V_{\beta;k}}|_{\alpha\beta} = \cL_{\beta;k}|_{\alpha\beta}$ does not work because it is not compatible with our choice on $k = 0$.} For a general $A$, we take an appropriate pull-back from some $A_k$. Now we form $P_{\alpha;A} := P_A(\cL_{\alpha;A})$. They come with restriction maps along $B' \to B$, and on overlaps, we have isomorphisms $P_{\alpha;A}|_{\alpha\beta} \cong P_{\beta;A}|_{\alpha\beta}$. The cocycles on overlaps are inner automorphisms because $P_A(\varphi)$ is a gauge transform whenever $\varphi$ is a gauge transform. The open cover $\V(\cL_0)$ can be seen to be admissible as specified in Definition~\ref{enhanced-sys-of-defo} by using that a distinguished splitting $B_1$ exists on every $p_0^{-1}(V_\alpha)$, using the general projection formula mentioned above, and using the appropriate Leray spectral sequence. This gives the enhanced system of deformations $\D(\cL_0)$. 
 
 If $f: (Y_A,\cL_A) \to S_A$ is an enhanced deformation of type $\D$, its restriction to $Y_0 \cap V_\alpha$ is isomorphic to $(V_{\alpha;A},\cL_{\alpha;A}) \to S_A$ via a specified isomorphism which differ by an inner automorphism on $V_\alpha \cap V_\beta$. Thus, we have an isomorphism from $P_A(\cL_A)|_\alpha$ to $P_{\alpha;A}$, differing by an inner automorphism on $V_\alpha \cap V_\beta$. In other words, $P_A(\cL_A) \to S_A$ is an enhanced generically log smooth deformation of $g_0: P_0(\cL_0) \to S_0$ of type $\D(\cL_0)$. This defines a natural transformation 
 $$\mathrm{ELD}^\D_{X_0/S_0}(\cL_0) \Rightarrow \mathrm{ELD}^{\D(\cL_0)}_{P_0(\cL_0)/S_0}.$$
 
 \begin{lemma}\label{line-bundle-fiber-bundle-defo-isom-gen-log-sm}\note{line-bundle-fiber-bundle-defo-isom-gen-log-sm}
  This is an isomorphism of deformation functors.
 \end{lemma}
 \begin{proof}
  Since $P_A(-)$ induces a bijection on automorphisms, this follows from tracking the gluing like in the proof of \cite[Prop.~4.2]{Felten2022}.
 \end{proof}

\par\vspace{\baselineskip}

An important special case is when both $f_0: X_0 \to S_0$ is log toroidal with respect to $a_0: S_0 \to A_Q$, and each local model $V_{\alpha;A} \to S_A$ is log toroidal with respect to $a_A: S_A \to A_Q$. In this situation, also $g_A: P_A(\cL_A) \to S_A$ is log toroidal with respect to $a_A: S_A \to A_Q$. To see this, first let $(Q \subset P,\F)$ be an elementary log toroidal datum in the sense of \cite[Defn.~3.1]{FFR2021}. Let $\hat P := P \times \NN$, and let $\hat \F := \{F \times \NN \ | \ F \in \F\} \cup \{P \times 0\}$. Mimicking the construction of $P_A(\cL_A)$ for $A_{P,\F}$, we form a commutative diagram 
\[
 \xymatrixrowsep{1em}\xymatrix{
  & A_{\hat P,\hat \F} \ar[dd] \ar[rr] \ar[dl]^p & & \underline{A}_{P,\F} \times A_\NN \ar[dl] \ar[dd] & & \\
  A_{P,\F} \ar[dd] \ar[rr] & & \underline{A}_{P,\F} \ar[dd] & & & \\
  & A_Q \times A_\NN \ar[dl] \ar[rr] & & \underline{A}_Q \times A_\NN \ar[dl] \ar[rr] & & A_\NN \ar[dl] \\
  A_Q \ar[rr] & & \underline{A}_Q \ar[rr] & & \Spec(\kk) & \\
 }
\]
 of log schemes. In this diagram, the log structure is defined everywhere on every entry, but it is not necessarily coherent. Since $A_Q \times A_\NN = A_{Q \oplus \NN}$, the lower two horizontal squares are Cartesian. Also the middle vertical square is Cartesian. Let $x$ be the coordinate of $A_\NN$ on the very right. By \cite[III, Prop.~1.7.3]{LoAG2018}, $\underline{A}_Q \times A_\NN$ is the log scheme with divisorial log structure defined by $\{x = 0\}$, and $\underline{A}_{P,\F} \times A_\NN$ carries also the divisorial log structure defined by $\{x = 0\}$. This shows that the arrows emanating from $A_{\hat P,\hat \F}$ are well-defined and give rise to a commutative diagram. By \cite[Cor.~3.11]{FFR2021}, on $A_{P,\F}$, we have a log smooth locus $U_{P/Q} = U_1 \cup U_2$ such that $A_{P,\F}|_{U_1} = A_P|_{U_1}$, and $A_{P,\F}|_{U_2} \to A_Q$ is strict and smooth. One can then show that the left vertical square is Cartesian on $U_{P/Q}$ by analyzing $U_1$ and $U_2$ separately. For $U_1$, one uses that $A_{P_1 \oplus P_2} \cong A_{P_1} \times A_{P_2}$, and for $U_2$ one uses Lemma~\ref{smooth-map-div-log-str} below, a result which seems to be missing in \cite{LoAG2018} although similar results are discussed. Now also the upper horizontal square is Cartesian on $U_{P/Q}$. Finally, the log structure on $\underline{A}_{P,\F} \times A_\NN$ is the one associated with the Deligne--Faltings log structure coming from the ideal $(x)$ because the map to $A_\NN$ is strict. Thus, on $p^{-1}(U_{P/Q})$, the map $p: A_{\hat P,\hat\F} \to A_{P,\F}$ is precisely the analogue of either $W'_\alpha$ or $W''_\alpha$ in the construction of $P_A(\cL_A)$.
 
 \begin{lemma}\label{smooth-map-div-log-str}\note{smooth-map-div-log-str}
  Let $Q$ be a sharp toric monoid, and let $p: X \to A_Q$ be a strict and smooth map. Then $X$ carries the divisorial log structure defined by $p^{-1}(D_Q)$.
 \end{lemma}
 \begin{proof}
  First assume that we have a factorization via an \'etale map $r: X \to A_Q \times \bAA^r$ and the projection $q: A_Q \times \bAA^r \to A_Q$. Following the proof of \cite[III, Thm.~1.9.4]{LoAG2018}, it is sufficient to show that 
  $$q^{-1}\underline\Gamma_{D_Q}(Div_{A_Q}^+) \to \underline\Gamma_{D_Q \times \bAA^r}(Div_{A_Q \times \bAA^r}^+)$$
  is an isomorphism in order to show that $A_Q \times \bAA^r$ carries the divisorial log structure defined by $D_Q \times \bAA^r$. Here, $\underline\Gamma_Y(Div_X^+)$ denotes the sheaf of effective Cartier divisors with support in $Y \subset X$. Now the argument in the proof of \cite[III, Lemma~1.6.7]{LoAG2018} shows that this is the case. Since every irreducible component of $D_Q \times \bAA^r$ is pure of codimension $1$ and normal, hence geometrically unibranch, we can apply \cite[III, Prop.~1.6.5]{LoAG2018} to the \'etale map $r: X \to A_Q \times \bAA^r$ and obtain, after untangling the definitions, that $X$ carries the divisorial log structure defined by $Y := r^{-1}(D_Q \times \bAA^r)$. Because every smooth map $p: X \to A_Q$ can, locally in the \'etale topology, be factorized as above, $X$ carries the divisorial log structure defined by $p^{-1}(D_Q)$ in the \'etale topology. Another application of \cite[III, Prop.~1.6.5]{LoAG2018} yields that this is the same as the divisorial log structure in the Zariski topology.
 \end{proof}

With this preparation, we show:

\begin{lemma}\label{log-toroidal-fiber-bundle}\note{log-toroidal-fiber-bundle}
 In Situation~\ref{gen-log-sm-fiber-bundle-sitn}, assume that $f_0: X_0 \to S_0$ is log toroidal with respect to $a_0: S_0 \to A_Q$, and that each local model $V_{\alpha;A} \to S_A$ is log toroidal with respect to $a_A: S_A \to A_Q$. Let $f_A: (Y_A,\cL_A) \to S_A$ be a generically log smooth deformation with a line bundle of type $\D$. Then $g_A: P_A(\cL_A) \to S_A$ is log toroidal with respect to $a_A: S_A \to A_Q$.
\end{lemma}
\begin{proof}
 If $A_{P,\F} \to A_Q$ is a local model of $V_{\alpha;A} \to S_A$, then $A_{\hat P,\hat\F} \to A_Q$ is a local model of $P_A(\cO_{V_{\alpha;A}})$ constructed over $V_{\alpha;A}$. Thus, $P_A(\cL_A)$ is log toroidal with respect to $a_A: S_A \to A_Q$.
\end{proof}



\chapter{Algebraic degenerations}\label{alg-degen-sec}\note{alg-degen-sec}

In this chapter, we study algebraic families instead of infinitesimal or formal ones. These algebraic families are obtained by lifting an ample line bundle to arbitrary order on a projective log Calabi--Yau space $f_0: X_0 \to S_0$, which is often possible by the results of Chapter~\ref{fiber-bundle-constr-sec}. We restrict our attention to the case $Q = \NN$, which is both simple and the most important case in practice. We work over $\kk = \CC$ for simplicity.

\begin{sitn}\label{alg-degen-sitn}\note{alg-degen-sitn}
 Let $\Lambda = \CC\llbracket t\rrbracket$, and let $f_0: X_0 \to S_0$ be a torsionless enhanced generically log smooth family of relative dimension $d$. Assume that $f_0: X_0 \to S_0$ is projective, and let $\cL_0$ be an ample line bundle on $X_0$. Assume that $X_0$ is connected. Let $\V = \{V_\alpha\}_\alpha$ be an $\cL_0$-admissible open cover of $X_0$, and let $\D$ be an enhanced system of deformations subordinate to $\D$. Assume that $\mathrm{ELD}^\D_{X_0/S_0}(\cL_0)$ is unobstructed. The family $f_0: X_0 \to S_0$ may or may not be log Calabi--Yau.
\end{sitn}
\begin{ex}
 Theorem~\ref{proper-log-CY-log-toroidal-unobstr-pairs} shows that we are in Situation~\ref{alg-degen-sitn} when the enhanced generically log smooth family is an actual generically log smooth family, all local models are log toroidal with respect to $a_A: S_A \to A_\NN$, and $f_0: X_0 \to S_0$ is log Calabi--Yau.
\end{ex}

In Situation~\ref{alg-degen-sitn}, let us fix a deformation $f_k: (X_k,\cL_k) \to S_k$ for every $k \geq 0$. Let $N > 0$ be such that $\cL_0^{\otimes N}$ is very ample and $H^n(X_0,\cL_0^{\otimes N}) = 0$ for $n \geq 1$, and let $\M_k := \cL_k^{\otimes N}$. Then $\M_0$ defines a closed embedding $\Phi_0: X_0 \to \PP^s$ for $s = h^0(X_0,\M_0) - 1$. By lifting the generators of $\M_0$, we obtain a surjection $\cO_{X_k}^{\oplus (s + 1)} \to \M_k$, which defines a morphism $\Phi_k: X_k \to \PP^s_{S_k}$. The formation of this morphism is compatible with base change so that $\Phi_k \times_{S_k} S_0 = \Phi_0$. This shows that $\Phi_k$ is a closed immersion, and that $\M_k$ is very ample.

The limit of the flat deformations $f_k: X_k \to S_k$ is a formal scheme $f_\infty: \mathfrak{X} \to \mathfrak{S}$ over $\mathfrak{S} = \mathrm{Spf}\,\CC\llbracket t\rrbracket$. Let us recall the proof that $\mathfrak{X}$ is a Noetherian formal scheme, as given in \cite[Prop.~21.1]{Hartshorne2010}: Let $W \subseteq X_0$ be affine, and let $B_k = \Gamma(W,\cO_{X_k})$. By \cite[09B8]{stacks}, the limit $B_\infty = \varprojlim_k B_k$ is $t$-adically complete, and $(W,\cO_{\mathfrak{X}}|_W)$ is the formal completion of $\Spec B_\infty$ in $\Spec B_0$. It thus remains to show that $B_\infty$ is Noetherian. Since $f_0: X_0 \to S_0$ is of finite type, we can find a surjection $A_0[\underline x] \to B_0$. We can lift the images of the variables to any order so that we obtain maps $P_k := A_k[\underline x] \to B_k$. By \cite[09ZW]{stacks}, this map is surjective. Thus, we have a surjective homomorphism $P_\infty := \varprojlim_k A_k[\underline x] \to B_\infty$.
Let $A = \CC\lsem t\rsem$. Then $P_\infty$ is the $t$-adic completion of $A[\underline x]$. In particular, $P_\infty$ is Noetherian, and so is its quotient $B_\infty$.\footnote{An earlier version of the manuscript has a mistake at this point. When denoting by $I_k$ the kernel of $A_k[\underline x] \to B_k$, it is true that $I_{k + 1} \otimes_{A_{k + 1}} A_k = I_k$. However, it may not be true that $I = \{f \in A[\underline x] \ |\ \forall k: \ f|_k \in I_k\}$ is an ideal with $I \otimes_A A_k = I_k$---the ideal $I$ might not surject onto $I_k$. In this case, it is no longer true that we would have $(A[\underline x]/I) \otimes_A A_k = B_k$. It is plausible that there are indeed affine (!) formal families $(f_k: X_k \to S_k)_k$ which are not the completion of any scheme of finite type over $\Spec \kk\lsem t\rsem$, but we currently do not have a proven counterexample.}

The closed immersions $\Phi_k$ give rise to a closed immersion $\Phi_\infty: \mathfrak{X} \to \mathfrak{P}^s$ of formal schemes over $\mathfrak{S}$, as can be seen by using \cite[Lemma~10.14.4]{EGAI}. Here, $\mathfrak{P}^s$ is the limit of $\PP^s_{S_k} \to S_k$, and it is also equal to the formal completion of $\PP^s_S \to S$ in $t \cdot \cO_{\PP^s_S}$, where $S = \Spec \CC\llbracket t\rrbracket$. Now \cite[Cor.~5.1.8]{EGAIII-1} states that closed subschemes of $\PP^s_S$ are in one-to-one correspondence with closed formal subschemes of $\mathfrak{P}^s$. Thus, there is a closed subscheme $\Phi: X \to \PP^s_S$ whose completion in $X_0 = X \cap \PP^s_0$ is $\mathfrak{X}$; it induces the closed immersions $\Phi_k: X_k \to \PP^s_{S_k}$ over $S_k$. This gives rise to the map $f: X \to S$, which is of finite type and proper, and which induces $f_k: X_k \to S_k$ after base change. It follows from \cite[0523]{stacks} that $X$ is flat over $S$ at all points in $X_0$. Since the flat locus is open and $f: X \to S$ is a closed morphism, this implies that $f: X \to S$ is flat. By \cite[045U]{stacks}, the Cohen--Macaulay locus of $f: X \to S$ is open, and it contains $X_0$, so we find that $f: X \to S$ is a Cohen--Macaulay morphism, and $X$ is Cohen--Macaulay since the base $S$ is so. Using \cite[02NM]{stacks}, we can show that $f: X \to S$ is of relative dimension $d$. Since $t$ is a non-zero divisor due to flatness, the Krull intersection theorem together with the reducedness of $X_0$ shows that $X$ is reduced at points in $X_0$. Then $X$ is reduced because the support of any section is closed. The generic fiber $X_\eta$ is reduced because it is the localization of $X$ in $t$. It follows from \cite[055J]{stacks} that $X_\eta$ is connected since $X_0$ is connected.
Thus, we have shown:

\begin{lemma}\label{algebraization-lemma}\note{algebraization-lemma}
 In Situation~\ref{alg-degen-sitn}, there is a morphism $f: X \to S$ of finite type over $S = \Spec \CC\llbracket t\rrbracket$ which induces, via base change, the underlying morphism of schemes of some enhanced generically log smooth deformation $f_k: X_k \to S_k$ of type $\D$. The morphism $f: X \to S$ is separated, flat, projective, Cohen--Macaulay, of relative dimension $d$, and has reduced connected fibers. The total space $X$ is connected, reduced, and Cohen--Macaulay.
\end{lemma}

Next, we investigate smoothness, integrality, and normality of the generic fiber $X_\eta$. The following definition of \emph{formal smoothings} seems to appear first in Tziolas' article \cite[Defn.~11.6]{Tziolas2010}, although the concept is probably much older, cf.~also \cite[0E7T]{stacks}. In \cite[p.~202]{Hartshorne2010}, Hartshorne defines the closely related notion of being \emph{formally smoothable}, and on page 214, he claims that this definition is new.

\begin{defn}\label{formal-smoothing-defn}\note{formal-smoothing-defn}
 Let $S_0 = \Spec \kk$, and let $f_0: X_0 \to S_0$ be a separated scheme of finite type, reduced, connected, Cohen--Macaulay of dimension $d \geq 1$. Let $S_k = \Spec \kk[t]/(t^{k + 1})$, and let $(f_k: X_k \to S_k)_k$ be a compatible family of flat deformations of $f_0: X_0 \to S_0$. Then $(f_k)_k$ is a \emph{formal smoothing} of $f_0: X_0 \to S_0$ if there is some $m \geq 0$ such that $t^m$ is in the $d$-th Fitting ideal $\mathrm{Fitt}_d(\Omega^1_{X_m/S_m})$ of the classical K\"ahler differentials $\Omega^1_{X_m/S_m}$ of $f_m: X_m \to S_m$. In this case, the property holds for all $k \geq m$.
\end{defn}

Note that the notion of a formal smoothing is \'etale local, i.e., when we have a system of compatible surjective \'etale maps $g_k: Y_k \to X_k$, then $(X_k \to S_k)_k$ is a formal smoothing if and only if $(Y_k \to S_k)_k$ is a formal smoothing. For a formal smoothing, we have:

\begin{lemma}\label{smooth-generic-fiber}\note{smooth-generic-fiber}
 Let $S = \Spec \kk\llbracket t\rrbracket$. In the situation of Definition~\ref{formal-smoothing-defn}, let $f: X \to S$ be a flat and separated morphism of finite type inducing the formal smoothing $(f_k)_k$. Then there is an open subset $W \subseteq X$ with $W \cap X_0 = X_0$ such that $W_\eta$ is smooth.
\end{lemma}
\begin{proof}
 This is a variant of \cite[0E7T]{stacks}, cf.~also \cite[Prop.~11.8]{Tziolas2010}.
\end{proof}

The converse holds as well, i.e., if $f: X \to S$ is a separated morphism of finite type, and if $X_\eta$ is smooth, then the induced system $(f_k: X_k \to S_k)_k$ is a formal smoothing.

The point now is that saturated log smooth maps give rise to formal smoothings in codimension $1$. Let us say that a log smooth and saturated morphism $f_0: X_0 \to S_0$ is \emph{para-smooth}\footnote{From Greek $\pi\alpha\rho\acute\alpha$, meaning \emph{next to},\,\emph{nearby}.}\index{para-smooth!log smooth morphism} if the locally unique log smooth deformations constitute a formal smoothing.

\begin{lemma}\label{para-smooth-locus}\note{para-smooth-locus}
 Let $f_0: U_0 \to S_0$ be a separated log smooth and saturated morphism over the standard log point $S_0 = \Spec(\NN \to \CC)$. Then there is an open subset $\tilde U_0 \subseteq U_0$ which is para-smooth, and such that $U_0 \setminus \tilde U_0$ has codimension $\geq 2$ in $U_0$.
\end{lemma}
\begin{proof}
 A saturated and log smooth morphism $f_0: U_0 \to S_0$ has local models $A_P \to A_\NN$ for saturated injections $h: \NN \to P$. Let $\rho = h(1)$, and let $R \subseteq P$ be the face generated by $\rho$. Then we have a factorization $A_P \to A_R \to A_\NN$. The fibers of $A_R \to A_\NN$ over $t \not= 0$ are smooth because $\NN \to R$ is vertical. The fibers over $z^\rho \not= 0$ of $A_P \to A_R$ are given by $A_{P/R}$, a normal scheme. When we remove the closure of the singular locus in $A_{P/R} \times A_R^*$ from $A_P$, then the complement is a deformation to a smooth general fiber, and moreover the removed locus has the right codimension.
\end{proof}
\begin{cor}
 In Situation~\ref{alg-degen-sitn}, there is a para-smooth open subset $\tilde U_0 \subseteq U_0 \subseteq X_0$ with $\mathrm{codim}(X_0 \setminus \tilde U_0,X_0) \geq 2$. Moreover, there is an open subset $\tilde U \subseteq X$ with $\tilde U \cap X_0 = \tilde U_0$ and such that $\tilde U_\eta$ is smooth. In particular, both $X$ and $X_\eta$ are normal.
\end{cor}
\begin{proof}
 The family $(f_k: \tilde U_k \to S_k)_k$ is a formal smoothing so that we can find an open subset $\tilde U \subseteq X$ such that $\tilde U \cap X_0 = \tilde U_0$ and $\tilde U_\eta$ is smooth. Recall that $X$ is Cohen--Macaulay. To show that $X$ is normal, let $x \in X$ be such that $\mathrm{dim}(\cO_{X,x}) = 1$. We have to show that $\cO_{X,x}$ is regular. If $x \in X_0$, then $\cO_{X_0,x}$ is reduced and of dimension $0$ because $f: X \to S$ is flat, hence $\cO_{X_0,x}$ is a field, and $t \in \cO_{X,x}$ generates the maximal ideal. Using the Krull intersection theorem, we find that $\cO_{X,x}$ is integral, and then it is a discrete valuation ring, hence regular. In the case $x \in X_\eta$, let $C = \overline{\{x\}}$ be the closure in $X$. Then $C$ is irreducible of dimension $d$. Since $f: X \to S$ is closed, we find $C \cap X_0 \not= \emptyset$. Let $c \in C \cap X_0$ be a closed point. The formula in \cite[02JS]{stacks} shows that $\mathrm{dim}_c(C \cap X_0) \geq d - 1$. Thus, we have $C \cap \tilde U_0 \not=\emptyset$, and hence $x \in \tilde U$. This shows that $\cO_{X,x}$ is regular, and hence that $X$ is normal. Since $X_\eta$ is a localization of $X$, it is normal as well.
\end{proof}

\subsubsection*{Singularities of the generic fiber}

Another consequence of Lemma~\ref{smooth-generic-fiber} and its converse is that, in Situation~\ref{alg-degen-sitn}, the smoothness of $X_\eta$ depends only on the system of deformations $\D$ if $f_0: X_0 \to S_0$ is proper.

\begin{defn}
 A system of deformations $\D$ is \emph{para-smooth} if $\{V_{\alpha;k}\}_k$ is a formal smoothing of $V_{\alpha;0}$.\index{para-smooth!system of deformations}
\end{defn}

In Situation~\ref{alg-degen-sitn}, if $f_0: X_0 \to S_0$ is proper, and $\D$ is para-smooth, then $X_\eta$ is smooth.

\begin{cor}\index{standard Gross--Siebert type}
 In Situation~\ref{alg-degen-sitn}, assume additionally that $f_0: X_0 \to S_0$ is proper, and that both $f_0: X_0 \to S_0$ and all deformations in $\D$ are log toroidal of standard Gross--Siebert type. Then the generic fiber $X_\eta$ is smooth.
\end{cor}
\begin{proof}
 We use \cite[Prop.~2.2]{GrossSiebertII}.
\end{proof}

We can also strengthen the formal and analytic smoothing result \cite[Thm.~1.1]{FFR2021} to an algebraic version.

\begin{cor}
 Let $V$ be a connected projective normal crossing space, assume that $\T^1_V$ is globally generated, and that $V$ is Calabi--Yau, i.e., $\omega_V \cong \cO_V$. Then there is a flat projective morphism $f: X \to S$ with $f^{-1}(0) = V$ and a connected smooth projective generic fiber $X_\eta$. 
\end{cor}
\begin{proof}
 By \cite[Prop.~6.9]{FFR2021}, we can endow $V$ with the structure of a log toroidal family $f_0: X_0 \to S_0$ of standard Gross--Siebert type.
\end{proof}
\begin{rem}
 In view of Chapter~\ref{log-modif-sec}, this remains true when $\omega_V^{-1}$ admits a log regular section $s_0$ with respect to the log structure $f_0: X_0 \to S_0$. Indeed, the system of deformations $\D(s_0)$ remains para-smooth although the log structure changes. An example of this is \cite[Thm.~1.1]{FFR2021}, which we have worked out in a variant in Theorem~\ref{smoothing-nc-spaces}.
\end{rem}

In a sense, the singularities of the generic fiber are always determined by $\D$, not only in the para-smooth case. The key insight is the following variant of Artin's approximation.

\begin{lemma}\label{rel-Artin-approx}\note{rel-Artin-approx}
 Let $R = \CC\llbracket t\rrbracket$ and $S = \Spec R$. Let $f: X \to S$ and $f': X' \to S$ be two morphisms of finite type, and assume that we have a system of compatible isomorphisms $\varphi_k: X_k \xrightarrow{\cong} X'_k$. Let $x \in X_0 \subset X$ be a point in the central fiber, and let $x' = \varphi_0(x) \in X'_0 \subset X'$. Let $c \geq 0$. Then there is a scheme $V$ of finite type over $S$ together with a point $v \in V$ and two \'etale $S$-morphisms $g: V \to X$ and $h: V \to X'$ with the following properties: $g(v) = x$ and we have an isomorphism $\kappa(x) \cong \kappa(v)$ of residue fields; $h(v) = x'$ and we have an isomorphism $\kappa(x') \cong \kappa(v)$ of residue fields; we have $\varphi_c \circ g_c = h_c$ for the base change $g_c: V_c \to X_c$ and $h_c: V_c \to X'_c$.
\end{lemma}
\begin{proof}
 The base ring $R$ is an excellent discrete valuation ring, so we can apply \cite[Thm.~1.10]{Artin1969}. We may assume that $X$ is affine; let $B := \Gamma(X,\cO_X) = R[y_1,...,y_n]/(f_1,...,f_m)$. We set 
 $$A := \cO_{X,x}^h, \quad A' := \cO_{X',x'}^h,$$
 the Henselian local rings, i.e., the local rings in the \emph{Nisnevich} topology. Note that for $A_k := A \otimes_R R/(t^{k + 1})$ and $A'_k$ analogously, we have $A_k = \cO_{X_k,x}^h$ and $A'_k = \cO_{X'_k,x'}^h$. The isomorphisms $\varphi_k$ induce compatible isomorphisms $\phi_k: A_k \cong A'_k$. In particular, we obtain composite homomorphisms $\psi_k: B \to A \to A_k \to A'_k$. In other words, we have a formal solution of the system of equations $f_j(\hat y_1,...,\hat y_n) = 0$ with values in the $t$-completion $\hat A' = \varprojlim A'_k$ of $A'$. By \cite[Thm.~1.10]{Artin1969}, we can find $y_1,...,y_n \in A'$ which satisfy $f_j(y_1,...,y_n) = 0$ and $y_i \equiv \hat y_i \mod{t^{c + 1}}$. Thus, we have a homomorphism $\psi: B \to A'$ which coincides with $\psi_k$ after mapping to $A'_c$ (but not necessarily for $k > c$). Now we can find an \'etale map $h: U' \to X'$ from an affine $S$-scheme $U'$ together with a point $u' \in U'$ such that $h(u') = x'$ and we have an isomorphism $\kappa(x') \cong \kappa(u')$, and such that we have a map $\chi: B \to \Gamma(U',\cO_{U'}) =: B'$ which coincides with $\psi: B \to A'$ after composition with $B' \to A'$. By shrinking $U'$ in the Zariski topology, we may assume that $B' \to A'$ is injective. Since $\chi: B \to B'$ induces an isomorphism of residue fields and of completions at $x$ and $u'$, we can assume that $\chi: B \to B'$ is \'etale after further shrinking $U'$ in the Zariski topology, by \cite[Prop~17.6.3]{EGAIV-4}. It remains to show that $\varphi_c \circ g_c = h_c$, or more precisely, that $\chi_c: B_c \to B'_c$ equals the composition 
 $$B_c \xrightarrow{\varphi_{c*}} \cO_{X'_c} \xrightarrow{h_c^*} B'_c.$$ By construction, the two maps are equal after composing with $B'_c \to A'_c$. By shrinking $U'_c$ in the Zariski topology, we can assume that $B'_c \to A'_c$ is injective. Then we choose $V$ as an open subset of $U'$ which, after base change to $S_c$, gives rise to the shrunk $U'_c$.
\end{proof}
\begin{rem}
 This result is different from the similar \cite[Cor.~2.6]{Artin1969} in Artin's original article in that the latter does not show that we recover the original isomorphism over $S_c$.
\end{rem}

\begin{cor}
 In Situation~\ref{alg-degen-sitn}, assume that we have, for each open $V_\alpha$, morphisms $U_\alpha \to S$ of finite type with $U_\alpha \times_S S_k = V_\alpha$ as schemes. Then the singularities of the generic fiber $X_\eta$ of $f: X \to S$ are \'etale equivalent to the singularities of $U_{\alpha;\eta}$.
\end{cor}
\begin{proof}
 Use that $f_0: X_0 \to S_0$ is proper to show that the result holds globally on $X_\eta$.
\end{proof}

This is particularly interesting for log toroidal families of elementary Gross--Siebert type since, in this case, we have complete control over the generic fiber of the local models, as explained in Lemma~\ref{elem-simplex-gen-fib} and Corollary~\ref{elem-GS-gen-fib-rigid}. We need the following preparation.

\begin{lemma}\label{etale-map-lift}\note{etale-map-lift}
 Let $S = \Spec \CC\llbracket t\rrbracket$, and let $f: X \to S$ be a morphism of finite type between affine schemes. Let $g_0: Y_0 \to X_0$ be an \'etale morphisms of finite type between affine schemes. Then there is an \'etale morphism $g: Y \to X$ of finite type between affine schemes such that $g$ induces $g_0: Y_0 \to X_0$ over $S_0$.
\end{lemma}
\begin{proof}
 Let $B = \Gamma(X,\cO_X)$, let $B_0 = \Gamma(X_0,\cO_{X_0})$, and let $C_0 = \Gamma(Y_0,\cO_{Y_0})$. Since $g_0: Y_0 \to X_0$ is \'etale, we have $C_0 = B_0[x_1,...,x_n]/(f_1,...,f_n)$ with $\Delta := \mathrm{det}(\frac{\partial f_j}{\partial x_i})$ invertible in $C_0$. Let $f_1',...,f_n' \in B[x_1,...,x_n]$ be arbitrary lifts of $f_1,...,f_n$, let $\Delta' := \mathrm{det}(\frac{\partial f'_j}{\partial x_i})$, and set 
 $$C := \big(B[x_1,...,x_n]/(f_1',...,f_n')\big)_{\Delta'}.$$
 Then $B \to C$ is \'etale and induces $B_0 \to C_0$ after base change.
\end{proof}

\begin{cor}
 In Situation~\ref{alg-degen-sitn}, assume that $f_0: X_0 \to S_0$ is proper, and assume that both $f_0: X_0 \to S_0$ and all deformations in $\D$ are log toroidal of elementary Gross--Siebert type. Then the generic fiber $X_\eta$ has toroidal singularities which are smooth in codimension $3$, locally rigid, terminal, Gorenstein, and abelian quotient singularities.\footnote{$\QQ$-factoriality does not translate to $X_\eta$ because this property is not \'etale local.}
\end{cor}
\begin{proof}
 Let $f: X \to S$ be the algebraic degeneration. By subdividing $\V$, we may assume that $\V$ is an affine cover of $X_0$. Let $W_\alpha \subseteq X$ be an affine open subset with $W_\alpha \cap X_0 = V_\alpha$. Around a point $x \in V_\alpha$, we may find an \'etale roof $Y_0 \to V_\alpha$, $Y_0 \to L_0$ where $L_0$ is the central fiber of a local model of elementary Gross--Siebert type. Let $Y \to W_\alpha$ be an \'etale morphism obtained from $Y_0 \to W_\alpha$ via Lemma~\ref{etale-map-lift}, and let $Y' \to L$ be one obtained from $Y_0 \to L_0$. Here, $L \to S$ is the local model of elementary Gross--Siebert type. Since infinitesimal log toroidal deformations of elementary Gross--Siebert type are locally unique, we have $Y_k \cong Y_k'$ for all $k \geq 0$. Now Lemma~\ref{rel-Artin-approx} shows that $W_\alpha$ and $L$ are \'etale locally isomorphic after possibly shrinking $W_\alpha$. Then the generic fiber has the stated properties on the new $W_\alpha$. Since $f_0: X_0 \to S_0$ is proper, the new $W_\alpha$ collectively cover $X$ and thus $X_\eta$.
\end{proof}

In other words, the generic fiber $X_\eta$ has very mild singularities, and furthermore, they cannot be improved by a flat deformation of $X_\eta$.

\begin{rem}
 In the preprint \cite{Ruddat2019local}, Ruddat gives a variant of Lemma~\ref{rel-Artin-approx} in which the isomorphism $\varphi_k$ does not need to be given for all $k \geq 0$. Instead, one fixes the first family $f: X \to S$, and then there is a number $N$, depending on $f: X \to S$, such that the statement of Lemma~\ref{rel-Artin-approx} holds when $\varphi_k$ is given for $k \leq 4N + 1$, for any flat $f': X' \to S$ of finite type. In order for this statement to hold, one needs the crucial assumption that the generic fiber of the first given family $f: X \to S$ is \emph{locally rigid}, i.e., $\T^1_{X_\eta/\eta} = 0$ for the first tangent sheaf in the sense of Lichtenbaum--Schlessinger. As we have just seen, this assumption is satisfied in the case of elementary Gross--Siebert type. The relevance of Ruddat's version is that, by \cite[Thm.~B.1]{RuddatSiebert2020}, one can find, for every $c \geq 0$, an \emph{analytic} family over a small disk $\Delta$ which coincides with the analytification of $f_c: X_c \to S_c$ after base change to $S_c$. Then the general fiber of this family must have the same singularities as the general fiber of the analytification of the local model of elementary Gross--Siebert type. The advantage of this picture is that we do not need to have an ample line bundle on $X_0$. However, for algebraic degenerations arising from Situation~\ref{alg-degen-sitn}, the weaker Lemma~\ref{rel-Artin-approx} is enough.
\end{rem}

\subsubsection*{Comparing two log structures}

On the one hand side, the central fiber $f_0: U_0 \to S_0$ and every deformation $f_k: U_k \to S_k$ carries a structure of a log smooth log morphism. On the other hand side, we can also define a log structure on $f: X \to S$ from its geometry. On the base $S$, the compactifying log structure defined by $t = 0$ coincides with the log structure coming from $S = \Spec(\NN \to \CC\llbracket t\rrbracket)$. On the total space, first assume that $f_0: U_0 \to S_0$ is vertical. Then $f_0: U_0 \to S_0$ is, locally in the \'etale topology, isomorphic to $V(\sigma) \times \bAA^{d - r} \to S_0$ for $V(\sigma)$ arising from Construction~\ref{prototype-constr}. This map is the central fiber of $U_S(\sigma) \times_S \bAA^{d-r}_S \to S$, whose log structure coincides with the compactifying log structure defined by $\{t = 0\} = V(\sigma) \times \bAA^{d - r}$ as explained in Chapter~\ref{toroidal-cr-sp-sec}. Thus, one may expect that the compactifying log structure on $X$ defined by $X_0$ coincides, after base change to $S_k$, with the given log structure on $f_k: U_k \to S_k$. In the general case, when $f_0: U_0 \to S_0$ may not be vertical, we have the \emph{horizontal locus} $H_k \subseteq U_k$, constructed in \cite[Constr.~2.19]{FeltenThesis}, outside of which $f_k: U_k \to S_k$ is vertical. The horizontal locus is flat over $S_k$, and its formation commutes with (strict) base change. Since $f_k: X_k \to S_k$ is a Cohen--Macaulay morphism, we can choose a canonical compactification $\bar H_k$ of $H_k$ in $X_k$ by applying $j_*$ to the ideal sheaf of $H_k \subseteq U_k$. In good cases, but possibly not always, the formation of $\bar H_k$ commutes with base change.\footnote{We currently do not have a counterexample, but we also see no reason why this should be true in general. Note that the condition only depends on the system of deformations $\D$.} In this case, we have a formal closed subscheme $\overline{\mathfrak{H}} \subseteq \mathfrak{X}$, and hence a unique closed subscheme $\bar H \subseteq X$ which induces $\bar H_k$ over $S_k$. Then one may expect that the compactifying log structure on $X$ induced by $X_0 \cup \bar H$ coincides, after base change, with the given log structure on $f_k: U_k \to S_k$.

Unfortunately, we do not yet know if the two log structures really coincide. So far, we can only provide the following partial result.\footnote{A stronger result will appear in the published version.}

\begin{lemma}\label{log-str-comp}\note{log-str-comp}
 In Situation~\ref{alg-degen-sitn}, assume that $f_0: U_0 \to S_0$ is vertical. The inclusion $X_0 \subset X$ defines a compactifying log structure on $X$ together with a structure of log morphism to $S$. Then the restriction of this log morphism to $U_0$ coincides with the original given structure of a log smooth log morphism $f_0: U_0 \to S_0$.
\end{lemma}
\begin{rem}
 We do not yet know if this is also true for the thickenings $f_k: U_k \to S_k$, although this seems plausible. We certainly know that the two are locally isomorphic, but it is not (yet) clear if they are also globally isomorphic.
\end{rem}
\begin{proof}
  Let $W \subseteq X$ be an affine open subset such that $W_0 := W \cap X_0 \subseteq U_0$. Let $g_0: V_0 \to W_0$ be an \'etale morphism with $V_0$ affine, and assume that we have another \'etale morphism $h_0: V_0 \to V(\sigma) \times \bAA^{d - r}$. When we apply Lemma~\ref{etale-map-lift} to $g_0$, we obtain an \'etale morphism $g: V \to W$, and when we apply the lemma to $h_0$, then we obtain an \'etale morphism $h: V' \to U_S(\sigma) \times_S \bAA^{d - r}_S$. We apply Lemma~\ref{rel-Artin-approx} to $V' \to U_S(\sigma) \times_S \bAA^{d - r}_S \to S$. Since both $V_k \to S_k$ and $V'_k \to S_k$ are log smooth deformations of the affine log smooth log scheme $V_0 \to S_0$, we have compatible isomorphisms $V_k \cong V'_k$ over $S_k$ up to arbitrary order. Then we can find \'etale neighborhoods $\bar V \to V$ and $\bar V' \to V'$ of any given point $v \in V_0$ together with an isomorphism $\phi: \bar V \cong \bar V'$ over $S$ which is compatible with $V_0 = V_0'$ on the central fiber. This shows the following: We can find an \'etale map $g: Y \to X$ with $g(Y) \cap X_0 = U_0$ such that the pull-back along $g_k: Y_k \to U_k$ of the given log structure is isomorphic to the pull-back along $g_k: Y_k \to U_k$ of the log structure induced as the compactifying log structure of $X_0 \subseteq X$. In other words: The two log structures are locally isomorphic in the \'etale topology.

  On $U_0$, we have the structure of a toroidal crossing space by taking the ghost sheaf of the log structure. Thus, we obtain a section $s \in \cL\cS_{U_0}$. Locally, the ghost sheaf of the log structure coming from $X_0 \subseteq X$ is isomorphic to $\cP$. Since $\cP$ has no non-trivial automorphisms, the two ghost sheaves coincide. Thus, $X_0 \subseteq X$ defines a section of $\cL\cS_{U_0}$ as well. The two sections must be locally the same, so they are in fact globally the same. In other words, the two log structures coincide on $U_0$.
\end{proof}



\chapter{Modifications of the log structure}\label{log-modif-sec}\note{log-modif-sec}
\index{modification of a log structure}

The main unobstructedness results apply to log Calabi--Yau spaces $f_0: X_0 \to S_0$. If, however, the anti-canonical bundle $\A^d_{X_0/S_0}$ is effective, then we can, in some cases, use a section $s_0$ of it to modify the log structure on $X_0$ and obtain a log Calabi--Yau space $f_0 \circ h_0: X_0(s_0) \to S_0$, whose deformation theory is closely related to the deformation theory of $f_0: X_0 \to S_0$. We work in the generality of enhanced generically log smooth families; while this situation is considerably more subtle than the case of plain generically log smooth families, it seems that this generality is necessary for some applications. The most general setup for this chapter is the following:

\begin{sitn}\label{log-modif-sitn}\note{log-modif-sitn}
 Let $f: X \to S$ be an enhanced generically log smooth family of relative dimension $d$ with Gerstenhaber calculus $\G\C^\bullet_{X/S}$, endowed with maps $\varpi^\bullet: \G^\bullet_{X/S} \to \V^\bullet_{X/S}$ and $\varpi^\bullet: \A^\bullet_{X/S} \to \W^\bullet_{X/S}$. Recall that $\varpi^\bullet$ must be an isomorphism on $U$, and that each $\A^i_{X/S}$ and $\G^p_{X/S}$ must be flat over $S$. By definition, $f: X \to S$ is log Gorenstein, and we have $\A^d_{X/S} = \W^d_{X/S}$. Let $\cL$ be a line bundle on $X$, and let $s \in \cL$ be a global section.
\end{sitn}

\section{Log pre-regular sections and $f \circ p: L(X) \to S$}\label{log-prereg-sec}\note{log-prereg-sec}

In this section, we study log pre-regular sections $s \in \cL$, and we investigate the generically log smooth family $f \circ p: L(X) \to S$ used to define the modification $f \circ h: X(s) \to S$. We turn $f \circ p: L(X) \to S$ into an enhanced generically log smooth family.

\begin{defn}\label{prereg-defn}\note{prereg-defn}\index{section!log pre-regular}
 In Situation~\ref{log-modif-sitn}, we say that $s$ is a \emph{log pre-regular section} if:
 \begin{itemize}
  \item the induced map $s: \cO \to \cL$ is injective;
  \item the effective Cartier divisor $H = \mathrm{div}(s)$ is flat over $S$; we write $i: H \to X$ for the embedding;
  \item the intersection $Z \cap H \subseteq H$ has codimension $\geq 2$ in every fiber of $f \circ i: H \to S$.
 \end{itemize}
 In this case, we denote the underlying scheme $\underline X$ of $X$, endowed with the Deligne--Faltings log structure defined by $s^\vee: \cL^\vee \to \cO_X$, by $\underline X(s)$. We write $X(s)$ for the fiber product\footnote{Formed in the category of all log schemes. The log structure on $\underline X(s)$ is fine and saturated; thus, $\underline X(s) \to \underline X$ is a saturated morphism, and $U(s)$ is fine and saturated as well.} of $X \to \underline X \leftarrow \underline X(s)$; this space carries a log structure on $U \subseteq X$. We denote the induced map by $h: X(s) \to X$; this is a map of schemes everywhere, and a map of log schemes on $U$.
\end{defn}

The notion of a log pre-regular section $s \in \cL$ is stable under base change; namely, since $\cO_H = \mathrm{coker}(\cL^\vee \to \cO_X)$ is flat over $S$, injectivity of $\cL^\vee \to \cO_X$ is preserved under base change. Similarly, the construction of $h: X(s) \to X$ commutes with base change.
We have a second, equivalent construction of $h: X(s) \to X$, which is more useful in studying $f \circ h: X(s) \to S$. For the sake of explicit computations, let us introduce a finite affine open cover $\V = \{V_\alpha\}_\alpha$ of $X$ such that $\cL|_\alpha$ is trivial. Let $e_\alpha \in \cL|_\alpha$ be a trivializing section, and let $s = s_\alpha \cdot e_\alpha$ with $s_\alpha \in \cO_X|_\alpha$. We denote the coordinate transforms by $\gamma_{\alpha\beta} \in \cO_X^*|_{\alpha\beta}$, i.e., $e_\alpha = \gamma_{\alpha\beta}e_\beta$. We denote the dual trivializing sections by $e_\alpha^\vee \in \cL^\vee|_\alpha$; then the coordinate transform on $\cL^\vee$ is $e_\alpha^\vee = \gamma_{\alpha\beta}^{-1}e_\alpha^\vee$. We have $s^\vee(e_\alpha^\vee) = s_\alpha$ under the map $s^\vee: \cL^\vee \to \cO_X$. The function $s_\alpha$ transforms as $s_\alpha = \gamma_{\alpha\beta}^{-1}s_\beta$. Let us denote the coordinate rings by $R_\alpha := \Gamma(V_\alpha,\cO_X)$.

Let 
$$L = \mathrm{Spec}_{\cO_X}\bigoplus_{\ell \geq 0} (\cL^\vee)^{\otimes \ell}$$
be the total space of $\cL$, and let $p: L \to X$ be the projection. On $p^{-1}(U)$, we endow $L$ with the log structure pulled back from $U \subseteq X$. The coordinate ring of $L$ on $W_\alpha := p^{-1}(V_\alpha)$ is $R_\alpha[x_\alpha]$ with $x_\alpha$ corresponding to $e_\alpha^\vee \in \cL|_\alpha$. On overlaps, we have $x_\alpha = \gamma_{\alpha\beta}^{-1}x_\beta$. The projection $p: L \to X$ is given by $R_\alpha \to R_\alpha[x_\alpha]$. Let $o: X \to L$ be the map corresponding to $0 \in \cL$, which is given by $R_\alpha[x_\alpha] \to R_\alpha, \: x_\alpha \mapsto 0$. The ideal sheaf of $o: X \to L$ is $\eps_o: p^*\cL^\vee \hookrightarrow \cO_L$ with $p^*e_\alpha^\vee \mapsto x_\alpha$. This gives rise to a Deligne--Faltings log structure on the underlying space $\underline L$ of $L$, which we denote by $\underline L(X)$. As above, we can form the fiber product $L(X)$ of $L \to \underline L \leftarrow \underline L(X)$, which is a log scheme on $p^{-1}(U)$. By abuse of notation, we will write $p: L(X) \to X$ as well.

Let $s: X \to L$ be the map corresponding to $s \in \cL$; it is given by $R_\alpha[x_\alpha] \to R_\alpha, \: x_\alpha \mapsto s_\alpha$. The ideal sheaf of $s: X \to L$ is $\eps_s: p^*\cL^\vee \hookrightarrow \cO_L$ with $p^*e_\alpha^\vee \mapsto x_\alpha - s_\alpha$. Then we define a log scheme $X'(s)$ by pulling back the log structure from $L(X)$ along $s: X \to L$. This log structure is defined on $s^{-1}(p^{-1}(U)) = U$. Since $s^*\eps_o: \cL^\vee \to \cO_X$ is precisely $s^\vee: \cL^\vee \to \cO_X$, and since the pull-back of Deligne--Faltings structures is the same as the pull-back of the associated log structures, it is easy to see that this second construction coincides with the first one in Definition~\ref{prereg-defn}.

\begin{lemma}
 In the above situation, we have:
 \begin{enumerate}[label=\emph{(\alph*)}]
  \item The map $p: L \to X$ is strict and smooth of relative dimension $1$. We have $\Omega^1_{L/X} \cong p^*\cL^\vee$.
  \item The map $p: L(X) \to X$ is smooth and log smooth of relative dimension $1$. We have $\Omega^1_{L(X)/X} = \cO_L \cdot \gamma_s$, where we have $\gamma_s|_{U \cap V_\alpha} = \delta_{L(X)/X}(m_\alpha)$ in a local chart.
 \end{enumerate}
\end{lemma}
\begin{proof}
 It is clear that $p: L \to X$ is strict and smooth. The notation $\Omega^1_{L/X}$ means the classical relative differentials. They are, on $W_\alpha$, generated by $dx_\alpha$. Since we have $dx_\alpha = \gamma_{\alpha\beta}^{-1}\cdot dx_\beta$ on overlaps, we find $\Omega^1_{L/X} \cong p^*\cL^\vee$. Since $\underline p: \underline L(X) \to \underline X$ is locally a base change of $A_\NN \to \Spec \ZZ$, it is smooth and log smooth. Let us denote by $\underline m_\alpha \in \M_{\underline L(X)}(W_\alpha)$ the section induced by $x_\alpha \in \eps_0(p^*\cL^\vee)$ via the construction of the associated log structure of a Deligne--Faltings structure. Then $\Omega^1_{\underline L(X)/\underline X}$ is locally free of rank $1$ and, on $W_\alpha$, generated by $\delta(\underline m_\alpha)$, where $\delta$ denotes the log part of the universal derivation. Due to $\underline m_\alpha = \gamma_{\alpha\beta}^{-1}\underline m_\beta$, we find $\delta(\underline m_\alpha) = \delta(\underline m_\beta)$ on overlaps, so $\Omega^1_{\underline L(X)/\underline X} \cong \cO_L$. Now $p: L(X) \to X$ is a base change of $\underline p: \underline L(X) \to \underline X$, so we have $\Omega^1_{L(X)/X} \cong \Omega^1_{\underline L(X)/\underline X}$ where the former sheaf is defined, and we consider $\Omega^1_{L(X)/X}$ to be defined globally via this isomorphism. This is obviously the same as taking the direct image from $p^{-1}(U)$.
\end{proof}

The composition $f \circ p: L(X) \to S$ is clearly a generically log smooth family. We now turn to the problem of endowing it with a Gerstenhaber calculus $(\G^\bullet_{L(X)/S},\A^\bullet_{L(X)/S})$ which turns it into an enhanced generically log smooth family, compatibly with $f: X \to S$. We set $\A^0_{L(X)/S} := \cO_L$. Consider the exact sequence 
\begin{equation}\label{W1LXS-seq}
 0 \to p^*\W^1_{X/S} \to \W^1_{L(X)/S} \to \Omega^1_{L(X)/X} \to 0.
\end{equation}\note{W1LXS-seq}It is clearly an exact sequence on $p^{-1}(U)$. Since $p$ is flat, $p^*\W^1_{X/S}$ is reflexive and hence $p^{-1}(Z)$-closed. Since $\Omega^1_{L(X)/X}$ is $p^{-1}(Z)$-closed as well, we can define the sequence as the direct image from $p^{-1}(U)$. If $m_\alpha \in \Gamma(p^{-1}(U) \cap W_\alpha,\M_{L(X)})$ is the image of $\underline m_\alpha$ under $L(X) \to \underline L(X)$, then $\delta(m_\alpha) \in \Gamma(p^{-1}(U) \cap W_\alpha,\W^1_{L(X)/S})$ is a preimage of the generator of $\Omega^1_{L(X)/X}$, so the map is surjective. Since $\Omega^1_{L(X)/X}$ is free, the exact sequence is locally split. Any two splittings differ by a map in $\cH om(\Omega^1_{L(X)/X},p^*\W^1_{X/S})$. There is, however, a set of distinguished splittings such that any two distinguished splittings differ by a map in $\cH om(\Omega^1_{L(X)/X},p^*\A^1_{X/S})$. One distinguished splitting is given by 
$$\gamma_s = \delta_{L(X)/X}(m_\alpha) \mapsto \delta_\alpha := \delta_{L(X)/S}(m_\alpha).$$
On overlaps, we have 
$$\delta_{L(X)/S}(\gamma_{\alpha\beta}^{-1}\cdot m_\beta) = \gamma_{\alpha\beta} \cdot d_{L(X)/S}(\gamma_{\alpha\beta}^{-1}) + \delta_{L(X)/S}(m_\beta).$$
Because of $d_{L(X)/S}(\gamma_{\alpha\beta}^{-1}) \in p^*\A^1_{X/S}$, the difference is given by a map to $p^*\A^1_{X/S} \subseteq p^*\W^1_{X/S}$. Thus, the class of distinguished splittings is globally consistent. If we modify the trivializing section $e_\alpha$ by a unit $u_\alpha$, then a similar computation shows that the class of distinguished splittings is independent of the choice of $e_\alpha$. If $B_1: \Omega^1_{L(X)/X} \to \W^1_{L(X)/S}$ is some distinguished splitting, then we set 
$$\A^1_{L(X)/S} := p^*\A^1_{X/S} \oplus B_1(\Omega^1_{L(X)/X}) \subseteq \W^1_{L(X)/S}.$$
This is independent of the choice of $B_1$, and thus globally well-defined; it fits into a split exact sequence 
$$0 \to p^*\A^1_{X/S} \to \A^1_{L(X)/S} \to \Omega^1_{L(X)/X} \to 0.$$
For $i \geq 2$, we set 
$$\A^i_{L(X)/S} := p^*\A^i_{X/S} + p^*\A^{i - 1}_{X/S} \wedge B_1(\Omega^1_{L(X)/X}) \subseteq \W^i_{L(X)/X},$$
where, again, $B_1: \Omega^1_{L(X)/X} \to \W^1_{L(X)/S}$ is some local distinguished splitting. This is independent of the choice of $B_1$ and hence globally well-defined. We have a diagram
\[
 \xymatrix{
  0 \ar[r] & p^*\W^i_{X/S} \ar[r]^-{E_i} & \W^i_{L(X)/S} \ar[r]^-{Q_i} & p^*\W^{i - 1}_{X/S} \otimes \Omega^1_{L(X)/X} \ar[r] & 0 \\
  0 \ar[r] & p^*\A^i_{X/S} \ar[r]^-{E_i} \ar@{^{(}->}[u] & \A^i_{L(X)/S} \ar[r]^-{Q_i} \ar@{^(->}[u] & p^*\A^{i - 1}_{X/S} \otimes \Omega^1_{L(X)/X} \ar[r] \ar@{^(->}[u] & 0. \\
 }
\]
First, we obtain both rows as exact sequences on $p^{-1}(U)$ via the canonical construction. The upper row is then obtained as the direct image from $p^{-1}(U)$; it is exact on the right because we can construct explicit preimages via a splitting $B_1$ from above. It follows from the definition of $\A^i_{L(X)/S}$ that the lower row is well-defined as a sequence of subsheaves of the upper row, and it is easy to see that it is exact. A distinguished splitting $B_1$ as above gives rise to a simultaneous (local) splitting of both rows. We denote the induced splitting on the left hand side by 
$$T_i: \W^i_{L(X)/S} \to p^*\W^i_{X/S};$$
it restricts to a map $T_i: \A^i_{L(X)/S} \to p^*\A^i_{X/S}$.

\begin{lemma}
 We have $T_i(\alpha) \wedge T_j(\beta) = T_{i + j}(\alpha \wedge \beta)$.
\end{lemma}
\begin{proof}
 Let $\delta := B_1(\gamma_s)$, depending on our choice of $B_1$. Then, locally, every element $\alpha \in\A^i_{L(X)/S}$ can be written uniquely as $\alpha = E_i(\alpha_1) + E_{i - 1}(\alpha_2) \wedge \delta$ with $\alpha_1 \in p^*\A^i_{X/S}$ and $\alpha_2 \in p^*\A^{i - 1}_{X/S}$. Then $T_i(E_i(\alpha_1) + E_{i - 1}(\alpha_2) \wedge \delta) = \alpha_1$. We have 
\begin{align}
 &(E_i(\alpha_1) + E_{i - 1}(\alpha_2) \wedge \delta) \wedge  (E_j(\beta_1) + E_{j - 1}(\beta_2) \wedge \delta) \nonumber \\ 
 =\ &E_{i + j}(\alpha_1 \wedge \beta_1) + E_{i + j - 1}(\alpha_1 \wedge \beta_2 + \alpha_2 \wedge \beta_1) \wedge \delta, \nonumber
\end{align}
and hence $T_i(\alpha) \wedge T_j(\beta) = T_{i + j}(\alpha \wedge \beta)$.
\end{proof}

The maps $E_i$ and $Q_i$ are independent of $B_1$. To capture the dependence of $T_i$ from $B_1$, let $\delta = B_1(\gamma_s)$ and $\delta' = B_1'(\gamma_s)$. Then there is a unique $\eps \in p^*\A^1_{X/S}$ with $\delta' = \delta + E_1(\eps)$. Now a direct computation yields $T_i'(\alpha) = T_i(\alpha) - Q_i(\alpha) \wedge \eps$, where we consider $Q_i(\alpha) \in p^*\W^{i - 1}_{X/S}$ by identifying the generator $\gamma_s$ of $\Omega^1_{L(X)/X}$ with $1$.

We denote the reflexive Gerstenhaber algebra of polyvector fields on $f \circ p: L(X) \to S$ by $\V^\bullet_{L(X)/S}$. By dualizing the split exact sequence above, we get a split exact sequence 
\begin{equation}\label{Gamma-LXS-seq}
 0 \to p^*\V^{p + 1}_{X/S} \otimes \Theta_{L(X)/X}^{1} \xrightarrow{I_p} \V^p_{L(X)/S} \xrightarrow{F_p} p^*\V^p_{X/S} \to 0
\end{equation}
\note{Gamma-LXS-seq}of reflexive sheaves. In the middle and on the right, we identify these sheaves with the duals via the contraction map.  We have $\Theta_{L(X)/X}^{1} = \cO_L \cdot \gamma_s^\vee$ for the formal dual of $\gamma_s$, and we identify on the left via 
$$p^*\V^{p + 1}_{X/S} \otimes \Theta^{1}_{L(X)/X} \to \cH om(p^*\W^{-p + 1}_{X/S} \otimes \Omega^1_{L(X)/X},\cO_L), \enspace \theta \otimes \gamma_s^\vee \mapsto (\alpha \otimes \gamma_s \mapsto (-1)^{p + 1} \theta \ \invneg \ \alpha).$$
We define an element $\delta^\vee \in \V^{-1}_{L(X)/S}$, a priori depending on the choice of $B_1$, by setting $\delta^\vee \ \invneg \ \delta = 1$ and $\delta^\vee \ \invneg \ \ E_1(\alpha) = 0$. However, we also have $\delta^\vee \ \invneg \ (\delta + E_1(\eps)) = 1$, so $\delta^\vee$ is actually independent of the choice of $B_1$. Dualizing the splitting $T_{-p}$ of $E_{-p}$ gives a splitting 
$$S_p: p^*\V^p_{X/S} \to \V^p_{L(X)/S}$$
of $F_p$. 

\begin{lemma}\label{formulae-LXS}\note{formulae-LXS}
 The following formulae hold for any distinguished splitting $B_1$:
 \begin{enumerate}[label=\emph{(\alph*)}]
  \item\label{deltavee} $\delta^\vee \ \invneg \ E_i(\alpha) = 0$ for $\alpha \in p^*\W^i_{X/S}$;
  \item\label{delta} $S_p(\theta) \vdash \delta = 0$ for $\theta \in p^*\V^p_{X/S}$;
  \item\label{EFdual} $\theta \ \invneg \ E_i(\alpha) = F_{-i}(\theta) \ \invneg \ \alpha$ for $\theta \in \V^{-i}_{L(X)/S}$, $\alpha \in p^*\W^i_{X/S}$;
  \item\label{Ip-comp} $I_p(F_{p + 1}(\theta) \otimes \gamma_s^\vee) = \theta \wedge \delta^\vee$ for $\theta \in \V^{p + 1}_{L(X)/S}$;
  \item\label{Fmult} $F_p(\theta) \wedge F_q(\xi) = F_{p + q}(\theta \wedge \xi)$ for $\theta \in \V^p_{L(X)/S}$, $\xi \in \V^q_{L(X)/S}$;
  \item\label{EF-adjunction} $\theta \ \invneg\ E_i(\alpha) = E_{p + i}(F_p(\theta) \ \invneg \ \alpha)$ for $\theta \in \V^p_{L(X)/S}$, $\alpha \in p^*\W^i_{X/S}$;
  \item\label{FE-dual-adjunction} $F_{p + i}(\theta \vdash E_i(\alpha)) = F_p(\theta) \vdash \alpha$ for $\theta \in \V^p_{L(X)/S}$, $\alpha \in p^*\W^i_{X/S}$;
  \item\label{ST-dual} $\theta \ \invneg \ T_i(\alpha) = S_{-i}(\theta) \ \invneg \ \alpha$ for $\theta \in p^*\V^{-i}_{X/S}$, $\alpha \in \W^i_{L(X)/S}$;
  \item\label{Smult} $S_p(\theta) \wedge S_q(\xi) = S_{p + q}(\theta \wedge \xi)$ for $\theta \in p^*\V^{p}_{X/S}$, $\xi \in p^*\V^q_{X/S}$;
  \item\label{ST-adjunction} $\theta \ \invneg \ T_i(\alpha) = T_{p + i}(S_p(\theta) \ \invneg \ \alpha)$ for $\theta \in p^*\V^p_{X/S}$, $\alpha \in \W^i_{L(X)/S}$;
  \item\label{SE-comp} $S_p(\theta) \ \invneg \ E_i(\alpha) = E_{p + i}(\theta \ \invneg \ \alpha)$ for $\theta \in p^*\V^{p}_{X/S}$, $\alpha \in p^*\W^i_{X/S}$;
  \item\label{SEd-comp} $S_p(\theta) \ \invneg \ (E_{i - 1}(\beta) \wedge \delta) = E_{p + i - 1}(\theta \ \invneg \ \beta) \wedge \delta$ for $\theta \in p^*\V^{p}_{X/S}$, $\beta \in p^*\W^{i - 1}_{X/S}$;
  \item $S_p(\theta) \vdash E_i(\alpha) = S_{p + i}(\theta \vdash \alpha)$ for $\theta \in p^*\V^p_{X/S}$, $\alpha \in p^*\W^i_{X/S}$;
  \item $(S_{p + 1}(\xi) \wedge \delta^\vee) \vdash E_i(\alpha) = (-1)^i S_{p + 1}(\xi \vdash \alpha) \wedge \delta^\vee$ for $\xi \in p^*\V^{p + 1}_{X/S}$, $\alpha \in p^*\W^i_{X/S}$.
 \end{enumerate}
\end{lemma}
\begin{proof}
 We start with \emph{\ref{deltavee}}. If $i = 0$, then we have $\delta^\vee \ \invneg \ E_i(\alpha) = 0$ for degree reasons. If $i = 1$, then we have $\delta^\vee \ \invneg \ E_i(\alpha) = 0$ by definition. Since $|\delta^\vee| = -1$, we have a formula for computing $\delta^\vee \ \invneg \ E_{i + j}(\alpha \wedge \beta)$ in terms of $\delta^\vee \ \invneg \ E_i(\alpha)$ and $\delta^\vee \ \invneg \ E_j(\beta)$, which yields $\delta^\vee \ \invneg \ E_i(\alpha) = 0$ for every $\alpha$ which is locally a $\wedge$-product of elements in degree $1$. This holds everywhere on $p^{-1}(U)$; since $p^*\W^i_{X/S}$ has injective restrictions to $p^{-1}(U)$, the claim follows. The proof of \emph{\ref{delta}} is similar, starting with $T_1(\delta) = 0$ and using \emph{\ref{ST-dual}}.
 
 The formula in \emph{\ref{EFdual}} holds by the definition of $F_{-i}$ as the dual of $E_i$.
 
 To see \emph{\ref{Ip-comp}}, contract both sides with $\alpha = E_{-p}(\alpha_1) + E_{-p + 1}(\alpha_2) \wedge \delta$, using the formula of \emph{\ref{EFdual}} and our identification above.
 
 We show \emph{\ref{Fmult}} at a stalk at $x \in p^{-1}(U)$. On the one hand side, we have 
 $$F_{-1}(\theta_r) \wedge ... \wedge F_{-1}(\theta_1) \ \invneg \ (\alpha_1 \wedge ... \wedge \alpha_r) = \sum_{\sigma \in S_r} (-1)^{\mathrm{sgn}(\sigma)} \prod_{j = 1}^r F_{-1}(\theta_j) \ \invneg\ \alpha_{\sigma(j)};$$
 on the other hand side, we have
 \begin{align}
  F_{-r}(\theta_r \wedge ... \wedge \theta_1) \ \invneg \ (\alpha_1 \wedge ... \wedge \alpha_r) &= (\theta_r \wedge ... \wedge \theta_1) \ \invneg \ E_1(\alpha_1) \wedge ... \wedge E_1(\alpha_r) \nonumber \\
  &= \sum_{\sigma \in S_r} (-1)^{\mathrm{sgn}(\sigma)} \prod_{j = 1}^r \theta_j \ \invneg\ E_1(\alpha_{\sigma(j)}). \nonumber
 \end{align}
 Since $p^*\W^r_{X/S}$ is locally free at $x \in p^{-1}(U)$, this implies $F_{-1}(\theta_1) \wedge ... \wedge F_{-1}(\theta_r) = F_{-r}(\theta_1 \wedge ... \wedge \theta_r)$, and since $p^*\V^{-1}_{X/S}$ is locally free at $x \in p^{-1}(U)$, the claim follows.
 
 It suffices to show the formula in \emph{\ref{EF-adjunction}} on $p^{-1}(U)$. The formula is obvious for $p = 0$; for $p = -1$ and $i = 0$, it is trivial, and for $p = -1$ and $i = 1$, it is the very definition of the map $F_{-1}$. We obtain the formula for $p = -1$ and arbitrary $i$ via the product formula for right contractions. Using \emph{\ref{Fmult}}, we generalize to $p \leq -2$. The proof of \emph{\ref{FE-dual-adjunction}} is analogous.
 
 The formula in \emph{\ref{ST-dual}} is nothing but the definition of $S_{-i}$ as the dual of $T_i$. Then \emph{\ref{Smult}} is analogous to \emph{\ref{Fmult}}, and \emph{\ref{ST-adjunction}} is analogous to \emph{\ref{EF-adjunction}}.
 
 The remaining four formulae are now easy once we use $T_1(\delta) = 0$.
\end{proof}

We set 
$$\G^{p}_{L(X)/S} :=\, I_p(p^*\G^{p + 1}_{X/S} \otimes \V^{-1}_{L(X)/X}) + S_p(p^*\G^p_{X/S}) \, \subseteq \V^p_{L(X)/S}.$$
A careful computation shows that, under a change of distinguished splitting $B_1$, we have $S'_p(\theta) = S_p(\theta) - I_p((\theta \vdash \eps) \otimes \gamma_s^\vee)$, where again $\eps \in p^*\A^1_{X/S}$ is such that $E_1(\eps) = B'_1(\gamma_s) - B_1(\gamma_s)$. Since $\theta \vdash \eps \in p^*\G^{p + 1}_{X/S}$ if $\theta \in p^*\G^p_{X/S}$, we find that $\G^p_{L(X)/S}$ is independent of the choice of $B_1$, and hence globally well-defined.

Every element $\theta \in \G^p_{L(X)/S}$ can be uniquely written as 
$$\theta = S_{p + 1}(\theta_1) \wedge \delta^\vee + S_p(\theta_2)$$
with $\theta_1 \in p^*\G^{p + 1}_{X/S}$ and $\theta_2 \in p^*\G^p_{X/S}$ once a local distinguished splitting $B_1$ is chosen.

\begin{prop}\label{modif-is-enhanced}\note{modif-is-enhanced}
 The family $f \circ p: L(X) \to S$, endowed with 
 $$\G\C^\bullet_{L(X)/S} = (\G^\bullet_{L(X)/S},\A^\bullet_{L(X)/S}),$$
 is an enhanced generically log smooth family and log Gorenstein.
\end{prop}
\begin{proof}
 By construction, we have $\A^0_{L(X)/S} = \W^0_{L(X)/S} = \cO_L$. By assumption, we have $\A^d_{X/S} = \W^d_{X/S}$, and this is a line bundle. Since $\W^{d + 1}_{L(X)/S} = p^*\W^d_{X/S} \otimes \Omega^1_{L(X)/X}$ as well as $\A^{d + 1}_{L(X)/S} = p^*\A^d_{L(X)/S} \otimes \Omega^1_{L(X)/S}$, we find $\A^{d + 1}_{L(X)/S} = \W^{d + 1}_{L(X)/S}$, and this is a line bundle.
 
 We show that $\A^\bullet_{L(X)/S} \subseteq \W^\bullet_{L(X)/S}$ is a differential graded subalgebra. Since $E_i: p^*\W^i_{X/S} \to \W^i_{L(X)/S}$ is compatible with the $\wedge$-product, we see that $\A^\bullet_{L(X)/S} \subseteq \W^\bullet_{L(X)/S}$ is closed under the $\wedge$-product. On $W_\alpha$, we have 
$$d_{L(X)/S}(x_\alpha) = x_\alpha \cdot \delta_{L(X)/S}(m_\alpha) \in B_1(\Omega^1_{L(X)/X})$$
for the original distinguished splitting $B_1$, which we have constructed first. Thus, we have $\partial(\cO_L) \subseteq \A^1_{L(X)/S} \subseteq \W^1_{L(X)/S}$; another direct computation yields that $\partial(\A^i_{L(X)/S}) \subseteq \A^{i + 1}_{L(X)/S}$, so $\A^\bullet_{L(X)/S}$ is a complex.

$\G^\bullet_{L(X)/S}$ is closed under the $\wedge$-product. Namely, let $\theta = S_{p + 1}(\theta_1) \wedge \delta^\vee + S_p(\theta_2)$ and $\xi = S_{q + 1}(\xi_1) \wedge \delta^\vee + S_q(\xi_2)$; then we have 
 $$(S_{p + 1}(\theta_1) \wedge \delta^\vee + S_p(\theta_2)) \wedge (S_{q + 1}(\xi_1) \wedge S_q(\theta_2)) = S_{p + q + 1}(\theta_1 \wedge \xi_2 + \xi_1 \wedge \theta_2) \wedge \delta^\vee + S_{p + q}(\theta_2 \wedge \xi_2)$$
 so that $\theta \wedge \xi \in \G^{p + q}_{L(X)/S}$. Similarly, we have 
 \begin{align}
  (S_{p + 1}(\theta_1) &\wedge \delta^\vee + S_p(\theta_2)) \ \invneg \ (E_i(\alpha_1) + E_{i - 1}(\alpha_2) \wedge \delta) \nonumber \\
  =\ &E_{p + i}((\theta_2 \ \invneg \ \alpha_1) + (-1)^{i - 1}(\theta_1 \ \invneg \ \alpha_2)) + E_{p + i - 1}(\theta_2 \ \invneg \ \alpha_2) \wedge \delta, \label{rcont-formula}
 \end{align}
 showing that $\theta \ \invneg \ \alpha \in \A^{p + i}_{L(X)/S}$ for $\theta \in \G^p_{L(X)/S}$ and $\alpha \in \A^i_{L(X)/S}$. By the Lie--Rinehart homotopy formula, $\G\C^\bullet_{L(X)/S}$ is closed under the Lie derivative $\cL_{-}(-)$. From \eqref{rcont-formula}, we also see that $\G\C^\bullet_{L(X)/S}$ is locally Batalin--Vilkovisky. Namely, if $\omega \in \A^d_{X/S}$ is a local volume form on $X$, then $E_d(p^*\omega) \wedge \delta \in \A^{d + 1}_{L(X)/S}$ is a local volume form on $L(X)$. Contracting with this form yields 
 $$(S_{p + 1}(\theta_1) \wedge \delta^\vee + S_p(\theta_2)) \ \invneg \ E_d(p^*\omega) \wedge \delta = E_{p + d + 1}((-1)^d\theta_1 \ \invneg \ p^*\omega) + E_{p + d}(\theta_2 \ \invneg \ p^*\omega) \wedge \delta,$$
 so this induces an isomorphism $\G^p_{L(X)/S} \cong \A^{p + d + 1}_{L(X)/S}$. Now the mixed Leibniz rule yields that $\G^\bullet_{L(X)/S}$ is closed under $[-,-]$. For, we have $[\theta,\xi] \in \G^{p + q + 1}_{L(X)/S}$ if and only if
 $$[\theta,\xi] \ \invneg \ (E_d(p^*\omega) \wedge \delta) \in \A^{p + q + d + 2}_{L(X)/S},$$
 and the latter is the case by evaluating the term with the mixed Leibniz rule, and using closedness under $\cL_{-}(-)$.
 We have 
 \begin{align}
  (S_{p + 1}&(\theta_1) \wedge \delta^\vee + S_p(\theta_2)) \vdash (E_i(\alpha_1) + E_{i - 1}(\alpha_2) \wedge \delta) \nonumber \\
  &= (-1)^i S_{p + i + 1}(\theta_1 \vdash \alpha_1) \wedge \delta^\vee + S_{p + i}(\theta_2 \vdash \alpha_1 + (-1)^{i - 1} \theta_1 \vdash \alpha_2), \nonumber
 \end{align}
 so $\G\C^\bullet_{L(X)/S}$ is closed under the left contraction $\vdash$.
 
 Each term $\G^p_{L(X)/S}$ and $\A^i_{L(X)/S}$ is flat over $S$ because of the locally split exact sequences these terms fit in. On $p^{-1}(U)$, the Gerstenhaber calculus $\G\C^\bullet_{L(X)/S}$ coincides with the reflexive Gerstenhaber calculus by construction. Finally, $\G^p_{L(X)/S}$ and $\A^i_{L(X)/S}$ are $p^{-1}(Z)$-pure in every fiber since they decompose locally as a direct sum of terms of the form $p^*\G^p_{X/S}$ and $p^*\A^i_{X/S}$, and they are $p^{-1}(Z)$-pure in every fiber by assumption.
\end{proof}

\section{Log quasi-regular and log regular sections}

In this section, we introduce the notions of a log quasi-regular section $s \in \cL$ and a log regular section $s \in \cL$. The latter will be the correct notion to study modifications $f \circ h: X(s) \to S$ in a deformation-theoretic setting. If the enhanced generically log smooth family is a plain generically log smooth family, i.e., its Gerstenhaber calculus is reflexive, then the two notions coincide.\footnote{We do not need the base change property for this to hold.}

Let $s \in \cL$ be a log pre-regular section, and let 
$$X(s) \xrightarrow{e} N(s) \xrightarrow{k} L(X)$$
be the first infinitesimal neighborhood, i.e., it is the closed subscheme of $L(X)$ defined by the square of the ideal sheaf $\eps_s: p^*\cL^\vee \to \cO_L, \, p^*e_\alpha^\vee \mapsto x_\alpha - s_\alpha,$ of $X(s) \subset L(X)$, and on $N(s) \cap p^{-1}(U)$, it is endowed with the inverse image of the log structure on $L(X)$. As topological spaces, we have $|X(s)| = |N(s)|$ and $|N(s) \cap p^{-1}(U)| = |U|$. We have an exact sequence 
$$0 \to \cL^\vee \xrightarrow{e_\alpha^\vee \mapsto x_\alpha - s_\alpha} \cO_{N(s)} \to \cO_{X(s)} \to 0,$$
which exhibits $\cL^\vee$ as the ideal sheaf of $X(s) \subset N(s)$. Let us write $\eps_N: \cL^\vee \to \cO_{N(s)}$ for the corresponding map.
A \emph{splitting} of $e: X(s) \to N(s)$ is a morphism $r: N(s) \to X(s)$ over $S$ with $r \circ e = \mathrm{id}_{X(s)}$; more precisely, $r: N(s) \to X(s)$ is supposed to be a morphism of schemes everywhere, and a morphism of log schemes on $U$. Local splittings form a sheaf on the topological space $X$, which we denote by 
$$\mathrm{Sp}_{X(s)/N(s)}.$$
Explicitly, a splitting consists of two maps $r^*: \cO_{X(s)} \to \cO_{N(s)}$ and $r^*: \M_{X(s)}|_U \to \M_{N(s)}|_U$. There is an action of $\V^{-1}_{X(s)/S} \otimes \cL^\vee = \D er_{X(s)/S}(\cL^\vee)$ on $\mathrm{Sp}_{X(s)/N(s)}$, which is, when $\theta = (D,\Delta)$ is a log derivation with values in $\cL^\vee$, explicitly given by
$$(\theta \odot r)^*(a) = r^*(a) + \eps_N \circ D(a), \quad (\theta \odot r)^*(m) = (1 + \eps_N \circ \Delta(m)) \cdot r^*(m).$$
Here, we consider $\V^\bullet_{X(s)/S} := j_*\Theta^{-\bullet}_{U(s)/S}$ as the complex of \emph{$Z$-closed} polyvector fields, as $U(s) \to S$ may not be log smooth. First, this action is given on $U$, but since both $
\V^{-1}_{X(s)/S} \otimes \cL^\vee$ and $\mathrm{Sp}_{X(s)/N(s)}$ are $Z$-closed\footnote{If we have a splitting on $W \cap U$ for some open subset $W \subseteq X$, then we can extend the classical part of the morphism to $W$ since we have $\cO_{X(s)}|_W = j_*\cO_{X(s)}|_{W \cap U}$ and $\cO_{N(s)}|_W = j_*\cO_{N(s)}|_{W \cap U}$; the latter holds because $\cO_{N(s)}$ is an extension of $\cO_{X(s)}$ by $\cL^\vee$.}, we obtain the action on $X$ as the direct image from $U$. By \cite[IV, Thm.~2.2.2]{LoAG2018}, $\mathrm{Sp}_{X(s)/N(s)}|_U$ is a pseudo-torsor over $\V^{-1}_{U(s)/S} \otimes \cL^\vee$. Then $\mathrm{Sp}_{X(s)/N(s)}$ is automatically a pseudo-torsor over $\V^{-1}_{X(s)/S} \otimes \cL^\vee$.

\begin{defn}\label{log-qreg-defn}\note{log-qreg-defn}\index{section!log quasi-regular}
 In Situation~\ref{log-modif-sitn}, let $s \in \cL$ be a log pre-regular section. Then $s$ is a \emph{log quasi-regular section} if $\mathrm{Sp}_{X(s)/N(s)}$ is a torsor over $\V^{-1}_{X(s)/S} \otimes \cL^\vee$, i.e., if local splittings exist everywhere.
\end{defn}

It is clear that every base change of a log quasi-regular section is again a log quasi-regular section. The proof of \cite[IV, Thm.~3.2.2]{LoAG2018} shows that $U(s) \to S$ is log smooth if and only if $s|_U$ is a log quasi-regular section. In particular, if $s \in \cL$ is a log quasi-regular section, then $f \circ h: X(s) \to S$ is a generically log smooth family. We have the following criterion for log quasi-regularity: 

\begin{lemma}\label{log-qreg-crit}\note{log-qreg-crit}
 Let $s \in \cL$ be a log pre-regular section. Then $s^*
 \W^1_{L(X)/S}$ is reflexive, and the canonical exact sequence of the embedding $X(s) \subset L(X)$ extends to a (not necessarily exact) sequence 
 $$0 \to \cL^\vee \xrightarrow{G_1} s^*\W^1_{L(X)/S} \xrightarrow{R_1} \W^1_{X(s)/S} \to 0$$
 with $\W^1_{X(s)/S} := j_*\Omega^1_{U(s)/S}$, even if $U(s) \to S$ is not log smooth.
 The section $s \in \cL$ is log quasi-regular if and only if this sequence is exact and locally split.
\end{lemma}
\begin{proof}
 We consider the diagram
\[
 \xymatrix{
  0 \ar[r] & \cL^\vee \ar[r]^-{\bar d_{X(s)}} & s^*\W^1_{L(X)/S} \ar[r]^-{ds^*} & \W^1_{X(s)/S} \ar@{.>}@/^1pc/[d]^{dr^*} \ar[r] & 0 \\
  & k^*\I_{N(s)} \ar[u]^0 \ar[r]^-{\bar d_{N(s)}} & k^*\W^1_{L(X)/S} \ar[u]^\pi \ar[r] & \W^1_{N(s)/S} \ar@{-->}[lu]^\phi \ar[u]^{de^*} \ar[r] & 0 \\
 }
\]
of sheaves on the topological space $X$. On $U$, the upper row is the canonical exact sequence of the strict closed immersion $U(s) \subset L(X)|_{p^{-1}(U)}$. It is exact on $U$ in the middle and on the right by \cite[IV, Prop.~2.3.2]{LoAG2018}. From the locally split exact sequence \eqref{W1LXS-seq}, it follows that $s^*\W^1_{L(X)/S}$ is reflexive. Hence we can define the upper row on $X$ as its direct image from $U$. The sheaf $\I_{N(s)}$ is the ideal sheaf of $N(s) \subset L(X)$; on $U$, the lower row is the canonical exact sequence of this embedding. It is also possible to show that $k^*p^*\W^1_{X/S}$ is $Z$-closed, hence $k^*\W^1_{L(X)/S}$ is $Z$-closed, and we can define the lower row on $X$ as the direct image from $U$. A direct computation shows that $\pi \circ \bar d_{N(s)} = 0$; hence, on $U$, $\pi$ descends to a map $\phi: \W^1_{N(s)/S} \to s^*\W^1_{L(X)/S}$. Since both source and target are $Z$-closed, this map can be extended to be defined on $X$.\footnote{Note that we have not shown surjectivity of the lower row on the right.}

First assume that $s \in \cL$ is log quasi-regular. Then, on $U$, the upper row is exact and locally split by \cite[IV, Prop.~2.3.2]{LoAG2018}. In particular, on $X$, it is exact on the left and in the middle. For a local splitting $r \in \mathrm{Sp}_{X(s)/N(s)}$, we have $ds^* \circ \phi \circ dr^* = \mathrm{id}$, first on $U$, and then on $X$. Thus, $ds^*$ is surjective on $X$, and the upper row is locally split exact everywhere. 

Conversely, assume that the sequence is exact and locally split. Let $\mathrm{Sp}'$ be the sheaf of splittings of this exact sequence. Under the assumption, this is a $\cH om(\W^1_{X(s)/S},\cL^\vee)$-torsor. The above construction yields a map 
$$\eta: \mathrm{Sp}_{X(s)/N(s)} \to \mathrm{Sp}'$$
of sheaves. A careful computation (using that elements of the form $\delta(m)$ generate $\Omega^1_{U(s)/S}$) shows that $\eta$ is equivariant for the $
\V^{-1}_{X(s)/S} \otimes \cL^\vee$-pseudo-torsor structure on the left and the $\cH om(\W^1_{X(s)/S},\cL^\vee)$-torsor structure on the right, with the obvious identification. By \cite[IV, Thm.~3.2.2]{LoAG2018}, $\mathrm{Sp}_{X(s)/N(s)}$ is a torsor on $U$. Thus, $\eta$ is an isomorphism on $U$. If $\rho \in \mathrm{Sp}'$ is a local section on an open subset $W \subseteq X$, then $\eta^{-1}(\rho|_{W \cap U})$ can be canonically extended to $W$; thus, $s \in \cL$ is a log quasi-regular section.
\end{proof}

Let $s \in \cL$ be a log quasi-regular section. From \eqref{Gamma-LXS-seq}, we find that $s^*\V^{-1}_{X/S}$ is locally free on $U$ and reflexive; thus $s^*\V^{-1}_{L(X)/S} \cong \cH om(s^*\W^1_{L(X)/S},\cO_X)$, and we obtain a locally split exact sequence 
$$0 \to \V^{-1}_{X(s)/S} \xrightarrow{J_{-1}} s^*\V^{-1}_{L(X)/S} \xrightarrow{H_{-1}} \cL \to 0$$
by dualizing the exact sequence of Lemma~\ref{log-qreg-crit}. The map $s^*\G^{-1}_{L(X)/S} \to s^*\V^{-1}_{L(X)/S}$ is injective, and after setting 
$$\G^{-1}_{X(s)/S} := (J_{-1})^{-1}(s^*\G^{-1}_{L(X)/S}) \subseteq \V^{-1}_{X(s)/S},$$
we have an exact sequence 
$$0 \to \G^{-1}_{X(s)/S} \to s^*\G^{-1}_{L(X)/S} \to \cL.$$
It is natural to ask for this map to be surjective, and, as it turns out, this is a very useful condition in many respects.

\begin{defn}\label{log-reg-defn}\note{log-reg-defn}\index{section!log regular}
 In Situation~\ref{log-modif-sitn}, let $s \in \cL$ be a log quasi-regular section. Then $s$ is a \emph{log regular section} if the canonical map $s^*\G^{-1}_{L(X)/S} \subseteq s^*\V^{-1}_{L(X)/S} \to \cL$ is surjective.
\end{defn}

Since the formation of the map $s^*\G^{-1}_{L(X)/S} \to \cL$ commutes with base change, the condition of being a log regular section is stable under base change. If $s \in \cL$ is a log regular section, then the exact sequence is locally split because $\cL$ is a line bundle.

If $r \in \mathrm{Sp}_{X(s)/N(s)}$ is a splitting, let us denote the induced splitting of $R_1: s^*\W^1_{L(X)/S} \to \W^1_{X(s)/S}$ by $C_1[r]$, and the induced splitting of $G_1: \cL^\vee \to s^*\W^1_{L(X)/S}$ by $U_1[r]$. With this notation, we construct a map 
$$\psi: \mathrm{Sp}_{X(s)/N(s)} \to \cH om(\W^1_{X/S},\cL^\vee)$$
as 
$$\psi(r) = U_1[r] \circ E_1: \enspace \W^1_{X/S} \xrightarrow{E_1} s^*\W^1_{L(X)/S} \xrightarrow{U_1[r]} \cL^\vee.$$
For $r' = \theta \odot r$ with $\theta \in \V^{-1}_{X(s)/S} \otimes \cL^\vee = \cH om(\W^1_{X(s)/S},\cL^\vee)$, we find 
$$U_1[r'](\alpha) - U_1[r](\alpha) = - \theta(R_1(\alpha))$$
for $\alpha \in s^*\W^1_{L(X)/S}$. Thus, 
$$\psi(r')(\beta) - \psi(r)(\beta) = -\theta(R_1 \circ E_1(\beta))$$
for $\beta \in \W^1_{X/S}$, so $\psi$ is $\V^{-1}_{X(s)/S} \otimes \cL^\vee$-equivariant when we let it act on $\cH om(\W^1_{X/S},\cL^\vee)$ via the negative of the embedding via the canonical map $dh^* = R_1 \circ E_1: \W^1_{X/S} \to \W^1_{X(s)/S}$, i.e., via $-Th_* \otimes \cL^\vee: \V^{-1}_{X(s)/S} \otimes \cL^\vee \to \V^{-1}_{X/S} \otimes \cL^\vee$. This map is injective since it is an isomorphism on $X \setminus H$, which is dense due to our log pre-regularity assumption.

\begin{lemma}
 Let $s \in \cL$ be a log quasi-regular section. Then the following statements are equivalent:
 \begin{enumerate}[label=\emph{(\alph*)}]
  \item\label{log-reg-via-surj} $s$ is a log regular section, i.e., $s^*\G^{-1}_{L(X)/S} \to \cL$ is surjective;
  \item\label{log-reg-via-sp} everywhere on $X$, there is locally a splitting $r \in \mathrm{Sp}_{X(s)/N(s)}$ such that 
  $$\psi(r) \in \G^{-1}_{X/S} \otimes \cL^\vee \subseteq \V^{-1}_{X/S} \otimes \cL^\vee.$$
 \end{enumerate}
\end{lemma}
\begin{proof}
 First, assume that $s \in \cL$ is a log regular section. Then we can find locally an element $\theta_\alpha \in s^*\G^{-1}_{L(X)/S}$ with $H_{-1}(\theta_\alpha) = e_\alpha$. In particular, this defines a splitting of $H_{-1}$ with $e_\alpha \mapsto \theta_\alpha$, and due to the isomorphism $\eta: \mathrm{Sp}_{X(s)/N(s)} \to \mathrm{Sp}'$, we can find a splitting $r \in \mathrm{Sp}_{X(s)/N(s)}$ such that $e_\alpha \mapsto \theta_\alpha$ is given by $V_{-1}[r]$, the dual of $U_1[r]$. Since $F_{-1}(\theta_\alpha) \in \G^{-1}_{X/S}$, we have 
 $$F_{-1} \circ V_{-1}[r]: \cL \to \G^{-1}_{X/S}.$$
 This implies $\psi(r) = U_1[r] \circ E_1 \in \G^{-1}_{X/S} \otimes \cL^\vee$, as desired. Conversely, if $\psi(r) \in \G^{-1}_{X/S} \otimes \cL^\vee$, then $F_{-1} \circ V_{-1}[r](e_\alpha) \in \G^{-1}_{X/S}$, so $\theta_\alpha := V_{-1}[r] \in s^*\G^{-1}_{L(X)/S}$ by the variant of \eqref{Gamma-LXS-seq} for $\G^\bullet$. Then $H_{-1}(\theta_\alpha) = e_\alpha$, so $s^*\G^{-1}_{L(X)/S} \to \cL$ is surjective, and $s$ is a log regular section.
\end{proof}

Let 
$$\psi^{-1}(\V^{-1}_{X/S} \otimes \cL^\vee) =: \ \widetilde{\mathrm{Sp}}  \ \subseteq \mathrm{Sp}_{X(s)/N(s)}$$
be the subsheaf of those splittings $r \in \mathrm{Sp}_{X(s)/N(s)}$ with $\psi(r) \in \G^{-1}_{X/S} \otimes \cL^\vee$. We have an induced map 
$$Th_*: \G^{-1}_{X(s)/S} \xrightarrow{J_{-1}} s^*\G^{-1}_{L(X)/S} \xrightarrow{F_{-1}} \G^{-1}_{X/S},$$
and under this map, we have 
$$\psi(\theta \odot r) = \psi(r) - (Th_* \otimes \cL^\vee)(\theta)$$
for $\theta \in \G^{-1}_{X(s)/S} \otimes \cL^\vee$, i.e., $\G^{-1}_{X(s)/S} \otimes \cL^\vee$ acts on $\widetilde{\mathrm{Sp}}$. If $r', r \in \widetilde{\mathrm{Sp}}$ are two local sections, then there is some $\theta \in \V^{-1}_{X(s)/S} \otimes \cL^\vee$ with $r' = \theta \odot r$. Since $\psi(\theta \odot r) = \psi(r) - (Th_* \otimes \cL^\vee)(\theta)$, we find $(Th_* \otimes \cL^\vee)(\theta) \in \G^{-1}_{X/S} \otimes \cL^\vee$. This implies $(J_{-1} \otimes \cL^\vee)(\theta) \in s^*\G^{-1}_{L(X)/S} \otimes \cL^\vee$, and hence $\theta \in \G^{-1}_{X(s)/S} \otimes \cL^\vee$ more or less by definition. Thus, $\widetilde{\mathrm{Sp}}$ is an $\G^{-1}_{X(s)/S} \otimes \cL^\vee$-torsor. We summarize this as follows:

\begin{prop}
 Let $s \in \cL$ be a log regular section. Then there is a distinguished submodule $\G^{-1}_{X(s)/S} \subseteq \V^{-1}_{X(s)/S}$ of log derivations on $f \circ h: X(s) \to S$, depending only on $s \in \cL$ and the enhanced generically log smooth family $f: X \to S$, together with a canonical $\G^{-1}_{X(s)/S} \otimes \cL^\vee$-torsor $\widetilde{\mathrm{Sp}}$ of special splittings of $e: X(s) \to N(s)$.
\end{prop}

Using $\widetilde{\mathrm{Sp}}$, we give yet another equivalent characterization of log regularity.

\begin{lemma}\label{special-log-der}\note{special-log-der}
 Let $s \in \cL$ be a log quasi-regular section. Then $s$ is log regular if and only if, locally on each $V_\alpha$, there is a log derivation $\theta = (D,\Delta) \in \G^{-1}_{X/S}$ with $D(s_\alpha) = 1 + a \cdot s_\alpha$ for some (locally defined) function $a \in \cO_X$.
\end{lemma}
\begin{proof}
 First assume that $s$ is log regular. Recall that $i: H \to X$ is the inclusion. We consider the diagram 
 \[
  \xymatrix{
   0 \ar[r] & i^*\W^1_{X/S} \ar[r]^-{i^*E_1} & i^*s^*\W^1_{L(X)/S} \ar[r]^-{i^*Q_1} & i^*s^*\Omega^1_{L(X)/X} \ar[r] & 0 \\
   0 \ar[r] & i^*\cL^\vee \ar@{-->}[u]^{M} \ar[r]^-{i^*G_1} & i^*s^*\W^1_{L(X)/S} \ar[r]^-{i^*R_1} \ar@{=}[u] & i^*s^*\W^1_{X(s)/S} \ar[r] & 0 \\
  }
 \]
 with two locally split exact rows. Since $i^*Q_1 \circ i^*G_1(i^*e_\alpha^\vee) = 0$, there is a map $M: i^*\cL^\vee \to i^*\W^1_{X/S}$ as indicated. For any splitting $U_1[r]$ of $G_1$, we have $i^*U_1[r] \circ M = \mathrm{id}_{i^*\cL^\vee}$.
 Now let $r \in \widetilde{\mathrm{Sp}}$ be a local splitting, and let $(D,\Delta) \in \G^{-1}_{X/S}$ be such that $\psi(r) = (D,\Delta) \otimes e_\alpha^\vee$. Then we have $D(s_\alpha) \otimes e_\alpha^\vee = U_1[r](ds_\alpha)$, and since $M(i^*e_\alpha^\vee) = -i^*ds_\alpha$, we have 
 $$i^*D(s_\alpha) \otimes i^*e_\alpha^\vee = -i^*U_1[r] \circ M(i^*e_\alpha^\vee) = -i^*e_\alpha^\vee.$$
 Thus, $1 + D(s_\alpha) \in \mathrm{Ann}(i^*\cL^\vee) = \I_H = (s_\alpha)$, and hence we can write $D(s_\alpha) = -1 - a \cdot s_\alpha$ for some $a \in \cO_X$. Replacing $(D,\Delta)$ with $(-D,-\Delta)$ yields the desired log derivation.
 
 Conversely, let $\theta = (D,\Delta)$ be a log derivation with $D(s_\alpha) = 1 + a \cdot s_\alpha$. Let 
 $$\xi := -a \cdot s^*\delta^\vee - S_{-1}(\theta)$$
 where $S_{-1}$ is the splitting associated with the original distinguished splitting $B_1(\gamma_s) = \delta_\alpha$. We have $\xi \in s^*\G^{-1}_{L(X)/S}$ because of $\theta \in \G^{-1}_{X/S}$. Since $H_{-1}: s^*\V^{-1}_{L(X)/S} \to \cL$ is the dual of $G_1: \cL^\vee \to s^*\W^1_{L(X)/S}$, and since $G_1(e_\alpha^\vee) = s_\alpha \cdot s^*\delta^\vee - ds_\alpha$, we have 
 $$H_{-1}(\xi)(e_\alpha^\vee) = \xi \ \invneg \ G_1(e_\alpha^\vee) = 1.$$
 Thus, $H_{-1}(\xi) = e_\alpha$, and $s^*\G^{-1}_{L(X)/S} \to \cL$ is surjective.
\end{proof}
\begin{rem}
 Note that this criterion is independent of the choice of $s_\alpha$: For another choice $u_\alpha \cdot s_\alpha$ with $u_\alpha$ invertible, we find $D(u_\alpha \cdot s_\alpha) = 1 + (u_\alpha \cdot a + D(u_\alpha)) \cdot s_\alpha$.
\end{rem}

\section{The Gerstenhaber calculus on $f \circ h: X(s) \to S$}

Let $s \in \cL$ be a log regular section. The exact sequence of Lemma~\ref{log-qreg-crit} yields a locally split exact sequence 
$$0 \to \W^{i - 1}_{X(s)/S} \otimes \cL^\vee \xrightarrow{G_i} s^*\W^i_{L(X)/S} \xrightarrow{R_i} \W^i_{X(s)/S} \to 0.$$
Namely, on $U$, we obtain this sequence as usual; since $s^*\W^i_{L(X)/S}$ is reflexive, we can extend it via the direct image to $X$; and finally, the right hand side is surjective because 
$$s^*\Omega^i_{L(X)/S}|_U \to \Omega^i_{U(s)/S}$$
is locally split on $X$ (!). Note that $s^*\A^i_{L(X)/S} \subseteq s^*\W^i_{L(X)/S}$.
We define 
$$\A^i_{X(s)/S} := R_i(s^*\A^i_{L(X)/S}) \subseteq \W^i_{X(s)/S}.$$
Dually, we have a locally split exact sequence 
$$0 \to \V^p_{X(s)/S} \xrightarrow{J_p} s^*\V^p_{L(X)/S} \xrightarrow{H_p} \V^{p + 1}_{X(s)/S} \otimes \cL \to 0$$
since $s^*\V^p_{L(X)/S} \cong \cH om(s^*\W^{-p}_{L(X)/S},\cO_X)$. Note that $s^*\G^p_{L(X)/S} \subseteq s^*\V^p_{L(X)/S}$. We define 
$$\G^p_{L(X)/S} := J_p^{-1}(s^*\G^p_{L(X)/S}) \subseteq \V^p_{X(s)/S},$$
extending our above definition of $\G^{-1}_{X(s)/S}$. These definitions make sense for log quasi-regular sections $s \in \cL$, but we need log regularity in order for them to be well-behaved.

Obviously, we have well-defined comparison maps 
$$dh^* = R_i \circ E_i: \A^i_{X/S} \to \A^i_{X(s)/S}, \quad Th_* = F_p \circ J_p: \G^p_{X(s)/S} \to \G^p_{X/S},$$
which are associated with $h: X(s) \to X$. Recall that $o: X \to L$ is the zero section. Since $L \setminus o(X) = L(X) \setminus o(X)$, the morphism $h: X(s) \to X$ is an isomorphism on $X \setminus H$, so $dh^*$ and $Th_*$ are isomorphisms on $X \setminus H$ on the level of $\V^\bullet$ and $\W^\bullet$. This remains true for $\G^\bullet$ and $\A^\bullet$ under the definition we have given:

\begin{lemma}\label{X-H-isom}\note{X-H-isom}
 The maps $dh^*: \A^i_{X/S} \to \A^i_{X(s)/S}$ and $Th_*: \G^p_{X(s)/S} \to \G^p_{X/S}$ are isomorphisms on $X \setminus H$.
\end{lemma}
\begin{proof}
 Both maps are injective because they are isomorphisms on the level of $\V^\bullet$ and $\W^\bullet$, so it suffices to show surjectivity. If $\alpha \in \A^i_{X(s)/S}$ inside $X \setminus H$, then, locally, we can find $\beta = E_i(\beta_1) + E_{i - 1}(\beta_2) \wedge \delta \in s^*\A^i_{L(X)/S}$ with $R_i(\beta) = \alpha$. Now 
 $$R_i(\beta) = R_i \circ E_i(\beta_1) + R_{i - 1} \circ E_{i - 1}(\beta_2) \wedge R_1(\delta)$$
 with $R_1(\delta) = s_\alpha^{-1} \cdot R_1 \circ E_1(ds_\alpha)$,\footnote{Here we take $\delta = B_1(\gamma_s)$ for our original distinguished splitting $B_1$ used to define the class of distinguished splittings.}\footnote{In the computation, we use $R_i(\alpha) \wedge R_j(\beta) = R_{i + j}(\alpha \wedge \beta)$, which holds by  definition of $R_i$.} so we have 
 $$\alpha = R_i(\beta) = R_i \circ E_i(\beta_1 + s_\alpha^{-1} \cdot \beta_2 \wedge ds_\alpha),$$
 i.e., $dh^*$ is surjective. For the surjectivity of $Th_*$, let $\theta \in \G^p_{X/S}$. Since $Th_*: \V^p_{X(s)/S} \to \V^p_{X/S}$ is an isomorphism on $X \setminus H$, we can find $\theta' \in \V^p_{X(s)/S}$ with $Th_*(\theta') = \theta$. Let $$S_p: \V^p_{X/S} \to s^*\V^p_{L(X)/S}$$
 be the splitting of $F_p$ associated with the original distinguished splitting $B_1$.
 Let $i = -p$ and $\alpha \in s^*\W^i_{L(X)/S}$. On the one hand side, we have 
 $$J_p(\theta') \ \invneg \ (E_i(\alpha_1) + E_{i - 1}(\alpha_2) \wedge \delta) = \theta \ \invneg \ \alpha_1 + \theta \ \invneg \ (s_\alpha^{-1} \cdot \alpha_2 \wedge ds_\alpha).$$
 On the other hand side, we have 
 $$S_p(\theta) \ \invneg \ (E_i(\alpha_1) + E_{i - 1}(\alpha_2) \wedge \delta) = \theta \ \invneg \ \alpha_1.$$
 Thus, we have 
 \begin{align}
  (J_p(\theta')\ -\ &S_p(\theta)) \ \invneg \ (E_i(\alpha_1) + E_{i - 1}(\alpha_2) \wedge \delta) = \theta \ \invneg \ (s_\alpha^{-1} \cdot \alpha_2 \wedge ds_\alpha) \nonumber \\
  &= (-1)^{i - 1} s_\alpha^{-1} \cdot (\theta \vdash ds_\alpha) \ \invneg \ \alpha_2 = s_\alpha^{-1} \cdot I_p((\theta \vdash ds_\alpha) \otimes \gamma_s^\vee) \ \invneg \ \alpha. \nonumber
 \end{align}
 Since $S_p(\theta) \in s^*\G^p_{L(X)/S}$ and $s_\alpha^{-1} \cdot I_p((\theta \vdash ds_\alpha) \otimes \gamma_s^\vee) \in s^*\G^p_{L(X)/S}$, we have $J_p(\theta') \in s^*\G^p_{L(X)/S}$, and hence $\theta' \in \G^p_{X(s)/S}$.
\end{proof}

We work again on the whole of $X$. We have 
$$G_1(1 \otimes e_\alpha^\vee) = s_\alpha \cdot s^*\delta - d_{X/S}(s_\alpha),$$
i.e., $G_1(1 \otimes e_\alpha^\vee) \in s^*\A^1_{L(X)/S}$, so we find $G_i(\A^{i - 1}_{X(s)/S} \otimes \cL^\vee) \subseteq s^*\A^i_{L(X)/S}$ since $s^*\A^\bullet_{L(X)/S} \subseteq s^*\W^\bullet_{L(X)/S}$ is closed under the $\wedge$-product. Thus, we have a sequence 
\begin{equation}\label{GR-A-seq}
 0 \to \A^{i - 1}_{X(s)/S} \otimes \cL^\vee \xrightarrow{G_i} s^*\A^i_{L(X)/S} \xrightarrow{R_i} \A^i_{X(s)/S} \to 0.
\end{equation}\note{GR-A-seq}
It is exact on the left and on the right by construction.

\begin{lemma}
 If $s \in \cL$ is a log regular section, then the sequence \eqref{GR-A-seq} is exact in the middle and locally split.
\end{lemma}
\begin{proof}
 We start with exactness in the middle. By construction, we have $R_i \circ G_i = 0$. From Lemma~\ref{X-H-isom}, it is not so hard to see that \eqref{GR-A-seq} is exact in the middle on $X \setminus H$. So let us work with the stalks at some $x \in H$. Let $\alpha \in s^*\A^i_{L(X)/S}$ be such that $R_i(\alpha) = 0$. From the exactness of the corresponding sequence for $\W^\bullet$, we find that there is some $\beta \in s^*\W^{i - 1}_{L(X)/S}$ with $\beta \wedge G_1(1 \otimes e_\alpha^\vee) = \alpha$. After decomposing $\beta = E_{i - 1}(\beta_1) + E_{i - 2}(\beta_2) \wedge \delta$ with $\beta_1 \in \W^{i - 1}_{X/S}$ and $\beta_2 \in \W^{i - 2}_{X/S}$, we have 
 $$\alpha = - E_i(\beta_1 \wedge ds_\alpha) + E_{i - 1}(s_\alpha \cdot \beta_1 + \beta_2 \wedge ds_\alpha) \wedge s^*\delta.$$
 This shows $\beta_1 \wedge ds_\alpha \in \A^i_{X/S}$ and $s_\alpha \cdot \beta_1 + \beta_2 \wedge ds_\alpha \in \A^{i - 1}_{X/S}$. Let $\theta \in \G^{-1}_{X/S}$ be a log derivation as obtained from Lemma~\ref{special-log-der}, i.e., $\theta \ \invneg \ ds_\alpha = 1 + a \cdot s_\alpha =: u$. Since we are at the stalk at $x \in H$, the function $u \in \cO_{X,x}$ is invertible. We set 
 $$\tilde\beta_1 := (-1)^{i - 1} u^{-1} \cdot \theta \ \invneg\ (\beta_1 \wedge ds_\alpha) \in \A^{i - 1}_{X/S};$$
 then we have $\tilde\beta_1 \wedge ds_\alpha = \beta_1 \wedge ds_\alpha$. Furthermore, we have 
 $$((-1)^i u^{-1} \cdot s_\alpha \cdot \theta \ \invneg \ (\beta_1 - \tilde\beta_1)) \wedge ds_\alpha = s_\alpha \cdot (\beta_1 - \tilde\beta_1),$$
 in other words, if $\tilde\gamma \in \W^{i - 1}_{X/S}$ is the left hand side without the $ds_\alpha$-term, then $(\tilde\gamma + \beta_2) \wedge ds_\alpha \in \A^{i - 1}_{X/S}$. After setting 
 $$\tilde\beta_2 = (-1)^{i - 2}u^{-1} \cdot \theta \ \invneg \ ((\tilde\gamma + \beta_2) \wedge ds_\alpha) \in \A^{i - 2}_{X/S},$$
 we have $s_\alpha \cdot \tilde\beta_1 + \tilde\beta_2 \wedge ds_\alpha = s_\alpha \cdot \beta_1 + \beta_2 \wedge ds_\alpha$. Thus, 
 $$\tilde\beta := E_{i - 1}(\tilde\beta_1) + E_{i - 2}(\tilde\beta_2) \wedge s^*\delta \in s^*\A^{i - 1}_{L(X)/S},$$
 and $\tilde\beta \wedge G_1(1 \otimes e_\alpha^\vee) = \alpha$. In other words, $G_i(R_{i - 1}(\tilde\beta) \otimes e_\alpha^\vee) = \alpha$, and \eqref{GR-A-seq} is exact in the middle.
 
 To show that the sequence is locally split, let $\theta = (D,\Delta) \in \G^{-1}_{X/S}$ be such that $D(s_\alpha) = 1 + a \cdot s_\alpha$, and let $\xi = - a \cdot s^*\delta^\vee - S_{-1}(\theta) \in s^*\G^{-1}_{L(X)/S}$. Then we obtain a splitting of $G_i$ as 
$$U_i: s^*\A^i_{L(X)/S} \to \A^{i - 1}_{X(s)/S} \otimes \cL^\vee, \quad \alpha \mapsto (-1)^{i - 1}R_{i - 1}(\xi \ \invneg \ \alpha) \otimes e_\alpha^\vee.$$
If $\beta \in \A^{i - 1}_{X(s)/S}$, then $G_i(\beta \otimes e_\alpha^\vee) = \tilde\beta \wedge G_1(e_\alpha^\vee)$ for any lift $\tilde\beta \in s^*\A^{i - 1}_{L(X)/S}$ with $R_{i - 1}(\tilde\beta) = \beta$, so $$U_i \circ G_i(\beta \otimes e_\alpha^\vee) = R_{i - 1}((-1)^{i - 1}(\xi \ \invneg \ \tilde\beta) \wedge G_1(e_\alpha^\vee) + \tilde\beta \wedge (\xi \ \invneg \ G_1(e_\alpha^\vee))) \otimes e_\alpha^\vee = \beta \otimes e_\alpha^\vee$$
since $\xi \ \invneg \ G_1(e_\alpha^\vee) = 1$.
\end{proof}

Next, we study the dual of the sequence \eqref{GR-A-seq}. If $\theta \in s^*\V^p_{L(X)/S}$ is such that $H_p(\theta) = \xi \otimes e_\alpha$ for some $\xi \in \V^{p + 1}_{X(s)/S}$, then we have 
$$\xi \ \invneg \ R_{i - 1}(\beta) = (\theta \vdash G_1(e_\alpha^\vee)) \ \invneg \ \beta$$
for every $\beta \in s^*\W^{i - 1}_{L(X)/S}$ for $i = -p$. In other words, we have 
$J_{p + 1}(\xi) = \theta \vdash G_1(e_\alpha^\vee)$. If $\theta \in s^*\G^p_{L(X)/S}$, then $\theta \vdash G_1(e_\alpha^\vee) \in s^*\G^{p + 1}_{L(X)/S}$, and hence $\xi \in \G^{p + 1}_{X(s)/S}$ by definition. Thus, we have a well-defined sequence 
\begin{equation}\label{JH-G-seq}
 0 \to \G^p_{X(s)/S} \xrightarrow{J_p} s^*\G^p_{L(X)/S} \xrightarrow{H_p} \G^{p + 1}_{X(s)/S} \otimes \cL \to 0,
\end{equation}\note{JH-G-seq}
which is exact on the left and in the middle.

\begin{lemma}
 The sequence \eqref{JH-G-seq} is exact on the right and locally split.
\end{lemma}
\begin{proof}
 We construct a local splitting of $H_p$. Let $\theta_\alpha = (D_\alpha,\Delta_\alpha) \in \G^{-1}_{X/S}$ be such that $D_\alpha(s_\alpha) = 1 + a \cdot s_\alpha$, and let $\xi_\alpha := - a \cdot s^*\delta^\vee - S_{-1}(\theta_\alpha) \in s^*\G^{-1}_{L(X)/S}$.\footnote{Note that this does not only depend on $\alpha$.} Let $\theta \in \G^{p + 1}_{X(s)/S}$. Since $H_p$ is surjective for $\V^\bullet$, there is some $\xi \in \V^p_{L(X)/S}$ with $H_p(\xi) = \theta \otimes e_\alpha$, and hence $J_{p + 1}(\theta) = \xi \vdash G_1(e_\alpha^\vee)$. We define the splitting by 
 $$V_p: \G^{p + 1}_{X(s)/S} \otimes \cL \to s^*\G^p_{L(X)/S}, \quad \theta \otimes e_\alpha \mapsto J_{p + 1}(\theta) \wedge \xi_\alpha.$$
 Then $H_p \circ V_p(\theta \otimes e_\alpha) = \theta' \otimes e_\alpha$ for some $\theta'$ with 
 $$J_{p + 1}(\theta') = (J_{p + 1}(\theta) \wedge \xi_\alpha) \vdash G_1(e_\alpha^\vee) = - (\xi \vdash G_1(e_\alpha^\vee)) \vdash G_1(e_\alpha^\vee) + J_{p + 1}(\theta) \wedge (\xi_\alpha \vdash G_1(e_\alpha^\vee)), $$
 i.e., $\theta' = \theta$ since $\xi_\alpha \vdash G_1(e_\alpha^\vee) = 1$. Thus $V_p$ is a splitting of $H_p$.
\end{proof}

With this preparation, we show:

\begin{prop}\label{WdXs-comp}\note{WdXs-comp}
 Let $s \in \cL$ be a log regular section. Assume that the pieces of $\G^p_{X/S}$ and $\A^i_{X/S}$ are torsionless in every fiber. Then $f \circ h: X(s) \to S$ is a log Gorenstein generically log smooth family. The relative log canonical bundle is 
 $$\W^d_{X(s)/S} \cong \W^d_{X/S} \otimes \cL.$$
 When we endow $f \circ h: X(s) \to S$ with $(\G^\bullet_{X(s)/S},\A^\bullet_{X(s)/S})$, this is an enhanced generically log smooth family whose pieces are torsionless in every fiber. The formation of the Gerstenhaber calculus commutes with base change in $S$.
\end{prop}
\begin{proof}
 Considering the exact sequence \eqref{GR-A-seq} for $i = d + 1$ together with its variant for $\W^\bullet$ and the embedding, we find that 
 $$\varpi^{d + 1} \otimes \cL^\vee: \A^d_{X(s)/S} \otimes \cL^\vee \to \W^d_{X(s)/S} \otimes \cL^\vee$$
 is an isomorphism, and that both sides are isomorphic to the line bundle $s^*\W^{d + 1}_{L(X)/S} \cong \W^d_{X/S} \otimes s^*\Omega^1_{L(X)/X} \cong \W^d_{X/S}$. The exact sequences \eqref{GR-A-seq} and \eqref{JH-G-seq} show that each piece $\G^p_{X(s)/S}$ and $\A^i_{X(s)/S}$ is flat over $S$, and that they are torsionless in every fiber since this holds for $\G^p_{X/S}$ and $\A^i_{X/S}$, and hence for $s^*\G^p_{L(X)/S}$ and $s^*\A^i_{L(X)/S}$.
 
 Because $J_0$ and $R_0$ are isomorphisms, we have $\A^0_{X(s)/S} = \cO_X$ and $\G^0_{X(s)/S} = \cO_X$. Since $R_i$ is compatible with the $\wedge$-product, $\A^\bullet_{X(s)/S} \subseteq \W^\bullet_{X(s)/S}$ is closed under the $\wedge$-product. To see that $\A^\bullet_{X(s)/S}$ is a complex, we consider the diagram 
 \[
  \xymatrix{
   \A^i_{L(X)/S} \ar[r] \ar@{^(->}[d] & s_*s^*\A^i_{L(X)/S} \ar[r] \ar@{^(->}[d] & s_*\A^i_{X(s)/S} \ar@{^(->}[d] & \\
   \W^i_{L(X)/S} \ar[r] & s_*s^*\W^i_{L(X)/S} \ar[r] & s_*\W^i_{X(s)/S} \ar@{=}[r] & j_*s_*\Omega^i_{U(s)/S}. \\
  }
 \]
 The lower composition is compatible with the de Rham differential $\partial$ since it is so on $p^{-1}(U)$. The upper row remains surjective. Then the claim follows from the fact that $\A^\bullet_{L(X)/S} \subseteq \W^\bullet_{L(X)/S}$ is a subcomplex.
 
 The graded subsheaf $\G^\bullet_{X(s)/S} \subseteq \V^\bullet_{X(s)/S}$ is closed under the $\wedge$-product because $J_p$ is compatible with the $\wedge$-product. Similar to the proof of Lemma~\ref{formulae-LXS}, we can show that 
 $$R_{p + i}(J_p(\theta) \ \invneg \ \alpha) = \theta \ \invneg \ R_i(\alpha)$$
 for $\theta \in \V^p_{X(s)/S}$ and $\alpha \in s^*\W^i_{L(X)/S}$. Using this formula, we see that $\G\C^\bullet_{X(s)/S}$ is closed under the right contraction $\invneg$. Similarly, we have $J_p(\theta) \vdash \alpha = J_{p + i}(\theta \vdash R_i(\alpha))$,\footnote{Note that this holds for arbitrary $\alpha \in s^*\W^i_{L(X)/S}$---the left contraction with an element in the image of $J_p$ is always in the image of $J_{p + i}$.} so $\G\C^\bullet_{X(s)/S}$ is closed under the left contraction $\vdash$.
 
 Next, we show that $\G\C^\bullet_{X(s)/S}$ is locally Batalin--Vilkovisky. Let $\omega \in \A^d_{X(s)/S}$ be a local volume form, and let $\tilde \omega \in s^*\A^d_{L(X)/S}$ be such that $R_d(\tilde\omega) = \omega$. A local volume form of $s^*\A^{d + 1}_{L(X)/S}$ is now given by $\hat\omega = \tilde\omega \wedge G_1(e_\alpha^\vee)$. Given $\omega$, this form is independent of the chosen lift $\tilde\omega$. We have to show that 
 $$\kappa_\omega: \G^p_{X(s)/S} \to \A^{p + d}_{X(s)/S}, \quad \theta \mapsto \theta \ \invneg\ \omega,$$
 is an isomorphism. By construction, the corresponding map on the level of $\V^p_{X(s)/S}$ and $\W^{p + d}_{X(s)/S}$ is an isomorphism. In particular, our $\kappa_\omega$ is injective, and for $\alpha \in \A^{p + d}_{X(s)/S}$, we can find $\theta \in \V^p_{X(s)/S}$ with $\theta \ \invneg \ \omega = \alpha$. We have to show that $\theta \in \G^p_{X(s)/S} \subseteq \V^p_{X(s)/S}$, i.e., $J_p(\theta) \in s^*\G^p_{L(X)/S}$. Since $\G\C^\bullet_{L(X)/S}$ is locally Batalin--Vilkovisky, it is sufficient to show $J_p(\theta) \ \invneg\ \hat\omega \in s^*\A^{p + d + 1}_{L(X)/S}$. If $\tilde \alpha \in s^*\A^{p + d}_{L(X)/S}$ is a lift of $\alpha$, then we find $R_{p + d}(J_p(\theta) \ \invneg \ \tilde\omega) = R_{p + d}(\tilde \alpha)$ from the above formulae. In particular, there is some $\beta \in \W^{p + d - 1}_{X(s)/S} \otimes \cL^\vee$ with $J_p(\theta) \ \invneg \ \tilde\omega = G_{p + d}(\beta) + \tilde\alpha$. By induction on $p$, we can show that 
 $$J_p(\theta) \ \invneg\ \hat\omega = (J_p(\theta) \ \invneg \ \tilde\omega) \wedge G_1(e_\alpha^\vee)$$
 although there is in general no easy formula for the action of the right contraction on a $\wedge$-product on the right. Thus, $J_p(\theta) \ \invneg \ \hat\omega = \tilde\alpha \wedge G_1(e_\alpha^\vee) \in s^*\A^{p + d + 1}_{L(X)/S}$, as desired.
 
 Now we obtain that $\G\C^\bullet_{X(s)/S}$ is closed under the Lie derivative $\cL_{-}(-)$ from the Lie--Rinehart homotopy formula, and that $\G\C^\bullet_{X(s)/S}$ is closed under the bracket $[-,-]$ from the mixed Leibniz rule.
 
 Considering the exact sequence \eqref{GR-A-seq}, we see by induction on $i$ that the formation of $\A^i_{X(s)/S}$ commutes with base change because the formation of $s^*\A^i_{L(X)/S}$ commutes with base change. Since $\G\C^\bullet_{X(s)/S}$ is locally Batalin--Vilkovisky, then the same is true for $\G^p_{X(s)/S}$.
\end{proof}

\section{Thickenings of log regular sections}

In this section, we show that a thickening $s \in \cL$ of a log regular section $s_0 \in \cL_0$ is again a log regular section. In other words, the condition of log regularity does not impose additional restrictions on infinitesimal deformations once we require it on the central fiber. We work in the following situation:

\begin{sitn}\label{log-modif-defo-sitn}\note{log-modif-defo-sitn}
 Let $S_0 = \Spec (Q \to \kk)$, and let $f_0: X_0 \to S_0$ be an enhanced generically log smooth family of relative dimension $d$ with Gerstenhaber calculus $\varpi^\bullet: (\G^\bullet_{X_0/S_0},\A^\bullet_{X_0/S_0}) \subseteq (\V^\bullet_{X_0/S_0},\W^\bullet_{X_0/S_0})$. Let $\cL_0$ be a line bundle on $X_0$, and let $s_0$ be a global section. We consider infinitesimal deformations $f_A: X_A \to S_A$ of $f_0: X_0 \to S_0$ with Gerstenhaber calculus $\varpi^\bullet: (\G^\bullet_{X_A/S_A},\A^\bullet_{X_A/S_A}) \subseteq (\V^\bullet_{X_A/S_A},\W^\bullet_{X_A/S_A})$. Then $\cL$ is a line bundle on $X_A$ with a chosen isomorphism $\cL|_0 = \cL_0$, and $s \in \cL$ is a global section with $s|_0 = s_0$.
\end{sitn}

First, we prove our claim for log pre-regular sections. 

\begin{lemma}
 Assume that $s_0 \in \cL_0$ is a log pre-regular section. Then $s \in \cL$ is a log pre-regular section.
\end{lemma}
\begin{proof}
 It is clear that $\mathrm{codim}(Z \cap H,H) \geq 2$ in every fiber since $f: X \to S$ and $f_0: X_0 \to S_0$ have the same fibers. The open subset $X_0 \setminus H_0 \subseteq X_0$ is scheme-theoretically dense, and by \cite[Rem.~4.12]{FeltenThesis} in the author's thesis, then the same is true for $X_A \setminus H_A \subseteq X_A$, so $\cO_{X_A} \to j_*\cO_{X_A \setminus H_A}$ is injective. Thus, $s^\vee: \cL^\vee \to \cO_{X_A}$ is injective. Now let $\mathbf{M}$ be the class of all finitely generated $A$-modules $M$ such that $\cL^\vee \otimes_A M \to \cO_{X_A} \otimes_A M$ remains injective. By assumption, we have $A/\m_A \in \mathbf{M}$, and the same argument that shows $\cL^\vee \to \cO_{X_A}$ injective also shows $A/I \in \mathbf{M}$ for all ideals $I \subseteq \m_A \subseteq A$. Since both $\cL^\vee$ and $\cO_{X_A}$ are flat over $S_A$, we find that every extension $M$ of two modules $M',M'' \in \mathbf{M}$ is in $\mathbf{M}$ as well. Since every finitely generated module $M$ has a filtration by submodules such that each quotient is of the form $A/I$ for some ideal $I \subseteq A$, we find that $\mathbf{M}$ is the class of all finitely generated $A$-modules. Then the sequence 
 $$0 \to \cL^\vee \to \cO_{X_A} \to \cO_{H_A} \to 0$$
 is universally exact over $A$, and since $\cO_{X_A}$ is flat over $A$, we find that $\cO_{H_A}$ is flat over $A$ by \cite[058P]{stacks}.
\end{proof}

On the log smooth locus $U_A$, it is easy to see that $s$ is log regular if $s_0$ is log regular. First, there is a splitting in $\mathrm{Sp}_{X_A(s)/N_A(s)}$ on $U_A \setminus H_A$, given by $p: L_A \to X_A$, so $G_1: \cL^\vee \to s^*\W^1_{L_A(X_A)/S_A}$ is injective. Since $\Omega^1_{U_0(s_0)/S_0}$ is locally free, the same is true for $\Omega^1_{U_A(s)/S_A}$, so the exact sequence in Lemma~\ref{log-qreg-crit} is locally split on $U_A$. On the log singular locus, this result is considerably more difficult to prove. We start with a lemma.

\begin{lemma}\label{log-qreg-surj}\note{log-qreg-surj}
 In Situation~\ref{log-modif-sitn}, let $s \in \cL$ be a log quasi-regular section. Let $\Q$ be the cokernel in the exact sequence 
 $$0 \to \V^{-1}_{X(s)/S} \xrightarrow{Th_*} \V^{-1}_{X/S} \xrightarrow{q} \Q \to 0.$$
 Let $V \subseteq X$ be an affine open subset. Then the induced map 
 $$\Gamma(V \cap U, \V^{-1}_{X/S}) \to \Gamma(V \cap U, \Q)$$
 is surjective.
\end{lemma}
\begin{proof}
 First, we consider the following diagram:
 \[
  \xymatrix{
  & 0 & 0 & 0 & \\
  0 \ar[r] & \cL^\vee \ar[r] \ar[u] & s^*\Omega^1_{L(X)/X} \ar[u] \ar[r] & j_*\Omega^1_{U(s)/U} \ar[u] \ar[r] & 0 \\
  0 \ar[r] & \cL^\vee \ar[u]^\cong \ar[r] & s^*\W^1_{L(X)/S} \ar[u] \ar[r] & \W^1_{X(s)/S} \ar[u] \ar[r] & 0 \\
  0 \ar[r] & 0 \ar[r] \ar[u] & \W^1_{X/S} \ar[u] \ar[r]^\cong & \W^1_{X/S} \ar[u] \ar[r] & 0 \\
  & 0 \ar[u] & 0 \ar[u] & 0 \ar[u] & \\
  }
 \]
 The middle row is exact by Lemma~\ref{log-qreg-crit} since $s \in \cL$ is a log quasi-regular section. The upper left horizontal map is an isomorphism on $X \setminus H$ and hence injective since both the source and the target are line bundles. On $U$, the upper row is exact in the middle and on the right by construction since $U(s) \to L(U) := L(X)|_{p^{-1}(U)}$ is a strict closed immersion. Since the upper left horizontal map is equal to $s^\vee: \cL^\vee \to \cO_X$ under $\cO_X \cong s^*\Omega^1_{L(X)/X}, \: 1 \mapsto s^*\gamma_s$, its cokernel is $\cO_H$. Since $s_\alpha$ is a non-zero divisor, the flat map $f \circ i: H \to S$ has Cohen--Macaulay fibers, i.e., we have $j_*\cO_{U \cap H} = \cO_H$. Thus, the map $\cO_H \to j_*\Omega^1_{U(s)/U}$ induced from the upper row is not only an isomorphism on $U$ but on $X$; in particular, the upper row is exact. Similar arguments show that the right column is exact. We obtain the surjectivity on the top since we already know that the middle column is exact and locally split, i.e., the upper left composition in the upper right square is surjective, so the right map must be surjective as well. Let us write $\cO_H$ for $j_*\Omega^1_{U(s)/U}$ from now on. By dualizing the above diagram with $\cH om(-,\cO_X)$, we obtain the following diagram:
 \[
  \xymatrix{
    \E xt^1(\cO_H,\cO_X) \ar[d] & \ar@{->>}[l] \cH om(\cL^\vee,\cO_X) \ar[d]^<<<{\cong} & \ar[l] \Theta^1_{L(X)/X} \ar@{^(->}[d] \ar[l] & \ar[l] 0 \ar[d] \\
    0_1 & \ar[l] \cH om(\cL^\vee,\cO_X) \ar[d] & \ar[l] s^*\V^{-1}_{L(X)/S} \ar[d] \ar@{-->}[ddr]_>>>>>>>>>{B} \ar@{-->}[ull]^>>>>>>>>{A}  & \ar@{_(->}[l] \V^{-1}_{X(s)/S} \ar[d] \\
    & 0 & \ar[l] \V^{-1}_{X/S} \ar[d] & \ar[l]_<<<<<{\cong} \V^{-1}_{X/S} \ar[d] \\
    & & 0_2 & \ar[l] \E xt^1(\cO_H,\cO_X) \ar[d] \\
    & & & \E xt^1(\W^1_{X(s)/S},\cO_X) \\
  }
 \]
 The spaces in the diagram at the places of $0_1$ and $0_2$ may not actually be $0$, but the horizontal map to $0_1$ and the vertical map to $0_2$ are surjective because they come from dualizing a locally split exact sequence. Thus, we can put a $0$ in their places in the diagram. From the upper row, we find $\E xt^1(\cO_H,\cO_X) \cong \cO_H \otimes \cL$. This sheaf satisfies $j_*(\cO_H \otimes \cL)|_U = \cO_H \otimes \cL$, so the upper horizontal left map is surjective for sections on $V \cap U$. Similarly, the map $s^*\V^{-1}_{L(X)/S} \to \cL$ is surjective for sections on $V \cap U$, so the map $A$ is surjective for sections on $V \cap U$. The maps $A$ and $B$ have the same source and target, and both have the same kernel $\Theta^1_{L(X)/X} + \V^{-1}_{X(s)/S}$. The map $A$ is surjective, and the map $B$ is surjective over $U$ since $\W^1_{X(s)/S}$ is locally free there. Thus, on $U$, there is unique automorphism $\gamma$ of $\E xt^1(\cO_H,\cO_X)$ with $B = \gamma \circ A$; in particular, the map $B$ is surjective for sections on $V \cap U$. Since the canonical map $\Q \to \E xt^1(\cO_H,\cO_X)$ is an isomorphism on $U$, we find that $\V^{-1}_{X/S} \to \Q$ is surjective for sections on $V \cap U$.
\end{proof}
\begin{cor}
 We have $\Q \cong \E xt^1(\cO_H,\cO_X) \cong \cO_H \otimes \cL$.
\end{cor}
\begin{proof}
 The map $\Gamma(V,\Q) \to \Gamma(V \cap U,\Q)$ is both injective because $\Q$ is a quotient of a reflexive sheaf by a reflexive subsheaf, and it is surjective as an easy application of Lemma~\ref{log-qreg-surj}. Thus, $\Q \to \E xt^1(\cO_H,\cO_X)$ is an isomorphism.
\end{proof}

With this preparation, we can prove that, in Situation~\ref{log-modif-defo-sitn}, an infinitesimal thickening $s$ of a log regular section $s_0$ is again a log regular section. Note that the proof below is not correct if we assume only log quasi-regular because the map $\V^{-1}_{X_{A'}/S_{A'}} \to \V^{-1}_{X_A/S_A}$ need not be surjective.

\begin{prop}
 In Situation~\ref{log-modif-defo-sitn}, if $s_0 \in \cL_0$ is a log regular section, then $s \in \cL$ is a log regular section.
\end{prop}
\begin{proof}
 We prove the result by induction over small extensions $A' \to A$ with kernel $I \subseteq A'$. So let $f': X_{A'} \to S_{A'}$ be a deformation of $f: X_A \to S_A$ as an enhanced generically log smooth family, and assume that we have a deformation $\cL'$ of $\cL$ and a section $s' \in \cL'$ with $s'|_A = s$. We consider the following diagram:
 \[
  \resizebox{14cm}{!}{
  \xymatrixcolsep{0.6em}\xymatrix{
    & & \V^{-1}_{X_A(s)/S_A} \otimes \cL^\vee \ar@{}[d]^{\circlearrowleft} & & \V^{-1}_{X_{A'}(s')/S_{A'}} \otimes \cL^\vee \ar@{}[d]^{\circlearrowleft} \ar[ll] & \V^{-1}_{X_0(s_0)/S_0} \otimes \cL_0^\vee \otimes I \ar[l] \\
   & & \mathrm{Sp}_{X_A(s)/N_A(s)} \ar@{^(->}[ddd]^\psi & & \mathrm{Sp}_{X_{A'}(s')/N_{A'}(s')} \ar@{^(->}[ddd]^{\psi'} \ar[ll] & \\
   & \G^{-1}_{X_A(s)/S_A} \otimes \cL^\vee \ar[uur]^\subseteq \ar@{}[d]^{\circlearrowleft} & & \G^{-1}_{X_{A'}(s')/S_{A'}} \otimes \cL^\vee \ar[uur]^\subseteq \ar@{}[d]^{\circlearrowleft} \ar@{->>}[ll] & & \G^{-1}_{X_0(s_0)/S_0} \otimes \cL_0^\vee \otimes I  \ar@{_(->}[ll]&  \\
   &\widetilde{\mathrm{Sp}}_{X_A(s)/N_A(s)} \ar[uur]_\subseteq \ar@{^(->}[ddd]^\psi & & \widetilde{\mathrm{Sp}}_{X_{A'}(s')/N_{A'}(s')} \ar[ll] \ar[uur]_\subseteq \ar@{^(->}[ddd]^{\psi'} & & \\
   & & \V^{-1}_{X_A/S_A} \otimes \cL^\vee & & \V^{-1}_{X_{A'}/S_{A'}} \otimes \cL'^\vee \ar[ll] & \\
   & & \\
   & \G^{-1}_{X_A/S_A} \otimes \cL^\vee \ar[uur]_\subseteq & & \G^{-1}_{X_{A'}/S_{A'}} \otimes \cL'^\vee \ar@{->>}[ll] \ar[uur]_\subseteq  & & \\
  }
  }
 \]
 Let $V \subseteq X$ be a (small) affine open subset. Since $s \in \cL$ is log regular, we have a section $r \in \widetilde{\mathrm{Sp}}_{X_A(s)/N_A(s)}$ over $V$. Let $\U = \{U_i\}_i$ be an affine open cover of $V \cap U$. Since the restriction map 
 $$\widetilde{\mathrm{Sp}}_{X_{A'}(s')/N_{A'}(s')} \to \widetilde{\mathrm{Sp}}_{X_A(s)/N_A(s)}$$
 is surjective on $U$,\footnote{This is because of the torsor structure that we have; we do, however, not yet know the surjectivity on $X$ because, there, the source might be empty.} we can find lifts $r'_i \in \widetilde{\mathrm{Sp}}_{X_{A'}(s')/N_{A'}(s')}$ of $r$ on $U_i$ if the $U_i$ are small enough. On the other hand, we can find an element $g' \in \G^{-1}_{X_{A'}/S_{A'}} \otimes \cL'^\vee$ over $V$ with $g'|_A = \psi(r)$. Since $r'_i|_A = r'_j|_A$, we have, due to the torsor structure, an element $\theta_{ij} \in \G^{-1}_{X_0(s_0)/S_0} \otimes \cL_0^\vee \otimes I$ over $U_i \cap U_j$ with $r'_j = \theta_{ij} \odot r'_i$. Then 
 $$\psi'(r'_j) - \psi'(r'_i) = Th_*(\theta_{ij}) \in \G^{-1}_{X_0/S_0} \otimes \cL_0^\vee \otimes I.$$
 On the other hand side, we also have $g'_i := \psi'(r'_i) - g' \in \G^{-1}_{X_0/S_0} \otimes \cL_0^\vee \otimes I$, and $g'_j - g'_i = \psi'(r_j') - \psi'(r'_i)$. When we denote the cohomology class represented by $\theta_{ij}$ as 
 $$[\theta] \in H^1(U \cap V,\: \G^{-1}_{X_0(s_0)/S_0} \otimes \cL_0^\vee \otimes I),$$
 this shows $Th_*[\theta] = 0 \in H^1(U \cap V,\: \G^{-1}_{X_0/S_0} \otimes \cL_0^\vee \otimes I)$. Lemma~\ref{log-qreg-surj} implies that $Th_*$ is injective as a map on $H^1(U \cap V,-)$, i.e., $[\theta] = 0$. Then there are sections $\theta_i \in \G^{-1}_{X_0(s_0)/S_0} \otimes \cL_0^\vee \otimes I$ over $U_i$ with $\theta_j - \theta_i = \theta_{ij}$. Then the local splittings $\tilde r'_i := (-\theta_i) \odot r'_i$ glue to a splitting $r'$ on $V \cap U$ with $r'|_A = r$, i.e., $s' \in \cL'$ is a log quasi-regular section. Since the formation of $s'^*\G^{-1}_{L_{A'}(X_{A'})/S_{A'}} \to \cL'$ commutes with base change, $s'$ is in fact a log regular section.
\end{proof}

\section{Infinitesimal automorphisms and isomorphisms}

We study infinitesimal automorphisms and isomorphisms in the following situation:

\begin{sitn}
 Let $S_0 = \Spec (Q \to \kk)$, and let $f_0: X_0 \to S_0$ be an enhanced generically log smooth family of relative dimension $d$ with Gerstenhaber calculus 
 $$\varpi^\bullet: (\G^\bullet_{X_0/S_0},\A^\bullet_{X_0/S_0}) \subseteq (\V^\bullet_{X_0/S_0},\W^\bullet_{X_0/S_0}).$$
 Note that we assume $\varpi^\bullet$ to be injective. Let $\cL_0$ be a line bundle on $X_0$, and let $s_0$ be a \emph{log regular} global section. We fix an infinitesimal deformation $f: X \to S$ of $f_0: X_0 \to S_0$ over $S = S_A$ with Gerstenhaber calculus $\varpi^\bullet: (\G^\bullet_{X/S},\A^\bullet_{X/S}) \subseteq (\V^\bullet_{X/S},\W^\bullet_{X/S})$. 
We have a line bundle $\cL$ on $X$ with a chosen isomorphism $\cL|_0 = \cL_0$, and $s \in \cL$ is a global section with $s|_0 = s_0$. Furthermore, we have a not necessarily small extension $A' \to A$ with kernel $I \subset A'$, and we consider a thickening $f': X' \to S'$ of $f: X \to S$ over $S' = S_{A'}$ as an enhanced generically log smooth family. We have a line bundle $\cL'$ on $X'$ with chosen isomorphism $\cL'|_A = \cL$, and we consider global sections $s' \in \cL'$ with $s'|_A = s$.
\end{sitn}

Let $f_1': X_1' \to S'$ and $f_2': X_2' \to S'$ be two thickenings of $f: X \to S$ as an enhanced generically log smooth family, and assume that we have line bundles $\cL_1'$ and $\cL_2'$ with sections $s_1'$ and $s_2'$, both being thickenings of $s \in \cL$. If $\varphi: X_1' \xrightarrow{\cong} X_2'$ is an isomorphism of enhanced generically log smooth families over $f: X \to S$, i.e., $(d\varphi^*,T\varphi_*)$ is compatible with the chosen Gerstenhaber calculi $\G\C^\bullet_{X_i'/S'}$ inside $(\V^\bullet_{X'_i/S'},\W^\bullet_{X_i'/S'})$, and if we are given an isomorphism $\psi: \cL_2' \cong \varphi_*\cL_1'$ with $\psi(s_2') = s_1'$, then we have a commutative diagram 
\[
 \xymatrix{
  X_1'(s_1') \ar[d]^{\varphi_s}_\cong \ar[r] & L'_1(X'_1) \ar[d]^\psi_\cong \ar[r] & L'_1 \ar[r] \ar[d]^\psi_\cong & X'_1 \ar[d]^\varphi_\cong \\
  X_2'(s_2') \ar[r] & L'_2(X'_2) \ar[r] & L'_2 \ar[r] & X'_2 \\
 }
\]
of isomorphisms which induce the identity after restriction to $S$. Note that the isomorphism $\varphi_s$ on the left hand side is an isomorphism of enhanced generically log smooth families, i.e., compatible with the Gerstenhaber calculus $\G\C^\bullet_{X_i'(s_i')/S'}$.

We compute these maps in the case of an inner automorphism of $f': X' \to S'$ over $f: X \to S$, with the same $\cL'$ but allowing two different sections $s_1'$ and $s_2'$. Recall that the automorphisms of $f': X' \to S'$ as a non-enhanced generically log smooth family are classified by log derivations in the kernel $I \cdot \V^{-1}_{X'/S'}$ of the possibly non-surjective map $\V^{-1}_{X'/S'} \to \V^{-1}_{X/S}$. The induced automorphism of the reflexive Gerstenhaber calculus is $\mathrm{exp}_{-\theta}$ where $\theta = (D,\Delta)$ is the relative log derivation. If $\theta \in I \cdot \G^{-1}_{X'/S'} \subseteq I \cdot \V^{-1}_{X'/S'}$, which is the kernel of the surjective map $\G^{-1}_{X'/S'} \to \G^{-1}_{X/S}$, then we say that $\theta$ induces an \emph{inner automorphism} of the enhanced generically log smooth family; in this case, $\G\C^\bullet_{X/S}$ is invariant under the gauge transform $\mathrm{exp}_{-\theta}$.

When we want to take the map $\psi: \cL' \to \cL'$ into account, then we need, in addition to $(D,\Delta)$, an $f'^{-1}(\cO_{S'})$-linear map $u: \cL' \to \cL'$ with $u(a \cdot e) = D(a) \cdot e + a \cdot u(e)$. The classifying space of such triples $(D,\Delta,u)$ fits into the Atiyah extension 
$$0 \to \cO_{X'} \to \V^{-1}_{X'/S'}(\cL') \to \V^{-1}_{X'/S'} \to 0,$$
which is locally split exact. Those triples $(D,\Delta,u)$ in $I \cdot \V^{-1}_{X'/S'}(\cL')$ give the infinitesimal automorphisms of $f': (X',\cL') \to S'$ as a non-enhanced generically log smooth family with a line bundle. When we consider only such $(D,\Delta,u)$ with $(D,\Delta) \in \G^{-1}_{X'/S'}$, then we obtain the modified Atiyah extension 
$$0 \to \cO_{X'} \to \G^{-1}_{X'/S'}(\cL') \to \G^{-1}_{X'/S'} \to 0,$$
which is locally split exact as well. The triples $(D,\Delta,u)$ in $I \cdot \G^{-1}_{X'/S'}(\cL')$ are precisely those with $(D,\Delta) \in I \cdot \G^{-1}_{X'/S'}$, and they give rise to an (inner) automorphism of $f': (X',\cL') \to S'$ as an enhanced generically log smooth family with a line bundle.

For $(D,\Delta,u) \in \G^{-1}_{X'/S'}(\cL')$, we have, by definition,  
$$\phi: \enspace \cO_{X'} \to \cO_{X'}, \quad a \mapsto \sum_{n = 0}^\infty \frac{D^n(a)}{n!} = a + D(a) + \frac{1}{2}D^2(a) + ... \ ,$$
as well as $\Phi: \M_{U'} \to \M_{U'}$ defined by the more complicated formula in Chapter~\ref{inf-auto}. On $\cL'$, we have 
$$\psi: \enspace \cL' \to \cL', \quad e \mapsto \sum_{n = 0}^\infty \frac{u^n(e)}{n!} = e + u(e) + \frac{1}{2}u^2(e) + ... \ . $$
We use the notations from Chapter~\ref{log-prereg-sec} on $f': X' \to S'$, i.e., $e_\alpha$ is a local trivializing section of $\cL'$ etc. When we define $w_\alpha \in \cO_{X'}$ by $u(e_\alpha) = w_\alpha \cdot e_\alpha$, then we have $\psi(e_\alpha) = v_\alpha \cdot e_\alpha$ with 
$$v_\alpha = 1 + \sum_{n = 1}^\infty \frac{[D + w_\alpha]^{n - 1}(w_\alpha)}{n!} = 1 + w_\alpha + \frac{D(w_\alpha) + w_\alpha^2}{2} + ... \ .$$
On the level of $\cO_{L'}$, we have $\psi: R_\alpha[x_\alpha] \to R_\alpha[x_\alpha]$ with $\psi(a) = \phi(a)$ for $a \in R_\alpha$ and $\psi(x_\alpha) = v_\alpha^{-1} \cdot x_\alpha$. For the embedding $\eps_o: p'^*\cL' \to \cO_{L'}, \: p'^*e_\alpha^\vee \mapsto x_\alpha$, we find that it is compatible with $\psi$ when we map 
$$p'^*\cL' \to p'^*\cL', \quad p'^*e_\alpha^\vee \mapsto v_\alpha^{-1} \cdot p'^*e_\alpha^\vee,$$
as is natural. Thus, we have an isomorphism of Deligne--Faltings structures, and hence an isomorphism $\psi: \underline L'(X') \to \underline L'(X')$, yielding the isomorphism $\psi: L'(X') \cong L'(X')$ over $\psi: L' \to L'$.

Let us write $s_1' = s_{\alpha;1} \cdot e_\alpha$ and $s_2' = s_{\alpha;2} \cdot e_\alpha$. Since $\psi(s_2') = s_1'$, we find $\phi(s_{\alpha;2}) = v_\alpha^{-1} \cdot s_{\alpha;1}$. This confirms that we have a commutative diagram 
\[
 \xymatrixcolsep{4em}\xymatrix{
  X'(s_1') \ar[d]_{v_\alpha^{-1} \cdot s_{\alpha;1} \leftmapsto s_{\alpha;2}} \ar[r]^{s_{\alpha;1} \leftmapsto x_\alpha} & L'(X') \ar[d]^{v_\alpha^{-1} \cdot x_\alpha \leftmapsto x_\alpha} \\
  X'(s_2') \ar[r]^{s_{\alpha;2} \leftmapsto x_\alpha} & L'(X'), \\
 }
\]
and, because the log structure is defined as the inverse image from $L'(X')$, a compatible isomorphism $\varphi_s: X'(s_1') \cong X'(s_2')$ of log schemes (and in fact enhanced generically log smooth families since our construction is functorial on that level). This allows us to construct both isomorphisms between different choices of $s'$ and automorphisms of each $X'(s')$.

First, we show that, locally, $X'(s_1') \cong X'(s_2')$ for any two choices $s_1'$ and $s_2'$ with $s_i'|_A = s$. In other words, up to (non-unique local) isomorphism, $X'(s')$ is independent of the choice of $s'$.

\begin{lemma}\label{s-independent}\note{s-independent}
 Let $s_1'$ and $s_2'$ be two sections of $\cL'$ with $s_i'|_A = s$. Then there is locally an element $\theta = (D,\Delta,u) \in I \cdot \G^{-1}_{X'/S'}(\cL')$ with $\psi(s_2') = s_1'$.
\end{lemma}
\begin{proof}
 First, we show the statement for a small extension, so assume that $I^2 = 0$, and more precisely that $I = (i)$ with $i^2 = 0$. We want to have $\psi(s_2') = s_1'$, which is equivalent to 
 $$s_{\alpha;2} + D(s_{\alpha;2}) = (1 - w_\alpha) \cdot s_{\alpha;1},$$
 where $u(e_\alpha) = w_\alpha \cdot e_\alpha$. Let $s_{\alpha;2} = s_{\alpha;1} + i \cdot x$, and 
 let $(D',\Delta') \in I \cdot \G^{-1}_{X'/S'}$ with $D'(s_{\alpha;2}) = 1 + y \cdot s_{\alpha;2}$, which exists by Lemma~\ref{special-log-der}. After setting $(D,\Delta) := -i \cdot x \cdot (D',\Delta')$, the above equation becomes 
 $$i \cdot x \cdot y \cdot s_{\alpha;2} = w_\alpha \cdot s_{\alpha;2}$$
 since $w_\alpha \cdot s_{\alpha;1} = w_\alpha \cdot s_{\alpha;2}$. Let $u'$ be arbitrary such that $(D,\Delta,u') \in I \cdot \G^{-1}_{X'/S'}(\cL')$, and let $w_\alpha'$ be such that $u'(e_\alpha) = w_\alpha' \cdot e_\alpha$. Then setting $u := u' + (i \cdot x \cdot y - w_\alpha')$ yields the desired element $(D,\Delta,u) \in I \cdot \G^{-1}_{X'/S'}(\cL')$. In the general case, we apply an inductive argument over small extensions. When $A' = A_n \to ... \to A_0 = A$ is a decomposition into small extensions, then we use the flatness and compatibility with base change of $\G^{-1}_{X'/S'}$ to lift the constructed isomorphism from $A_k$ to $A_{k + 1}$, thus producing the situation that $s_1'|_{A_k} = s_2'|_{A_k}$.
\end{proof}
\begin{rem}
 Note that this proof does not work for $\V^{-1}_{X'/S'}$ since the base change property may be violated, i.e., we may have constructed an isomorphism in a step $A_{k + 1} \to A_k$ which does not lift to $A'$.
\end{rem}

After we have established that there is, up to an isomorphism which can be explicitly constructed, locally only one $X'(s')$, we study the automorphisms of $f': X'(s') \to S'$.

\begin{lemma}\label{s-invariant-der}\note{s-invariant-der}
 Let $\theta = (D,\Delta,u) \in I \cdot \G^{-1}_{X'/S'}(\cL')$. Then we have $\psi(s') = s'$ if and only if $u(s') = 0$.
\end{lemma}
\begin{proof}
 We have $\psi(s') = \sum_{n = 0}^\infty \frac{u^n(s')}{n!}$, so $\psi(s') = s'$ is equivalent to $\sum_{n = 1}^\infty \frac{u^n(s')}{n!} = 0$. If $u(s') = 0$, then obviously $\psi(s') = s'$. Conversely, assume that $u(s') \not= 0$ but $\psi(s') = s'$. Then there is some $m > 0$ with $u(s') \in I^m \cdot \cL'$ but $u(s') \notin I^{m + 1} \cdot \cL'$. However, $u^2(s') \in I^{m + 1} \cdot \cL'$ because $(D,\Delta,u) \in I \cdot \G^{-1}_{X'/S'}(\cL')$. Now $\sum_{n = 1}^\infty \frac{u^n(s')}{n!} = 0$ implies $u(s') \in I^{m + 1} \cdot \cL'$, a contradiction.
\end{proof}

In view of Lemma~\ref{s-invariant-der}, we write 
$$\G^{-1}_{X'/S'}(\cL',s') := \{(D,\Delta,u) \in \G^{-1}_{X'/S'}(\cL') \ | \ u(s') = 0\},$$
and analogously for $\V^{-1}_{X'/S'}(\cL',s')$. Thus, the automorphisms of $f': X'(s') \to S'$ which come from $X'$ are in $I \cdot \G^{-1}_{X'/S'}(\cL') \cap \G^{-1}_{X'/S'}(\cL',s')$.\footnote{At this point, we do not yet know that $I \cdot \G^{-1}_{X'/S'}(\cL') \cap \G^{-1}_{X'/S'}(\cL',s') = I \cdot \G^{-1}_{X'/S'}(\cL',s')$; however, we will see this below. The problem is that we do not know here that the injectivity of $\G^{-1}_{X'/S'}(\cL',s') \to \G^{-1}_{X'/S'}(\cL')$ is preserved under base change.} A priori, they are \emph{outer} automorphisms of $f': X'(s') \to S'$, but we will show now that they are in fact inner automorphisms, i.e., we construct a map 
$$\mu: I \cdot \G^{-1}_{X'/S'}(\cL') \cap \G^{-1}_{X'/S'}(\cL',s') \to I \cdot \G^{-1}_{X'(s')/S'}$$
such that $\psi_s: X'(s') \cong X'(s')$ is the gauge transform induced by $\mu(D,\Delta,u)$.

First, we have to compute the gauge transform on $(\V^\bullet_{L'(X')/S'},\W^\bullet_{L'(X')/S'})$ induced by $\psi: L'(X') \cong L'(X')$. In order to do so, we start by comparing the Atiyah extension with the exact sequence \eqref{Gamma-LXS-seq}. Unfortunately, we have not found a direct functorial comparison map, and we have to construct it explicitly on charts.

\begin{lemma}\label{Atiyah-comp-map}\note{Atiyah-comp-map}
 We have an isomorphism 
 \[
  \xymatrix{
   0 \ar[r] & p'^*\cO_{X'} \ar[r] \ar[d]^{1 \mapsto -\gamma_s^\vee} & p'^*\V^{-1}_{X'/S'}(\cL') \ar[r] \ar[d]^c & p'^*\V^{-1}_{X'/S'} \ar[r] \ar@{=}[d] & 0 \\
  0 \ar[r] & \Theta^1_{L'(X')/X'} \ar[r] & \V^{-1}_{L'(X')/S'} \ar[r] & p'^*\V^{-1}_{X'/S'} \ar[r] & 0 \\
  }
 \]
 of short exact sequences.\footnote{The map $1 \mapsto -\gamma_s^\vee$ rather than $1 \mapsto \gamma_s^\vee$ is, though somewhat surprising, not a mistake.} This isomorphism is compatible with $\G^{-1}$, i.e., when replacing $\V^{-1}$ with $\G^{-1}$, we still have an isomorphism of short exact sequences.
\end{lemma}
\begin{proof}
 Given $(D,\Delta) \in \V^{-1}_{X'/S'}$ on $V_\alpha$, we can find a unique $u_\alpha$ such that $u_\alpha(e_\alpha) = 0$ and $(D,\Delta,u_\alpha) \in \V^{-1}_{X'/S'}(\cL')$. Namely, start with an arbitrary $u$, and then set $u_\alpha(e) := u(e) - w_\alpha \cdot e$, where $u(e_\alpha) = w_\alpha \cdot e_\alpha$. This defines a local splitting 
$$K: \enspace \V^{-1}_{X'/S'} \to \V^{-1}_{X'/S'}(\cL')$$
of the Atiyah extension. Obviously, it splits the Atiyah extension for $\G^{-1}$ instead of $\V^{-1}$ as well. If $K': \V^{-1}_{X'/S'} \to \V^{-1}_{X'/S'}(\cL')$ is the analogous splitting on $V_\beta$, then we have 
$$K'(D,\Delta) - K(D,\Delta) = \Delta(\gamma_{\alpha\beta})$$
on overlaps $V_\alpha \cap V_\beta$, where $\gamma_{\alpha\beta}$ is the transition function with $e_\alpha = \gamma_{\alpha\beta}e_\beta$, and we consider $\Delta(\gamma_{\alpha\beta})$ as the multiplication operator on $\cL'$, i.e., embedded via $\cO_{X'} \to \V^{-1}_{X'/S'}(\cL')$.

On the other hand side, we also have our original distinguished splitting $B_1$ of \eqref{W1LXS-seq}. Let $B_1$ be the splitting on $V_\alpha$, and let $B'_1$ be the splitting on $V_\beta$. Let 
$$S_{-1},\: S_{-1}': \enspace p'^*\V^{-1}_{X'/S'} \to \V^{-1}_{L'(X')/S'}$$
be the two induced splittings of \eqref{Gamma-LXS-seq}. Then a slightly tedious direct computation shows that 
$$S_{-1}'(p'^*\theta) - S_{-1}(p'^*\theta) = -\Delta(\gamma_{\alpha\beta}) \cdot \delta^\vee$$
for $\theta = (D,\Delta) \in \V^{-1}_{X'/S'}$. Now we define the comparison map 
$$c: \enspace p'^*\V^{-1}_{X'/S'}(\cL') \to \V^{-1}_{L'(X')/S'}, \quad a + b \cdot p'^*K(D,\Delta) \mapsto -a \cdot \delta^\vee + b \cdot S_{-1}(p'^*(D,\Delta)).$$
The above formulae for $K$ and $K'$ respective $S_{-1}$ and $S_{-1}'$ show that $c$ is well-defined. The same computation also shows that $c$ is independent of the choice of $e_\alpha$ as a local trivialization of $\cL'$, and hence canonical. It is obvious that the diagram in the statement above commutes with this definition of $c$. Since the maps on the left and on the right are isomorphisms, the map in the middle is an isomorphism as well. By construction, $c$ maps $p'^*\G^{-1}_{X'/S'}(\cL')$ into $\G^{-1}_{L'(X')/S'}$. Since we still have a morphism of short exact sequences with isomorphisms on the sides, $c$ restricted to $p'^*\G^{-1}_{X'/S'}(\cL')$ is an isomorphism as well, mapping onto $\G^{-1}_{L'(X')/S'}$.
\end{proof}

With this comparison map $c$, we get the expected gauge transform.

\begin{lemma}
 Let $\theta = (D,\Delta,u) \in I \cdot \V^{-1}_{X'/S'}(\cL')$. Then the gauge transform on $\V\,\W^\bullet_{L'(X')/S'}$ induced by $\psi: L'(X') \to L'(X')$ is given by $\mathrm{exp}_{-c(p'^*\theta)}$. In particular, if $\theta \in I \cdot \G^{-1}_{X'/S'}(\cL')$, then the gauge transform preserves the Gerstenhaber subcalculus $\G\C^\bullet_{L'(X')/S'}$.
\end{lemma}
\begin{proof}
 We write $\psi: L'(X') \to L'(X')$ as above for the morphism induced by $(D,\Delta,u)$. First, we show that $\psi^*(b) = \mathrm{exp}_{-c(p'^*\theta)}(b)$ for functions $b \in \cO_{L'}$. For a function $a \in \cO_{X'}$, we have $[-c(p'^*\theta),p'^*a] = p'^*D(a)$. Thus, we have 
 $$\psi^*(p'^*a) = p'^*\varphi^*(a) = \sum_{n = 0}^\infty \frac{p'^*D^n(a)}{n!} = \mathrm{exp}_{-c(p'^*\theta)}(p'^*a).$$
 Since $[S_{-1}(p'^*(D,\Delta)),x_\alpha] = -S_{-1}(p'^*(D,\Delta)) \ \invneg \ (x_\alpha \cdot \delta_\alpha) = 0$, we find 
 $$[-c(p'^*\theta),x_\alpha] = -p'^*(w_\alpha) \cdot x_\alpha$$
 where $w_\alpha$ is such that $u(e_\alpha) = w_\alpha \cdot e_\alpha$. This shows 
 $$\mathrm{exp}_{-c(p'^*\theta)}(x_\alpha) = \left(1 + \sum_{n = 1}^\infty(-1) \cdot \frac{[D - p'^*w_\alpha]^{n - 1}(p'^*w_\alpha)}{n!}\right) \cdot x_\alpha = v_\alpha^{-1} \cdot x_\alpha,$$
 where the second equality is as in the proof of Lemma~\ref{geom-auto-gauge-trafo-corr}, i.e., $\psi^*(x_\alpha) = \mathrm{exp}_{-c(p'^*\theta)}(x_\alpha)$. Since $\cO_{L'}$ is, as a ring, locally generated by $\cO_{X'}$ and $x_\alpha$, we obtain our claim on $\cO_{L'}$.
 
 From here, we proceed as in the proof of Lemma~\ref{geom-auto-gauge-trafo-corr}. First, we obtain $d\psi^*(\partial a) = \mathrm{exp}_{-c(p'^*\theta)}(\partial a)$ for functions $a \in \cO_{L'}$, and then we have $d^1\psi^* = \mathrm{exp}_{-c(p'^*\theta)}$ first on the strict locus of $f': L'(X') \to S'$ and then everywhere. Both sides are compatible with the $\wedge$-product, so we get the equality in all degrees. Afterward, we get $T^1\psi^* = \mathrm{exp}_{-c(p'^*\theta)}$ from the formula $T^1\psi^*(\xi) \ \invneg \ \alpha = \mathrm{exp}_{-c(p'^*\theta)}(\xi) \ \invneg \ \alpha$, which is obtained from the definition of $T^1\psi^*$ by applying $\mathrm{exp}_{-c(p'^*\theta)}$ in the cases where we already know our claimed formula. Finally, we obtain all degrees $T^p\psi^*$ similarly to $d^i\psi^*$.
\end{proof}

Since 
$$H_{-1}(s'^*c(p'^*\theta)) = (s'^*c(p'^*\theta) \ \invneg \ G_1(e_\alpha^\vee)) \cdot e_\alpha$$
under the map $H_{-1}: s'^*\V^{-1}_{L'(X')/S'} \to \cL'$, a direct computation shows that 
$$H_{-1}(s'^*c(p'^*\theta)) = -(w_\alpha \cdot s_\alpha + D(s_\alpha)) \cdot e_\alpha = -u(s),$$
where $\theta = (D,\Delta,u)$, and $w_\alpha$ is such that $u(e_\alpha) = w_\alpha \cdot e_\alpha$. Thus, we have $H_{-1}(s'^*c(p'^*\theta)) = 0$ if and only if $u(s) = 0$, i.e., if and only if $\theta \in \V^{-1}_{X'/S'}(\cL',s')$. If in fact $\theta \in \G^{-1}_{X'/S'}(\cL')$, then we obtain $s'^*c(p'^*\theta) \in \G^{-1}_{X'(s')/S'}$. Finally, by considering the maps to the variants for $f: X \to S$, we find $s'^*c(p'^*\theta) \in I \cdot \G^{-1}_{X'(s')/S'}$. This defines our map 
$$\mu: I \cdot \G^{-1}_{X'/S'}(\cL') \cap \G^{-1}_{X'/S'}(\cL',s') \to I \cdot \G^{-1}_{X'(s')/S'}, \quad \theta \mapsto s'^*c(p'^*\theta).$$

\begin{lemma}\label{gauge-transform-on-X-s}\note{gauge-transform-on-X-s}
 Let $\theta = (D,\Delta,u) \in I \cdot \V^{-1}_{X'/S'}(\cL') \cap \V^{-1}_{X'/S'}(\cL',s')$. Then the gauge transform on $\V\,\W^\bullet_{X'(s')/S'}$ induced by $\varphi_s: X'(s') \cong X'(s')$ is given by $\mathrm{exp}_{-\mu(\theta)}$. If 
 $$\theta \in I \cdot \G^{-1}_{X'/S'}(\cL') \cap \G^{-1}_{X'/S'}(\cL',s'),$$
 then the gauge transform preserves the Gerstenhaber subcalculus $\G\C^\bullet_{X'(s')/S'}$.
\end{lemma}
\begin{proof}
 Instead of tracking the actions through the construction, we apply the following trick: Recall that $h: X'(s') \to X'$ is an isomorphism over $X' \setminus H'$. In particular, we have a commutative diagram $h \circ \varphi_s = \varphi \circ h$ of isomorphisms on $X' \setminus H'$. Under the identification given by $h$, we have $(d\varphi_s^*,T\varphi_s^*) = (d\varphi^*,T\varphi^*)$; the latter is the gauge transform $\mathrm{exp}_{-\theta}$. Expanding our definition of $\mu$, we find $Th_*\mu(D,\Delta,u) = (D,\Delta)$, so, on $X' \setminus H'$, the automorphism $(d\varphi_s^*,T\varphi_s^*)$ is given by the gauge transform $\mathrm{exp}_{-Th^*(D,\Delta)} = \mathrm{exp}_{-\mu(\theta)}$. Now $(d\varphi_s^*,T\varphi_s^*)$ and $\mathrm{exp}_{-\mu(\theta)}$ are two automorphisms of $\V\,\W^\bullet_{X'(s')/S'}$ which coincide on the scheme-theoretically dense open subset $X' \setminus H'$, so they must be equal due to local freeness on $U'$ and reflexivity of $\V^p$ and $\W^i$.
\end{proof}

As in the proof of Lemma~\ref{gauge-transform-on-X-s}, we consider now the map 
$$Th_*: I \cdot \G^{-1}_{X'(s')/S'} \to I \cdot \G^{-1}_{X'/S'}.$$
By construction, we have $Th_* \circ \mu(D,\Delta,u) = (D,\Delta)$. To exploit this map, we establish the following basic facts about $\V^{-1}_{X'/S'}(\cL',s') \to \V^{-1}_{X'/S'}$. Parts \emph{\ref{sD-us-equality}} and \emph{\ref{svuv-Dsv-equality}} are not needed for our argument, but we found them worth mentioning here.

\begin{lemma}\label{VLs-basics}\note{VLs-basics}
 The following hold:
 \begin{enumerate}[label=\emph{(\alph*)}]
  \item\label{VLs-V-inj} the map $\V^{-1}_{X'/S'}(\cL',s') \to \V^{-1}_{X'/S'}$ is injective;
  \item\label{VLs-V-criterion} $(D,\Delta) \in \V^{-1}_{X'/S'}$ is in the image of the map if and only if $D(s_\alpha) = a \cdot s_\alpha$ for some local function $a \in \cO_{X'}$, i.e., if and only if $D$ preserves the ideal $\cL'^\vee \subseteq \cO_{X'}$ of $H' \subset X'$.
  \item\label{sD-us-equality} if $(D,\Delta,u) \in \V^{-1}_{X'/S'}(\cL',s')$, then $s' \circ D = u \circ s'$ where $s': \cO_{X'} \to \cL'$ is the map induced by the section $s'$;
  \item\label{svuv-Dsv-equality} if $(D,\Delta,u) \in \V^{-1}_{X'/S'}(\cL',s')$, then there is a unique map $u^\vee: \cL'^\vee \to \cL'^\vee$ with $s'^\vee \circ u^\vee = D \circ s'^\vee$, where $s'^\vee: \cL'^\vee \to \cO_{X'}$ is the canonical map associated with the section $s'$; this map $u^\vee$ is $f'^{-1}(\cO_{S'})$-linear and satisfies $u^\vee(a \cdot e^\vee) = D(a) \cdot e^\vee + a \cdot u^\vee(e^\vee)$.
 \end{enumerate}
\end{lemma}
\begin{proof}
 First we show \emph{\ref{VLs-V-inj}}. Let $(D,\Delta,u)$ and $(D,\Delta,u')$ be two elements of $\V^{-1}_{X'/S'}(\cL',s')$. Then there is some $a \in \cO_{X'}$ with $v(e) = u(e) + a \cdot e$.  Since $u(s) = v(s) = 0$, we obtain $a \cdot s = 0$; since each $s_\alpha$ is a non-zero divisor, we get $a = 0$. 
 For \emph{\ref{VLs-V-criterion}}, first observe that, if $(D,\Delta,u) \in \V^{-1}_{X'/S'}(\cL',s')$, then 
 $$0 = u(s_\alpha \cdot e_\alpha) = D(s_\alpha) \cdot e_\alpha + s_\alpha \cdot u(e_\alpha),$$
 so $D(s_\alpha) = a \cdot s_\alpha$ for some $a \in \cO_{X'}$. Conversely, if $(D,\Delta) \in \V^{-1}_{X'/S'}$ with $D(s_\alpha) = a \cdot s_\alpha$, we can find first some locally defined $u: \cL' \to \cL'$ with $(D,\Delta,u) \in \V^{-1}_{X'/S'}(\cL')$. Let $x \in \cO_{X'}$ be such that $u(e_\alpha) = x \cdot e_\alpha$. Then $\hat u(e) := u(e) - (a + x) \cdot e$ defines a preimage of $(D,\Delta)$ since $\hat u(s) = 0$.
 For \emph{\ref{sD-us-equality}}, note that $u(b \cdot s) = D(b) \cdot s + b \cdot u(s) = D(b) \cdot s$. 
 For \emph{\ref{svuv-Dsv-equality}}, a map $u^\vee$ with the desired property must satisfy 
 $$s_0^\vee(u^\vee(b \cdot e_\alpha^\vee)) = b \cdot D(s_\alpha) + s_\alpha \cdot D(b) = s_\alpha \cdot (b \cdot a + D(b)),$$
 where again $D(s_\alpha) = a \cdot s_\alpha$. Thus, we must have $u^\vee(b \cdot e_\alpha^\vee) = (b \cdot a + D(b)) \cdot e_\alpha^\vee$. A direct computation shows that this is independent of the choice of $e_\alpha$ and hence well-defined on overlaps; another direct computation shows the claimed relations.
\end{proof}

From part \emph{\ref{VLs-V-inj}} of Lemma~\ref{VLs-basics} it follows that $\mu$ is injective. Now let $\theta' = (D',\Delta') \in I \cdot \G^{-1}_{X'(s')/S'}$, and let $(D,\Delta) =  Th_*(D',\Delta')$. Then $D = D'$ as a derivation on $\cO_{X'}$, and we have $D(s_\alpha) = D'(s_\alpha) = s_\alpha \cdot \Delta'(m_\alpha)$, i.e., we have some $(D,\Delta,u) \in \V^{-1}_{X'/S'}(\cL',s')$ lifting $(D,\Delta)$ by part \emph{\ref{VLs-V-criterion}} of Lemma~\ref{VLs-basics}. When we look into the actual construction, then we find $(D,\Delta,u) \in I \cdot \G^{-1}_{X'/S'}(\cL') \cap \G^{-1}_{X'/S'}(\cL',s')$, i.e., $(D,\Delta,u)$ is in the domain of $\mu$, and we have $\mu(D,\Delta,u) = \theta'$ since $Th_*$ is injective. Thus, $\mu$ is an isomorphism.

In the same way as $\mu$, we can also define 
$$\tilde\mu: \enspace \G^{-1}_{X'/S'}(\cL',s') \to \G^{-1}_{X'(s')/S'}, \quad \theta \mapsto s'^*c(p'^*\theta).$$
Just as $\mu$, also $\tilde\mu$ is an isomorphism. Because the formation of $\G^{-1}_{X'(s')/S'}$ commutes with base change, so does the formation of $\G^{-1}_{X'/S'}(\cL',s')$. This implies that 
$$I \cdot \G^{-1}_{X'/S'}(\cL') \cap \G^{-1}_{X'/S'}(\cL',s') = I \cdot \G^{-1}_{X'/S'}(\cL',s').$$

\begin{cor}\label{VLs-autom}\note{VLs-autom}
 Every element $\theta = (D,\Delta,u) \in I \cdot \G^{-1}_{X'/S'}(\cL',s')$ induces a (geometric) automorphism of $f' \circ h': X'(s') \to S'$ whose action on the Gerstenhaber calculus is the gauge transform $\mathrm{exp}_{-\mu(\theta)}$, and every inner automorphism of $f' \circ h': X'(s') \to S'$ is of this form for a unique $\theta$.
\end{cor}

\section{Enhanced systems of deformations}

Suppose we are in the following situation:

\begin{sitn}
 Let $f_0: X_0 \to S_0$ be an enhanced generically log smooth family of relative dimension $d$ over $S_0 = \Spec (Q \to \kk)$. Assume that $\varpi^\bullet: \G\C_{X_0/S_0}^\bullet \to \V\,\W^\bullet_{X_0/S_0}$ is injective. Let $\D$ be an enhanced system of deformations of $f_0: X_0 \to S_0$, subordinate to an affine open cover $\V = \{V_\alpha\}$. Let $\cL_0$ be a line bundle on $X_0$, and let $s_0 \in \cL_0$ be a log regular global section. Assume that $\cL_0|_\alpha$ is trivial.
\end{sitn}

We construct an enhanced system of deformations $\D(s_0)$ for $f_0 \circ h_0: X_0(s_0) \to S_0$. Let us write $\cO_{\alpha;A}$ for the structure sheaf on $V_{\alpha;A}$. We choose 
\begin{enumerate}[label=(\arabic*)]
 \item a trivialization $e_{\alpha;0}$ of $\cL_0|_\alpha$.
\end{enumerate}
Then we choose, for every $\alpha$,
\begin{enumerate}[label=(\arabic*),resume]
 \item  sections $s_{\alpha;A} \in \cO_{\alpha;A}$ 
\end{enumerate}
which are compatible with restrictions, and such that $s_{\alpha;0} \cdot e_{\alpha;0} = s|_\alpha$.  We set 
$$\tilde V_{\alpha;A} := V_{\alpha;A}(s_{\alpha;A}),$$
the modification of the local model $V_{\alpha;A}$ by the section $s_{\alpha;A} \in \cO_{\alpha;A}$. 
This gives a collection of generically log smooth deformations of $X_0(s_0)|_\alpha$ which is compatible with base change. Next, for every $\alpha \not= \beta$, we choose a collection of
\begin{enumerate}[label=(\arabic*),resume]
 \item isomorphisms $\gamma^*_{\alpha\beta;A}: \cO_{\beta;A}|_{\alpha\beta} \cong \cO_{\alpha;A}|_{\alpha\beta}$
\end{enumerate}
 which is compatible with restrictions, and which induces the isomorphism\footnote{Now we have $e_\beta = \gamma_{\alpha\beta}e_\alpha$ rather than our previous formula $e_\alpha = \gamma_{\alpha\beta} e_\beta$, which does not fit in well with our conventions about enhanced systems of deformations.} 
$$\gamma^*_{\alpha\beta;0}: (\cO_{X_0}|_\beta)|_{\alpha\beta} \cong (\cL_0|_\beta)|_{\alpha\beta} = (\cL_0|_\alpha)|_{\alpha\beta} \cong (\cO_{X_0}|_\alpha)_{\alpha\beta}$$
---which comes from our choice of $e_{\alpha;0}$---on the central fiber. These isomorphisms must be $\cO$-linear for $\psi_{\alpha\beta;A}: V_{\alpha;A}|_{\alpha\beta} \cong V_{\beta;A}|_{\alpha\beta}$. Additionally, we assume that $(\gamma_{\alpha\beta;A}^*)^{-1} = \gamma_{\beta\alpha;A}^*$. By applying Lemma~\ref{s-independent} inductively along $A_k$, we can find
\begin{enumerate}[label=(\arabic*),resume]
 \item $\theta_{\alpha\beta;A} = (D_{\alpha\beta;A},\Delta_{\alpha\beta;A},u_{\alpha\beta;A}) \in \m_A \cdot \G^{-1}_{\alpha;A}(\cO_{\alpha;A})|_{\alpha\beta}$
\end{enumerate}
with $\mathrm{exp}_{-\theta_{\alpha\beta;A}}(\gamma_{\alpha\beta;A}^*(s_{\beta;A})) = s_{\alpha;A}$ and $\theta_{\alpha\beta;A} = -\theta_{\beta\alpha;A}$ under the identification given by $(\psi^*_{\alpha\beta;A},\gamma^*_{\alpha\beta;A})$, and which are compatible with restrictions. Then we replace $\psi_{\alpha\beta;A}$ with $\hat\psi_{\alpha\beta;A} := \psi_{\alpha\beta;A} \circ \mathrm{exp}(-\theta_{\alpha\beta;A})$, where $\mathrm{exp}(-\theta)$ is the automorphism of enhanced generically log smooth families which acts via $\mathrm{exp}_{-\theta}$ on the Gerstenhaber calculus. The new cocycles are gauge transforms in $\m_A \cdot \G^{-1}_{\alpha;A}|_{\alpha\beta\gamma}$ as well, so we still have an enhanced system of deformations. The new $\hat\psi^*_{\alpha\beta;A}$ together with $\mathrm{exp}_{-\theta_{\alpha\beta;A}} \circ \gamma^*_{\alpha\beta;A}$ yields the comparison isomorphism 
$$\tilde\psi_{\alpha\beta;A}: \enspace \tilde V_{\alpha;A}|_{\alpha\beta} \cong \tilde V_{\beta;A}|_{\alpha\beta}$$
for the enhanced system of deformations $\D(s_0)$. These isomorphisms are compatible with restrictions. By construction, they also preserve the Gerstenhaber calculi. Also by construction, the cocycles are in $\m_A \cdot \G^{-1}_{V_{\alpha;A}(s_{\alpha;A})/S_A}|_{\alpha\beta\gamma}$. Thus, we have indeed constructed an enhanced system of deformations $\D(s_0)$ for $f_0 \circ h_0: X_0(s_0) \to S_0$. It gives rise to a deformation functor 
$$\mathrm{ELD}_{X_0(s_0)/S_0}^{\D(s_0)}: \mathbf{Art}_Q \to \mathbf{Set}.$$

\begin{lemma}
 This deformation functor is, up to canonical isomorphism, independent of the choices that we made in the construction of $\D(s_0)$.
\end{lemma}
\begin{proof}
 Let $\D(s_0)$ and $\D'(s_0)$ be two choices of enhanced systems of deformations which come out of the above construction for different choices of the parameters. Let us denote the local models of the two systems by $\tilde V_{\alpha;A}$ respective $\tilde V'_{\alpha;A}$. Then there is a (non-unique) isomorphism $\bar\chi_\alpha: \tilde V_{\alpha;A} \to \tilde V'_{\alpha;A}$ which comes from Lemma~\ref{s-independent}. An enhanced generically log smooth deformation $\tilde X_A \to S_A$ of $f_0: X_0(s_0) \to S_0$ of type $\D(s_0)$ comes with isomorphisms $\chi_\alpha: \tilde X_A|_\alpha \cong \tilde V_{\alpha;A}$. We define the image of our comparison isomorphism between the deformation functors by taking the enhanced generically log smooth family $\tilde X_A \to S_A$ together with the comparison isomorphisms $\bar\chi_\alpha \circ \chi_\alpha: \tilde X_A|_\alpha \cong \tilde V'_{\alpha;A}$. This is again an enhanced generically log smooth deformation because the two comparison isomorphisms $\tilde\psi_{\alpha\beta;A}$ and $\tilde\psi'_{\alpha\beta;A}$ differ by gauge transforms. Any two choices of $\bar\chi_\alpha$ differ by a gauge transform, so the map is independent of the choice of $\bar\chi_\alpha$ on the level of the deformation functor since the images are equivalent as enhanced generically log smooth deformations. The map is an isomorphism since we can define the analogous map in the opposite direction, and the composition is the identity.
\end{proof}

We can define a similar deformation functor 
$$\mathrm{ELD}_{X_0/S_0}^\D(\cL_0,s_0): \mathbf{Art}_Q \to \mathbf{Set}$$
which classifies enhanced generically log smooth deformations of $f_0: X_0 \to S_0$ of type $\D$ together with a line bundle $\cL$ deforming $\cL_0$, and a section $s \in \cL$ with $s|_0 = s_0$. We have the obvious map 
$$\mathrm{ELD}_{X_0/S_0}^\D(\cL_0,s_0) \to \mathrm{ELD}_{X_0(s_0)/S_0}^{\D(s_0)}$$
of deformation functors. By Lemma~\ref{s-invariant-der} and Corollary~\ref{VLs-autom}, the objects classified by the two deformation functors have the same automorphisms. Then the map is an isomorphism of deformation functors with the argument employed, for example, in Proposition~\ref{LDD-GDefD-iso}. On the other hand side, we have also a forgetful morphism 
$$\mathrm{ELD}_{X_0/S_0}^\D(\cL_0,s_0) \xrightarrow{\Phi} \mathrm{ELD}_{X_0/S_0}^\D(\cL_0) \xrightarrow{\Psi} \mathrm{ELD}_{X_0/S_0}^\D$$
of deformation functors. Here, the middle term denotes enhanced generically log smooth deformations with a line bundle.

\begin{lemma}
 The following statements hold:
 \begin{enumerate}[label=\emph{(\alph*)}]
  \item If $H^1(X_0,\cL_0) = 0$, then $\Phi$ is smooth and, in particular, surjective.
  \item If $H^2(X_0,\cO_{X_0}) = 0$, then $\Psi$ is smooth and, in particular, surjective.
  \item If $H^1(X_0,\cO_{X_0}) = 0$, then $\Psi$ is injective.
  \item If $\cL_0 = (\W^d_{X_0/S_0})^\vee$ (respective $\W^d_{X_0/S_0}$), then $\Psi$ admits a splitting 
  $$\Sigma:\enspace \mathrm{ELD}_{X_0/S_0}^\D \to \mathrm{ELD}_{X_0/S_0}^\D(\cL_0)$$
  given by $\cL_A = (\W^d_{X_A/S_A})^\vee$ (respective $\W^d_{X_A/S_A}$). In particular, $\Psi$ is surjective.
 \end{enumerate}
\end{lemma}

For us, the most interesting case is $\cL_0 = (\W^d_{X_0/S_0})^\vee$ since, then, $f_0 \circ h_0: X_0(s_0) \to S_0$ is log Calabi--Yau by Proposition~\ref{WdXs-comp}.

\begin{cor}\label{original-modif-unob-equiv}\note{original-modif-unob-equiv}
 Let $\cL_0 = (\W^d_{X_0/S_0})^\vee$. If $H^1(X_0,\cL_0) = 0$ and either $H^1(X_0,\cO_{X_0}) = 0$ or $H^2(X_0,\cO_{X_0}) = 0$, then $\mathrm{ELD}_{X_0/S_0}^\D$ is unobstructed if and only if $\mathrm{ELD}_{X_0(s_0)/S_0}^{\D(s_0)}$ is unobstructed. If we have only $H^1(X_0,\cL_0) = 0$, then the unobstructedness of $f_0 \circ h_0: X_0(s_0) \to S_0$ still implies the unobstructedness of $f_0: X_0 \to S_0$.
\end{cor}

\begin{ex}
 Let $f_0: X_0 \to S_0$ be a vertical enhanced generically log smooth family which is log Fano, i.e., $\omega_{X_0/S_0}^\vee = (\W^d_{X_0/S_0})^\vee$ is an ample line bundle. If the Gorenstein scheme $X_0$ satisfies Kodaira vanishing for the canonical bundle, then the conditions of the corollary are satisfied. This is, for example, the case if $X_0$ has semi-log-canonical singularities, cf.~\cite[Cor.~6.6]{KSS2010} and the dual statement in the sense of \cite[III, Thm.~7.6]{Hartshorne1977}.\footnote{While Kodaira vanishing in the sense of $H^q(X,L^{-1}) = 0$ for $0 \leq q \leq d - 1$ and an ample line bundle $L$ needs that $X$ is Cohen--Macaulay, Fujino has given a generalization of the dual statement $H^q(X,\omega_X \otimes L) = 0$ for $q \geq 1$ without the Cohen--Macaulay assumption on $X$, see \cite[Thm.~1.8]{Fujino2014} and  \cite[Thm.~1.3]{fujino2015kodaira}.} In particular, this is the case for a toroidal crossing space $(V,\cP,\bar\rho)$ by Lemma~\ref{slc-sing-tor-cr}.
\end{ex}

\section{Log regular sections in practice}

We start with an elementary result that helps us to study log regularity of sections $s \in \cL$ on local models.

\begin{lemma}
 Let $f: X \to S$ be an enhanced generically log smooth family, and let $g: X' \to X$ be an \'etale cover. Then $s \in \cL$ is log regular if and only if $g^*s \in g^*\cL$ is log regular.
\end{lemma}
\begin{proof}
 Injectivity of $\cL^\vee \to \cO_X$, flatness of $\cO_H$ over the base (\cite[0584]{stacks}), and having codimension $\geq 2$ for $H \cap Z \subseteq H$ are all properties which are equivalent on $X$ and $X'$ under an \'etale cover $g: X' \to X$, so $s \in \cL$ is log pre-regular if and only if $g^*s \in g^*\cL$ is log pre-regular. The formation of the sequence in Lemma~\ref{log-qreg-crit} commutes with the faithfully flat map $g: X' \to X$, so it is locally split exact on $X$ if and only if it is locally split exact on $X'$. Thus, $s \in \cL$ is log quasi-regular if and only if $g^*s \in g^*\cL$ is log quasi-regular. Since also the formation of the map $s^*\G^{-1}_{L(X)/S} \to \cL$ commutes with $g: X' \to X$, the claim follows.
\end{proof}

Next, we show that the construction of \cite[Lemma~6.10]{FFR2021} yields log regular sections in our sense, so our theory is applicable in that setting. Let $g: H \to S$ be a torsionless enhanced generically log smooth family. We obtain an enhanced generically log smooth family $f: X \to S$ by setting $X = H \times \bAA^1$; when $q: H \times \bAA^1 \to H$ is the projection, then the two-sided Gerstenhaber calculus is given by 
$$\A^i_{X/S} = q^*\A^i_{H/S} \oplus q^*\A^{i - 1}_{H/S} \wedge dx \quad \mathrm{and} \quad \G^p_{X/S} = q^*\G^p_{H/S} \oplus q^*\G^{p + 1}_{H/S} \wedge \partial_x,$$ where $x$ is the variable in $\bAA^1$. Now we set $\cL = \cO_X$ and $s = x$. 

\begin{lemma}\label{log-reg-constr}\note{log-reg-constr}
 Consider the enhanced generically log smooth family $f: H \times \bAA^1_x \to S$. Then $x \in \cO_X =: \cL$ is a log regular section.
\end{lemma}
\begin{proof}
 This section is obviously log pre-regular, and we have $X(s) = H \times A_\NN$. The first exact sequence for $s^*\W^1_{L(X)/S}$ becomes 
$$0 \to q^*\W^1_{H/S} \oplus \cO_X \cdot dx \to q^*\W^1_{H/S} \oplus \cO_X \cdot dx \oplus \cO_X \cdot s^*\delta_\alpha \to \cO_X \cdot s^*\gamma_s \to 0.$$
Then the second sequence is 
$$0 \to \cL^\vee \xrightarrow{e_\alpha^\vee \mapsto x \cdot s^*\delta_\alpha - dx} q^*\W^1_{H/S} \oplus \cO_X \cdot dx \oplus \cO_X \cdot s^*\delta_\alpha \xrightarrow{R_1}  q^*\W^1_{H/S} \oplus \cO_X \cdot \frac{dx}{x} \to 0.$$
The map on the left is obviously injective, so we already know that the sequence is exact on the left and in the middle from its construction.
We have $R_1(dx) = x \cdot \frac{dx}{x}$ and $R_1(s^*\delta_\alpha) = \frac{dx}{x}$, and on $q^*\W^1_{H/S}$, the map is the identity; in particular, $R_1$ is surjective, the sequence is locally split, and $s \in \cL$ is log quasi-regular. The dual exact sequence is 
$$0 \to q^*\V^{-1}_{H/S} \oplus \cO_X \cdot x\partial_x \to q^*\V^{-1}_{H/S} \oplus \cO_X \cdot \partial_x \oplus \cO_X \cdot s^*\delta^\vee \xrightarrow{H_{-1}} \cL \to 0.$$
Here, $H_{-1}(\partial_x) = e_\alpha$ and $H_{-1}(s^*\delta^\vee) = x \cdot e_\alpha$. Thus, $H_{-1}$ remains surjective even after restricting to $s^*\G^{-1}_{L(X)/S} \subseteq s^*\V^{-1}_{L(X)/S}$, so that $s \in \cL$ is indeed log regular.
\end{proof}

At least in the log smooth case, a converse is true: When $f: X \to S$ is log smooth, and $s \in \cL$ a log regular section, then the zero locus $H$ of $s$ is log smooth as well; then both $X$ and $H \times \bAA^1$ have the same local models around points $x \in H$, so $f: X \to S$ together with $s \in \cL$ arises \'etale locally from this construction. We do not know if the converse is true in general.

\begin{defn}\label{log-transversal-defn}\note{log-transversal-defn}\index{section!log transversal}
 In Situation~\ref{log-modif-sitn}, we say that a log regular section $s \in \cL$ is \emph{log transversal} if $f: X \to S$ together with $s \in \cL$ arises from the above construction \'etale locally around every point $x \in H$.\footnote{We leave it to the reader to construct the two-sided Gerstenhaber calculus on $H \subset X$ to turn it into an enhanced generically log smooth family so that the comparison between $X$ and $H \times \bAA^1$ makes sense in that case.}
\end{defn}

In the log transversal case, if $H \to S$ is log toroidal, then both $f: X \to S$ and $f \circ h: X(s) \to S$ are log toroidal.\footnote{It is probably also true that, if $f: X \to S$ is log toroidal in this case, then also $H \to S$ is log toroidal.}

\begin{opbm}
 Find an example of $f: X \to S$ with a log regular section $s \in \cL$ which is not log transversal, or show that log regularity implies log transversality.
\end{opbm}

\vspace{\baselineskip}

Here are two examples which are not log regular.

\begin{ex}
 Let $f_0: X_0 \to S_0$ be the central fiber of $xy = tz$. Let $\cL_0 = \cO_{X_0}$, and let $s = z$. Then $s \in \cL$ is not even log pre-regular since $H \cap Z$ has codimension $1$ in $H$. Also the criterion in Lemma~\ref{log-qreg-crit} is violated. Namely, we have $s^*\W^1_{L_0(X_0)/S_0} = \W^1_{X_0/S_0} \oplus \cO_{X_0} \cdot s^*\delta_\alpha$. The map $G_1: \cL^\vee \to s^*\W^1_{L_0(X_0)/S_0}, \: e_\alpha^\vee \mapsto z \cdot s^*\delta_\alpha - dz$, can be computed explicitly since we know $\W^1_{X_0/S_0}$ explicitly by \cite[Prop.~7.3]{FFR2021}. It turns out that $G_1$ is injective, but the cokernel is not reflexive, so the map $R_1$ in Lemma~\ref{log-qreg-crit} is not surjective.
\end{ex}

\begin{ex}
 Let $f_0: X_0 \to S_0$ be the central fiber of $xy = tzw$ (see also Example~\ref{xy-tzw-example}), let $\cL = \cO_{X_0}$, and $s = z + w$. The family $f_0: X_0 \to S_0$ is log toroidal. We have $H_0 = \Spec \kk[x,y,z,w]/(xy,z + w)$. Since $H_0 \cap Z = \{0\}$, the section $s \in \cL$ is log pre-regular. The derivations of the family $f: X \to S$ given by $xy = tzw$ can be computed as the kernel $\T_{X/S} = \Theta^1_{X/S}$ of the matrix 
 $$R = \begin{pmatrix}
        y & x & -tz & -tw \\
       \end{pmatrix}: \enspace \cO_X^{\oplus 4} \to \cO_X.$$
 When $\theta \in \T_{X/S}$, then $\theta \ \invneg \ (dz + dw)$ is the image of the matrix 
 $$M = \begin{pmatrix}
        0 & 0 & 1 & 1 \\
       \end{pmatrix}: \enspace \T_{X/S} \to \cO_X.$$
 Thus, $dz + dw = M^\vee(1)$ for the dual matrix $M^\vee: \cO_X \to \W^1_{X/S} = \cH om(\T_{X/S},\cO_X)$. This allows us to compute $dz + dw \in \W^1_{X_0/S_0}$ as the image $M_0^\vee(1)$, and then we can compute the map 
 $$G_1: \enspace \cL^\vee \to s^*\W^1_{L_0(X_0)/S_0} = \cO_{X_0} \cdot s^*\delta_\alpha \oplus \W^1_{X_0/S_0}, \quad e_\alpha^\vee \mapsto (z + w) \cdot s^*\delta_\alpha - d(z + w),$$
 explicitly. It turns out that $G_1$ is injective, and the cokernel is reflexive. Thus, the sequence in Lemma~\ref{log-qreg-crit} is exact. However, the exact sequence is not locally split. Namely, if this were the case, then the dual $G_1^\vee$ would be surjective, which is not the case. Thus, $s = z + w \in \cO_{X_0}$ is not log (quasi-)regular.
\end{ex}

Finally, we do not yet know in general under which conditions global log regular sections exist.

\begin{opbm}
 Give criteria on $f_0: X_0 \to S_0$ for the existence of a log regular section $s_0 \in \cL_0$. For example, is it true that a global log regular section $s_0 \in \cL_0$ exists if $f_0: X_0 \to S_0$ is log smooth, and $\cL_0$ is very ample? How to deal with the ample case when $\cL_0$ is not necessarily globally generated?
\end{opbm}

\section{Smoothing normal crossing spaces}\label{smoothing-nc-sec}\note{smoothing-nc-sec}

As an application of the theory developed in this chapter, we discuss a variant of \cite[Thm.~1.1]{FFR2021}, the original motivation to develop the theory in this chapter. Recall that a normal crossing space $V$ carries a canonical structure of a toroidal crossing space $(V,\cP,\bar\rho)$, that we have a natural injection $\eta: \cL\cS_V \to \T^1_V$ into the first tangent sheaf $\T^1_V = \E xt^1(\Omega^1_V,\cO_V)$, and that the latter is a line bundle on the double locus $D = V_{sing}$. The normal crossing space $V$ admits a filtration according to the rank of $\cP$, which is discussed in Chapter~\ref{toroidal-cr-sp-sec}.

\begin{thm}[Smoothing normal crossing spaces]\label{smoothing-nc-spaces}\note{smoothing-nc-spaces}\index{normal crossing space!smoothing}\index{Gross--Siebert type!standard Gross--Siebert type}
 Let $V/\CC$ be a proper normal crossing space such that every open stratum $S^\circ \in [\cS_kV]$ is quasi-projective for every $k \geq 0$. Assume that $\omega_V^\vee$ is globally generated as a line bundle on $V$, and that $\T^1_V$ is globally generated as a line bundle on $D = V_{sing}$. Then there is a closed subset $Z \subseteq V$, a section $s \in \Gamma(V \setminus Z,\cL\cS_V)$, and a section $e_0 \in \Gamma(V,\omega_V^\vee)$ such that:
 \begin{enumerate}[label=\emph{(\roman*)}]
  \item $(V,Z,s)$ is a well-adjusted triple;
  \item the associated generically log smooth family $f_0: X_0 \to S_0$ is log toroidal of standard Gross--Siebert type;
  \item $e_0 \in \omega_V^\vee$ is a log regular section;
  \item the modification $g_0: X_0(e_0) \to S_0$ is log Calabi--Yau and log toroidal, and every deformation in $\D(e_0)$ is log toroidal as well, where $\D$ is the system of log toroidal deformations of standard Gross--Siebert type for $f_0: X_0 \to S_0$, and $\D(e_0)$ is its modification.
 \end{enumerate}
 In particular, $V$ admits a formal smoothing. If $V$ is projective, then $V$ admits an algebraic smoothing $f: X \to \Spec \CC\llbracket t\rrbracket$.
\end{thm}

The local models of $f_0: X_0 \to S_0$ will be given by 
$$M(r;q) = \Spec \CC[x_0,x_1,...,x_r,t,u_1,...,u_q]/(x_0 \cdot ... \cdot x_r - t\cdot u_1)$$
for $r \geq 1$ and $q \geq 1$. This is indeed a local model of standard Gross--Siebert type. Namely, we choose $M'_\RR = \RR^r$, the standard simplex 
$$\tau = \mathrm{Conv}(0,e_1,...,e_r) \subseteq M'_\RR,$$
and $\Delta_1 = \tau$, $\Delta_2 = ... = \Delta_q = 0$. Then we have 
$$\check\psi_0(n) = \check \psi_1(n) = -\mathrm{inf}\{0,n_1,...,n_r\}$$
and $\check\psi_2 = ... = \check\psi_q = 0$. Thus, for the monoid 
$$P = \{n + a_0e_0^* + \sum_{j = 1}^q a_je_j^* \ | \ a_i \geq \check\psi_i(n) \},$$
the monoid algebra $\CC[P]$ is generated by 
$$x_0 = z^{(-1,...,-1;1;1,0,...,0)}, \enspace x_i = z^{(0,...,1,...,0;0;0, ...,0)}, \enspace t = z^{(0,...,0;1;0,...,0)}, \enspace u_i = z^{(0,...,0;0;0,...,1,...,0)}.$$
Finally, it is easy to check that $\Delta_+$ is a standard simplex. 

Let $M_0 = M(r;q) \times_{\bAA^1} \Spec \CC$ be the central fiber. The exact sequence 
$$0 \to \cO_{M_0} \xrightarrow{1 \mapsto dt} \cO_{M_0}\langle dx_0,...,dx_r\rangle \to \Omega^1_{M_0} \to 0$$
of the usual log smooth deformation $x_0 \cdot ... \cdot x_r = t$ defines a surjection $\cO_{M_0} \to \T^1_{M_0}$. The image is the reference section $s_0 \in \cL\cS_{M_0} \subseteq \T^1_{M_0}$. Now a direct computation shows that the extension class of $x_0 \cdot ... \cdot x_r = tu_1$ is $s = u_1 \cdot s_0$, so outside the zero locus of $s$, this is also the class of the induced log structure.

Using the reference log structure on $M_0$, we get 
$$\Omega^{r + q}_{M_0/S_0} = \cO_{M_0} \cdot \frac{dx_1}{x_1} \wedge ... \wedge \frac{dx_r}{x_r} \wedge du_1 \wedge ... \wedge du_q =: \cO_{M_0} \cdot \eps_0$$
Since $\omega_{M_0}$ is independent of the chosen log structure (and equal to the canonical bundle of the underlying Gorenstein scheme), we find $\omega_{M_0} = \cO_{M_0} \cdot \eps_0$, and thus $\omega_{M_0}^\vee = \cO_{M_0} \cdot \eps_0^\vee$. From Lemma~\ref{log-reg-constr}, we see that $u_2\cdot \eps_0^\vee$ is a log regular (indeed log transversal) section of $\omega_{M_0}^\vee$ (of course assuming $q \geq 2$ now). This will be the local model for the modified and log Calabi--Yau family $g_0: X_0(e_0) \to S_0$.

\begin{lemma}\label{Vse-transv}\note{Vse-transv}
 The normal crossing space $M_0$ together with $s = u_1 \cdot s_0 \in \T^1_{M_0}$ and $e_0 = u_2 \cdot \eps_0^\vee \in \omega_{M_0}^\vee$ has the following properties:
 \begin{enumerate}[label=\emph{(\roman*)}]
  \item on every open stratum $S^\circ \in [\cS_kM_0]$ for $k \geq 0$, the intersection $E \cap S^\circ \subset S^\circ$ is a smooth effective Cartier divisor for $E = \mathrm{div}(e_0)$;
  \item on every open stratum $S^\circ \in [\cS_kM_0]$ for $k \geq 1$, the intersection $Z \cap S^\circ \subset S^\circ$ is a smooth effective Cartier divisor for $Z = \mathrm{div}(s)$;
  \item for every open stratum $S^\circ \in [\cS_kM_0]$ for $k \geq 1$, the intersection $Z \cap E \cap S^\circ$ is a smooth effective Cartier divisor both on $Z \cap S^\circ$ and on $E \cap S^\circ$.
 \end{enumerate}
\end{lemma}
\begin{proof}
 All statement are clear.
\end{proof}
\begin{defn}
 We say that a triple $(V,s,e_0)$ with a normal crossing space $V$, a section $s \in \T^1_V$, and a section $e_0 \in \omega_V^\vee$ is \emph{well-formed} if the statements of Lemma~\ref{Vse-transv} hold.
\end{defn}

In the situation of Theorem~\ref{smoothing-nc-spaces}, we construct a well-formed triple $(V,s,e_0)$. Note the following version of Bertini's theorem, which follows from the method of the proof of \cite[III, Cor.~10.9]{Hartshorne1977}.

\begin{thm}[Bertini]\index{Bertini's theorem}
 Let $X/\CC$ be a smooth quasi-projective variety, let $\cL$ be a line bundle on $X$, and let $\mathfrak{d} \subseteq H^0(X,\cL)$ be a finite-dimensional linear system without base points, i.e., $\mathfrak{d} \otimes_\CC \cO_X \to \cL$ is surjective. Then there is a dense Zariski open subset $\mathfrak{u} \subseteq \mathfrak{d}$ such that $\mathrm{div}(e) \subset X$ is a smooth effective Cartier divisor for every $e \in \mathfrak{u}$, possibly empty or with several connected components.
\end{thm}

Note that the statement is also true for a homomorphism $\mathfrak{d} \to H^0(X,\cL)$ which need not be injective.

\begin{cor}
 In the situation of Theorem~\ref{smoothing-nc-spaces}, there are global sections $s \in \T^1_V$ and $e_0 \in \omega_V^\vee$ such that $(V,s,e_0)$ is a well-formed triple.
\end{cor}
\begin{proof}
 Each open stratum $S^\circ \in [\cS_kV]$ is a smooth quasi-projective variety. We apply Bertini's theorem to all open strata $S^\circ$ separately. First, we construct $e_0 \in \omega_V^\vee$ with the desired property by intersecting the relevant open subsets of $H^0(V,\omega_V^\vee)$ for all open strata $S^\circ$. Then we can construct a dense open subset $\mathfrak{u} \subseteq H^0(D,\T^1_V)$ such that the condition on $Z \cap S^\circ \subset S^\circ$ is satisfied. Since also $E \cap S^\circ$ decomposes into smooth quasi-projective varieties, we achieve that $Z \cap E \cap S^\circ \subset E \cap S^\circ$ is a smooth effective Cartier divisor by possibly shrinking $\mathfrak{u}$. Then $Z \cap E \cap S^\circ \subset Z \cap S^\circ$ must be a smooth effective Cartier divisor as well.
\end{proof}
\begin{rem}
 In \cite[Thm.~1.1]{FFR2021}, it was assumed instead that $D = V_{sing}$ itself is projective (while not assuming that $V$ is projective), and that an appropriate section $e_0 \in \omega_V^\vee$ is already given.\footnote{This is the meaning of the term \emph{effective} anti-canonical bundle in \cite[Thm.~1.1]{FFR2021}.} Then we do not need the quasi-projectivity of the strata in $[\cS_0V]$, and for $S^\circ \in [\cS_kV]$ for $k\geq 1$, it is automatic.
\end{rem}

We investigate the local structure of well-formed triples $(V,s,e_0)$. Recall that, given a local ring $R = \cO_{X,x}$ of a smooth scheme $X/\CC$ at a $\CC$-valued point $x \in X$, we have 
$$\Omega^1_{R/\CC} \cong \bigoplus_{k = 1}^n R \cdot dy_k$$
for any regular system of parameters $y_1,...,y_n \in \m_R$. Thus, for every $R$-module $M$ and every sequence $m_1,...,m_n \in M$, there is a unique $\CC$-linear derivation $D: R \to M$ with $D(y_k) = m_k$. We write $\partial_i$ for the unique derivation with $\partial_i(y_k) = \delta_{ik}$. Recall the following basic fact:

\begin{lemma}
 Let $R = \cO_{X,x}$ be a regular local ring of dimension $n \geq 1$, essentially of finite type over $\CC$, and let $y_1,...,y_n \in \m_R$ be a regular system of parameters. For an element $f \in \m_R$, the following statements are equivalent:
 \begin{enumerate}[label=\emph{(\roman*)}]
  \item $R/(f)$ is a regular local ring of dimension $n - 1$;
  \item $f \notin \m_R^2$;
  \item there is a $\CC$-linear derivation $D: R \to R$ with $D(f) \notin \m_R$;
  \item there is an index $1 \leq i \leq n$ with $\partial_i(f) \notin \m_R$.
 \end{enumerate}
\end{lemma}
\begin{proof}
 This essentially follows from the discussion in \cite[07PD]{stacks}.
\end{proof}

With this preparation, we also get:

\begin{lemma}
 Let $R = \cO_{X,x}$ be a regular local ring of dimension $n \geq 2$, essentially of finite type over $\CC$, and let $y_1,...,y_n \in \m_R$ be a regular system of parameters. For elements $f,g \in \m_R$, the following statements are equivalent:
 \begin{enumerate}[label=\emph{(\roman*)}]
  \item both $R/(f)$ and $R/(g)$ are regular local rings of dimension $n - 1$, and $R/(f,g)$ is a regular local ring of dimension $n - 2$;
  \item there are indices $1 \leq i,j \leq n$ such that $\partial_i(f) \notin \m_R$, $\partial_j(g) \notin \m_R$, and $\partial_i(f)\partial_j(g) - \partial_i(g)\partial_j(f) \notin \m_R$.
 \end{enumerate}
\end{lemma}
\begin{proof}
 Let $f_k := [\partial_k(f)] \in R/\m_R = \kk$ and $g_k := [\partial_k(g)] \in R/\m_R = \kk$. Then we have $[f] = \sum_{k = 1}^n f_k [y_k]$ and $[g] = \sum_{k = 1}^n g_k[y_k]$ in $\m_R/\m_R^2$. First suppose that we have indices $i,j$ with $f_i \not= 0$, $g_j \not= 0$, and $f_ig_j - f_jg_i \not= 0$. In particular, we have $i \not= j$, and also $f \not= 0$, $g \not= 0$. By the above lemma, $R/(f)$ and $R/(g)$ are regular local rings of dimension $n - 1$. For notational simplicity, suppose that $i = 1$ and $j = 2$. Note that $f,y_2,...,y_n$ is a regular system of parameters for $R/(f)$, so $\bar y_2,...,\bar y_n \in \m_{R/(f)}$ is a regular system of parameters for $R/(f)$. Now 
 $$[\bar g] = \sum_{k = 1}^n g_k[\bar y_k] = \sum_{k = 2}^n (g_k - g_1f_k/f_1)[\bar y_k]$$
 because $0 = [\bar f] = \sum_{k = 1}^n f_k[\bar y_k]$ and $f_1 \not= 0$. Since $g_2f_1 - g_1f_2 \not= 0$, we have $[\bar g] \not= 0 \in \m_{R(f)}/\m_{R/(f)}^2$, so $\bar g \notin \m_{R/(f)}^2$, and hence $R/(f,g)$ is regular of dimension $n - 2$.
 
 Conversely, assume the regularity statement. Suppose that there is some $f_i \not= 0$ such that $g_i = 0$. Because $R/(g)$ is regular of dimension $n - 1$, there is some $g_j \not= 0$ with $j \not= i$. Then $f_ig_j - f_jg_i = f_ig_j \not= 0$, so we are finished. Similarly, when there is some $g_j \not= 0$ with $f_j = 0$, our claim follows. Thus, we may assume that $f_i = 0$ if and only if $g_i = 0$, and that there is some index $j$ with $f_j,g_j \not= 0$. Now suppose that the claim is wrong. In particular, for each $i$, we have $f_i = 0$ or $f_ig_j = f_jg_i$, i.e., $f_i/f_j = g_i/g_j$. Since $f_i = 0$ implies $g_i = 0$, we always have $f_i/f_j = g_i/g_j$. But then $[f]/f_j = [g]/g_j$ in $\m_R/\m_R^2$. This shows $\bar g \in \m_{R/(f)}^2$, contradicting the regularity of $R/(f,g)$ (under the assumption that this ring has dimension $n - 2$).
\end{proof}

Let $p: Y \to X$ be a morphism between two schemes of finite type over $\CC$. Then $p: Y \to X$ is \'etale if and only if the induced map $p^*: \widehat\cO_{X,p(y)} \to \widehat\cO_{Y,y}$ of complete local rings is an isomorphism for all $\CC$-valued points $y \in Y$. Namely, by \cite[0C4G]{stacks}, the map $p^*:\cO_{X,p(y)} \to \cO_{Y,y}$ is flat, and then $p: Y \to X$ is \'etale at $y$ by \cite[02GU, (6)]{stacks}. This allows us to show the following structure result for well-formed triples. The proof essentially follows a part of \cite[Thm.~2.6]{GrossSiebertII}.

\begin{prop}\label{smoothing-nc-local-model}\note{smoothing-nc-local-model}\index{Gross--Siebert type!standard Gross--Siebert type}
 Let $(V,s,e_0)$ be a well-formed triple, and let $v \in D = V_{sing} \subset V$ be a $\CC$-valued point.
 \begin{enumerate}[label=\emph{(\roman*)}]
  \item Assume that $s(v) = 0$ and $e_0(v) = 0$. Then there is an \'etale neighborhood $\chi: (W,w) \to (V,v)$ of $v$ together with an \'etale morphism $\phi: W \to M_0(r;q)$
 to 
 $$M_0(r;q) = \Spec \CC[x_0,...,x_r,u_1,...,u_q]/(x_0 \cdot ... \cdot x_r)$$
 for $q \geq 2$ with $\phi(w) = 0$ such that $\chi^*(s) = \phi^*(u_1 \cdot s_0) \in \T^1_W$ and $\chi^*(e_0) = v \cdot \phi^*(u_2 \cdot \eps_0^\vee) \in \omega_W^\vee$ for some invertible function $v \in \cO_W^*$, where $s_0 \in \T^1_{M_0}$ and $\eps_0^\vee \in \omega_{M_0}^\vee$ are as above.
 \item Assume that $s(v) = 0$ and $e_0(v) \not= 0$. Then there is an \'etale neighborhood $\chi: (W,w) \to (V,v)$ of $v$ together with an \'etale morphism $\phi: W \to M_0(r;q)$ for $q \geq 1$ with $\phi(w) = 0$ such that $\chi^*(s) = \phi^*(u_1 \cdot s_0) \in \T^1_W$.
 \end{enumerate}
\end{prop}
\begin{proof}
 Note that the statement makes sense because $p^*\T^1_X = \T^1_Y$ and $p^*\omega_X^\vee = \omega_Y^\vee$ for an \'etale morphism $p: Y \to X$. We show the first statement, and the proof of the second statement is similar. Let $T^\circ \in [\cS_kV]$ be the open stratum with $v \in T^\circ$. Then we set $r = k$ and $q = d - k$, where $d = \mathrm{dim}(V)$. Since $V$ is a normal crossing space, we can find an \'etale neighborhood $\chi: (W,w) \to (V,v)$ together with an \'etale morphism $\psi: W \to M_0(r;q) =: M_0$ with $\psi(w) = 0$. By shrinking $W$, we can assume that $W$ is affine. Since $s_0$ trivializes $\T^1_{M_0}$, and $\eps_0^\vee$ trivializes $\omega_{M_0}^\vee$, we can find functions $f,g \in \cO_W$ with $\chi^*s = f\cdot \psi^*(s_0)$ and $\chi^*e_0 = g \cdot \psi^*(\eps_0^\vee)$. We have $f(w) = 0$ and $g(w) = 0$. Let us denote the stratum of $W$ by $T^\circ$ as well; by shrinking $W$, we can assume that $T^\circ \subset W$ is a (smooth) closed subscheme. Let us denote the images of $f,g$ in $\bar R := \cO_{T^\circ}$ by $\bar f,\bar g$. By the well-formedness assumption, $\bar R/\bar f$, $\bar R/\bar g$, and $\bar R/(\bar f,\bar g)$ are regular local rings of dimension $q - 1$ respective $q - 2$. Now $y_1 := \psi^*u_1,...,y_q := \psi^*u_q$ forms a regular system of parameters of $\bar R$. Thus, we can find indices $1 \leq i,j \leq q$ with $\partial_i(\bar f) \notin \m_{\bar R}$, $\partial_j(\bar g) \notin \m_{\bar R}$, and $\partial_i(\bar f)\partial_j(\bar g) - \partial_i(\bar g)\partial_j(\bar f) \notin \m_{\bar R}$. For simplicity of notation, let us assume $i = 1$ and $j = 2$. Now we define a morphism 
 $\phi: W \to M_0(r;q)$ by $\phi^*(u_1) = f$, $\phi^*(u_2) = g$, and $\phi^*(u_k) = \psi^*(u_k)$ for $k \geq 3$ as well as $\phi^*(x_k) = \psi^*(x_k)$. Note that $\phi^*(u_k) \in \m_{W,w}$ for all $1 \leq k \leq q$ as well as $\phi^*(x_k) \in \m_{W,w}$, so we have $\phi(w) = 0$. Since $f,g,y_3,...,y_q$ is a system of parameters for $\bar R$---here we use the condition on the derivations---, we find that $\phi^*(\m_0)$ generates $\m_{w} \subseteq \cO_{W,w}$. Thus, the induced map $\phi^*: \widehat\cO_{M_0,0} \to \widehat\cO_{W,w}$ of complete local rings is surjective. However, we also have an isomorphism $\psi^*: \widehat\cO_{M_0,0} \cong \widehat\cO_{W,w}$ from our first \'etale morphism $\psi$, so $\phi^* \circ (\psi^*)^{-1}$ is a surjective endomorphism of a Noetherian ring, hence an isomorphism. Since $\phi: W \to M_0(r;q)$ is of finite type, it is \'etale at $w$, hence \'etale after further shrinking $W$.
 
 Now let $\pi: M_0(r;q) \to M_0(r;0)$ be the projection. On the one hand side, we have $s_0 = \pi^*(\tilde s_0)$ for the corresponding section on $M_0(r;0)$. On the other hand side, we have $\pi \circ \phi = \pi \circ \psi$. Thus, we have $\chi^*(s) = f \cdot \psi^*(s_0) = \phi^*(u_1) \cdot \phi^*(s_0)$. The case of $\eps_0^\vee$ is more complicated. We have $\eps_0 = \pi^*\tilde\eps_0 \otimes du_1 \wedge ... \wedge du_q$, so 
 \begin{align}
  \phi^*(\eps_0) &= (\pi \circ \phi)^*\tilde\eps_0 \otimes df \wedge dg \wedge d\psi^*(u_3) \wedge ... \wedge d\psi^*(u_q) \nonumber \\
  &= (\pi \circ \psi)^*\tilde\eps_0 \otimes (\partial_1(f)\partial_2(g) - \partial_2(f)\partial_1(g)) d\psi^*(u_1) \wedge ... \wedge d\psi^*(u_q)  \nonumber \\  
  &= (\partial_1(f)\partial_2(g) - \partial_2(f)\partial_1(g)) \cdot \psi^*(\eps_0) \nonumber
 \end{align}
 Here, $\partial_k: \cO_W \to \cO_W$ is the induced derivation from $\partial_{u_k}: \cO_{M_0} \to \cO_{M_0}$. These derivations restrict to $T^\circ$, so after shrinking $W$, we know that the factor is invertible on $W$. Then we find $\chi^*(e_0) = g \cdot \psi^*(\eps_0^\vee) = v \cdot \phi^*(u_2) \cdot \phi^*(\eps_0^\vee)$ for an invertible function $v$.
\end{proof}

\begin{proof}[Proof of Theorem~\ref{smoothing-nc-spaces}]
 $(V,Z,s)$ is obviously a well-adjusted triple. From the local model of Proposition~\ref{smoothing-nc-local-model}, we see that the log structure on $f_0: X_0 \to S_0$ is log toroidal of standard Gross--Siebert type. Namely, the map $\eta: \cL\cS_V \to \T^1_V$ is compatible with \'etale pull-back. Outside the double locus $D$, the section $e_0 \in \omega_V^\vee$ is log regular because it is the strict embedding of a smooth divisor in a strict log smooth space. Inside $D$, Proposition~\ref{smoothing-nc-local-model} shows that $e_0$ is log regular up to multiplication with a unit, but then $e_0 \in \omega_V^\vee$ itself is log regular because this notion is invariant under isomorphisms of the triple $(V,\cL,e_0)$. The modification $g_0: X_0(e_0) \to S_0$ is log toroidal because the pull-back from $M_0(r;q)$ is log toroidal, and because this pull-back is isomorphic to the pull-back from $V$ since the two sections of $\omega_V^\vee$ differ only by a unit locally. Then $g_0: X_0(e_0) \to S_0$ is log Calabi--Yau because the chosen line bundle is the anti-canonical bundle $\omega_V^\vee$. The thickenings of $u_2$ on the infinitesimal log toroidal deformations of $M_0(r;q)$ are log transversal so that the modification in $u_2$ is log toroidal as well. When pulling these deformations back to $W$ along $\phi: W \to M_0(r;q)$, we obtain local log toroidal deformations of $g_0: X_0(e_0) \to S_0$ which coincide with the modifications of the log toroidal deformations of $f_0: X_0 \to S_0$. Thus, the deformations of type $\D(e_0)$ are log toroidal. The formal smoothing is obtained from the unobstructedness of $\mathrm{LD}^{\D(e_0)}_{X_0(e_0)/S_0}$ by Theorem~\ref{perfect-G-calc-log-toroidal} via the map $\Psi \circ \Phi$ to $\mathrm{LD}^\D_{X_0/S_0}$. The algebraic smoothing in the projective case is obtained from the unobstructedness of $\mathrm{LD}^{\D(e_0)}_{X_0(e_0)/S_0}(\cL_0)$ for a very ample line bundle $\cL_0$ as in Chapter~\ref{alg-degen-sec}.
\end{proof}

\begin{rem}
 In the projective case, the algebraic degeneration $f: X \to S = \Spec \CC\llbracket t\rrbracket$ endowed with the divisorial log structure from the central fiber is a log toroidal family of standard Gross--Siebert type, as follows from the method of Lemma~\ref{log-str-comp} and the discussion before it. Thus, we have a well-behaved log de Rham complex $\W^\bullet_{X/S}$ on $f: X \to S$, whose Hodge--de Rham spectral sequence degenerates at $E_1$ by \cite[Cor.~8.7]{FeltenThesis}. Since the central fiber of this log morphism $f: X \to S$ coincides with $f_0: X_0 \to S_0$ by Lemma~\ref{log-str-comp}, we can compute the Hodge numbers of the smoothing $X_\eta$ as the log Hodge numbers of $f_0: X_0 \to S_0$.
\end{rem}


\part{Additional material}


\chapter{From the fibers to the total space}\label{fiber-total-space-sec}\note{fiber-total-space-sec}

We collect results about going from the fibers of a map to the total space. Let $\Lambda \to R$ be a flat ring homomorphism of finite type between Noetherian rings inducing a morphism of schemes $f: X \to S$, and let $M$ be a finitely generated $R$-module inducing a coherent sheaf $\F$ on $X$.

\begin{lemma}
 Assume that $\F_s = 0$ on $X_s$ for all points $s \in S$. Then $\F = 0$. If $f: X \to S$ maps closed points to closed points (for example if $S$ is Jacobson), then it is sufficient to check closed points $s \in S$.
\end{lemma}
\begin{proof}
 Let $x \in X$ be a point with image $s = f(x) \in S$. Then we have a factorization 
 $$x = \Spec \kappa(x) \xrightarrow{i_{x,s}} X_s \xrightarrow{i_s} X$$
 of $i_x: x \to X$. Thus, we have $i_x^*\F = i_{x,s}^*\F_s = 0$, so $\F_x/(\m_x \cdot \F_x) = 0$. Since $\F_x$ is a finitely generated $\cO_{X,x}$-module, Nakayama's lemma implies $\F_x = 0$. Then $\F = 0$ follows from \cite[Prop.~3.8]{AM1969}. The second claim is then clear.
\end{proof}
\begin{rem}
 If $\F$ is not finitely generated, the statement is false, at least the part about the closed points. Let $\Lambda = R = \CC[x]$, and let $M$ be the field of fractions of $\CC[x]$. Then $i_x^*M = 0$ for all closed points $x \in X = \bAA^1$.
\end{rem}

\begin{lemma}\label{zero-in-thick-fibers}\note{zero-in-thick-fibers}
 Let $m \in M = \Gamma(X,\F)$, and assume that $m|_{X_A} = 0$ for all fibers $X_A \to \Spec A$ over Artinian local rings $A = \cO_{S,s}/\m_s^{k + 1}$ for $k \geq 0$ and points $s \in S$. Then $m = 0$.
\end{lemma}
\begin{proof}
 Assume that $m \not= 0$. Then there is a point $x \in X$ with $m \not= 0$ in $\F_x$. Since $\F_x$ is a finitely generated $\cO_{X,x}$-module and $\cO_{X,x}$ is a Noetherian local ring, the Krull intersection theorem implies that we can find $k + 1 \geq 1$ such that $m \notin \m_x^{k + 1} \cdot \F_x$. Let $B = \cO_{X,x}/\m_x^{k + 1}$; then $m \not= 0$ in $\F \otimes_R B$. Since the inclusion $\Spec B \to X$ factors through $X_A$ with $A = \cO_{S,f(x)}/\m_{f(x)}^{k + 1}$, this is a contradiction to our assumption.
\end{proof}

Let $Z \subseteq X$ be a closed subset. Following \cite[Ch.~IV, \S5.9]{EGAIV-2}, a coherent sheaf $\F$ on $X$ is \emph{$Z$-pure} if $H^0(X,\F) \to H^0(X \setminus Z, \F)$ is injective, and \emph{$Z$-closed} if $H^0(X,\F) \to H^0(X \setminus Z,\F)$ is bijective. By \cite[Ch.~IV, Prop.~5.10.2]{EGAIV-2}, a coherent sheaf $\F$ is $Z$-pure if and only if $\mathrm{depth}_{\cO_{X,z}}(\F_z) \geq 1$ for every $z \in Z$, and $Z$-closed if and only if $\mathrm{depth}_{\cO_{X,z}}(\F_z) \geq 2$ for every $z \in Z$. In demonstrating how to go from the fibers to the total space with these properties, we use that 
$$\mathrm{depth}_B(N) = \mathrm{depth}_A(A) + \mathrm{depth}_{B \otimes_A k}(N \otimes_A k)$$
for a local homomorphism $A \to B$ of Noetherian local rings, the residue field $k$ of $A$, and a finitely generated $B$-module $N$ which is flat over $A$. This statement is a special case of \cite[Ch.~IV, Prop.~6.3.1]{EGAIV-2}.

\begin{lemma}\label{injective-in-fibers}\note{injective-in-fibers}
 Let $\F$ be flat over $S$. Let $U \subseteq X$ be an open subset, and assume that $H^0(X_s,\F_s) \to H^0(U_s,\F_s)$ is injective for all points $s \in S$. Then $H^0(X,\F) \to H^0(U,\F)$ is injective.
\end{lemma}
\begin{proof}
 Let $Z := X \setminus U$. We have to show $\mathrm{depth}_{\cO_{X,z}}(\F_z) \geq 1$ for all $z \in Z$. For a point $z \in Z$, let $s := f(z)$, and let $X_s$ be the fiber over $s$. Then we have $z \in X_s$ and $\cO_{X_s,z} = \cO_{X,z} \otimes_{\cO_{S,s}} \kappa(s)$. By assumption, we have $\mathrm{depth}_{\cO_{X_s,z}}((\F_s)_z) \geq 1$. Since $\mathrm{depth}({\cO_{S,s}}) \geq 0$, we find $\mathrm{depth}_{\cO_{X,z}}(\F_z) \geq 1$.
\end{proof}

\begin{lemma}\label{bijective-in-fibers}\note{bijective-in-fibers}
 Let $\F$ be flat over $S$. Let $U \subseteq X$ be an open subset, and assume that 
 $H^0(X_s,\F_s) \to H^0(U_s,\F_s)$
 is bijective for all points $s \in S$. Then $H^0(X,\F) \to H^0(U,\F)$ is bijective.
\end{lemma}
\begin{proof}
 The proof is the same as for the previous lemma, with $1$ replaced by $2$.
\end{proof}



\chapter{Multilinear differential operators}\label{multilin-diff-op-sec}\note{multilin-diff-op-sec}

Let $f: X \to S$ be a morphism of schemes, and let $\F$ and $\G$ be quasi-coherent $\cO_X$-modules. Following \cite[0G3P]{stacks} or \cite[\S 16.8]{EGAIV-4}, a \emph{differential operator of order $k$} is, inductively, an $f^{-1}(\cO_S)$-linear map of sheaves $D: \F \to \G$ such that $D_r(f) := D(rf) - rD(f): \F \to \G$ is a differential operator of order $k - 1$ for every $r \in \cO_X$. For $k = 0$, a differential operator of order $0$ is, by definition, nothing but an $\cO_X$-linear homomorphism. Here, we give a generalization for multilinear operators 
$$\mu: \F_1 \times ... \times \F_n \to \G,$$
the cases of most interest for us being the Lie bracket $[-,-]$ and the Lie derivative $\cL$ in a Gerstenhaber calculus.
Due to lack of reference, we have to develop everything from scratch, although the theory is very similar to the classical case $n = 1$. The theory works in great generality, without any Noetherianity or other finiteness assumptions. It comes in two very similar flavors, one for a total degree $k \geq 0$, and one for a tuple $(k_1,...,k_n)$ as degree.

\section{Multilinear differential operators in commutative algebra}
\index{multilinear differential operators!on rings}

Let $\phi: A \to R$ be a ring homomorphism, and let $F_1, ..., F_n$ as well as $G$ be $R$-modules. 

\begin{defn}
 Let 
 $$\mu: F_1 \times ... \times F_n \to G$$
 be an $A$-multilinear map. Then $\mu$ is a \emph{multilinear differential operator of total order $0$} if it is $R$-multilinear.
 For $1 \leq \ell \leq n$ and $r \in R$, we set $L_r: G \to G, \: g \mapsto rg,$ and
 $$L_{r;\ell}: F_1 \times ... \times F_n \to F_1 \times ... \times F_n, \quad (f_1,...,f_n) \mapsto (f_1, ..., rf_\ell, ...,f_n). $$
 For $r \in R$ and $1 \leq \ell \leq n$, we define 
 $$\Psi_{r;\ell}(\mu) :=  \mu \circ L_{r;\ell} - L_r \circ \mu: \enspace F_1 \times ... \times F_n \to G,$$
 $$\quad (f_1,...,f_n) \mapsto \mu(f_1,...,rf_\ell,...,f_n) - r\mu(f_1,...,f_n).$$
 These operations commute, i.e., $\Psi_{r;\ell} \circ \Psi_{s;m}(\mu) = \Psi_{s;m} \circ \Psi_{r;\ell}(\mu)$.
 Then, for $k \geq 1$, we say that $\mu$ is a \emph{multilinear differential operator of total order $k$} if $\Psi_{r;\ell}(\mu)$ is a multilinear differential operator of total order $k - 1$ for every $r \in R$ and every $1 \leq \ell \leq n$. We denote the set of multilinear differential operators of order $k$ by $\mathrm{Diff}^k_{R/A}(F_1,...,F_n;G)$. We consider it as an $R$-module via $r \cdot \mu := L_r \circ \mu$.
 
 Similarly, we say that $\mu$ is a \emph{multilinear differential operator of order $(0,...,0)$} if it is $R$-linear. Then, for a tuple $\underline k = (k_1,...,k_n)$ with $k_\ell \geq 0$ for all $1 \leq \ell \leq n$, we say that $\mu$ is a \emph{multilinear differential operator of order $\underline k$} if $\Psi_{r;\ell}(\mu)$ is a multilinear differential operator of order $(k_1,...,k_\ell - 1,...,k_n)$ for all $r \in R$ and all $1 \leq \ell \leq n$ with $k_\ell \geq 1$, and if it is $R$-linear in the $\ell$-th entry for $\ell$ with $k_\ell = 0$. We denote the set of multilinear differential operators of order $\underline k$ by $\mathrm{Diff}^{\underline k}_{R/A}(F_1,...,F_n;G)$. It is an $R$-module via $r \cdot \mu := L_r \circ \mu$.
\end{defn}

Multilinear differential operators of total order $k$ and order $\underline k$ are closely related. If $\mu$ is a multilinear differential operator of total order $k$, then it is also of order $(k,...,k)$; if $\mu$ is of order $\underline k = (k_1,...,k_n)$, then it is also of total order $|\underline k| := \sum_{\ell = 1}^n k_\ell$. In the following, we work out the theory for total order $k$, and then we state without proof the variant for order $\underline k$.

\vspace{\baselineskip}

When 
$$\nu: G_1 \times ... \times G_n \to H$$
is a multilinear differential operator of total order $k_0$, and 
$$\mu_i: F_{i1} \times ... \times F_{im_i} \to G_i$$
are multilinear differential operators of total order $k_i$, then 
$$\nu \circ (\mu_1,...,\mu_n): F_{11} \times ... \times F_{nm_n} \to H$$
is a multilinear differential operator of total order $k = k_0 + \sum_{i = 1}^n k_i$. Namely, this is true for $k_0 = k_i = 0$, and we have 
$$\Psi_{r;i\ell}(\nu \circ (\mu_1,...,\mu_n)) = \Psi_{r;i}(\nu) \circ (\mu_1,...,\mu_n) + \nu \circ (\mu_1,...,\Psi_{r;\ell}(\mu_i), ...,\mu_n).$$
Here, $\Psi_{r;i\ell}$ is applied to the $F_{i\ell}$-entry. We call $\nu \circ (\mu_1,...\mu_n)$ the \emph{multi-composition}.\index{multi-composition} Multi-composition yields an $A$-multilinear map
\begin{align}
 \Gamma_{R/A}: \enspace \mathrm{Diff}^{k_0}_{R/A}(G_1,...,G_n;H) &\times \mathrm{Diff}^{k_1}_{R/A}(F_{11},...,F_{1m_1};G_1) \times ... \nonumber \\
 ... \times \mathrm{Diff}^{k_n}_{R/A}&(F_{n1},...,F_{nm_n};G_n) \to \mathrm{Diff}^{k_0 + k_1 + ... + k_n}_{R/A}(F_{11},...,F_{nm_n};H). \nonumber
\end{align}

\subsection{An analog of the module of principal parts}

Let $\mu \in \mathrm{Diff}^k_{R/A}(F_1,...,F_n;G)$. Then for all sequences $r_1,...,r_k \in R$ and $\ell_1,...,\ell_k \in \{1,...,n\}$, we have 
$$\sum_{I \subseteq \{1,...,k\}} \prod_{i \notin I}(-r_i) \mu\left(f_1 \cdot\prod_{i \in I, \ell_i = 1} r_i,\ ...,\ f_n \cdot \prod_{i \in I,\ell_i = n} r_i\right) = 0$$
for all $f_1 \in F_1,...,f_n \in F_n$, where the sum is over all subsets $I \subseteq \{1,...,k\}$. Conversely, if an $A$-multilinear map 
$$\mu: F_1 \times ... \times F_n \to G$$
satisfies the above equation for all $r_1,...,r_k$ and $\ell_1,...,\ell_k$, then $\mu \in \mathrm{Diff}^k_{R/A}(F_1,...,F_n;G)$.

Our first aim now is to construct an analog of the module of principal parts. Let us consider the free $R$-module
$$\mathrm{Free}_R(F_1,...,F_n) := \bigoplus_{(f_1,...,f_n) \in F_1 \times ... \times F_n} R \cdot [f_1,...,f_n]$$
on generators $[f_1,...,f_n]$; inside it, let $\mathrm{I}^k_{R/A}(F_1,...,F_n)$ be the $R$-submodule generated by 
$$[f_1,...,f_\ell + f'_\ell, ...,f_n] - [f_1,...,f_\ell, ..., f_n] - [f_1,...,f_\ell',...,f_n]$$
and 
$$[f_1,...,a\cdot f_\ell, ...,f_n] - a\cdot[f_1,...,f_\ell,...,f_n]$$
for $a \in A$ as well as 
\begin{equation}\label{HPF}\tag{HPF}
 \sum_{I \subseteq \{0,...,k\}} \prod_{i \notin I}(-r_i) \left[f_1 \cdot\prod_{i \in I, \ell_i = 1} r_i,\ ...,\ f_n \cdot \prod_{i \in I,\ell_i = n} r_i\right]
\end{equation}
for $r_0,...,r_{k} \in R$, $1 \leq \ell_0,...,\ell_k \leq n$, and $f_1 \in F_1,...,f_n \in F_n$. Here, \eqref{HPF} stands for ``higher product formula''. Let 
$$\mathrm{P}^k_{R/A}(F_1,...,F_n) := \mathrm{Free}_R(F_1,...,F_n) / \mathrm{I}^k_{R/A}(F_1,...,F_n)$$ be the quotient.

\begin{lemma}
 The map 
 $$\mu^u_{R/A}: F_1 \times ... \times F_n \to \mathrm{P}^k_{R/A}(F_1,...,F_n), \quad (f_1,...,f_n) \mapsto [f_1,...,f_n],$$
 is an $A$-multilinear differential operator of total order $k$. It is universal in the sense that every $A$-multilinear differential operator of total order $k$
 $$\mu : F_1 \times ... \times F_n \to G$$
 factors uniquely through an $R$-linear homomorphism 
 $$h: \mathrm{P}^k_{R/A}(F_1,...,F_n) \to G$$
 as $\mu = h \circ \mu^u_{R/A}$. In particular, $\mathrm{Diff}^k_{R/A}(F_1,...,F_n;G) = \mathrm{Hom}_R(\mathrm{P}^k_{R/A}(F_1,...,F_n),G)$.
\end{lemma}
\begin{proof}
 $\mu_u$ is an $A$-multilinear differential operator of total order $k$ because all required relations are satisfied by construction. These relations are preserved when postcomposing with $R$-linear homomorphisms so that $h \circ \mu^u_{R/A}$ is always an $A$-multilinear differential operator of total order $k$. This defines a map 
 $$\mathrm{Hom}_R(\mathrm{P}^k_{R/A}(F_1,...,F_n),G) \to \mathrm{Diff}^k_{R/A}(F_1,...,F_n;G).$$
 When $\mu$ is an element on the right, then we can set $h_\mu[f_1,...,f_n] := \mu(f_1,...,f_n)$; due to the relations for $\mu$, this is well-defined, and we have $h_\mu \circ \mu^u_{R/A} = \mu$. If $h \circ \mu^u_{R/A} = h' \circ \mu^u_{R/A}$, then $h[f_1,...,f_n] = h'[f_1,...,f_n]$, so $h = h'$ since this is an $R$-generating system of the module $\mathrm{P}^k_{R/A}(F_1,...,F_n)$.
\end{proof}
\begin{var}
 There is a universal $A$-multilinear differential operator of order $\underline k$
 $$\mu^u_{R/A}: F_1 \times ... \times F_n \to \mathrm{P}^{\underline k}_{R/A}(F_1,...,F_n), \quad (f_1,...,f_n) \mapsto [f_1,...,f_n].$$
 In particular, $\mathrm{Diff}^{\underline k}_{R/A}(F_1,...,F_n;G) = \mathrm{Hom}_R(\mathrm{P}^{\underline k}_{R/A}(F_1,...,F_n),G)$.
\end{var}

Next, we exhibit a canonical generating system for $\mathrm{P}^k_{R/A}(F_1,...,F_n)$; this allows us to conclude that $\mathrm{P}^k_{R/A}(F_1,...,F_n)$ is finitely generated in good cases. Let $X \subset R$ be a generating system of $R$ as an $A$-\emph{algebra}, and let
$$M(X) := \{\prod_{i = 1}^\ell x_i \ | \ x_1,...,x_\ell \in X\} \cup \{1\}$$
 be the set of monomial expressions in $X$. If $E_i$ is a system of generators of $F_i$ as an $R$-module, then $\mathrm{P}^k_{R/A}(F_1,...,F_n)$ will be generated, as an $R$-module, by elements $[p_1e_1,...,p_ne_n]$ with $e_i \in E_i$ and $p_i \in M(X)$. However, a subset of these generators is already sufficient:

\begin{lemma}\label{principal-parts-fg}\note{principal-parts-fg}
 Let $M^n(k;X)$ be the set of vectors $(p_1,...,p_n) \in M(X)^n$ with $\sum \mathrm{deg}(p_i) \leq k$. Then $\mathrm{P}^k_{R/A}(F_1,...,F_n)$ is generated, as an $R$-module, by $[p_1e_1,...,p_ne_n]$ with $(p_1,...,p_n) \in M^n(k;X)$ and $e_i \in E_i$. In particular, if $\phi: A \to R$ is of finite type, and each $F_i$ is a finitely generated $R$-module, then $\mathrm{P}^k_{R/A}(F_1,...,F_n)$ is a finitely generated $R$-module as well.
\end{lemma}
\begin{proof}
 For $k = 0$, we have $M(0;X) = \{1\}$ and $\mathrm{P}^0_{R/A}(F_1,...,F_n) = F_1 \otimes_R ... \otimes_R F_n$, so the statement is true. For $k \geq 1$, let $G(k) \subseteq \mathrm{P}^k_{R/A}(F_1,...,F_n)$ be the submodule generated by the specified elements. Then, by definition, $[p_1e_1,...,p_ne_n] \in G(k)$ for $(p_i)_i \in M^n(k;X)$ and $e_i \in E_i$. We show by induction that, for all $\ell \geq k$, we have $[p_1e_1,...,p_ne_n] \in G(k)$ for $(p_i)_i \in M(\ell;X)$ and $e_i \in E_i$. We can assume that $\prod_{i = 1}^n p_i$ has degree $\geq k + 1$; then we can find $(p_i')_i \in M^n(\ell - k - 1;X)$ and monomials $q_i$ such that $\prod_{i = 1}^n q_i$ has degree $k + 1$ and $p_i'q_i = p_i$. By using formula \eqref{HPF}, we can express $[p_1e_1,...,p_ne_n]$ in terms of the form $Q \cdot [P_1e_1,...,P_ne_n]$ with $Q \in M(X)$ and $(P_i)_i \in M(\ell - 1;X)$. By induction hypothesis, these terms are contained in $G(k)$, so $[p_1e_1,...,p_ne_n] \in G(k)$ as well. Since $\mathrm{P}^k_{R/A}(F_1,...,F_n)$ is generated, as an $R$-module, by $[p_1e_1,...,p_ne_n]$ with $p_i \in M(X)$ and $e_i \in E_i$ for all $k$, our claim follows.
\end{proof}

\begin{var}
 Let $M^n(\underline k;X)$ be the set of vectors $(p_1,...,p_n) \in M(X)^n$ with $\mathrm{deg}(p_\ell) \leq k_\ell$. Then $\mathrm{P}^{\underline k}_{R/A}(F_1,...,F_n)$ is generated, as an $R$-module, by $[p_1e_1,...,p_ne_n]$ with $(p_1,...,p_n) \in M^n(\underline k;X)$ and $e_i \in E_i$. In particular, if $\phi: A \to R$ is of finite type, and each $F_i$ is a finitely generated $R$-module, then $\mathrm{P}^{\underline k}_{R/A}(F_1,...,F_n)$ is a finitely generated $R$-module as well.
\end{var}

\subsection{Localization on $R$}

In order to study the behavior of multilinear differential operators on sheaves, we study localization, so let $S \subset R$ be a multiplicative system. We start with a rather strong uniqueness result of extensions from $R$ to $S^{-1}R$, which is enforced by \eqref{HPF}.

\begin{lemma}\label{diff-loc-uniqueness}\note{diff-loc-uniqueness}
 Let $F_1,...,F_n$ be $R$-modules, and let $G_S$ be an $S^{-1}R$-module. Let 
 $$\mu,\nu: S^{-1}F_1 \times ... \times S^{-1}F_n \to G_S$$
 be two $A$-multilinear differential operators of total order $k$, and assume that 
 $$\mu(f_1,...,f_n) = \nu(f_1,...,f_n)$$
 for $f_i \in F_i$. Then $\mu = \nu$.
\end{lemma}
\begin{proof}
 We prove the statement by induction on $k$. For $k = 0$, the statement follows from $S^{-1}R$-linearity. So assume that we know the statement for some $k$, and let $\mu,\nu$ be $A$-multilinear differential operators of total order $k + 1$ which satisfy $\mu(f_1,...,f_n) = \nu(f_1,...,f_n)$ for $f_i \in F_i$. For $1 \leq \ell \leq n$ and $s \in S$, both $\Psi_{s;\ell}(\mu)$ and $\Psi_{s;\ell}(\nu)$ are $A$-multilinear differential operators of total order $k$; by induction hypothesis, we have $\Psi_{s;\ell}(\mu) = \Psi_{s;\ell}(\nu)$ since they are equal for $f_i \in F_i$. For $f_i \in F_i$ and $s_1 \in S$, we have 
 \begin{align}
  \mu(f_1/s_1,f_2,...,f_n) &= s_1^{-1} \cdot(\mu(f_1,f_2,...,f_n) - \Psi_{s;\ell}(\mu)(f_1/s_1,f_2,...,f_n)) \nonumber \\
  &= s_1^{-1} \cdot(\nu(f_1,f_2,...,f_n) - \Psi_{s;\ell}(\nu)(f_1/s_1,f_2,...,f_n)) \nonumber \\
  &= \nu(f_1/s_1,f_2,...,f_n). \nonumber
 \end{align}
 Repeating this argument inductively over $\ell$, we show that $\mu(f_1/s_1,...,f_\ell/s_\ell,f_{\ell + 1},...,f_n) = \nu(f_1/s_1,...,f_\ell/s_\ell,f_{\ell + 1},...,f_n)$ for all $\ell$ and hence $\mu = \nu$.
\end{proof}

If $\mu \in \mathrm{Diff}^k_{R/A}(F_1,...,F_n;G)$ is given, then an extension over $S^{-1}R$ is not only unique if it exists but actually exists. This is the computationally most intricate part of the basic theory of multilinear differential operators. Our proof closely follows \cite[0G36]{stacks}, which is the same statement for $n = 1$.

\begin{lemma}\label{diff-ext-loc}\note{diff-ext-loc}
 Let $S \subset R$ be a multiplicative system. Then, for every $k \geq 0$, there is a unique map
 $$\rho_k: \mathrm{Diff}^k_{R/A}(F_1,...,F_n;G) \to \mathrm{Diff}^k_{S^{-1}R/A}(S^{-1}F_1,...,S^{-1}F_n;S^{-1}G)$$
 such that 
 \[
  \xymatrix{
  F_1 \times ... \times F_n \ar[r]^-\mu \ar[d] & G \ar[d] \\
  S^{-1}F_1 \times ... \times S^{-1}F_n \ar[r]^-{\rho_k(\mu)} & S^{-1}G \\
  }
 \]
 commutes. The map is compatible among all $k$ in the sense that $\rho_k(\mu) = \rho_{k + 1}(\mu)$ if both are defined. The map is compatible with the multiplication operators in that, for $r \in R$ and $1 \leq \ell \leq n$, we have $L_r \circ \rho_k(\mu) = \rho_k(L_r \circ \mu)$ and $\rho_k(\mu) \circ L_{r;\ell} = \rho_k(\mu \circ L_{r;\ell})$; in particular, we have $\rho_{k - 1}(\Psi_{r;\ell}(\mu)) = \Psi_{r;\ell}(\rho_k(\mu))$.
\end{lemma}
\begin{proof}
 In this proof, we use the shorter notation $\mu_{r;\ell} := \Psi_{r;\ell}(\mu)$. The uniqueness follows from Lemma~\ref{diff-loc-uniqueness}. We show existence by induction on $k$. For $k = 0$, we obtain $\rho_0(\mu)$ as the localization $S^{-1}\mu$. For the induction step, we construct $\hat\mu = \rho_k(\mu)$ in several steps. Let $\hat F_i$ be the image of $F_i$ in $S^{-1}F_i$. First, we define 
 $$\hat\mu_0: \hat F_1 \times ... \times \hat F_n \to S^{-1}G \quad \mathrm{via} \quad \hat\mu_0(f_1,...,f_n) := \mu(f_1,...,f_n) \in S^{-1}G.$$
 We have to show that this is well-defined. So assume that $f_\ell = f_\ell' \in S^{-1}F_\ell$, i.e., $sf_\ell = sf_\ell'$ for some $s \in S$. Then we have 
 \begin{align}
  ss'\mu(f_1,...,f_\ell,...,f_n) &= s'\mu(f_1,...,sf_\ell,...,f_n) - s'\mu_{s;\ell}(f_1,...,f_\ell,...,f_n) \nonumber \\
  &= s'\mu(f_1,...,sf_\ell',...,f_n) - s'\mu_{s;\ell}(f_1,...,f_\ell',...,f_n) \nonumber \\
  &= ss'\mu(f_1,...,f_\ell',...,f_n) \nonumber
 \end{align}
 for some $s' \in S$ with $s'\mu_{s;\ell}(f_1,...,f_\ell,...,f_n) = s'\mu_{s;\ell}(f_1,...,f_\ell',...,f_n)$, which exists by induction hypothesis. Thus, $\hat\mu_0(f_1,...,f_n)$ is well-defined. It is easy to see that $\hat\mu_0$ is $A$-multilinear.
 
 For the next step, let $\hat\mu_{r;1} := \rho_{k - 1}(\mu_{r;1})$ for $r \in R$, which exists by the induction hypothesis. Note that 
 $$\hat\mu_{r;1}(f_1,...,f_n) = \hat\mu_0(rf_1,...,f_n) - r\hat\mu_0(f_1,...,f_n)$$
 for $f_i \in \hat F_i$.
 Then we define a map 
 $$\hat\mu_1: S^{-1}F_1 \times \hat F_2 \times ... \times \hat F_n \to S^{-1}G$$
 with the formula
 $$\hat\mu_1(f_1/s_1,f_2,...,f_n) := s_1^{-1} \cdot (\hat\mu_0(f_1,f_2,...,f_n) - \hat\mu_{s_1;1}(f_1/s_1,f_2,...,f_n)).$$
To show that it is well-defined, assume $f_1/s_1 = f_1'/s_1'$. Without loss of generality, we can assume that $f_1' = f_1t$ and $s_1' = s_1t$ for some $t \in S$. We have 
$$\mu_{s_1t;1} \circ L_{s_1;1} - L_{s_1} \circ \mu_{s_1t;1} = \mu_{s_1;1} \circ L_{s_1t;1} - L_{s_1t} \circ \mu_{s_1;1}$$
as operators in $\mathrm{Diff}^{k - 1}_{R/A}(F_1,...,F_n;G)$ by direct computation. Applying $\rho_{k - 1}$, the same equation holds for $\mu$ replaced with $\hat\mu$. We also have 
$$s_1t\hat\mu_0(f_1,...,f_n) -s_1\hat\mu_0(tf_1,...,f_n) = \hat\mu_{s_1;1}(tf_1,...,f_n) - \hat\mu_{s_1t;1}(f_1,...,f_n)$$
by direct computation. From this, we obtain 
$$\hat\mu_1(f_1/s_1,f_2,...,f_n) = \hat\mu_1(f'_1/s_1',f_2,...,f_n)$$
with the above definition applied to either side, and hence $\hat\mu_1(f_1/s_1,f_2,...,f_n)$ is well-defined in $S^{-1}G$. It is easy to see that $\hat\mu_1$ is $A$-multilinear.

Now let $r \in R$. Then we have $(\mu_{r;1})_{s_1;1} = (\mu_{s_1;1})_{r;1}$; consequently, we have $(\hat\mu_{r;1})_{s_1;1} = (\hat\mu_{s_1;1})_{r;1}$. Then we obtain 
\begin{align}
 (\hat\mu_1)_{r;1}(f_1/s_1,...,f_n) :&= \hat\mu_1(rf_1/s_1,...,f_n) - r\hat\mu_1(f_1/s_1,...,f_n) \nonumber\\ 
 &= \frac{1}{s_1}(\hat\mu_0)_{r;1}(f_1,...,f_n) - \frac{1}{s_1}(\hat\mu_{s_1;1})_{r;1}(f_1/s_1,...,f_n) \nonumber \\
 &= \frac{1}{s_1}\hat\mu_{r;1}(f_1,...,f_n) - \frac{1}{s_1}(\hat\mu_{r;1})_{s_1;1}(f_1/s_1,...,f_n) \nonumber \\
 &= \hat\mu_{r;1}(f_1/s_1,...,f_n). \nonumber
\end{align}
For $s \in S$, we have 
\begin{align}
 (\hat\mu_1)_{1/s;1}(f_1/s_1,...,f_n) :&= \hat\mu_1(f_1/(s_1s),...,f_n) - \frac{1}{s}\hat\mu_1(f_1/s_1,...,f_n) \nonumber \\
 &= -\frac{1}{ss_1}\hat\mu_{ss_1;1}(f_1/(s_1s),...,f_n) + \frac{1}{ss_1}\hat\mu_{s_1;1}(f_1/s_1,...,f_n) \nonumber \\ 
 &= - (L_{1/s} \circ \hat\mu_{s;1} \circ L_{1/s;1})(f_1/s_1,...,f_n) \nonumber 
\end{align}
by applying $\hat\mu_{s_1s;1} = \hat\mu_{s_1;1} \circ L_{s;1} + L_{s_1} \circ \hat\mu_{s;1}$, which holds since it holds over $R$, and then we apply $\rho_{k - 1}$.

In the next step, we construct inductively a map 
$$\hat\mu_p: S^{-1}F_1 \times ... \times S^{-1}F_p \times \hat F_{p + 1} \times ... \times \hat F_n \to S^{-1}G$$
by the formula 
$$\hat\mu_p(f_1/s_1,...,f_p/s_p,...,f_n) := s_p^{-1} \cdot (\hat\mu_{p - 1}(f_1/s_1,...,f_p,...,f_n) - \hat\mu_{s_p;p}(f_1/s_1,...,f_p/s_p,...,f_n)),$$
where, as above, $\hat\mu_{s_p;p} := \rho_{k - 1}(\mu_{s_p;p})$. Similar to the case $\hat\mu_1$, we obtain that $\hat\mu_p$ is well-defined and $A$-multilinear. Also similar to the case $\hat\mu_1$, we can show that 
$$(\hat\mu_p)_{r;p} = \hat\mu_{r;p} \quad \mathrm{and} \quad (\hat\mu_p)_{1/s;p} = -L_{1/s} \circ \hat\mu_{s;p} \circ L_{1/s;p}.$$
We wish to show these equations as well for $p$ replaced with $1 \leq \ell \leq p - 1$, i.e., 
$$(\hat\mu_p)_{r;\ell} = \hat\mu_{r;\ell} \quad \mathrm{and} \quad (\hat\mu_p)_{1/s;\ell} = -L_{1/s} \circ \hat\mu_{s;\ell} \circ L_{1/s;\ell},$$
where again $\hat\mu_{r;\ell} := \rho_{k - 1}(\mu_{r;\ell})$.
They are not yet shown because $\hat\mu_p$ has a larger domain than $\hat\mu_\ell$. By induction on $p$, we can assume that we already know these equations for $\hat\mu_{p - 1}$. Then the claim follows from two straightforward computations.

We set $\hat\mu := \hat\mu_n$; it is a well-defined $A$-multilinear map which makes the diagram in the statement of the Lemma commute. For $r \in R$, we have $(\hat\mu)_{r;\ell} = \rho_{k - 1}(\mu_{r;\ell})$, so this is an $A$-multilinear differential operator of total order $k - 1$ by the induction hypothesis. Similarly, $(\hat\mu)_{1/s;\ell}$ is an $A$-multilinear differential operator of total order $k - 1$ for $s \in S$. Then 
$$(\hat\mu)_{r/s;\ell} = (\hat\mu)_{r;\ell} \circ L_{1/s;\ell} + L_r \circ (\hat\mu)_{1/s;\ell}$$
is an $A$-multilinear differential operator of total order $k - 1$ as well. Thus, $\hat\mu$ is an $A$-multilinear differential operator of total order $k$.
The compatibility claims all follow from Lemma~\ref{diff-loc-uniqueness}.
\end{proof}

\begin{rem}
 Localization commutes with multi-composition. Namely, the diagram in Lemma~\ref{diff-ext-loc} commutes with $\nu \circ (\mu_1,...\mu_n)$ and $\rho(\nu) \circ (\rho(\mu_1),...,\rho(\mu_n))$, so uniqueness shows that 
 $$\rho(\nu \circ (\mu_1,...,\mu_n)) = \rho(\nu) \circ (\rho(\mu_1),...,\rho(\mu_n)).$$
\end{rem}
\begin{var}
 In the situation of Lemma~\ref{diff-ext-loc}, if $\mu$ is a multilinear differential operator of order $\underline k = (k_1,...,k_n)$, then $\rho_k(\mu)$ is a multilinear differential operator of order $\underline k$, where $k = \sum_{\ell = 1}^n k_\ell$.
\end{var}

With this preparation, we can show that the formation of $\mathrm{P}^k_{R/A}(F_1,...,F_n)$ commutes with localization.

\begin{lemma}\label{diff-loc-compat}\note{diff-loc-compat}
 We have a canonical isomorphism
 $$\Phi_k: S^{-1}\mathrm{P}^k_{R/A}(F_1,...,F_n) \to \mathrm{P}^k_{S^{-1}R/A}(S^{-1}F_1,...,S^{-1}F_n)$$
 of $S^{-1}R$-modules.
\end{lemma}
\begin{proof}
 First note that $S^{-1}\mathrm{Free}_R(F_1,...,F_n) = \mathrm{Free}_{S^{-1}R}(F_1,...,F_n)$. Then we get an $S^{-1}R$-homomorphism 
 $$\tilde\Phi_k: \mathrm{Free}_{S^{-1}R}(F_1,...,F_n) \to \mathrm{P}^k_{S^{-1}R/A}(S^{-1}F_1,...,S^{-1}F_n)$$
 by $\tilde\Phi_k[f_1,...,f_n] := [f_1,...,f_n]$. It descends to a map $\Phi_k$ from $S^{-1}\mathrm{P}^k_{R/A}(F_1,...,F_n)$ since every relation in $S^{-1}\mathrm{I}^k_{R/A}(F_1,...,F_n)$ is also in $\mathrm{I}^k_{S^{-1}R/A}(S^{-1}F_1,...,S^{-1}F_n)$.
 
 Next, Lemma~\ref{diff-ext-loc} gives us an $A$-multilinear differential operator 
 $$\nu_k := \rho_k(\mu^u_{R/A}): S^{-1}F_1 \times ... \times S^{-1}F_n \to S^{-1}\mathrm{P}^k_{R/A}(F_1,...,F_n)$$
 of total order $k$. Thus, we have a unique $S^{-1}R$-homomorphism 
 $$h_k: \mathrm{P}^k_{S^{-1}R/A}(S^{-1}F_1,...,S^{-1}F_n) \to S^{-1}\mathrm{P}^k_{R/A}(F_1,...,F_n)$$
 with $h_k \circ \mu^u_{S^{-1}R/A} = \nu_k$. On the other hand side, we also have $\Phi_k \circ \nu_k = \mu^u_{S^{-1}R/A}$ by Lemma~\ref{diff-loc-uniqueness} since they are equal on $F_1 \times ... \times F_n$. Then the universal property of $\mu^u_{S^{-1}R/A}$ shows that $\Phi_k \circ h_k = \mathrm{Id}$; in particular, $h_k$ is injective. Now $S^{-1}\mathrm{P}^k_{R/A}(F_1,...,F_n)$ is, as an $S^{-1}R$-module, generated by $\nu_k(f_1,...,f_k)$ for $f_1 \in F_1,...,f_n \in F_n$. Thus, $h_k$ is also surjective. But then $\Phi_k$ must be an isomorphism as well.
\end{proof}
\begin{var}
 For $\underline k = (k_1,...,k_n)$, we have a canonical isomorphism
 $$\Phi_{\underline k}: S^{-1}\mathrm{P}^{\underline k}_{R/A}(F_1,...,F_n) \to \mathrm{P}^{\underline k}_{S^{-1}R/A}(S^{-1}F_1,...,S^{-1}F_n)$$
 of $S^{-1}R$-modules.
\end{var}

\subsection{Base change}

Let $\phi: A \to R$ and $\psi: A \to B$ be ring homomorphisms, and set $T := R \otimes_A B$. Given an $A$-multilinear differential operator $\mu \in \mathrm{Diff}^k_{R/A}(F_1,...,F_n;G)$, we can form the $A$-multilinear map 
$$\mu \otimes_A B: \enspace F_1 \otimes_A B \times ... \times F_n \otimes_A B \to G \otimes_A B,$$
which is in fact a multilinear differential operator for $T/B$. It is clear that this construction commutes with multi-compositions. We have the analog of the well-known base change result for $\Omega^1_{R/A}$.

\begin{lemma}\label{P-base-change}\note{P-base-change}
 In the above situation, we have a canonical isomorphism 
 $$C^k_{B/A}: \mathrm{P}^k_{R/A}(F_1,...,F_n) \otimes_A B \to \mathrm{P}^k_{T/B}(F_1 \otimes_A B, ..., F_n \otimes_A B)$$
 of $T$-modules.
\end{lemma}
\begin{proof}
 The proof is very similar to that of Lemma~\ref{diff-loc-compat}. We have $\mathrm{Free_R}(F_1,...,F_n) \otimes_A B = \mathrm{Free}_T(F_1,...,F_n)$, and then we obtain the map $C^k_{B/A}$ by preservation of relations. On the other hand, we have a $B$-multilinear differential operator 
 $$\mu_{R/A}^u \otimes_A B: F_1 \otimes_A B \times ... \times F_n \otimes_A B \to \mathrm{P}^k_{R/A}(F_1,...,F_n) \otimes_A B$$
 which gives rise to a map 
 $$h_k: \mathrm{P}^k_{T/B}(F_1 \otimes_A B, ..., F_n \otimes_A B) \to \mathrm{P}^k_{R/A}(F_1,...,F_n) \otimes_A B$$
 with $h_k \circ \mu_{T/B}^u = \mu^u_{R/A} \otimes_A B$. By the universal property of $\mu^u_{T/B}$, we have $C^k_{B/A} \circ h_k = \mathrm{Id}$, so $h_k$ is injective. However, as a $T$-module, the target is generated by $(\mu^u_{R/A} \otimes_A B)(f_1,...,f_n)$ for $f_i \in F_i$, so $h_k$ is surjective as well. Then $C^k_{B/A}$ is an isomorphism.
\end{proof}
\begin{var}
 Let $\psi: A \to B$ be a ring homomorphism, and let $T := R \otimes_A B$. Let $\underline k = (k_1,...,k_n)$. Then we have a canonical isomorphism 
 $$C^{\underline k}_{B/A}: \mathrm{P}^{\underline k}_{R/A}(F_1,...,F_n) \otimes_A B \to \mathrm{P}^{\underline k}_{T/B}(F_1 \otimes_A B, ..., F_n \otimes_A B)$$
 of $T$-modules.
\end{var}

\subsection{Localization on $A$}

For the construction of a sheaf, we also need the following basic compatibility. Recall that $\Spec (A) \to \Spec(A')$ is an open immersion if and only if $A' \to A$ is flat, of finite presentation, and an epimorphism of rings. This is more general than being a localization in a multiplicative system $T \subset A'$.

\begin{lemma}\label{loc-A-compat}\note{loc-A-compat}
 Let $\psi: A' \to A$ be flat, of finite presentation, and an epimorphism of rings. Then
 $$\mathrm{P}^k_{R/A'}(F_1,...,F_n) = \mathrm{P}^k_{R/A}(F_1,...,F_n).$$
\end{lemma}
\begin{proof}
 Let $\phi' := \phi \circ \psi$. Then, using \cite[04VN]{stacks}, we obtain a Cartesian diagram 
 \[
  \xymatrix{
   A \ar[r]^-{\cong} & A \otimes_{A'} A \ar[r] & R \otimes_{A'} A  \\
   A' \ar[r] \ar[u] & A \ar[r] \ar[u]^{\cong} & R  \ar[u]^{\cong} \\
  }
 \]
 of rings. Then we have 
 \begin{align}
  \mathrm{P}^k_{R/A}(F_1,...,F_n) &= \mathrm{P}^k_{R/A}(F_1 \otimes_{A'} A,...,F_n \otimes_{A'} A) =  \mathrm{P}^k_{R/A'}(F_1,...,F_n) \otimes_{A'} A \nonumber \\
  &= \mathrm{P}^k_{R/A'}(F_1,...,F_n) \otimes_R R = \mathrm{P}^k_{R/A'}(F_1,...,F_n) \nonumber 
 \end{align}
 by Lemma~\ref{P-base-change}.
\end{proof}
\begin{var}
 In the setting of Lemma~\ref{loc-A-compat}, we have
 $$\mathrm{P}^{\underline k}_{R/A'}(F_1,...,F_n) = \mathrm{P}^{\underline k}_{R/A}(F_1,...,F_n)$$
 for $\underline k = (k_1,...,k_n)$.
\end{var}

\subsection{\'Etale localization of multilinear differential operators}

First, we given an alternative, more geometric description of $\mathrm{P}^k_{R/A}(F_1,...,F_n)$ in the case of a total degree $k$. Let us consider the exact sequence 
\begin{equation}\label{diagonal-embedding}
 0 \to I_\Delta \to R \otimes_A R \otimes_A ... \otimes_A R \xrightarrow{r_0 \otimes ... \otimes r_n \,\mapsto\, r_0 \cdot ... \cdot r_n} R \to 0
\end{equation}
which corresponds to the diagonal embedding. The kernel $I_\Delta$ is generated by elements of the form 
$$1 \otimes 1 \otimes ... \otimes r \otimes ... \otimes 1 - r \otimes 1 \otimes ... \otimes 1 \otimes ... \otimes 1.$$
In particular, for $k \geq 0$, the power $I_\Delta^{k + 1}$ is generated by elements of the form 
$$\sum_{I \subseteq \{0,...,k\}}\left(\prod_{i \notin I} (-r_i)\right) \otimes \left(\prod_{i \in I,\ell_i = 1}r_i\right) \otimes ... \otimes \left(\prod_{i \in I,\ell_i = n}r_i\right)$$
for $r_0,...,r_k \in R$ and $1 \leq \ell_0, ..., \ell_k \leq n$. Taking the tensor product of the exact sequence \eqref{diagonal-embedding} over $T = R \otimes_A R \otimes_A ... \otimes_A R$ with $B = R \otimes_A F_1 \otimes ... \otimes F_n$, we obtain
$$I_\Delta^{k + 1} \otimes_T B \to R \otimes_A F_1 \otimes_A ... \otimes_A F_n \to (T/I_\Delta^{k + 1}) \otimes_T B \to 0.$$
Comparing with the definition of $\mathrm{P}^k_{R/A}(F_1,...,F_n)$, we find a canonical isomorphism 
$$(T/I_\Delta^{k + 1}) \otimes_T (R \otimes_A F_1 \otimes_A ... \otimes_A F_n) \cong \mathrm{P}^k_{R/A}(F_1,...,F_n).$$
Now let $\psi: R \to R'$ be a ring map. It is easy to see that we have a canonical induced homomorphism of $R'$-modules 
$$R' \otimes_R \mathrm{P}^k_{R/A}(F_1,...,F_n) \to \mathrm{P}^k_{R'/A}(R' \otimes_R F_1,...,R' \otimes_R F_n), \enspace [f_1,...,f_n] \mapsto [1 \otimes f_1,...,1 \otimes f_n].$$
This map fits into a commutative diagram
\[
 \xymatrix{
  F_1 \times ... \times F_n \ar[d] \ar[r] & \mathrm{P}^k_{R/A}(F_1,...,F_n) \ar[r] & R' \otimes_R \mathrm{P}^k_{R/A}(F_1, ..., F_n) \ar[d] \\
  R' \otimes_R F_1 \times ... \times R' \otimes_R F_n \ar[rr] & & \mathrm{P}^k_{R'/A}(R' \otimes_R F_1,...,R' \otimes_R F_n) \\
 }
\]
\begin{lemma}\label{general-diff-loc}\note{general-diff-loc}
 Assume that the right vertical map is an isomorphism. Let $G'$ be an $R'$-module, and let $G$ be the $R'$-module $G'$ considered as an $R$-module via $\psi: R \to R'$. Then, for every $\mu \in \mathrm{Diff}^k_{R/A}(F_1,...,F_n;G)$, there is a unique $\nu \in \mathrm{Diff}^k_{R'/A}(R' \otimes_R F_1,...,R' \otimes_R F_n;G')$ such that 
 \[
  \xymatrix{
   F_1 \times ... \times F_n \ar[d] \ar[r]^-\mu & G \ar@{=}[d] \\
   R' \otimes_R F_1 \times ... \times R' \otimes_R F_n \ar[r]^-\nu & G' \\
  }
 \]
 commutes.
\end{lemma}
\begin{proof}
 This follows from $\mathrm{Hom}_R(E,G) = \mathrm{Hom}_{R'}(R' \otimes_R E,G')$ for $E = \mathrm{P}^k_{R/A}(F_1,...,F_n)$.
\end{proof}

Now the conditions of Lemma~\ref{general-diff-loc} are satisfied if $\psi: R \to R'$ is \'etale.

\begin{prop}\label{etale-diff-loc}\note{etale-diff-loc}
 Assume that $\psi: R \to R'$ is \'etale, in particular of finite presentation. Then 
 $$R' \otimes_R \mathrm{P}^k(F_1,...,F_n) \to \mathrm{P}^k_{R'/A}(R' \otimes_R F_1,...,R' \otimes_R F_n)$$
 is an isomorphism.
\end{prop}
\begin{proof}
 We work geometrically. Let $S = \Spec A$, $X = \Spec R$, and $X' = \Spec R'$. Let $g: X' \to X$ and $f: X \to S$ be the induced maps. Let $M = X \times_S X \times_S ... \times_S X$, and we define $M'$ analogously. We consider the commutative diagram 
 \[
  \xymatrix{
   & & \Delta_k' \ar[d] \ar[dl] & \Delta'_0 \ar[d] \ar[l] \ar@/^3em/[ddd]^g \\
   & M' \ar[dl]_{q'} \ar[d]^{h'} & M' \times_M \Delta_k \ar[d] \ar[l] & M' \times_M \Delta_0 \ar[d] \ar[l] \\
   X' \ar[d]^g & X' \times_S X \times_S ... \times_S X \ar[d]^h \ar[l] & X' \times_X \Delta_k \ar[d] \ar[l] & X' \times_X \Delta_0 \ar[d] \ar[l] \\
   X & M \ar[l]^q & \Delta_k \ar[l] & \Delta_0 \ar[l] \\
  }
 \]
 where $q: M \to X$ and $q': M' \to X'$ are the projections onto the first factors, $\Delta_0 \to M$ and $\Delta_0' \to M'$ are the diagonal embeddings, the lower and middle squares are Cartesian, and $\Delta_k \to M$ as well as $\Delta'_k \to M'$ are closed immersions defined by $\I_\Delta^{k + 1}$ respective $\I_{\Delta'}^{k + 1}$. The lower and middle vertical arrows are \'etale. Since $g$ is \'etale as well, we find that $\Delta_0' \to M' \times_M \Delta_0$ is \'etale. Since this map is also a closed immersion because $\Delta_0' \to X'$ is an isomorphism, we find that $\Delta_0' \to M' \times_M \Delta_0$ is the embedding of a connected component. Note that $M' \times_M \Delta_0 \to M'$ is defined by $\cO_{M'} \cdot \I_{\Delta}$, and $\Delta_0' \to M'$ is defined by $\I_{\Delta'}$. Thus, these two ideals coincide on the open subset $U' := M' \setminus (M' \times_M \Delta_0 \setminus \Delta'_0)$, and hence $\cO_{M'} \cdot \I_\Delta^{k + 1}$ coincides with $\I_{\Delta'}^{k + 1}$ on $U'$. Thus, $\Delta_k' \to M' \times_M \Delta_k$ is also the inclusion of a connected component; in particular, this map is an \'etale closed immersion. Then the base change $P \to M' \times_M \Delta_0$ of this map to $M' \times_M \Delta_0$ is an \'etale closed immersion, and $\Delta_0' \to P$ is \'etale. This map is also a closed immersion because $\Delta_0' \to X' \times_X \Delta_0$ is an isomorphism. It is then easy to see that the upper right square is Cartesian as well. Now $\Delta'_k \to X' \times_X \Delta_k$ is a closed immersion because it is the thickening of a closed immersion, see \cite[09ZW]{stacks}, and then it must be an isomorphism because it is \'etale and a homeomorphism. In a nutshell, the outer square in 
 \[
  \xymatrix{
   X' \ar[d]^g & M' \ar[l]^{q'} \ar[d]^p & \Delta'_k \ar[d]^{p_k} \ar[l]^{i'} \\
   X & M \ar[l]^q & \Delta_k. \ar[l]^i \\
  }
 \]
 is Cartesian. After setting $B = R \otimes_A F_1 \otimes_A ... \otimes_A F_n$ as a module on $M$ and $B' = R' \otimes_A (R' \otimes_R F_1) \otimes_A ... \otimes_A (R' \otimes_R F_n)$ as a module on $M'$, we have 
 \begin{align}
  g^*\mathrm{P}^k_{R/A}(F_1,...,F_n) &\cong g^*(q \circ i)_*i^*B \cong (q' \circ i')_*p_k^*i^*B \nonumber \\
  &\cong (q' \circ i')_*(i')^*B' \cong \mathrm{P}^k_{R'/A}(R' \otimes_R F_1,...,R' \otimes_R F_n) \nonumber
 \end{align}
 because $g$ is \'etale,\footnote{Note that flat base change holds when the other morphism is qcqs; since $\Delta_k \to X$ is a morphism of affine schemes, this is always true.} and this isomorphism coincides with the canonical isomorphism.
\end{proof}

\'Etale localization is compatible with multi-compositions and with base change $A \to B$.

\section{Multilinear differential operators on schemes}
\index{multilinear differential operators!on schemes}

Let $f: X \to S$ be a morphism of schemes, and let $\F_1,...,\F_n$ as well as $\G$ be quasi-coherent $\cO_X$-modules.

\begin{defn}
 Let 
 $$\mu: \F_1 \times ... \times \F_n \to \G$$
 be an $f^{-1}(\cO_S)$-multilinear map. We say that it is a \emph{multilinear differential operator of total order $0$} if it is $\cO_X$-multilinear. Then we say inductively that it is a \emph{multilinear differential operator of total order $k$} if, for every open $V \subseteq X$, every $1 \leq \ell \leq n$, and every $r \in \Gamma(V,\cO_X)$, the map
 $$\Psi_{r;\ell}(\mu) := \mu \circ L_{r;\ell} - L_r \circ \mu: \enspace \F_1|_V \times ... \times \F_n|_V \to \G|_V$$
 is a multilinear differential operator of total order $k - 1$. Multilinear differential operators of total order $k$ form a sheaf $\Diff^k_{X/S}(\F_1,...,\F_n;\G)$ of $\cO_X$-modules via $r \cdot \mu := L_r \circ \mu$.
 
 Similarly, we say that $\mu$ is a \emph{multilinear differential operator of order $(0,...,0)$} if it is $\cO_X$-multilinear. Then, for a tuple $\underline k = (k_1,...,k_n)$ with $k_\ell \geq 0$ for all $1 \leq \ell \leq n$, we say that $\mu$ is a \emph{multilinear differential operator of order $\underline k$} if $\Psi_{r;\ell}(\mu)$ is a multilinear differential operator of order $(k_1,...,k_\ell - 1,...,k_n)$ for all open $V \subseteq X$, all $r \in \Gamma(V,\cO_X)$, and all $1 \leq \ell \leq n$ with $k_\ell \geq 1$, and if it is $\cO_X$-linear in the $\ell$-th entry for $\ell$ with $k_\ell = 0$. We denote the sheaf of multilinear differential operators of order $\underline k$ by $\Diff^{\underline k}_{X/S}(\F_1,...,\F_n;\G)$. It is an $\cO_X$-module via $r \cdot \mu := L_r \circ \mu$.
\end{defn}

\begin{prop}\label{multilin-diff-op-shvs}\note{multilin-diff-op-shvs}
 In the above situation, there is a quasi-coherent sheaf $\cP^k_{X/S}(\F_1,...,\F_n)$ on $X$ together with an $f^{-1}(\cO_S)$-multilinear differential operator 
 $$\mu_{X/S}^u: \F_1 \times ... \times \F_n \to \cP^k_{X/S}(\F_1,...,\F_n)$$
 which induces an isomorphism $\cH om_{\cO_X}(\cP^k_{X/S}(\F_1,...,\F_n),\G) \cong \Diff_{X/S}^k(\F_1,...,\F_n;\G)$ for every quasi-coherent sheaf $\G$ via $h \mapsto h \circ \mu^u_{X/S}$. If $b: T \to S$ is a morphism of schemes and $c: Y := X \times_S T \to X$ is the fiber product, then we have a canonical isomorphism 
 $$C^k_{T/S}: c^*\cP^k_{X/S}(\F_1,...,\F_n) \cong \cP^k_{Y/T}(c^*\F_1,...,c^*\F_n)$$
 which induces a pull-back map 
 $$c^*: \enspace \Diff^k_{X/S}(\F_1,...,\F_n;\G) \to c_*\Diff^k_{Y/T}(c^*\F_1,...,c^*\F_n;c^*\G)$$
 for every quasi-coherent sheaf $\G$.
\end{prop}
\begin{proof}
 First, we fix an affine open $U \subseteq X$ such that there is some affine open $V \subseteq S$ with $U \subseteq f^{-1}(V)$. Let $A = \Gamma(V,\cO_S)$ and $R = \Gamma(U,\cO_X)$. If $S \subset R$ is a multiplicative system, let $U_S := \Spec (S^{-1}R) \subseteq U$. We set 
 $$\Gamma(U_S,\cP^k_{U/S}(\F_1,...,\F_n)) := \mathrm{P}^k_{S^{-1}R/A}(S^{-1}F_1,...,S^{-1}F_n)$$
 where $F_i := \Gamma(U,\F_i)$. By Lemma~\ref{loc-A-compat}, this is independent of the choice of $V$. By Lemma~\ref{diff-loc-compat}, we can set $\cP^k_{U/S}(\F_1,...,\F_n)$ to be the quasi-coherent sheaf associated with $\mathrm{P}^k_{R/A}(F_1,...,F_n)$. By Lemma~\ref{diff-ext-loc}, we obtain an $A$-multilinear differential operator 
 $$\mu_{U/S}^u: \F_1|_U \times ... \times \F_n|_U \to \cP^k_{U/S}(\F_1,...\F_n)$$
 first on every $U_S$, and then, by the sheaf property, on every open subset of $U$. Another application of Lemma~\ref{loc-A-compat} together with a careful topological argument shows that $\mu^u_{U/S}$ is $f^{-1}(\cO_S)$-multilinear. By induction on $k$, we show that a map is a multilinear differential operator of total order $k$ on the level of sheaves if and only if it is on all affine opens $U_S$, so $\mu_{U/S}^u$ is indeed an $f^{-1}(\cO_S)$-multilinear differential operator of total order $k$. This also shows 
 $$\Gamma(U_S,\Diff^k_{X/S}(\F_1,...,\F_n;\G)) = \mathrm{Diff}^k_{S^{-1}R/A}(S^{-1}F_1,...,S^{-1}F_n;S^{-1}G).$$
 In particular, the homomorphism  
 $$\cH om_{\cO_U}(\cP^k_{U/S}(\F_1,...,\F_n),\G) \to \Diff_{U/S}^k(\F_1|_U,...,\F_n|_U;\G|_U);$$
 induced by $\mu_{U/S}^u$ is an isomorphism.
 
 If $U_1, U_2 \subseteq X$ are two affine open subsets as above, then, by the universal property, $\cP^k_{U_1/S}(\F_1,...,\F_n;\G)$ and $\cP^k_{U_2/S}(\F_1,...,\F_n;\G)$ are canonically isomorphic on the overlap $U_1 \cap U_2$; in particular, there is a unique gluing $\cP^k_{X/S}(\F_1,...,\F_n;\G)$ to a quasi-coherent sheaf on $X$ together with a multilinear differential operator  $\mu_{X/S}^u$ of total order $k$ into it; $\mu_{X/S}^u$ has the desired universal property.
 
 Consider a base change along $b: T \to S$ as in the statement of the proposition. Then we obtain a multilinear differential operator of total order $k$
 $$\nu: \F_1 \times ... \times \F_n \to c_*c^*\F_1 \times ... \times c_*c^*\F_n \to c_*\cP^k_{Y/T}(c^*\F_1,...,c^*\F_n)$$
 relative to $X/S$. Thus, we have unique homomorphism 
 $$h: \cP^k_{X/S}(\F_1,...,\F_n) \to c_*\cP^k_{Y/T}(c^*\F_1,...,c^*\F_n)$$
 with $h \circ \mu_{X/S}^u = \nu$. The adjoint map is the map $C^k_{B/A}$ of Lemma~\ref{P-base-change} on appropriate affine open subsets, so it is an isomorphism; this is the map $C^k_{T/S}$ above. To form the pull-back map $c^*$ on the level of multilinear differential operators, we use the natural map 
 $$c^*\cH om(\cP^k_{X/S}(\F_1,...,\F_n),\G) \to \cH om(c^*\cP^k(\F_1,...,\F_n),c^*\G),$$
 which need not be an isomorphism.\footnote{Already the example of derivations with values in $\cO_X$ shows that $c^*$ does not necessarily induce an isomorphism $c^*\Diff^k_{X/S}(\F_1,...,\F_n;\G) \to \Diff^k_{Y/T}(c^*\F_1,...,c^*\F_n;c^*\G)$.}
\end{proof}
\begin{rem}
 The pull-back map $c^*$ is compatible with multi-compositions. This is because $c^*$ is given by $(-) \otimes_A B$ on affine open subsets $\Spec R \subseteq X$, which is compatible with multi-compositions.
\end{rem}

\begin{var}\label{multilin-diff-op-shvs-var}\note{multilin-diff-op-shvs-var}
 Proposition~\ref{multilin-diff-op-shvs} holds with $k$ replaced by $\underline k = (k_1,...,k_n)$ as well.
\end{var}

\begin{lemma}\label{pp-fg-shvs}\note{pp-fg-shvs}
 Assume that $f: X \to S$ is locally of finite type, and that $\F_1,...,\F_n$ are quasi-coherent sheaves of finite type. Then $\cP^k_{X/S}(\F_1,...,\F_n)$ and $\cP^{\underline k}_{X/S}(\F_1,...,\F_n)$ are quasi-coherent sheaves of finite type.
\end{lemma}
\begin{proof}
 Let $U \subseteq X$ be an affine open as in the proof of Proposition~\ref{multilin-diff-op-shvs}. By further shrinking $U$ and $V$, we can assume that $R$ is a finitely generated $A$-algebra. By \cite[01PB]{stacks}, each $R$-module $F_i$ is finitely generated. Thus, by Lemma~\ref{principal-parts-fg}, $\mathrm{P}^k_{R/A}(F_1,...,F_n)$ is a finitely generated $R$-module, and then $\cP^k_{X/S}(\F_1,...,\F_n)$ is of finite type.
\end{proof}

We also have \'etale localization on the level of schemes.

\begin{lemma}\label{etale-diff-loc-shvs}\note{etale-diff-loc-shvs}\index{multilinear differential operators!\'etale localization}
 In the above situation, let $g: X' \to X$ be \'etale, in particular locally of finite presentation. Then we have a canonical isomorphism 
 $$g^*\cP^k_{X/S}(\F_1,...,\F_n) \to \cP^k_{X'/S}(g^*\F_1,...,g^*\F_n)$$
 which induces a pull-back map 
 $$g^*: \Diff^k_{X/S}(\F_1,...,\F_n;\G) \to g_*\Diff^k_{X'/S}(g^*\F_1,...,g^*\F_n;g^*\G)$$
 for every quasi-coherent sheaf $\G$. This pull-back map is compatible with multi-compositions and base change.
\end{lemma}



\chapter{Spectral sequences}\label{spectral-seq-sec}\note{spectral-seq-sec}

\section{A criterion for $E_1$-degeneration over $\kk$}\label{spectral-seq-k-sec}\note{spectral-seq-k-sec}

Here, we give a criterion for a spectral sequence to degenerate at the page $E_1$, which we use in the proof of Lemma~\ref{QMC-unob-lemma}. The criterion goes at least back to \cite[Defn.~4.13]{KKP2008}, where it is only implicit; it is also claimed explicitly just after \cite[Ass.~5.4]{ChanLeungMa2023}, but none of the two references give a proof. Therefore, we provide the proof here. In fact, we give a more precise criterion which also applies to some form of partial degeneration as defined in Definition~\ref{degeneration-level-k} below.

Let $(B^{\bullet,\bullet},\partial,\bar\partial)$ be a double complex of $\kk$-vector spaces, bounded below in both indices. Let us first review the construction of the spectral sequence. We have inclusions 
$$I_r^{p,q} \subseteq K_r^{p,q} \subseteq B^{p,q}$$
such that $E^{p,q}_{r + 1} = K_r^{p,q}/I_r^{p,q}$, starting with 
$$K_0^{p,q} = \mathrm{ker}(\bar\partial: B^{p,q} \to B^{p,q + 1}), \quad I_0^{p,q} = \mathrm{im}(\bar\partial: B^{p,q - 1} \to B^{p,q}).$$
For $r \geq 1$, an \emph{$r$-zigzag} is a sequence $\alpha_\bullet = (\alpha_1,...,\alpha_{r})$ with $\alpha_i \in B^{p - 1 + i, q + 1 - i}$ and $\bar\partial\alpha_1 = 0$ as well as $\bar\partial\alpha_{i + 1} + \partial\alpha_i = 0$ for $1 \leq i \leq r - 1$. Then we have, for $r \geq 0$, that $\alpha \in K_r^{p,q}$ if and only if there is an $(r + 1)$-zigzag $\alpha_\bullet$ with $\alpha_1 = \alpha$. The differential is given by 
$$d_{r + 1}[\alpha] = [\partial\alpha_{r + 1}] \in E_{r + 1}^{p + r + 1,q - r}.$$
These claims can be proven by induction on $r$, using that $d_{r + 1}[\alpha] \in I_{r}^{p + r + 1, q - r} \subseteq I_{r + 1}^{p + r + 1,q - r}$ if and only if there is an element $\beta^0$ and a sequence of $i$-zigzags $\beta^i_\bullet$ of increasing length such that $0 = d_{r + 1}[\alpha] + \bar\partial\beta^0 + \sum_i \partial\beta^i_i$.

\begin{defn}\label{degeneration-level-k}\note{degeneration-level-k}
 We say that the first spectral sequence of $(B^{\bullet,\bullet},\partial,\bar\partial)$ degenerates at $E_1$ \emph{in $[k - 1,k]$}\index{degeneration!in$[k - 1,k]$} if $d_r^{p,q} = 0$ for $p + q = k - 1$ and $r \geq 1$.
 This is equivalent to $K_0^{p,q} = K_1^{p,q} = ... = K_\infty^{p,q}$ for $p + q = k - 1$. We say that the spectral sequence degenerates \emph{at level $k$}\index{degeneration!at $k$} if it degenerates in $[k - 1,k]$ and in $[k,k + 1]$. This is equivalent to $E^{p,q}_1 = E_2^{p,q} = ... = E^{p,q}_\infty$.
\end{defn}

For the criterion, we form a new double complex $(B^{\bullet,\bullet}\fpsh,\hb \dl,\db)$ of $\kk\fpsh$-modules with 
$$B^{p,q}\fpsh := \left\{\sum_{i = 0}^\infty b_i\hb^i \ \middle| \ b_i \in B^{p,q}\right\},$$
i.e., in general $B^{p,q}\fpsh \not= B^{p,q} \otimes_\kk \kk\fpsh$. They are only equal if $B^{p,q}$ has finite dimension, which we do not assume. However, concerning the total complexes, we have an equality
$$\bigoplus_{p + q = k} B^{p,q}\fpsh = B^k\fpsh := \left\{\sum_{i = 0}^\infty b_i\hb^i \ \middle| \ b_i \in B^k = \bigoplus_{p + q = k} B^{p,q}\right\}$$
because the direct sum is finite. Note that we have modified $\dl$ to $\hb \dl$.

Since $\kk\fpsh$ is a discrete valuation ring, a $\kk\fpsh$-module is flat if and only if it is torsion-free.

\begin{prop}\label{spec-seq-alternative-formulation}\note{spec-seq-alternative-formulation}
 In the above situation, the first spectral sequence of $(B^{\bullet,\bullet},\dl,\db)$ degenerates at $E_1$ in $[k - 1,k]$ for $k \in \ZZ$ if and only if $H^k(B^\bullet\fpsh,\db + \hb \dl)$ is a flat $\kk\fpsh$-module. In particular, the spectral sequence degenerates at $E_1$ if and only if $H^k(B^\bullet\llbracket \hbar \rrbracket, \bar\partial + \hbar\partial)$ is a flat $\kk\llbracket\hbar\rrbracket$-module for all $k$.
\end{prop}
\begin{proof}
 First assume that the first spectral sequence of $(B^{\bullet,\bullet},\dl,\db)$ degenerates at $E_1$ in $[k - 1, ]$, i.e., we have $K^{p,q}_0 = ... = K^{p,q}_\infty$ for $p + q = k - 1$. Consider the double complex $\tilde B^{\bullet,\bullet} := (B^{\bullet,\bullet}\fpsh,\dl,\bar \dl)$ with the unaltered differentials $\dl$ and $\bar\dl$. Then we have $\tilde K_r^{p,q} = K_r^{p,q}\fpsh$ and $\tilde I_r^{p,q} = I_r^{p,q}\fpsh$. 
 
 Next, we consider the double complex $\check B^{\bullet,\bullet} = (B^{\bullet,\bullet}\fpsh,\hb\dl,\bar\dl)$ with the modified differential. We have $\check d_0 = \tilde d_0$, so we find $\check K_0^{p,q} = \tilde K_0^{p,q}$. We already know that $\tilde K_0^{p,q} = \tilde K_1^{p,q} = ... = \tilde K^{p,q}_\infty$ for $p + q = k - 1$. If $\tilde\alpha \in \tilde K^{p,q}_r$, then we can find an $(r + 1)$-zigzag $\tilde\alpha_\bullet$ with $\tilde\alpha_1 = \tilde\alpha$. Setting $\check\alpha_i = \hbar^{i - 1}\tilde\alpha_i$, we obtain an $(r + 1)$-zigzag $\check\alpha_\bullet$ for $\check\alpha := \tilde\alpha$ with respect to the modified double complex $\check B^{\bullet,\bullet}$. Thus, we have $\check\alpha \in \check K^{p,q}_r$, and hence $\check K^{p,q}_r = \tilde K^{p,q}_r = \check K^{p,q}_0$. This shows $\check K^{p,q}_0 = ... = \check K^{p,q}_\infty$ for $p + q = k - 1$, hence $\check I^{p,q}_1 = \check I^{p,q}_2 = ... = \check I^{p,q}_\infty$ for $p + q = k$. In particular, $\check E^{p,q}_\infty$ is a $\kk\fpsh$-submodule of $\check E^{p,q}_1 = \tilde E^{p,q}_1 = E^{p,q}_1\fpsh$ so that $\check E^{p,q}_\infty$ is torsion-free and hence flat. The abutment $H^k(B^\bullet\fpsh, \db + \hb\dl)$ is a flat $\kk\fpsh$-module since it has a filtration by the flat $\kk\fpsh$-modules $\check E^{p,q}_\infty$.

 Conversely, assume that the first spectral sequence of $(B^{\bullet,\bullet},\dl,\bar\dl)$ does \emph{not} degenerate at $E_1$ in $[k - 1,k]$. Let $\alpha \in K^{p,q}_{r - 1}$ be such that $\alpha \notin K^{p,q}_r$ for some $r \geq 1$ and $p + q = k - 1$. Then there is an $r$-zigzag $\alpha_\bullet$ with $\alpha_1 = \alpha$ but no such $(r + 1)$-zigzag. In particular, there is no $\gamma \in B^{p + r,q - r}$ such that $\db\gamma = \beta := -\partial\alpha_r$. Note that $\partial\beta = 0$ and $\bar\partial\beta = 0$, so also $(\bar\partial + \hb\partial)(\beta) = 0$, i.e., $\beta$ defines a class $[\beta]_{tot} \in H^{k}(B^\bullet\fpsh,\db + \hb\dl)$. We have $[\beta]_{tot} \not= 0$ because otherwise there would be a $\gamma$ with $\bar\partial\gamma = \beta$ as above. 
 Now let 
 $$\alpha' := \sum_{i = 1}^r\hb^{i} \alpha_i.$$
 A direct computation shows that $(\db + \hb\dl)(\alpha') = \hb^{r + 1}\beta$. In other words, we have $\hb^{r + 1}[\beta]_{tot} = 0$ so that $H^{k}(B^\bullet\fpsh,\db + \hb\dl)$ has $\hb$-torsion and is not flat.
\end{proof}

\section{Spectral sequences over Artinian local rings}\label{spectral-sequence-A-sec}\note{spectral-sequence-A-sec}

We prove Proposition~\ref{spectral-sequence-Artin-ring} below about the behavior of spectral sequences over Artinian local rings under base change. In this section, $A$ is an Artinian local $\kk$-algebra with residue field $\kk$ and maximal ideal $\m_A \subset A$.

As pointed out in \cite{Wahl1976}, Nakayama's lemma holds for a general $A$-module $M$, not just for finitely generated ones. More precisely, if $M$ is an $A$-module, and $N \subseteq M$ is a submodule such that the induced map $N \to M_0 := M \otimes_A \kk$ is surjective, then $N = M$. 

By \cite[051G]{stacks}, an $A$-module $M$ is flat if and only if it is projective if and only if it is free. By \cite[Lemma~0.2.1]{Wahl1976}, in an exact sequence 
$$0 \to M' \to M \to M'' \to 0$$
of $A$-modules, if two of them are flat, so is the third. In particular, if we have an exact sequence of finite length, and all modules but one are flat, the remaining one is flat as well.
The following result is \cite[Thm.~A.1]{Wahl1976}.

\begin{prop}\label{Artin-complex-base-change}\note{Artin-complex-base-change}
 Let $C^\bullet$ be a complex of flat $A$-modules (not necessarily bounded), and let 
 $$\phi_A^q: H^q(C^\bullet) \otimes_A \kk \to H^q(C^\bullet \otimes_A \kk)$$
 be the canonical map. Then:
 \begin{enumerate}[label=\emph{(\alph*)}]
  \item If $\phi_A^q$ is surjective, then it is an isomorphism.
  \item If $H^q(C^\bullet)$ is flat over $A$, then $\phi_A^q$ is injective.
  \item If $\phi_A^q$ is injective, and $\phi_A^{q - 1}$ is an isomorphism, then $\phi_A^q$ is an isomorphism.
  \item Any two of the following statements together imply the third:
  \begin{enumerate}[label=\emph{(\roman*)}]
   \item $\phi_A^q$ is an isomorphism.
   \item $\phi_A^{q - 1}$ is an isomorphism.
   \item $H^q(C^\bullet)$ is flat over $A$.
  \end{enumerate}
 \end{enumerate}
\end{prop}

This allows us to prove the following result for spectral sequences. Recall that we assume $\partial\bar\partial + \bar\partial\partial = 0$ in a double complex.

\begin{prop}\label{spectral-sequence-Artin-ring}\note{spectral-sequence-Artin-ring}
 Let $(C^{\bullet,\bullet},\partial,\bar\partial)$ be a double complex of flat $A$-modules, bounded from below in both variables. Let $(C^\bullet,\partial + \bar\partial)$ be the associated total complex. Let $C_0^{\bullet,\bullet} := C^{\bullet,\bullet} \otimes_A \kk$, and assume that the first spectral sequence 
 $$'\!E_{0;1}^{p,q} = H^q(C_0^{p,\bullet},\bar\partial) \Rightarrow H^{p + q}_0 := H^{p + q}(C_0^\bullet,\partial + \bar\partial)$$
 associated with $C_0^{\bullet,\bullet}$ degenerates at $E_1$. Assume furthermore that the restriction map 
 $$H^n_A := H^{n}(C^\bullet,\partial + \bar\partial) \to H^n_0 := H^{n}(C_0^\bullet,\partial + \bar\partial)$$
 is surjective for all $n$. Let 
 $$'\!E_{A;1}^{p,q} = H^q(C^{\bullet,\bullet},\bar\partial) \Rightarrow H^{p + q}_A = H^{p + q}(C^\bullet,\partial + \bar\partial)$$
 be the first spectral sequence associated with $C^{\bullet,\bullet}$. We have an induced map $'\!E_{A;r} \to {'\!E}_{0;r}$ of spectral sequences.
 Then the following hold:
 \begin{enumerate}[label=\emph{(\alph*)}]
  \item The filtered pieces $F^pH^n_A$ of $'\!E_{A;r}$ are flat $A$-modules, and the canonical maps 
  $$\rho^p: F^pH^n_A \to F^pH^n_0$$
  are surjective, inducing an isomorphism $F^pH_A^n \otimes_A \kk \cong F^pH_0^n$.
  \item Each $'\!E_{A;r}^{p,q}$ is a flat $A$-module, and the canonical map $'\!E_{A;r}^{p,q} \to {'\!E}_{0;r}^{p,q}$ is surjective, inducing an isomorphism ${'\!E}_{A;r}^{p,q} \otimes_A \kk \cong {'\!E}_{0;r}^{p,q}$. In particular, $H^q(C^{p,\bullet},\bar\partial)$ is a flat $A$-module, and $H^q(C^{p,\bullet},\bar\partial) \to H^q(C_0^{p,\bullet},\bar\partial)$ is surjective, inducing an isomorphism $H^q(C^{p,\bullet},\bar\partial) \otimes_A \kk \cong H^q(C_0^{p,\bullet},\bar\partial)$.
  \item The spectral sequence $'\!E_{A;r}$ degenerates at $E_1$. 
 \end{enumerate}
\end{prop}
\begin{proof}
 Since $A \to \kk$ can be split into a succession of small extensions, we may prove the result by induction on the length of $A$. So let $A \to \bar A$ be a small extension, and assume that the result holds for $\bar A$. Note that $\bar C^{\bullet,\bullet} := C^{\bullet,\bullet} \otimes_A \bar A$ satisfies the assumptions. We obtain an exact sequence 
 $$0 \to C_0^{\bullet,\bullet} \otimes_\kk I \to C^{\bullet,\bullet} \to \bar C^{\bullet,\bullet} \to 0$$
 of double complexes, where $I \subset A$ is the kernel of $A \to \bar A$. First, we show that 
 $$\rho^p: F^pH^n_A \xrightarrow{r^p} F^pH^n_{\bar A} \to F^pH^n_0$$
 is surjective. By the induction hypothesis, it is sufficient to show that $r^p$ is surjective. Let $d = \partial + \bar\partial$ for short. For $p << 0$, we have $F^pH^n = H^n$, so in this case $r^p$ is surjective. To show that in fact all restrictions $r^p$ are surjective, assume the contrary and let $p$ be the lowest value such that $r^p$ is not surjective. Let $\bar c = (\bar c^{i,j})_{i \geq p}$ with $\bar c^{i,j} \in \bar C^{i,j}$ and $d(\bar c) = 0$ be a representative of a class $[\bar c] \in F^pH^n(\bar C^\bullet,d)$ which is not in the image of $r^p$. By assumption, $r^{p - 1}$ is surjective. Thus, we can find a $(\tilde c^{i,j})_{i \geq p - 1}$ with $\tilde c^{i,j} \in C^{i,j}$ and $d(\tilde c) = 0$ such that $[\tilde c]|_{\bar A} = [\bar c]$. This means that $\tilde c|_{\bar A} + d(\bar b) = \bar c$ for some $\bar b \in \bar C^{n - 1}$. Since $C^{n - 1} \to \bar C^{n - 1}$ is surjective, we can find a lift $b \in C^{n - 1}$ of $\bar b$. When we set $c := \tilde c + db \in C^n$, then this element satisfies $d(c) = 0$ and $c|_{\bar A} = \bar c$. In particular, for $i < p$, we have $c^{i,j} \in I \cdot C^{i,j} = C_0^{i,j} \otimes_\kk I$. We construct a new element $e \in I \cdot C^n$. For $i < p$, we set $e^{i,j} = c^{i,j}$. Let $m$ be the smallest index with $e^{m,n-m} \not= 0$. By the construction of the differentials $d_r$ of the spectral sequence $'\!E_{0;r}$, 
 $$\pm\bar\partial e^{m + r,n - m - r} = \mp \partial e^{m + r - 1,n - m - r + 1}$$
 represents $d_r[e^{m,n - m}]$ until we reach $e^{p - 1,n - p + 1}$. However, by assumption, $'\!E_{0;r}$ degenerates at $E_1$ so that $d_r[e^{m,n - m}] = 0$. Thus, we can construct all further $e^{i,j}$ for $i \geq p$ by choosing a preimage in $C_0^{i,j} \otimes_\kk I$ of $-\partial e^{i - 1, j + 1}$ under $\bar\partial$, which exists because $[-\partial e^{i - 1,j + 1}] = \pm d_r[e^{m,n - m}]$ is the zero class. This element $e \in C_0^n \otimes_\kk I$ satisfies $d(e) = 0$, and we have $(c - e)^{i,j} = 0$ for $i < p$. Thus, $d(c - e) = 0$, and $c - e$ defines a class in $F^pH^n_A$ with $r^p[c - e] = [\bar c]$. Hence, $r^p$ is surjective.
 
 Now since $F^pH^n_A \to F^pH^n_0$ is surjective, the induced map 
 $$'\!E_{A;\infty}^{p,q} = F^pH_A^{p + q}/F^{p + 1}H^{p + q}_A \to F^pH^{p + q}_0/F^{p + 1}H^{p + q}_0 = {'\!E}_{0;\infty}^{p,q}$$
 is surjective as well. By assumption, we have $'\!E_{0;\infty}^{p,q} = {'\!E}_{0;1}^{p,q}$, and $'\!E_{A;\infty}^{p,q} = Z_{A;\infty}^{p,q}/B_{A;\infty}^{p,q}$ is a subquotient of $'\!E_{A;1}^{p,q}$ for $A$-submodules $B_{A;\infty}^{p,q} \subseteq Z_{A;\infty}^{p,q} \subseteq {'\!E}_{A;1}^{p,q}$. In particular, $Z_{A;\infty}^{p,q} \to {'\!E}_{0;1}^{p,q}$ is surjective, and Nakayama's lemma (as discussed at the beginning of the section) shows that $Z_{A;\infty}^{p,q} = {'\!E}_{A;1}^{p,q}$. Then $d_1: {'\!E}_{A;1}^{p,q} \to {'\!E}_{A;1}^{p + 1,q}$ must be the zero map, for otherwise $Z_{A;\infty}^{p,q} \subsetneq {'\!E}_{A;1}^{p,q}$. Since this holds for all $(p,q)$, we have $d_1 = 0$. Repeating this argument step by step, we find $d_r = 0$ for all $r \geq 1$; in particular, $'\!E_{A;r}$ degenerates at $E_1$. 
 
 Now we have $'\!E_{A;\infty}^{p,q} = {'\!E}_{A;1}^{p,q}$ due to degeneration at $E_1$. Thus, the induced map $'\!E_{A;1}^{p,q} \to {'\!E}_{0;1}^{p,q}$ is surjective. Proposition~\ref{Artin-complex-base-change} applied to $(C^{p,\bullet},\bar\partial)$ yields that $'\!E_{A;1}^{p,q} \otimes_A \kk \cong {'\!E}_{0;1}^{p,q}$ is an isomorphism, and that $'\!E_{A;1}^{p,q}$ is a flat $A$-module. Since $'\!E_{A;r}^{p,q} = {'\!E}_{A;1}^{p,q}$, they are also flat and surject to $'\!E_{0;r}^{p,q}$, inducing isomorphisms. 
 
 Since $'\!E_{A;\infty}^{p,q} \cong {'\!E}_{A;1}^{p,q}$ is flat, and since $H^n_A$ is flat as well by Proposition~\ref{Artin-complex-base-change} applied to $(C^\bullet,d)$, the filtered pieces $F^pH^n_A$ must be flat.  Finally, we show $F^pH^n_A \otimes_A \kk \cong F^pH^n_0$ by induction on $p$.
 For $p << 0$, we have $F^pH^n_A = H^n_A$, so it follows from Proposition~\ref{Artin-complex-base-change} applied to $(C^\bullet,d)$. Then we have an exact sequence 
 $$0 \to F^{p + 1}H^n_A \to F^pH^n_A \to {'\!E}_{A;\infty}^{p,n - p} \to 0$$
 with only flat modules so that it remains exact after applying $(-) \otimes_A \kk$. This shows that $F^{p + 1}H_A^n \otimes_A \kk \to F^{p + 1}H^n_0$ is injective, and surjectivity follows from surjectivity of $F^{p + 1}H^n_A \to F^{p + 1}H^n_0$.
\end{proof}

\begin{cor}\label{spectral-sequence-Artin-ring-base-change}\note{spectral-sequence-Artin-ring-base-change}
 In the above situation, let $A \to B$ be a homomorphism of Artinian local $\kk$-algebras with residue field $\kk$. Then the induced maps 
 $$H^q(C^{p,\bullet},\bar\partial) \otimes_A B \to H^q(C^{p,\bullet} \otimes_A B,\bar\partial)$$
 and $F^pH^n_A \otimes_A B \to F^pH^n_B$ are isomorphisms of $B$-modules.
\end{cor}
\begin{proof}
 Let $M_A$ be any of the modules over $A$ in the statement, and let $M_B$ respective $M_0$ be their variant over $B$ respective $\kk$. Since $M_A \to M_0$ is surjective, the induced map $M_A \otimes_A B \to M_0$ is surjective as well; then the image of $M_A \otimes_A B \to M_B$ surjects onto $M_0 \cong M_B \otimes_B \kk$, so $M_A \otimes_A B \to M_B$ is surjective by Nakayama's lemma. Since $M_B$ is a flat $B$-module, the inclusion of the kernel $K$ of $M_A \otimes_A B \to M_B$ is universally injective. Thus $K \otimes_B \kk = 0$ and hence $K = 0$.
\end{proof}



\chapter{Analytification}\label{analyt-sec}\note{analyt-sec}

Let $X/\CC$ be a scheme locally of finite type. By Serre's famous article \cite{SerreGAGA} and subsequent work, we have a complex analytic space $X^{an}$ together with a map $\varphi: X^{an} \to X$ of locally $\CC$-ringed spaces which is universal among all morphisms of locally $\CC$-ringed spaces from a complex analytic space $\Z$ to $X$; see \cite[XII]{Grothendieck2003} for a nice exposition. In this chapter, we prove the existence of a unique analytification of first-order differential operators between quasi-coherent sheaves in Proposition~\ref{analytify-diff-op-1}. Furthermore, we extend the classical comparison of the algebraic and analytic cohomologies on a proper scheme $X$ from coherent sheaves to quasi-coherent sheaves in Lemma~\ref{qcoh-gaga}. For neither of the two results we could find a reference. Both results are needed to show Theorem~\ref{perfect-G-calc-log-toroidal} in Chapter~\ref{perfect-G-calc-sec}.

We write $\cO_X^{an}$ for the structure sheaf of $X^{an}$. This notation is less heavy than $\cO_{X^{an}}$, and it is correct in that $\cO_{X^{an}} \cong \varphi^*\cO_X$.

\section{Preliminaries}\label{analyt-preliminaries-sec}\note{analyt-preliminaries-sec}

\textbf{The sequence topology.} As described in \cite{GR1971}, every ring of convergent power series $K_n := \CC\{x_1,...,x_n\}$ carries a special topology induced from an exhaustion by Banach algebras, the \emph{sequence topology}.\index{sequence topology} More precisely, for every $t = (t_1,...,t_n) \in \RR^n_+$, we have a map 
$$|| \cdot ||_t: \CC\llbracket x_1,...x_n\rrbracket \to \RR_{\geq 0} \cup \{\infty\}, \quad \sum_{\nu} a_\nu x^\nu \mapsto \sum_\nu |a_\nu| \cdot t^\nu;$$
then $B_t\{x_1,...,x_n\} := \{f \in \CC\llbracket  x_1,...,x_n\rrbracket \ | \ || f ||_t < \infty\}$ is a Banach algebra, and $K_n$ is the union of these for all $t \in \RR^n_+$.

The sequence topology can be extended in a canonical way to all analytic $\CC$-algebras $A$ such that every ideal $I \subseteq A$ is closed, and in every surjection $\phi: A \to A/I$ of analytic $\CC$-algebras, the sequence topology on $A/I$ is the quotient topology of the one on $A$. Every homomorphism $\phi: A \to B$ of analytic $\CC$-algebras is continuous. Furthermore, it can be canonically extended to all finitely generated $A$-modules $M$ such that $A^{\oplus n}$ carries the product topology, in every surjection $h: M \to N$, the sequence topology on $N$ is the quotient of the one on $M$, and all homomorphisms $h: M \to N$ of finitely generated $A$-modules are continuous. Every injection $h: M \to N$ exhibits $M$ as a closed subset of $N$. The sequence topology is always Hausdorff.

We extend the sequence topology to arbitrary $A$-modules as follows: Every $A$-module $M$ is the direct limit of the system $M_\alpha$ of its finitely generated submodules. Then we endow $M$ with the colimit topology of this directed system. This construction has the following properties:

\begin{lemma}
 Let $A$ be an analytic $\CC$-algebra, and let $M$ be an $A$-module. 
 \begin{enumerate}[label=\emph{(\alph*)}]
  \item Let $(M_i)_{i \in I}$ be a directed system of finitely generated $A$-modules with colimit $M$. Then the colimit topology of the system $(M_i)_i$ coincides with the colimit topology of the system $(M_\alpha)_\alpha$.
  \item Let $h: M \to N$ be $A$-linear. Then $h$ is continuous for the sequence topologies on $M$ and $N$.
  \item If $p: M \to N$ is $A$-linear and surjective, then the sequence topology on $N$ is the quotient of the sequence topology on $M$.
  \item If $i: M \to N$ is $A$-linear and injective, then $M$ is a closed subset of $N$, and the sequence topology on $M$ is the induced topology from the sequence topology on $N$.
 \end{enumerate}
\end{lemma}
\begin{proof}
 All statements are more or less straightforward. Let $M_{\alpha(i)}$ be the image of $M_i$ under the map $f_i: M_i \to M$ of the directed system $(M_i)_i$, and let $\bar f_i: M_i \to M_{\alpha(i)}$ be the induced map. Let $U \subseteq M$ be a subset. Then $f_i^{-1}(U) = \bar f_i^{-1}(U \cap M_{\alpha(i)})$. Since the sequence topology on $M_{\alpha(i)}$ is the quotient topology of the sequence topology on $M_i$, we have that $U \cap M_{\alpha(i)}$ is open in $M_{\alpha(i)}$ if and only if $f_i^{-1}(U)$ is open in $M_i$. Now if $U$ is open in the sequence topology on $M$, then $U \cap M_{\alpha(i)}$ is open in $M_{\alpha(i)}$, thus $f_i^{-1}(U)$ is open in $M_i$, and hence $U$ is open in the colimit topology of $(M_i)_i$. Conversely, if $U$ is open in the colimit topology of $(M_i)_i$, then $U \cap M_{\alpha(i)}$ is open in $M_{\alpha(i)}$ for all $i$. Since $M$ is the colimit of $(M_i)_i$, for every $M_\alpha$, there is some $i$ with $M_\alpha \subseteq M_{\alpha(i)}$. Then $U \cap M_\alpha$ is open in $M_\alpha$ because $M_\alpha \to M_{\alpha(i)}$ is continuous, so $U$ is open in the sequence topology on $M$. The proofs of the other statements are similar.
\end{proof}
\begin{rem}
 Every open in the product topology on $M \times N = M \oplus N$ is open in the sequence topology. However, the converse is not clear.\footnote{In general, filtered colimits do not commute with finite products in the category of topological spaces. This holds for compactly generated Hausdorff spaces (maybe under some additional hypotheses on the filtered system and the maps between the spaces), but, in general, the sequence topology even on a finitely generated $A$-module $M$ is not compactly generated.} Correspondingly, it is not clear if the sequence topology on $M$ is Hausdorff if $M$ is not finitely generated.
\end{rem}

\par\vspace{\baselineskip}

\noindent\textbf{Quasi-coherent analytic sheaves.}\index{analytic sheaf!quasi-coherent}
Let $\X$ be a complex analytic space. By definition, a sheaf $\F$ of $\cO_\X$-modules is coherent if it is locally of finite presentation. We say that it is \emph{quasi-coherent} if it admits locally some presentation, and it is \emph{locally generated by sections} if we can find an open cover $\{\U_i\}_i$ of $\X$ such that $H^0(\U_i,\F) \otimes_\CC \cO_{\U_i} \to \F|_{\U_i}$ is surjective. 

Following \cite[Defn.~1.1.6]{KashiwaraKawai1981}, we say that a sheaf $\F$ of $\cO_\X$-modules is \emph{pseudo-coherent}\index{analytic sheaf!pseudo-coherent}\footnote{Compare this also with the usage of the term in \cite{BourbakiAC12}. In \cite{stacks}, the term is used for a different concept.} if, for every open subset $\U \subseteq \X$, every $\cO_\X$-submodule $\E \subseteq \F|_\U$ which is locally of finite type is actually coherent. 

Renaming for a moment the concept of quasi-coherence in \cite[Defn.~2.1.1]{Conrad2006}, we say that a sheaf $\F$ of $\cO_\X$-modules is \emph{locally ind-coherent} if we can find an open cover $\{\U_i\}_i$ of $\X$ such that each $\F|_{\U_i}$ is a (small filtered) direct limit of a system $\{\F_{i,\alpha}\}$ of coherent $\cO_{\U_i}$-modules. It follows from the proof of \cite[Lemma~2.1.8]{Conrad2006} that the image of an $\cO_\X$-linear map $h: \E \to \F$ from a coherent sheaf $\E$ to a locally ind-coherent sheaf $\F$ is coherent. Thus, every locally ind-coherent sheaf is pseudo-coherent.

Cartan's Theorem A states that $H^0(\X,\F) \otimes_\CC \cO_\X \to \F$ is surjective if $\X$ is a Stein space and $\F$ a coherent analytic sheaf.\footnote{One might be tempted to believe that a finite number of global sections in $H^0(X,\F)$ would be sufficient to generate $\F$. However, this is not true. For an example, take a sheaf of ideals $\I$ in $\cO_\CC$ generated by a sequence of functions as in the example in \cite[p.~181]{Henriksen1952}. The ideal sheaf $\I$ is coherent because, on every bounded domain $U \subseteq \CC$, finitely many of the generators are sufficient. But globally, there is no finite number of sections generating $\I$ because they would also generate $\I(\CC) \subseteq \cO_\CC(\CC)$, which is not the case because this ideal is not finitely generated.} Thus, after refining the open cover $\{\U_i\}$ in the definition of a locally ind-coherent sheaf $\F$ such that each $\U_i$ is a Stein space, we find that a locally ind-coherent sheaf $\F$ is also locally generated by sections. Conversely, if a sheaf $\F$ of $\cO_\X$-modules is both pseudo-coherent and locally generated by sections, then it is locally ind-coherent.

By \cite[Lemma~2.1.9]{Conrad2006}, kernels, cokernels, extensions, and tensor products of locally ind-coherent sheaves are locally ind-coherent. In particular, a locally ind-coherent\index{analytic sheaf!(locally) ind-coherent} sheaf is quasi-coherent in the usual sense because the kernel of the surjection $H^0(\U_i,\F) \otimes_\CC \cO_{\U_i} \to \F|_{\U_i}$ is locally generated by sections. Conversely, a quasi-coherent analytic sheaf $\F$ is locally the cokernel of a map between locally ind-coherent sheaves, so every quasi-coherent sheaf is locally ind-coherent. Thus, the two notions are equivalent.

If $\F$ is a quasi-coherent analytic sheaf on $\X$, then, for every $x \in \X$, we endow the stalk $\F_x$ with the sequence topology constructed above. If $h: \F \to \G$ is an $\cO_\X$-linear map of quasi-coherent analytic sheaves, then the induced map $h: \F_x \to \G_x$ is continuous.

If $X/\CC$ is a scheme locally of finite type and $\F$ is a quasi-coherent sheaf, then $\F^{an}$ is a quasi-coherent analytic sheaf.

\par\vspace{\baselineskip}

\noindent\textbf{A GAGA result for quasi-coherent sheaves.}\index{analytic sheaf!GAGA} Let $X/\CC$ be a scheme of finite type with analytification $X^{an}$. For every $\cO_X$-module $\F$, there is a canonical and functorial map 
\begin{equation}\label{cohom-comp}
 H^p(X,\F) \to H^p(X^{an},\F^{an}). 
\end{equation}
As is well-known, this map is an isomorphism if $X/\CC$ is proper and $\F$ is coherent. However, it is true for quasi-coherent sheaves as well.\footnote{Although Lemma~\ref{qcoh-gaga} is in the preliminaries section, it is not needed in the proof of Proposition~\ref{analytify-diff-op-1} below. Instead, we need it directly in the proof of Theorem~\ref{perfect-G-calc-log-toroidal}.}

\begin{lemma}\label{qcoh-gaga}\note{qcoh-gaga}
 Let $X/\CC$ be proper, and let $\F$ be a quasi-coherent $\cO_X$-module. Then the map in \eqref{cohom-comp} is an isomorphism of $\CC$-vector spaces.
\end{lemma}
\begin{proof}
 Since $X$ is Noetherian, $\F$ is the direct limit of its coherent subsheaves. Thus, we have a directed system $\{\F_\alpha\}$ of coherent sheaves with colimit $\F$. Since $(-)^{an}$ is left adjoint to $\varphi_*$, it preserves colimits. Thus, $\F^{an}$ is the colimit of the directed system $\{\F_\alpha^{an}\}$ of coherent analytic sheaves. Since the map in \eqref{cohom-comp} is functorial, we get a commutative diagram 
 \[
  \xymatrix{
   \varinjlim_\alpha H^p(X,\F_\alpha) \ar[r] \ar[d] & H^p(X,\F) \ar[d] \\
   \varinjlim_\alpha H^p(X^{an},\F_\alpha^{an}) \ar[r] & H^p(X^{an},\F^{an}) \\
  }
 \]
 of cohomology $\CC$-vector spaces. The upper horizontal arrow is an isomorphism by \cite[01FF]{stacks} because $X$ is a Noetherian topological space. The left vertical arrow is an isomorphism because \eqref{cohom-comp} is an isomorphism for each coherent sheaf $\F_\alpha$. Since $X^{an}$ is a compact Hausdorff space, the lower horizontal arrow is an isomorphism by \cite[Thm.~4.12.1]{Godement1973}; namely, cohomology with compact support is the same as cohomology on a compact space.
\end{proof}

\par\vspace{\baselineskip}

\noindent\textbf{Density.} Let $X$ be an affine scheme of finite type over $\CC$, with analytification $X^{an}$. Let $\F$ be a quasi-coherent sheaf on $X$, and let $x \in X^{an}$ be a point. Then we have an induced map $q: H^0(X,\F) \to \F_{\varphi(x)} \to \F_x^{an}$.

\begin{lemma}\label{density-lemma}\note{density-lemma}
 The image of $q: H^0(X,\F) \to \F_x^{an}$ is dense in $\F_x^{an}$ for the sequence topology.
\end{lemma}
\begin{proof}
 First, consider the map $q: P_n := \CC[x_1,...,x_n] \to \CC\{x_1,...,x_n\} =: K_n$. Let $A \subseteq K_n$ be a closed subset which contains the image of $q$. Then, for every $t \in \RR^n_+$, the intersection $A \cap B_t\{x_1,...,x_n\}$ is a closed subset of $B_{t;n} := B_t\{x_1,...,x_n\}$ which contains the image of $q$ as well. If $f \in B_{t;n}$, set $f_k := \sum_{|\nu| \leq k} a_\nu x^\nu \in \CC[x_1,...,x_n] \subseteq A \cap B_{t;n}$, where $|\nu| := \nu_1 + ... + \nu_n$. Then $(f_k)_k$ converges to $f$, so $f \in A \cap B_{t;n}$ because $A$ is closed, and thus $A \cap B_{t;n} = B_{t;n}$. This implies $A = K_n$.
 
 Next, let $R = \Gamma(X,\cO_X)$ be the coordinate ring of $X$. Then we can find a presentation $R = \CC[x_1,...,x_n]/I$. After a coordinate transformation, we can assume that $x \in X$ corresponds to $\{x_1 = ... = x_n = 0\}$. The analytic local ring of $X^{an}$ at $x$ is $K_n/I \cdot K_n$, and we have a commutative diagram 
 \[
  \xymatrix{
   \CC[x_1,...,x_n] \ar[d] \ar[r] & \CC\{x_1,...,x_n\} \ar[d] \\
   \CC[x_1,...,x_n]/I \ar[r] & K_n/I \cdot K_n. \\
  }
 \]
 Since $K_n \to K_n/I \cdot K_n$ is continuous and surjective, we find that the image of $\CC[x_1,...x_n]/I$ is dense.
 
 Considering the product topology on $\cO_{X,x}^{an,\oplus k}$, we obtain the claim for finite free $\cO_X$-modules. Then, we obtain it for coherent sheaves $\F$ from some surjection $\cO_X^{\oplus k} \to \F$ since $\cO_{X,x}^{an,\oplus k} \to \F_x^{an}$ is continuous and surjective. Finally, if $\F$ is quasi-coherent, we have $\F = \varinjlim \F_\alpha$ for the coherent subsheaves $\F_\alpha$ of $\F$, and at the same time, we have $\F_x^{an} = \varinjlim \F_{\alpha,x}^{an}$. Since the image of $H^0(X,\F_\alpha) \to \F_{\alpha,x}^{an}$ is dense, we obtain the claim for $\F$ because the sequence topology on $\F_x^{an}$ is the colimit topology of the sequence topologies on the $\F_{\alpha,x}^{an}$.
\end{proof}

\par\vspace{\baselineskip}

\noindent\textbf{Locally bounded operators.}\index{operator!locally bounded}
For technical reasons, we need to introduce the following notion. Let $\F$ and $\G$ be two pseudo-coherent analytic sheaves on $\X$. Then a map $h: \F \to \G$ of sheaves of sets, not necessarily $\cO_\X$-linear, is called \emph{locally bounded} if the following holds: For every open subset $\U \subseteq \X$ and every coherent analytic subsheaf $\E$ of $\F|_\U$, there is an open cover $\{\U_i\}$ of $\U$ and a coherent analytic subsheaf $\E'_i$ of $\G|_{\U_i}$ such that $h$ maps $\E|_{\U_i}$ into $\E'_i \subseteq \G|_{\U_i}$.

\par\vspace{\baselineskip}

\noindent\textbf{Derivations and analytification.} If $f: \X \to \cS$ is a morphism of complex analytic spaces, then there is a coherent analytic sheaf $\Omega^1_{\X/\cS}$ together with a canonical $f^{-1}(\cO_\cS)$-linear derivation 
$$d^{an}: \cO_\X \to \Omega^1_{\X/\cS}.$$
Unlike the simpler algebraic case, it is only universal among $f^{-1}(\cO_\cS)$-linear derivations from $\cO_\X$ into a coherent analytic sheaf $\E$, but (most likely) not among derivations into an arbitrary $\cO_\X$-module. According to \cite[Exp.~14]{SemCartan13}, which unfortunately does not discuss the universal property, the canonical derivation can be constructed as follows: Let 
$$\Delta: \X \to \cP := \X \times_\cS \X$$
be the diagonal, and let $\I$ be its ideal sheaf as a closed analytic subspace. The two projections give rise to two maps $\mathrm{pr}_1^*,\mathrm{pr}_2^*: \cO_\X \to \cO_\cP$, whose composition with $\cO_\cP \to \cO_\Delta$ is the same. Then $d^{an} = \mathrm{pr}_2^* - \mathrm{pr}_1^*$ is their difference, considered as a map $\cO_\X \to \I/\I^2$. The latter is supported on $\Delta \subseteq \cP$, so it can be considered as a sheaf on $\X$.

Now let $f: X \to S$ be a morphism of finite type between schemes $X$ and $S$ which are of finite type over $\CC$. Let $f^{an}: X^{an} \to S^{an}$ be the analytification. The (now for all $\cO_X$-modules) universal derivation 
$$d: \cO_X \to \Omega^1_{X/S}$$
arises from the same construction via the algebraic diagonal $\Delta: X \to P := X \times_S X$. Thus, we have a commutative diagram 
\[
  \xymatrix{
   \varphi_*\cO_X^{an} \ar[r]^{\varphi_*d^{an}} & \varphi_*\Omega^1_{X^{an}/S^{an}} \\
   \cO_X \ar[u] \ar[r]^d & \Omega^1_{X/S} \ar[u] \\
  }
 \]
 and an isomorphism $\varphi^*\Omega^1_{X/S} \cong \Omega^1_{X^{an}/S^{an}}$.

\par\vspace{\baselineskip}

\noindent\textbf{First-order differential operators.}\index{first-order differential operator!analytification} Let $f: \X \to \cS$ be a morphism of complex analytic spaces. Let $\F$ and $\G$ be two quasi-coherent sheaves on $\X$. Then a \emph{differential operator of first order relative to $\X/\cS$} is an $f^{-1}(\cO_\cS)$-linear map $D: \F \to \G$ such that, for every local section $a \in \cO_\X$, the map $D_a: \F \to \G, \: f \mapsto D(af) - aD(f)$, is $\cO_\X$-linear. These operators form a sheaf $\Diff^1_{\X/\cS}(\F,\G)$ of $\cO_\X$-modules under the $\cO_\X$-action given by $(a \cdot D)(f) := a \cdot D(f)$.

\begin{lemma}
 Let $\F$ and $\G$ be coherent analytic sheaves. Then $\Diff^1_{\X/\cS}(\F,\G)$ is coherent.
\end{lemma}
\begin{proof}
 Let $\tilde\Omega^1_{\X/\cS}$ be the sheaf of K\"ahler differential forms of the morphism $f: \X \to \cS$ of locally ringed spaces. Its universal property gives rise to a canonical $\cO_\X$-linear homomorphism $\chi: \tilde \Omega^1_{\X/\cS} \to \Omega^1_{\X/\cS}$ to the usual analytic differential forms which induces an isomorphism 
 $$\cH om(\Omega^1_{\X/\cS},\G) \to \cH om(\tilde\Omega^1_{\X/\cS},\G)$$
 for every coherent analytic sheaf $\G$ by the universal property of $\Omega^1_{\X/\cS}$. By \cite[0G3V]{stacks}, applied to the morphism $f: \X \to \cS$ of locally ringed spaces, we obtain a short exact sequence 
 $$0 \to \tilde\Omega^1_{\X/\cS} \otimes \F \to \tilde\cP^1_{\X/\cS}(\F) \to \F \to 0$$
 of $\cO_\X$-modules with the sheaf of principal parts in the middle. Applying $\cH om(-,\G)$, we find a long exact sequence
 $$0 \to \cH om(\F,\G) \xrightarrow{a} \Diff^1_{\X/\cS}(\F,\G) \xrightarrow{b} \cH om(\tilde\Omega^1_{\X/\cS} \otimes \F,\G) \xrightarrow{c} \E xt^1(\F,\G).$$
 From the tensor-hom adjunction, we find that the third sheaf is coherent. Since the last one is coherent as well, so is the kernel of the map $c$ between them. But then the image of $b$ is coherent, and thus the sheaf of first-order differential operators is an extension of coherent sheaves, so coherent itself.
\end{proof}

We use our notion of locally bounded operators in the following result.

\begin{lemma}\label{diff-op-continuity}\note{diff-op-continuity}
 Let $D: \F \to \G$ be a locally bounded first-order differential operator relative to $\X/\cS$ between quasi-coherent analytic sheaves. Then, for every $x \in \X$, the induced map $D: \F_x \to \G_x$ is continuous for the sequence topologies on $\F_x$ and $\G_x$.
\end{lemma}
\begin{proof}
 We start with the case where both $\F$ and $\G$ are coherent. In this case, every operator is locally bounded, so this condition is empty. By \cite[III,\ \S 4]{GR1971}, a derivation $D: \cO_{\X,x} \to \G_x$ is continuous for the sequence topologies if $\G$ is coherent. Every first-order differential operator $D: \cO_{\X,x} \to \G_x$ can be decomposed as $D = D_0 + D_1$ with an $\cO_{\X,x}$-linear map $D_0(a) := aD(1)$, and a derivation $D_1(a) := D(a) - D_0(a)$. Since both $D_0$ and $D_1$ as well as the summation map $\G_x \times \G_x \to \G_x$ are continuous, we find that every first-order differential operator $D: \cO_{\X,x} \to \G_x$ is continuous. Now if $\F$ is coherent, we can find a surjection $\pi: \cO_{\X,x}^{\oplus n} \to \F_x$. The sequence topology on $\F_x$ is the quotient topology along $\pi$, so $D$ is continuous if and only if $D \circ \pi$ is continuous. The sum $\cO_{\X,x}^{\oplus n}$ carries the product topology, so $D$ is continuous if and only if the maps $D \circ \pi_i$, with $\pi_i: \cO_{\X,x} \to \cO_{\X,x}^{\oplus n} \to \F_x$ the $n$ summands, are continuous. However, they are first-order differential operators, so they are continuous. In the more general quasi-coherent case, we write $\F_x = \varinjlim \F_{\alpha,x}$. Let $U \subseteq \G_x$ be an open subset. By our definition of the sequence topology on $\F_x$, we have to show that $D^{-1}(U) \cap \F_{\alpha,x}$ is open for all $\alpha$. Since $D$ is locally bounded by assumption, there is a finitely generated submodule $\G_{\alpha,x}$ of $\G_x$ which contains $D(\F_{\alpha,x})$. Then $D^{-1}(U) \cap \F_{\alpha,x} = D^{-1}(\G_{\alpha,x} \cap U) \cap \F_{\alpha,x} = D_\alpha^{-1}(\G_{\alpha,x} \cap U)$ for $D_\alpha: \F_{\alpha,x} \to \G_{\alpha,x}$. Since both $D_\alpha$ and $\G_{\alpha,x} \to \G_x$ are continuous, the claim follows.
\end{proof}

\section{Analytification of first-order differential operators}

Let $f: X \to S$ be a morphism of finite type between schemes $X$ and $S$ which are of finite type over $\CC$.\footnote{Although some of the theory works for locally Noetherian schemes as well, we restrict here to the Noetherian case as we shall only be interested in Noetherian schemes anyway. Then we can use that every quasi-coherent sheaf on a scheme $X$ is the colimit of its coherent subsheaves.} Let $\F$ and $\G$ be two quasi-coherent sheaves on $X$. Then a differential operator of first order relative to $X/S$ is an $f^{-1}(\cO_S)$-linear map $D: \F \to \G$ such that, for every local section $a \in \cO_X$, the map $D_a: \F \to \G, \: f \mapsto D(af) - aD(f)$, is $\cO_X$-linear. In this section, we prove the following result.

\begin{prop}\label{analytify-diff-op-1}\note{analytify-diff-op-1}\index{first-order differential operator!analytification}
 In the above situation, we have the following:
 \begin{enumerate}[label=\emph{(\alph*)}]
  \item\label{analyt-exist} There is a unique locally bounded differential operator 
 $$D^{an}: \F^{an} \to \G^{an}$$
 of first order relative to $X^{an}/S^{an}$ such that 
 $$\varphi_*D^{an} \circ c_\F = c_\G \circ D: \quad \F \to \varphi_*\G^{an}$$
 where $c_\F: \F \to \varphi_*\F^{an}$ and $c_\G: \G \to \varphi_*\G^{an}$ are the adjunction maps. 
 \item Equivalently, we can characterize $D^{an}$ as the unique locally bounded differential operator $D^{an}: \F^{an} \to \G^{an}$ of first order relative to $X^{an}/S^{an}$ such that 
 $$D^{an} \circ c'_\F = c'_\G \circ \varphi^{-1}D: \quad \varphi^{-1}\F \to \G^{an}$$
 where $c'_\F: \varphi^{-1}\F \to \F^{an}$ and $c'_\G: \varphi^{-1}\G \to \G^{an}$ are the adjunction maps.
 \item\label{analyt-precompose} If $\E$ is quasi-coherent and $h: \E \to \F$ is $\cO_X$-linear, then $D^{an} \circ h^{an} = (D \circ h)^{an}$.
 \item\label{analyt-postcompose} If $\cH$ is quasi-coherent and $h: \G \to \cH$ is $\cO_X$-linear, then $h^{an} \circ D^{an} = (h \circ D)^{an}$.
 \item\label{analyt-two-diff-comp}  If $E: \G \to \cH$ is another differential operator of first order relative to $X/S$ between quasi-coherent sheaves, then $E^{an} \circ D^{an} = 0$ if and only if $E \circ D = 0$.
 \end{enumerate}
\end{prop}

We will obtain this as a consequence of a number of preliminary results. 

\begin{lemma}\label{locally-bounded-uniqueness}\note{locally-bounded-uniqueness}
 Let $\F$ and $\G$ be two quasi-coherent sheaves on $X$, and let $h_1,h_2: \F^{an} \to \G^{an}$ be two locally bounded maps of sheaves of sets such that $h_1,h_2: \F^{an}_x \to \G^{an}_x$ are continuous for the sequence topologies. Assume that $\varphi_*h_1 \circ c_\F = \varphi_*h_2 \circ c_\F$ as maps $\F \to \varphi_*\F^{an} \to \varphi_*\G^{an}$. Then $h_1 = h_2$. In particular, if $h: \F \to \G$ is $\cO_X$-linear, then $h^{an}: \F^{an} \to \G^{an}$ is the unique $\cO_X^{an}$-linear map with $\varphi_*h^{an} \circ c_\F = c_\G \circ h$.
\end{lemma}
\begin{proof}
 It suffices to show that $h_1 = h_2$ at the stalks $\F_x^{an} \to \G_x^{an}$. Let $U \subseteq X$ be an affine open subset with $x \in U$. Since $\varphi_*h_1 \circ c_\F = \varphi_*h_2 \circ c_\F$, the compositions of the two maps at the stalks with $H^0(U,\F) \to \F_x^{an}$ are equal. Now let $A_x \subseteq \F_x^{an}$ be the subset where $h_1$ and $h_2$ are equal, i.e., we have just seen that the image of $q: H^0(U,\F) \to \F_x^{an}$ is contained in $A_x$. Let $M \subseteq \F_x^{an}$ be a finitely generated $\cO_{X^{an},x}$-submodule.\footnote{The proof is so subtle because we do not know if the sequence topology on $\G_x^{an}$ is Hausdorff.} Since $\F$ is pseudo-coherent, we can find, locally around $x$, a coherent subsheaf $\E \subseteq \F$ with $\E_x = M$. Since both $h_1$ and $h_2$ are locally bounded, we can find, again locally around $x$, another coherent subsheaf $\E' \subseteq \G$ such that both $h_1$ and $h_2$ map $\E$ into $\E'$. Let $N \subseteq \G_x^{an}$ be the stalk of $\E'$ at $x$. Then we have two induced continuous maps $h_1,h_2: M \to N$; continuity of these maps follows from the fact that the sequence topology on $N$ is the induced subspace topology from $\G_x^{an}$. Since $N$ is finitely generated, and hence its sequence topology is Hausdorff, the set where $h_1$ and $h_2$ agree on $M$ is closed, i.e., $A_x \cap M$ is closed. Then $A_x \subseteq \F_x^{an}$ is closed by the definition of the sequence topology on $\F_x^{an}$. Since $A_x$ contains a dense subset of $\F_x^{an}$, we have $A_x = \F_x^{an}$, and hence $h_1 = h_2$.
\end{proof}

\begin{prop}
 Let $D: \cO_X \to \F$ be a derivation relative to $X/S$ with values in a coherent (!) sheaf $\F$. Then there is a unique derivation $D^{an}: \cO_X^{an} \to \F^{an}$ relative to $X^{an}/S^{an}$ satisfying $\varphi_*D^{an} \circ c_\cO = c_\F \circ D$ for the adjunction maps $c_\cO: \cO \to \varphi_*\cO_X^{an}$ and $c_\F: \F \to \varphi_*\F^{an}$.
\end{prop}
\begin{proof}
 Consider the diagram 
 \[
  \xymatrix{
   \varphi_*\cO_X^{an} \ar[r]^{\varphi_*d^{an}} & \varphi_*\Omega^1_{X^{an}/S^{an}} \ar[r]^{\varphi_*h^{an}} & \varphi_*\F^{an} \\
   \cO_X \ar[u] \ar[r]^d & \Omega^1_{X/S} \ar[u] \ar[r]^h & \F. \ar[u] \\
  }
 \]
 We define $D^{an} := h^{an} \circ d^{an}$. This is a derivation over $X^{an}/S^{an}$ which makes the diagram commutative. If we have another derivation $D'$ with this property, then it factors as $D' = h' \circ d^{an}$. Using that $d: \cO_X \to \Omega^1_{X/S}$ is universal also for derivations with values in $\varphi_*\F^{an}$, we can show  that $h' = h^{an}$. Then we have $D' = D^{an}$.
\end{proof}

We now obtain Proposition~\ref{analytify-diff-op-1} in the special case where both $\F$ and $\G$ are coherent.

\begin{prop}\label{analytify-diff-op-coherent}\note{analytify-diff-op-coherent}
 Let $D: \F \to \G$ be a differential operator of first order relative to $X/S$ between coherent sheaves. Then there is a unique differential operator $D^{an}: \F^{an} \to \G^{an}$ of first order relative to $X^{an}/S^{an}$ satisfying $\varphi_*D^{an} \circ c_\F = c_\G \circ D$.
\end{prop}
\begin{proof}
 By definition, for every $a \in \cO_X$, the map $H_a(f) := D(af) - aD(f)$ is $\cO_X$-linear. This defines a sheaf homomorphism 
 $$H: \cO_X \to \cH om(\F,\G).$$
 By direct computation, this is a derivation relative to $X/S$. Thus, we have an analytification 
 $$H^{an}: \cO_X^{an} \to \cH om(\F,\G)^{an} \xrightarrow{\cong} \cH om(\F^{an},\G^{an}),$$
 where the right hand isomorphism comes from the fact that $\varphi: X^{an} \to X$ is flat and $\F$ is finitely presented (otherwise, it is just a map). Then we define 
 $$D^{an}: \varphi^{-1}\F \otimes_{\varphi^{-1}\cO_X} \cO_{X}^{an} \to \varphi^{-1}\G \otimes_{\varphi^{-1}\cO_X} \cO_X^{an}, \quad f \otimes a \mapsto D(f) \otimes a + H_a(f \otimes 1).$$
 If $b \in \varphi^{-1}\cO_X$, then $D^{an}(f \otimes ba) = D^{an}(fb \otimes a)$, showing that the map is well-defined. It is also a differential operator of order $1$ relative to $X^{an}/S^{an}$. Furthermore, it satisfies $\varphi_*D^{an} \circ c_\F = c_\G \circ D$. Then, one may show uniqueness by using the uniqueness of $H$ and of the composition $\varphi^{-1}\F \to \F^{an} \to \G^{an}$.
\end{proof}

Before we can prove Proposition~\ref{analytify-diff-op-1}, we need a final preparation.

\begin{lemma}\label{alg-diff-op-bounded}\note{alg-diff-op-bounded}
 Let $D: \F \to \G$ be a differential operator of first order relative to $X/S$ between quasi-coherent sheaves. Then $D$ is globally bounded, i.e., for every coherent subsheaf $\E \subseteq \F$, the image $D(\E)$ is contained in a coherent subsheaf $\E' \subseteq \G$.
\end{lemma}
\begin{proof}
 As a consequence of \cite[01PF]{stacks}, it is sufficient to work on an affine open subset $U \subseteq X$. Let $e_1,...,e_n$ be a set of generators of $\E$, and let $a_1,...,a_m$ be a set of generators of $\cO_X(U)$ as a $\CC$-algebra. Let $\E' \subseteq \G$ be the submodule generated by $D(e_j)$ for $1 \leq j \leq n$ and $D(a_ie_j)$ for $1 \leq i \leq m$, $1 \leq j \leq n$. Since $D$ is $\CC$-linear and satisfies $D(abf) - aD(bf) - bD(af) + abD(f) = 0$ for $a,b \in \cO_X$, $f \in \F$, we find that $D(\E(U))$ is contained in $\E'(U)$. Then there is a unique first-order differential operator $E: \E \to \E'$ with $\Gamma(U,E) = \Gamma(U,D)$ by Proposition~\ref{multilin-diff-op-shvs}. Applying Lemma~\ref{diff-loc-uniqueness} to standard open subsets of $U$ and, if necessary, enlarging $\E'$ for each standard open individually, we find that $D|_U = E$. Thus, $D(\E) \subseteq \E'$.
\end{proof}

\begin{proof}[Proof of Proposition~\ref{analytify-diff-op-1}]
 We start with \emph{\ref{analyt-exist}}.
 Let $\{\F_\alpha\}_\alpha$ be a directed system of coherent subsheaves of $\F$ with colimit $\F$. For each $\alpha$, let $\G_\alpha$ be the intersection of all (quasi-)coherent subsheaves of $\G$ which contain the image of $D_\alpha := D|_{\F_\alpha}: \F_\alpha \to \G$. Then $\G_\alpha$ is a coherent subsheaf of $\G$ by Lemma~\ref{alg-diff-op-bounded}, and for $\alpha \leq \beta$, we have $\G_\alpha \subseteq \G_\beta$. Proposition~\ref{analytify-diff-op-coherent} yields a differential operator $D_\alpha^{an}: \F_\alpha^{an} \to \G_\alpha^{an}$ for each $\alpha$. For $\alpha \leq \beta$, let us write $i_{\alpha\beta}: \F_\alpha \to \F_\beta$ respective $j_{\alpha\beta}: \G_\alpha \to \G_\beta$ for the inclusions. Then we have $D_\beta^{an} \circ i_{\alpha\beta}^{an} = j_{\alpha\beta}^{an} \circ D_\alpha^{an}$ because both are the unique analytification of the differential operator $D_\beta \circ i_{\alpha\beta} = j_{\alpha\beta} \circ D_\alpha$. Now $\F^{an}$ is the direct limit of $\{\F_\alpha^{an}\}_\alpha$ in the category of abelian sheaves on $X^{an}$, so there is an induced map $D^{an}: \F^{an} \to \G^{an}$ coming from the compositions $\F^{an}_\alpha \to \G^{an}_\alpha \to \G^{an}$. It is a differential operator of first order relative to $X^{an}/S^{an}$ because its restriction to each $\F_\alpha^{an} \subseteq \F^{an}$ is. It satisfies $\varphi_*D^{an} \circ c_\F = c_\G \circ D$ because the compositions of the two maps with $i_{\alpha\beta}: \F_\alpha \to \F$ are the same for all $\alpha$, and $\F$ is the direct limit of $\{\F_\alpha\}_\alpha$ in the category of abelian sheaves on $X$.
 
 To show that $D^{an}$ is locally bounded, let $\U \subseteq X^{an}$ be an open subset, and let $\E \subseteq \F^{an}|_\U$ be a coherent analytic subsheaf. Let $\{\U_i\}_i$ be an open cover of $\U$ such that each $\E|_{\U_i}$ is finitely generated. By refining $\{\U_i\}_i$ if necessary, we can assume that we have, for each index $i$, an index $\alpha_i$ with $\E|_{\U_i} \subseteq \F_{\alpha_i}^{an}|_{\U_i}$. Namely, at every point $x \in \U_i$, the germs of the finitely many generators of $\E|_{\U_i}$ must all be contained in the stalk $\F^{an}_{\alpha,x}$ for some $\alpha$. Then $D^{an}$ maps $\E|_{\U_i}$ into the coherent analytic subsheaf $\G^{an}_{\alpha_i}|_{\U_i}$ of $\G^{an}|_{\U_i}$. Thus, $D^{an}$ is locally bounded. Uniqueness follows from Lemma~\ref{locally-bounded-uniqueness}, using Lemma~\ref{diff-op-continuity}.

 For \emph{\ref{analyt-precompose}} and \emph{\ref{analyt-postcompose}}, first note that the statement holds if all involved modules are coherent. Then, to see \emph{\ref{analyt-precompose}}, use a representation of $\E$ as a direct limit of coherent subsheaves $\E_\alpha$, let $\F_\alpha := h(\E_\alpha)$, and let $\G_\alpha$ be constructed from this $\F_\alpha$ like in the construction of $D^{an}$ above. To see \emph{\ref{analyt-postcompose}}, start with a representation of $\F$ as a direct limit of coherent subsheaves $\F_\alpha$, let $\G_\alpha$ be as above, and let $\cH_\alpha := h(\G_\alpha)$. Then use that this $\cH_\alpha$ is the smallest submodule of $\cH$ which contains $h \circ D(\F_\alpha)$.
 
 For \emph{\ref{analyt-two-diff-comp}}, note that $E^{an} \circ D^{an}$ is locally bounded and induces continuous maps $\F_x^{an} \to \cH_x^{an}$ on the stalks. If $E \circ D = 0$, then $\varphi_*E^{an} \circ \varphi_*D^{an} \circ c_\F = 0$, so $E^{an} \circ D^{an} = 0$ by Lemma~\ref{locally-bounded-uniqueness}. For the converse, first note that $c_\cH: \cH \to \varphi_*\cH^{an}$ is injective. Namely, using the Krull intersection theorem, we find that $\E_{\varphi(x)} \to \E^{an}_x$ is injective for $\E$ coherent; then, however, it is easy to generalize to quasi-coherent $\E$ using the direct limit description. Now the injectivity of $c_\cH$ implies $E \circ D = 0$.
\end{proof}



\chapter{Toroidal crossing spaces}
\label{toroidal-cr-sp-sec}\note{toroidal-cr-sp-sec}

In this chapter, we give a synopsis of the basic theory of toroidal crossing spaces. The first appearance of the name and the essential idea of the concept is by Schr\"oer and Siebert in \cite{SchroerSiebert2006}, which has a slightly different focus than we have. Most theory was developed by Gross and Siebert in \cite{GrossSiebertI}, but toroidal crossing spaces do not show up as a concept and a definition there. Finally, in \cite{FFR2021}, we have used toroidal crossing spaces as a concept, but the treatment is rather sloppy. As of now, toroidal crossing spaces are the single most important source of generically log smooth families, the main subject of this text.

A toroidal crossing space is a generalization of a \emph{normal} crossing space which is particularly suitable to the study of degenerations and more natural from the perspective of logarithmic geometry. 

Let $R$ be a discrete valuation $\kk$-algebra with residue field $\kk$, choose a uniformizer $t \in \m_R$, and write $S = \Spec R$.
A normal crossing space $V$ has local models $\{z_1 \cdot ... \cdot z_r = 0\}$. Every semistable degeneration $f: X \to S$ (with central fiber $V$) becomes a log smooth and saturated as well as vertical\footnote{In the sense of \cite[I, Defn.~4.3.1]{LoAG2018}.} morphism once we endow both the base and the total space with the compactifying log structure defined by $t = 0$, as usual in the \'etale topology, as also the local models are in the \'etale topology. This induces a log structure $\M_{X_0}$ on $V$, turning it into a log scheme $X_0 = (V,\M_{X_0})$, and a log smooth morphism to $S_0 = \Spec (\NN \to \kk)$. Different semistable degenerations to $V$ may induce different log structures on $V$, but they all have the same ghost sheaf, namely $\cP := \nu_*\underline\NN_{\tilde V}$ where $\nu: \tilde V \to V$ is the normalization. The image of the generator $1_{S_0} \in \M_{S_0}$ in $\cP$ is always $\nu_*(1_{\tilde V})$, where $1_{\tilde V} \in \underline\NN_{\tilde V}$ is the global section which is equal to $1$ in every stalk. The isomorphism classes of log structures on $V$ together with the structure of a log morphism to $S_0$ which can arise from local semistable degenerations to $V$ form not only a presheaf but indeed a sheaf $\cL\cS_V$ on $V$, which turns out to be isomorphic to the subsheaf of local generators inside the line bundle $\T^1_V = \E xt^1(\Omega^1_V,\cO_V)$ on the double locus $D = V_{sing}$ of $V$. $d$-semistability of $V$, a well-known necessary condition for the existence of a \emph{global} semistable degeneration to $V$, then becomes equivalent to the existence of a global section in $\cL\cS_V$. Generically log smooth families over $S_0$ now arise from sections of $\cL\cS_V$ on $V \setminus Z$ which cannot be extended across the log singular locus $Z \subseteq X$.

Toroidal crossing spaces arise when we replace semistable degenerations with the more general degenerations $f: X \to S$ which become log smooth, saturated, and vertical once we endow source and target with the compactifying log structure defined by $t = 0$.

\begin{defn}\index{toroidal crossing degeneration}
 A \emph{toroidal crossing degeneration} is a separated morphism of finite type $f: X \to S$ of some constant relative dimension $d$ which becomes log smooth, saturated, and vertical once we endow $X$ and $S$ with the compactifying log structure defined by $t = 0$. In particular, the log structure $\M_X$ is fine and saturated, the morphism $f: X \to S$ is flat, the generic fiber $X_\eta$ is smooth, the central fiber $X_0$ is Gorenstein, and the total space $X$ is normal and Cohen--Macaulay.
\end{defn}

The following construction yields such degenerations:

\begin{constr}\label{prototype-constr}\note{prototype-constr}\index{toroidal crossing space!prototype}
 Let $N \cong \ZZ^r$ be a lattice with dual $M = \mathrm{Hom}(N,\ZZ)$, and set $N_\RR = N \otimes_\ZZ \RR$. Let $\sigma \subseteq N_\RR$ be a lattice polytope of full dimension $r$; in particular, it is bounded. The cone over $\sigma \times \{1\}$ is 
 $$C(\sigma) = \{(sn,s) \ | n \in \sigma, s \geq 0 \ \} \subseteq N_\RR \times \RR ;$$
 it gives rise to a sharp toric monoid $P_\sigma := C(\sigma)^\vee \cap (M \times \ZZ)$, which consists of the lattice points in the dual cone $C(\sigma)^\vee$. 
 The monoid $P_\sigma$ defines the affine toric variety 
 $$U(\sigma) = \Spec \kk[P_\sigma];$$ 
 it is Gorenstein according to the criterion in \cite[p.~126]{Oda1988} because we have $\mathrm{int}(P_\sigma) = \rho_\sigma + P_\sigma$. Here, $\mathrm{int}(P_\sigma)$ denotes the \emph{interior}---the elements which are not in any proper face of $P_\sigma$---, and $\rho_\sigma = (0,1) \in P_\sigma \subseteq M \times \ZZ$ is the so-called \emph{Gorenstein degree}. The toric boundary of $U(\sigma)$ is 
 $$V(\sigma) := \Spec \kk[P_\sigma]/(z^{\rho_\sigma}) = \{z^{\rho_\sigma} = 0\}  \subseteq U(\sigma)$$
 since the ideal of the toric boundary---i.e., $\kk[\mathrm{int}(P_\sigma)]$---is generated by $z^{\rho_\sigma}$. We obtain a family
 $$U(\sigma) \to \bAA^1_t, \quad t \mapsto z^{\rho_\sigma},$$
 with central fiber $V(\sigma)$ and smooth generic fiber.
 We denote the canonical log structure of $U(\sigma)$---it is the divisorial log structure of the inclusion $V(\sigma) \subseteq U(\sigma)$---by $\M_{U(\sigma)}$; then $(U(\sigma),\M_{U(\sigma)}) \to A_\NN$ is a log smooth, saturated, and vertical family, and $(V(\sigma),\M_{V(\sigma)}) \to S_0$ is log smooth, saturated, and vertical as well. We denote the base change of $U(\sigma) \to \bAA^1_t$ to $S = \Spec R$ along $\kk[t] \ni t \mapsto t \in R$ by 
 $$U_S(\sigma) \to S.$$
 On $U_S(\sigma)$, the induced log structure from $U(\sigma)$ coincides with the compactifying log structure defined by $t = 0$,\footnote{To see this, use \cite[I, Prop.~1.6.3]{LoAG2018} together with the proof of \cite[I, Lemma~1.6.7]{LoAG2018}. Use that a reflexive sheaf of rank $1$ on $U(\sigma)$ which is locally free outside $V(\sigma)$ is a line bundle if and only if its pull-back to $U_S(\sigma)$ is a line bundle.} as is the case on $S$.
\end{constr} 
\begin{ex}\label{two-comp-prototype}\note{two-comp-prototype}
 Let $N = \ZZ$ and $\sigma = [0,k] \subseteq \RR = \ZZ \otimes_\ZZ \RR$ for $k \geq 1$. Then $P_\sigma \subseteq M \oplus \ZZ$ is generated by the three vectors $(-1,k),(0,1),(1,0)$. Thus, $U(\sigma) = \Spec \kk[x,y,t]/(xy - t^k)$ is the two-dimensional $A_{k - 1}$-singularity. The central fiber of the family is $V(\sigma) \cong \Spec \kk[x,y]/(xy)$; in particular, it does not depend on $k$.
\end{ex}

The name \emph{toroidal crossing degeneration} is justified by the analogy with normal crossing degenerations through the following result. In other words, a toroidal crossing degeneration has \emph{toric} crossings \'etale locally.

\begin{lemma}
 Let $f: X \to S$ be a separated morphism of finite type. Then $f: X \to S$ is a toroidal crossing degeneration if and only if it is, locally in the \'etale topology on $X$, isomorphic to families of the form $U_S(\sigma) \times_S \bAA^{d - r}_S \to S$.\footnote{Note that this is independent of the choice uniformizer $t \in \m_R$ as a different choice differs by an automorphism of the base $S$.}
\end{lemma}
\begin{proof}
 We have already seen that every family of the specified form is a toroidal crossing degeneration. For the converse, the local model is clear around points $x \in X_\eta$ since $X_\eta$ is smooth. If $x \in X_0$, let $h: \NN \to P$ be the map on ghost stalks. By assumption, $P$ is a sharp toric monoid. By \cite[IV, Thm.~3.3.1]{LoAG2018}, $f: X \to S$ has a local model by $A_h \times \bAA^{d - r}: A_P \times \bAA^{d - r} \to A_\NN$. It thus suffices to show that $A_h: A_P \to A_\NN$ is of the form $U(\sigma) \to \bAA^1_t$. The map $h: \NN \to P$ is injective by \cite[IV, Thm.~3.3.1]{LoAG2018}. Verticality implies $h(1) \in \mathrm{int}(P)$, and saturatedness implies that $K_h = h(1) + P \subseteq \mathrm{int}(P)$ is a radical ideal. Then $K_h$ is the intersection of the prime ideals containing it, hence $K_h = \mathrm{int}(P)$. Thus, $P$ is a Gorenstein toric monoid, which are all of the form $P_\sigma$.
\end{proof}

As in the case of normal crossings, the local structure of the central fiber of a toroidal crossing degeneration $f: X \to S$ is dictated by the local models $U_S(\sigma) \to S$. Thus, they are all of the following form:

\begin{defn}\index{toroidal crossing singularities}
 A \emph{space with toroidal crossing singularities} is a separated scheme $V/\kk$ of finite type, pure of some dimension $d$, which is, locally in the \'etale topology, isomorphic to local models of the form $V(\sigma) \to \kk$. In particular, $V$ is reduced and Gorenstein.
\end{defn}

Our goal is now to understand toroidal crossing degenerations to a given space $V$ with toroidal crossing singularities.
When studying semistable degenerations $f: X \to S$ to a normal crossing space $V$, then the local model $z_1 \cdot ... \cdot z_r = t$ of $f: X \to S$ around a point $x \in V$ is dictated by the geometry of $V$: the number $r$ is precisely the number of \'etale local components of $V$ which meet at $x$. Thus, in a way, all semistable degenerations to $V$ are the same around $x \in V$ even though they do not induce the same log structure on $V$. Example~\ref{two-comp-prototype} shows that this is no longer true for a space with toroidal crossing singularities: each toroidal crossing degeneration $\{xy = t^k\}$ has the same central fiber $\{xy = 0\}$, but the total spaces have non-isomorphic singularities. However, to have a well-behaved theory analogous to the case of normal crossing spaces, we wish to prescribe the toroidal crossing degeneration up to \'etale local isomorphy, i.e., fix the local model at every point. Now a local model is always of the form $A_P \times \bAA^{d - r} \to A_\NN$ so that it can be recovered from the induced map on ghost sheaves of the log structure (plus the dimension to recover $d - r$ where $r$ is the rank of $P$). Thus, we endow $V$ with a sheaf of monoids $\cP$, which is supposed to be the ghost sheaf of the log structure induced on the central fiber, and with a global section $\bar\rho \in \cP$, which is supposed to be the image of the generator $1_{S_0} \in \M_{S_0}$ under the induced map of log schemes. Then \cite[IV, Thm.~3.3.1]{LoAG2018} implies that any toroidal crossing degeneration whose induced map in the central fiber on the level of ghost sheaves is $\NN \to \cP, \: 1 \mapsto \bar\rho,$ has, around each point $\bar v \in V$, the same local model $A_P \times \bAA^{d - r} \to A_\NN$ with $P = \cP_{\bar v}$. A \emph{toroidal crossing space} is such a triple $(V,\cP,\bar\rho)$. In the following sections, we compile the basic theory of $(V,\cP,\bar\rho)$.

\vspace{\baselineskip}

The following result is sometimes useful.

\begin{lemma}\index{semi-log-canonical singularities}\label{slc-sing-tor-cr}\note{slc-sing-tor-cr}
 Let $V$ be a space with toroidal crossing singularities. Then $V$ has semi-log-canonical singularities.
\end{lemma}
\begin{proof}
 Our reference for semi-log-canonical singularities is \cite{Kollar2013}. Since $V$ is demi-normal and Gorenstein, what we have to show is that $(\bar V,\bar D)$ is log canonical, where $\nu: \bar V \to V$ is the normalization of $V$, and $\bar D \subset \bar V$ is the conductor locus. The construction of $(\bar V,\bar D)$ commutes with \'etale morphisms $f: Y \to V$, and the condition of having log canonical singularities is local in the \'etale topology. Thus, it is sufficient to show that $V(\sigma)$ has semi-log-canonical singularities. The normalization $\bar V(\sigma)$ of $V(\sigma)$ is the disjoint union of the irreducible components of $V(\sigma)$ because they are normal. The conductor locus $D(\sigma) \subseteq V(\sigma)$ is the locus where at least two components meet with its reduced scheme structure, and $\bar D(\sigma) \subseteq \bar V(\sigma)$ is the toric boundary of each irreducible component with its reduced scheme structure. Thus, $(\bar V(\sigma),\bar D(\sigma))$ is log canonical by \cite[Cor.~11.4.25]{Cox2011}, and $V(\sigma)$ is semi-log-canonical.
\end{proof}
\begin{cor}\label{tor-cr-Kodaira-van}\note{tor-cr-Kodaira-van}
 Let $V$ be a space with toroidal crossing singularities which is projective and of dimension $d$. Let $\cL$ be an ample line bundle on $V$. Then $H^k(V,\omega_V \otimes \cL) = 0$ for $k \geq 1$ and $H^k(V,\cL^\vee) = 0$ for $0 \leq k \leq d - 1$. Here, $\omega_V$ is the canonical bundle of the Gorenstein scheme $V$.
\end{cor}
\begin{proof}
 The first statement is explicit in \cite[Thm.~1.8]{Fujino2014} (for slc singularities which may not be Cohen--Macaulay), and the second statement is explicit in \cite[Cor.~6.6]{KSS2010} (over $\kk = \CC$, for weakly slc singularities which are required to be Cohen--Macaulay). Both statements are equivalent in our case by \cite[III, Thm.~7.6]{Hartshorne1977}.
\end{proof}

\section{Spaces with a toric ghost sheaf}

This is the first device we need to define on our way towards toroidal crossing spaces. It is a space $V$ together with a sheaf of monoids $\cP$ that might occur as the ghost sheaf of a log structure. First, we need to specify which sheaves $\cP$ can occur as the ghost sheaves of a fine and saturated log structure.

\begin{defn}
 A sheaf of monoids $\cP$ in the \'etale topology of $V$ is \emph{toric} if for every $\bar v \in V$, the stalk $\cP_{\bar v}$ is a toric monoid---i.e., fine, saturated, and $\cP_{\bar v}^{gp}$ is torsion-free---and there is an \'etale neighborhood $U$ and a surjection $P_U \to \cP|_U$ from a constant sheaf of monoids which is an isomorphism on the stalks at $\bar v$.
\end{defn}

With this definition, a log scheme $(V,\M)$ is fine and saturated if and only if $\overline\M$ is a toric sheaf of monoids. This follows from the argument given in \cite[Prop.~3.7]{GrossSiebertI}. Unfortunately, there in the original setting, the argument is wrong; namely, if we assume $\overline\M$ only to be fine---as Gross--Siebert does---the exact sequence in the proof of \cite[Prop.~3.7]{GrossSiebertI} does not necessarily split. Adding the saturatedness hypothesis remedies this situation (cf.~also \cite[Thm.~1.2.7]{LoAG2018}). 

\begin{defn}\index{toric ghost sheaf}\index{toric ghost sheaf!space with a toric ghost sheaf}
 A \emph{space with a toric ghost sheaf} is a pair $(V,\cP)$ where $V$ is a reduced variety and $\cP$ is a toric sheaf of monoids.
\end{defn}

\begin{ex}\label{prototype-toric-ghost}
 Let $V(\sigma)$ be as in Construction~\ref{prototype-constr}. Then with $\cP_\sigma := i^{-1}\overline\M_{U(\sigma)}$, we obtain a space with a toric ghost sheaf $(V(\sigma),\cP_\sigma)$. Here, $i: V(\sigma) \to U(\sigma)$ denotes the inclusion. 
\end{ex}

Our goal is to classify all log structures on $V$ that have the given sheaf $\cP$ as the ghost sheaf. We consider the log structure as a triple $(\M,\alpha,q)$ where $\M$ is a sheaf of monoids, $\alpha: \M \to \cO_V$ is a homomorphism turning $\M$ into a log structure, and $q: \M \to \cP$ is a homomorphism establishing $\cP$ as the ghost sheaf; an isomorphism between $(\M,\alpha,q)$ and $(\M',\alpha',q')$ is an isomorphism $\varphi: \M \cong \M'$ which is compatible with all data. Isomorphism classes of log structures do \emph{not} form a sheaf, so we define:

\begin{defn}
 Let $(V,\cP)$ be a space with a toric ghost sheaf. Then $\hat\cL_V$ is the sheaf of equivalence classes of log structures $(\M,\alpha,q)$ where two log structures are \emph{equivalent} if they are locally isomorphic.\footnote{The sections of $\hat\cL_V$ are not in general log structures on their locus of definition; rather, they are log structures on the open subsets of a cover which are equivalent on overlaps.}
\end{defn}

To compute the stalk $\widehat\cL_{V,\bar v}$ at a geometric point $\bar v \in V$, we construct a canonical injection 
$$\iota_{V,\bar v}: \widehat\cL_{V,\bar v} \to \E xt^1(\cP^{gp}, \cO_V^*)_{\bar v}$$
and determine its image. Here, $\cP^{gp}$ is the sheaf of groups associated to $\cP$. Note that the target is \emph{not} the extension group $\mathrm{Ext}^1(\cP^{gp}_{\bar v},\cO^*_{V,\bar v})$.

For $\iota_{V,\bar v}$, we actually have two constructions, which yield the same map. Firstly, a germ $[(\M,\alpha,q)] \in \widehat\cL_{V,\bar v}$---represented by a log structure defined in a neighborhood of $\bar v$---induces a short exact sequence 
$$1 \to \cO_V^* \to \M^{gp} \xrightarrow{q^{gp}} \cP^{gp} \to 0$$
of abelian groups and thus an extension class in $\E xt^1(\cP^{gp}, \cO_V^*)_{\bar v}$. Injectivity follows from the argument in \cite[Prop.~3.11, Cor.~3.12]{GrossSiebertI}. The second construction is more involved.

\begin{constr}
 
We set $P := \cP_{\bar v}$, which is a toric monoid. Because $\cP$ is a toric sheaf of monoids, there is---possibly after shrinking $V$ to an appropriate \'etale neighborhood---a surjection $P_V \to \cP$ which is an isomorphism at $\bar v$. Here, $P_V$ denotes the constant sheaf with stalk $P$. Groupification yields the surjection $P^{gp}_V \to \cP^{gp}$ of abelian sheaves, whose kernel $\R$ is called the \emph{relation sheaf}. Since $\E xt^1(P^{gp}_V,\cO_V^*) = 0$, we get a presentation 
\begin{equation}\label{ext-present}
 \cH om(P^{gp}_V,\cO_V^*) \to \cH om(\R,\cO_V^*) \to \E xt^1(\cP^{gp},\cO_V^*) \to 0.
\end{equation}
The stalk $\cH om(\R,\cO_V^*)_{\bar v}$ has an explicit description: An element $p \in P^{gp}$ induces an element in $\Gamma(V,\cP^{gp})$; its \emph{support} is 
$$\mathrm{supp}(p) = \{\bar x \in V \ | \ 0 \not= p_{\bar x} \in \cP_{\bar x}^{gp} \} \subseteq V,$$
which is already a closed subset of $V$, so we do not need to take its closure to obtain the support. For $n \in \ZZ \setminus \{0\}$, we have $\mathrm{supp}(np) = \mathrm{supp}(p)$, and for $p,q \in P^{gp}$ we have 
$$(V \setminus \mathrm{supp}(p)) \cap (V \setminus \mathrm{supp}(q)) \subseteq (V \setminus \mathrm{supp}(p + q)).$$
We say $p$ and $q$ are \emph{equifacial} if for every face $F \subseteq P$ it holds that $p + q \in F^{gp}$ implies $p \in F^{gp}$ and $q \in F^{gp}$. In this case, the above inequality is an equality.
According to \cite{GrossSiebertI}, after writing $j_p: V \setminus \mathrm{supp}(p) \subseteq V$ for the inclusion, the stalk $\cH om(\R,\cO_V^*)_{\bar v}$ is 
$$\left\{ (h_p)_{p \in P^{gp}} \ | \ h_p \in (j_{p*}\cO^*_{V \setminus \mathrm{supp}(p)})_{\bar v}, h_p \cdot h_q = h_{p + q} \ \mathrm{on} \ (V \setminus \mathrm{supp}(p)) \cap (V \setminus \mathrm{supp}(q))\right\};$$
at this point, there is a minor mistake in the statement of \cite[Prop.~3.14]{GrossSiebertI} where they write $V \setminus \mathrm{supp}(p + q)$ instead of $(V \setminus \mathrm{supp}(p)) \cap (V \setminus \mathrm{supp}(q))$. Equivalently, we can drop the $h_p$ for $p \notin P$ and write 
$$\cH om(\R,\cO_V^*)_{\bar v} = \left\{ (h_p)_{p \in P} \ | \ h_p \in (j_{p*}\cO^*_{V \setminus \mathrm{supp}(p)})_{\bar v}, \ h_p \cdot h_q = h_{p + q} \ \mathrm{on} \ V \setminus \mathrm{supp}(p + q)\right\}$$
for the stalk. Now let $[(\M,\alpha,q)] \in \hat\cL_{V,\bar v}$. The surjection $q^{gp}: \M^{gp}_{\bar v} \to \cP^{gp}_{\bar v}$ of stalks splits because $\cP_{\bar v}$ is toric; the splitting then induces locally around $\bar v$ a splitting $\sigma: \cP \to \M$ of the sheaf homomorphism $\M \to \cP$. The composition as in the diagram
\[
 \xymatrix{
  P_V \ar@{>>}[d] \ar@/^1pc/[rrd]^\varphi & & \\
  \cP \ar@/^/[r]^\sigma & \M \ar@/^/[l]^q \ar[r]^{\alpha} & \cO_V \\
 }
\]
yields a map $\varphi: P \to \Gamma(V,\cO_V)$; then $\varphi(p)|_{V\setminus\mathrm{supp}(p)}$ is an invertible function. Indeed, we have $p_{\bar x} = 0 \in \cP_{\bar x}^{gp}$ if and only if $\varphi(p)_{\bar x} \in \cO_{V,\bar x}^*$.  Thus, $\varphi(p)$ defines an element $h_p$ of $(j_{p*}\cO^*_{V \setminus \mathrm{supp}(p)})_{\bar v}$, and we get $(h_p)_{p \in P} \in \cH om(\R,\cO_V^*)_{\bar v}$. The image in $\E xt^1(\cP^{gp},\cO_V^*)_{\bar v}$ is independent of the choice of splitting $\sigma$; this yields our second construction of $\iota_{V,\bar v}$, which coincides with the first one.
\end{constr}

We determine the image of $\iota_{V,\bar v}$. Assume first that $\xi = [(h_p)_p]$ comes from a log structure $(\M,\alpha,q)$. Then each $h_p$ extends by $0$ to $V$---this means that there are $\tilde h_p := \varphi(p) \in \cO_{V,\bar v}$ such that $\tilde h_p \mapsto h_p$ under $\cO_{V,\bar v} \to (j_{p*}\cO_{V \setminus \mathrm{supp}(p)})_{\bar v}$ and $\tilde h_p \mapsto 0$ under $\cO_{V,\bar v} \to \cO_{\mathrm{supp}(p),\bar v}$.
Such an extension $\tilde h_p$ of $h_p$ is unique if it exists; in fact, the map 
$$\cO_{V,\bar v} \to \cO_{\mathrm{supp}(p),\bar v} \times (j_{p*}\cO_{V \setminus \mathrm{supp}(p)})_{\bar v}$$
is injective because $V$ is reduced. Conversely, if each $h_p$ admits an extension $\tilde h_p$ by $0$, then $P \to \cO_V, \ p \mapsto \tilde h_p$ is a chart of a log structure (in a neighborhood of $\bar v$) because it is a homomorphism due to uniqueness of $\tilde h_p$. Then $\xi = [(h_p)_p]$ is the image of this log structure under $\iota_{V,\bar v}$.

\begin{lemma}[cf.~\cite{GrossSiebertI},~3.14]
 A germ $\xi = [(h_p)_p] \in \E xt^1(\cP^{gp},\cO_V^*)_{\bar v}$ is in the image of 
 $$\iota_{V,\bar v}: \widehat\cL_{V,\bar v} \to \E xt^1(\cP^{gp},\cO_V^*)_{\bar v}$$
 if and only if every $h_p$ extends to $V$ by $0$. This property does not depend on the choice of representative $(h_p)_p \in \cH om(\R,\cO_V^*)_{\bar v}$.
\end{lemma}

We introduce an equivalence relation on the germs of log structures---the $\xi = [(h_p)_p]$ that extend by $0$---called the \emph{type}. Let $V(p) = \overline{V \setminus \mathrm{supp}(p)}$.

\begin{defn}[Type]\index{type!log structures of the same type}
 Two germs $\xi = [(h_p)_p], \xi' = [(h'_p)_p]$ that extend by $0$ are \emph{of the same type} if for every $p \in P$, we can find $e_p \in \cO^*_{V(p),\bar v}$ such that $h'_p = e_ph_p \in (j_{p*}\cO^*_{V \setminus \mathrm{supp}(p)})_{\bar v}$.
\end{defn}

Note the canonical restriction map $\cO^*_{V(p),\bar v} \to (j_{p*}\cO^*_{V \setminus \mathrm{supp}(p)})_{\bar v}$, which we apply to $e_p$. There is no condition relating $e_p$ and $e_q$; moreover, there is no condition for $p \in P^{gp} \setminus P$. Of course, being of the same type is an equivalence relation.

For later use in studying toroidal crossing spaces, we give the following lemma; it gives a nice---though rather peculiar---construction of two log structures which are of the same type. This result is implicit in \cite{GrossSiebertI}.

\begin{lemma}\label{same-type-lemma}
 Let $V(\sigma)$ be as in Example~\ref{prototype-toric-ghost}, and let $\tau_1,\tau_2: U \to V(\sigma)$ be two \emph{smooth} morphisms---with the same source and target---with $\tau_1(\bar u) = \tau_2(\bar u) = 0 \in V(\sigma)$ for a point $\bar u \in U$. Assume that we have an isomorphism $\varphi: \tau_1^{-1}\cP_\sigma \cong \tau_2^{-1}\cP_\sigma$ of toric ghost sheaves which respects the two isomorphisms $P_\sigma \cong (\tau_1^{-1}\cP_\sigma)_{\bar u}$ and $P_\sigma \cong (\tau_2^{-1}\cP_\sigma)_{\bar u}$. In particular, $(U,\tau_1^{-1}\cP_\sigma)$ is a space with a toric ghost sheaf.  Then the two log structures $(\tau_1)^*_{log}\M_{V(\sigma)}$ and $(\tau_2)^*_{log}\M_{V(\sigma)}$---induced from $\M_{V(\sigma)}$ via the two maps---are of the same type in $\bar u$.
\end{lemma}

\section{Pre-ghost structures}
We turn from the \emph{absolute} problem of classifying log structures to the \emph{relative} problem of classifying log morphisms to $S_0 = \Spec (\NN \to \kk)$. Once we fix a log structure $\M$ on $V$, the datum of a log morphism to $S_0$ is equivalent to the datum of a global section $\rho \in \Gamma(V,\M)$ with $\alpha(\rho) = 0$. At the level of ghost sheaves, this yields $\bar\rho := [\rho] \in \Gamma(V,\overline\M)$. When classifying log morphisms, we assume $\bar\rho \in \Gamma(V,\cP)$ to be fixed.

\begin{defn}\index{pre-ghost structure}\index{pre-ghost structure!space with a pre-ghost structure}
 A \emph{space with a pre-ghost structure} is a triple $(V,\cP,\bar\rho)$ where $(V,\cP)$ is a space with a toric ghost sheaf, and $\bar\rho \in \Gamma(V,\cP)$ is a global section which is nowhere nilpotent, i.e., for every $\bar v \in V$ and every $n > 0$, we have $n \cdot \bar\rho_{\bar v} \not= 0 \in \cP_{\bar v}$. We assume furthermore that $\bar\rho$ is primitive, i.e., for every $\bar v \in V$, we \emph{cannot} write $\bar\rho_{\bar v} = nq \in \cP_{\bar v}$ with $q \in \cP_{\bar v}$ and $n > 0$.
\end{defn}
\begin{rem}
 The notion of pre-ghost structure is not explicit in \cite{GrossSiebertI}; moreover, $\bar\rho$ is not explicitly assumed to be primitive in the corresponding discussion. The effect of primitivity is that $\cP^{gp}_{\bar v}/\ZZ\bar\rho_{\bar v}$ is a free abelian group; that will simplify our arguments below. We do not know if the condition is necessary.
\end{rem}
\begin{ex}\label{prototype-preghost}
 Let $V(\sigma)$ be as in Construction~\ref{prototype-constr}. According to Example~\ref{prototype-toric-ghost}, it can be enhanced to a space with a toric ghost sheaf $(V(\sigma),\cP_\sigma)$; the element $\rho_\sigma \in P_\sigma$ defines a global section $\bar\rho_\sigma \in \cP_\sigma$, which yields a space with a pre-ghost structure $(V(\sigma),\cP_\sigma,\bar\rho_\sigma)$.
\end{ex}

A log morphism from $V$ to $S_0$ is a quadruple $(\M,\alpha,q,\rho)$ where $\alpha: \M \to \cO_V$ is a log structure, $q: \M \to \cP$ exhibits $\cP$ as the ghost sheaf, and $\rho \in \Gamma(V,\M)$ satisfies $q(\rho) = \bar\rho$ and $\alpha(\rho) = 0$. 

\begin{defn}
 Let $(V,\cP,\bar\rho)$ be a space with a pre-ghost structure. Then $\widehat{\cL\cS}_V$ is the sheaf of equivalence classes of log morphisms $(\M,\alpha,q,\rho)$ where two log morphisms are \emph{equivalent} if they are locally isomorphic on $V$.
\end{defn}

Let $\xi \in \widehat{\cL\cS}_{V,\bar v}$ be a germ, given by a log morphism $(\M,\alpha,q,\rho)$ defined in a neighborhood of $\bar v \in V$. It gives rise to an extension
$$0 \to \cO_V^* \to \M^{gp}/\rho \to \cP^{gp}/\bar\rho \to 0;$$
thus, we obtain a canonical map
$$\bar\iota_{V,\bar v}: \widehat{\cL\cS}_{V,\bar v} \to \E xt^1(\cP^{gp}/\bar\rho,\cO_V^*)_{\bar v},$$
which is injective. We can reconstruct the log structure $(\M,\alpha,q)$ from the extension by pull-back along $\cP^{gp} \to \cP^{gp}/\bar\rho$. We can reconstruct $\rho$ as well; namely, we get the subsheaf $\ZZ\rho \subseteq \M^{gp}$ as the kernel of the map to $\M^{gp}/\rho$, and we choose the ``correct'' generator $\rho$ (instead of the wrong one $-\rho$) because only one of them maps to $\cP \subseteq \cP^{gp}$. We have a commutative diagram 
\begin{equation}
 \label{abs-rel-embedding-square}
 \xymatrix{
  \widehat{\cL\cS}_{V,\bar v} \ar[r]^-{\bar\iota_{V,\bar v}} \ar[d] & \E xt^1(\cP^{gp}/\bar\rho,\cO_V^*)_{\bar v} \ar[d] \\
  \widehat{\cL}_{V,\bar v} \ar[r]^-{\iota_{V,\bar v}} & \E xt^1(\cP^{gp},\cO_V^*)_{\bar v} \\
 }
\end{equation}
where the left vertical arrow is the forgetful map $(\M,\alpha,q,\rho) \mapsto (\M,\alpha,q)$.

The explicit description of $\E xt^1(\cP^{gp},\cO_V^*)_{\bar v}$ in the absolute case has an analog for the sheaf $\E xt^1(\cP^{gp}/\bar\rho,\cO_V^*)_{\bar v}$ in the relative case. Because the surjection $P_V^{gp} \to \cP^{gp}$ is an isomorphism on the stalk at $\bar v$, there is---for $V$ small enough and connected---a unique $\tilde\rho \in P^{gp} = \Gamma(V,P_V^{gp})$ with $\tilde\rho \mapsto \bar\rho \in \Gamma(V,\cP^{gp})$. This gives a short exact sequence
$$0 \to \R \to P_V^{gp}/\tilde\rho \to \cP^{gp}/\bar\rho \to 0$$
with the same sheaf of relations $\R$ as above.
Using $\E xt^1(P_V^{gp}/\tilde\rho,\cO_V^*) = 0$, we find a morphism 
\[
 \xymatrix{
 \cH om(P_V^{gp}/\tilde\rho,\cO_V^*) \ar[r] \ar[d] & \cH om(\R,\cO_V^*) \ar[r] \ar@{=}[d] & \E xt^1(\cP^{gp}/\bar\rho,\cO_V^*) \ar[d] \ar[r] & 0 \\
 \cH om(P_V^{gp},\cO_V^*) \ar[r] & \cH om(\R,\cO_V^*) \ar[r] & \E xt^1(\cP^{gp},\cO_V^*) \ar[r] & 0 \\
 }
\]
of presentations. Thus, a germ of a log morphism $(\M,\alpha,q,\rho)$ can be represented by functions $(h_p)_p \in \cH om(\R,\cO_V^*)_{\bar v}$ as above---we just divide by a smaller subgroup. The result is:

\begin{lemma}
 A germ $\xi = [(h_p)_p] \in \E xt^1(\cP^{gp}/\bar\rho,\cO_V^*)_{\bar v}$ is in the image of 
 $$\bar\iota_{V,\bar v}: \widehat{\cL\cS}_{V,\bar v} \to \E xt^1(\cP^{gp}/\bar\rho,\cO_V^*)_{\bar v} = \mathrm{coker}(\cH om(P_V^{gp}/\tilde\rho,\cO_V^*)_{\bar v} \to \cH om(\R,\cO_V^*)_{\bar v})$$
 if and only if every $h_p$ extends to $V$ by $0$. This property does not depend on the choice of representative $(h_p)_p \in \cH om(\R,\cO_V^*)_{\bar v}$.
\end{lemma}
\begin{cor}
 The square \eqref{abs-rel-embedding-square} is Cartesian, and the forgetful map $\widehat{\cL\cS}_{V,\bar v} \to \widehat\cL_{V,\bar v}$ is surjective.
\end{cor}

\begin{rem}
 Intuitively, the reader might feel that the surjectivity would conflict with the additional condition $\alpha(\rho) = 0$ in our definition of $\widehat{\cL\cS}_V$. But because $\mathrm{supp}(\bar\rho) = V$, we have $h_{\bar\rho} = 0 \in (j_{\bar\rho *}\cO_\emptyset^*)_{\bar v}$; in particular, $h_{\bar\rho} = 0$ as a function after extending by $0$.
\end{rem}

\section{Ghost structures}

Ghost structures are not just pre-ghost structures that enjoy a special property; they encode more information. Once we have a ghost structure, we can specify a subsheaf $\cL\cS_V \subseteq \widehat{\cL\cS}_V$ whose sections (i) correspond to actual log morphisms, not just equivalence classes,\footnote{There is one globally defined log morphism on the locus of definition of $s \in \cL\cS_V$, not just on the open subsets of a cover, and then only equivalent on overlaps.} and (ii) correspond to log morphisms which are log smooth, not just some morphisms.

\begin{defn}\index{ghost structure}\index{ghost structure!space with a ghost structure}
 A \emph{space with a ghost structure} $V^g$ consists of a space with a pre-ghost structure $(V,\cP,\bar\rho)$ and 
 \begin{itemize}
  \item an \'etale cover $\{\pi_i: U_i \to V\}_{i = 1,...,m}$
  \item a choice of geometric point $\bar v_i \in U_i$ for each $i = 1, ..., m$.
 \end{itemize}
 Then $P_i := \cP_{\bar v_i}$ is a toric monoid, and $\bar\rho_i := \bar\rho_{\bar v_i} \in P_i$; they define a log map $f_i: (V_i,\M_i) \to S_0$ as follows: The homomorphism $\NN \to P_i, \ 1 \mapsto \bar\rho_i$ defines a log morphism 
 $$A_{P_i} := \Spec (P_i \to \kk[P_i]) \to \Spec (\NN \to \kk[\NN]) =: A_\NN;$$
 then we obtain $f_i: (V_i,\M_i) \to S_0$ as the fiber product along the inclusion $S \to A_\NN$. The underlying space $V_i$ is the toric 
 divisor $\{z^{\bar\rho_i} = 0\} \subseteq \Spec \kk[P_i]$. Forgetting the log structure, we obtain the space with a pre-ghost structure $(V_i,\overline\M_i,\bar\rho_i)$. Now, as part of the datum, we have
 \begin{itemize}
  \item \emph{smooth} morphisms $\sigma_i: U_i \to V_i$ with $\sigma_i(\bar v_i) = 0 \in \Spec\kk[P_i]$.
 \end{itemize}
 We assume that there is an isomorphism $\phi_i: \pi_i^{-1}(\cP) \cong \sigma_i^{-1}(\overline\M_i)$ of sheaves of monoids on $U_i$ which is on the stalk at $\bar v_i$ compatible with the two identifications $\cP_{\bar v_i} = P_i$---by definition of $P_i$---and $\overline\M_{i,0} = P_i$---by construction of $\M_i$---; in particular, $\phi_i(\bar\rho) = \sigma_i^*\bar\rho_i$. Such an isomorphism is unique if it exists because two isomorphisms would coincide on $\cP_{\bar v_i}$, and because there is a surjection $P_i \to \cP|_{U_i}$ of sheaves of monoids. Finally, we assume that the two log structures $(\sigma_i)^*_{log}\M_i$ (on $U_i$) and $(\sigma_j)^*_{log}\M_j$ (on $U_j$) are of the same type on overlaps $U_i \times_V U_j$; this means, the log structures are of the same type in every point.
\end{defn}
\begin{defn}
 Let $V^g$ be a space with a ghost structure. The subsheaf $\cL\cS_V \subseteq \widehat{\cL\cS}_V$ consists of equivalence classes of log morphisms $(\M,\alpha,q,\rho)$ such that in every point, the log structure $(\M,\alpha,q)$ is of the same type as prescribed by the ghost structure---i.e., of the same type as $(\sigma_i)^*_{log}\M_i$.
\end{defn}

A \emph{log smooth structure} of \emph{ghost type} $V_g$ is a log morphism $(\M,\alpha,q,\rho) \in \cL\cS_V$ such that the corresponding map $(V,\M) \to S_0$ is log smooth; the key point about ghost structures is that---under one more technical hypothesis---they correspond precisely with sections of $\cL\cS_V$. The purpose of all the above structures is to ensure that $(V,\M) \to S_0$ is log smooth.

\begin{prop}[\cite{GrossSiebertI},~3.20]
 Let $V^g$ be a space with a ghost structure, and assume that $\cP_{\bar\eta}^{gp}$ is generated by $\bar\rho_{\bar\eta}$ at geometric generic points $\bar\eta$ of irreducible components of $V$. Then for an \'etale open $U \subseteq V$, sections in $\Gamma(U,\cL\cS_V)$ correspond to actual log structures $(\M,\alpha,q,\rho)$ on $U$, and the log morphism is log smooth.
\end{prop}

\section{Toroidal crossing spaces}

Now we are prepared to define toroidal crossing spaces.

\begin{defn}[Toroidal crossing space]\label{tor-cross-defn}\index{toroidal crossing space}
 A \emph{toroidal crossing space} is a space with a pre-ghost structure $(V,\cP,\bar\rho)$ such that for every point $\bar v \in V$, there is a polytope $\sigma \subseteq N_\RR$ such that $\cP_{\bar v} = P_\sigma$ and $\bar\rho_{\bar v} = \rho_\sigma$ where $\bar\rho_{\bar v} \in \cP_{\bar v}$ is the germ, and $\rho_\sigma \in P_\sigma$ is the Gorenstein degree. Moreover, there is an \'etale neighborhood $\pi: U \to V$ of $\bar v$ and a \emph{smooth} morphism $\tau: U \to V(\sigma)$ with $\tau(\bar v) = 0$ such that we have an isomorphism 
 $$\phi: \cP|_U \to \tau^{-1}(\cP_\sigma)$$ of toric ghost sheaves which respects the canonical identifications of the respective stalks with $P_\sigma$. In particular, there is at most one such isomorphism; if it exists (we assume this here), then it is compatible with $\bar\rho_{\bar v}$.
\end{defn}

To turn a toroidal crossing space $(V,\cP,\bar\rho)$ into a space with a ghost structure, we choose points $\bar v_1,...,\bar v_m$ of $V$ such that the corresponding \'etale neighborhoods $\pi_i: U_i \to V$ of Definition~\ref{tor-cross-defn} form a cover. After writing $\sigma_i$ for the polytope associated to $\bar v_i$, we obtain smooth morphisms $\tau_i: U_i \to V(\sigma_i)$ with $\tau_i(\bar v_i) = 0$ as required. The two induced log morphisms are of the same type on $U_i \times_V U_j$ due to Lemma~\ref{same-type-lemma}; namely, for a point $\bar v \in U_i \times_V U_j$ with polytope $\tilde\sigma$, the two log morphisms are induced by two smooth morphisms to the same space $V(\tilde\sigma)$, as follows from the following Lemma.

\begin{lemma}\label{local-model-at-point}
 Consider $V(\sigma) \to S_0$ as a log morphism, and let $\bar v \in V(\sigma)$ be a $k$-valued point with associated polytope $\tilde\sigma$, i.e., we have $\overline\M_{V(\sigma),\bar v} = P_{\tilde\sigma}$. Then there is an open neighborhood $\bar v \in W \subseteq V(\sigma)$ and a smooth and strict morphism $s: W \to V(\tilde \sigma)$ with $s(\bar v) = 0$.
\end{lemma}

A priori, the subsheaf $\cL\cS_V \subseteq \widehat{\cL\cS}_V$ might depend on the choices we made to turn $V$ into a space with a ghost structure. However, it is in fact \emph{independent} of the choices; namely, the argument which proves that the two log structures are of the same type also applies for choosing other points, other \'etale neighborhoods, and other smooth morphisms to $V(\sigma_i)$.

\begin{cor}
 Let $(V,\cP,\bar\rho)$ be a toroidal crossing space. Then there is a well-defined subsheaf $\cL\cS_V \subseteq \widehat{\cL\cS}_V$ of log smooth morphisms.
\end{cor}

\begin{ex}\label{prototype-tor-cr-space}\note{prototype-tor-cr-space}
 Our prototypes of toroidal crossing spaces from Example~\ref{prototype-preghost} are toroidal crossing spaces. This is not a tautology, but needs Lemma~\ref{local-model-at-point}.
\end{ex}

\begin{ex}[Normal crossing spaces]\label{normal-crossing-toroidal-crossing}\note{normal-crossing-toroidal-crossing}\index{toroidal crossing space!normal crossing space}
 A normal crossing space $V$ has a canonical structure of a toroidal crossing space. Namely, let $\nu: \tilde V \to V$ be the normalization map, and let $\underline\NN_{\tilde V}$ be the constant sheaf on $\tilde V$ with stalk $\NN$, formed in the \'etale topology; it has a distinguished global section $1_{\tilde V} \in \Gamma(\tilde V,\underline\NN_{\tilde V})$. Now the space with a pre-ghost structure
 $$(V,\nu_*\underline\NN_{\tilde V},\bar\rho_V := \nu_*1_{\tilde V})$$
 is a toroidal crossing space. All ghost stalks are of the form $\NN^r$ with $\bar\rho_{\bar v} = (1,...,1)$ and $r \leq d = \mathrm{dim}\ V$. 
\end{ex}

\begin{rem}
 A normal crossing space $V$ might carry other structures of a toroidal crossing space besides the canonical one. For example, taking $\sigma = [0,k] \subseteq \RR$, we get a toroidal crossing space $V(\sigma)$ with underlying space $\{xy = 0\} \subseteq \bAA^2_{x,y}$; for $k = 1$, this is the canonical structure of a toroidal crossing space, but for $k \geq 2$, it is another one.
\end{rem}

\begin{ex}[Toroidal embeddings]\label{toroidal-emb}
 Let $U \subseteq X$ be a toroidal embedding---of the smooth $U$ in the normal variety $X$---in the sense of \cite{KKMS1973}, and write $V := X \setminus U = \bigcup_i E_i$ for the decomposition of the boundary in irreducible components; assume that $U \subseteq X$ is \emph{without self-intersections}, i.e., each $E_i$ is a normal variety. We endow $X$ with the divisorial log structure $\M_X$ induced from $V \subseteq X$; a priori, we should carry out this construction in the \'etale topology, but because $U \subseteq X$ is without self-intersections, we get the same answer in the Zariski topology, cf.~\cite[III.~Prop.~1.6.5]{LoAG2018}. We have an isomorphism $\overline\M_X \cong \underline\Gamma_V(\mathrm{Div}_X^+)$, the sheaf of effective Cartier divisors with support in $V$, formed in the Zariski topology. When we set 
 $$\cP := \overline\M_X|_V =  \Gamma_V(\mathrm{Div}_X^+)|_V,$$ 
 then $(V,\cP)$ is a space with a toric ghost sheaf. If $X$ is \emph{Gorenstein}, then $V \subseteq X$ itself is an effective Cartier divisor. It gives a global section $\bar\rho \in \cP$; by definition of a toroidal embedding, now $(V,\cP,\bar\rho)$ is a toroidal crossing space in our sense. In \cite{KKMS1973}, a \emph{conical polyhedral complex} is associated to the toroidal embedding $U \subseteq X$; its cones are dual to the monoids of effective Cartier divisors on an open neighborhood $\mathrm{Star}(Y)$ of a stratum of $\bigcup_i E_i$, so we may consider this conical polyhedral complex a precursor to the structure of a toroidal crossing space.
\end{ex}

\begin{rem}
 The definition and theory of a toroidal crossing space by Schr\"oer--Siebert in \cite{SchroerSiebert2006} is slightly different. They are not so much interested in log morphisms, but in log structures that we classify with $\hat\cL_V$. Moreover, they restrict to \emph{tame} toroidal crossing spaces in the sense of our definition below.
\end{rem}

\section{The geometry and stratification of $V$}\index{toroidal crossing space!stratification}

Let $V = (V,\cP,\bar\rho)$ be a toroidal crossing space. If $\bar v$ is a geometric point over $v \in V$, then $\cP_{\bar v}$ depends only on $v$ and not on the choice of $\bar v$. This is because, on a roof $V \leftarrow U \rightarrow V(\sigma)$ as in the definition of a toroidal crossing space, we have $\cP|_U = \tau^{-1}(\cP_\sigma)$, i.e., the stalks of the restriction of $\cP$ to the Zariski topology on $U$ do not change under a further \'etale open $\pi': U' \to U$. It is, however, not always true that $\cP_{\bar v} = \cP^{Zar}_v$ when $\cP^{Zar}$ is the restriction of $\cP$ to the Zariski topology of $V$, as the example of the nodal cubic with $\cP$ as constructed in Example~\ref{normal-crossing-toroidal-crossing} shows. Since $\cP_{\bar v}$ is independent of the geometric point $\bar v$ over $v$, the following loci are well-defined, for $m \geq 0$:
\begin{align}
 \U_mV &:= \{v \in V \ | \ \mathrm{rk}(\cP_{\bar v}) \leq m + 1\} \nonumber \\
 \C_mV &:= \{v \in V \ | \ \mathrm{rk}(\cP_{\bar v}) \geq m + 1\} \nonumber \\
 \cS_mV &:= \{v \in V \ | \ \mathrm{rk}(\cP_{\bar v}) = m + 1\} \enspace = \enspace \U_mV \cap \C_mV \nonumber 
\end{align}
It is easy to see from the local models that $\U_mV \subseteq V$ is open, and $\C_mV \subseteq V$ is closed. Hence, $\cS_mV \subseteq V$ is locally closed. We also see from the local models that $\U_0V = \cS_0V$ is smooth, and that its complement $\C_1V$ is the locus where at least two \'etale local components meet, which is also equal to the non-normal locus of $V$. On $\cS_1V$, we have only double normal crossing singularities given by $\{xy = 0\}$; the locus $\C_2V$ contains more complicated singularities, and here always meet at least three \'etale local components.

We derive from the local models as well that $\cS_mV$ is smooth of dimension $d - m$, where $d = \mathrm{dim}(V)$. We denote the set of irreducible components of $\cS_mV$ by $[\cS_mV]$. An irreducible component $S^\circ \in [\cS_mV]$ is called an \emph{open stratum} of $(V,\cP,\bar\rho)$, and its closure $S = \overline{S^\circ}$ is called a \emph{closed stratum}. The closed strata in $[\cS_0V]$ are precisely the irreducible components of $V$; we also denote them by $V_1, ..., V_c$.

If $v \in \cS_1V$, then we must have 
$$\cP_{\bar v} \cong \langle (-1,k),(0,1),(1,0)\rangle \subseteq \ZZ^2$$
with $\bar\rho_{\bar v} = (0,1)$ for some $k \geq 1$. The corresponding polytope is $\sigma = [0,k] \subseteq \RR$. We get a function 
$$\kappa: \enspace \cS_1V\to \NN, \quad v \mapsto k,$$
which is locally constant because the structure of toroidal crossing space is locally around $\bar v$ pulled back from $V(\sigma)$---no change of $k$ is possible. Thus, we consider it as a function 
$$\kappa: \enspace [\cS_1V] \to \NN.$$
We say that $k = \kappa(D^\circ)$ is the \emph{kink} of the open stratum $D^\circ \in [\cS_1V]$ of codimension $1$. When we consider a piecewise linear function $h: \ZZ \to \ZZ$ with 
$$\cP_{\bar v} \cong \{(m_1,m_2) \in \ZZ^2 \ | \ m_2 \geq h(m_1)\},$$
i.e., $h(m) = -k\cdot m$ for $m \leq 0$ and $h(m) = 0$ for $m \geq 0$, then $k$ becomes the change of slope of $h$---hence the name \emph{kink}. We say that a toroidal crossing space $(V,\cP,\bar\rho)$ is \emph{tame} if the kink is always $k = 1$.

In analogy with normal crossing spaces, we say that a toroidal crossing space $(V,\cP,\bar\rho)$ is \emph{simple} if all closed strata $V_i \in [\cS_0V]$ are normal, i.e., if all irreducible components $V_1,...,V_c$ are normal. Simple toroidal crossing spaces admit special charts:

\begin{lemma}\label{special-chart-tor-cr}\note{special-chart-tor-cr}
 Let $(V,\cP,\bar\rho)$ be a simple toroidal crossing space, and let $v \in V$. Then there is an \'etale map $\pi: U \to V$, a geometric point $\bar u \in U$ with $\pi(u) = v$, and a smooth map $\tau: U \to V(\sigma)$ with $\tau(u) = 0$ as in the definition of a toroidal crossing space, with the following property: The map $\tau: U \to V(\sigma)$ induces a bijection between the irreducible components of $V(\sigma)$ and of $U$. Every irreducible component of $U$ contains $u \in U$. The map $\pi: U \to V$ induces a bijection between the irreducible components of $U$ and those irreducible components of $V$ which contain $v$. 
\end{lemma}
\begin{proof}
 We leave the rather easy proof of the first two statements to the reader. The most interesting part is the last statement. Let $u_1, ..., u_e$ be the generic points of the irreducible components $U_1,...,U_e$ of $U$. Since $\pi: U \to V$ is \'etale, each $v_i := \pi(u_i)$ has a local ring of dimension $0$, so they are generic points of irreducible components of $V$. Let us denote them by $V_i$. If there are two indices $i,j$ with $v_i = v_j$, then $U_i \cup U_j \subseteq \pi^{-1}(V_i)$. However, by assumption, $V_i$ is normal, so $\pi^{-1}(V_i)$ is normal, and it cannot contain two different irreducible components which both contain $u \in U$. Thus, we have an injection from the irreducible components of $U$ to those irreducible components of $V$ which contain $v$. It must be surjective because any specialization $\eta \rightsquigarrow v$ from the generic point of an irreducible component of $V$ can be lifted along $\pi: U \to V$.
\end{proof}
\begin{cor}
 Let $(V,\cP,\bar\rho)$ be a simple toroidal crossing space. Then every closed stratum $S \in [\cS_mV]$ is normal.
\end{cor}
\begin{proof}
 Let $\eta$ be the generic point of $S^\circ$, and let $V_{i_1},...,V_{i_e}$ be the irreducible components of $V$ which contain $\eta$. Let $W = \bigcap_{k = 1}^e V_{i_k}$, and let $w \in W$ be a point. Choose $\pi: U \to V$ and $\tau: U \to V(\sigma)$ as in Lemma~\ref{special-chart-tor-cr}. Then $\pi^{-1}(W)$ is the intersection of some of the irreducible components of $U$, and hence it is equal to $\tau^{-1}(W_\sigma)$, where $W_\sigma \subseteq V(\sigma)$ is the corresponding intersection of irreducible components of $V(\sigma)$. Thus, $W$ is normal. When we apply this argument to $w = \eta$, then, on the one hand side, $\pi^{-1}(W)$ is the intersection of all irreducible components of $U$. On the other hand side, $\pi^{-1}(S^\circ)$ is the locus where, on $U$, the rank of $\cP_{\bar v}$ is maximal, and hence this is equal to $\tau^{-1}(\{0\})$. Thus, we have $\pi^{-1}(W) = \pi^{-1}(S^\circ)$, showing that $S \subseteq W$ is a maximal closed subset with respect to inclusion, i.e., an irreducible component. Hence, $S$ is normal.
\end{proof}

In a simple toroidal crossing space $(V,\cP,\bar\rho)$, each closed stratum $S \in [\cS_mV]$ is a toroidal variety, and its boundary $\partial S = S \setminus S^\circ$ turns it into a toroidal embedding without self-intersections.

\begin{ex}
 Let $V$ be a nodal cubic with node $v \in V$, considered as a  toroidal crossing space by Example~\ref{normal-crossing-toroidal-crossing}. Then $\U_0V = \cS_0V = V \setminus \{v\}$ is irreducible although it has, \'etale locally around $v$, two strata of codimension $0$. Since $V = \overline{V \setminus \{v\}}$ is not normal, $V$ is not simple. We have $\C_1V = \cS_1V = \{v\}$.
\end{ex}

\section{Toroidal crossing degenerations to $V$}

Let $R$ be a discrete valuation $\kk$-algebra with residue field $\kk$, and $S = \Spec R$. Let $f: X \to S$ be a toroidal crossing degeneration. Example~\ref{prototype-tor-cr-space} shows that $V = X_0$ endowed with $\cP := \overline\M_{X_0}$ and the corresponding global section $\bar\rho \in \cP$ is a toroidal crossing space. Conversely, when $(V,\cP,\bar\rho)$ is a toroidal crossing space, and $f: X \to S$ is a toroidal crossing degeneration, and we have an isomorphism $V \cong X_0$ of schemes, then there is at most one map $q: \M_{X_0} \to \cP$ which turns $(X_0,\M_{X_0}) \to S_0$ into a section of $\cL\cS_V$. Namely, any two such maps differ by an automorphism of $\cP$, and when we analyze the supports of sections of $\cP$ in the local models, we find that the identity is the only automorphism of $\cP$ as a sheaf of monoids. Thus, the following definition makes sense.

\begin{defn}
 Let $(V,\cP,\bar\rho)$ be a toroidal crossing space. Then a \emph{toroidal crossing degeneration to $V$} is a toroidal crossing degeneration $f: X \to S$ together with an isomorphism $V \cong X_0$ as schemes such that there is a map $q: \M_{X_0} \to \cP$ which turns the log morphism $f_0: X_0 \to S_0$ into a section of $\cL\cS_V$. 
\end{defn}

This is equivalent to that $(\overline\M_{X_0},f^*(1_{S_0}))$ is locally isomorphic to $(\cP,\bar\rho)$ as a sheaf of monoids with a distinguished section. Namely, since $(\cP,\bar\rho)$ has no non-trivial automorphisms, the local isomorphisms glue to a global isomorphism, and then we obtain a section of $s \in \widehat{\cL\cS}_V$. The argument which shows that $\cL\cS_V \subseteq \widehat{\cL\cS}_V$ is well-defined now shows that in fact $s \in \cL\cS_V$.

Locally, every section $s \in \cL\cS_V$ arises from a toroidal crossing degeneration---at least when $R$ is complete. By the Cohen structure theorem, we have $R = \kk\lsem t\rsem$ in this case.

\begin{lemma}
	Let $R = \kk\lsem t\rsem$, and let $(V,\cP,\bar\rho)$ be a toroidal crossing space. Let $s \in \cL\cS_V$ be a section. Then every geometric point $\bar v \in V$ has an \'etale neighborhood $W$ such that there is a toroidal crossing degeneration $f: X \to S$ to $V|_W$ which induces $s \in \cL\cS_V$.\footnote{The proof of a somewhat stronger statement in an earlier version of the manuscript is wrong.}
\end{lemma}
\begin{proof}
	Let $f_0: X_0 \to S_0$ be the log morphism corresponding to $s \in \cL\cS_V$. Then $f_0: X_0 \to S_0$ is log smooth, saturated, and vertical. Every geometric point $\bar v \in V$ admits a strict \'etale neighborhood $g_0: W_0 \to X_0$ such that there is a strict \'etale morphism 
	$$h_0: W_0 \to V(\sigma) \times \bAA^{d - r} =: L_0.$$ By shrinking $W_0$, we can assume that it is affine. Lemma~\ref{etale-map-lift} shows that there is an \'etale morphism of finite type $h: W \to L$ of schemes, where $L := U_S(\sigma) \times \bAA^{d - r}$ is the base change of the local model to $S = \Spec \CC\lsem t\rsem$. Now $W \to S$ is a toroidal crossing degeneration, and the divisorial log structure defined by $t = 0$ is the given log structure on $W_0$ after restriction to the central fiber.
\end{proof}

This justifies the choice of $\cL\cS_V \subseteq \widehat{\cL\cS}_V$: every toroidal crossing degeneration to $V$ whose ghost sheaf is isomorphic to $(\cP,\bar\rho)$ induces a section of $\cL\cS_V$, and every section of $\cL\cS_V$ arises locally from a toroidal crossing degeneration.

\section{The local description of $\cL\cS_V$}\label{sec-local-LS}

In \cite[Thm.~3.22]{GrossSiebertI}, Gross--Siebert gives a description of the sheaf $\cL\cS_V$ for the prototypes of toroidal crossing spaces, which we briefly explain here. Let $\sigma \subseteq N_\RR$ be a polytope as in Construction~\ref{prototype-constr}, and let $V = (V,\cP,\bar\rho) = V(\sigma) \times \GG_m^{d-r}$ be the toroidal crossing space which we obtain from the prototype $(V(\sigma),\cP_\sigma,\bar\rho_\sigma)$ by pulling back the structure along the smooth projection $V \to V(\sigma)$. If $\tau \subseteq \sigma$ is a face, then the dual face $\check\tau \subseteq M_\RR$ defines an irreducible closed toric stratum $V_\tau \subseteq V$ with $V_\tau \cong \Spec k[(\check \tau \cap M) \oplus \ZZ^{d - r}]$; these strata are glued along lower-dimensional strata. The irreducible components of $V$ are $\{V_v \ | \ v \in \sigma \mathrm{\ is\ a\ vertex}\}$. For each face $\tau \subseteq \sigma$, we denote by $\tau^\parallel \subseteq N$ the intersection of the tangent space in $N_\RR$ at $\tau$ with the lattice $N$. For edges $\omega \subseteq \sigma$, this tangent space is one-dimensional, i.e., $\omega^\parallel \cong \ZZ$; thus, we can (and do) choose primitive generators $d_\omega \in \omega^\parallel$. We get a choice of vertices $v^+_\omega$ and $v^-_\omega$ of $\omega$ (such that $d_\omega$ points from $v^-_\omega$ to $v^+_\omega$) and an orientation on the edge $\omega$. For each two-dimensional face $\tau \subseteq \sigma$, we choose a \emph{sign vector} 
$$\epsilon_\tau: \{\mathrm{edges\ of\ }\sigma\} \to \{-1,0,1\}$$
such that $\epsilon_\tau(\omega) = 0$ if and only if $\omega \not\subseteq \tau$, and $\sum_\omega \epsilon_\tau(\omega)\omega$ is an oriented boundary of $\tau$. With these choices, Gross--Siebert finds:

\begin{thm}[\cite{GrossSiebertI},~3.22]\label{LS-local-description}\note{LS-local-description}
 The sheaf $\cL\cS_V$ is isomorphic to the subsheaf of $\bigoplus_{\mathrm{dim}\ \omega = 1} \cO_{V_\omega}^*$ defined as follows. If $U \subseteq V$ is an open subset, then $\Gamma(U,\cL\cS_V)$ consists of $(f_\omega)$ such that, for every two-dimensional face $\tau$ of $\sigma$, we have 
 $$\prod_{\mathrm{dim}\,\omega \, =\, 1} d_\omega \otimes f_\omega^{\epsilon_\tau(\omega)}|_{V_\tau} = 1 \in N \otimes_\ZZ \Gamma(U,\cO_{V_\tau}^*).$$
\end{thm}

The functions $f_\omega \in \cO_{V_\omega}^*$ in the theorem are referred to as \emph{slab functions}.\index{slab function} Gross--Siebert finds in \cite[Thm.~3.28]{GrossSiebertI} that, in their setting, the local descriptions can be glued, and the slab functions are sections of global line bundles $\N_\omega$ on the (normalizations of) strata of codimension $1$ of $X_0(B,\cP,s)$.

\begin{rem}
 In the general case, the global nature of the line bundles $\N_\omega$ is not yet entirely clear. In current work in progress (\cite{CortiRuddatLog}, personal communication), Corti and Ruddat define, for each closed codimension-$1$ stratum $\rho \in [\cS_1V]$ of a simple toroidal crossing space, a \emph{wall bundle} $\cL_\rho$ on $\rho$, which will be the global analog of $\N_\omega$ above. Then Theorem~\ref{LS-local-description} essentially globalizes. This will allow us to construct sections of $\cL\cS_V$ more easily in a general global setting because the $\cL_\rho$ are coherent sheaves, as opposed to the more complicated $\E xt^1(\cP^{gp}/\bar\rho,\cO_V^*)$. An important application is the construction of resolutions of log singularities. Given a toroidal crossing space $(V,\cP,\bar\rho)$ together with a section $s \in \Gamma(V \setminus Z,\cL\cS_V)$, one wishes to construct a proper birational map $\pi: Y \to V$ such that $Y$ carries the structure of a saturated log smooth log morphism, and such that $\pi: Y \to V$ is an isomorphism over $V \setminus Z$. Given a good understanding of the behavior of $\cL\cS_V$, one can construct a resolution as follows: Let us assume that $s \in \Gamma(V \setminus Z,\cL\cS_V)$ is given by functions $f_\rho \in \cL_\rho$ which vanish precisely in $Z$. Then one constructs a modification $\pi: Y \to V$ such that, on $Y$, each $f_\rho$ extends to a nowhere vanishing section of the new $\cL_\rho'$. At least in good situations, this yields a section of $\cL\cS_Y$ which extends $s$ from $V \setminus Z = Y \setminus \pi^{-1}(Z)$ to the whole of $Y$.
 
 For applications in deformation theory, one sometimes wants to consider \emph{crepant} resolutions of log singularities, often of affine log singularities, because then we have $\Theta^1_{Y/S_0} \cong \Omega^{d - 1}_{Y/S_0}$. When insisting that $\pi: Y \to V$ should be crepant, examples suggest (again personal communication) that it is not always possible to have that $Y$ is a toroidal crossing space with a section of $\cL\cS_Y$. Instead, some singularity like $\{xy = 0\} \subset \frac{1}{r}(1,-1,a,-a)$ will persist already on the level of the underlying space of $Y$. This has lead Corti and Ruddat to investigate \emph{generically} toroidal crossing spaces,\index{toroidal crossing space!generically} which are (more or less) spaces with a pre-ghost structure $(V,\cP,\bar\rho)$ such that the local models of toroidal crossing spaces do not necessarily exist everywhere but only around the generic points of the strata. Here, one assumes that a stratification is a priori given, which coincides with the stratification from the toroidal crossing space structure where both are defined. Corti and Ruddat impose a compatibility condition, called \emph{viability},\index{toroidal crossing space!viability} which is sufficient to construct the wall bundles $\cL_\rho$, and to define a sheaf $\cL\cS_V$ which classifies some log morphisms to the standard log point $S_0$. Due to the possibility of coherent log singularities that are already build in into $(V,\cP,\bar\rho)$, this approach is far more general and flexible than the classical notion of toroidal crossing spaces which we discuss in this chapter, allowing to construct log structures on spaces that do not admit the structure of a toroidal crossing space in our sense. Details and precise statements will appear in \cite{CortiRuddatLog}.
\end{rem}

\vspace{\baselineskip}

Since $\U_1V \cap V_\omega \subseteq V_\omega$ is schematically dense, we deduce from Theorem~\ref{LS-local-description}:

\begin{cor}\label{purity-LS}
 Let $(V,\cP,\bar\rho)$ be a toroidal crossing space. Then, for each (Zariski) open $U \subseteq V$, the restriction map $\Gamma(U, \cL\cS_V) \to \Gamma(U \cap \U_1V,\cL\cS_V)$ is injective.
\end{cor}

\section{The map $\cL\cS_V \to \T^1_V$}

In Chapter~\ref{sec-local-LS}, we gave Gross--Siebert's local description of $\cL\cS_V$, whose globalization is not yet completely clear. In this chapter, we explain another approach to a global description of $\cL\cS_V$, which is particularly useful for normal crossing spaces. More generally, to obtain useful information from this approach, we need the toroidal crossing space to be tame as defined above. The reference for this approach is \cite{FFR2021}; the ideas trace back to Schr\"oer--Siebert's work \cite{SchroerSiebert2006}.

Let $(V,\cP,\bar\rho)$ be a toroidal crossing space, and let $U \subseteq V$ be an affine open subset. Let $(\M,\alpha,q,\rho)$ be a log morphism that represents a class in $\Gamma(U,\cL\cS_V)$; thus, we have a log smooth morphism $U \to S_0$. We set $S_1 := \Spec(\NN \to \kk[t]/(t^2))$ where $1 \mapsto t$; it is a first order log thickening of $S_0$. By log smooth deformation theory, there is---up to non-unique isomorphism---a unique log smooth deformation $U_1 \to S_1$. In particular, forgetting the log structure, we get a closed embedding $i: \underline U \subseteq \underline U_1$ of schemes; it gives rise to an extension
\begin{equation}\label{LS-extension}
 0 \to \cO_U \to i^*\Omega^1_{\underline U_1} \to \Omega^1_{\underline U} \to 0
\end{equation}
of sheaves on $U$. Here, $\Omega^1_{\underline U_1}$ are the \emph{absolute} K\"ahler differential forms of $\underline U_1$; the left inclusion is given by $1 \mapsto i^*dt$. The isomorphism class of the extension is an element in $\Gamma(U,\E xt^1(\Omega^1_{\underline U},\cO_U))$; this yields a map
$$\eta_V: \cL\cS_V \to \E xt^1(\Omega^1_{\underline V},\cO_V) =: \T^1_V$$
of sheaves of sets. Because the log smooth deformation $U_1 \to S_1$ is only unique up to isomorphism, so is the extension; thus, a global section of $\cL\cS_V$ gives rise to a section of $\E xt^1(\Omega^1_{\underline V},\cO_V)$, but not to a global extension in $\mathrm{Ext}^1(\Omega^1_{\underline V},\cO_V)$.

The target $\T^1_V$ is a coherent sheaf, but the source $\cL\cS_V$ has, up to now, no similar structure. However, we easily endow it with an $\cO_V^*$-action by setting 
$$\lambda \cdot [(\M,\alpha,q,\rho)] := [(\M,\alpha,q,\lambda^{-1}\rho)]$$
---i.e., we change $\rho \in \M$ to $\lambda^{-1}\rho$ without changing its image $\bar\rho \in \cP$. Note that we invert the function $\lambda \in \cO_V^*$. The target $\T^1_V$ has a canonical $\cO_V^*$-action as well---it is a coherent sheaf.

\begin{prop}[\cite{FFR2021},~5.1]\label{eta-equivariance}
 The map $\eta_V: \cL\cS_V \to \T^1_V$ is $\cO_V^*$-equivariant.
\end{prop}

\begin{rem}
 At the level of $\bar\iota_{V,\bar v}: \widehat{\cL\cS}_{V,\bar v} \to \E xt^1(\cP^{gp}/\bar\rho,\cO_V^*)_{\bar v}$, we see the $\cO^*_{V,\bar v}$-action as follows: If $(h_p)_p$ represents a germ $M \in \cL\cS_{V,\bar v}$ and $\phi: P^{gp} \to \cO^*_{V,\bar v}$ is a homomorphism, then $(\phi(p)h_p)_p$ represents $\phi(\bar\rho)^{-1}\cdot M$.
\end{rem}

In order to be useful to describe sections of $\cL\cS_V$, we want the map $\eta_V: \cL\cS_V \to \T^1_V$ to be injective; however, this is not always the case.

\begin{ex}\label{one-dim-non-inj}
 This is an example of a one-dimensional toroidal crossing space such that $\eta_V$ is not injective. Let $N = \ZZ$ and $\sigma = [0,2] \subseteq \RR = N_\RR$. We denote the associated toroidal crossing space by $(V,\cP,\bar\rho)$; it consists of two intersecting lines $L_x$ and $L_y$. In the interiors of the lines, the stalk of $\cP$ is $\NN$; in the intersection point, we have $\cP_0 = \langle (-1,2),(1,0)\rangle \subseteq \ZZ^2$. For every $\lambda \in \kk^*$, the log smooth family 
 $$\Spec \kk[x,y,t]/(xy - \lambda t^2) \to \Spec \kk[t], \quad t \mapsto t,$$
 induces a section $s_\lambda \in \Gamma(V,\cL\cS_V)$. If $\lambda \not= \mu$, then also $s_\lambda \not= s_\mu$ 
 However, we have 
 $$\eta_V(s_\lambda) = \eta_V(s_\mu) = 0$$
 because the unique first order log smooth deformation has $\Spec \kk[x,y,t]/(xy,t^2)$
 as underlying space. Thus $\eta_V: \cL\cS_V \to \T_V^1$ is not injective; in fact, it is the zero map because every section of $\cL\cS_V$ is of the form $s_\lambda$.
\end{ex}

In order to find a criterion for injectivity of $\eta_V: \cL\cS_V \to \T_V^1$, we consider (at a point $\bar v \in V$) the forgetful map 
$$\cL\cS_{V,\bar v} \subseteq \widehat{\cL\cS}_{V,\bar v} \to \hat\cL_{V,\bar v};$$
it is invariant under the $\cO^*_{V,\bar v}$-action since this action does not affect the underlying log structure. Conversely, when two germs $M,M' \in \cL\cS_{V,\bar v}$ have the same underlying log structure, then they differ only by an invertible function. Thus, the fibers of the forgetful map are precisely the $\cO^*_{V,\bar v}$-orbits. In general, there might be many orbits, but, for $\cP_{\bar v} \cong \NN^r$, the situation is particularly simple.

\begin{lemma}
 Let $(V,\cP,\bar\rho)$ be a toroidal crossing space, and let $\bar v \in V$ be a point with $\cP_{\bar v} \cong \NN^r$. Then $\cO^*_{V,\bar v}$ acts transitively on $\cL\cS_{V,\bar v}$.
\end{lemma}
\begin{proof}
 The statement is in \cite[Lemma 5.2]{FFR2021}. In a nutshell, this holds because $\NN^r$ is free, so homomorphisms from $\NN^r$ can be constructed by specifying values on a basis. 
\end{proof}

The stalk of the ghost sheaf in Example~\ref{one-dim-non-inj} is not free; it is in fact $\cP_0 = \langle (-1,2),(1,0)\rangle \subseteq \ZZ^2$. But if instead $\cP_{\bar v} \cong \NN^r$, then the map $\eta_V: \cL\cS_V \to \T_V^1$ turns out to be injective.

\begin{lemma}[\cite{FFR2021},~5.2]
 Let $(V,\cP,\bar\rho)$ be a toroidal crossing space, and let $\bar v \in V$ be a geometric point with $\cP_{\bar v} \cong \NN^r$. Then:
 \begin{enumerate}[label=\emph{(\roman*)}]
  \item For $M \in \cL\cS_{V,\bar v}$, the map $\mu_M: \cO_{V,\bar v} \to \T^1_{V,\bar v}$ is surjective.
  \item The map $\eta_{V,\bar v}: \cL\cS_{V,\bar v} \to \T^1_{V,\bar v}$ is injective.
  \item The image of $\eta_{V,\bar v}$ is $(\T_{V,\bar v}^1)^* \subseteq \T^1_{V,\bar v}$, the elements that generate $\T_{V,\bar v}^1$ as an $\cO_{V,\bar v}$-module. 
 \end{enumerate}
\end{lemma}

Injectivity is not restricted to toroidal crossing spaces with $\cP_{\bar v} \cong \NN^r$ for all $\bar v \in V$; it extends to tame toroidal crossing spaces as defined above.

\begin{thm}
 If $(V,\cP,\bar\rho)$ is a \emph{tame} toroidal crossing space, then $\eta_V: \cL\cS_V \to \T_V^1$ is injective. 
\end{thm}
\begin{proof}
 As in \cite{FFR2021}, this follows from Corollary~\ref{purity-LS}.
\end{proof}

\section{Generically log smooth families from $\cL\cS_V$}

Given a projective or just proper toroidal crossing space $(V,\cP,\bar\rho)$, there is often no global section of $\cL\cS_V$ to endow it with a global log structure, cf.~\cite{GrossSiebertI}. Even if there is one, it is often enlightening for the (flat) deformation theory of $V$ to allow certain log singularities. In our theory, we achieve this by selecting a \emph{log singular locus} $Z \subseteq V$ of codimension $\geq 2$ and a section $s \in \Gamma(V \setminus Z,\cL\cS_V)$. Since $V$ is Cohen--Macaulay, this gives rise to a generically log smooth family over $S_0$, which we denote by $(V,Z,s)$. In order to have a well-behaved theory, we impose a condition:

\begin{defn}\label{well-adj-def}\note{well-adj-def}
 The triple $(V,Z,s)$ is \emph{well-adjusted} if $Z \subseteq \C_1V$ is contained in the double locus and the intersection $Z \cap S$ with a closed stratum $S \in [\cS_mV]$ is properly contained in $S$. For simplicity, we also require that $V$ is pure of dimension $d$.
\end{defn}
\begin{rem}
 We imagine $Z$ to show inadequate behavior, claiming space that it does not deserve, if this condition is violated; hence the name.
\end{rem}

The codimension condition \eqref{CC} is automatic if $(V,Z,s)$ is well-adjusted; $Z$ may be empty or non-reduced. As we will see below, the well-adjustedness guarantees a nice interaction between the log structure and the underlying space. We also need the following result.

\begin{lemma}
 Let $f_0: X_0 \to S_0$ be the generically log smooth family associated with a well-adjusted triple $(V,Z,s)$. Then $f_0: X_0 \to S_0$ is log Gorenstein.
\end{lemma}
\begin{proof}
 The underlying scheme $V$ is Gorenstein, so its dualizing sheaf $\omega_V$ is a line bundle. Since the log morphism $f_0: U_0 \to S_0$ is vertical, we have $\W^d_{U_0/S_0} \cong \omega_V|_{U_0}$ by \cite[Prop.~2.8]{FFR2021}. Thus $\W^d_{X_0/S_0} \cong \omega_V$ is a line bundle.
\end{proof}



\chapter{Log toroidal families of Gross--Siebert type}\label{elem-GS-type-sec}\note{elem-GS-type-sec}

The goal of this chapter is to give a justification for working with generically log smooth families instead of log smooth families, which would be technically simpler, and in particular, to give a justification for our definition of a system of deformations. We show that systems of deformations are not just a theoretical gadget but actually occur in 'nature'. The following results are known to the experts, already contained in publications, but probably have been somewhat mysterious to the outsiders. The interpretation of $\B_{X_0/S_0}$ given below is probably new, as well as the extension of the deformation theory to more general situations and the basic results, although we do not use them here.

\section{Log toroidal families}

Log toroidal families are generically log smooth families which admit specific local models, controlled by \emph{elementary log toroidal data}, short ETDs. Log toroidal families have been introduced in \cite{FFR2021} and are studied extensively in the author's thesis \cite{FeltenThesis}. Here, we give a brief summary for the reader's convenience.

An ETD $(Q \subset P,\F)$ consists of a saturated injection $Q \to P$ of sharp toric monoids and a collection $\F$ of facets, i.e., maximal faces, in $P$ such that all facets which do not contain $Q$ are contained in $\F$. Then $A_Q = \Spec (Q \to \ZZ[Q])$ is the affine toric variety $\Spec \ZZ[Q]$ endowed with the divisorial log structure from the full toric boundary $D_Q$. The facets in $\F$ give rise to a divisor in $\Spec \ZZ[P]$ which is a part of the toric boundary. We denote $\Spec \ZZ[P]$ endowed with the divisorial log structure from this divisor by $A_{P,\F}$. Then the inclusion $Q \to P$ defines a log morphism $f: A_{P,\F} \to A_Q$. In the discussion before \cite[Prop.~3.9]{FFR2021}, we construct an open subset $U_{P/Q} \subseteq A_{P,\F}$, only depending on $Q \to P$, such that $f: A_{P,\F} \to A_Q$ becomes a generically log smooth family. We say that $f: A_{P,\F} \to A_Q$ is an \emph{elementary log toroidal family}. Then, a \emph{log toroidal family}, is a generically log smooth family which admits local models by elementary log toroidal families in the following sense.

\begin{defn}[Local models]\label{local-model-def}\note{local-model-def}
 Let $f: X \to S$ be a generically log smooth family. For a geometric 
 point $\bar s \in S$, a \emph{base chart} at $\bar s$ is a strict 
 \'etale morphism $(\tilde S, \bar s) \to (S, \bar s)$ and a map
 $a: \tilde S \to A_Q$ given by a chart $Q \to \M_{\tilde S}$ of the log structure which is neat at $\bar s$.
 For a geometric point $\bar x \in X$ and a base chart at $f(\bar x)$, a \emph{local model} at $\bar x$ is a diagram
\begin{equation}
\tag{LM}\label{LM}
\quad \begin{aligned}
  \xymatrix@R-2pc{
   & (V,g^{-1}(U)) \ar[ldd]_g \ar@{.>}[ddr]^h \ar[dddd] & & \\
   \\
   (X,U) \ar[dddd]_f & & (L, U_L) \ar[ddl] \ar[ddr]^c & \\ 
   \\
   & \tilde S \ar[ldd] \ar[ddr]^a & & (A_{P,\F}, U_{P/Q}) \ar[ddl]  \\
   \\
   S & & A_Q & \\
  } 
\end{aligned}
\end{equation}
where $g: V \to X$ is an \'etale neighbourhood of $\bar x$ (of underlying schemes) and the bottom right diagonal map is given by an ETD $(Q \subset P,\F)$.
The solid arrows are morphisms of schemes and log morphisms on the specified opens, whereas $h: V \to L$
 is an \'etale morphism only of underlying schemes. The bottom right diamond is Cartesian, in particular $U_L=c^{-1}(U_{P/Q})$. Moreover, we have an open $\tilde U \subset V$ 
 satisfying \eqref{CC}, such that $\tilde U \subset g^{-1}(U) \cap h^{-1}(U_L)$ and there is an isomorphism $g^*\M_X \cong h^*\M_L$ of the two log structures on $\tilde U$ such that the composed maps to $\tilde S$ coincide. Finally, we have $c \circ h(\bar x) = 0 \in A_{P,\F}$.
 
 For a base chart $S \leftarrow \tilde S \to A_Q$ and a Zariski 
 open $W \subset X$, a \emph{local model} is a diagram \eqref{LM} as above with $g(V) = W$ (and no requirement on some $\bar x \in X$).
\end{defn}

Note that the specified open subset of log smoothness in the family does not need to coincide with the open subset $U_{P/Q}$ of the local model.

\begin{defn}[Log toroidal families]\label{toroid-def}\note{toroid-def}
 A \emph{log toroidal family} is a generically log smooth family 
 $f: X \to S$ such that there is a Zariski open cover $X = \bigcup_i W_i$ 
 and local models for the $W_i$ (over various base charts 
 $S \leftarrow \tilde S_i \to A_{Q_i}$).
 \end{defn}
 
If $S \cong \Spec (Q \to B)$, then $\tilde S = S$ and $a: S \to A_Q$ given 
by the chart $Q \to B$ is a base chart. We say $f: X \to S$ is 
\emph{log toroidal with respect to $a: S \to A_Q$} if we can choose 
the local models over this base chart. In the setup of infinitesimal deformation theory, we will be always in this situation.

For more information on log toroidal families, the reader may consult \cite{FFR2021,FeltenThesis}.

\section{Local models of Gross--Siebert type}\label{local-model-GS-type-sec}\note{local-model-GS-type-sec}

They are specific ETDs which arise from the Gross--Siebert program; we consider them also as examples of well-adjusted triples $(V,Z,s)$; we will then later study well-adjusted triples that are locally isomorphic to the local models of Gross--Siebert type.

\subsubsection*{The construction of the local model}

The reference for the construction is \cite[Constr.~2.1]{GrossSiebertII}; it starts from a tuple 
$$(M',N',\tau,\Delta_1,...,\Delta_q)$$ where $M' \cong \ZZ^d$ is a lattice, $N'$ is its dual lattice, $\tau \subset M'_\RR$ is a convex lattice polytope of full dimension $d$, and $\Delta_1, ..., \Delta_q \subset M'_\RR$ are convex lattice polytopes which may not be of full dimension. Here, $q = 0$ is allowed, i.e., there may not be any $\Delta_i$. For convenience, we also set $\Delta_0 := \tau$.
 We denote the dual fan of $\Delta_i$ in $N'_\RR$ by $\check\Sigma_i$, including $\Delta_0 = \tau$. We assume that 
\begin{center}
 $\check\Sigma_0$ is a subdivision of each $\check\Sigma_i$.
\end{center}
Each $\Delta_i$ gives rise to a piecewise linear function
$$\check\psi_i(n) := - \mathrm{inf}\{\langle m,n\rangle \ | \ m \in \Delta_i\}$$
on $N'_\RR$. On each cone $\sigma \in \check\Sigma_i$, the function $\check\psi_i$ is linear; since $\check\Sigma_0$ is a subdivision of each $\check\Sigma_i$, the function $\check\psi_i$ is linear, in particular, on each $\sigma \in \check\Sigma_0$. 

Gross--Siebert defines a sharp toric monoid
$$P' = \{n + a_0e_0^* \in N' \oplus \ZZ \ | \ a_0 \geq \check\psi_0(n)\},$$
which is exactly what we get when we apply Construction~\ref{prototype-constr} to $\tau$. The Gorenstein degree is here $\rho' = (0,1) \in N' \oplus \ZZ$; it induces an injective and saturated homomorphism $\theta': \NN \to P', \: 1 \mapsto \rho',$ of sharp toric monoids. 

There is a canonical subset $E' \subset P'$ such that every $p' \in P'$ has a unique decomposition $p' = e' + k\rho'$ for $e' \in E$ and $k \in \NN$. Explicitly, we have 
$$E' = \{n + a_0e_0^* \ | \ a_0 = \check\psi_0(n)\}.$$
A face $F \subseteq P'$ is called \emph{essential} if $F \subseteq E'$; $E'$ is the union of the essential faces. Under the isomorphism 
$$\phi': \enspace N' \to E', \quad n \mapsto (n,\check\psi_0(n)),$$
the essential faces in $P'$ correspond exactly to the toric monoids in 
$$\E_0 := \{N' \cap \sigma \ | \sigma \in \check\Sigma_0\},$$
the set of intersections between cones in $\check\Sigma_0$ and the lattice $N'$. Since $\rho'$ generates the interior of $P'$, every face in $P'$ is essential.

As in Construction~\ref{prototype-constr}, we obtain a log smooth and saturated morphism; this time, we denote it by $f': \bAA_{P'} \to \bAA^1_t$. It defines a toroidal crossing space $V(\tau)$; it is stratified by toric strata of the form $V_{F'} := \Spec \kk[F'] \subset \bAA_{P'}$ for essential faces $F' \subseteq E'$. Our well-adjusted triple $(V,Z,s)$ will live on the toroidal crossing space $V := V(\tau) \times \bAA^q$. 

To construct $(V,Z,s)$, we use the remaining polytopes $\Delta_1,...\Delta_q$. They give rise to a sharp toric monoid
$$P := \{n + a_0e_0^* + \sum_{i = 1}^q a_ie_i^* \in N' \oplus \ZZ \oplus \ZZ^{q} \ | \ \forall i: a_i \geq \check\psi_i(n)\}.$$
It is Gorenstein with Gorenstein degree $s = \sum_{i = 0}^q e_i^*$.
However, for the construction, Gross--Siebert uses the distinguished element $\rho = e_0^* \in P$, defining an injective and saturated homomorphism $\theta: \NN \to P,\: 1 \mapsto \rho$. The well-adjusted triple $(V,Z,s)$ will be obtained as the central fiber of 
$$f: \bAA_P := \Spec \kk[P] \to \bAA^1_t,$$
where the total space $\bAA_P$ is endowed with the compactifying log structure coming from the central fiber.

The map $f: \bAA_P \to \bAA^1_t$ is a generically log smooth family. Over $\bAA^1_t \setminus \{0\}$, the log structure is trivial. The family is trivial as well, with fiber $\Spec \kk[P_\rho/\ZZ\rho]$ where
$$P_\rho/\ZZ\rho = \{n + \sum_{i = 1}^q a_ie_i^* \in N' \oplus \ZZ^q \ | \ a_i \geq \check\psi_i(n)\}.$$
After setting
$$\Delta_+ := \mathrm{Conv}\left(\bigcup_{i = 1}^q \Delta_i \times \{e_i\}\right) \subset M'_\RR \oplus \RR^q$$
with $K_+$ the cone spanned by $\Delta_+$, we have $P_\rho/\ZZ\rho = K_+^\vee \cap (N' \oplus \ZZ^q)$. In general, the family is not smooth over $\bAA^1_t \setminus \{0\}$; we will later impose additional conditions on $\Delta_+$ that bound the 'badness' of the singularities in the general fiber.

To describe the central fiber $f^{-1}(0)$, we use the canonical set
$$E = \{n + \sum_{i = 0}^q a_ie_i^* \ | \ a_0 = \check\psi_0(n)\}$$
such that every $p \in P$ has a unique decomposition $p = e + k\rho$ with $e \in E$ and $k \geq 0$. Then $f^{-1}(0)$ is stratified by toric strata of the form $V_F := \Spec \kk[F] \subset \bAA_P$ for \emph{essential faces} $F \subseteq E$.
The bijection
$$\phi: \enspace N' \oplus \NN^q \to E, \quad (n,a_1,...,a_q) \mapsto n + \check\psi_0(n)e_0^* + \sum_{i = 1}^q (a_i + \check\psi_i(n))e_i^*,$$
allows us to describe the essential faces:
If $F' \in \E_0$ and $G \subseteq \NN^q$ is a face, then $\phi(F' + G)$ is an essential face of $P$. Every essential face of $P$ is of this form for unique $F' \in \E_0$ and $G \subseteq \NN^q$. We denote the set of essential faces by $\E_0 \oplus \NN^q$.

The bijection $\phi$ induces an isomorphism of underlying schemes between the central fiber of $f^{-1}(0)$ and $V(\tau) \times \bAA^q$. In particular, it has a stratification coming from $V(\tau)$. If $\rho$ is a face of $\tau$ and $F' = \sigma_\rho \cap N'$ for the cone $\sigma_\rho \in \check\Sigma_0$ associated with $\rho$, then we write $V_\rho := \Spec \kk[F'] \times \bAA^q$ for the corresponding stratum.

It is worthwhile to have a closer look at the log singular locus inside the central fiber. We denote the set of edges in $\tau$ by $\Omega(\tau)$. For $1 \leq i \leq q$,
$$\Omega_i := \{\omega \ | \ \sigma_\omega \in \check\Sigma_i\} \subseteq \Omega(\tau)$$
is the subset of those edges whose corresponding cones $\sigma_\omega$ are contained in $\check\Sigma_i$; this is the case if and only if the function $\check\psi_i$ bends along $\sigma_\omega$, i.e., $\check\psi_i$ is not linear on the union of the two adjacent maximal cones of $\check\Sigma_0$. Let us write $F'_\omega = N' \cap \sigma_\omega$ for the associated toric monoid in $N'$, and let 
$$F_{\omega;i} := F'_\omega \oplus \{(a_1,...,a_q) \in \NN^q \ | \ a_i = 0\} \in \E_0 \oplus \NN^q.$$
Via $\phi: N' \oplus \NN^q \cong E$, this defines a closed subset $\Spec \kk[F_{\omega;i}]$ of the stratum $V_\omega$ of codimension $1$. Gross--Siebert defines 
$$Z_i := \bigcup_{\omega \in \Omega_i} \Spec \kk[F_{\omega;i}] \qquad \mathrm{and} \qquad Z := \bigcup_{i = 1}^q Z_i.$$
Then $Z$ is pure of dimension $d + q - 2$, and each component $\Spec \kk[F_{\omega;i}]$ is affine toric. Since the second summand in $F'_\omega$ is properly contained in $\NN^q$, the intersection $Z \cap V_\rho$ is a proper subset of the stratum $V_\rho = \Spec \kk[F'_\rho \oplus \NN^q]$ for every face $\rho \subset \tau$.

The complement $U = \bAA_P \setminus Z$ is a union of open subsets of the form $U_F = \Spec \kk[P_F]$ for faces $F \subseteq P$ and the monoid localization $P_F = P + F^{gp}$. We have to take those faces $F$ which either contain $\rho$ or which are essential and satisfy $F \not\subset F_{\omega;i}$ for all $1 \leq i \leq q$ and $\omega \in \Omega_i$. Let us denote the latter set of faces by $\F$ for short.

If $\rho \in F$, then $U_F \subseteq f^{-1}(\bAA^1_t \setminus \{0\})$ does not contain any point of the central fiber. There might be additional singular points that are not yet captured by $Z$ as defined above---so we have to enlarge $Z$ in order to obtain a generically log smooth family---, but they are all accounted for by $\Delta_+$. If $F \in \F$, decompose $F = \phi(F' + G)$ with $F' \in \E_0$ and $G \subseteq \NN^q$ a face. A careful computation yields an isomorphism $P_F \cong P'_{F'} \oplus \NN^q$ compatibly with $\rho$ and $\rho'$ so that $f: \bAA_P \to \bAA^1_t$ is log smooth on $U_F$. In particular, $Z$ captures all log singularities of $f: \bAA_P \to \bAA^1_t$ inside $f^{-1}(0)$.

Recall that, in order to define a section of $\cL\cS_V$ on the toroidal crossing space $V = V(\tau) \times \bAA^q$, we must exhibit an isomorphism between the ghost sheaves of $\bAA_P$ and $\bAA_{P'} \times \bAA^q$ on the two central fibers---identified via $\phi$---which is compatible with the maps to the base $\bAA^1_t$, and which turns the two log structures into log structures of the same type. Both is achieved by $P_F \cong P'_{F'} \oplus \NN^q$ as follows: We have a commutative diagram
\[
 \xymatrix{
 U_F \ar[r] \ar[d]^\cong & \bAA_P \ar[r]^{f} & \bAA^1_t \ar@{=}[d] \\
 U_{F'} \times \bAA^q \ar[r] & \bAA_{P'} \times \bAA^q \ar[r]& \bAA^1_t \\
 }
\]
of log schemes. It induces an isomorphism $f^{-1}(0) \cap U_F \cong V \cap U_{F'}$ on the central fiber; the underlying morphism of schemes is \emph{not} the restriction of $f^{-1}(0) \cong V(\tau) \times \bAA^q$ via $\phi$ but differs from it by an automorphism compatible with the pre-ghost structure on $V$. Taking this automorphism into account, we obtain---only on the central fiber, not the whole of $U_F$---the correct identifications of the ghost sheaf, and it also implies that both log structures are of the same type because this is always the case if two log structures are obtained from the same one along two different \'etale maps that are compatible with each other on the level of the ghost sheaves in a suitable sense.

\subsubsection*{The local model in codimension $1$}

Let us first study an example of the construction.

\begin{ex}\label{codim-1-monoid}\note{codim-1-monoid}
 Let $\ell \geq 1$, and let $m_i \geq 0$ for $1 \leq i \leq q$ for some $q \geq 1$. Then we can take $M' = \ZZ$ and $\tau = [0,\ell]$, $\Delta_i = [0,m_i]$. The functions are $\check\psi_0(n) = -\mathrm{inf}\{\ell n,0\}$ and $\check\psi_i(n) = -\mathrm{inf}\{m_in,0\}$. Then we have by definition
 $$P = \{n + \sum_{i = 0}^q a_ie_i^* \ | \ \forall i: a_i \geq \check\psi_i(n)\}.$$
 Let $e = 1$ be the generator of $N'$. Then
 $$g_x := e,\quad g_y := -e + \ell e_0^* + \sum_{i = 1}^q m_ie_i^*, \quad g_t := e_0^*, \quad g_i := e_i^* \enspace \mathrm{for}\enspace i = 1, ..., q$$
 are elements of $P$, and they generate it as a monoid. With $x = z^{g_x}$ etc., we obtain
 $$\kk[P] = \kk[x,y,t,z_1,...,z_q]/(xy - t^\ell \prod_{i = 1}^q z_i^{m_i}).$$
 The central fiber is $\kk[x,y,z_1,...,z_q]$, and the log singular locus is given, with a natural possibly non-reduced scheme structure, by $\{\prod_{i = 1}^q z_i^{m_i} = 0\}$. We denote this family (with a possible extension) by 
 $$L(\ell; m_1,...,m_q; s) := L(\ell; m_1,...,m_q) \times \bAA^s := \Spec \kk[P] \times \bAA^s \to \bAA^1_t.$$
\end{ex}

Let us go back to the general case. For every $\omega \in \Omega(\tau)$, we have an open subset $U_\omega := U_{F_\omega} = \Spec \kk[P_{F_\omega}]$ with $F_\omega := \phi(F'_\omega)$. Together, they cover the locus $\U_1V$ on the central fiber where at most two components intersect. With 
$$\check\psi_{i;\omega}(n) := \mathrm{sup}\{\check\psi_i(n + f) - \check\psi_i(f) \ | \ f \in \phi(F'_\omega)\},$$
we have 
$$P_\omega := P_{F_\omega} = \{n + \sum_{i = 0}^q a_ie_i^* \ | \ \forall i: a_i \geq \check\psi_{i;\omega}(n)\}.$$
Let $\sigma^+$ and $\sigma^-$ be the two maximal cones which are adjacent to $\sigma_\omega$ in $\check\Sigma_0$, and let $F^+ = N' \cap \sigma^+$ and $F^- = N' \cap \sigma^-$. Let $H^+ = F^+ + F_\omega'^{gp}$ and $H^- = F^- + F_\omega'^{gp}$. Then $\check\psi_{i;\omega}$ is linear on both $H^+$ and $H^-$ for all $0 \leq i \leq q$. If $h \in H^+$ is an element with $H^+ = F_\omega'^{gp} \oplus h \cdot \NN$ and $H^- = F_\omega'^{gp} \oplus (-h) \cdot \NN$---such an element exists---, then we find an isomorphism
$$Q \oplus F_\omega'^{gp} \xrightarrow{\cong} P_\omega, \quad g_x \mapsto h + \sum_{i = 0}^q \check\psi_{i;\omega}(h)e_i^*,\: g_y \mapsto -h + \sum_{i = 0}^q \check\psi_{i;\omega}(-h)e_i^*,\: g_t \mapsto e_0^*, \: g_i \mapsto e_i^*,$$
where $Q$ is the monoid $P$ of Example~\ref{codim-1-monoid} with $\ell = \check\psi_{0;\omega}(h) + \check\psi_{0;\omega}(-h) \geq 1$ and $m_i = \check\psi_{i;\omega}(h) + \check\psi_{i;\omega}(-h) \geq 0$. We have $m_i \geq 1$ if and only if $\omega \in \Omega_i$. Summarizing the discussion, we find that $f: \bAA_P \to \bAA^1_t$ is (locally) isomorphic to $L(\ell;m_1,...,m_q;d-1) \to \bAA^1_t$ on $U_\omega$. In particular, the central fiber is locally isomorphic to $L_0(\ell;m_1,...,m_q;d-1)$ on $\U_1V$ where $L_0(-)$ has the obvious meaning.

\subsubsection*{The case where $\Delta_+$ is an elementary simplex or a standard simplex}

In \cite{GrossSiebertII}, the case of the construction where $\Delta_+$ is an elementary simplex or a standard simplex in $M'_\RR \oplus \RR^q$ (of some dimension; it cannot have full dimension) plays a special role because then $f: \bAA_P \to \bAA^1_t$ is particularly well-behaved. To avoid any confusion, we fix the following definition.

\begin{defn}\label{simplex-type-defn}\note{simplex-type-defn}
 Let $M$ be a lattice, and $P \subseteq M_\RR$ be a lattice polytope, not necessarily of full dimension.
 \begin{enumerate}[label=(\roman*)]
  \item We say that $P$ is a \emph{simplex} if the number of vertices of $P$ is equal to $\mathrm{dim}(P) + 1$, the minimal number possible.
  \item We say that $P$ is an \emph{elementary simplex}\index{simplex!elementary} if it is a simplex, and it contains no other lattice points than its vertices.
  \item We say that $P$ is a \emph{standard simplex}\index{simplex!standard} if it is an elementary simplex, and if, for some (equivalently any) numbering $v_0,...,v_d$ of its vertices, we can extend $v_1 - v_0, ..., v_d - v_0$ to a basis of the lattice $M$. This is equivalent to that $v_1 - v_0,...,v_d - v_0$ is a basis of the monoid $A_\ZZ(P) - v_0$ where $A(P)$ is the real affine space spanned by $P$, and $A_\ZZ(P) := A(P) \cap M$.\footnote{To see the equivalence of the two formulations, consider the torsion in the quotient of $M$ respective $A_\ZZ(P)$ by the sublattice generated by $v_1 - v_0, ..., v_d - v_0$.}
 \end{enumerate}
\end{defn}
\begin{ex}
 $\mathrm{Conv}((0,1,0),(1,1,0),(1,0,0),(0,0,2)) \subseteq \RR^3$ is an elementary simplex which is not a standard simplex.
\end{ex}

We study the standard case first.

\begin{lemma}[\cite{GrossSiebertII}, Prop.~2.2]\label{standard-simplex-gen-fib}\note{standard-simplex-gen-fib}
 The general fiber of $f: \bAA_P \to \bAA^1_t$ is smooth if and only if $\Delta_+$ is a standard simplex.
\end{lemma}
\begin{proof}
 The general fiber is $\Spec \kk[P_\rho/\ZZ\rho]$ with $P_\rho/\ZZ\rho = K_+^\vee \cap (N' \oplus \ZZ^q)$. Thus, the general fiber is smooth if and only if $K_+$ is a regular cone. The primitive ray generators of $K_+$ are precisely the vertices $v_0,...,v_d$ of $\Delta_+$; they form part of a basis of $M' \oplus \ZZ^q$ if and only if $\Delta_+$ is a standard simplex. For the non-trivial direction, note that $M' \oplus \ZZ^q = \ZZ v_0 \oplus \{(m,a_1,...,a_q) \ | \ \sum_{i = 1}^q a_i = 1\}$.
\end{proof}

When $\Delta_+$ is an elementary simplex, the singularities of the general fiber are very mild.

\begin{lemma}[\cite{GrossSiebertII}, Prop.~2.2]\label{elem-simplex-gen-fib}\note{elem-simplex-gen-fib}
 Assume that $\Delta_+$ is an elementary simplex. Then the general fiber of $f: \bAA_P \to \bAA^1_t$ is normal, Gorenstein, terminal, $\QQ$-factorial, smooth in codimension $3$, and has abelian quotient singularities.
\end{lemma}
\begin{proof}
 As a toric variety, the general fiber $\Spec \kk[P_\rho/\ZZ\rho]$ is normal and Cohen--Macaulay. We obtain from the criterion at the beginning of \cite[\S 6]{Altmann1995} that it is Gorenstein because the primitive ray generators of $K_+$ are precisely the vertices of $\Delta_+$. Then \cite[Prop.~11.4.12]{Cox2011} shows that the general fiber has terminal singularities. A dimension count yields that $\Spec \kk[P_\rho/\ZZ\rho]$ is a simplicial toric variety, and hence it is $\QQ$-factorial and has abelian quotient singularities. Then \cite[Prop.~11.4.22]{Cox2011} yields that $\Spec \kk[P_\rho/\ZZ\rho]$ is, as a simplicial Gorenstein terminal toric variety, smooth in codimension $3$.\footnote{$\QQ$-factoriality is the only one of these properties which is only Zariski local but not \'etale local. In particular, we cannot conclude that a space which has the general fiber as an \'etale local model is $\QQ$-factorial.}
\end{proof}

In particular, we can apply the following well-known theorem of Altmann, which is proven in \cite[\S 6]{Altmann1995}:

\begin{thm}\label{Altmann-rigid}\note{Altmann-rigid}\index{rigidity!affine toric variety}
 Let $Y$ be a $\QQ$-Gorenstein affine toric variety. Assume that $Y$ is smooth in codimension $3$. Then $Y$ is rigid, i.e., $\T^1_Y = 0$.
\end{thm}
\begin{cor}\label{elem-GS-gen-fib-rigid}\note{elem-GS-gen-fib-rigid}
 Assume that $\Delta_+$ is an elementary simplex. Then the general fiber of the morphism $f: \bAA_P \to \bAA^1_t$ is rigid.
\end{cor}

Thus, the general fiber does not admit any non-trivial local deformation.

Next, we investigate the local model in codimension $1$ in the elementary case. Assume that $\Delta_+$ is an elementary simplex. Since, for $1 \leq i \leq q$, each $\Delta_i$ is a face of $\Delta_+$, they must be elementary simplices as well. Since each vertex of $\Delta_i$ shows up in $\Delta_+$, we find $\sum_{i = 1}^q \mathrm{dim}(\Delta_i) \leq d = \mathrm{dim}(\tau)$.
Let $\omega \in \Omega(\tau)$ be an edge, and set 
$$m_{i;\omega} := \check\psi_{i;\omega}(h) + \check\psi_{i;\omega}(-h), \quad m_\omega := \sum_{i = 1}^q m_{i;\omega},$$
where we maintain the notations from the previous section. As stated above, we have $m_{i;\omega} \not= 0$ if and only if $\omega \in \Omega_i$. In this case, $\Delta_i$ has an edge $\omega_i$ which is parallel to $\omega$; the lattice length of this edge is $m_{i;\omega}$, i.e., the edge contains $m_{i;\omega} + 1$ lattice points, including the two vertices. Since each $\Delta_i$ is an elementary simplex, we have $m_{i;\omega} \leq 1$. Now assume that we have two indices $i \not= j$ with both $m_{i;\omega} \not= 0$ and $m_{j;\omega} \not= 0$. Let $\omega_k^{\pm}$ be the two vertices of the edge $\omega_k$ with $k = i,j$. Then the four vertices $\omega_i^\pm \times e_i$ and $\omega_j^\pm \times e_j$ of $\Delta_+$ lie in a common plane, contradicting the assumption that $\Delta_+$ is a simplex. Thus, $m_\omega \leq 1$. This means that the central fiber of $f: \bAA_P \to \bAA^1_t$ is particularly simple on $\U_1V$ because it is locally isomorphic to the central fiber of 
$$\Spec \CC[x,y,t,z_1,...,z_q,u_1,...,u_s]/(xy - t^\ell z_i) \to \bAA^1_t$$
for some $1 \leq i \leq q$ and $s = d - 1$. The equation $xy = t^\ell \prod_{i = 1}^q z_i^{m_i}$ simplifies to $xy = t^\ell z_i$.

\subsubsection*{Some additional examples}

We have already seen in Section~\ref{examples-intro-sec} some examples of this construction. Here are some additional examples. In particular, we will see that there are log toroidal families of elementary Gross--Siebert type which are not of standard Gross--Siebert type, a phenomenon that appears only in relative dimension $\geq 4$.

\begin{figure}
 \begin{mdframed}
  \begin{center}
   \includegraphics[scale=0.8]{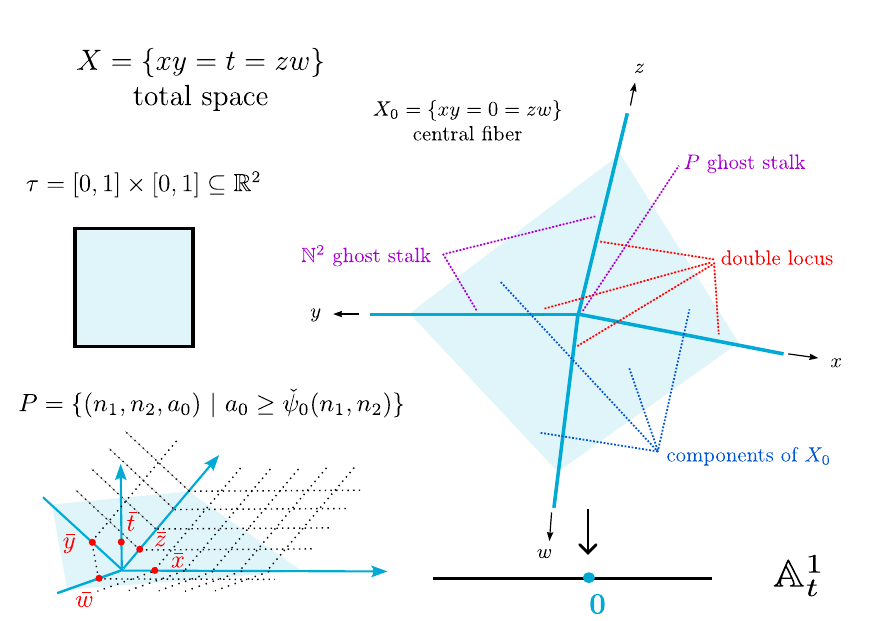}
  \end{center}
  \caption{Example~\ref{four-surfaces-example}}
 \end{mdframed}
\end{figure}

\begin{ex}\label{four-surfaces-example}\note{four-surfaces-example}
 We set $\tau = [0,1] \times [0,1] \subseteq \RR^2$ and $q = 0$. Then 
 $$\check\psi_0(n_1,n_2) = -\mathrm{inf}\{0,n_1,n_2,n_1 + n_2\}.$$
 The monoid 
 $$P = \{(n_1,n_2,a_0) \ | \ a_0 \geq \check \psi_0(n_1,n_2)\}$$
 is generated by 
 $$\bar x = (1,0,0), \enspace \bar y = (-1,0,1), \enspace \bar z = (0,1,0), \enspace \bar w = (0,-1,1).$$
 We also set $\bar t = (0,0,1)$. Thus, we have an isomorphism $\kk[P] \cong \kk[x,y,z,w]/(xy - zw)$. Indeed, the surjection must be an isomorphism because both sides are integral of the same dimension $3$. This is the same monoid as in Example~\ref{xy-tz-example}, but now the map is given by $t = xy = zw$. This family is log smooth and vertical. The central fiber consists of four copies of $\bAA^2$. Each of them intersects two other components in a copy of $\bAA^1$, and the third other component in a single point $\{0\}$. The intersection of all four components is this single point $\{0\}$. Obviously, the central fiber is not a normal crossing space, but it is a toroidal crossing space.
\end{ex}

\begin{figure}
 \begin{mdframed}
  \begin{center}
   \includegraphics[scale=0.8]{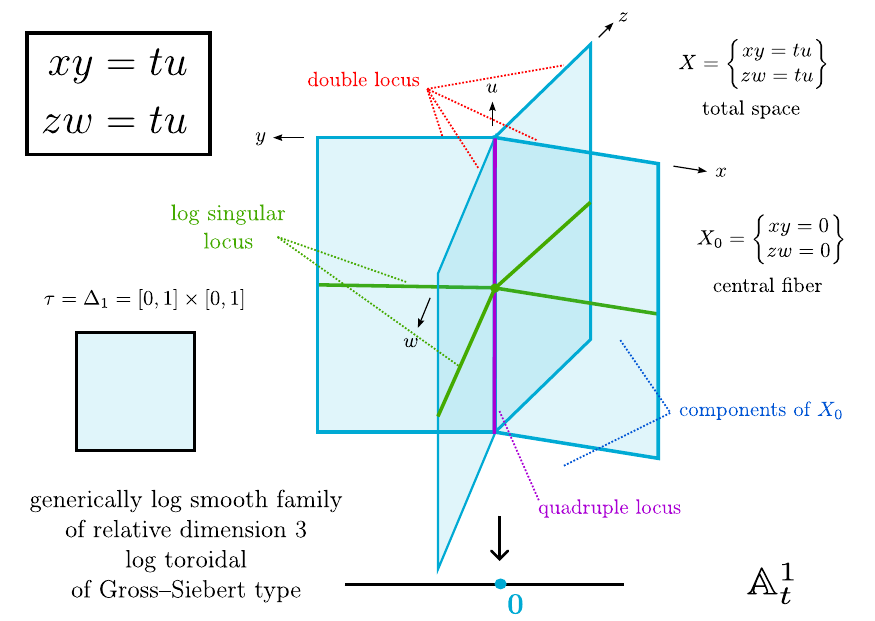}
  \end{center}
  \caption{Example~\ref{four-lines-log-sing-example}}
 \end{mdframed}
\end{figure}

\begin{ex}\label{four-lines-log-sing-example}\note{four-lines-log-sing-example}
 We take again $\tau = [0,1] \times [0,1] \subseteq \RR^2$, but this time, we set $q = 1$ and $\Delta_1 = \tau$. Then we have 
 $$\check \psi_0(n_1,n_2) = - \mathrm{inf}\{0,n_1,n_2,n_1 + n_2\} = \check \psi_1(n_1,n_2).$$
 The monoid 
 $$P = \{(n_1,n_2,a_0,a_1) \ | \ a_i \geq \check\psi_i(n_1,n_2), \ i = 0,1\}$$
 is generated by 
 \begin{center}
  \begin{tabular}{lll}
  $\bar x = (1,0,0,0)$, & $\bar y = (-1,0,1,1)$, & $\bar t = (0,0,1,0)$, \\
  $\bar z = (0,1,0,0)$, & $\bar w = (0,-1,1,1)$, & $\bar u = (0,0,0,1)$.
 \end{tabular}
 \end{center}
 Thus, we have $\kk[x,y,z,w,t,u]/(xy - tu, zw - tu) \cong \kk[P]$. The nearby fiber has an $A_1$-threefold singularity so that the family is not log toroidal of standard Gross--Siebert type. In fact, $\Delta_+ \cong [0,1] \times [0,1]$ is not even a simplex. In the central fiber, the log singular locus is given by $u = 0$ inside the double locus. It consists of four lines, one in each component of $X_0$, which meet at $0$ inside the quadruple locus.
\end{ex}

\begin{figure}
 \begin{mdframed}
  \begin{center}
   \includegraphics[scale=0.8]{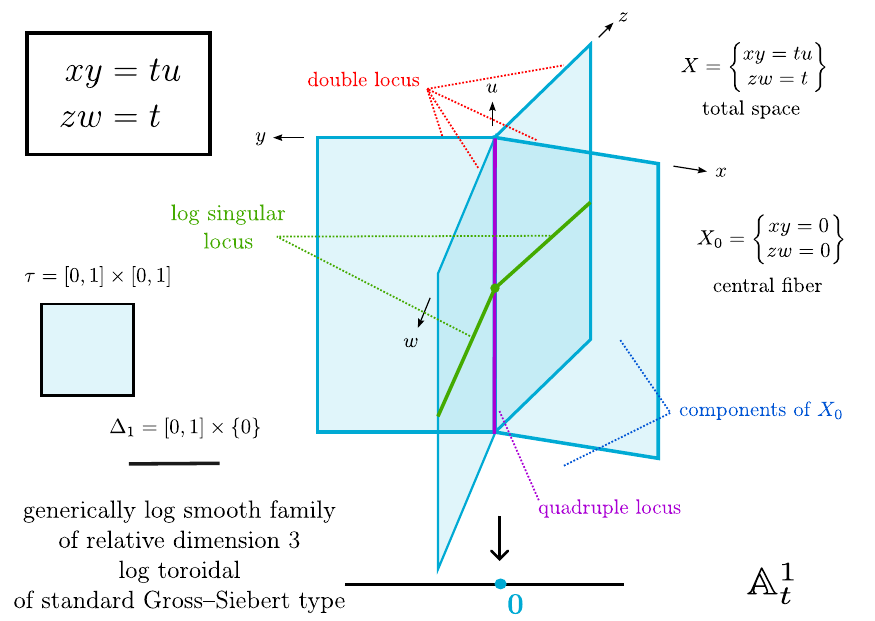}
  \end{center}
  \caption{Example~\ref{two-lines-log-sing-example}}
 \end{mdframed}
\end{figure}

\begin{ex}\label{two-lines-log-sing-example}\note{two-lines-log-sing-example}
 We take again $\tau = [0,1] \times [0,1] \subseteq \RR^2$ and $q = 1$ as well, but this time, we set $\Delta_1 = [0,1] \times \{0\} \subseteq \RR^2$. Then we have 
 $$\check\psi_0(n_1,n_2) = -\mathrm{inf}\{0,n_1,n_2,n_1 + n_2\}, \quad \check\psi_1(n_1,n_2) = -\mathrm{inf}\{0,n_1\}.$$
 The monoid 
 $$P = \{(n_1,n_2,a_0,a_1) \ | \ a_i \geq \check\psi_i(n_1,n_2), \ i = 0,1\}$$
 is generated by 
 \begin{center}
  \begin{tabular}{lll}
  $\bar x = (1,0,0,0)$, & $\bar y = (-1,0,1,1)$, & $\bar t = (0,0,1,0)$, \\
  $\bar z = (0,1,0,0)$, & $\bar w = (0,-1,1,0)$, & $\bar u = (0,0,0,1)$.
 \end{tabular}
 \end{center}
 The equations are slightly different now: We obtain $xy = tu$ and $zw = t$. Now the nearby fiber is smooth so that $f: X := \bAA_P \to \bAA^1_t$ is log toroidal of standard Gross--Siebert type. In fact, we have $\Delta_+ \cong [0,1]$, a standard simplex. As a scheme, the central fiber is the as in the previous Example~\ref{four-lines-log-sing-example}. However, in this example, the log singular locus $Z$ consists only of two lines in diametrally opposed components of the double locus, i.e., the other two components of the double locus contain only the intersection point $0$ of the two lines in the quadruple locus as log singular locus. When taking the equations $xy = t(u + 1)$ and $zw = t(u - 1)$, then we can even pull the four branches of the previous Example~\ref{four-lines-log-sing-example} apart into two times two branches.
\end{ex}

\begin{ex}\label{elementary-not-standard-example}\note{elementary-not-standard-example}
 This is an example of a log toroidal family which is of elementary Gross--Siebert type but not of standard Gross--Siebert type. We start with the quadrangle
 $$\tau = \mathrm{Conv}((1,0),(0,1),(-1,0),(0,-1)) \subseteq \RR^2,$$
 so we have 
 $$\check\psi_0(n_1,n_2) = -\mathrm{inf}\{n_1,-n_1,n_2,-n_2\} = \mathrm{sup}\{|n_1|,|n_2|\}.$$
 Thus, 
 $$P' = \{(n_1,n_2,a_0) \ | \ a_0 \geq \check\psi_0(n_1,n_2)\}$$
 is the cone over a square with edge length $2$. This monoid has nine generators and $20$ relations (according to a computation in Macaulay2), so we will not describe the family in terms of generators and relations. To obtain our example of a family which is of elementary but not standard Gross--Siebert type, we set $q = 2$ and take the two diagonal intervals $\Delta_1 = \mathrm{Conv}((0,0),(1,1))$ and $\Delta_2 = \mathrm{Conv}((0,0),(1,-1))$. Then 
 $$\check\psi_1(n_1,n_2) = -\mathrm{inf}\{0,n_1 + n_2\}, \quad \check\psi_2(n_1,n_2) = -\mathrm{inf}\{0,n_1 - n_2\},$$
 and the family is given by the monoid 
 $$P = \{(n_1,n_2,a_0,a_1,a_2) \ | \ a_i \geq \check\psi_i(n_1,n_2), \ i = 0,1,2\}.$$
 By definition, we have 
 \begin{align}
  \Delta_+ &= \mathrm{Conv}\left(\bigcup_{q = 1}^2 \Delta_i \times \{e_i\}\right) \nonumber \\
  &= \mathrm{Conv}((0,0,1,0),(1,1,1,0),(0,0,0,1),(1,-1,0,1)) \subseteq \RR^4. \nonumber
 \end{align}
 Thus, $\Delta_+$ is contained in the affine space 
 $$(0,0,1,0) + \RR \cdot (1,0,0,0) + \RR \cdot (0,1,0,0) + \RR \cdot (0,0,-1,1).$$
 After the affine transformation given by these vectors, we have 
 $$\Delta_+ \cong \mathrm{Conv}((0,0,0),(1,1,0),(0,0,1),(1,-1,1)) \subseteq \RR^3.$$
 This is an elementary simplex, but the differences of these points do not form a basis of $\ZZ^3 \subseteq \RR^3$, i.e., this is not a standard simplex.
\end{ex}

\section{Log toroidal families of (elementary) Gross--Siebert type}

A log toroidal family of (elementary) Gross--Siebert type is a log toroidal family whose local models arise from the construction in the previous section, with $\Delta_+$ an elementary simplex in the elementary case. We give a slightly modified definition here for the case of well-adjusted triples.

\begin{defn}\label{log-tor-GS-type}\note{log-tor-GS-type}
 Let $(V,Z,s)$ be a well-adjusted triple, and let $f_0: X_0 \to S_0$ be the associated generically log smooth family. Let $A \in \mathbf{Art}_\NN$. A generically log smooth deformation $f: X_A \to S_A$ of $f_0: X_0 \to S_0$ with log smooth locus $U_A$ is called a \emph{log toroidal family of Gross--Siebert type} if we can find finitely many commutative diagrams
 \[
  \xymatrix{
   X_A \ar[d]^{f_A} & V_i \ar[l]_{g_i} \ar[r]^{h_i} & L_i \ar[r] \ar[d] & \bAA_{P_i} \ar[d]^{f_i}\\
   S_A \ar@{=}[rr] & & S_A \ar[r] & \bAA^1_t\\
  }
 \]
 where $f_i: \bAA_{P_i} \to \bAA^1_t$ is a local model with log singular locus $Z_i$ associated with data $(M'_i,N'_i,\tau_i,\Delta_{1;i}, ..., \Delta_{q_i;i})$, the map $S_A \to \bAA^1_t$ is the canonical morphism, $L_i$ is the fiber product of generically log smooth families, $h_i: V_i \to L_i$ is \'etale, $g_i: V_i \to X_A$ is \'etale, we have $h_i^{-1}(L_i \setminus Z_i) = g_i^{-1}(U_A)$, both $g_i$ and $h_i$ are strict \'etale on the locus where the log structures are defined, and $X_A = \bigcup_i g_i(V_i)$ is an open covering. We say it is a \emph{log toroidal family of elementary Gross--Siebert type}\index{Gross--Siebert type!elementary Gross--Siebert type} if all local models can be chosen such that $\Delta_{+;i}$ is an elementary simplex. We say it is a \emph{log toroidal family of standard Gross--Siebert type}\index{Gross--Siebert type!standard Gross--Siebert type} if all local models can be chosen such that $\Delta_{+;i}$ is a standard simplex.
\end{defn}

 A couple of explanations are in order. In the first step, the definition should be read with $A = \kk$, i.e., it gives a definition for when the generically log smooth family $f_0: X_0 \to S_0$ associated with $(V,Z,s)$ is a log toroidal family of (elementary/standard) Gross--Siebert type. As opposed to the general case, we require the local models over the base chart $S_0 \to A_\NN$, and we insist that the log singular locus $Z$ on $f_0: X_0 \to S_0$ coincides with the log singular locus in the local model which we have specified above---no change of the log singular locus is allowed here, unlike in our general definition above.
 
 In the second step, the definition should be read with a general $A$. By base change, if $f_A: X_A \to S_A$ is a log toroidal family of (elementary/standard) Gross--Siebert type, then $f_0: X_0 \to S_0$ is as well. Thus, in the second step, we can assume that we already have a log toroidal family of (elementary/standard) Gross--Siebert type, and we are defining its deformations. The local models of the deformation are not required to be the same as the ones exhibited for the central fiber. So if one has two different local models of (elementary) Gross--Siebert type, and they happen to define the same central fiber, then both define valid and different log toroidal deformations of (elementary/standard) Gross--Siebert type. We will see, however, later that, in the elementary case at least, these deformations are locally unique up to non-unique isomorphism, in the same way as log smooth deformations are.

\section{Unisingular deformations}

In \cite{GrossSiebertII}, Gross and Siebert show that toric log Calabi--Yau spaces admit a well-behaved deformation theory. This deformation theory works because toric log Calabi--Yau spaces admit local models as constructed above with $\Delta_+$ an elementary simplex. We guide the reader through the argument and formulate the results in their natural generality of log toroidal families of elementary Gross--Siebert type. The theory that we spell out here is certainly of interest for more general singularities, but then deformations will be, in general, no longer locally unique.

\subsubsection*{The unisingular deformation functor}

The definition of unisingular deformations is inspired by the local model $xy = t^\ell z_i$ that we obtained above for $\Delta_+$ an elementary simplex. As a variant of our $L(\ell;m_1,...,m_q;s)$, let 
$$L(\ell) := \Spec k[x,y,z,t]/(xy - t^\ell z) \to \bAA^1_t$$
with the divisorial log structures induced by $t = 0$; it becomes generically log smooth once we remove the point $Z(\ell) = \{p_0\}$ in the central fiber $L(\ell)_0 = \Spec k[x,y,z]/(xy)$. To adapt to arbitrary dimension of $f_0: X_0 \to S_0$, we write $L(\ell;s) := L(\ell) \times \bAA^s$.

\begin{defn}\label{unisingular-defn}\note{unisingular-defn}
 Let $(V,Z,s)$ be a well-adjusted triple with associated generically log smooth family $f_0: X_0 \to S_0$. Then $(V,Z,s)$ is \emph{unisingular} if every point $\bar v \in Z \cap \cS_1V$ admits an \'etale neighborhood $\pi: W_0 \to X_0$ and an \'etale morphism $\tau: W_0 \to L(\ell;s)_0$ with $\pi^{-1}(Z) = \tau^{-1}(Z(\ell;s)_0)$ that can be enhanced to a strict \'etale morphism $W_0 \setminus \pi^{-1}(Z) \to L(\ell;s)_0 \setminus Z(\ell;s)_0$ over $S_0$.
\end{defn}

\begin{rem}
 We require the condition only for $\bar v \in Z \cap \cS_1V$. Outside of $Z$, no condition is necessary because there the family is log smooth. For points of $Z$ in deeper strata like $\cS_2V$, we do not require any condition. Furthermore, for every $\bar v \in Z \cap \cS_1V$, there is only one possible value of $n$ for which we have an isomorphism as above. Namely, on $L(\ell)_0$, the parameter $\ell$ can be reconstructed from the stalk of the ghost sheaf at points in $\{x = 0, y = 0, z \not= 0\}$, and on $V$, the parameter $\ell$ can be reconstructed from the stalks of $\cP$ in points $\bar v \in \cS_1V$. Since $\cP$ is (locally) constant on strata $\cS \subseteq \cS_1V$, also the possible $\ell$ is locally constant on $\cS_1V$. 'uni' in the name unisingular is Latin for 'one'; it refers to the exponent of $z$ in the equation $xy = t^\ell z$.
\end{rem}

\begin{ex}
 Let $f: \bAA_P \to \bAA^1_t$ be a local model of elementary Gross--Siebert type, coming from data $(M',N',\tau,\Delta_1,...,\Delta_q)$ with $\Delta_+$ an elementary simplex. Then $V = V(\tau) \times \bAA^q$ is a toroidal crossing space, and the log singular locus $Z$ of the central fiber of $f: \bAA_P \to \bAA^1_t$ together with the log structure $s \in \cL\cS_V(V \setminus Z)$ forms a well-adjusted triple $(V,Z,s)$. Let $f_0: X_0 \to S_0$ be the associated generically log smooth family. The above computation of $f_0: X_0 \to S_0$ on $\U_1V$ shows that $(V,Z,s)$ is unisingular. If $\Delta_+$ is not an elementary simplex, then $(V,Z,s)$ might not be unisingular.
\end{ex}

\begin{ex}
 Let $(V,Z,s)$ be a well-adjusted triple such that the associated generically log smooth family $f_0: X_0 \to S_0$ is a log toroidal family of elementary Gross--Siebert type. Then $(V,Z,s)$ is unisingular. Namely, if we denote the toroidal crossing space in the local model $f_i: \bAA_{P_i} \to \bAA^1_t$ of Definition~\ref{log-tor-GS-type} by $V_i$, then we have $g_i^{-1}(\U_1V) = h_i^{-1}(\U_1V_i)$. Thus, the loci where we need to have the local model $L_0(\ell;s)$ coincide on $V$ and on the local model of $f_0: X_0 \to S_0$---and, on the latter, we have them by the previous example.
\end{ex}

\begin{defn}
 Let $(V,Z,s)$ be a unisingular well-adjusted triple. Then a \emph{unisingular deformation} of $f_0: X_0 \to S_0$ is a generically log smooth deformation $f_A: X_A \to S_A$ such that for every $\bar v \in Z \cap \cS_1V$, there is an \'etale neighborhood $\pi: W_A \to X_A$ and an \'etale morphism $\tau: W_A \to L(\ell;s)_A$ with $\pi^{-1}(Z) = \tau^{-1}(Z(\ell;s)_0)$ which can be enhanced to a strict and smooth log morphism $W_A \setminus \pi^{-1}(Z) \to L(\ell;s)_A \setminus Z(\ell;s)_0$ over $S_A$. We denote the functor of isomorphism classes of unisingular deformations by 
$$\mathrm{LD}_{X_0/S_0}^{uni}: \mathbf{Art}_\NN \to \mathbf{Set}.$$
\end{defn}
\begin{rem}
 Again, we have no condition on deeper strata.
\end{rem}
\begin{ex}
 Let $(V,Z,s)$ be a well-adjusted triple such that $f_0: X_0 \to S_0$ is a log toroidal family of elementary Gross--Siebert type. Then every generically log smooth deformation $f_A: X_A \to S_A$ which is a log toroidal family of elementary Gross--Siebert type is a unisingular deformation.
\end{ex}

\begin{lemma}
 The functor 
$$\mathrm{LD}_{X_0/S_0}^{uni}: \mathbf{Art}_\NN \to \mathbf{Set}$$
of unisingular deformations is a deformation functor, i.e., it satisfies $(H_0)$, $(H_1)$, and $(H_2)$.
\end{lemma}
\begin{proof}
 The condition $(H_0)$ is clear. For $(H_1)$, let $B = A'' \times_A A'$ for some morphism $A' \to A$ and a surjective morphism $A'' \to A$, let $X_A \to S_A$ be a unisingular deformation, and let $X_{A'} \to S_{A'}$, $X_{A''} \to S_{A''}$ be unisingular liftings (they come with a \emph{choice} of map $X_A \to X_{A'}$, $X_A \to X_{A''}$, which is not encoded in the deformation functor). We have to show that the generically log smooth family 
 $$X_{A'} \sqcup_{X_A} X_{A''} := (|X_0|,\ \cO_{X_{A'}} \times_{\cO_{X_A}} \cO_{X_{A''}},\ \M' \times_\M \M'')$$
 is unisingular. After appropriate shrinking of the \'etale neighborhoods of the point $\bar v \in Z \cap \cS_1V$, this amounts to showing that the canonical map 
 $$L(n)_{A'} \sqcup_{L(n)_A} L(n)_{A''} \to L(n)_B$$
 is an isomorphism where $L(n)_B := L(n) \times_{\bAA^1} S_B$ etc. This is true on the level of $\cO$ because both are flat thickenings of $L(n)_{A'}$, cf.~also \cite[Cor.~3.6]{Schlessinger1968}; it is true on the level of $\M$ because it holds for $\cO^*$ and $\overline\M$. Finally, $(H_2)$ follows from $(H_1)$ because $\mathrm{LD}_{X_0/S_0}^{gen}$ satisfies $(H_2)$. 
\end{proof}

\subsubsection*{Unisingular deformations on $\tilde U_0 \subseteq X_0$}

The power of unisingular deformations relies on the existence of unique deformations outside the log singular locus in deeper strata of codimension $\geq 2$. 

\begin{prop}[\cite{GrossSiebertII}, Lemma~2.15]
 Let $(V,Z,s)$ be a unisingular well-adjusted triple with associated generically log smooth family $f_0: X_0 \to S_0$, let $U_0 = X_0 \setminus Z$ be the log smooth locus, and let $W_0 \subseteq X_0$ be an affine open contained in 
$$\tilde U_0 := \U_1V \cup U_0 = U_0 \cup (Z \cap \cS_1V).$$
 Then there is a unisingular deformation $W_A \to S_A$, which is unique up to non-unique isomorphism. Moreover, when $S_A \to S_{A'}$ is a small thickening by an ideal $I \subseteq A'$ and a unisingular deformation $\tilde U_A \to S_A$ is given, then 
\begin{itemize}
 \item the automorphisms of a lifting $\tilde U_{A'} \to S_{A'}$ are in $H^0(\tilde U_0, \Theta^1_{\tilde U_0/S_0} \otimes I)$;
 \item the different liftings are classified by $H^1(\tilde U_0, \Theta^1_{\tilde U_0/S_0} \otimes I)$;
 \item the obstruction to the existence of a lifting is in $H^2(\tilde U_0,\Theta^1_{\tilde U_0/S_0} \otimes I)$.
\end{itemize}
\end{prop}
\begin{proof}
 First, we have to show that any two unisingular deformations $W_A,W'_A$ are isomorphic in an \'etale neighborhood of $\bar v \in \tilde U_0$. If $\bar v \in U_0$, this is obvious, and if $\bar v \in Z \cap \cS_1V$, then it follows with the proof in \cite[Lemma 2.15]{GrossSiebertII} since they both have the same local model $L(n;r)_A$.  Since automorphisms of a lifting (as a generically log smooth family) are classified by $\Theta^1_{\tilde U_0/S_0} \otimes I$, the rest follows from the classical theory of smooth deformations.
\end{proof}

\subsubsection*{Unisingular deformations of affine spaces}

We study liftings of unisingular deformations on affine spaces.

\begin{prop}[\cite{GrossSiebertII}, Lemma~2.16]
 Let $(V,Z,s)$ be a unisingular well-adjusted triple with associated generically log smooth family $f_0: X_0 \to S_0$. Assume that $V$ is \emph{affine}. Let $f_A: X_A \to S_A$ be a unisingular deformation of $f_0: X_0 \to S_0$, and let $S_A \to S_{A'}$ be a small extension by an ideal $I \subseteq A'$. Assume that a lifting $f_{A'}: X_{A'} \to S_{A'}$ exists. Then \emph{all} liftings have the same underlying scheme $X_{A'}$.
\end{prop}
\begin{proof}
 Let $Y_{A'} \to S_{A'}$ be another lifting. By the previous proposition, after restriction to $\tilde U_0$, the difference to $X_{A'} \to S_{A'}$ is captured by an element 
 $$\theta \in H^1(\tilde U_0, \Theta^1_{X_0/S_0} \otimes I) \subseteq \mathrm{Ext}^1_{\cO_{\tilde U_0}}(W^1_{\tilde U_0/S_0} \otimes I, \cO_{\tilde U_0}).$$
 Here, $W^1$ denotes the reflexive differential forms as in Chapter~\ref{W-sec}. According to \cite[Thm.~4.4]{Vistoli1999}, the difference of the underlying flat deformations of $\tilde U_0$ is captured by an element
 $$\underline\theta \in \mathrm{Ext}^1_{\cO_{\tilde U_0}}(\Omega^1_{\underline{\tilde U_0}/k} \otimes I,\cO_{\tilde U_0}).$$
 Denoting $\psi$ the natural map between the two extension groups, we have $\underline\theta = \psi(\theta)$. Following \cite[Lemma~2.16]{GrossSiebertII}, we have $\psi(\theta) = 0$ by a diagram chase which uses that $(V,Z,s)$ is well-adjusted and that $V$ is affine. In particular, $\cO_{Y_{A'}}|_{\tilde U_0} \cong \cO_{X_{A'}}$, and this isomorphism extends to the whole space.
\end{proof}
\begin{rem}
 The proof of the proposition fails if we drop the assumption that $V$ is affine. This uniqueness holds only locally. Note also that we do not prove the existence of the first lifting---indeed, there might be no lifting at all. If the lifting always exists, then $\mathrm{LD}_{X_0/S_0}^{uni}$ is unobstructed (by definition).
\end{rem}

In fact, the proof of \cite[Lemma~2.16]{GrossSiebertII} even shows that $\theta$ is in the kernel of the map 
$$H^1(\tilde U_0, \Theta^1_{X_0/S_0} \otimes I) \to H^1(\tilde U_0, \T_{X_0/S_0} \otimes I)$$
induced by the embedding of $\Theta^1_{X_0/S_0}$ into the classical differential forms $\T_{X_0/S_0}$;
by the long exact sequence of 
$$0 \to \Theta^1_{X_0/S_0} \to \T_{X_0/S_0} \to \B_{X_0/S_0} \to 0$$
(where $\B_{X_0/S_0}$ is defined as the quotient), $\theta$ is in the cokernel
$$T(X_0/S_0) \otimes I \ := \ \mathrm{coker}\left(H^0(\tilde U_0, \T_{X_0/S_0} \otimes I) \to H^0(\tilde U_0, \B_{X_0/S_0} \otimes I)\right).$$
Conversely, every $\theta \in T(X_0/S_0) \otimes I$ actually gives rise to a unisingular lifting of $X_A \to S_A$ to $S_{A'}$ because the extension of the structure sheaf to the whole space is already given.

\begin{cor}
 In the above situation---in particular, $V$ is affine and one lifting is given---the set of all unisingular liftings is a torsor under $T(X_0/S_0) \otimes I$. In particular, if $T(X_0/S_0)$ is of finite dimension, then $\mathrm{LD}_{X_0/S_0}^{uni}$ has a hull.
\end{cor}

\begin{ex}\label{affine-T-is-0}\note{affine-T-is-0}
 Let $f_0: X_0 \to S_0$ be the central fiber of a local model $f: \bAA_P \to \bAA^1_t$ of elementary Gross--Siebert type, i.e., $\Delta_+$ is an elementary simplex. Then the argument at the end of the proof of \cite[Prop.~2.16]{GrossSiebertII}, using \cite[Lemma~2.14]{GrossSiebertII}, shows that $T(X_0/S_0) = 0$. Thus unisingular liftings are unique if they exist. Conversely, the local model provides a unisingular deformation over each $A \in \mathbf{Art}_\NN$, so that there is at least one unisingular deformation for each $A \in \mathbf{Art}_\NN$. By induction over small extensions, every unisingular deformation over $A$ is isomorphic to that given one. In particular, $\mathrm{LD}_{X_0/S_0}^{uni} = \{*\}$. The hull is $\kk\llbracket t\rrbracket$, and the miniversal, in fact universal, family is given by pull-back of $f: \bAA_P \to \bAA^1_t$ to $\kk\llbracket t\rrbracket$.
\end{ex}

\begin{ex}
 Let us give just one example beyond the original context of local models of elementary Gross--Siebert type. Let 
 $$X = \Spec \kk[x,y,z,w,u,t]/(xy - w^2 - t^2,\: zw - tu),$$
 and consider the map $f: X \to \bAA^1_t$. The central fiber has three irreducible components $\{x = w = 0\}$, $\{y = w = 0\}$, and $\{z = 0\}$; any two of them intersect in a copy of $\bAA^2$, and the intersection of all three of them is a line $\bAA^1_u$. The log singular locus of the central fiber $f: X_0 \to S_0$ is two lines $z^2 = u^2$ in $\{x = y = w = 0\}$ and one line $u = 0$ in each of the two other strata of codimension $1$. With appropriate choices of the remaining data, we obtain a unisingular triple. An explicit computation in Macaulay2 yields 
 $$\mathrm{dim}_\kk\: T(X_0/S_0) = 1.$$
 In particular, $\mathrm{LD}_{X_0/S_0}^{uni}$ has a hull and thus a miniversal family. Unfortunately, I haven't yet been able to find this family. I have tried several families over $\Spec \kk[t,s]$ with $s$ the parameter corresponding to $T(X_0/S_0)$, but none of them had a bijective (equivalently non-zero) Kodaira--Spencer map (defined in an appropriate sense). Maybe $\mathrm{LD}_{X_0/S_0}^{uni}$ is obstructed so that the hull is not $\kk\llbracket t,s\rrbracket$.
\end{ex}

\subsubsection*{Unisingular deformations of families of elementary Gross--Siebert type}

Let $(V,Z,s)$ be a well-adjusted triple such that $f_0: X_0 \to S_0$ is a log toroidal family of elementary Gross--Siebert type. Let $\B = \B_{X_0/S_0}$. On an affine open $V_0 \subseteq X_0$, we can compute $T(V_0/S_0)$ as 
$$T(V_0/S_0) = \mathrm{coker}(\Gamma(V_0,\B) \to \Gamma(V_0 \cap \tilde U_0,\B))$$
because $\Gamma(V_0,\T_{X_0/S_0}) \to \Gamma(V_0,\B)$ is surjective. After writing $j: \tilde U_0 \to X_0$ for the inclusion, we form the (quasi-coherent) cokernel 
$$\B \to j_*\B|_{\tilde U_0} \to \C_{X_0/S_0} \to 0$$
and get $T(V_0/S_0) = \Gamma(V_0,\C_{X_0/S_0})$. When we form $\C$ on the local model, then we get $\C = 0$ by Example~\ref{affine-T-is-0}. Since the formation of $\C$ commutes with flat morphisms, we obtain $\C_{X_0/S_0} = 0$ locally in the \'etale topology of $X_0$ by pull-back along $h_i: V_i \to L_i$. Because $\C_{X_0/S_0}$ is quasi-coherent, we get $\C_{X_0/S_0} = 0$. Thus, $T(V_0/S_0) = 0$, and unisingular liftings on affine opens on $X_0$ are always unique.

Every unisingular deformation $f_A: X_A \to S_A$ is in fact a log toroidal family of elementary Gross--Siebert type. This is proved via induction over small extensions in $\mathbf{Art}_\NN$. If $f': X_{B'} \to S_{B'}$ is a unisingular deformation, then its pull-back to $B$ along a small extension $B' \to B$ must admit local models. These local models can be lifted to $B'$, and the lifts must be locally isomorphic to $f': X_{B'} \to S_{B'}$. Thus, the latter family admits local models as well.

If $f: X_B \to S_B$ is a unisingular deformation, then the local models show that $X_B$ can be covered by \'etale opens where the deformation can be lifted to $B'$. This is not yet the existence of a unique deformation over every affine open, but we can get this from the standard theory of \cite[Expos\'e~III]{Grothendieck2003}, applied to our context.

\begin{prop}
 Let $(V,Z,s)$ be a well-adjusted triple, and assume that the associated generically log smooth family $f: X_0 \to S_0$ is a log toroidal family of elementary Gross--Siebert type. Let $B' \to B$ be a small extension with kernel $I$, and assume that we have a unisingular deformation $f: X_B \to S_B$. Then:
 \begin{enumerate}[label=\emph{(\roman*)}]
  \item The automorphisms of a given lifting $f': X_{B'} \to S_{B'}$ lie in $H^0(X_0,\Theta^1_{X_0/S_0}) \otimes_\kk I$.
  \item Given one lifting, the set of all liftings is a torsor under $H^1(X_0,\Theta^1_{X_0/S_0}) \otimes_\kk I$.
  \item The obstruction to the existence of a lifting is in $H^2(X_0,\Theta^1_{X_0/S_0}) \otimes_\kk I$.
 \end{enumerate}
\end{prop}
\begin{proof}
 Let $j: U_0 \to X_0$ be the inclusion. Using Lemma~\ref{sheaf-of-autom}, we find that the automorphisms of a lifting lie in $j_*(\Theta^1_{U_0/S_0} \otimes_\kk I)$. This is equal to $\Theta^1_{X_0/S_0} \otimes_\kk I$. The rest follows from general theory (use a \v{C}ech cover in the \'etale topology).
\end{proof}

\begin{cor}\label{sys-defo-exists}\note{sys-defo-exists}
 Let $f_0: X_0 \to S_0$ be as in the theorem. Then it admits a system of deformations $\D$ as defined in Definition~\ref{sys-of-defo}.
\end{cor}
\begin{proof}
 The central fiber $f_0: X_0 \to S_0$ is log Gorenstein because it comes from a toroidal crossing space. Let $\{V_\alpha\}_\alpha$ be an affine cover. For each $k$, we can choose inductively a unisingular deformation $V_{\alpha;k}$ over $S_k$ which is compatible with restrictions along $S_k \to S_{k + 1}$. They have the base change property by \cite[Thm.~8.2]{FFR2021} since the local models have it. Given $\alpha, \beta$, we have $V_{\alpha;0}|_{\alpha\beta} = V_{\beta;0}|_{\alpha\beta}$. Then, by induction on $k$ and by using the uniqueness of unisingular deformations an the affine open $V_\alpha \cap V_\beta$, we can lift this to an isomorphism $\psi_{\alpha\beta;k}: V_{\alpha;k}|_{\alpha\beta} \cong V_{\beta;k}|_{\alpha\beta}$ which is compatible with the restrictions along $S_k \to S_{k + 1}$. Every general $A \in \mathbf{Art}_Q$ admits a canonical map $S_A \to S_k$ for all $k >> 0$; they are compatible with $S_k \to S_{k + 1}$ and with every map $S_B \to S_{B'}$. Then we can define $V_{\alpha;A}$, $\rho_{\alpha;BB'}$, and $\psi_{\alpha\beta;A}$ by pull-back.
\end{proof}

\subsubsection*{The Interpretation of $\B_{X_0/S_0}$}

We provide an intrinsic interpretation of the role of $\B := \B_{X_0/S_0}$, which came out of a long exact sequence in the previous section. This material seems new; it is, as far as I see, not contained in the source \cite{GrossSiebertII} of the rest of the material. Let $f_0: X_0 \to S_0$ be the family associated to a unisingular well-adjusted triple, let $f_A: X_A \to S_A$ be a thickening, let $S_A \to S_{A'}$ be a small thickening by $I \subseteq A'$, and assume that we have a lifting $f_{A'}: X_{A'} \to S_{A'}$. On $\tilde U_0$, we introduce a \emph{presheaf} $\D$ of generically log smooth deformations which are locally isomorphic to the given deformation $f_{A'}: X_{A'} \to S_{A'}$:

\begin{defn}
 Let $W_0 \subseteq \tilde U_0$. Then a section in $\Gamma(W_0,\D)$ is an isomorphism class of quadruples $(\M,\alpha,\pi,\rho)$ where $\alpha: \M \to \cO_{X_{A'}}|_{W_0 \cap U_0}$ is a log structure on $W_0 \cap U_0$, $\pi: \M \to \M_{X_A}|_{W_0 \cap U_0}$ defines a strict closed embedding, and $\rho \in \Gamma(W_0 \cap U_0, \M)$ satisfies $\pi(\rho) = f_A^*(1)$ for $1 \in \M_{S_A}$ and $\alpha(\rho) = f_{A'}^*(t)$. These data are required to fit locally on $W_0$---not only on $W_0 \cap U_0$---into a commutative diagram
 \[
  \xymatrix{
   \M \ar[r]^{\Phi}_{\cong} \ar[d]^\alpha & \M_{X_{A'}}|_{W_0 \cap U_0} \ar[d] \\
   \cO_{X_A'}|_{W_0 \cap U_0} \ar[r]^{\phi}_{\cong} & \cO_{X_A'}|_{W_0 \cap U_0} \\
  }
 \]
 which defines an isomorphism of liftings of $f_A: X_A \to S_A$. An isomorphism between $(\M,\alpha,\pi,\rho)$ and $(\M',\alpha',\pi',\rho')$ is an isomorphism $\Phi: \M \to \M'$ of sheaves of monoids such that $\alpha' \circ \Phi = \alpha$, $\pi' \circ \Phi = \pi$, and $\Phi(\rho) = \rho'$.

\end{defn}

\begin{rem}
 The presheaf $\D$ is a sheaf on $\tilde U_0$.
\end{rem}
\begin{proof}
It suffices to show that $(\M,\alpha,\pi,\rho)$ has no non-trivial automorphisms. Every automorphism is in particular an automorphism of the log smooth lifting $(X_{A'}|_{W_0 \cap U_0},\M) \to S_{A'}$, i.e., it is given by a log derivation
$$\theta \in H^0(U_0 \cap W_0, \Theta^1_{X_0/S_0} \otimes I) = H^0(W_0, \Theta^1_{X_0/S_0} \otimes I)$$
as is discussed in \cite[Lemma~2.10]{GrossSiebertII} and more generally in \cite{Felten2022}. On the structure sheaf, the automorphism is trivial, so $\theta$ maps to $0$ under the forgetful map $\Theta^1_{X_0/S_0} \otimes I \to \T_{X_0/S_0} \otimes I$; the latter is injective, so $\theta = 0$.
\end{proof}

Let $\I = I \cdot \cO_{X_{A'}}$ be the kernel of the map $\cO_{X_{A'}} \to \cO_{X_A}$. Let $\A ut$ be the sheaf of automorphisms of $X_{A'}$ as a \emph{scheme} which induce the identity on $X_A$ and are compatible with $f_{A'}: X_{A'} \to S_{A'}$. If $\phi: \cO_{X_{A'}} \to \cO_{X_{A'}}$ is such an automorphism, then we obtain a derivation 
$$\delta_\phi = \phi - \mathrm{id}: \cO_{X_0} \to \I \cong \cO_{X_0} \otimes_k I;$$
this defines an isomorphism $\A ut \cong \T_{X_0/S_0} \otimes_k I$ of sheaves of groups. Now $\A ut$ acts on $\D$ via composition:
$$\A ut \times \D \to \D, \quad (\phi, (\M,\alpha,\pi,\rho)) \mapsto (\M,\phi \circ \alpha, \pi,\rho).$$

\begin{lemma}
 The action $\A ut \times \D \to \D$ induces an action $(\B \otimes_k I) \times \D \to \D$ via the sequence
 $$0 \to \Theta^1_{X_0/S_0} \otimes_k I \to \T_{X_0/S_0} \otimes_k I \to \B \otimes_k I \to 0.$$
 Under this action, $\D$ is a trivial $(\B \otimes_k I)$-torsor.
\end{lemma}
\begin{proof}
 First, we show that the action descends to $\B \otimes_k I$. Let $W_0 \subseteq \tilde U_0$ be an open, let $f': X' \to S_{A'}$ be a deformation representing a class $\gamma \in \Gamma(W_0,\D)$, and assume that $\phi \in \Gamma(W_0,\A ut)$ is in the image of $\Theta^1_{X_0/S_0} \otimes_k I \to \T_{X_0/S_0} \otimes_k I$; this means that we can find $\Phi: \M|_{W_0 \cap U_0} \to \M|_{W_0 \cap U_0}$ such that $(\phi,\Phi)$ is an automorphism of $X'$, cf.~\cite[Lemma~2.10]{GrossSiebertII}. Now the diagram 
 \[
  \xymatrix{
   \M \ar[r]^\Phi \ar[d]^{\phi \circ \alpha} & \M \ar[d]^\alpha\\
   \cO_{X_{A'}} \ar@{=}[r] & \cO_{X_{A'}} \\
  }
 \]
 shows that $\phi \cdot \gamma = \gamma$ in $\Gamma(W_0,\D)$, so the action descends. The arguments that $\D$ is a torsor under $\B \otimes_k I$ are similar. It is trivial because the given lifting $f_{A'}: X_{A'} \to S_{A'}$ defines a global section in $\Gamma(\tilde U_0,\D)$.
\end{proof}


\newpage

\backmatter



\chapter{Appendix}

\newpage

\section{The script for Example~\ref{base-change-violation-t3-w3}}
\label{t3-w3-script-sec}\note{t3-w3-script-sec}

We provide an annotated Macaulay2 script to compute that not all infinitesimal automorphisms lift in Example~\ref{base-change-violation-t3-w3}.

\noindent\hrulefill

\begin{lstlisting}[breaklines]

-- -----------------------------------------------

-- this is a script to show that for a general generically log smooth deformation, automorphisms over affine opens do not always lift
-- the example that we give is given by the equations xy = w^3 + t^3 and zw = tu

-- -----------------------------------------------

-- we will use the method reflexify of the following package to compute the map into the bidual

loadPackage("Divisor")

-- ---------------------------------------------------

-- definition of the family over A^1_t

OX = QQ[x,y,z,w,t,u]/(x*y - w^3 - t^3, z*w - t*u)

-- classical relative tangent space T_X/S; it is canonically isomorphic to the relative log tangent space once we use the compactifying log structures coming from t = 0

relativeTangentMap = matrix{{y,x,0,-3*w^2,0},{0,0,w,z,-t}}

TXS = kernel(relativeTangentMap)


-- --------------------------------------------------

-- we compute the base change to the Artinian ring Q[t]/(t^3)

OX2 = OX/ideal(t^3)

preThetaXS2 = TXS ** OX2  -- this is isomorphic to the relative log tangent sheaf on the log smooth locus

ThetaXS2 = target(reflexify(preThetaXS2,ReturnMap=>true)) -- this is the reflexive hull; it is isomorphic to the relative log tangent sheaf

-- ---------------------------------------------------

-- we compute further base change to Q[t]/(t^2)

OX1 = OX2/ideal(t^2)

r21 = map(OX1,OX2)

pThetaXS21 = tensor(r21,ThetaXS2) -- this is isomorphic to the relative log tangent sheaf on the log smooth locus

c21 = reflexify(pThetaXS21,ReturnMap=>true)

ThetaXS1 = target(c21) -- the reflexive hull is isomorphic to the relative log tangent sheaf

k21 = map(ThetaXS1,ThetaXS2,r21,matrix c21 * matrix coverMap(pThetaXS21))  

-- this is the restriction map from the relative log tangent sheaf over Q[t]/(t^3) to the one over Q[t]/(t^2)


-- ------------------

-- we compute further base change to S0 = Spec (N -> Q[t]/(t))

OX0 = OX1/ideal(t)

r10 = map(OX0,OX1)

pThetaXS10 = tensor(r10,ThetaXS1)

c10 = reflexify(pThetaXS10,ReturnMap=>true)

ThetaXS0 = target(c10)

k10 = map(ThetaXS0,ThetaXS1,r10,matrix c10 * matrix coverMap(pThetaXS10))

-- this is the restriction map from ThetaXS1 to ThetaXS0

k20 = k10 * k21 -- the restriction from ThetaXS2 to ThetaXS0

-- note that using k20, k21, and k10 is difficult from a programming perspective because Macaulay2 does not easily work with maps between modules over different rings

-- -------------------------------------------------

K20 = kernel k20 -- one of the few operations supported for maps between modules over different rings

-- unfortunately, K20 is not given as a submodule of ThetaXS2 but as a submodule of the generators module; we fix this with the following map

embK20 = inducedMap(ThetaXS2,K20,gens ThetaXS2)

Aut20 = image embK20

K10 = kernel k10

embK10 = inducedMap(ThetaXS1,K10,gens ThetaXS1)

Aut10 = image embK10

-- in principle, we now have the desired restriction map k21: Aut20 -> Aut10 that we wish to show not surjective; unfortunately, studying this map is pretty complicated in Macaulay2---we can't just command image(Aut20 -> ThetaXS2 -> ThetaXS1) and then compare it to Aut10

-- ---------------------------------------

-- we compute k21(Aut10) as a subset of ThetaXS1

G = source coverMap(K20)

p21 = k21 * embK20 * coverMap(K20) -- the image of p21 in ThetaXS1 is k21(Aut20); the source G of this map is a free module

r = (rank G) - 1 -- the rank of G - 1 = number of generators in this module

ze = p21(G_0) -- the image of the first generator of G in ThetaXS1 happens to be 0; thus, this object is the zero object in ThetaXS1

-- because OX2 -> OX1 is surjective, the image of p21 is the OX1-submodule spanned by all p21(G_i); we make a list that contains precisely the non-zero image vectors

L = {} -- empty list

for i from 0 to r do if (p21(G_i) != ze) then L = append(L,p21(G_i));

-- next, we wish to create the submodule in ThetaXS1 that is generated by the vectors in the list; since Macaulay2 still remembers them as coming from another ring, it cannot easily form this submodule; instead we extract the symbols of this list and use them to define a new matrix over OX1

st = toString(L)

"L-saving-file.txt" << "use OX1" << endl << "ML = matrix" << st << close

load "L-saving-file.txt" -- it is here that we had to remove all zero entries because they would not be interpreted over OX1 but over ZZ -- at least if I make only the list in the file and try to make the matrix on the M2 prompt

-- unfortunately, now the degrees of ML and of ThetaXS1 do not match; we need the following comparison map, which seems to me can only be defined by explicitly spelling out the matrix (it has been generated with an id command, then printed to a file, and from there copied into this script); we need the multiplication with t in one entry below in order to make it consider the matrix over the correct ring OX1 instead of ZZ

P = inducedMap(ThetaXS1, image ML, matrix {{1, t*0, 0, 0, 0, 0, 0, 0, 0, 0}, {0, 1, 0, 0, 0, 0, 0, 0, 0, 0}, {0, 0, 1, 0, 0, 0, 0, 0, 0, 0}, {0, 0, 0, 1, 0, 0, 0, 0, 0, 0}, {0, 0, 0, 0, 1, 0, 0, 0, 0, 0}, {0, 0, 0, 0, 0, 1, 0, 0, 0, 0}, {0, 0, 0, 0, 0, 0, 1, 0, 0, 0}, {0, 0, 0, 0, 0, 0, 0, 1, 0, 0}, {0, 0, 0, 0, 0, 0, 0, 0, 1, 0}, {0, 0, 0, 0, 0, 0, 0, 0, 0, 1}})

-- now image P is the image of Aut20 under k21 in ThetaXS1

-- a sanity check:

assert(isSubset(image P,Aut10)) -- this is what we expect 

-- the result:

assert(not isSubset(Aut10, image P))

-- hence image P is a proper subset of Aut10, and there are morphisms in Aut10 that do not lift to Aut20

-- --------------------------------------

-- we compute how much the two differ

Q = Aut10 / image P

Qs = Q ** OX1/ann(Q) -- Q and Qs are the same as QQ-vector spaces

Qsp = prune Qs

-- the ring of Qsp, i.e., OX1/ann(Q), happens to be isomorphic to QQ

assert(isFreeModule(Qsp))

assert(rank Qsp == 1)

-- thus Qsp, and hence Q, is a QQ-vector space of dimension 1








\end{lstlisting}

\noindent\hrulefill


\newpage


\newpage

\section{Notations for enhanced generically log smooth families}
 
 Let $f: X \to S$ be an enhanced generically log smooth family. Then we use the following notations:
 
 \vspace{\baselineskip}
 
 \begin{tabular}{p{2cm}l}
  $\W^i_{X/S}$ & reflexive relative log differential forms \\
  $\V^p_{X/S}$ & reflexive polyvector fields in negative degrees \\
  $\V\,\W^\bullet_{X/S}$ & reflexive two-sided Gerstenhaber calculus \\
  $\A^i_{X/S}$ & distinguished differential forms \\
  $\G^p_{X/S}$ & distinguished polyvector fields \\
  $\G\C^\bullet_{X/S}$ & two-sided Gerstenhaber calculus of distinguished forms and fields \\
  $\varpi^\bullet$ & embedding $\G\C^\bullet_{X/S} \to \V\,\W^\bullet_{X/S}$ \\
  $\Theta^1_{X/S}$ & relative log derivations \\
  $\Gamma^1_{X/S}$ & distinguished log derivations, considered inside $\Theta^1_{X/S}$ \\
  $\T_{X/S}$ & relative classical derivations \\
  $\T^1_{X/S}$ & classical first relative tangent sheaf (from cotangent complex) \\
 \end{tabular}

\section{Formulae in the curved case}\label{overview-form-sec}\note{overview-form-sec}

In this section, we give an overview over our definitions in the curved case. For the bigraded case, just omit the predifferential, and for the singly graded case, take the total degree. In this section, $\Lambda$ is a complete local Noetherian $\kk$-algebra with residue field $\kk$.

\section*{$\Lambda$-linear curved Lie algebras}

\begin{itemize}
 \item for every $\mathbf{i \geq 0}$, the object $L^i$ is a \textbf{flat} $\Lambda$-module which is \textbf{complete} with respect to the $\m_\Lambda$-adic topology; if $\theta \in L^i$, then its \textbf{total degree} is $|\theta| = i - 1$ (!);
 \item $\Lambda$-bilinear \textbf{Lie bracket} $[-,-]: L^i \times L^j \to L^{i + j}$ such that 
 $$[\theta,\xi] = -(-1)^{(|\theta| + 1)(|\xi| + 1)}[\xi,\theta];$$
 \textbf{Jacobi identity}
  $$[\theta,[\xi,\eta]] = [[\theta,\xi],\eta] + (-1)^{(|\theta| + 1)(|\xi| + 1)}[\xi,[\theta,\eta]];$$
  \item $\Lambda$-linear \textbf{predifferential} $\bar\partial: L^i \to L^{i + 1}$ with 
  $\bar\partial[\theta,\xi] = [\bar\partial\theta,\xi] + (-1)^{|\theta| + 1}[\theta,\bar\partial \xi]$;
  \item $\ell \in \m_\Lambda \cdot L^2$ with $\bar\partial^2(\theta) = [\ell,\theta]$ and $\bar\partial(\ell) = 0$.
\end{itemize}

\section*{$\Lambda$-linear curved Lie--Rinehart algebras}
\begin{itemize}
 \item for every $\mathbf{i \geq 0}$, the objects $F^i$ and $T^i$ are \textbf{flat} $\Lambda$-modules which are \textbf{complete} with respect to the $\m_\Lambda$-adic topology; if $a \in F^i$, then its \textbf{total degree} is $|a| = i$; if $\theta \in T^i$, then its \textbf{total degree} is $|\theta| = i - 1$;
 \item $\Lambda$-bilinear \textbf{product} $-\wedge -: F^i \times F^j \to F^{i + j}$ and a \textbf{unit element} $1 \in F^0$ with 
 $$a \wedge (b \wedge c) = (a \wedge b) \wedge c, \quad a \wedge b = (-1)^{|a||b|} b \wedge a, \quad 1 \wedge a = a;$$
 \item $\Lambda$-bilinear \textbf{multiplication} $\ast: F^i \times T^j \to T^{i + j}$ with 
 $$(a \wedge b) \ast \theta = a \ast (b \ast \theta), \quad 1 \ast \theta = \theta;$$
 \item $\Lambda$-bilinear \textbf{derivative} $\nabla^F: T^i \times F^j \to F^{i + j}$ with $a \ast \nabla^F_\theta(b) = \nabla^F_{a \ast \theta}(b)$;
 \item $\Lambda$-bilinear \textbf{Lie bracket} $[-,-] = \nabla^T: T^i \times T^j \to T^{i + j}$ with 
 $$[\theta,\xi] = -(-1)^{(|\theta| + 1)(|\xi| + 1)}[\xi,\theta];$$
 \item for $P = F,T$, the two \textbf{Jacobi identities}
 $$\nabla^P_{[\theta,\xi]}(p) = \nabla^P_\theta \nabla^P_\xi(p) - (-1)^{(|\theta| + 1)(|\xi| + 1)} \nabla^P_\xi\nabla^P_\theta(p);$$
 \item for $P = F,T$, the two \textbf{odd Poisson identities}
  $$\nabla^P_\theta(a \ast^P p) = \nabla_\theta(a) \ast^P p + (-1)^{(|\theta| + 1)|a|}a \ast^P \nabla^P_\theta(p);\footnote{$\ast^F = \wedge$ and $\ast^T = \ast$}$$
  \item two $\Lambda$-linear \textbf{predifferentials} $\bar\partial: F^i \to F^{i + 1}$ and $\bar\partial: T^i \to T^{i + 1}$ satisfying the four \textbf{derivation rules}
  $$\bar\partial\nabla^P_\theta(p) = \nabla^P_{\bar\partial\theta}(p) + (-1)^{|\theta| + 1}\nabla^P_\theta(\bar\partial p)$$
  and 
  $$\bar\partial(a \ast^P p) = \bar\partial(a) \ast^P p + (-1)^{|a|} a \ast^P \bar\partial p;$$
  \item $\ell \in \m_\Lambda$ with $\bar\partial^2(a) = \nabla^F_\ell(a)$ and $\bar\partial^2(\theta) = \nabla^T_\ell(\theta) = [\ell,\theta]$ as well as $\bar\partial(\ell) = 0$.
\end{itemize}

\section*{$\Lambda$-linear Lie--Rinehart pairs}

A $\Lambda$-linear Lie--Rinehart algebra plus:

\begin{itemize}
 \item for every $\mathbf{i \geq 0}$, the object $E^i$ is a \textbf{flat} $\Lambda$-module which is \textbf{complete} with respect to the $\m_\Lambda$-adic topology; if $e \in E^i$, then its \textbf{total degree} is $|e| = i$;
 \item $\Lambda$-bilinear \textbf{multiplication} $\ast^E: F^i \times E^j \to E^{i + j}$ with $(a \wedge b) \ast^E e = a \ast^E (b \ast^E e)$ and $1 \ast^E e = e$;
 \item $\Lambda$-bilinear \textbf{derivative} $\nabla^E: T^i \times E^j \to E^{i + j}$ with $a \ast^E \nabla^E_\theta(e) = \nabla^E_{a \ast \theta}(e)$;
 \item \textbf{Jacobi identity} 
 $$\nabla^E_{[\theta,\xi]}(e) = \nabla^E_\theta\nabla^E_\xi(e) - (-1)^{(|\theta| + 1)(|\xi| + 1)}\nabla^E_\xi\nabla^E_\theta(e);$$
 \item \textbf{derivation rule}
 $$\nabla^E_\theta(a \ast^E e) = \nabla^F_\theta(a) \ast^E e + (-1)^{(|\theta| + 1)|a|}a \ast^E \nabla^E_\theta(e);$$
 \item $\Lambda$-linear \textbf{predifferential} $\bar\partial: E^i \to E^{i + 1}$ satisfying the \textbf{derivation rules}
 $$\bar\partial\nabla^E_\theta(e) = \nabla^T_{\bar\partial\theta}(e) + (-1)^{|\theta| + 1}\nabla_\theta^E(\bar\partial e)$$
 and 
 $$\bar\partial(a \ast^E e) = \bar\partial(a) \ast^E e + (-1)^{|a|}a \ast^E \bar\partial(e)$$
 as well as $\bar\partial^2(e) = \nabla^E_\ell(e)$.
 
\end{itemize}

\newpage

\section*{$\Lambda$-linear curved Gerstenhaber algebras and calculi}

\paragraph{$\Lambda$-linear curved Gerstenhaber \textsl{algebras}}

\begin{itemize}
  \item for every $\mathbf{-d \leq p \leq 0}$ and $\mathbf{q \geq 0}$, the object $G^{p,q}$ is a \textbf{flat} $\Lambda$-module which is \textbf{complete} with respect to the $\m_\Lambda$-adic topology; if $\theta \in G^{p,q}$, then its \textbf{total degree} is $|\theta| = p + q$;
  \item $\Lambda$-bilinear $\wedge$-\textbf{product} $-\wedge-: G^{p,q} \times G^{p',q'} \to G^{p + p',q + q'}$
 and $\mathbf{1} \in G^{0,0}$ such that 
  $$\theta \wedge (\xi \wedge \eta) = (\theta \wedge \xi) \wedge \eta; \qquad \theta \wedge \xi = (-1)^{|\theta||\xi|} \xi \wedge \theta; \qquad 1 \wedge \theta = \theta;$$
  unadorned powers $\theta^n := \theta \wedge ... \wedge \theta$;
 \item $\Lambda$-bilinear \textbf{Lie bracket} $[-,-]: G^{p,q} \times G^{p',q'} \to G^{p + p' + 1,q + q'}$
 such that
  $$[\theta,\xi] = - (-1)^{(|\theta| + 1)(|\xi| + 1)}[\xi,\theta]; \qquad [\theta,1] = 0;$$ \textbf{Jacobi identity}
  $$[\theta,[\xi,\eta]] = [[\theta,\xi],\eta] + (-1)^{(|\theta| + 1)(|\xi| + 1)}[\xi,[\theta,\eta]];$$
  in particular, for $\varphi$ with $|\varphi|$ even: \quad $[\varphi,[\varphi,\xi]] = [\frac{1}{2}[\varphi,\varphi],\xi]$;
  \item \textbf{odd Poisson identities} 
  \begin{align}
   [\theta,\xi \wedge \eta] &= [\theta,\xi] \wedge \eta + (-1)^{(|\theta| + 1)|\xi|}\xi \wedge [\theta,\eta], \nonumber \\
   [\theta \wedge \xi, \eta] &= \theta \wedge [\xi,\eta] + (-1)^{(|\eta| + 1)|\xi|}[\theta,\eta] \wedge \xi; \nonumber
  \end{align}
  in particular, for $\varphi$ with $|\varphi|$ even, we have $[\varphi^n,\theta] = n[\varphi,\theta] \wedge \varphi^{n - 1}$ for $n \geq 1$;
  \item $\Lambda$-linear \textbf{predifferential}
  $\bar\partial: G^{p,q} \to G^{p,q + 1}$ with \quad $\bar\partial(1) = 0$ \quad and \textbf{derivation rules}
  $$\bar\partial[\theta,\xi] = [\bar\partial\theta,\xi] + (-1)^{|\theta| + 1}[\theta,\bar\partial \xi] \quad \mathrm{and} \quad \bar\partial(\theta\wedge\xi) = \bar\partial\theta \wedge \xi + (-1)^{|\theta|}\theta\wedge\bar\partial \xi;$$
  in particular, $\bar\partial(\varphi^n) = n\cdot \bar\partial\varphi \wedge \varphi^{n - 1}$ for $|\varphi|$ even;
  \item $\ell \in \m_\Lambda\cdot G^{-1,2}$\quad with \quad $\bar\partial^2(\theta) = [\ell,\theta]$;  \qquad $\bar\partial(\ell) = 0$; \qquad $\ell \wedge \ell = 0$.
 \end{itemize}
 
\paragraph{$\Lambda$-linear curved \textsl{two-sided} Gerstenhaber calculi}\hspace{1cm}

\vspace{.2cm}

\noindent A $\Lambda$-linear curved one-sided Gerstenhaber \textbf{calculus} (next page) plus:

\begin{itemize}
 \item $\Lambda$-bilinear \textbf{left contraction} $\vdash: G^{p,q} \times A^{i,j} \to G^{p + i, q + j}$ satisfying $\theta \vdash 1 = \theta$ and \\ $\theta \vdash (\alpha \wedge \beta) = (\theta \vdash \alpha) \vdash \beta$;
 \item for $\alpha \in A^{0,j}$, the identity $\theta \vdash \alpha = \theta \wedge (1 \vdash \alpha)$;
 \item for $\alpha \in A^{1,j}$, the identity 
 $$(\theta \wedge \xi) \vdash \alpha = (-1)^{|\xi||\alpha|}(\theta \vdash \alpha) \wedge \xi + \theta \wedge (\xi \vdash \alpha);$$
 \item for $\theta \in G^{-i,q}$ and $\alpha \in A^{i,j}$, the identity $\lambda(\theta \vdash \alpha) = (-1)^{i}\,\theta \ \invneg \ \alpha$;
  \item for $\omega \in A^{d,j}$, the identity $(\theta \vdash \alpha) \ \invneg \ \omega = (-1)^{|\alpha||\omega|} (\theta \ \invneg \ \omega) \wedge \alpha$;
  \item for $\omega^\vee \in G^{-d,q}$, the identity $\omega^\vee \vdash (\theta \ \invneg \ \alpha) = (-1)^{|\omega^\vee||\theta|}\, \theta \wedge (\omega^\vee \vdash \alpha)$;
  \item for $\theta \in G^{-1,q}$, the \textbf{special left mixed Leibniz rule}
  $$[\theta,\xi \vdash \alpha] = [\theta,\xi] \vdash \alpha + (-1)^{(|\theta| + 1)|\xi|}\,\xi \vdash \cL_\theta(\alpha);$$
  \item derivation rule $\bar\partial(\theta \vdash \alpha) = \bar\partial\theta \vdash \alpha + (-1)^{|\theta|}\, \theta \vdash \bar\partial\alpha$.
\end{itemize}

\newpage

\paragraph{$\Lambda$-linear curved \textsl{one-sided} Gerstenhaber calculi}\hspace{1cm}

\vspace{.2cm}

\noindent A $\Lambda$-linear curved Gerstenhaber algebra plus:

\begin{itemize}
  \item for every $\mathbf{0 \leq i \leq d}$ and $\mathbf{j \geq 0}$, the object $A^{i,j}$ is a \textbf{flat} $\Lambda$-module which is \textbf{complete} with respect to the $\m_\Lambda$-adic topology; if $\alpha \in A^{i,j}$, then its \textbf{total degree} is $|\alpha| = i + j$;
  \item $\Lambda$-bilinear $\wedge$-\textbf{product} $-\wedge - : A^{i,j} \times A^{i',j'} \to A^{i+i',j+j'}$
  and $\mathbf{1} \in A^{0,0}$ such that 
  $$(\alpha \wedge \beta) \wedge \gamma = \alpha \wedge (\beta \wedge \gamma); \quad \alpha \wedge \beta = (-1)^{|\alpha||\beta|}\beta \wedge \alpha; \quad 1 \wedge \alpha = 1; $$
  \item $\Lambda$-linear \textbf{de Rham differential} $\partial: A^{i,j} \to A^{i + 1,j}$ with $\partial^2 = 0$,\quad $\partial(1) = 0$, and the \textbf{derivation rule}
  $$\partial(\alpha \wedge \beta) = \partial(\alpha) \wedge \beta + (-1)^{|\alpha|} \alpha \wedge \partial(\beta);$$
  \item $\Lambda$-bilinear \textbf{contraction map}
   $\invneg \ : G^{p,q} \times A^{i,j} \to A^{p + i,q + j}$
  such that 
  $$1 \ \invneg \ \alpha = \alpha \quad \mathrm{and} \quad (\theta \wedge \xi) \ \invneg \ \alpha = \theta \ \invneg \ (\xi \ \invneg \ \alpha);$$
  \item $\Lambda$-bilinear \textbf{Lie derivative} $\cL_{-}(-): G^{p,q} \times A^{i,j} \to A^{p + i + 1, q + j}$ such that 
  $$\cL_{[\theta,\xi]}(\alpha) = \cL_\theta(\cL_\xi(\alpha)) - (-1)^{(|\theta| + 1)(|\xi| + 1)}\cL_\xi(\cL_\theta(\alpha)) \quad \mathrm{and} \quad \cL_1(\alpha) = 0;$$
  \item \textbf{mixed Leibniz rule}
  \begin{align}
    \theta \ \invneg \ \cL_\xi(\alpha) &= (-1)^{|\xi| + 1}([\theta,\xi] \ \invneg \ \alpha) + (-1)^{|\theta|(|\xi| + 1)}\cL_\xi(\theta \ \invneg \ \alpha) \nonumber \\
   \Leftrightarrow \:\: \qquad [\theta,\xi] \ \invneg\ \alpha &= (-1)^{|\xi| + 1}\: \theta \ \invneg \ \cL_\xi(\alpha) - (-1)^{(|\theta| + 1)(|\xi| + 1)}\cL_\xi(\theta \ \invneg\ \alpha) \nonumber \\
   \Leftrightarrow \qquad \cL_\theta(\xi \ \invneg \ \alpha) &= (-1)^{(|\theta| + 1)|\xi|}\: \xi \ \invneg\ \cL_\theta(\alpha) + [\theta,\xi] \ \invneg \ \alpha; \nonumber
  \end{align}
  \item \textbf{Lie--Rinehart homotopy formula}
  $$(-1)^{|\theta|}\cL_\theta(\alpha) = \partial(\theta \ \invneg \ \alpha) - (-1)^{|\theta|}(\theta \ \invneg \ \partial\alpha);$$
  consequently, 
  $$\cL_{\theta \wedge \xi}(\alpha) = (-1)^{|\xi|}\cL_\theta(\xi\ \invneg\ \alpha) + \theta\ \invneg\ \cL_\xi(\alpha);$$
  \item for $\varphi$ with $|\varphi|$ even, $\cL_{\varphi^n}(\alpha) = n \varphi^{n - 1} \ \invneg\ \cL_\varphi(\alpha) + \frac{n(n-1)}{2}([\varphi,\varphi] \wedge \varphi^{n - 2}) \ \invneg\ \alpha$ for $n \geq 2$;
  \item for $\theta \in G^{0,q}$, the identity $\theta \ \invneg\ \alpha = (\theta \ \invneg \ 1) \wedge \alpha$ holds;
  \item for $\theta \in G^{-1,q}$, the identity
  $$\theta\ \invneg\ (\alpha \wedge \beta) = (\theta\ \invneg\ \alpha) \wedge \beta + (-1)^{|\alpha||\theta|}\alpha \wedge (\theta\ \invneg\ \beta)$$
  holds; consequently, we have for $\theta \in G^{-1,q}$
  $$\cL_\theta(\alpha \wedge \beta) = \cL_\theta(\alpha) \wedge \beta + (-1)^{|\alpha|(|\theta| + 1)} \alpha \wedge \cL_\theta(\beta);$$
  \item $\Lambda$-linear \textbf{predifferential}
  $\bar\partial: A^{i,j} \to A^{i,j + 1}$
  with \enspace $\bar\partial^2(\alpha) = \cL_\ell(\alpha)$ \enspace and \enspace $\partial\bar\partial + \bar\partial\partial = 0$;
  \item \textbf{derivation rules}
  $$\bar\partial(\alpha \wedge \beta) = \bar\partial\alpha \wedge\beta + (-1)^{|\alpha|}\alpha \wedge\bar\partial\beta \quad \mathrm{and} \quad \bar\partial(\theta \ \invneg\ \alpha) = (\bar\partial\theta) \ \invneg \ \alpha + (-1)^{|\theta|} \theta \ \invneg\ \bar\partial\alpha;$$
  \item $\Lambda$-linear isomorphism $\lambda: G^{0,q} \to A^{0,q}, \: \theta \mapsto \theta \ \invneg\ 1,$ with $\lambda(\theta \wedge \xi) = \lambda(\theta) \wedge \lambda(\xi)$.
 \end{itemize}

\newpage

\section*{$\Lambda$-linear curved Batalin--Vilkovisky algebras and calculi}

\paragraph{$\Lambda$-linear curved Batalin--Vilkovisky \textsl{algebras}}

 \begin{itemize}
  \item for every $\mathbf{-d \leq p \leq 0}$ and $\mathbf{q \geq 0}$, the object $G^{p,q}$ is a \textbf{flat} $\Lambda$-module which is \textbf{complete} with respect to the $\m_\Lambda$-adic topology; if $\theta \in G^{p,q}$, then its \textbf{total degree} is $|\theta| = p + q$;
  \item $\Lambda$-bilinear $\wedge$-\textbf{product} $-\wedge-: G^{p,q} \times G^{p',q'} \to G^{p + p',q + q'}$
 and $\mathbf{1} \in G^{0,0}$ such that 
  $$\theta \wedge (\xi \wedge \eta) = (\theta \wedge \xi) \wedge \eta; \qquad \theta \wedge \xi = (-1)^{|\theta||\xi|} \xi \wedge \theta; \qquad 1 \wedge \theta = \theta;$$
  unadorned powers $\theta^n := \theta \wedge ... \wedge \theta$;
 \item $\Lambda$-bilinear \textbf{Lie bracket} $[-,-]: G^{p,q} \times G^{p',q'} \to G^{p + p' + 1,q + q'}$
 such that
  $$[\theta,\xi] = - (-1)^{(|\theta| + 1)(|\xi| + 1)}[\xi,\theta]; \qquad [\theta,1] = 0;$$ \textbf{Jacobi identity}
  $$[\theta,[\xi,\eta]] = [[\theta,\xi],\eta] + (-1)^{(|\theta| + 1)(|\xi| + 1)}[\xi,[\theta,\eta]];$$
  in particular, for $\varphi$ with $|\varphi|$ even: \quad $[\varphi,[\varphi,\xi]] = [\frac{1}{2}[\varphi,\varphi],\xi]$;
  \item \textbf{odd Poisson identities} 
  $$[\theta,\xi \wedge \eta] = [\theta,\xi] \wedge \eta + (-1)^{(|\theta| + 1)|\xi|}\xi \wedge [\theta,\eta]$$
  and 
  $$[\theta \wedge \xi, \eta] = \theta \wedge [\xi,\eta] + (-1)^{(|\eta| + 1)|\xi|}[\theta,\eta] \wedge \xi; $$
  in particular, for $\varphi$ with $|\varphi|$ even, we have $[\varphi^n,\theta] = n[\varphi,\theta] \wedge \varphi^{n - 1}$ for $n \geq 1$;
  \item $\Lambda$-linear \textbf{predifferential}
  $\bar\partial: G^{p,q} \to G^{p,q + 1}$ with \quad $\bar\partial(1) = 0$ \quad and \textbf{derivation rules}
  $$\bar\partial[\theta,\xi] = [\bar\partial\theta,\xi] + (-1)^{|\theta| + 1}[\theta,\bar\partial \xi] \quad \mathrm{and} \quad \bar\partial(\theta\wedge\xi) = \bar\partial\theta \wedge \xi + (-1)^{|\theta|}\theta\wedge\bar\partial \xi;$$
  in particular, $\bar\partial(\varphi^n) = n\cdot \bar\partial\varphi \wedge \varphi^{n - 1}$ for $|\varphi|$ even;
  \item $\ell \in \m_\Lambda\cdot G^{-1,2}$\quad with \quad $\bar\partial^2(\theta) = [\ell,\theta]$;  \qquad $\bar\partial(\ell) = 0$; \qquad $\ell \wedge \ell = 0$;
  \item $\Lambda$-linear \textbf{Batalin--Vilkovisky operator} $\Delta: G^{p,q} \to G^{p + 1,q}$
  with
  $$\Delta(1) = 0; \qquad \Delta^2 = 0; \qquad \Delta[\theta,\xi] = [\Delta(\theta),\xi] + (-1)^{|\theta| + 1}[\theta,\Delta(\xi)],$$
  and the \textbf{Bogomolov--Tian--Todorov formula}
  $$(-1)^{|\theta|}[\theta,\xi] = \Delta(\theta \wedge \xi) - \Delta(\theta) \wedge \xi - (-1)^{|\theta|} \theta \wedge \Delta(\xi);$$
  for $\varphi$ with $|\varphi|$ even, we have $\Delta(\varphi^n) = n\Delta(\varphi) \wedge \varphi^{n - 1} + \frac{n(n - 1)}{2}[\varphi,\varphi] \wedge \varphi^{n - 2}$ for $n \geq 2$;
  \item $y \in \m_\Lambda \cdot G^{0,1}$\quad with \quad $\bar\partial\Delta(\theta) + \Delta(\bar\partial\theta) = [y,\theta]$ \quad and \quad $\bar\partial y + \Delta\ell = 0$ \quad as well as
  $$y \wedge y = 0; \quad y \wedge \ell + \ell \wedge y = 0; \quad [\ell,y] + [y,\ell] = 0.$$
 \end{itemize}

\newpage

\paragraph{$\Lambda$-linear curved \textsl{one-sided} Batalin--Vilkovisky calculi}\hspace{1cm}

\vspace{.2cm}

\noindent A $\Lambda$-linear curved Batalin--Vilkovisky algebra plus:

\begin{itemize}
  \item for every $\mathbf{0 \leq i \leq d}$ and $\mathbf{j \geq 0}$, the object $A^{i,j}$ is a \textbf{flat} $\Lambda$-module which is \textbf{complete} with respect to the $\m_\Lambda$-adic topology; if $\alpha \in A^{i,j}$, then its \textbf{total degree} is $|\alpha| = i + j$;
  \item $\Lambda$-bilinear $\wedge$-\textbf{product} $-\wedge - : A^{i,j} \times A^{i',j'} \to A^{i+i',j+j'}$
  and $\mathbf{1} \in A^{0,0}$ such that 
  $$(\alpha \wedge \beta) \wedge \gamma = \alpha \wedge (\beta \wedge \gamma); \quad \alpha \wedge \beta = (-1)^{|\alpha||\beta|}\beta \wedge \alpha; \quad 1 \wedge \alpha = 1; $$
  \item $\Lambda$-linear \textbf{de Rham differential} $\partial: A^{i,j} \to A^{i + 1,j}$ with $\partial^2 = 0$,\quad $\partial(1) = 0$, and the \textbf{derivation rule}
  $$\partial(\alpha \wedge \beta) = \partial(\alpha) \wedge \beta + (-1)^{|\alpha|} \alpha \wedge \partial(\beta);$$
  \item $\Lambda$-bilinear \textbf{contraction map}
   $\invneg \ : G^{p,q} \times A^{i,j} \to A^{p + i,q + j}$
  such that 
  $$1 \ \invneg \ \alpha = \alpha \quad \mathrm{and} \quad (\theta \wedge \xi) \ \invneg \ \alpha = \theta \ \invneg \ (\xi \ \invneg \ \alpha);$$
  \item $\Lambda$-bilinear \textbf{Lie derivative} $\cL_{-}(-): G^{p,q} \times A^{i,j} \to A^{p + i + 1, q + j}$ such that 
  $$\cL_{[\theta,\xi]}(\alpha) = \cL_\theta(\cL_\xi(\alpha)) - (-1)^{(|\theta| + 1)(|\xi| + 1)}\cL_\xi(\cL_\theta(\alpha)) \quad \mathrm{and} \quad \cL_1(\alpha) = 0;$$
  \item \textbf{mixed Leibniz rule}
  \begin{align}
    \theta \ \invneg \ \cL_\xi(\alpha) &= (-1)^{|\xi| + 1}([\theta,\xi] \ \invneg \ \alpha) + (-1)^{|\theta|(|\xi| + 1)}\cL_\xi(\theta \ \invneg \ \alpha) \nonumber 
  \end{align}
  \item \textbf{Lie--Rinehart homotopy formula}
  $$(-1)^{|\theta|}\cL_\theta(\alpha) = \partial(\theta \ \invneg \ \alpha) - (-1)^{|\theta|}(\theta \ \invneg \ \partial\alpha);$$
  consequently, 
  $$\cL_{\theta \wedge \xi}(\alpha) = (-1)^{|\xi|}\cL_\theta(\xi\ \invneg\ \alpha) + \theta\ \invneg\ \cL_\xi(\alpha);$$
  \item for $\varphi$ with $|\varphi|$ even, $\cL_{\varphi^n}(\alpha) = n \varphi^{n - 1} \ \invneg\ \cL_\varphi(\alpha) + \frac{n(n-1)}{2}([\varphi,\varphi] \wedge \varphi^{n - 2}) \ \invneg\ \alpha$ for $n \geq 2$;
  \item for $\theta \in G^{0,q}$, the identity $\theta \ \invneg\ \alpha = (\theta \ \invneg \ 1) \wedge \alpha$ holds;
  \item for $\theta \in G^{-1,q}$, the identity
  $$\theta\ \invneg\ (\alpha \wedge \beta) = (\theta\ \invneg\ \alpha) \wedge \beta + (-1)^{|\alpha||\theta|}\alpha \wedge (\theta\ \invneg\ \beta)$$
  holds; consequently, we have for $\theta \in G^{-1,q}$
  $$\cL_\theta(\alpha \wedge \beta) = \cL_\theta(\alpha) \wedge \beta + (-1)^{|\alpha|(|\theta| + 1)} \alpha \wedge \cL_\theta(\beta);$$
  \item $\Lambda$-linear \textbf{predifferential}
  $\bar\partial: A^{i,j} \to A^{i,j + 1}$
  with \enspace $\bar\partial^2(\alpha) = \cL_\ell(\alpha)$ \enspace and \enspace $\partial\bar\partial + \bar\partial\partial = 0$;
  \item \textbf{derivation rules}
  $$\bar\partial(\alpha \wedge \beta) = \bar\partial\alpha \wedge\beta + (-1)^{|\alpha|}\alpha \wedge\bar\partial\beta \quad \mathrm{and} \quad \bar\partial(\theta \ \invneg\ \alpha) = (\bar\partial\theta) \ \invneg \ \alpha + (-1)^{|\theta|} \theta \ \invneg\ \bar\partial\alpha;$$
  \item \textbf{volume element} $\omega \in A^{d,0}$ inducing isomorphisms of $\Lambda$-modules
  $$\kappa: G^{p,q} \to A^{p + d,q}, \quad\theta \mapsto \theta \ \invneg\ \omega,\quad \mathrm{and}\quad \upsilon: A^{0,q} \to A^{d,q},\quad \alpha \mapsto \alpha \wedge \omega;$$
  \item $\Delta(\theta) \ \invneg \ \omega = \partial(\theta \ \invneg\ \omega)$\quad
  and\quad $y \ \invneg\ \omega = \bar\partial\omega$;
  \item $\Lambda$-linear isomorphism $\lambda: G^{0,q} \to A^{0,q}, \: \theta \mapsto \theta \ \invneg\ 1,$ with $\lambda(\theta \wedge \xi) = \lambda(\theta) \wedge \lambda(\xi)$.
 \end{itemize}
A $\Lambda$-linear curved \textbf{two-sided BV calculus} is a $\Lambda$-bilinear curved one-sided BV calculus which is at the same time a $\Lambda$-linear curved two-sided Gerstenhaber calculus.



\bibliography{solv-MC-note-v2.bib}
\bibliographystyle{plain}

\newpage

\printindex

\end{document}